\documentclass[twoside,12pt]{Latex/Classes/PhDthesisPSnPDF}

\usepackage{amsfonts,amsmath,braket,soul,color,bbm,mathtools}
\usepackage{amssymb,dsfont,amsbsy,amsthm,tensor,graphicx,epstopdf}
\usepackage{caption,multicol,textcomp,wrapfig,enumerate}
\usepackage[list=true]{subcaption}
\usepackage{paralist,setspace}
\usepackage{enumitem}

\newcommand{\id}{\mathbbm{1}}
\newcommand{\tr}{\mathrm{tr}}

\newcommand{\be}{\begin{equation}}
\newcommand{\ee}{\end{equation}}
\newcommand{\beq}{\begin{eqnarray}}
\newcommand{\eeq}{\end{eqnarray}}

\DeclareMathAlphabet\mathbfcal{OMS}{cmsy}{b}{n}

\let\oldsqrt\sqrt
\def\sqrt{\mathpalette\DHLhksqrt}
\def\DHLhksqrt#1#2{%
\setbox0=\hbox{$#1\oldsqrt{#2\,}$}\dimen0=\ht0
\advance\dimen0-0.2\ht0
\setbox2=\hbox{\vrule height\ht0 depth -\dimen0}%
{\box0\lower0.4pt\box2}}

\DeclareMathOperator{\Real}{Re}

\DeclareMathOperator{\sign}{sign}

\DeclareMathOperator{\OR}{OR}

\newcommand{\mathsym}[1]{{}}
\newcommand{\unicode}[1]{{}}

\newcommand\Rb{\mathbb{R}}
\newcommand\Cb{\mathbb{C}}
\newcommand\Mb{\mathbb{M}}
\newcommand\Zb{\mathbb{Z}}
\newcommand\Hb{\mathbb{H}}
\DeclareMathOperator{\Tr}{Tr}

\DeclareMathOperator{\mo}{mod}

\DeclareMathOperator{\iden}{id}
\DeclareMathOperator{\Inn}{Inn}

\newglossaryentry{k}
{
  name={\ensuremath{k}},
  description={a field, typically \ensuremath{\Rb} or \ensuremath{\Cb}},
  sort=k
}

\newglossaryentry{x}
{
  name={\ensuremath{x}},
  description={a fixed invertible element of the algebra \ensuremath{A}},
  sort=x
}

\newglossaryentry{Sigma}
{
  name={\ensuremath{\Sigma}},
  description={a two-dimensional compact manifold, orientable but not necessarily closed},
  sort=Se
}

\newglossaryentry{sigma}
{
  name={\ensuremath{\sigma}},
  description={the unique Frobenius algebra automorphism satisfying \ensuremath{\varepsilon(a \cdot b)=\varepsilon(\sigma(b) \cdot a)} for all \ensuremath{a,b \in A} known as the Nakayama automorphism},
  sort=Sd
}

\newglossaryentry{Sigma_g}
{
  name={\ensuremath{\Sigma_g}},
  description={a closed oriented surface of genus \ensuremath{g}},
  sort=Sf
}

\newglossaryentry{Sigma_k}
{
  name={\ensuremath{\Sigma^k}},
  description={a closed unorientable surface of genus \ensuremath{k}},
  sort=Sg
}

\newglossaryentry{C}
{
  name={\ensuremath{C_{abc}}},
  description={amplitude of an oriented triangle with edges labelled by states \ensuremath{a,b,c}},
  sort=Cb
}

\newglossaryentry{R}
{
  name={\ensuremath{R}},
  description={vertex amplitude},
  sort=R
}

\newglossaryentry{field_unit}
{
  name={\ensuremath{1_k}},
  description={unit element of the field \ensuremath{k}},
  sort=aab
}

\newglossaryentry{m}
{
  name={\ensuremath{m}},
  description={multiplication map on the vector space \ensuremath{A}},
  sort=m
}

\newglossaryentry{T}
{
  name={\ensuremath{T}},
  description={a triangulation of a piecewise-linear surface},
  sort=T
}

\newglossaryentry{ast}
{
  name={\ensuremath{\ast}},
  description={a linear map on the vector space \ensuremath{A} that is typically an involution},
  sort=m
}

\newglossaryentry{eps}
{
  name={\ensuremath{\varepsilon}},
  description={a non-degenerate linear map \ensuremath{A \to k} called a Frobenius form},
  sort=e
}

\newglossaryentry{unit}
{
  name={\ensuremath{1}},
  description={unit element of an algebra},
  sort=aaa
}

\newglossaryentry{Gamma}
{
  name={\ensuremath{\Gamma}},
  description={a trivalent graph composed of nodes and arrows referred to as a defect graph},
  sort=Gc
}

\newglossaryentry{B}
{
  name={\ensuremath{B^{ab}}},
  description={amplitude for a pair of identified edges labelled with states \ensuremath{a,b} and belonging to triangles with the same orientation},
  sort=Bc
}

\newglossaryentry{Ss}
{
  name={\ensuremath{S^{ab}}},
  description={amplitude for a pair of identified edges labelled with states \ensuremath{a,b} and belonging to triangles with opposite orientation},
  sort=sc
}

\newglossaryentry{S}
{
  name={\ensuremath{S}},
  description={set of states used to label edges of a triangulation},
  sort=sb
}

\newglossaryentry{Tb}
{
  name={\ensuremath{\mathfrak{T}(\mathfrak{b})}},
  description={set of states used to label an arrow \ensuremath{\mathfrak{b}} of a defect graph},
  sort=Tb
}

\newglossaryentry{s}
{
  name={\ensuremath{s}},
  description={a spin structure on a surface},
  sort=sa
}

\newglossaryentry{Z}
{
  name={\ensuremath{Z}},
  description={partition function},
  sort=zb
}

\newglossaryentry{lambda}
{
  name={\ensuremath{\lambda}},
  description={a crossing map \ensuremath{A \otimes A \to A \otimes A} deemed compatible with \ensuremath{(C,B,R)}},
  sort=lambda
}

\newglossaryentry{bi}
{
  name={\ensuremath{\tilde{\lambda}}},
  description={a linear map \ensuremath{H \times H \to k} satisfying the axioms of a bicharacter},
  sort=lambdatilde
}

\newglossaryentry{ZA}
{
  name={\ensuremath{\mathcal{Z}(A)}},
  description={centre of the algebra \ensuremath{A}},
  sort=zc
}

\newglossaryentry{z}
{
  name={\ensuremath{z}},
  description={preferred element on an FHK model},
  sort=z
}

\newglossaryentry{chi}
{
  name={\ensuremath{\chi}},
  description={preferred element on a spin model associated with odd spin structures},
  sort=xb
}

\newglossaryentry{eta}
{
  name={\ensuremath{\eta}},
  description={preferred element on a spin model associated with even spin structures},
  sort=heta
}

\newglossaryentry{phi}
{
  name={\ensuremath{\varphi}},
  description={a map \ensuremath{A \to A}, the curl map associated with a spin model},
  sort=vb
}

\newglossaryentry{w}
{
  name={\ensuremath{w}},
  description={preferred element on an KM model},
  sort=w
}

\newglossaryentry{A}
{
  name={\ensuremath{A}},
  description={a vector space which is usually also a Frobenius algebra},
  sort=aac
}

\newglossaryentry{V}
{
  name={\ensuremath{V}},
  description={a vector space, an \ensuremath{A}-\ensuremath{A} bimodule},
  sort=V
}

\newglossaryentry{H}
{
  name={\ensuremath{H}},
  description={a finite group of order \ensuremath{|H|}},
  sort=H
}

\newglossaryentry{G}
{
  name={\ensuremath{G}},
  description={a graph constructed from the dual of a triangulation and boundary data with evaluation denoted as \ensuremath{|G|}},
  sort=Ga
}

\newglossaryentry{Gg}
{
  name={\ensuremath{G_{\Gamma}}},
  description={a graph constructed from the dual of a triangulated surface with defects and its boundary data, with evaluation denoted as \ensuremath{|G_{\Gamma}|}},
  sort=Gb
}

\newglossaryentry{Ct}
{
  name={\ensuremath{C}},
  description={a \ensuremath{k}-valued trilinear form on the vector space \ensuremath{A}},
  sort=Ca
}

\newglossaryentry{l}
{
  name={\ensuremath{l}},
  description={a left action \ensuremath{A \otimes V \to V}},
  sort=l
}

\newglossaryentry{r}
{
  name={\ensuremath{r}},
  description={a right action \ensuremath{A \otimes A \to V}},
  sort=l
}

\newglossaryentry{Bb}
{
  name={\ensuremath{B}},
  description={a \ensuremath{k}-valued bilinear form on the vector space \ensuremath{A^{\gls{ast}}}},
  sort=Ba
}

\newglossaryentry{Bi}
{
  name={\ensuremath{B^{-1}}},
  description={a \ensuremath{k}-valued bilinear form on the vector space \ensuremath{A}},
  sort=Bb
}
\newtheorem{definition}{Definition}[chapter]
\newtheorem{theorem}[definition]{Theorem}
\newtheorem{lemma}[definition]{Lemma}
\newtheorem{proposition}[definition]{Proposition}
\newtheorem{corollary}[definition]{Corollary}

\theoremstyle{definition}
\newtheorem{example}[definition]{Example}

\ifpdf  
    \pdfcatalog { /PageMode (/UseOutlines)
                  /OpenAction (fitbh)  }
\fi

\title{Spin state sum models \\ in \\ two dimensions}

\ifpdf
  \author{\href{sara.oriana@gmail.com}{Sara Oriana Gomes Tavares, MSc}}
\collegeordept{}
\university{}
  \crest{\includegraphics[width=4cm]{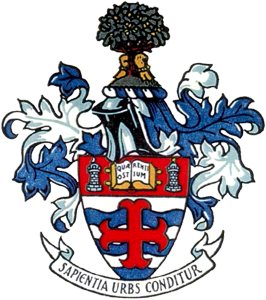}}
  
\else
  \author{Sara Oriana Gomes Tavares}
  \collegeordept{School of Mathematical Sciences}
  \university{University of Nottingham}
  \crest{\includegraphics[width=3cm]{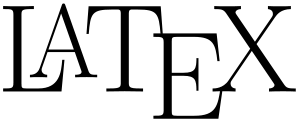}}
\fi
\renewcommand{\submittedtext}{Thesis submitted to the University of Nottingham \\ for the degree of Doctor of Philosophy}
\degree{}
\degreedate{July 2015}

\setcitestyle{square}
\begin{document}
%\language{english}

% sets line spacing
\renewcommand\baselinestretch{1.2}
\baselineskip=18pt plus1pt

%: ----------------------- generate cover page ------------------------

\maketitle  % command to print the title page with above variables

\frontmatter

\begin{abstracts}
\addcontentsline{toc}{chapter}{Abstract}

We propose a new type of state sum model for two-dimensional surfaces that takes into account topology and spin. The definition used -- new to the literature -- provides a rich class of extended models called spin models. Both examples and general properties are studied. Most prominently, we find this type of model can depend on a surface spin structure through parity alone and we explore explicit cases that feature this behaviour.

Further directions for the two dimensional world are analysed: we introduce a source of new information -- defects -- and show how they can enlarge the class of spin models available.  

\end{abstracts}

\begin{acknowledgements}
\addcontentsline{toc}{chapter}{Acknowledgements}

I remember clearly the reason why I chose to come to Nottingham -- my first meeting with John Barrett where after so many interviews I finally found the scientist that offered me what I was looking for: mathematics with a physics flavour to give it substance. It is hard to believe these four years have passed me by already and I thank my supervisor, Prof John Barrett, first and foremost for guiding me through a world of lattices, defects and graphs and showing me by example the value of conceptual honesty.

It is a truth universally-accepted that life in the quantum gravity group thrives on the back of pub nights, all-student groups, house festivities and Zakopane conferences. I thank Carlos and Steven, Gianluca and Alex for my first insights into life in Nottingham and very shameful Wii performances. I praise Ricardo, Manuel and Benito for their renewal of maths-dominance in the postgraduate theatre. I expect Johnny and Hugo to be entertaining the people of Poland for many years to come.

Shamefully, when I moved from London in 2010 I was not prepared for the accents of the North. I thank Antony for not dismissing me as a nitwit when I failed to understand a single word out of his mouth; I thank Ian, Drew, Mike and Tom for the amazing linguistic training that going to the pub with them offers and for showing me how an after-viva celebration should be.

I cannot forget the people that gave me freedom from doing research alone. Thank you Nico, Dave, Ivette and Lisa for having me on your conferences. Thank you Kristen for sharing with me your passion around literature and films (and making the best pancakes ever).  

M\~{a}e, Pai, Rosana, F\'{a}bio, Gata and Ninja -- thank you to all of you. For all the missed calls and emails, and all the worrying. It's done! 

\vspace{1cm}

The author acknowledges financial support from Funda\c{c}\~{a}o para a Ci\^{e}ncia e Tecnologia through grant no. SFRH/BD/68757/2010. 

\vspace{0.5cm}

\includegraphics[width=12cm]{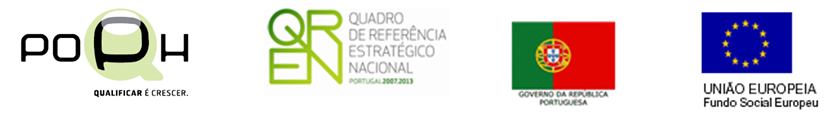}
\end{acknowledgements}

%: ----------------------- contents ------------------------

\setcounter{secnumdepth}{3} % organisational level that receives a numbers
\setcounter{tocdepth}{3}    % print table of contents for level 3
\tableofcontents            % print the table of contents

% the main text starts here with the introduction, 1st chapter,...
\mainmatter

\renewcommand{\chaptername}{} % uncomment to print only "1" not "Chapter 1"

%%%%%%%%%%%%%%%%%%%%%%%%%%%%%%%%%%%%%%%%%%%%%%%%
%%%%%%%%%%%%%%%%%%%%%%%%%%%%%%%%%%%%%%%%%%%%%%%%
\chapter{Introduction}
%%%%%%%%%%%%%%%%%%%%%%%%%%%%%%%%%%%%%%%%%%%%%%%%
%%%%%%%%%%%%%%%%%%%%%%%%%%%%%%%%%%%%%%%%%%%%%%%%

%%%%%%%%%%%%%%%%%%%%%%%%%%%%%%%%%%%%%%%%%%%%%%%%
%%%%%%%%%%%%%%%%%%%%%%%%%%%%%%%%%%%%%%%%%%%%%%%%
\section{What is a state sum model?}
%%%%%%%%%%%%%%%%%%%%%%%%%%%%%%%%%%%%%%%%%%%%%%%%
%%%%%%%%%%%%%%%%%%%%%%%%%%%%%%%%%%%%%%%%%%%%%%%%
The two-dimensional world is one of simplicity -- or so our physical intuition would tell us. General relativity's independence of space-time coordinates indicates that in these worlds there are no local degrees of freedom. Gravity can therefore be dealt with exactly and the manifold theories that model it find validation here, from dynamical triangulations \cite{Ambjorn} to loop quantum gravity \cite{Oriti}. 

The traditional approaches leave out one important physical notion -- spin. The main motivation behind this thesis is mending this oversight. Our starting point are topological theories -- the mathematical construction that encodes the essence of a theory that does not depend on local properties. Nonetheless, before introducing this concept we will first understand how it can be seen as naturally arising from what are known as `state sum models'. Throughout it is assumed the reader is familiar with linear algebra, tensor calculus on vector spaces and the fundamentals of differential geometry to the standard of references \cite{Nakahara} and \cite{MDG}. 

We treat continuous space-time as a limit. We start therefore not from a description that relies on smooth manifolds but from something more rudimentary: a simplex. Intuitively, we can see simplices as the building blocks of polyhedra: a collection of vertices, edges and triangles assembled together to form a hollow three-dimensional object which might possess holes and boundaries. The space thus obtained (a polyhedron) is a special kind of topological manifold, a piecewise-linear manifold. 

\begin{definition}
Let $X$ be a topological manifold. If there exists a polyhedron $\gls{T}$ and a homeomorphism $f \colon T \to X$ then $X$ is said to be a piecewise-linear manifold and the pair $(T,f)$ or simply $T$ is called a triangulation of $X$.
\end{definition}

As the number of components of the triangulation grows the `smoother' it will appear -- in this sense we will naively regard the continuous as the limit when the triangulation becomes ever finer. In other words, the limit when a piecewise-linear manifold approaches smoothness. The interpretation one gives to a theory developed on a triangulation can be of fundamental or technical value. On one hand, we can regard the existence of a continuous world as an idealisation: discreteness should be recovered at a certain scale for which the current primordial candidate is the Planck scale. On the other, the presence of this kind of `lattice' can be regarded as an instrument which allows us to simplify the theory but must not be regarded as exact. 

To avoid convergence issues we assume the number of vertices $v$, edges $e$ and triangles $t$ in the triangulation to be finite -- this means its manifold counterpart is compact. 

\begin{definition}
Let $k$ be a field. A state sum model on a triangulation $T$ consists of: one, assignments $v \mapsto A(v)$, $e \mapsto A(e)$ and $t \mapsto A(t)$ where the numbers $A(v), A(e), A(t) \in k$ are referred to as amplitudes; and two, an evaluation map  
\begin{align}
Z(T)= \sum_{v \in T}\sum_{e \in T}\sum_{t \in \gls{T}} A(v)A(e)A(t)  
\end{align}
referred to as the partition function.
\end{definition}

Given the generality of the formulation above it is perhaps not surprising for a number of theories to fall under this classification -- remarkably, discretised versions of gauge theory \cite{Oeckl}. We are, however, interested in a particular class of these models: those which do not depend on the choice of triangulation. By this we mean $Z(T)$ and $Z(T')$ must match if the piecewise-linear manifolds associated with $\gls{T}$ and $T'$ can be regarded as the same, which is to say there is a homeomorphism between those spaces. We thus circumvent having to make a philosophical choice on how to interpret a triangulation: we work with a discrete structure but the model is overall independent of such a choice. 

We are interested in a construction that is invariant under the action of homeomorphisms, known as topological, because of their relation to the concept of `general covariance'. You will have noticed thus far there has been no mention of a metric -- precisely, a theory that is said to be topological is defined on a manifold (up to homeomorphism) without a metric structure present and all amplitudes one calculates will be topological invariants. On the other hand, a quantum field theory where no a priori choice of metric is made is said to be generally covariant. However, any observables in such a theory will be necessarily only dependent on the topology of the manifold as well. We can therefore regard generally-covariant theories as topological theories and the relation is in fact reciprocal \cite{Witten3}. The general covariance principle of gravity is therefore the motivation behind studying topological theories.  

Chapter \S\ref{chapter:pure_models} is dedicated to the revision of two major contributions to this area \cite{Fukuma}\cite{Karimipour} -- it is a more detailed introduction to the models this thesis tries to generalise. The formalism used in \S\ref{chapter:pure_models} is re-imagined and expanded in chapter \S\ref{sec:diagram}. 

%%%%%%%%%%%%%%%%%%%%%%%%%%%%%%%%%%%%%%%%%%%%%%%%
%%%%%%%%%%%%%%%%%%%%%%%%%%%%%%%%%%%%%%%%%%%%%%%%
\section{Spin geometry}
%%%%%%%%%%%%%%%%%%%%%%%%%%%%%%%%%%%%%%%%%%%%%%%%
%%%%%%%%%%%%%%%%%%%%%%%%%%%%%%%%%%%%%%%%%%%%%%%%

Now that we know what a state sum model is we must discuss what information is to be encoded through the amplitudes $A$. In classical discrete gauge theory these assignments would represent information from a principle $\mathcal{G}$-bundle, where $\mathcal{G}$ is a Lie group \cite{Oeckl}. In such theories the amplitudes are used to encode geometrical information not accessible merely through a triangulation, since it merely possesses information about the topology of the piecewise-linear manifold we regard as our space-time. Chapter \S\ref{ch:spin} is dedicated to extending this principle to a different type of extrinsic data: spin. 

First, however, we must understand what is meant by such an elusive concept as `spin information'. A concept that would immediately come to mind on a physical scenario is that of a fermion: a Grassmann-valued field which is acted on by a special type of group, the double cover of $SO(n)$ (where $n$ is typically the space-time dimension). Perhaps the most intuitive way of using amplitudes to encode spin is therefore to present a discretised version of a fermion field on a triangulation. Such an attempt can be found in the works of Hamber \cite{Hamber},  Fairbairn \cite{Fairbairn}, and Barrett, Kerr and Louko \cite{Kerr}. However, with the exception of one-dimensional models the introduction of this fermion field comes at a high price: in all known examples topological invariance is lost.  

To address this issue we have abandoned the idea of a fermion as the fundamental concept to use. We have instead directed our attention to something more primitive: a spin structure. Intuitively, we would see this structure as a pre-requisite: it must exist if a theory based on fermions is to be possible. This statement is made more precise through the following definition \cite{Cimasoni} and by recalling that a fermion field can be seen as a section of a spin bundle \cite{Nakahara}.

\begin{definition} \label{def:spin-intro}
Let $M$ be an oriented $n$-dimensional Riemannian manifold, and let $P_{SO} \to M$ be the principal $SO(n)$-bundle associated to its tangent bundle. A spin structure on $M$ is a principal $Spin(n)$-bundle $P \to M$ together with a 2-fold covering map $P \to P_{SO}$ which restricts to the covering map $Spin(n) \to SO(n)$ on each fiber.
\end{definition} 

In chapter \S\ref{ch:spin} we describe how the presence of a spin structure can be used to inform the amplitudes of a state sum model -- a definition of spin structure equivalent to \ref{def:spin-intro} but more tailored to the topological setting will be used. We show that the amount of spin information obtained in this manner is restricted: for example, the $2^{2g}$ inequivalent spin structures a Riemann surface of genus $g$ can be equipped with fall into one of two categories and only those categories can be distinguished by our model. This work is based on the paper \cite{Tavares} co-authored with John Barrett that has been accepted for publication by the journal \emph{Communications in Mathematical Physics}. 

Chapters \S\ref{ch:defects} and \S\ref{ch:spin_defects} pave the way to merging the formalism of spin models with that of `defects' -- extra sources of information. It is shown through simple examples how more spin structure information can be extracted in the presence of such new data. Finally, future directions of research are analysed through the final remarks of chapter \S\ref{sec:cat}. 

%%%%%%%%%%%%%%%%%%%%%%%%%%%%%%%%%%%%%%%%%%%%%%%%

%%%%%%%%%%%%%%%%%%%%%%%%%%%%%%%%%%%%%%%%%%%%%%%%%%%%%%%%%%%%%%%%%%%%%%
%%%%%%%%%%%%%%%%%%%%%%%%%%%%%%%%%%%%%%%%%%%%%%%%%%%%%%%%%%%%%%%%%%%%%%
\chapter{Pure state sum models}\label{chapter:pure_models}
%%%%%%%%%%%%%%%%%%%%%%%%%%%%%%%%%%%%%%%%%%%%%%%%%%%%%%%%%%%%%%%%%%%%%%
%%%%%%%%%%%%%%%%%%%%%%%%%%%%%%%%%%%%%%%%%%%%%%%%%%%%%%%%%%%%%%%%%%%%%%

%%%%%%%%%%%%%%%%%%%%%%%%%%%%%%%%%%%%%%%%%%%%%%%%%%%%%%%%%%%%%%%%%%%%%%
%%%%%%%%%%%%%%%%%%%%%%%%%%%%%%%%%%%%%%%%%%%%%%%%%%%%%%%%%%%%%%%%%%%%%%
\section{Naive state sum models} \label{sec:lattice_tft}
%%%%%%%%%%%%%%%%%%%%%%%%%%%%%%%%%%%%%%%%%%%%%%%%%%%%%%%%%%%%%%%%%%%%%%
%%%%%%%%%%%%%%%%%%%%%%%%%%%%%%%%%%%%%%%%%%%%%%%%%%%%%%%%%%%%%%%%%%%%%%

This section reviews the construction of state sum models according to the work of Fukuma, Hosono and Kawai~\cite{Fukuma}, with the calculation of examples. These state sum models are called naive state sum models to distinguish them from the generalisation to diagrammatic ones in \S\ref{sec:diagram}.

%%%%%%%%%%%%%%%%%%%%%%%%%%%%%%%%%%%%%%%%%%%%%%%%%%%%%%%%%%%%%%%%%%%%%%
\begin{figure}
\centering
\begin{subfigure}[t!]{0.47\textwidth}
                \centering
		\includegraphics{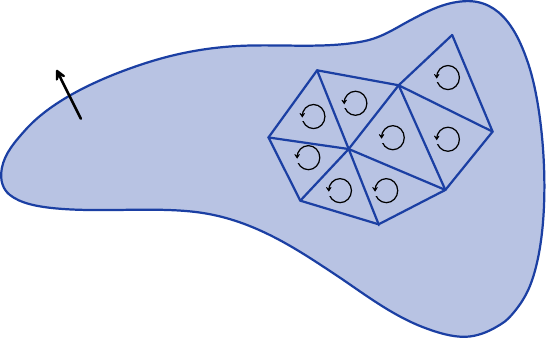}
		\caption[A patch of a triangulated surface]{\emph{A patch of a triangulated surface}. Each triangle inherits the orientation induced by the overall orientation of the surface.}
		\label{fig:glued_orient}
\end{subfigure}
\hspace{5mm}
\begin{subfigure}[t!]{0.47\textwidth}
                	\centering
		\vspace{8.8mm}
		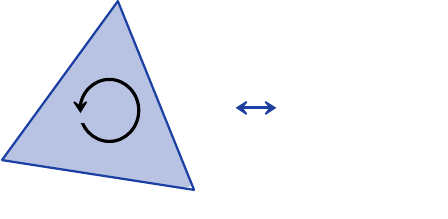
		\vspace{7.2mm}
		\caption[Associating amplitudes with triangles]{\emph{Associating amplitudes with triangles}. Each edge on a triangle is associated with one of a finite set of states \gls{S}. The amplitude for this oriented triangle is $\gls{C}$.}
		\label{fig:tri_label}
\end{subfigure}
\label{fig:triang_and_cyclicity}
\caption{Constructing a state sum model}
\end{figure} 
%%%%%%%%%%%%%%%%%%%%%%%%%%%%%%%%%%%%%%%%%%%%%%%%%%%%%%%%%%%%%%%%%%%%%%

The idea of a state sum model is to calculate an amplitude for a given triangulated manifold, possibly with a boundary. To simplify the description of the model, triangulations of the manifold are allowed to be degenerate: two simplices can intersect in more than one face. Degenerate triangulations are a particular case of a more general construction of complexes -- cell decompositions -- that are commonly used in the description of state sum models \cite{Oeckl}.

The amplitudes are numbers in a field \gls{k}, for which the main examples of interest here are $\gls{k}=\Rb$ or $\Cb$. A surface \gls{Sigma} is a two-dimensional compact manifold, orientable but not necessarily closed.  The surfaces are triangulated, and since they are compact, the number of vertices, edges and triangles is finite. The orientation of $\gls{Sigma}$ induces an orientation on each triangle. This means a triangle has a specified cyclic order of its vertices and these orientations are coherent as to preserve the overall orientation of the surface (see figure~\ref{fig:glued_orient}). Two triangles which share an edge in the triangulation are referred to as `glued' through that edge. 

A naive state sum model on an oriented triangulated surface $\gls{Sigma}$ has a set of amplitudes for each vertex, edge and triangle. These are glued together using a superposition of all states to give an overall amplitude to $\gls{Sigma}$.   

Each edge on a triangle is associated with one of a finite set of states \gls{S} and the amplitude for the oriented triangle with edge states $a,b,c \in \gls{S}$ is \gls{C} $\in \gls{k}$, as shown in figure~\ref{fig:tri_label}. These amplitudes are required to satisfy invariance under rotations, 
\begin{align} \label{eq:C_cycle}   
\gls{C}=C_{bca}=C_{cab},
\end{align}
which is to say they must respect the cyclic symmetry of an oriented triangle. If the orientation is reversed then the amplitude is $C_{bac}$ and therefore not necessarily equal to $\gls{C}$. 

The triangles are glued together using a matrix \gls{B} associated to each edge of the triangulation not on its boundary (an interior edge). Since the formalism for naive state sum models does not distinguish the two triangles meeting at the edge then one must require symmetry,
\begin{align} \label{eq:metric}
\gls{B}=B^{ba}.
\end{align}
Note that this condition is relaxed in \S \ref{sec:diagram}, together with a modification of the cyclic symmetry \eqref{eq:C_cycle}. 

Finally, each interior vertex has amplitude \gls{R} $\in \gls{k}$. This is a slight generalisation of the  formalism presented in \cite{Fukuma}, where $\gls{R}=\gls{field_unit}$  was assumed.

%%%%%%%%%%%%%%%%%%%%%%%%%%%%%%%%%%%%%%%%%%%%%%%%%%%%%
\begin{figure}
\centering
\begin{subfigure}[t!]{0.47\textwidth}
                \centering
		\vspace{7.3mm}
		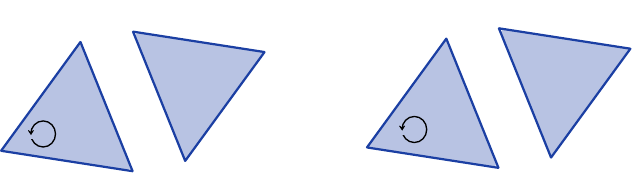 
		\vspace{3mm}
		\caption[Symmetry]{\emph{Symmetry}. The symmetry relation $\gls{B}=B^{ba}$ implies that the left- and right-hand sides of the equation above are equal, i.e., the amplitude for the two triangles glued together is invariant under a rotation by
$\pi$.}
		\label{fig:tri_metric}
\end{subfigure}
\hspace{5mm}
\begin{subfigure}[t!]{0.47\textwidth}
                \centering
		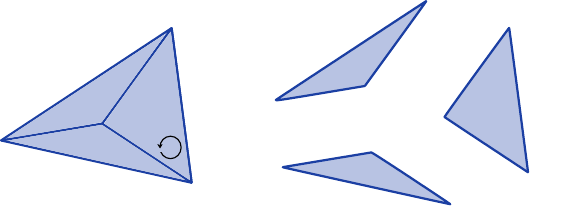
		\vspace{3mm}
		\caption[A triangulation of the disk]{\emph{A triangulation of the disk}. The partition function $Z_{abc}$ is constructed from the constants $\gls{Ct}$ associated to each triangle and matrices $\gls{B}$ associated to each interior edge.}
		\label{fig:disk}
\end{subfigure}
\label{fig:metric_and_disk}
\caption[Constructing partition functions]{Constructing partition functions}
\end{figure} 
%%%%%%%%%%%%%%%%%%%%%%%%%%%%%%%%%%%%%%%%%%%%%%%%%%%%%%

All the data needed to calculate the amplitude of a surface is now defined. Each edge in each triangle has a variable $a\in \gls{S}$. For a given value of each of these variables, the amplitude of a triangle $t$ is $A(t)=\gls{C}$ (with $a,b,c$ the three variables on the three edges), and likewise the amplitude of an edge $e$ is $A(e)=\gls{B}$. The amplitude of the surface is called the partition function and is given by the formula that involves summing over the states on all interior edges,
\begin{align} \label{eq:Z_surf}
\gls{Z}(\text{\footnotesize{boundary states}})=\gls{R}^{V}\sum_{\text{interior states}}\left(\prod_{\text{triangles }t}A(t)\prod_{\text{interior edges }e}A(e)\right),
\end{align}
with $V$ the number of interior vertices. For example, the amplitude of the triangulated disk of figure~\ref{fig:disk} is
\begin{align} \label{eq:Z_disk}
Z_{abc}=\gls{R}\,C_{e'dc}\,C_{af'e}\,C_{fbd'}\,B^{dd'}B^{ee'}B^{ff'},
\end{align}
using the Einstein summation convention for each paired index (a convention we will adopt from now on). The resulting partition function depends on the boundary data $a,b,c$, which are not summed. 

The formalism can be interpreted in terms of linear algebra. The states $a \in \gls{S}$ correspond to basis elements $e_a$ of a vector space \gls{A}. The amplitude $\gls{C}$ is the value of a trilinear form $\gls{Ct}\colon A\times A\times A\to \gls{k}$ on basis elements, 
$\gls{Ct}(e_a,e_b,e_c)= \gls{C}$. The form $\gls{Ct}$ can also be viewed as a linear map on the tensor product, 
\begin{align} \label{eq:C_trilinear}
\gls{Ct} \colon A \otimes A \otimes A \to \gls{k}. 
\end{align}
On the other hand, $\gls{Bb}$ can be viewed as a bilinear form on $A^*$, the algebraic dual of $\gls{A}$: $\gls{Bb}\colon A^*\times A^*\to \gls{k}$, with matrix elements $\gls{B}= B(e^a,e^b)$,  using the dual basis elements $e^a$. Equivalently we can write
\begin{align} \label{eq:B_element}
\gls{Bb}=e_a\otimes e_b \, \gls{B} \in A \otimes A.
\end{align} 
This linear algebra perspective means it is possible to regard state sum models as isomorphic if they are related by a change of basis; this is used henceforth. 

The bilinear form $\gls{Bb}$ can be used to `raise indices' -- it is combined with $\gls{Ct}$ to create an $A \otimes A \to \gls{A}$ map. Thus, using the definition $C_{ab}{}^c = C_{abd}\,B^{dc}$ there is a multiplication map $\gls{m} \colon A \otimes A \to \gls{A}$ with components 
\begin{equation}
\label{eq:multiplicationmap} \gls{m}(e_a \otimes e_b)= C_{ab}{}^c\, e_c.
\end{equation} 
The notations $\gls{m}(e_a \otimes e_b)=e_a \cdot e_b$ will be used interchangeably. The state sum model data can also be used to determine a distinguished element of $\gls{A}$,
\begin{align}
\gls{m}(\gls{Bb})=e_a\cdot e_b\,\gls{B}.
\end{align}
Throughout it is assumed the data for the state sum model are non-degenerate: $\gls{R}\neq 0$, $\gls{Bb}(\cdot,a)=0 \Rightarrow a=0$ and $\gls{Ct}(\cdot,\cdot,a)=0 \Rightarrow a=0$. This means $\gls{Bb}$ has an inverse $\gls{Bi}=B_{ab}\,e^a\otimes e^b \in A^* \otimes A^*$. This is defined by 
\begin{equation}
B_{ac}B^{cb}=\delta_a^{b}.
\label{eq:snake}
\end{equation}
 This determines a bilinear form on $\gls{A}$ with components $\gls{Bi}(e_a,e_b)=B_{ab}$ and can be used to lower indices. Note that this discussion of the formalism in terms of linear algebra does not depend on the symmetry of $\gls{Bb}$, and these definitions will also be used in later sections where the symmetry of $\gls{Bb}$ is dropped.

A topological state sum is one for which the partition function of a surface is independent of the triangulation. This is made precise by the following definition.

\begin{definition}
A state sum model is said to be topological if $\gls{Z}(\gls{Sigma})=\gls{Z}(\Sigma')$ whenever $\gls{Sigma}$ and $\Sigma'$ are two closed oriented triangulated surfaces on which the state sum model is defined and there is a piecewise-linear homeomorphism $f\colon \gls{Sigma} \to \Sigma'$ that preserves the orientation.
\end{definition}

Any two triangulations of a surface are connected by a sequence of the two Pachner moves, shown in figures~\ref{fig:Pach1} and \ref{fig:Pach2}, or their inverses. For a closed manifold this result is proved in  \cite{Pachner,Lickorish-moves}. In fact this result can be extended to a manifold with boundary \cite{CowardLackenby}, but this result is not used here. Thus it is sufficient to check for each Pachner move that the partition functions for the disk on the two sides of the move are equal.

%%%%%%%%%%%%%%%%%%%%%%%%%%%%%%%%%%%%%%%%%%%%%%%%%%%%%
\begin{figure}
\centering
\begin{subfigure}[t!]{0.47\textwidth}
                \centering
		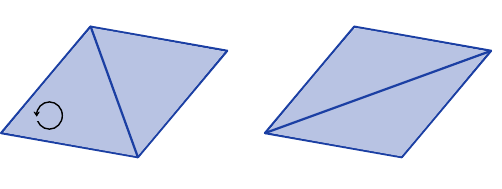 
		\vspace{3mm}
		\caption[Pachner move 2-2]{\emph{Pachner move $2$-$2$}.}
		\label{fig:Pach1}
\end{subfigure}
\hspace{5mm}
\begin{subfigure}[t!]{0.47\textwidth}
                \centering
		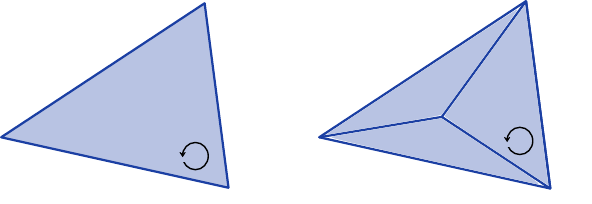
		\vspace{1.5mm}
		\caption[Pachner move 1-3]{\emph{Pachner move $1$-$3$}. }
		\label{fig:Pach2}
\end{subfigure}
\label{fig:Pachner_moves}
\caption{Topological moves}
\end{figure} 
%%%%%%%%%%%%%%%%%%%%%%%%%%%%%%%%%%%%%%%%%%%%%%%%%%%%%%

In the case of topological state sum models there is a connection between the vector space $\gls{A}$ and a Frobenius algebra. Recall that the dimension of an algebra is the dimension of its underlying vector space. A Frobenius algebra is a finite-dimensional associative algebra $\gls{A}$ with unit $\gls{unit} \in \gls{A}$ and a linear map $\gls{eps} \colon \gls{A} \to \gls{k}$ that determines a non-degenerate bilinear form $\gls{eps} \circ \gls{m}$ on $\gls{A}$. The linear map $\gls{eps}$ is called the Frobenius form. A Frobenius algebra is called symmetric if $\gls{eps} \circ \gls{m}$ is a symmetric bilinear form. Let $\gls{Bb}\in A\otimes A$ be the inverse of $\gls{Bi}=\gls{eps}\circ \gls{m}$ according to \eqref{eq:snake}. Then the Frobenius algebra is called special if $\gls{m}(\gls{Bb})$ is a non-zero multiple of the identity element.

A naive state sum model that obeys the Pachner moves is the type of model discussed by Fukuma, Hosono and Kawai, and so these are called FHK state sum models. The following result is a more precisely-stated version of their result in \cite{Fukuma}.

\begin{theorem}\label{theo:state_sum_algebra}
Non-degenerate naive state sum model data determine an FHK state sum model if and only if the multiplication map $\gls{m}$, the bilinear form $\gls{Bb}$ and the distinguished element $\gls{m}(\gls{Bb})$ determine on $\gls{A}$ the structure of a symmetric special Frobenius algebra with identity element $\gls{unit}=\gls{R}.\gls{m}(\gls{Bb})$. 
\end{theorem}
\begin{proof}
The proof begins by showing that the data determine a symmetric Frobenius algebra. The first Pachner move, shown in figure~\ref{fig:Pach1}, can be written
\begin{align} \label{eq:Pach1}
C_{ab}{}^e\, C_{ecd}= C_{bc}{}^e\, C_{aed}
\end{align}
and is equivalent to associativity of the multiplication. To see this note that using the notation \eqref{eq:multiplicationmap} of a multiplication, $(e_a \cdot e_b) \cdot e_c=C_{ab}{}^e\,C_{ec}{}^{f}\,e_f$ and $e_a \cdot (e_b \cdot e_c)=C_{bc}{}^e\,C_{ae}{}^f \,e_f$; hence, the identity \eqref{eq:Pach1} is $\gls{Bi}(e_a \cdot (e_b \cdot e_c),e_d)=\gls{Bi}((e_a\cdot e_b)\cdot e_c,e_d)$. Since the bilinear form $\gls{Bi}$ is non-degenerate this is equivalent to having an associative multiplication $\gls{m}$. A linear functional can be defined by setting $\gls{eps}(x)= \gls{Bi}(x,\gls{unit})$. The cyclic symmetry \eqref{eq:C_cycle} implies that $\gls{Bi}(x\cdot y,z)=\gls{Bi}(x,y\cdot z)$ and so $\gls{eps}(x\cdot y)=\gls{Bi}(x\cdot y,\gls{unit})=\gls{Bi}(x,y)$, which is non-degenerate and symmetric. 

The move in figure~\ref{fig:Pach2} requires the partition function of the disk \eqref{eq:Z_disk} to  equal $\gls{C}$. This is equivalent to
\begin{align}C_{ab}{}^c&=\gls{R}\,C_{ed}{}^c\,C_{af'}{}^e\,C_{fb}{}^d\,B^{ff'}\notag\\
&=\gls{R}\,C_{f'd}{}^h\,C_{ah}{}^c\,C_{fb}{}^d\,B^{ff'}
\end{align}
using associativity. For non-degenerate $\gls{Ct}$, and rewriting $C_{ab}{}^c=C_{ah}{}^c\delta_b^h$, this is equivalent to
\begin{align}\label{eq:Pachner_13}\delta_b^h&=\gls{R}\,C_{f'd}{}^h\,C_{fb}{}^d\,B^{ff'}\notag\\
&=\gls{R}\,C_{f'f}{}^d\,C_{db}{}^h\,B^{ff'}.
\end{align}
Recognising that $\gls{m}(\gls{Bb})=B^{ff'}C_{f'f}{}^d\,e_d$, expression \eqref{eq:Pachner_13} implies that $\gls{R}.\gls{m}(\gls{Bb})$ must be the unit element for multiplication, and hence $\gls{A}$ is an algebra; it is therefore a symmetric special Frobenius algebra. It is worth noting that the non-degeneracy of $\gls{Ct}$ is necessary here, as without it the algebra need not even be unital.

Conversely, given a symmetric Frobenius algebra with linear functional $\gls{eps}$, this defines a non-degenerate and symmetric bilinear form $\gls{Bi}= \gls{eps} \circ \gls{m}$ with property \eqref{eq:C_cycle}. The fact that the algebra is unital implies that $\gls{Ct}$ is non-degenerate. Finally, associativity and the property $\gls{R}.\gls{m}(\gls{Bb})=\gls{unit}$ guarantee the Pachner moves are satisfied, meaning the state sum model created is an FHK model. 
\end{proof}

For the cases  $\gls{k}=\Rb$ or $\Cb$ of interest in this paper the Frobenius algebras, and hence the state sum models, are easily classified.  The results for the symmetric Frobenius algebras in this section are stated here, with the proof of the classification given in a more general context in theorem~\ref{theo:diagram_semi_simple} of \S \ref{sec:diagram}.

Let $\Mb_{n}(\Cb)$ denote the algebra of $n \times n$ matrices over $\Cb$.
An FHK state sum model over the field $\Cb$ is isomorphic, by a change of basis, to one in which the algebra is a direct sum of matrix algebras, \begin{align} \label{eq:complex_matrix_algebra}
\gls{A} = \bigoplus\limits_{i=1}^N \Mb_{n_i}(\Cb).
\end{align}
The Frobenius form on an element $a = \oplus_i a_i$ is defined using the matrix trace on each factor:
\begin{align} \label{eq:cxfrob}
\gls{eps}(a) = \gls{R} \sum\limits_{i=1}^N n_i \Tr (a_i).
\end{align}

For the real case, the classification uses the division rings  $\Rb$, $\Cb$ and $\Hb$ (the quaternions) regarded as algebras over $\Rb$; these are denoted $\Rb$, $\Cb_{\Rb}$ and $\Hb_{\Rb}$, and the dimension of the division ring $D$ as an $\Rb$-algebra is denoted
$|D|$; thus $|\Rb|=1$, $|\Cb_\Rb|=2$, $|\Hb_\Rb|=4$. The imaginary unit in $\Cb$ is denoted $\hat{\imath}$ and the corresponding units for the quaternions $\hat{\imath}$, $\hat{\jmath}$ and $\hat{k}$.
 The real part of a quaternion is defined as $\Real(t+x\hat{\imath}+y\hat{\jmath}+z\hat{k})= t$ and the conjugate by $\overline{t+x\hat{\imath}+y\hat{\jmath}+z\hat{k}}=t-x\hat{\imath}-y\hat{\jmath}-z\hat{k}$. By abuse of notation we use $\Real(w)$ and $\overline{w}$ to denote the real part and conjugate of a complex number $w$ as well. The $n\times n$ matrices with entries in $D$ are denoted $\Mb_{n}(D)$ and are algebras over $\Rb$.

An FHK state sum model over the field $\Rb$  is isomorphic by a change of basis to one in which 
\begin{align} \label{eq:real_matrix_algebra}
 \gls{A} = \bigoplus\limits_{i=1}^N \Mb_{n_i}(D_i), \quad \text{ with }D_i=\Rb,\Cb_{\Rb},\text{or }\Hb_{\Rb}.
\end{align}
 The Frobenius form is defined by
\begin{align} \label{eq:realfrob}
\gls{eps}(a) =\gls{R} \sum\limits_{i=1}^N  |D_i|\,n_i \Real  \Tr(a_i).
\end{align}
The fact that these formulas do determine Frobenius algebras is proved here.
\begin{lemma} \label{lem:Frobenius_forms}
The equations \eqref{eq:cxfrob} and \eqref{eq:realfrob} determine symmetric Frobenius forms such that $\gls{R}.\gls{m}(\gls{Bb})=\gls{unit}$.
\end{lemma}
\begin{proof} That \eqref{eq:cxfrob} determines a symmetric Frobenius form follows from the fact that $\Tr(xy)$ is a non-degenerate symmetric bilinear form on $\Mb_{n}(\Cb) \ni x,y$. (For the sake of simplicity, matrix algebra multiplication is denoted $m(x,y)=xy$.)
For \eqref{eq:realfrob} there are three separate cases to handle: $\Mb_{n}(D)$ for $D=\Rb$, $\Cb_{\Rb}$ and $\Hb_{\Rb}$. The bilinear form $\Real \Tr (xy) $ reduces to $\Tr(xy)$ in the first case and this is non-degenerate on $\Mb_{n}(\Rb)$. In the $D=\Cb_\Rb$ case, if $\Real\Tr(xy)=0$ and $\Real\Tr(x(\hat{\imath}y))=0$ then $\Tr(xy)=0$. So $\Real\Tr(xy)=0$ for all $y\in\Mb_{n}(\Cb_\Rb)$ implies that $x=0$. Thus $\Real\Tr(xy)$ is a non-degenerate form. Finally, a similar proof works for $D=\Hb_\Rb$. In all these cases the bilinear form determined by $\Real\Tr$ is symmetric.

Let $\gls{k}=\Cb$. A basis for \eqref{eq:complex_matrix_algebra} is given by elementary matrices $\lbrace e_{lm}^{i}\rbrace_{l,m=1,n_i}^{i=1,N}$ satisfying $\left(e_{lm}^{i}\right)_{rs}=\delta_{lr}\delta_{ms}$. Then
\begin{align}\label{B_form_complex}
\gls{Bb}=\frac{\gls{field_unit}}{\gls{R}}\sum_{i,lm}\frac{\gls{field_unit}}{n_i}\,e_{lm}^{i}\otimes e_{ml}^i,
\end{align}
as can be verified by applying identity \eqref{eq:snake} to the above expression and using equation \eqref{eq:cxfrob}. Let $\gls{unit}=\oplus_i 1_i$; noticing $\sum_{lm}e^{i}_{lm}e^i_{ml}=n_i 1_i$, it is straightforward to conclude that $\gls{m}(\gls{Bb})=\gls{R}^{-1}\gls{unit}$. 

Suppose now that $\gls{k}=\Rb$ and let $\gls{A}$ be as in \eqref{eq:real_matrix_algebra}. Choose as a basis for the $i$-th component of $\gls{A}$ either $\lbrace e_{lm}^i\rbrace$, $\lbrace e_{lm}^i, \hat{\imath}\,e_{lm}^i\rbrace$ or $\lbrace e_{lm}^i, \hat{\imath}\,e_{lm}^i,\hat{\jmath}\,e_{lm}^i,\hat{k}\,e_{lm}^i\rbrace$ according to $D_i=\Rb$, $\Cb_{\Rb}$ or $\Hb_{\Rb}$, respectively. The element $\gls{Bb}$ associated with \eqref{eq:realfrob} will then take the form
\begin{align}
\gls{Bb}=\frac{\gls{field_unit}}{\gls{R}}\sum_{i,lm}\sum_{w_i}\frac{\gls{field_unit}}{|D_i|n_i}\, w_i\,e_{lm}^{i}\otimes_{\Rb} \overline{w_i}\,e_{ml}^i, \hspace{4mm}
w_i=
\begin{cases}
1 &(D_i=\Rb)\\
1,\hat{\imath} &(D_i=\Cb_{\Rb})\\
1,\hat{\imath},\hat{\jmath},\hat{k} &(D_i=\Hb_{\Rb})
\end{cases}.
\end{align}
Since the product $w_i\overline{w_i}=1$ for all $i$ then $\sum_{lm,w_i} w_i\,e_{lm}^{i} \overline{w_i}\,e_{ml}^i=n_i|D_i|1_i$. The identity $\gls{m}(\gls{Bb})=\gls{R}^{-1}\gls{unit}$ is therefore satisfied.
\end{proof}

The partition function for a surface can now be calculated for these examples. Let  $\gls{Sigma_g}$ denote an oriented surface of genus $g$.
Gluing two triangles together gives the partition function \eqref{eq:Pach1} of the disk with four boundary edges labelled with states $a,b,c,d$ which is equal to $\gls{eps}(e_a\cdot e_b\cdot e_c\cdot e_d)$. Gluing these boundary edges to make the sphere $\Sigma_0$ by identifying the states $a,d$ and $b,c$ results in the partition function
\begin{equation}\label{eq:sphere} \gls{Z}(\Sigma_0)=\gls{R}^3\,\gls{eps}(e_a\cdot e_b\cdot e_c\cdot e_d)B^{ad}B^{bc}=\gls{R}\,\gls{eps}(\gls{unit}).\end{equation}
Gluing opposite edges results in the torus 
\begin{equation}\label{eq:torus} \gls{Z}(\Sigma_1)=\gls{R}\,\gls{eps}(e_a\cdot e_b\cdot e_c\cdot e_d)B^{ac}B^{bd}=\gls{R}\,\gls{eps}(\gls{z}),\end{equation}
with $\gls{z}=e_a\cdot e_b\cdot e_c\cdot e_d\,B^{ac}B^{bd}$.
The surface $\gls{Sigma_g}$ for $g>0$ can be constructed from a disk with $4g$ boundary edges as presented in figure~\ref{fig:4gpoly}. This results in the partition function
\begin{align} \label{eq:invariant}
\gls{Z}(\gls{Sigma_g})=\gls{R}\,\gls{eps}\left( \gls{z}^g\right)
\end{align}
valid for all $g$. Although a specific orientation was picked when constructing expression \eqref{eq:invariant} the result is actually independent of orientation. Such a symmetry of the partition function is to be expected as it is easy to show orientation-reversing homeomorphisms exist for closed surfaces. Alternatively, this invariance can be proved directly through the partition function. For example, for the torus the two possible partition functions corresponding to two different orientations are given by expression \eqref{eq:torus} and $\gls{Z}'(\Sigma_1)=\gls{R}\,\gls{eps}(e_d\cdot e_c\cdot e_b\cdot e_a)B^{ac}B^{bd}$. By relabelling $(d,c,b,a) \to (a,b,c,d)$ and using the bilinear form symmetry it is established the two invariants are indeed equal.

%%%%%%%%%%%%%%%%%%%%%%%%%%%%%%%%%%%%%%%%%%%%%%%%%%%%%%%%%%
\begin{figure}
\centering
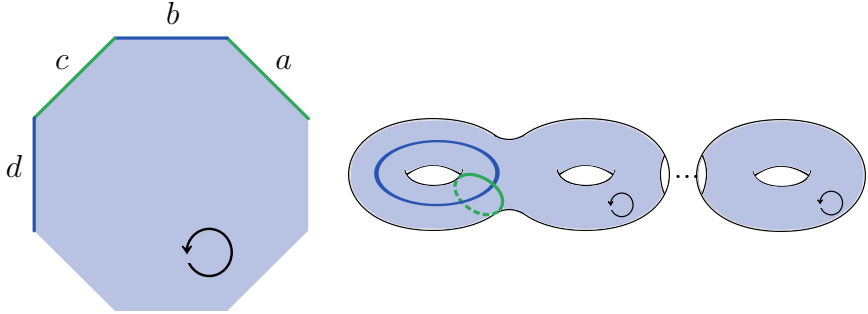  
\caption[Building a genus $g$ surface]{\emph{Building a genus $g$ surface.} $\gls{Sigma_g}$ is constructed from a disk with $4g$ boundary edges (internal edges are omitted). Edges are identified following the pattern shown on the right: $a \leftrightarrow c$, $b \leftrightarrow d$.  On the left, it can be seen how the glued edges give rise to curves on the surface.} 
\label{fig:4gpoly}
\end{figure} 
%%%%%%%%%%%%%%%%%%%%%%%%%%%%%%%%%%%%%%%%%%%%%%%%%%%%%%%%%%

The classification of FHK state sum models gives an explicit expression for $\gls{Z}(\gls{Sigma_g})$. This is based on the following calculations for the partition function in the case of simple algebras. For $\gls{A}=\Mb_n(\Cb)$, choose as a basis the elementary matrices $\lbrace e_{lm}\rbrace_{l,m=1,n}$. Then for a Frobenius form \eqref{eq:cxfrob} the element $\gls{z}$ is given by $\gls{z}=\gls{R}^{-2}n^{-2}\sum_{lm,rs}e_{lm}e_{rs}e_{ml}e_{sr}=\gls{R}^{-2}n^{-2}\gls{unit}$. This gives the partition function 
\begin{align}\label{eq:partitionfunction}
\gls{Z}(\gls{Sigma_g}, \Mb_n(\Cb))=\gls{R}^{2-2g}n^{2-2g},
\end{align}
a result also found in \cite{Lauda}. The same conclusion holds for $\Mb_{n}(\Rb)$, now with $\gls{R}\in\Rb$. For the case of $\Mb_{n}(\Cb_{\Rb})$, the element $\gls{z}$ again takes the form $\gls{z}=\gls{R}^{-2}n^{-2}\gls{unit}$ but it produces a new partition function
\begin{align}
\gls{Z}(\gls{Sigma_g}, \Mb_n(\Cb_{\Rb}))=2\gls{R}^{2-2g}n^{2-2g}
\end{align}
due to the extra factor of $|\Cb_{\Rb}|=2$ present in the Frobenius form \eqref{eq:realfrob}. Further details of this calculation are explained in the more general example~\ref{ex:semi_simple_canonical}.

Finally, for $\Mb_n(\Hb_{\Rb})$ it can be shown that $\gls{z}=4^{-1}\gls{R}^{-2}n^{-2}\gls{unit}$. Full details can be found in example~\ref{ex:semi_simple_canonical}. The partition function reads
\begin{align}
\gls{Z}(\gls{Sigma_g}, \Mb_n(\Hb_{\Rb}))=2^{2-2g}\gls{R}^{2-2g}n^{2-2g}.
\end{align}
Given the information gathered above, the most general form of an invariant from a symmetric Frobenius algebra can be stated. 
\begin{theorem} \label{lem:inv} 
Let $\gls{A}$ be a symmetric special Frobenius algebra over the field $\gls{k}=\Cb$ or $\Rb$, as in theorem~\ref{theo:state_sum_algebra}. The topological invariant $\gls{Z}(\gls{Sigma_g})$ constructed from $\gls{A}$ and an orientable surface $\gls{Sigma_g}$ is 
\begin{align} \label{eq:gen_inv}
\gls{Z}(\gls{Sigma_g})&=\gls{R}^{2-2g}\sum\limits_{i=1}^N  n_i^{2-2g} \hspace{10mm}\text{ if } \gls{k}=\Cb \text{, or }\\
\label{eq:gen_inv-real}
\gls{Z}(\gls{Sigma_g})&=\gls{R}^{2-2g}\sum\limits_{i=1}^N f(i,g) n_i^{2-2g}, \hspace{1mm}
f(i,g)=\begin{cases} 1 & (D_i = \Rb) \\ 2 & (D_i=\Cb_{\Rb}) \\ 2^{2-2g} & (D_i=\Hb_{\Rb})\end{cases}
\text{ if }\gls{k}=\Rb.
\end{align}
\end{theorem}
Another example of a Frobenius algebra is given by the complex group algebra. Recall an algebra can be built from any finite group $\gls{H}$ by taking formal linear combinations of the group elements. This algebra, denoted $\Cb \gls{H}$, is isomorphic to the algebra of $H$-valued complex functions which has elements $f = \sum_{h \in \gls{H}} f(h)h$, $f(h)\in \Cb$ and product defined according to
\begin{align} \label{eq:prod_group}
(f \cdot f') (h) = \sum_{l \in \gls{H}}f(l)f'(l^{-1}h).
\end{align}
A Frobenius form is $\gls{eps}(f)=\gls{R}\,f(\gls{unit})$. This form is the unique symmetric special Frobenius form such that $\gls{R}.\gls{m}(\gls{Bb})=\gls{unit}$. The Peter-Weyl decomposition~\cite{Dieck} gives an isomorphism with a complex matrix algebra satisfying the conditions of theorem~\ref{theo:state_sum_algebra}. The general form of the invariant associated with the group algebra is therefore
\begin{align}\label{eq:group_inv}
\gls{Z}(\gls{Sigma_g})=\gls{R}^{2-2g}\sum_{i \in I}(\dim i)^{2-2g},
\end{align}  
where each $i$ labels an irreducible group representation, a result that is given for a Lie group in \cite{Witten}.  Expression \eqref{eq:group_inv} agrees with the results of \cite{Fukuma} when $\gls{R}=\gls{field_unit}$.

%%%%%%%%%%%%%%%%%%%%%%%%%%%%%%%%%%%%%%%%%%%%%%
%%%%%%%%%%%%%%%%%%%%%%%%%%%%%%%%%%%%%%%%%%%%%%
\section{Unoriented state sum models} \label{sec:unoriented}
%%%%%%%%%%%%%%%%%%%%%%%%%%%%%%%%%%%%%%%%%%%%%%
%%%%%%%%%%%%%%%%%%%%%%%%%%%%%%%%%%%%%%%%%%%%%%

In this section we will provide a particular extension of the state sum model construction that allows us to create partition functions not only for oriented surfaces but also for non-orientable ones. Within this section alone, a surface $\Sigma$ is not required to be orientable.

The key idea is understanding the role played by orientation in the original models and therefore grasp how relaxing that condition might be possible. This generalisation was first studied by Karimipour and Mostafazadeh \cite{Karimipour} who extended the work of Fukuma, Hosono and Kawai to all surfaces. 

Let us recall that a naive state sum model -- not necessarily topological -- was built from the association of an amplitude to each triangle, edge and vertex. In particular, a map $\gls{Ct} \colon A \otimes A \otimes A \to \gls{k}$ is to be associated with a triangle, a map $\gls{Bb} \colon A^* \otimes A^* \to \gls{k}$ with each pair of identified edges and a constant $\gls{R} \in \gls{k}$ with every vertex. Explicitly or implicitly in the construction, the existence of a global orientation for the surface is reflected in the properties of $\gls{Bb}$ and $\gls{Ct}$.

The global orientation induces an orientation in each triangle and every triangle is equipped in this way with the same orientation; the induced orientation on each triangle is used to determine the cyclic symmetry of $\gls{Ct}$.
Moreover, this also means $\gls{Bb}$ is always associated to identified edges that belong to triangles with the same orientation.
Finally, for surfaces with boundaries, the global orientation induces an orientation on the boundary.   

We will relax the need for a global orientation through the following definition.

\begin{definition} \label{def:semi-oriented}
Let $T$ be a triangulation of a not necessarily orientable surface. $T$ is called semi-oriented if each of its triangles and each of its boundaries comes equipped with an orientation. 
\end{definition}

%%%%%%%%%%%%%%%%%%%%%%%%%%%%%%%%%%%%%%%%%%%%%%%%%%%%%
\begin{figure}[t!]
\hspace{4mm}
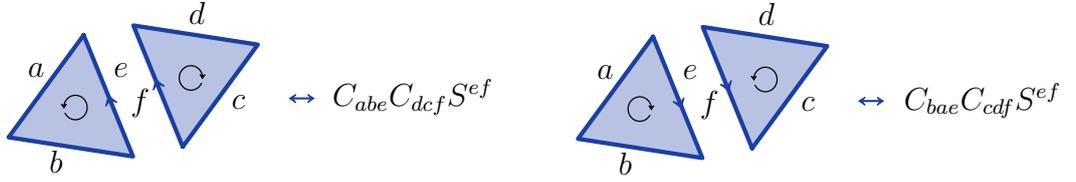 
\caption[Matrix $\gls{Ss}$]{\emph{Matrix $\gls{Ss}$}. The existence of triangles with different orientations gives rise to the inclusion of a new matrix, $\gls{Ss}$, on the model.}
\label{fig:CS}
\end{figure} 
%%%%%%%%%%%%%%%%%%%%%%%%%%%%%%%%%%%%%%

We will construct unoriented state sum models using semi-oriented triangulations. We assume each triangle $t$ is still to be associated with the map $\gls{Ct}$; however, the cyclic symmetry of $\gls{Ct}$ is now determined by the orientation of $t$ which varies from triangle to triangle. As illustrated in figure~\ref{fig:CS}, each pair of identified edges that belong to triangles with the same orientation will continue to be associated with the map $\gls{Bb}$; however, to a pair of identified edges that belongs to triangles with opposite orientations we will associate a new map,
\begin{align}
S \colon A^* \otimes A^* \to \gls{k}.
\end{align}
This map associates basis elements $e^a,e^b \in A^*$ with the matrix element $\gls{Ss}$. The order of labels (as for $\gls{B}$) is assumed not to matter:
\begin{align}
\gls{Ss}=S^{ba}.
\end{align}
The contraction $S_a{}^{b} \equiv B_{ac}S^{cb}$ can be used to define a new linear map $e_a \mapsto S_{a}{}^{b}e_b$ that we shall denote as $\gls{ast} \colon \gls{A} \to \gls{A}$.

Finally, we introduce a new consistency condition for the boundaries of the triangulated surface. Suppose $t$ is a triangle adjacent to a boundary $\partial$. Let the adjacent edges of $t$ and $\partial$ be labeled by $a$ and $b$ respectively.  We say the orientations of $\partial$ and $t$ are compatible if the orientations induced by $t$ and $\partial$ on the adjacent edges are opposite. If the orientations of $t$ and $\partial$ are compatible we include a $\gls{B}$ factor in the model; otherwise, we must include the factor $S^{ab}$. 

The partition function of an unoriented state sum model is still described by equation \eqref{eq:Z_surf} as the triangle edges adjacent to a boundary are considered to be interior ones. Nevertheless, we should note that edge amplitudes $A(e)$ can now read either $A(e)=\gls{B}$ or $A(e)=\gls{Ss}$.

The notion of an unoriented state sum model which is topological is now made precise.

\begin{definition}
An unoriented state sum model is said to be topological if $\gls{Z}(\gls{Sigma})=\gls{Z}(\gls{Sigma}')$ whenever $\gls{Sigma}$ and $\gls{Sigma}'$ are two triangulated not necessarily orientable surfaces on which the model is defined and there is a piecewise-linear homeomorphism $f \colon \gls{Sigma} \to \gls{Sigma}'$.
\end{definition}

Lickorish \cite{Lickorish-moves} gives us an equivalence between homeomorphic piecewise-linear manifolds with boundary and invariance under Pachner moves --  an equivalence that does not involve the notion of orientation. 

%%%%%%%%%%%%%%%%%%%%%%%%%%%%%%%%%%%%%%%%%%%%%%%%%%%%%
\begin{figure}[t!]
\hspace{30mm}
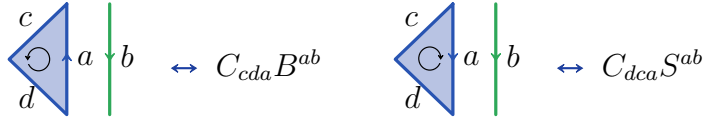 
\caption[Boundary rules]{\emph{Boundary rules.} According to definition~\ref{def:semi-oriented} any boundary $\partial$ on a semi-oriented triangulation comes with an orientation. Let $t$ be an internal triangle adjacent to $\partial$. If the orientations of $\partial$ and $t$ are compatible they are identified using $\gls{B}$; otherwise, $\gls{Ss}$ must be used.}
\label{fig:boundary}
\end{figure} 
%%%%%%%%%%%%%%%%%%%%%%%%%%%%%%%%%%%%%%

A naive state sum model together with a map $\gls{ast}$ that satisfy Pachner moves is exactly the type discussed by Karimipour and Mostafazadeh \cite{Karimipour}. Therefore, this type of model will be known as a KM model. Note that if $M$ is orientable and all local orientations of triangles are chosen to be the same, a KM model must reduce to an FHK model. Therefore, KM models are necessarily FHK models if the map $\gls{ast}$ is disregarded.  

KM models are FHK models with one extra structure: an involution map $\gls{ast} \colon \gls{A} \to \gls{A}$. For our purposes, an involution denotes a $\gls{k}$-linear map which is an anti-homomorphism, $(a \cdot b)^{\gls{ast}} = b^{\gls{ast}} \cdot a^{\gls{ast}}$, squares to the identity, $\gls{ast}\gls{ast}=\text{id}$, and preserves the algebra unit, $\gls{unit}=\gls{unit}^{\gls{ast}}$. The following is the analogous of theorem \ref{theo:state_sum_algebra} for unoriented surfaces, a result that first appeared in \cite{Karimipour} for the case of real algebras. Note that in what follows we need not assume the matrix $\gls{Ss}$ is invertible. 

\begin{theorem} \label{theo:KM} FHK state sum model data together with a linear map $\gls{ast} \colon \gls{A} \to \gls{A}$ determine a KM model if and only if $\gls{ast}$ is an algebra involution and $\gls{eps} \circ \gls{ast} = \gls{eps}$.
\end{theorem}
\begin{proof}
It will be sufficient to be concerned with moves involving triangles with incompatible orientations, as the remaining are satisfied by requiring the data to determine an FHK model. 

First, note one can only compare regions of the triangulation which share the same boundary: this means both the algebraic data and induced orientation of each boundary edge must match. Second, recall the assumption that a KM model must be independent of internal choices of orientation. This is enough to guarantee that the partition functions of the two triangulations depicted in figure~\ref{fig:Unor1} must match. Algebraically, this means
\begin{align}
C_{abg}C_{cdk}C_{efh}C_{k'g'h'}S^{gg'}S^{hh'}B^{kk'}&=C_{abg}C_{cdk}C_{efh}C_{k'h'g'}B^{gg'}B^{hh'}S^{kk'}\hspace{10mm} \notag\\
\Leftrightarrow \hspace{25mm} C_{k'g'h'}S^{gg'}S^{hh'}B^{kk'}&=C_{k'h'g'}B^{gg'}B^{hh'}S^{kk'}
\end{align}  
where the equivalence is guaranteed by the non-degeneracy of $\gls{Ct}$. Using the bilinear form to raise and lower indices and the cyclic symmetry of $\gls{C}$ we can manipulate this identity to a particularly suitable form:
\begin{align} \label{eq:ast_antihom}
S_g{}^{g'}S_{h}{}^{h'}C_{g'h'}{}^k=C_{hg}{}^{k'}S_{k'}{}^k.
\end{align}  
Note that $(e_g)^{\gls{ast}}\cdot (e_h)^{\gls{ast}}=S_g{}^{g'}S_{h}{}^{h'}C_{g'h'}{}^k e_k$ and $(e_h \cdot e_g)^{\gls{ast}}=C_{hg}{}^{k'}S_{k'}{}^k e_k$ according to the definition of $\gls{ast}$. Linear independence therefore guarantees that \eqref{eq:ast_antihom} is equivalent to $(e_g)^{\gls{ast}}\cdot (e_h)^{\gls{ast}}=(e_h \cdot e_g)^{\gls{ast}}$, which must be valid for every basis element. Note this identity also automatically guarantees that $\gls{unit}^{\gls{ast}}=\gls{unit}$: the relation $e_h^{\gls{ast}}=(\gls{unit} \cdot e_h)^{\gls{ast}}=\gls{unit}^{\gls{ast}}\cdot e_h^{\gls{ast}}$ holds and an algebra unit must be unique. Therefore, if the relation depicted in figure~\ref{fig:Unor1} is to be satisfied, the map $\gls{ast} \colon \gls{A} \to \gls{A}$ must be an algebra anti-homomorphism.  

%%%%%%%%%%%%%%%%%%%%%%%%%%%%%%%%%%%%%%%%%%%%%%%%%%%%%
\begin{figure}
\centering
\begin{subfigure}[t!]{0.46\textwidth}
                \centering
		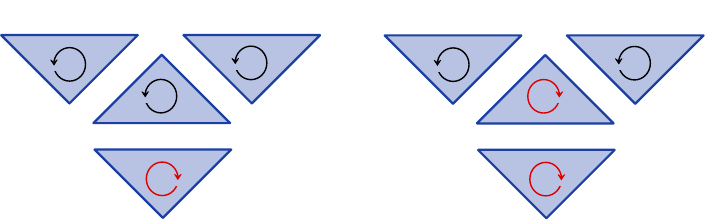 
		\vspace{2mm}
		\caption{}
		\label{fig:Unor1}
\end{subfigure}
\hspace{5mm}
\begin{subfigure}[t!]{0.46\textwidth}
                \centering
		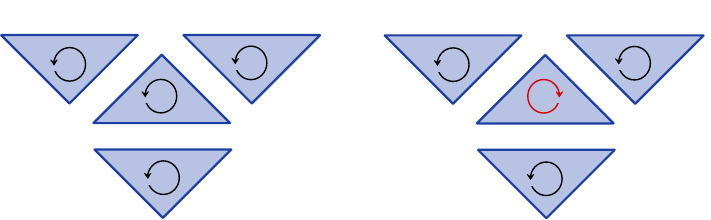
		\vspace{2mm}
		\caption{}
		\label{fig:Unor2}
\end{subfigure}
\caption[Orientation reversal]{\emph{Orientation reversal}. Invariance under Pachner moves has already been established since any KM model gives rise to an FHK model. One needs only establish that partition functions are invariant under orientation reversal of internal triangles.}
\label{fig:unor_moves}
\end{figure} 
%%%%%%%%%%%%%%%%%%%%%%%%%%%%%%%%%%%%%%%%%%%%%%%%%%%%%%

The partition functions associated with the triangulations depicted in figure~\ref{fig:Unor2} must also match. This means, again using the non-degeneracy of $\gls{Ct}$, that the identity
\begin{align}
C_{k'g'h'}S^{gg'}S^{hh'}S^{kk'}&=C_{k'h'g'}B^{gg'}B^{hh'}B^{kk'}
\end{align}  
must hold. By contracting both sides with $B_{bh}B_{ck}$ and using the symmetries of both $\gls{Ct}$ and $\gls{Ss}$ one obtains
\begin{align}
S_{b}{}^{h'}S_{c}{}^{k'}C_{h'k'}{}^{g'}S_{g'}{}^{g}&=C_{cb}{}^{g}.
\end{align}  
One straightforwardly recognises the relation between the right hand-side of this relation and the product $e_c \cdot e_b = C_{cb}{}^{g}e_g$. With a little more work one can also identify $((e_b)^{\gls{ast}} \cdot (e_c)^{\gls{ast}})^{\gls{ast}}=S_{b}{}^{h'}S_{c}{}^{k'}C_{h'k'}{}^{g'}S_{g'}{}^{g}e_g$. Since it is already known that $\gls{ast}$ must be an anti-homomorphism it follows that the identity of figure~\ref{fig:Unor2} imposes the relation $(e_c \cdot e_b)^{\gls{ast}\gls{ast}}=e_c \cdot e_b$. Since this must be true for every basis element and $\gls{Ct}$ is non-degenerate one concludes $\gls{ast}$ must square to the identity. Therefore, if $\gls{ast}$ is to determine a KM model it must be an algebra involution -- an algebra anti-automorphism that squares to the identity.

The symmetry requirement $\gls{Ss}=S^{ba}$ can be equivalently written as $S_{a}{}^{c}B_{cb}=S_{b}{}^{c}B_{ca}$. The inner product satisfies $\gls{Bi}=\gls{eps} \circ \gls{m}$ which means the symmetry of $\gls{Ss}$ can be translated as the property $\gls{eps} (e_a^{\gls{ast}} \cdot e_b)=\gls{eps}(e_a \cdot e_b^{\gls{ast}})$. Let $\gls{unit}=\sum_b 1_be_b=\sum_b 1_be_b^{\gls{ast}}=\gls{unit}^{\gls{ast}}$ where $1_b=1_b^{\gls{ast}}$ since $\gls{ast}$ is assumed to be linear. Linearity also allows one to rewrite the symmetry requirement as $\gls{eps} (e_a^{\gls{ast}} \cdot \gls{unit})=\gls{eps}(e_a \cdot \gls{unit})$ or, in other words, $\gls{eps} \circ \gls{ast} = \gls{eps}$. 

The converse of the proof is trivial: any algebra involution $\gls{ast}$ defined through $e_a^{\gls{ast}}=S_{a}{}^b e_b$ that satisfies $\gls{eps} \circ \gls{ast} = \gls{eps}$ gives rise to a map $S \colon A^* \otimes A^* \to \gls{k}$ via $\gls{Ss}=B^{ac}S_{c}{}^b$.  This maps will then necessarily satisfy the moves of figure~\ref{fig:unor_moves}.
\end{proof}

It is necessary to stress that the choice of $\gls{ast}$ structure does not influence the invariants constructed for closed orientable surfaces \cite{Fuchs}. Any differences may only arise for the non-orientable manifolds -- explicit examples will be provided later in this section.

\begin{corollary} 
Let $\gls{Sigma}$ be a closed orientable surface. Then the partition function $\gls{Z}(\gls{Sigma})$ constructed from $\gls{Sigma}$ and a KM model does not depend on the choice of involution $\gls{ast}$.
\end{corollary}
\begin{proof}
Fix an orientation for $\gls{Sigma}$. One is free to choose the orientations of triangles and boundaries in the semi-oriented triangulation to match those induced by the surface orientation. Therefore, the partition function will not depend on the map $S$. On the other hand, any other semi-oriented triangulation can be constructed from this particular choice by flipping the orientation of internal triangles. Since a KM model is invariant under these transformations, $\gls{Z}(\gls{Sigma})$ will not depend on $\gls{ast}$. 
\end{proof}

According to the classification of theorem \ref{theo:state_sum_algebra}, KM models must correspond to semi-simple algebras that are compatible with a choice of $\gls{ast}$ structure. We must therefore clarify the role of the $\gls{ast}$ operation when constructing topological invariants. Unlike an FHK model, where the choice of maps $\gls{Ct}$, $\gls{Bb}$ and constant $\gls{R}$ uniquely determine an algebra $\gls{A}$ and the Frobenius form $\gls{eps}$ it comes equipped with, there are multiple $\gls{ast}$ structures that can turn an FHK into a KM model. Our first step will therefore be to address the classification of involutions in semi-simple algebras. 

To tackle this classification problem we will need one main result: if $\gls{ast}$ and $\bullet$ are two involutions for an algebra $\gls{A}$ then there exists an automorphism $\omega \colon \gls{A} \to \gls{A}$ such that $\bullet = \omega \circ \gls{ast}$ \cite{Garrett, Fuchs}. To see this consider the map $\omega = \bullet \circ \gls{ast}$. This is a linear map that preserves the identity since $\gls{ast}$ and $\bullet$ are both linear and $\gls{unit}^{\gls{ast}}=\gls{unit}^{\bullet}=\gls{unit}$. It preserves multiplication, $\left((a \cdot b)^{\gls{ast}}\right)^{\bullet}=\left(b^{\gls{ast}}\cdot a^{\gls{ast}}\right)^{\bullet}=(a^{\gls{ast}})^{\bullet} \cdot (b^{\gls{ast}})^{\bullet}$, and is therefore a homomorphism. Finally, it has an inverse, $\gls{ast} \circ \bullet$ which allows one to conclude $\omega$ is an automorphism. It is then easy to see that $\bullet =(\bullet \circ \gls{ast}) \circ \gls{ast} = \omega \circ \gls{ast}$.  

On what follows, quaternionic hermitian conjugation $\ddagger$ of a matrix $a \in \Mb_{n}(\Hb_{\Rb})$ is defined as $a^{\ddagger}=\overline{a}^{\tr}$.
\begin{lemma} \label{lem:class_inv}
Let $A=\Mb_{n}(D)$ define a KM model. Then every involution $\bullet$ acts as $a^{\bullet}=sa^{\gls{ast}}s^{-1}$ for some invertible element $s \in \gls{A}$ satisfying $s=\mu s^{\gls{ast}}$ such that:
\begin{itemize} 
\item for $D=\Cb$ or $D=\Rb$, $\gls{ast}$ is the matrix transposition $\tr$ and $\mu=\pm \gls{field_unit}$;
\item for $D=\Cb_{\Rb}$, $\gls{ast}$ is either matrix transposition $\tr$ or hermitian conjugation $\dagger$ and $\mu=\pm \gls{field_unit}$ or $|\mu|=\gls{field_unit}$ respectively;
\item for $D=\Hb_{\Rb}$, $\gls{ast}$ is the quaternionic hermitian conjugation $\ddagger$ and $\mu=\pm \gls{field_unit}$.
\end{itemize}
\end{lemma}
\begin{proof}
First, the case $D=\Cb$ is studied. Since $\gls{A}$ is simple and its centre is trivial any automorphism $\omega$ is an inner one. Therefore, there exists an invertible element $s \in \gls{A}$ such that 
\begin{align} \label{eq:star_auto}
\omega (a)=sas^{-1}
\end{align}
for all $a \in \gls{A}$. This means that any two involutions $\bullet$, $\gls{ast}$ can be related according to
\begin{align} \label{eq:star_iso}
a^{\bullet}=s a^{\gls{ast}} s^{-1}.
\end{align}
This means that for $D=\Cb$, if a $\gls{ast}$ structure is proved to satisfy $\gls{eps} \circ \gls{ast} = \gls{eps}$ so will all other involutions given the cyclic symmetry of the Frobenius form: $\gls{eps}(a^{\bullet})=\gls{eps}(s a^{\gls{ast}} s^{-1})=\gls{eps}(a^{\gls{ast}})$. This conclusion extends also to the real algebras $\Mb_{n}(\Rb)$ and $\Mb_{n}(\Hb_{\Rb})$: as their centre is isomorphic to $\Rb$ all automorphisms are of the form \eqref{eq:star_auto} and their Frobenius forms are also symmetric. 

A natural choice of involution for $\Mb_{n}(\Rb)$ and $\Mb_{n}(\Cb)$ is the matrix transposition $\tr$. For the quaternion-valued matrices, however, $\tr$ is not an involution because the elements of $\Hb_{\Rb}$ do not commute with each other. Nevertheless, matrix transposition can be combined with quaternion conjugation $a \mapsto \overline{a}$ to obtain the involution $\ddagger$. 

One must check these involutions do respect the condition $\gls{eps} \circ \gls{ast} = \gls{eps}$. Since $\Tr(a)=\Tr(a)^{\tr}$ it is easy to conclude the condition holds for $\Mb_n(\Cb)$ where $\gls{eps}(a)=\gls{R}\,n\Tr(a)$. For $\Mb_{n}(\Rb)$, $\gls{eps}(a)=2R\, n\Real\Tr(a)$ so $\gls{eps} \circ \tr = \gls{eps}$ also in this case. Finally, for $\Mb_{n}(\Hb)$ one has $\gls{eps}(a^{\ddagger})=4R\,n \Real\Tr(a^{\ddagger})=4R\,n\Real\overline{\Tr(a)}=4R\,n\Real\Tr(a)=\gls{eps}(a)$. 

To conclude the analysis of $\Mb_{n}(\Rb)$, $\Mb_{n}(\Cb)$ and $\Mb_{n}(\Hb_{\Rb})$, one needs to verify the condition $s=\pm s^{\gls{ast}}$ is satisfied where $\gls{ast}$ is either $\tr$ or $\ddagger$ as appropriate. Applying the identity \eqref{eq:star_iso} twice one learns that $a^{\bullet\bullet}=s(s^{-1})^{\gls{ast}}(a^{\gls{ast}})^{\gls{ast}}s^{\gls{ast}}s^{-1}$. Using the fact both $\bullet$ and $\gls{ast}$ square to the identity one concludes the element $s(s^{-1})^{\gls{ast}}$ must be central. For the algebras under consideration this means $s(s^{-1})^{\gls{ast}}$ must be proportional to the identity. Therefore, there exists some $\mu \in \gls{k}$ ($\gls{k}=\Rb$ or $\gls{k}=\Cb$ as appropriate) such that $s^{\gls{ast}}=\mu s$. If we apply this relation twice we obtain $s^{\gls{ast}\gls{ast}}=s=\mu^2 s$. In other words, $s^{\gls{ast}}=\pm s$. Note that if $s$ is either real or complex $s^{\gls{ast}}=-s$ can only occur if $n$ is even; otherwise $s$ would not be invertible. 

For $D=\mathbb{C}_{\Rb}$, the centre of the algebra is no longer trivial: $\mathcal{Z}(\Mb_n(\mathbb{C}_{\Rb}))=\mathbb{C}_{\Rb}.\gls{unit}$. This means an automorphism $\omega \colon A \to A$ does not act as in \eqref{eq:star_auto} but instead as 
\begin{align}
a \mapsto s\psi(a)s^{-1},
\end{align}
where $\psi$ acts entry-wise on $a$ as a $\mathbb{C}_{\Rb}$-automorphism. Let $\lambda + \hat{\imath}\mu \in \mathbb{C}_{\Rb}$ where $\hat{\imath}$ is the imaginary unit. Since $\psi$ is $\Rb$-linear by definition it must satisfy $\psi(\lambda + \hat{\imath}\mu)=\lambda + \psi(\hat{\imath})\mu$. Consequently $\psi(\hat{\imath}^2)=\psi(-\gls{field_unit})=-\gls{field_unit}$. But since $\psi$ is an automorphism the relation $\psi(\hat{\imath}^2)=\psi(\hat{\imath})^2$ is verified; in other words $\psi(\hat{\imath})^2=-\gls{field_unit}$. This condition has only two possible solutions: $\psi(\hat{\imath})=\pm\hat{\imath}$. Therefore, $\psi$ can only be the identity map or the complex conjugate: $\psi(a)=a$ or $\psi(a)=\overline{a}$.

One must also investigate which consequences the relation $\bullet=\omega \circ \tr$ has for $s$, since also in this case the transpose operation is a natural choice of involution. Note that according to the definition of $\omega$ one has $a^{\bullet}=s\psi(a^{\tr})s^{-1}$. Applying this identity twice one obtains $a^{\bullet\bullet}=s\psi\left( \left( s\psi(a^{\tr})s^{-1}\right)^{\tr}\right)s^{-1}$. The map $\psi$ acts entry-wise; therefore, $\psi(a^{\tr})=\psi(a)^{\tr}$. Since it is an automorphism, one also knows $\psi(ab)=\psi(a)\psi(b)$. These simplifications hold
\begin{align}\label{eq:non_real_complex}
a=\big(s\psi(s^{-1})^{\tr}\big)a\big(s\psi(s^{-1})^{\tr}\big)^{-1}
\end{align}
where the fact $\gls{ast}$, $\tr$ and $\psi$ all square to the identity was also used. This means $s\psi(s^{-1})^{\tr}$ must be a central element; in other words $s=\mu\psi(s)^{\tr}$ for $\mu \in \Cb_{\Rb}$. Since either $\psi(a)=a$ or $\psi(a)=\overline{a}$, $s$ must either satisfy $s=\pm s^{\tr}$ or $s=\mu s^{\dagger}$ where $|\mu|=\gls{field_unit}$, according to the choice of $\Cb_{\Rb}$-automorphism.
\end{proof}

In light of the classification introduced in lemma \ref{lem:class_inv}, the following definition will prove useful.

\begin{definition}
Let $A=\Mb_{n}(D)$ together with an involution $\bullet$ define a KM model. Let $a^{\bullet}=sa^{\gls{ast}}s^{-1}$ with $s=\mu s^{\gls{ast}}$ for all $a \in A$. If $\gls{ast}=\tr$ and $\mu=+\gls{field_unit}$ ($\mu=-\gls{field_unit}$) the involution $\bullet$ is called symmetric (anti-symmetric). If $\gls{ast}=\ddagger$ and $\mu=+\gls{field_unit}$ ($\mu=-\gls{field_unit}$) the involution $\bullet$ is called hermitian (anti-hermitian). Finally, if $\gls{ast}=\ddagger$, $\bullet$ is referred to as hermitian for all possible $\mu$. 
\end{definition}

Since we are primarily interested in understanding under which circumstances the new structure $\gls{ast}$ will give rise to different partition functions, the analysis done in lemma \ref{lem:class_inv} will suffice. However, for the reader interested in understanding equivalence classes of $\gls{ast}$ structures in the context of KM models that treatment can be found in appendix \ref{app:TBA}. The following lemma tells us how the study we conducted for involutions in simple matrix algebras can be applied to semi-simple ones. This is a result recovered from \cite{Fuchs}.

\begin{lemma}
Let $A=\bigoplus\limits_{i=1}^N A_i$ where $A_i=\Mb_{n_i}(\Cb)$ (if $A$ is complex) or $A_i=\Mb_{n_i}(D_i)$ with $D_i=\Rb$, $\Cb_{\Rb}$ or $\Hb_{\Rb}$ (if $A$ is real). Let $A$, equipped with an involution $\gls{ast}$, determine a KM model. Then, the KM model is isomorphic to one in which $\gls{ast}$ decomposes as a direct sum of involutions, $\gls{ast}=\bigoplus\limits_{i=1}^N \gls{ast}_i$ where $\gls{ast}_i$ is an involution for $A_i$.
\end{lemma}
\begin{proof}
As with lemma \ref{lem:class_inv}, one exploits the fact that any two involutions $\gls{ast}$, $\bullet$ can be related by an automorphism $\omega$ via $\bullet =\omega \circ \gls{ast}$. A satisfactory classification of simple algebra involutions is already known. Therefore, one needs only to relate each $\gls{ast}$ to the straightforward choice of involution $\bigoplus\limits_{i=1}^N \gls{ast}_i$ where each $\gls{ast}_i$ is as in lemma \ref{lem:class_inv}. In what follows an element $a \in A$ is decomposed as $a=\oplus_i a_i$.

Consider the idempotent decomposition of the identity given by $\gls{unit}=\oplus_{i=1}^N 1_i$ where $1_i$ is the identity for $A_i$ and $1_i \cdot 1_j = \delta_{ij} 1_i$. This decomposition is unique up to reordering. This follows from one, the Artin-Wedderburn theorem \cite{Cohn} that tells us the decomposition of $A$ as $\bigoplus\limits_{i=1}^N \Mb_{n_i}(\Cb)$ or $\bigoplus\limits_{i=1}^N \Mb_{n_i}(D_i)$ determines the $n_i$ and the $D_i$ uniquely up to permutation of the $i$ and two, the fact $\gls{unit}\cdot a=a$ implies $1_i \cdot a_i = a_i$ and therefore $1_i$ must be the (unique) identity for the $A_i$. Note that for an automorphism $\omega \colon A \to A$ the set $\left\lbrace \omega(1_i)\right\rbrace$ is also a set of idempotents: 
\begin{align}
\oplus_{i=1}^N \omega (1_i)&=\omega (\oplus_{i=1}^N 1_i)=\omega (\gls{unit})=\gls{unit} ,\\
\omega(1_i) \cdot \omega (1_j) &= \omega(1_i \cdot 1_j) = \delta_{ij} \,\omega (1_i).
\end{align}
Because the set $\lbrace 1_i\rbrace$ is unique the automorphism must act on idempotents as a permutation $\pi$: $\omega (1_i)=1_{\pi(i)}$. Then $\omega(a_i)=\omega(a_i \cdot 1_i)=\omega(a_i) \cdot 1_{\pi(i)}$. This means the restriction $\omega_i \equiv \omega|_{A_i}$ gives rise to an isomorphism $A_{i} \to A_{\pi(i)}$. Since the $A_i$ are matrix algebras this implies their dimensions  must match: $n_i=n_{\pi(i)}$. Moreover, if $A$ is real having $A_i \simeq A_{\pi(i)}$ also implies $|D_i|=|D_{\pi(i)}|$.

Take as a canonical choice of involution for $A$ the map $\gls{ast}=\bigoplus\limits_{i=1}^N \gls{ast}_i$ where each $\gls{ast}_i$ is a canonical choice for $A_i$ according to lemma \ref{lem:class_inv}. Furthermore, note that for each such choice we have $\gls{ast}_i=\gls{ast}_{\pi(i)}$. Then, any other involution $\bullet$ reads
\begin{align}
(a)^{\bullet}=\bigoplus_{i=1}^N s_ia_i^{\gls{ast}_i}s_i^{-1} \in \bigoplus\limits_{i=1}^N A_{\pi(i)}.
\end{align}
Applying this identity twice along with the relation $a=(a)^{\bullet\star}$ gives rise to
\begin{align}
a_i = (s_{\pi(i)}(s_i^{-1})^{\gls{ast}_{i}})a_i (s_i^{-1})^{\gls{ast}_{i}})^{-1} 
\end{align}
for $i=1,\cdots,N$. This means the element $s_{\pi(i)}(s_i^{-1})^{\gls{ast}_{i}}$ is central in $A_i \simeq A_{\pi(i)}$ or in other words $s_{\pi(i)}=\lambda_is_i^{\gls{ast}_{i}}$ with $\lambda_i \in \Cb$ or $\mathcal{Z}(D_i)$.

Let $\omega' \colon A \to A$ be an algebra automorphism. By showing the maps $\gls{Ct}$ and $\gls{Bb}$ are invariant under the action of $\omega'$ one proves that FHK models are left invariant by the action of any such automorphism. Note that $\gls{C}=\gls{eps} (e_a \cdot e_b \cdot e_c )$ and $B_{ab}=\gls{eps} (e_a \cdot e_b)$ so it is necessary only to show the map $\gls{eps}$ is invariant under the action of $\omega'$. If $\omega'$ is an inner automorphism it is necessarily of the form $\omega'(a)=tat^{-1}$ for some $t \in A$. It is then true that $\gls{eps} \circ \omega' = \gls{eps}$ given the symmetry of the Frobenius form.

One must prove that not only inner automorphisms but also permutations do not affect FHK models. To do so, note one needs only verify $\gls{eps} \circ \omega' =\gls{eps}$ when $\omega'$ is a permutation: $\omega' (a_i)=a_{\pi(i)}$. Recall that each Frobenius form decomposes as $\gls{eps} = \bigoplus_{i=1}^N \varepsilon_i$ so $\gls{eps} \circ \omega' = \bigoplus_{i=1}^N \varepsilon_{\pi(i)}$. Therefore, one must show $\varepsilon_i=\varepsilon_{\pi(i)}$. According to equations \eqref{eq:cxfrob} and \eqref{eq:realfrob} each $\gls{eps}$ depends on $A_i$ only through $n_i$ and $|D_i|$. Since it has already been established that for $A_i \simeq A_{\pi(i)}$ both $n_i=n_{\pi(i)}$ and $|D_i|=|D_{\pi(i)}|$ it is clear $\gls{eps} \circ \omega = \gls{eps}$. 

On the other hand, if $\omega'$ is to be an isomorphism not only of FHK models but also of KM ones then it must preserve involutions. This means that KM models with involutions $\bullet$ and $\#$ will necessarily be equivalent if $\omega' (a^{\bullet})=\omega'(a)^{\#}$. This allows one to regard involutions $\bullet$ and $\#$ as equivalent if they are related through an automorphism $\omega'$ as $\#=\omega' \circ \bullet \circ (\omega')^{-1}$. This freedom allows one to bring the $\bullet$ involution into an equivalent standard form. Choosing such an automorphism of $A$ to act as $\omega'(a_i)=t_ia_it_i^{-1} \in A_{i}$, means $\#$ can be written according to
\begin{align}
a_i^{\#}= (t_{\pi(i)}s_{i}t_i^{\gls{ast}_{i}})a_i^{\gls{ast}_{i}}(t_{\pi(i)}s_{i}t_i^{\gls{ast}_{i}})^{-1}.
\end{align}
Note there are no restrictions on the choice of the $t_i$. If $\pi(i) \neq i$ one can set for example $t_{\pi(i)}=s_i^{-1}$ and $t_{i}=1_i$; this means $\#$ would reduce to the canonical choice of $\gls{ast}$. If $\pi(i)=i$ one immediately has a decomposition $\#= \bigoplus_{i=1}^N\#_i$ which concludes the proof.
\end{proof}

%%%%%%%%%%%%%%%%%%%%%%%%%%%%%%%%%%%%%%%%%%%%%%%%%%%%%
\begin{figure}[t!]
\centering
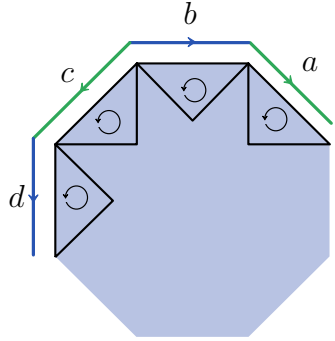 
\caption[Triangulation of a closed non-orientable surface]{\emph{Triangulation of a closed non-orientable surface}. Identified edges are represented with the same colour. For each tuple of states $(a,b,c,d)$ thus identified we obtain a copy of the projective plane.}
\label{fig:plane}
\end{figure} 
%%%%%%%%%%%%%%%%%%%%%%%%%%%%%%%%%%%%%%

As in section \ref{sec:lattice_tft}, we wish to introduce a general formula for computing closed-surface invariants. Let us denote non-orientable surfaces of genus $k$ as $\gls{Sigma_k}$ -- genus for non-orientable surfaces is the number of projective planes that we would have to connect to make the surface. Recall that gluing together two triangles with the same orientation along an edge gives rise to the partition function \eqref{eq:Pach1} of the disk with four boundary edges labelled with states $a,b,c,d$. This partition function can be further identified with $\gls{R}^4\,\gls{eps}(e_a\cdot e_b\cdot e_c\cdot e_d)$. By gluing the boundary edges as in figure \ref{fig:plane}, one makes the projective plane $\Sigma^1$ which results in the partition function
\begin{align}  \label{eq:projective}
\gls{Z}(\Sigma^1)=\gls{R}^{2} \gls{eps} (e_a \cdot e_b \cdot e_c \cdot e_d) S^{ac}S^{bd}=\gls{eps}(\gls{w}).
\end{align}
where $\gls{w}=\gls{R}\,e_a \cdot e_b \cdot e_c \cdot e_d \, S^{ac}S^{bd}$. As highlighted in figure \ref{fig:plane}, a non-orientable surface $\gls{Sigma_k}$ can be constructed from a disk with $4k$ boundary edges
\begin{align}
\gls{Z}(\gls{Sigma_k})=\gls{R}\,\gls{eps}(\gls{w}^k),
\end{align}
where the factor $\gls{R}$ comes from the one vertex shared by all connected $\Sigma^1$ copies. Furthermore, we know that the connected sum of $2k+1$ projective planes is equivalent to the connected sum of $k$ tori and a single projective plane. The algebraic counterpart of this identity in KM models is established through the following lemma.

\begin{lemma} \label{lem:klein_equals_torus}
The identity $\gls{w} \cdot \gls{z}=\gls{w}^3$ holds. Moreover, the element $\gls{w} \in \gls{ZA}$, the centre of $\gls{A}$.
\end{lemma}
\begin{proof}
Note that equation \eqref{eq:snake} is equivalent to the identity $\gls{eps} (y \cdot e_a)e_b \gls{B}=y$ for all $y \in A$. By recalling that $\gls{eps} \circ \gls{ast} = \gls{eps}$ and using the symmetry of $\gls{eps}$ one can further conclude $\gls{eps} (y \cdot e_a^{\gls{ast}})e_b\gls{B}=y^{\gls{ast}}$. By applying this identity to $y\cdot x$ one obtains $\gls{eps} (y \cdot x \cdot e_a^{\gls{ast}})e_b\gls{B}=x ^{\gls{ast}} \cdot y^{\gls{ast}}= x^{\gls{ast}} \cdot \gls{eps} (y \cdot e_a^{\gls{ast}})e_b\gls{B}$. The non-degeneracy of $\gls{eps}$ then guarantees the identity $x \cdot e_a^{\gls{ast}} \otimes e_b\gls{B}= e_a^{\gls{ast}} \otimes x^{\gls{ast}} \cdot e_b \gls{B}$.  

The shorthand notation $e^{a}=\gls{B}e_b$ will be used for the remainder of the proof. A few auxiliary identities that are established elsewhere are needed:
\begin{inparaenum}[1)]
\item for all $x \in A$, $x \cdot e_a \otimes e^a =e_a \otimes e^a \cdot x$ -- see lemma \ref{separable};
\item for all $x \in A$, $p(x)= e_a \cdot x \cdot e^a \in \gls{ZA}$ -- see lemma \ref{lem:projector}; and
\item the element $z$ is central -- see lemma \ref{lem:properties_eta_chi}.
\end{inparaenum}

The expression $\gls{w}=\gls{R}\,e^a \cdot e^b \cdot e_a^{\gls{ast}} \cdot e_b^{\gls{ast}}$ can be simplified. Using the identity $e^b \cdot e_a^{\gls{ast}} \cdot e_b^{\gls{ast}} = e_a \cdot e^b \cdot  e_b^{\gls{ast}}$ and the fact $\gls{R}.\gls{m}(\gls{Bb})=\gls{unit}$ one concludes $\gls{w}=e^b \cdot e_b^{\gls{ast}}$. Therefore, $\gls{w}^3 = (e^a \cdot e_a^{\gls{ast}}) \cdot  (e^b \cdot e_b^{\gls{ast}}) \cdot  ( e^c \cdot e_c^{\gls{ast}} )$. Since one can re-write $e_a^{\gls{ast}} \cdot  e^b \cdot e_b^{\gls{ast}} = e^b  \cdot e_a \cdot e_b^{\gls{ast}}$ one finds $\gls{w}^3=(e^a \cdot e^b \cdot e_a)\cdot (e^c \cdot e_b \cdot e_c^{\gls{ast}})$. By observing $e^a \cdot e^b \cdot e_a$ is a central element, one further concludes that $\gls{w}^3=e^c \cdot (e^a \cdot e^b \cdot e_a \cdot e_b) \cdot e_c^{\gls{ast}}$. Recognising $\gls{z}=e^a \cdot e^b \cdot e_a \cdot e_b$ and knowing it is a central element one obtains $\gls{w}^3=(e^c \cdot e_c^{\gls{ast}}) \cdot \gls{z} = \gls{w} \cdot \gls{z}$. 

Finally, showing that $\gls{w} \in \gls{ZA}$ is straightforward: $y \cdot \gls{w} = y \cdot e^a \cdot e_a^{\gls{ast}} = e^a \cdot y^{\gls{ast}} \cdot e_a^{\gls{ast}} = e^a \cdot e_a^{\gls{ast}} \cdot y^{\gls{ast}\gls{ast}}=\gls{w} \cdot y$ for any $y \in \gls{A}$.
\end{proof}

Having studied the classification of involutions in semi-simple matrix algebras over $\Rb$ and $\Cb$ we are now able to present concrete expressions for $\gls{Z}(\gls{Sigma_k})$. As in section~\S\ref{sec:lattice_tft}, the most general form of the invariant is based on calculations performed for the case of simple algebras. 

For $\gls{A}=\Mb_n(\Cb)$, we know that $\gls{Bb}= (\gls{R}\,n)^{-1}\sum_{lm}e_{lm} \otimes e_{ml}$ if we choose as a basis $\lbrace e_{lm}\rbrace$ the elementary matrices. We also know that an involution $\gls{ast}$ acts as $a^{\gls{ast}}=sa^{\tr}s^{-1}$ where $s$ is either symmetric or antisymmetric. This means that $\gls{w}=(\gls{R}\,n)^{-1}\sum_{lm}se_{lm}^{\tr}s^{-1}e_{ml}$. Since $\tr$ is an involution in its own right we know the identity $\sum_{lm}e_{lm}^{\tr}s^{-1}e_{ml}=(s^{-1})^{\tr}\sum_{lm}e_{lm}^{\tr}e_{ml}$ holds. Furthermore, $\sum_{lm}e_{ml}e_{ml}=\gls{unit}$ and $s(s^{-1})^{\tr}=\pm \gls{unit}$ which means $\gls{w}=\gamma(\gls{R}\,n)^{-1}\gls{unit}$ where $\gamma=\pm \gls{field_unit}$ is determined by the symmetry of $\gls{ast}$. The partition function therefore reads
\begin{align} \label{eq:non_complex}
\gls{Z}\left(\gls{Sigma_k},\Mb_{n}(\Cb),\gls{ast} \right)=\gamma^{2-k} \gls{R}^{2-k}n^{2-k}.
\end{align}
The same conclusion would hold for $\Mb_{n}(\Rb)$ with $\gls{R} \in \Rb$. Note that independently of the choice of $\gls{ast}$ this model cannot distinguish between orientable surfaces of genus $g$ and unorientable surfaces of genus $2g$; in other words it cannot distinguish a torus from a Klein bottle. This is a general feature of models where the element $\gls{w}$ is invertible, since in this case lemma \ref{lem:klein_equals_torus} implies $\gls{z}=\gls{w}^2$. 

The case $\gls{A}=\Mb_{n}(\Cb_{\Rb})$ is our most interesting simple example. As one might expect, if we choose $\gls{ast}=\tr$ as our canonical involution then the partition functions created do not differ from \eqref{eq:non_complex} apart from a factor of $2$. However, if we use $\gls{ast}=\dagger$, the result is very different. Recall that for $\Mb_{n}(\Cb_{\Rb})$ we have $\gls{Bb}=(2\gls{R}\,n)^{-1}\sum_{\omega,lm}\omega e_{lm} \otimes_{\Rb} \overline{\omega} e_{ml}$ with $\omega=1,\hat{\imath}$ as in lemma \ref{lem:Frobenius_forms}. This means $\gls{w}=(2\gls{R}\,n)^{-1}\left(\sum_{\omega} \omega^2\right) \left(\sum_{lm}se_{lm}s^{-1}e_{lm}\right)$. The identity $\gls{w}=0$ follows from the fact $\sum_{\omega} \omega^2 = 0$ independently of the choice of $s$. We note this indistinguishability was to be expected since there is an automorphism relating all hermitian involutions to $s=\text{diag}(1,\cdots,1,-1,\cdots,-1)$ (for more details see appendix~\ref{app:TBA}). In conclusion,
\begin{align}
\gls{Z}\left(\gls{Sigma_k},\Mb_{n}(\Cb_{\Rb}),\gls{ast} \right)=2\gamma^{2-k} \gls{R}^{2-k}n^{2-k}
\end{align}
where $\gamma=-1,0,+1$ according to the choice of $\ast$. This means not all $w$ are invertible and that is therefore possible to distinguish between tori and Klein bottles. 

Finally, we study the case $\gls{A}=\Mb_n(\Hb_{\Rb})$. Similarities with the $\Mb_{n}(\Cb_{\Rb})$ case allow us to conclude $\gls{w}=(4\gls{R}\,n)^{-1}\left(\sum_{\omega} \omega^2\right) \left(\sum_{lm}se_{lm}s^{-1}e_{lm}\right)$. However, in this case $\omega=1,\hat{\imath},\hat{\jmath},\hat{k}$ which means $\sum_{\omega}\omega^2=-2$. On the other hand, $\sum_{lm}se_{lm}s^{-1}e_{lm}=s(s^{-1})^{\ddagger}$. This is enough to conclude the partition function reads
\begin{align}
\gls{Z}\left(\gls{Sigma_k},\Mb_{n}(\Hb_{\Rb}),\gls{ast} \right)= 2^{2-k}\gamma^{2-k}\gls{R}^{2-k}n^{2-k}
\end{align}
where $\gamma=\pm 1$ is again determined by the symmetry of $\gls{ast}$: $s=-\gamma s^{\ddagger}$. 

We are now in a position to state what is the most general form of the partition function that a KM model can associate with a non-orientable surface.

\begin{theorem}
Let $\gls{A}$ and $\gls{ast}$ define a KM model. Then we can write without loss of generality $\gls{ast}=\oplus_{i=1}^N \gls{ast}_i$. The topological invariant $\gls{Z}(\gls{Sigma_k})$ constructed from $\gls{A}$, $\gls{ast}$ and a non-orientable surface $\gls{Sigma_k}$ is
\begin{align}
\gls{Z}(\gls{Sigma_k})=\gls{R}^{2-k}\sum_i \gamma(i)n_i^{2-k}
\end{align}
if $\gls{A}$ is an algebra over $\Cb$. If $\gls{A}$ is instead real, the partition function reads 
\begin{align}
\gls{Z}(\gls{Sigma_k})=\gls{R}^{2-k}\sum_i f(i,k)\gamma(i)^{2-k}n_i^{2-k}, \hspace{2mm}
f(i,k)=
\left\lbrace
\begin{array}{ll}
1 & (\text{if }D=\Rb)\\
2 & (\text{if }D=\Cb_{\Rb})\\
2^{2-k} & (\text{if }D=\Hb_{\Rb})
\end{array}
\right. .
\end{align}
The function $\gamma$ is defined as follows.
\begin{align*}
\gamma(i)=
\left\lbrace
\begin{array}{ll}
+1 & (\gls{ast}_i \text{ is symmetric }) \vee (\gls{ast}_i \text{ is anti-hermitian and }D=\Hb_{\Rb})\\
\,0 & (\gls{ast}_i \text{ is hermitian} \wedge D=\Cb_{\Rb})\\
-1 & (\gls{ast}_i \text{ is anti-symmetric }) \vee (\gls{ast}_i \text{ is hermitian and }D=\Hb_{\Rb})
\end{array}
\right. 
\end{align*}
\end{theorem}
Let us recover the group algebra example of section \S \ref{sec:lattice_tft}. For $\Cb \gls{H}$ with involution $-1$, defined as the linear extension of the inverse operation in group elements, we obtain
\begin{align}
\gls{Z}(\gls{Sigma_k})=
\gls{R}^{2-k}\sum_{i \in I}\gamma(i)^{2-k}(\dim i)^{2-k}.
\end{align} 
A Lie group version of the equation above is proposed by Witten \cite{Witten} when dealing with the zero area limit of Yang-Mills theory in two dimensions, for non-orientable surfaces.

%%%%%%%%%%%%%%%%%%%%%%%%%%%%%%%%%%%%%%%%%%%%%%%%%%%%%%%%%%%%%%%%  
%%%%%%%%%%%%%%%%%%%%%%%%%%%%%%%%%%%%%%%%%%%%%%%%%%%%%%%%%%%%%%%%
\chapter{Planar and spherical state sum models} \label{sec:diagram}
%%%%%%%%%%%%%%%%%%%%%%%%%%%%%%%%%%%%%%%%%%%%%%%%%%%%%%%%%%%%%%%%
%%%%%%%%%%%%%%%%%%%%%%%%%%%%%%%%%%%%%%%%%%%%%%%%%%%%%%%%%%%%%%%%

%%%%%%%%%%%%%%%%%%%%%%%%%%%%%%%%%%%%%%%%%%%%%%%%%%%%%%%%%%%%%%%%
%%%%%%%%%%%%%%%%%%%%%%%%%%%%%%%%%%%%%%%%%%%%%%%%%%%%%%%%%%%%%%%%
\section{Planar models} \label{sec:planar}
%%%%%%%%%%%%%%%%%%%%%%%%%%%%%%%%%%%%%%%%%%%%%%%%%%%%%%%%%%%%%%%%
%%%%%%%%%%%%%%%%%%%%%%%%%%%%%%%%%%%%%%%%%%%%%%%%%%%%%%%%%%%%%%%%

The definition of a naive state sum model for a triangulated surface $\gls{Sigma}$, a two-dimensional orientable manifold, relied on two key components: the assignment of states $a,b,c \in \gls{S}$ to edges of the triangulation $\gls{T}$ and the subsequent allocation of amplitudes to vertices, edges and triangles of $\gls{T}$, that depend on such states. 

Let $v$ denote an internal vertex, $(e,e')$ a pair of identified edges labeled with states $a,b$, and $t$ a triangle associated with states $a,b,c$. The naive state sum model introduced in section \S\ref{sec:lattice_tft} defined amplitudes as $A(v)=\gls{R}$, $A(e,e')=\gls{B}$ and $A(t)=\gls{C}$. However, to make such definitions meaningful extra conditions on $\gls{Ct}$ and $\gls{Bb}$ had to be introduced -- the relation between labeled internal edges and triangles, and such maps had to be made unequivocal. Such conditions were the symmetry requirements encoded in equations \eqref{eq:C_cycle} and \eqref{eq:metric}.

In this chapter we will relax these extra conditions. We will understand how a more general algebraic framework can be used for state sum models if a more sophisticated method to define the model is employed. This new framework uses a diagrammatic calculus to determine the combinatorics of the partition function. 

Within the new framework, we will require not only a triangulation $\gls{T}$ of a surface but also the dual graph that comes associated with it. Recall that the dual of a triangulation can be constructed through barycentric division: a unique dual vertex, a node, is placed in every triangle and dual edges, legs, connect all such nodes by intersecting edges on the triangulation. The dual of a boundary edge is seen as terminating in empty space. Any region of the dual graph that is bounded by legs is called a face.

The dual graph in itself possesses no more information than the original triangulation. However, we introduce to the former a crucial new feature: legs must point up or downwards and these directions are regarded as distinct.  For example, the triangles of figure~\ref{fig:triangle_amplitude_diag} are distinguished. Therefore, a triangulation can be reconstructed uniquely from this trivalent graph $\gls{G}$ but not the converse.    

%%%%%%%%%%%%%%%%%%%%%%%%%%%%%%%%%%%%%%%%%%%%%%%%%%%%%
\begin{figure}[t!]
                \centering
		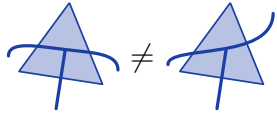 
		\vspace{3mm}
		\caption[Triangle amplitude]{\emph{Triangle amplitude.} The diagram is determined by the dual graph with a choice of legs pointing upwards or downwards.}
		\label{fig:triangle_amplitude_diag}
\end{figure}
%%%%%%%%%%%%%%%%%%%%%%%%%%%%%%%%%%%%%%%%%%%%%%%%%%%%%

Consider the surface $M \subset \Rb^2$ that comes with a trivalent graph $\gls{G}$ as described above. Note that the plane $\Rb^2$ is considered to have a standard orientation, so that $M$ is an oriented manifold. (Without loss of generality we work with an anti-clockwise orientation.) A diagrammatic state sum model on $M$ has a set of amplitudes for each node, leg and face. Each leg on the graph will be associated with a state $a,b,c \in \gls{S}$. The oriented triangle $t$ with edge states $a,b,c$ depicted on the right of figure~\ref{fig:triangle_amplitude_diag} is associated with the amplitude $A(t)=\gls{C} \in \gls{k}$. Two legs $(l,l')$ labelled by $a,b$ are joined together through a matrix $A(l,l')=\gls{B} \in k$. Their graphical representation is depicted below, where the order of the indices is implicitly determined by the orientation of $M$.
\begin{align}
\begin{aligned}
%% Creator: Inkscape 0.48.2, www.inkscape.org
%% PDF/EPS/PS + LaTeX output extension by Johan Engelen, 2010
%% Accompanies image file 'diagBC.pdf' (pdf, eps, ps)
%%
%% To include the image in your LaTeX document, write
%%   \input{<filename>.pdf_tex}
%%  instead of
%%   \includegraphics{<filename>.pdf}
%% To scale the image, write
%%   \def\svgwidth{<desired width>}
%%   \input{<filename>.pdf_tex}
%%  instead of
%%   \includegraphics[width=<desired width>]{<filename>.pdf}
%%
%% Images with a different path to the parent latex file can
%% be accessed with the `import' package (which may need to be
%% installed) using
%%   \usepackage{import}
%% in the preamble, and then including the image with
%%   \import{<path to file>}{<filename>.pdf_tex}
%% Alternatively, one can specify
%%   \graphicspath{{<path to file>/}}
%% 
%% For more information, please see info/svg-inkscape on CTAN:
%%   http://tug.ctan.org/tex-archive/info/svg-inkscape
%%
\begingroup%
  \makeatletter%
  \providecommand\color[2][]{%
    \errmessage{(Inkscape) Color is used for the text in Inkscape, but the package 'color.sty' is not loaded}%
    \renewcommand\color[2][]{}%
  }%
  \providecommand\transparent[1]{%
    \errmessage{(Inkscape) Transparency is used (non-zero) for the text in Inkscape, but the package 'transparent.sty' is not loaded}%
    \renewcommand\transparent[1]{}%
  }%
  \providecommand\rotatebox[2]{#2}%
  \ifx\svgwidth\undefined%
    \setlength{\unitlength}{253.67686768bp}%
    \ifx\svgscale\undefined%
      \relax%
    \else%
      \setlength{\unitlength}{\unitlength * \real{\svgscale}}%
    \fi%
  \else%
    \setlength{\unitlength}{\svgwidth}%
  \fi%
  \global\let\svgwidth\undefined%
  \global\let\svgscale\undefined%
  \makeatother%
  \begin{picture}(1,0.16438468)%
    \put(0,0){\includegraphics[width=\unitlength]{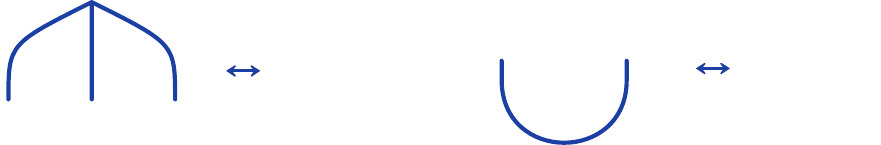}}%
    \put(-0.00112409,0.00774565){\color[rgb]{0,0,0}\makebox(0,0)[lb]{\smash{$a$}}}%
    \put(0.09471793,0.00751313){\color[rgb]{0,0,0}\makebox(0,0)[lb]{\smash{$b$}}}%
    \put(0.1889982,0.00826189){\color[rgb]{0,0,0}\makebox(0,0)[lb]{\smash{$c$}}}%
    \put(0.33045484,0.07295678){\color[rgb]{0,0,0}\makebox(0,0)[lb]{\smash{$C_{abc},$}}}%
    \put(0.5606451,0.11658956){\color[rgb]{0,0,0}\makebox(0,0)[lb]{\smash{$a$}}}%
    \put(0.69907035,0.11853901){\color[rgb]{0,0,0}\makebox(0,0)[lb]{\smash{$b$}}}%
    \put(0.86561706,0.07644772){\color[rgb]{0,0,0}\makebox(0,0)[lb]{\smash{$B^{ab}.$}}}%
  \end{picture}%
\endgroup%

\end{aligned}
\end{align}
The algebraic data are again the maps $\gls{Ct}$ and $\gls{Bb}$ as defined in equations \eqref{eq:C_trilinear} and \eqref{eq:B_element}, and a constant $A(f)=\gls{R} \in k$ now to be associated with each face $f$. The new consistency condition is a replacement of the symmetry requirements \eqref{eq:C_cycle} and \eqref{eq:metric} on $\gls{C}$ and $\gls{B}$, with the one equation
\begin{align}\label{eq:BCequation}
C_{abc}\,B^{cd}=B^{de}C_{eab},
\end{align}
which does not imply $\gls{Bb}$ must be symmetric.

The data $C_{abc}$ and $B^{ab}$ together with the face amplitude $\gls{R}\in \gls{k}$ determine an evaluation of the diagram $\gls{G}$, that we shall denote $|G| \in k$. This means that although the partition function $\gls{Z}(M)$ could be defined through an equation similar to \eqref{eq:Z_surf} it is more natural to simply set
\begin{align} \label{eq:diag-partition}
Z(M)=|G|.
\end{align}
A simple example for the $M$ consisting of two triangles is shown in figure \ref{fig:diag-ssm-glue}.

Either side of \eqref{eq:BCequation} can be taken as the definition of $C_{ab}{}^{d}$, the components of a multiplication map $\gls{m}$ as in equation \eqref{eq:multiplicationmap}. The diagrammatic counterpart is below. Similar expressions are used to define a vertex with two or three legs pointing upwards:
\begin{align}
\begin{aligned}
\hspace{15mm}
%% Creator: Inkscape 0.48.2, www.inkscape.org
%% PDF/EPS/PS + LaTeX output extension by Johan Engelen, 2010
%% Accompanies image file 'diag_mult.pdf' (pdf, eps, ps)
%%
%% To include the image in your LaTeX document, write
%%   \input{<filename>.pdf_tex}
%%  instead of
%%   \includegraphics{<filename>.pdf}
%% To scale the image, write
%%   \def\svgwidth{<desired width>}
%%   \input{<filename>.pdf_tex}
%%  instead of
%%   \includegraphics[width=<desired width>]{<filename>.pdf}
%%
%% Images with a different path to the parent latex file can
%% be accessed with the `import' package (which may need to be
%% installed) using
%%   \usepackage{import}
%% in the preamble, and then including the image with
%%   \import{<path to file>}{<filename>.pdf_tex}
%% Alternatively, one can specify
%%   \graphicspath{{<path to file>/}}
%% 
%% For more information, please see info/svg-inkscape on CTAN:
%%   http://tug.ctan.org/tex-archive/info/svg-inkscape
%%
\begingroup%
  \makeatletter%
  \providecommand\color[2][]{%
    \errmessage{(Inkscape) Color is used for the text in Inkscape, but the package 'color.sty' is not loaded}%
    \renewcommand\color[2][]{}%
  }%
  \providecommand\transparent[1]{%
    \errmessage{(Inkscape) Transparency is used (non-zero) for the text in Inkscape, but the package 'transparent.sty' is not loaded}%
    \renewcommand\transparent[1]{}%
  }%
  \providecommand\rotatebox[2]{#2}%
  \ifx\svgwidth\undefined%
    \setlength{\unitlength}{275.20251465bp}%
    \ifx\svgscale\undefined%
      \relax%
    \else%
      \setlength{\unitlength}{\unitlength * \real{\svgscale}}%
    \fi%
  \else%
    \setlength{\unitlength}{\svgwidth}%
  \fi%
  \global\let\svgwidth\undefined%
  \global\let\svgscale\undefined%
  \makeatother%
  \begin{picture}(1,0.20128155)%
    \put(0,0){\includegraphics[width=\unitlength]{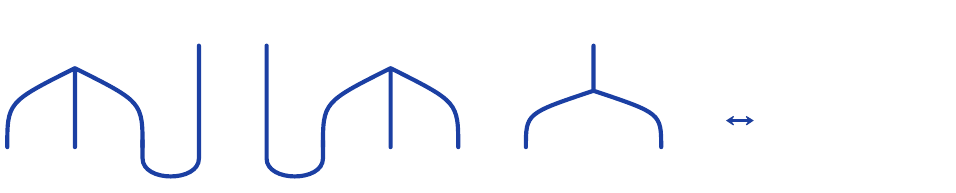}}%
    \put(-0.00103616,0.00299496){\color[rgb]{0,0,0}\makebox(0,0)[lb]{\smash{$a$}}}%
    \put(0.06965996,0.0035841){\color[rgb]{0,0,0}\makebox(0,0)[lb]{\smash{$b$}}}%
    \put(0.20044757,0.17855681){\color[rgb]{0,0,0}\makebox(0,0)[lb]{\smash{$d$}}}%
    \put(0.27114364,0.17796767){\color[rgb]{0,0,0}\makebox(0,0)[lb]{\smash{$d$}}}%
    \put(0.4019313,0.00299496){\color[rgb]{0,0,0}\makebox(0,0)[lb]{\smash{$a$}}}%
    \put(0.47203823,0.0035841){\color[rgb]{0,0,0}\makebox(0,0)[lb]{\smash{$b$}}}%
    \put(0.54332344,0.0035841){\color[rgb]{0,0,0}\makebox(0,0)[lb]{\smash{$a$}}}%
    \put(0.68294807,0.00417325){\color[rgb]{0,0,0}\makebox(0,0)[lb]{\smash{$b$}}}%
    \put(0.612252,0.17796776){\color[rgb]{0,0,0}\makebox(0,0)[lb]{\smash{$d$}}}%
    \put(0.83082062,0.06779968){\color[rgb]{0,0,0}\makebox(0,0)[lb]{\smash{$C_{ab}{}^d$}}}%
    \put(0.2302863,0.07204744){\color[rgb]{0,0,0}\makebox(0,0)[lb]{\smash{$=$}}}%
    \put(0.50190606,0.07170235){\color[rgb]{0,0,0}\makebox(0,0)[lb]{\smash{$=$}}}%
  \end{picture}%
\endgroup%

\label{eq:point_up}
\end{aligned}.
\end{align}

%%%%%%%%%%%%%%%%%%%%%%%%%%%%%%%%%%%%%%%%%%%%%%%%%%%%%
\begin{figure}[t!]
                \centering
		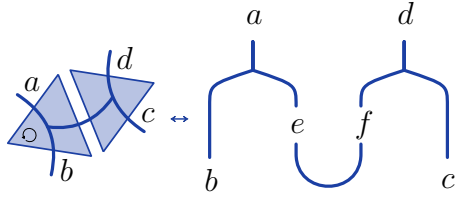
		\vspace{2.4mm}
		\caption[Gluing triangles]{\emph{Gluing triangles.} The fact $B$ is not assumed to be symmetric is translated into a lack of rotational symmetry.}
		\label{fig:diag-ssm-glue}
\end{figure} 
%%%%%%%%%%%%%%%%%%%%%%%%%%%%%%%%%%%%%%%%%%%%%%%%%%%%%%

As in section \S\ref{sec:lattice_tft}, we will restrict our attention to non-degenerate data: $R\neq 0$, $B(\cdot,y)=0 \Rightarrow y=0$ and $C(\cdot,\cdot,y)=0 \Rightarrow y=0$. Under these conditions $B^{ab}$ has an inverse $B_{ab}$ defined by \eqref{eq:snake} -- its graphical representation can be found below. 
\begin{align}
\begin{aligned}
\hspace{20mm}
%% Creator: Inkscape 0.48.2, www.inkscape.org
%% PDF/EPS/PS + LaTeX output extension by Johan Engelen, 2010
%% Accompanies image file 'diagB_1.pdf' (pdf, eps, ps)
%%
%% To include the image in your LaTeX document, write
%%   \input{<filename>.pdf_tex}
%%  instead of
%%   \includegraphics{<filename>.pdf}
%% To scale the image, write
%%   \def\svgwidth{<desired width>}
%%   \input{<filename>.pdf_tex}
%%  instead of
%%   \includegraphics[width=<desired width>]{<filename>.pdf}
%%
%% Images with a different path to the parent latex file can
%% be accessed with the `import' package (which may need to be
%% installed) using
%%   \usepackage{import}
%% in the preamble, and then including the image with
%%   \import{<path to file>}{<filename>.pdf_tex}
%% Alternatively, one can specify
%%   \graphicspath{{<path to file>/}}
%% 
%% For more information, please see info/svg-inkscape on CTAN:
%%   http://tug.ctan.org/tex-archive/info/svg-inkscape
%%
\begingroup%
  \makeatletter%
  \providecommand\color[2][]{%
    \errmessage{(Inkscape) Color is used for the text in Inkscape, but the package 'color.sty' is not loaded}%
    \renewcommand\color[2][]{}%
  }%
  \providecommand\transparent[1]{%
    \errmessage{(Inkscape) Transparency is used (non-zero) for the text in Inkscape, but the package 'transparent.sty' is not loaded}%
    \renewcommand\transparent[1]{}%
  }%
  \providecommand\rotatebox[2]{#2}%
  \ifx\svgwidth\undefined%
    \setlength{\unitlength}{109.32792969bp}%
    \ifx\svgscale\undefined%
      \relax%
    \else%
      \setlength{\unitlength}{\unitlength * \real{\svgscale}}%
    \fi%
  \else%
    \setlength{\unitlength}{\svgwidth}%
  \fi%
  \global\let\svgwidth\undefined%
  \global\let\svgscale\undefined%
  \makeatother%
  \begin{picture}(1,0.33245624)%
    \put(0,0){\includegraphics[width=\unitlength]{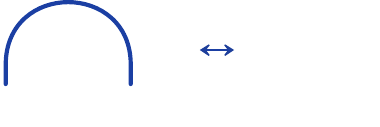}}%
    \put(-0.00260827,0.00753896){\color[rgb]{0,0,0}\makebox(0,0)[lb]{\smash{$a$}}}%
    \put(0.32399568,0.00821179){\color[rgb]{0,0,0}\makebox(0,0)[lb]{\smash{$b$}}}%
    \put(0.70205035,0.17689689){\color[rgb]{0,0,0}\makebox(0,0)[lb]{\smash{$B_{ab}$}}}%
  \end{picture}%
\endgroup%

\end{aligned}
\end{align}
The defining relation \eqref{eq:snake} can then be seen as the following snake identity.
\begin{align}
\begin{aligned}
\hspace{15mm}
%% Creator: Inkscape 0.48.2, www.inkscape.org
%% PDF/EPS/PS + LaTeX output extension by Johan Engelen, 2010
%% Accompanies image file 'inverse_metric.pdf' (pdf, eps, ps)
%%
%% To include the image in your LaTeX document, write
%%   \input{<filename>.pdf_tex}
%%  instead of
%%   \includegraphics{<filename>.pdf}
%% To scale the image, write
%%   \def\svgwidth{<desired width>}
%%   \input{<filename>.pdf_tex}
%%  instead of
%%   \includegraphics[width=<desired width>]{<filename>.pdf}
%%
%% Images with a different path to the parent latex file can
%% be accessed with the `import' package (which may need to be
%% installed) using
%%   \usepackage{import}
%% in the preamble, and then including the image with
%%   \import{<path to file>}{<filename>.pdf_tex}
%% Alternatively, one can specify
%%   \graphicspath{{<path to file>/}}
%% 
%% For more information, please see info/svg-inkscape on CTAN:
%%   http://tug.ctan.org/tex-archive/info/svg-inkscape
%%
\begingroup%
  \makeatletter%
  \providecommand\color[2][]{%
    \errmessage{(Inkscape) Color is used for the text in Inkscape, but the package 'color.sty' is not loaded}%
    \renewcommand\color[2][]{}%
  }%
  \providecommand\transparent[1]{%
    \errmessage{(Inkscape) Transparency is used (non-zero) for the text in Inkscape, but the package 'transparent.sty' is not loaded}%
    \renewcommand\transparent[1]{}%
  }%
  \providecommand\rotatebox[2]{#2}%
  \ifx\svgwidth\undefined%
    \setlength{\unitlength}{204.21323242bp}%
    \ifx\svgscale\undefined%
      \relax%
    \else%
      \setlength{\unitlength}{\unitlength * \real{\svgscale}}%
    \fi%
  \else%
    \setlength{\unitlength}{\svgwidth}%
  \fi%
  \global\let\svgwidth\undefined%
  \global\let\svgscale\undefined%
  \makeatother%
  \begin{picture}(1,0.26728199)%
    \put(0,0){\includegraphics[width=\unitlength]{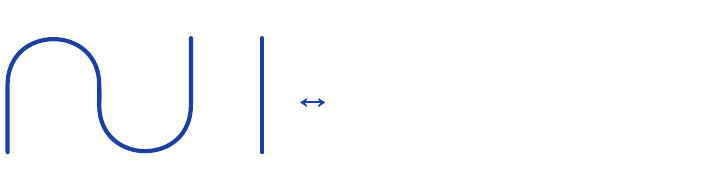}}%
    \put(-0.00139637,0.00403609){\color[rgb]{0,0,0}\makebox(0,0)[lb]{\smash{$a$}}}%
    \put(0.25583694,0.23586379){\color[rgb]{0,0,0}\makebox(0,0)[lb]{\smash{$b$}}}%
    \put(0.35507833,0.23665762){\color[rgb]{0,0,0}\makebox(0,0)[lb]{\smash{$b$}}}%
    \put(0.35746016,0.00483003){\color[rgb]{0,0,0}\makebox(0,0)[lb]{\smash{$a$}}}%
    \put(0.50191845,0.10878811){\color[rgb]{0,0,0}\makebox(0,0)[lb]{\smash{$B_{ac}B^{cb}=\delta_{a}^b$}}}%
    \put(0.30267889,0.10962891){\color[rgb]{0,0,0}\makebox(0,0)[lb]{\smash{$=$}}}%
  \end{picture}%
\endgroup%

\label{fig:inverse_metric}
\end{aligned}
\end{align}
Using this inverse, equation \eqref{eq:BCequation} can equivalently be written as either of the two equations
\begin{align} \label{eq:gen_cycle}
C_{eab}\,B_{dc}B^{de}=\gls{C}=C_{bce}\,B_{ad}B^{ed},
\end{align}
which have the following graphical counterpart.
\begin{align}
\begin{aligned}
\hspace{1mm}
%% Creator: Inkscape 0.48.2, www.inkscape.org
%% PDF/EPS/PS + LaTeX output extension by Johan Engelen, 2010
%% Accompanies image file 'diag_cyclicity.pdf' (pdf, eps, ps)
%%
%% To include the image in your LaTeX document, write
%%   \input{<filename>.pdf_tex}
%%  instead of
%%   \includegraphics{<filename>.pdf}
%% To scale the image, write
%%   \def\svgwidth{<desired width>}
%%   \input{<filename>.pdf_tex}
%%  instead of
%%   \includegraphics[width=<desired width>]{<filename>.pdf}
%%
%% Images with a different path to the parent latex file can
%% be accessed with the `import' package (which may need to be
%% installed) using
%%   \usepackage{import}
%% in the preamble, and then including the image with
%%   \import{<path to file>}{<filename>.pdf_tex}
%% Alternatively, one can specify
%%   \graphicspath{{<path to file>/}}
%% 
%% For more information, please see info/svg-inkscape on CTAN:
%%   http://tug.ctan.org/tex-archive/info/svg-inkscape
%%
\begingroup%
  \makeatletter%
  \providecommand\color[2][]{%
    \errmessage{(Inkscape) Color is used for the text in Inkscape, but the package 'color.sty' is not loaded}%
    \renewcommand\color[2][]{}%
  }%
  \providecommand\transparent[1]{%
    \errmessage{(Inkscape) Transparency is used (non-zero) for the text in Inkscape, but the package 'transparent.sty' is not loaded}%
    \renewcommand\transparent[1]{}%
  }%
  \providecommand\rotatebox[2]{#2}%
  \ifx\svgwidth\undefined%
    \setlength{\unitlength}{211.79963379bp}%
    \ifx\svgscale\undefined%
      \relax%
    \else%
      \setlength{\unitlength}{\unitlength * \real{\svgscale}}%
    \fi%
  \else%
    \setlength{\unitlength}{\svgwidth}%
  \fi%
  \global\let\svgwidth\undefined%
  \global\let\svgscale\undefined%
  \makeatother%
  \begin{picture}(1,0.19254725)%
    \put(0,0){\includegraphics[width=\unitlength]{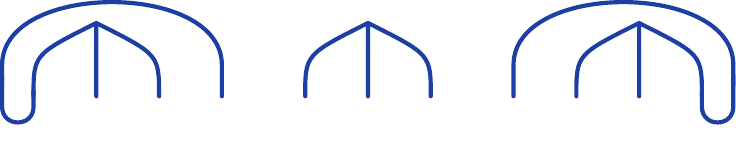}}%
    \put(0.4036248,0.00423107){\color[rgb]{0,0,0}\makebox(0,0)[lb]{\smash{$a$}}}%
    \put(0.48787939,0.00389149){\color[rgb]{0,0,0}\makebox(0,0)[lb]{\smash{$b$}}}%
    \put(0.57358835,0.00402278){\color[rgb]{0,0,0}\makebox(0,0)[lb]{\smash{$c$}}}%
    \put(0.68360709,0.00428548){\color[rgb]{0,0,0}\makebox(0,0)[lb]{\smash{$a$}}}%
    \put(0.77211525,0.0042086){\color[rgb]{0,0,0}\makebox(0,0)[lb]{\smash{$b$}}}%
    \put(0.85756127,0.00447118){\color[rgb]{0,0,0}\makebox(0,0)[lb]{\smash{$c$}}}%
    \put(0.61898256,0.08428064){\color[rgb]{0,0,0}\makebox(0,0)[lb]{\smash{$=$}}}%
    \put(0.12039813,0.00496038){\color[rgb]{0,0,0}\makebox(0,0)[lb]{\smash{$a$}}}%
    \put(0.20465273,0.0046208){\color[rgb]{0,0,0}\makebox(0,0)[lb]{\smash{$b$}}}%
    \put(0.29036168,0.00475209){\color[rgb]{0,0,0}\makebox(0,0)[lb]{\smash{$c$}}}%
    \put(0.33569595,0.08428064){\color[rgb]{0,0,0}\makebox(0,0)[lb]{\smash{$=$}}}%
  \end{picture}%
\endgroup%

\label{fig:diagram_cycle}
\end{aligned}
\end{align}
If $\gls{Bb}$ is symmetric, condition \eqref{eq:gen_cycle} reduces to cyclicity as presented in \eqref{eq:C_cycle}. 

Now the Pachner moves are introduced. A Pachner move preserves the boundary of a triangulation and it is assumed that the corresponding dual edges do not change in a neighbourhood of the boundary, so remain either upward or downward-pointing.
\begin{definition} 
A planar state sum model is a non-degenerate diagrammatic state sum model for any compact $M\subset \Rb^2$ satisfying the Pachner moves.
\end{definition}

The planar state sum models depend on the details of the diagram in the neighbourhood of the boundary. Nonetheless, due to the identities for $\gls{Ct}$ and $\gls{Bb}$, the interior of the graph can be moved by a triangulation-preserving homeomorphism (fixing the boundary) to any convenient graph in order to construct the required algebraic expression. The partition function of a disk is no longer symmetric under cyclic permutations of the boundary edges, but has a more refined mapping property that generalises \eqref{eq:gen_cycle}. Note that this is why the diagrammatic state sum models escape the conclusion of \S\ref{sec:lattice_tft} that $\gls{Bb}$ is symmetric for the naive models. These mappings of boundaries and the boundary data are not studied further in this thesis. It will be assumed that any mapping of surfaces is the identity mapping in a neighbourhood of the boundary. 

The result below is a refinement of theorem~\ref{theo:state_sum_algebra} and its proof develops the properties of the graphical calculus.
\begin{theorem} \label{theo:diagram}
Non-degenerate diagrammatic state sum model data determine a planar state sum if and only if the multiplication map $\gls{m}$, the bilinear form $B$ and the distinguished element $m(B)$ determine on $\gls{A}$ the structure of a special Frobenius algebra with identity element 
\begin{align}
\begin{aligned}
%% Creator: Inkscape 0.48.2, www.inkscape.org
%% PDF/EPS/PS + LaTeX output extension by Johan Engelen, 2010
%% Accompanies image file 'diag_1.pdf' (pdf, eps, ps)
%%
%% To include the image in your LaTeX document, write
%%   \input{<filename>.pdf_tex}
%%  instead of
%%   \includegraphics{<filename>.pdf}
%% To scale the image, write
%%   \def\svgwidth{<desired width>}
%%   \input{<filename>.pdf_tex}
%%  instead of
%%   \includegraphics[width=<desired width>]{<filename>.pdf}
%%
%% Images with a different path to the parent latex file can
%% be accessed with the `import' package (which may need to be
%% installed) using
%%   \usepackage{import}
%% in the preamble, and then including the image with
%%   \import{<path to file>}{<filename>.pdf_tex}
%% Alternatively, one can specify
%%   \graphicspath{{<path to file>/}}
%% 
%% For more information, please see info/svg-inkscape on CTAN:
%%   http://tug.ctan.org/tex-archive/info/svg-inkscape
%%
\begingroup%
  \makeatletter%
  \providecommand\color[2][]{%
    \errmessage{(Inkscape) Color is used for the text in Inkscape, but the package 'color.sty' is not loaded}%
    \renewcommand\color[2][]{}%
  }%
  \providecommand\transparent[1]{%
    \errmessage{(Inkscape) Transparency is used (non-zero) for the text in Inkscape, but the package 'transparent.sty' is not loaded}%
    \renewcommand\transparent[1]{}%
  }%
  \providecommand\rotatebox[2]{#2}%
  \ifx\svgwidth\undefined%
    \setlength{\unitlength}{77.2421875bp}%
    \ifx\svgscale\undefined%
      \relax%
    \else%
      \setlength{\unitlength}{\unitlength * \real{\svgscale}}%
    \fi%
  \else%
    \setlength{\unitlength}{\svgwidth}%
  \fi%
  \global\let\svgwidth\undefined%
  \global\let\svgscale\undefined%
  \makeatother%
  \begin{picture}(1,0.64095785)%
    \put(0,0){\includegraphics[width=\unitlength]{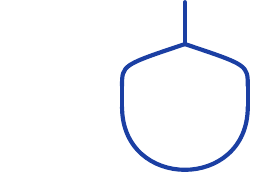}}%
    \put(-0.00369172,0.18591206){\color[rgb]{0,0,0}\makebox(0,0)[lb]{\smash{$1=R$}}}%
    \put(0.98022656,0.18591206){\color[rgb]{0,0,0}\makebox(0,0)[lb]{\smash{.}}}%
  \end{picture}%
\endgroup%

\label{eq:diag_1}
\end{aligned}
\end{align}
\end{theorem}

\begin{proof}
The proof of theorem~\ref{theo:state_sum_algebra} will be followed very closely. The essential difference relies on the translation of Pachner moves into the new diagrammatic model. 

Suppose that $(\gls{Ct},B,\gls{R})$ is the data for a planar state sum model. As before, define $A$ to be the vector space spanned by $S$. Consider the 2-2 move depicted in figure~\ref{fig:Pach1}. Its graphical counterpart is given below.
%%%%%%%%%%%%%%%%%%%%%%%%%%%%%%%%%%%%%%%%%%%%%%%%%%%%%
\begin{align}
\begin{aligned}
%% Creator: Inkscape 0.48.2, www.inkscape.org
%% PDF/EPS/PS + LaTeX output extension by Johan Engelen, 2010
%% Accompanies image file 'diagPachner22.pdf' (pdf, eps, ps)
%%
%% To include the image in your LaTeX document, write
%%   \input{<filename>.pdf_tex}
%%  instead of
%%   \includegraphics{<filename>.pdf}
%% To scale the image, write
%%   \def\svgwidth{<desired width>}
%%   \input{<filename>.pdf_tex}
%%  instead of
%%   \includegraphics[width=<desired width>]{<filename>.pdf}
%%
%% Images with a different path to the parent latex file can
%% be accessed with the `import' package (which may need to be
%% installed) using
%%   \usepackage{import}
%% in the preamble, and then including the image with
%%   \import{<path to file>}{<filename>.pdf_tex}
%% Alternatively, one can specify
%%   \graphicspath{{<path to file>/}}
%% 
%% For more information, please see info/svg-inkscape on CTAN:
%%   http://tug.ctan.org/tex-archive/info/svg-inkscape
%%
\begingroup%
  \makeatletter%
  \providecommand\color[2][]{%
    \errmessage{(Inkscape) Color is used for the text in Inkscape, but the package 'color.sty' is not loaded}%
    \renewcommand\color[2][]{}%
  }%
  \providecommand\transparent[1]{%
    \errmessage{(Inkscape) Transparency is used (non-zero) for the text in Inkscape, but the package 'transparent.sty' is not loaded}%
    \renewcommand\transparent[1]{}%
  }%
  \providecommand\rotatebox[2]{#2}%
  \ifx\svgwidth\undefined%
    \setlength{\unitlength}{344.9278474bp}%
    \ifx\svgscale\undefined%
      \relax%
    \else%
      \setlength{\unitlength}{\unitlength * \real{\svgscale}}%
    \fi%
  \else%
    \setlength{\unitlength}{\svgwidth}%
  \fi%
  \global\let\svgwidth\undefined%
  \global\let\svgscale\undefined%
  \makeatother%
  \begin{picture}(1,0.18036112)%
    \put(0,0){\includegraphics[width=\unitlength]{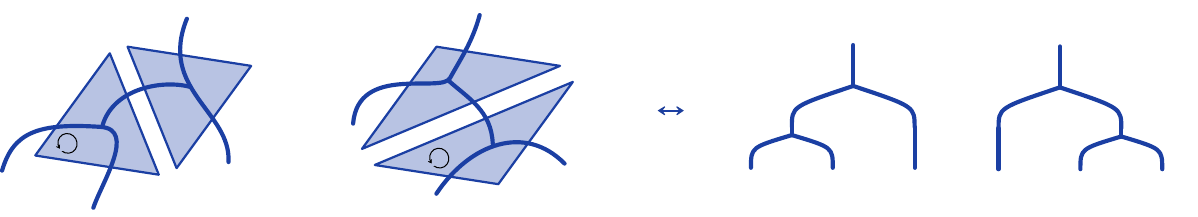}}%
    \put(0.09307965,0.00661993){\color[rgb]{0,0,0}\makebox(0,0)[lb]{\smash{$b$}}}%
    \put(0.04889558,0.11755027){\color[rgb]{0,0,0}\makebox(0,0)[lb]{\smash{$a$}}}%
    \put(0.1649963,0.1434026){\color[rgb]{0,0,0}\makebox(0,0)[lb]{\smash{$d$}}}%
    \put(0.19366894,0.07618639){\color[rgb]{0,0,0}\makebox(0,0)[lb]{\smash{$c$}}}%
    \put(0.37557575,0.00238954){\color[rgb]{0,0,0}\makebox(0,0)[lb]{\smash{$b$}}}%
    \put(0.46347386,0.06913576){\color[rgb]{0,0,0}\makebox(0,0)[lb]{\smash{$c$}}}%
    \put(0.4268105,0.14152239){\color[rgb]{0,0,0}\makebox(0,0)[lb]{\smash{$d$}}}%
    \put(0.32528111,0.1236608){\color[rgb]{0,0,0}\makebox(0,0)[lb]{\smash{$a$}}}%
    \put(0.24417811,0.07286071){\color[rgb]{0,0,0}\makebox(0,0)[lb]{\smash{$=$}}}%
    \put(0.62054953,0.00299521){\color[rgb]{0,0,0}\makebox(0,0)[lb]{\smash{$a$}}}%
    \put(0.68964595,0.00252523){\color[rgb]{0,0,0}\makebox(0,0)[lb]{\smash{$b$}}}%
    \put(0.75733213,0.00299521){\color[rgb]{0,0,0}\makebox(0,0)[lb]{\smash{$c$}}}%
    \put(0.82736865,0.00252523){\color[rgb]{0,0,0}\makebox(0,0)[lb]{\smash{$a$}}}%
    \put(0.89576002,0.00299528){\color[rgb]{0,0,0}\makebox(0,0)[lb]{\smash{$b$}}}%
    \put(0.96368126,0.00323034){\color[rgb]{0,0,0}\makebox(0,0)[lb]{\smash{$c$}}}%
    \put(0.70563718,0.16223006){\color[rgb]{0,0,0}\makebox(0,0)[lb]{\smash{$d$}}}%
    \put(0.87817531,0.16167472){\color[rgb]{0,0,0}\makebox(0,0)[lb]{\smash{$d$}}}%
    \put(0.78519271,0.05830724){\color[rgb]{0,0,0}\makebox(0,0)[lb]{\smash{$=$}}}%
  \end{picture}%
\endgroup%

\label{fig:diagram_mult2}
\end{aligned}
\end{align}
%%%%%%%%%%%%%%%%%%%%%%%%%%%%%%%%%%%%%%%%%%%%%%%%%%%%%%
The multiplication map is therefore associative, as in theorem~\ref{theo:state_sum_algebra}. 

Next, the definition of a multiplication through \eqref{eq:point_up} together with equation \eqref{fig:inverse_metric} implies
\begin{align} 
\begin{aligned}
%% Creator: Inkscape 0.48.2, www.inkscape.org
%% PDF/EPS/PS + LaTeX output extension by Johan Engelen, 2010
%% Accompanies image file 'functional.pdf' (pdf, eps, ps)
%%
%% To include the image in your LaTeX document, write
%%   \input{<filename>.pdf_tex}
%%  instead of
%%   \includegraphics{<filename>.pdf}
%% To scale the image, write
%%   \def\svgwidth{<desired width>}
%%   \input{<filename>.pdf_tex}
%%  instead of
%%   \includegraphics[width=<desired width>]{<filename>.pdf}
%%
%% Images with a different path to the parent latex file can
%% be accessed with the `import' package (which may need to be
%% installed) using
%%   \usepackage{import}
%% in the preamble, and then including the image with
%%   \import{<path to file>}{<filename>.pdf_tex}
%% Alternatively, one can specify
%%   \graphicspath{{<path to file>/}}
%% 
%% For more information, please see info/svg-inkscape on CTAN:
%%   http://tug.ctan.org/tex-archive/info/svg-inkscape
%%
\begingroup%
  \makeatletter%
  \providecommand\color[2][]{%
    \errmessage{(Inkscape) Color is used for the text in Inkscape, but the package 'color.sty' is not loaded}%
    \renewcommand\color[2][]{}%
  }%
  \providecommand\transparent[1]{%
    \errmessage{(Inkscape) Transparency is used (non-zero) for the text in Inkscape, but the package 'transparent.sty' is not loaded}%
    \renewcommand\transparent[1]{}%
  }%
  \providecommand\rotatebox[2]{#2}%
  \ifx\svgwidth\undefined%
    \setlength{\unitlength}{360.69047852bp}%
    \ifx\svgscale\undefined%
      \relax%
    \else%
      \setlength{\unitlength}{\unitlength * \real{\svgscale}}%
    \fi%
  \else%
    \setlength{\unitlength}{\svgwidth}%
  \fi%
  \global\let\svgwidth\undefined%
  \global\let\svgscale\undefined%
  \makeatother%
  \begin{picture}(1,0.10540295)%
    \put(0,0){\includegraphics[width=\unitlength]{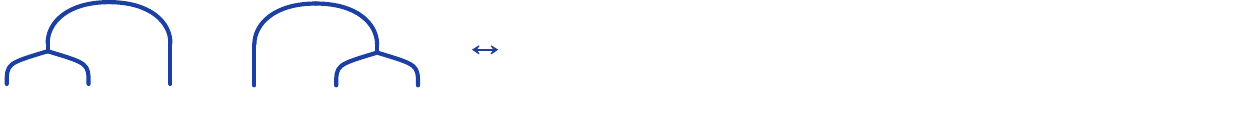}}%
    \put(-0.00079058,0.00273456){\color[rgb]{0,0,0}\makebox(0,0)[lb]{\smash{$a$}}}%
    \put(0.06528623,0.00228511){\color[rgb]{0,0,0}\makebox(0,0)[lb]{\smash{$b$}}}%
    \put(0.13001444,0.00273456){\color[rgb]{0,0,0}\makebox(0,0)[lb]{\smash{$c$}}}%
    \put(0.19699027,0.00228511){\color[rgb]{0,0,0}\makebox(0,0)[lb]{\smash{$a$}}}%
    \put(0.26239285,0.00273462){\color[rgb]{0,0,0}\makebox(0,0)[lb]{\smash{$b$}}}%
    \put(0.32734585,0.00295941){\color[rgb]{0,0,0}\makebox(0,0)[lb]{\smash{$c$}}}%
    \put(0.15665747,0.05562939){\color[rgb]{0,0,0}\makebox(0,0)[lb]{\smash{$=$}}}%
    \put(0.44911432,0.05770264){\color[rgb]{0,0,0}\makebox(0,0)[lb]{\smash{$B^{-1}(e_a \cdot e_b, e_c)=B^{-1}(e_a, e_b \cdot e_c)$}}}%
  \end{picture}%
\endgroup%

\label{eq:associative_bilinear}
\end{aligned}
\end{align}
 This means that a functional $\gls{eps}\colon A\to \gls{k}$ can be defined by $\varepsilon(x)= B^{-1}(x,1)$. However, there are no additional symmetry requirements that $\varepsilon$ must obey. 

To simplify the exposition of the 1-3 Pachner move, a 2-2 move was performed on the two left-most triangles of figure~\ref{fig:Pach2}. The relation
%%%%%%%%%%%%%%%%%%%%%%%%%%%%%%%%%%%%%%%%%%%%%%%%%%%%%
\begin{align}
\begin{aligned}
%% Creator: Inkscape 0.48.2, www.inkscape.org
%% PDF/EPS/PS + LaTeX output extension by Johan Engelen, 2010
%% Accompanies image file 'diag_pach_31.pdf' (pdf, eps, ps)
%%
%% To include the image in your LaTeX document, write
%%   \input{<filename>.pdf_tex}
%%  instead of
%%   \includegraphics{<filename>.pdf}
%% To scale the image, write
%%   \def\svgwidth{<desired width>}
%%   \input{<filename>.pdf_tex}
%%  instead of
%%   \includegraphics[width=<desired width>]{<filename>.pdf}
%%
%% Images with a different path to the parent latex file can
%% be accessed with the `import' package (which may need to be
%% installed) using
%%   \usepackage{import}
%% in the preamble, and then including the image with
%%   \import{<path to file>}{<filename>.pdf_tex}
%% Alternatively, one can specify
%%   \graphicspath{{<path to file>/}}
%% 
%% For more information, please see info/svg-inkscape on CTAN:
%%   http://tug.ctan.org/tex-archive/info/svg-inkscape
%%
\begingroup%
  \makeatletter%
  \providecommand\color[2][]{%
    \errmessage{(Inkscape) Color is used for the text in Inkscape, but the package 'color.sty' is not loaded}%
    \renewcommand\color[2][]{}%
  }%
  \providecommand\transparent[1]{%
    \errmessage{(Inkscape) Transparency is used (non-zero) for the text in Inkscape, but the package 'transparent.sty' is not loaded}%
    \renewcommand\transparent[1]{}%
  }%
  \providecommand\rotatebox[2]{#2}%
  \ifx\svgwidth\undefined%
    \setlength{\unitlength}{357.37328835bp}%
    \ifx\svgscale\undefined%
      \relax%
    \else%
      \setlength{\unitlength}{\unitlength * \real{\svgscale}}%
    \fi%
  \else%
    \setlength{\unitlength}{\svgwidth}%
  \fi%
  \global\let\svgwidth\undefined%
  \global\let\svgscale\undefined%
  \makeatother%
  \begin{picture}(1,0.36569759)%
    \put(0,0){\includegraphics[width=\unitlength]{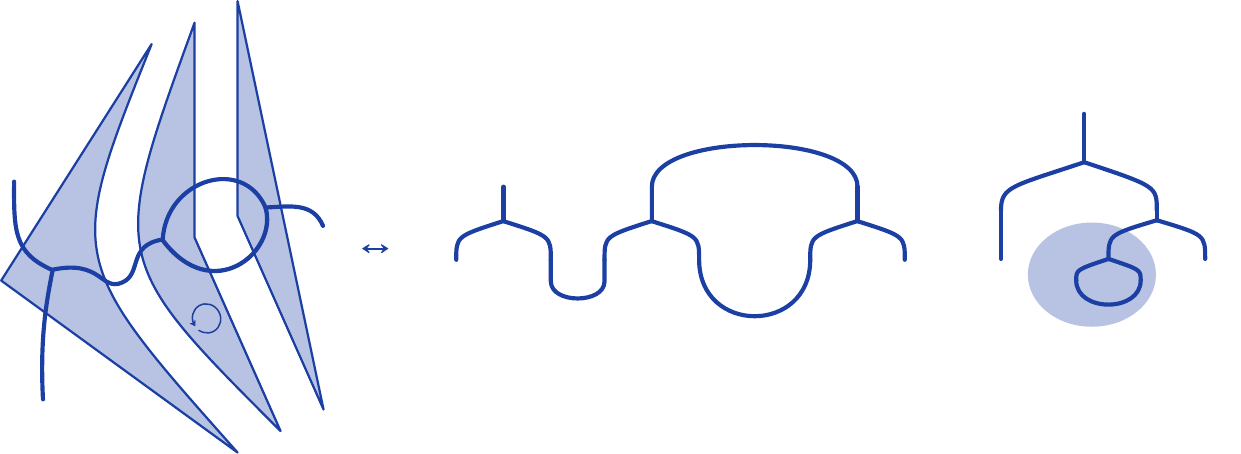}}%
    \put(0.00678601,0.23597322){\color[rgb]{0,0,0}\makebox(0,0)[lb]{\smash{$a$}}}%
    \put(0.0281087,0.00570804){\color[rgb]{0,0,0}\makebox(0,0)[lb]{\smash{$b$}}}%
    \put(0.25449214,0.15658026){\color[rgb]{0,0,0}\makebox(0,0)[lb]{\smash{$c$}}}%
    \put(0.39876076,0.23824156){\color[rgb]{0,0,0}\makebox(0,0)[lb]{\smash{$a$}}}%
    \put(0.36155939,0.11801791){\color[rgb]{0,0,0}\makebox(0,0)[lb]{\smash{$b$}}}%
    \put(0.72313771,0.11711041){\color[rgb]{0,0,0}\makebox(0,0)[lb]{\smash{$c$}}}%
    \put(0.31755289,0.1561265){\color[rgb]{0,0,0}\makebox(0,0)[lb]{\smash{$R$}}}%
    \put(0.86649879,0.29585827){\color[rgb]{0,0,0}\makebox(0,0)[lb]{\smash{$a$}}}%
    \put(0.80026233,0.12028632){\color[rgb]{0,0,0}\makebox(0,0)[lb]{\smash{$b$}}}%
    \put(0.83837099,0.13253546){\color[rgb]{0,0,0}\makebox(0,0)[lb]{\smash{$R$}}}%
    \put(0.96494605,0.11983257){\color[rgb]{0,0,0}\makebox(0,0)[lb]{\smash{$c$}}}%
    \put(0.74998507,0.16928553){\color[rgb]{0,0,0}\makebox(0,0)[lb]{\smash{$=$}}}%
  \end{picture}%
\endgroup%

\label{fig:diag_pach_31}
\end{aligned}
\end{align}
%%%%%%%%%%%%%%%%%%%%%%%%%%%%%%%%%%%%%%%%%%%%%%%%%%%%%%%
is obtained. It was simplified using the definition of multiplication components and associativity. The 1-3 Pachner move predicts the expression above must equal $C_{bc}{}^a$. Since $\gls{Ct}$ is assumed to be non-degenerate one concludes the highlighted element, $\gls{R}\,\gls{m}(\gls{Bb})$, must satisfy $R\,m(B)=\gls{unit}$. Therefore, $\gls{A}$ is a special Frobenius algebra. 

Conversely, given a special Frobenius algebra with multiplication $m$ and a linear functional $\gls{eps}$, a non-degenerate bilinear form is defined by $B^{-1}= \varepsilon \circ m$, with property \eqref{eq:gen_cycle}. As previously stated, the fact the algebra is unital implies the non-degeneracy of $C$, while associativity and the relation $R\, m(B)=1$ guarantee invariance under Pachner moves. The diagrammatic state sum model created is therefore planar.
\end{proof}

It is worth noting that having $m(B)$ proportional to the algebra unit is a non-trivial restriction on Frobenius algebras. The following arguments show that this condition implies the algebra must be separable, a concept we will define shortly. Note that some presentations of these state sum models \cite{Lauda,Runkel} assume from the outset the algebra is of this type. There are a number of equivalent definitions of the separability condition; the most convenient one for the purpose of this work is as follows \cite{KadisonStolin}, where the vector space $A\otimes A$ is a bimodule over $A$ with the actions $x\triangleright (u\otimes v)=(x\cdot u)\otimes v$ and $(u\otimes v)\triangleleft x=u\otimes(v\cdot x)$.
\begin{definition}[Separable algebra]
An algebra $A$ is called separable if there exists $t \in A \otimes A$ such that $x\triangleright t=t\triangleleft x$ for all $x \in A$ and $m(t)=1 \in A$.
\end{definition}

The relevance of this definition to the state sum models is given in the following lemma.
\begin{lemma}\label{separable} A special Frobenius algebra is a separable algebra.
\end{lemma} 

\begin{proof} Define $\gls{R}\in \gls{k}$ by $\gls{m}(\gls{Bb})=R^{-1}\gls{unit}$. Using the basis $\{e_a\}$ of the Frobenius algebra $\gls{A}$ with Frobenius form $\gls{eps}$, define $\gls{B}=\varepsilon(e_a\cdot e_b)$, $B^{ab}B_{bc}=\delta^a_c$, and set $t=R\,e_a\otimes e_b\,B^{ab}$. Then the identity $\varepsilon (y\cdot e_a)\,e_b\,B^{ab} =y$ for all $y\in A$ follows. Using this identity twice, one finds $\varepsilon(y\cdot x\cdot e_a)\,e_b\,B^{ab}=y\cdot x=\varepsilon(y\cdot e_a)\,e_b\cdot x\,B^{ab}$ which can diagrammatically be depicted as 
\begin{align}
\begin{aligned}
%% Creator: Inkscape 0.48.2, www.inkscape.org
%% PDF/EPS/PS + LaTeX output extension by Johan Engelen, 2010
%% Accompanies image file 'diag_separable.pdf' (pdf, eps, ps)
%%
%% To include the image in your LaTeX document, write
%%   \input{<filename>.pdf_tex}
%%  instead of
%%   \includegraphics{<filename>.pdf}
%% To scale the image, write
%%   \def\svgwidth{<desired width>}
%%   \input{<filename>.pdf_tex}
%%  instead of
%%   \includegraphics[width=<desired width>]{<filename>.pdf}
%%
%% Images with a different path to the parent latex file can
%% be accessed with the `import' package (which may need to be
%% installed) using
%%   \usepackage{import}
%% in the preamble, and then including the image with
%%   \import{<path to file>}{<filename>.pdf_tex}
%% Alternatively, one can specify
%%   \graphicspath{{<path to file>/}}
%% 
%% For more information, please see info/svg-inkscape on CTAN:
%%   http://tug.ctan.org/tex-archive/info/svg-inkscape
%%
\begingroup%
  \makeatletter%
  \providecommand\color[2][]{%
    \errmessage{(Inkscape) Color is used for the text in Inkscape, but the package 'color.sty' is not loaded}%
    \renewcommand\color[2][]{}%
  }%
  \providecommand\transparent[1]{%
    \errmessage{(Inkscape) Transparency is used (non-zero) for the text in Inkscape, but the package 'transparent.sty' is not loaded}%
    \renewcommand\transparent[1]{}%
  }%
  \providecommand\rotatebox[2]{#2}%
  \ifx\svgwidth\undefined%
    \setlength{\unitlength}{201.36540557bp}%
    \ifx\svgscale\undefined%
      \relax%
    \else%
      \setlength{\unitlength}{\unitlength * \real{\svgscale}}%
    \fi%
  \else%
    \setlength{\unitlength}{\svgwidth}%
  \fi%
  \global\let\svgwidth\undefined%
  \global\let\svgscale\undefined%
  \makeatother%
  \begin{picture}(1,0.22623053)%
    \put(0,0){\includegraphics[width=\unitlength]{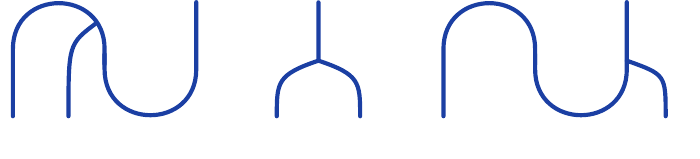}}%
    \put(-0.00141611,0.00836089){\color[rgb]{0,0,0}\makebox(0,0)[lb]{\smash{$y$}}}%
    \put(0.31426321,0.11159446){\color[rgb]{0,0,0}\makebox(0,0)[lb]{\smash{$=$}}}%
    \put(0.07804143,0.00836089){\color[rgb]{0,0,0}\makebox(0,0)[lb]{\smash{$x$}}}%
    \put(0.37600721,0.00836089){\color[rgb]{0,0,0}\makebox(0,0)[lb]{\smash{$y$}}}%
    \put(0.49519352,0.00836089){\color[rgb]{0,0,0}\makebox(0,0)[lb]{\smash{$x$}}}%
    \put(0.61437983,0.00836089){\color[rgb]{0,0,0}\makebox(0,0)[lb]{\smash{$y$}}}%
    \put(0.54684092,0.11165569){\color[rgb]{0,0,0}\makebox(0,0)[lb]{\smash{$=$}}}%
    \put(0.93221,0.00836089){\color[rgb]{0,0,0}\makebox(0,0)[lb]{\smash{$x$}}}%
    \put(0.99241506,0.1207863){\color[rgb]{0,0,0}\makebox(0,0)[lb]{\smash{.}}}%
  \end{picture}%
\endgroup%

\end{aligned}
\end{align}
Then, the non-degeneracy of $\varepsilon$ guarantees that  $x \triangleright t=t \triangleleft x$ for all $x\in A$: 
\begin{align}
\begin{aligned}
%% Creator: Inkscape 0.48.2, www.inkscape.org
%% PDF/EPS/PS + LaTeX output extension by Johan Engelen, 2010
%% Accompanies image file 'diag_separable_frob.pdf' (pdf, eps, ps)
%%
%% To include the image in your LaTeX document, write
%%   \input{<filename>.pdf_tex}
%%  instead of
%%   \includegraphics{<filename>.pdf}
%% To scale the image, write
%%   \def\svgwidth{<desired width>}
%%   \input{<filename>.pdf_tex}
%%  instead of
%%   \includegraphics[width=<desired width>]{<filename>.pdf}
%%
%% Images with a different path to the parent latex file can
%% be accessed with the `import' package (which may need to be
%% installed) using
%%   \usepackage{import}
%% in the preamble, and then including the image with
%%   \import{<path to file>}{<filename>.pdf_tex}
%% Alternatively, one can specify
%%   \graphicspath{{<path to file>/}}
%% 
%% For more information, please see info/svg-inkscape on CTAN:
%%   http://tug.ctan.org/tex-archive/info/svg-inkscape
%%
\begingroup%
  \makeatletter%
  \providecommand\color[2][]{%
    \errmessage{(Inkscape) Color is used for the text in Inkscape, but the package 'color.sty' is not loaded}%
    \renewcommand\color[2][]{}%
  }%
  \providecommand\transparent[1]{%
    \errmessage{(Inkscape) Transparency is used (non-zero) for the text in Inkscape, but the package 'transparent.sty' is not loaded}%
    \renewcommand\transparent[1]{}%
  }%
  \providecommand\rotatebox[2]{#2}%
  \ifx\svgwidth\undefined%
    \setlength{\unitlength}{113.8421875bp}%
    \ifx\svgscale\undefined%
      \relax%
    \else%
      \setlength{\unitlength}{\unitlength * \real{\svgscale}}%
    \fi%
  \else%
    \setlength{\unitlength}{\svgwidth}%
  \fi%
  \global\let\svgwidth\undefined%
  \global\let\svgscale\undefined%
  \makeatother%
  \begin{picture}(1,0.2656821)%
    \put(0,0){\includegraphics[width=\unitlength]{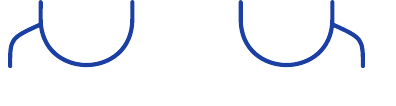}}%
    \put(0.42615874,0.16886726){\color[rgb]{0,0,0}\makebox(0,0)[lb]{\smash{$=$}}}%
    \put(0.88995869,0.00724001){\color[rgb]{0,0,0}\makebox(0,0)[lb]{\smash{$x$}}}%
    \put(0.97428595,0.20400362){\color[rgb]{0,0,0}\makebox(0,0)[lb]{\smash{.}}}%
    \put(-0.00250484,0.00724001){\color[rgb]{0,0,0}\makebox(0,0)[lb]{\smash{$x$}}}%
  \end{picture}%
\endgroup%

\label{eq:prop_sep}
\end{aligned}
\end{align}
Also, $m(t)=R\, m(B)=1$.
\end{proof}

For a field $k$ of characteristic zero, separability for an algebra is equivalent to it being both finite dimensional and semisimple \cite{KadisonStolin,Aguiar}. Therefore, if $k=\Rb$ or $\Cb$ these Frobenius algebras are easily classified. 

Consider the complex algebra $A=\Mb_{n}(\Cb)$ with Frobenius form $\varepsilon(a)=\Tr(xa)$ for some fixed invertible element $\gls{x} \in A$. This determines the non-degenerate bilinear form $B^{-1}(a,b)=\Tr(xab)$. Let $\lbrace e_{lm} \rbrace_{l,m=1,n}$ be the basis of elementary matrices such that $(e_{lm})_{rs}=\delta_{lr}\delta_{ms}$. Then $B$ must be given by
\begin{align}
B=\sum_{lm} e_{lm}\, x^{-1} \otimes e_{ml} \in A \otimes A.
\end{align}
The defining equation \eqref{eq:snake} is satisfied since the cyclicity of the trace guarantees $\sum_{lm}\Tr(xae_{lm}x^{-1})e_{ml}=\sum_{lm}\Tr(ae_{lm})e_{ml}=a$ for all $a \in A$. Moreover, the distinguished element satisfies $m(B)=\Tr(x^{-1})1$. This identity follows from noticing that $p(a)=\sum_{lm} e_{lm}ae_{ml}=\Tr(a)1$, where the map $p$ is proportional to a projector $\gls{A} \to A$ with the centre of $A$, $\gls{ZA}$, as its image. (The properties of $p$ are established in a more general context through lemma~\ref{lem:projector}.) Thus our example will define a planar state sum model if $\gls{R}^{-1}=\Tr(\gls{x}^{-1})$. This particular example will be used to prove the theorem below.
%%%%%%%%%%%%%% end

\begin{theorem} \label{theo:diagram_semi_simple}
A planar state sum model over the field $k = \Cb$ or $\Rb$ is isomorphic by a change of basis to one in which the algebra is a direct sum of
matrix algebras over $\Cb$ or division rings $\Rb,\Cb_{\Rb},\Hb_{\Rb}$ and the Frobenius form is determined by a fixed invertible element $x =\oplus_i\, x_i \in A$. For a complex algebra
\begin{align} \label{eq:complex_semi}
A = \bigoplus\limits_{i=1}^N \Mb_{n_i}(\Cb), 
\end{align}
the functional takes the form
\begin{align}
\gls{eps}(a) = \sum\limits_{i=1}^N \Tr (x_ia_i).
\end{align}
The element $x$ must satisfy the relations $R\Tr(x^{-1}_i)=1$ for all $i=1,\cdots,N$. For a real algebra
\begin{align} \label{eq:real_semi}
A = \bigoplus\limits_{i=1}^N \Mb_{n_i}(D_i) \text{ with }D_i=\Rb,\Cb_{\Rb},\Hb_{\Rb}
\end{align}
the Frobenius form is given by
\begin{align}
\varepsilon(a) = \sum\limits_{i=1}^N \Real  \Tr(x_ia_i).
\end{align}
The element $x$ must satisfy the relations
\begin{align}
R^{-1}=
\begin{cases}
\Tr(x_i^{-1}) & (D_i = \Rb) \\
2\Tr(x_i^{-1}) &  (D_i = \Cb_{\Rb}) \\
4 \Real\Tr(x_i^{-1}) &  (D_i = \Hb_{\Rb})
\end{cases}
\end{align}
 for all $i=1,\cdots,N$.
\end{theorem}
\begin{proof}
The classification of Frobenius forms on an algebra~\cite{Lauda,Kock} shows that any two Frobenius forms $\gls{eps}$, $\tilde{\varepsilon}$ are related by an invertible element $\gls{x} \in \gls{A}$ as $\varepsilon(a)=\tilde{\varepsilon}(xa)$. Thus, for the complex case, one can write
\begin{align}
\varepsilon(a)=\sum_i\Tr(x_ia_i)
\end{align}
using the decomposition $x=\oplus_i \, x_i$ and lemma \ref{lem:Frobenius_forms}. Let the unit element be decomposed as $\gls{unit}=\oplus_i \, 1_i$; from the example of a simple matrix algebra previously studied, one concludes $R\, m(B_i)=1_i$ with $1_i$ the unit element in $\Mb_{n_i}(\Cb)$. Consequently, setting $\gls{R}\, \gls{m}(\gls{Bb})=1$ gives the relation $R\Tr(x_i^{-1})=1$. 

As established in \S\ref{sec:lattice_tft}, $\Real\Tr$ is a Frobenius functional for a matrix algebra over a real division ring. Thus, for an algebra \eqref{eq:real_semi}, one can write
\begin{align}
\varepsilon(a)=\sum_i\Real\Tr(x_ia_i).
\end{align}
It is easy to verify the bilinear form $B$ associated with this Frobenius functional satisfies
\begin{align} \label{eq:B_form}
B=\sum_{i,\,w_i,\,lm}\,w_i\,e_{lm}^{i}\,x_i^{-1}\otimes \, \overline{w_i} \, e_{ml}^i
\end{align}
using the basis defined in lemma~\ref{lem:Frobenius_forms}; one then finds 
\begin{align}
m(B)=\sum_{i,\,w_i} w_i\Tr(x_i^{-1}) \overline{w_i}\, 1_i \hspace{1mm}.
\end{align}
For the identity $R\, m(B)=1$ to hold one must have $R^{-1}=\sum_{w_i}w_i\Tr(x_i^{-1})\, \overline{w_i}$ for all $i$. If $D_i=\Rb$ or $\Cb_{\Rb}$, then $\overline{w_i}$ and $\Tr(x_i^{-1})$ commute, which means the expression reduces to $R^{-1}=\Tr(x_i^{-1})$ and $R^{-1}=2\Tr(x_i^{-1})$ respectively. If $D_i=\Hb_{\Rb}$, the expression reduces to $R^{-1}=4\Real\Tr(x_i^{-1})$ -- the non-real components of the trace are automatically cancelled.

\end{proof}

As one might expect, the study of state sum models done in  \S\ref{sec:lattice_tft} for the disk can be regarded as a special case of theorem \ref{theo:diagram_semi_simple}.

\begin{corollary} \label{cor:FHK}
An FHK state sum model on the disk over the field $\gls{k} = \Cb$ or $\Rb$ is a planar state sum model in the conditions of theorem~\ref{theo:diagram_semi_simple} where the Frobenius form is symmetric. If the algebra is of the form \eqref{eq:complex_semi} then $\gls{x}=\gls{R}\oplus_i n_i1_i$; if it is of the form \eqref{eq:real_semi} then $x=R\oplus_i|D_i|n_i1_i$. The data $\gls{A}$ and $R$ therefore uniquely determine the Frobenius form of an FHK model.
\end{corollary}
\begin{proof}
This is a special case of theorem~\ref{theo:diagram_semi_simple} where $\gls{eps}$ must be symmetric. This means $x$ must be a central element and can, therefore, be written as $x=\oplus_i \lambda_i 1_i$. The constants $\lambda_i$ must be in $\Cb$ if the underlying field is $\Cb$ or if $D_i=\Cb_{\Rb}$; otherwise, they must be real numbers (recall that only real numbers commute with all the quaternions). Each of these constants must then satisfy $R^{-1}=\lambda_i^{-1}n_i$ in the complex case or $R^{-1}=\lambda_i^{-1}|D_i|n_i$ in the real one. In other words $x=R\oplus_in_i1_i$ or $x=R\oplus_i|D_i|n_i1_i$, respectively.
\end{proof}
%%%%%%%%%%%%%%%%%%%%%%%%%%%%%%%%%%%%%%%%%%%%%%%%%%%%%%%%%%%%%%%%
%%%%%%%%%%%%%%%%%%%%%%%%%%%%%%%%%%%%%%%%%%%%%%%%%%%%%%%%%%%%%%%%
\section{Spherical models}\label{sec:spherical}
%%%%%%%%%%%%%%%%%%%%%%%%%%%%%%%%%%%%%%%%%%%%%%%%%%%%%%%%%%%%%%%%
%%%%%%%%%%%%%%%%%%%%%%%%%%%%%%%%%%%%%%%%%%%%%%%%%%%%%%%%%%%%%%%%

The main objective of our treatment is the creation of partition functions for closed surfaces. Within the naive framework this extension of the formalism from the disk to any Riemann surface appeared as rather natural. We possessed a mechanism, the matrix $\gls{Bb}$, to identify internal edges and we just extended its use to the identification of disk boundaries -- thus creating topological invariants for closed surfaces. The introduction of the diagrammatic calculus, however, will show us how, although intuitive, this reasoning is not the only valid one and should be treated more carefully. We start our extension of planar models with the example of the sphere. 

Suppose that $M$ is a subset of the sphere, $M\subset\Sigma_0$, with a chosen orientation. Here, we describe the sphere as the completion of $\Rb^2$ with a `point at infinity' $p$. To be more precise, we are using the description of the sphere that relates it to $\Rb^2$ through a stereographic projection \cite{MDG}. The point $p$ could be chosen for example to be the North pole of the sphere.

If $M$ has a non-trivial boundary then it is topologically equivalent to $M$ seen as an $\Rb^2$ subset and this equivalence is independent of where the point $p$ is chosen to lie. A planar model for $M \subset \Sigma_0$ can therefore be defined as the planar model for $M$ seen as a subset of $\Rb^2$ through an orientation-preserving isomorphism of $\Sigma_0-\{p\}$ to $\Rb^2$, with the point $p$ chosen not to lie in the dual graph of the triangulation of $M$. A new condition, however, emerges: the freedom to move $p$ around the sphere corresponds to the spherical move \cite{BarrettWestbury}
\begin{equation}\label{eq:blob}
\begin{aligned}
%% Creator: Inkscape 0.48.2, www.inkscape.org
%% PDF/EPS/PS + LaTeX output extension by Johan Engelen, 2010
%% Accompanies image file 'sphere_blob.pdf' (pdf, eps, ps)
%%
%% To include the image in your LaTeX document, write
%%   \input{<filename>.pdf_tex}
%%  instead of
%%   \includegraphics{<filename>.pdf}
%% To scale the image, write
%%   \def\svgwidth{<desired width>}
%%   \input{<filename>.pdf_tex}
%%  instead of
%%   \includegraphics[width=<desired width>]{<filename>.pdf}
%%
%% Images with a different path to the parent latex file can
%% be accessed with the `import' package (which may need to be
%% installed) using
%%   \usepackage{import}
%% in the preamble, and then including the image with
%%   \import{<path to file>}{<filename>.pdf_tex}
%% Alternatively, one can specify
%%   \graphicspath{{<path to file>/}}
%% 
%% For more information, please see info/svg-inkscape on CTAN:
%%   http://tug.ctan.org/tex-archive/info/svg-inkscape
%%
\begingroup%
  \makeatletter%
  \providecommand\color[2][]{%
    \errmessage{(Inkscape) Color is used for the text in Inkscape, but the package 'color.sty' is not loaded}%
    \renewcommand\color[2][]{}%
  }%
  \providecommand\transparent[1]{%
    \errmessage{(Inkscape) Transparency is used (non-zero) for the text in Inkscape, but the package 'transparent.sty' is not loaded}%
    \renewcommand\transparent[1]{}%
  }%
  \providecommand\rotatebox[2]{#2}%
  \ifx\svgwidth\undefined%
    \setlength{\unitlength}{141.73972168bp}%
    \ifx\svgscale\undefined%
      \relax%
    \else%
      \setlength{\unitlength}{\unitlength * \real{\svgscale}}%
    \fi%
  \else%
    \setlength{\unitlength}{\svgwidth}%
  \fi%
  \global\let\svgwidth\undefined%
  \global\let\svgscale\undefined%
  \makeatother%
  \begin{picture}(1,0.58091433)%
    \put(0,0){\includegraphics[width=\unitlength]{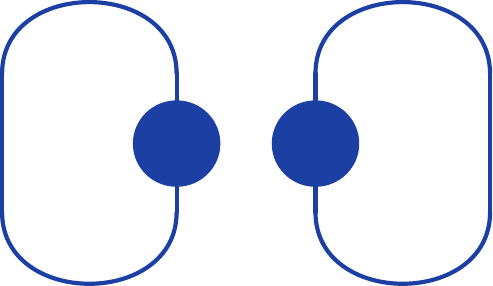}}%
    \put(0.47174085,0.2609378){\color[rgb]{0,0,0}\makebox(0,0)[lb]{\smash{$=$}}}%
  \end{picture}%
\endgroup%

\end{aligned}
\end{equation}
where $\begin{aligned}%% Creator: Inkscape 0.48.2, www.inkscape.org
%% PDF/EPS/PS + LaTeX output extension by Johan Engelen, 2010
%% Accompanies image file 'blob.pdf' (pdf, eps, ps)
%%
%% To include the image in your LaTeX document, write
%%   \input{<filename>.pdf_tex}
%%  instead of
%%   \includegraphics{<filename>.pdf}
%% To scale the image, write
%%   \def\svgwidth{<desired width>}
%%   \input{<filename>.pdf_tex}
%%  instead of
%%   \includegraphics[width=<desired width>]{<filename>.pdf}
%%
%% Images with a different path to the parent latex file can
%% be accessed with the `import' package (which may need to be
%% installed) using
%%   \usepackage{import}
%% in the preamble, and then including the image with
%%   \import{<path to file>}{<filename>.pdf_tex}
%% Alternatively, one can specify
%%   \graphicspath{{<path to file>/}}
%% 
%% For more information, please see info/svg-inkscape on CTAN:
%%   http://tug.ctan.org/tex-archive/info/svg-inkscape
%%
\begingroup%
  \makeatletter%
  \providecommand\color[2][]{%
    \errmessage{(Inkscape) Color is used for the text in Inkscape, but the package 'color.sty' is not loaded}%
    \renewcommand\color[2][]{}%
  }%
  \providecommand\transparent[1]{%
    \errmessage{(Inkscape) Transparency is used (non-zero) for the text in Inkscape, but the package 'transparent.sty' is not loaded}%
    \renewcommand\transparent[1]{}%
  }%
  \providecommand\rotatebox[2]{#2}%
  \ifx\svgwidth\undefined%
    \setlength{\unitlength}{8.40000381bp}%
    \ifx\svgscale\undefined%
      \relax%
    \else%
      \setlength{\unitlength}{\unitlength * \real{\svgscale}}%
    \fi%
  \else%
    \setlength{\unitlength}{\svgwidth}%
  \fi%
  \global\let\svgwidth\undefined%
  \global\let\svgscale\undefined%
  \makeatother%
  \begin{picture}(1,1.61712173)%
    \put(0,0){\includegraphics[width=\unitlength]{blob.pdf}}%
  \end{picture}%
\endgroup%
\end{aligned}$ consists of a diagram that is the same on both sides of the equation. Alternatively, this move can be also understood as making the arc on the left-hand side larger until it passes the point at infinity on the sphere, when it then re-enters the planar diagram as an arc on the right-hand side.

A sufficient condition that guarantees  \eqref{eq:blob} holds for any matrix representing $\begin{aligned}%% Creator: Inkscape 0.48.2, www.inkscape.org
%% PDF/EPS/PS + LaTeX output extension by Johan Engelen, 2010
%% Accompanies image file 'blob.pdf' (pdf, eps, ps)
%%
%% To include the image in your LaTeX document, write
%%   \input{<filename>.pdf_tex}
%%  instead of
%%   \includegraphics{<filename>.pdf}
%% To scale the image, write
%%   \def\svgwidth{<desired width>}
%%   \input{<filename>.pdf_tex}
%%  instead of
%%   \includegraphics[width=<desired width>]{<filename>.pdf}
%%
%% Images with a different path to the parent latex file can
%% be accessed with the `import' package (which may need to be
%% installed) using
%%   \usepackage{import}
%% in the preamble, and then including the image with
%%   \import{<path to file>}{<filename>.pdf_tex}
%% Alternatively, one can specify
%%   \graphicspath{{<path to file>/}}
%% 
%% For more information, please see info/svg-inkscape on CTAN:
%%   http://tug.ctan.org/tex-archive/info/svg-inkscape
%%
\begingroup%
  \makeatletter%
  \providecommand\color[2][]{%
    \errmessage{(Inkscape) Color is used for the text in Inkscape, but the package 'color.sty' is not loaded}%
    \renewcommand\color[2][]{}%
  }%
  \providecommand\transparent[1]{%
    \errmessage{(Inkscape) Transparency is used (non-zero) for the text in Inkscape, but the package 'transparent.sty' is not loaded}%
    \renewcommand\transparent[1]{}%
  }%
  \providecommand\rotatebox[2]{#2}%
  \ifx\svgwidth\undefined%
    \setlength{\unitlength}{8.40000381bp}%
    \ifx\svgscale\undefined%
      \relax%
    \else%
      \setlength{\unitlength}{\unitlength * \real{\svgscale}}%
    \fi%
  \else%
    \setlength{\unitlength}{\svgwidth}%
  \fi%
  \global\let\svgwidth\undefined%
  \global\let\svgscale\undefined%
  \makeatother%
  \begin{picture}(1,1.61712173)%
    \put(0,0){\includegraphics[width=\unitlength]{blob.pdf}}%
  \end{picture}%
\endgroup%
\end{aligned}$ is 
\begin{equation} 
B_{ca}B^{cb}=B_{ac}B^{bc}.\label{eq:spherical}
\end{equation} 
The meaning of \eqref{eq:spherical} is easier to understand in the context of Frobenius algebras. 
\begin{definition}\label{Nakayama}
A Frobenius algebra has an automorphism $\gls{sigma} \colon \gls{A} \to A$ determined uniquely by the relation $\gls{eps}(x \cdot y)=\varepsilon(\sigma(y) \cdot x)$ for all $x,y \in A$. This map is known as the Nakayama automorphism.
\end{definition}

\begin{lemma}
Let $A$ be a Frobenius algebra. Then the following are equivalent:
\begin{enumerate} [label=A\arabic{*}), ref=(A\arabic{*})]
\item Equation \eqref{eq:spherical} \label{item:equation}
\item $\sigma^2=\iden$ \label{item:sigma}
\item $B^{-1}$ decomposes into the direct sum of a symmetric bilinear form and an antisymmetric one.\label{item:B}
\end{enumerate}
\end{lemma}
\begin{proof}
Note that equation \eqref{eq:spherical} can be rewritten as $(B^{-1}\gls{Bb}^{\tr})^2=\gls{unit}$ in matrix notation. The definition of $\sigma$ then implies that $\varepsilon(e_a \cdot e_b)=\varepsilon(\sigma(e_b)\cdot e_a)$ or, equivalently, $\gls{B}=\sigma_b{}^c B_{ca}$. By contracting both sides with $B^{ad}$ one can conclude that $\sigma_{b}{}^d=B_{ab}B^{ad}$ or, as matrices, $\sigma=B^{-1}B^{\tr}$. The equivalence between \ref{item:equation} and \ref{item:sigma} is then immediate.

Suppose $\gls{Bb}^{-1}$ is as in \ref{item:B}. Then the vectors $v$ that lie in the symmetric or antisymmetric subspaces of $\gls{A}$ satisfy $B^{-1}v=\pm (B^{-1})^{\tr}v$. If $B^{\tr}$ is applied to this equation the identity $B^{\tr}B^{-1}v=\pm v$ is obtained. Since it must be true for all $v$ it is also equivalent to  $(B^{\tr}B^{-1})^2=\gls{unit}$, which implies \ref{item:equation}. On the other hand, if \eqref{eq:spherical} is satisfied then $(B^{\tr}B^{-1})^2=1$. The eigenspaces with eigenvalues $\pm1$ give the direct sum decomposition of  \ref{item:equation}.
\end{proof}
 
For the case of triangulations of $M=\Sigma_0$ (with no boundary) the condition \eqref{eq:spherical} is not required. In these cases, $\begin{aligned}%% Creator: Inkscape 0.48.2, www.inkscape.org
%% PDF/EPS/PS + LaTeX output extension by Johan Engelen, 2010
%% Accompanies image file 'blob.pdf' (pdf, eps, ps)
%%
%% To include the image in your LaTeX document, write
%%   \input{<filename>.pdf_tex}
%%  instead of
%%   \includegraphics{<filename>.pdf}
%% To scale the image, write
%%   \def\svgwidth{<desired width>}
%%   \input{<filename>.pdf_tex}
%%  instead of
%%   \includegraphics[width=<desired width>]{<filename>.pdf}
%%
%% Images with a different path to the parent latex file can
%% be accessed with the `import' package (which may need to be
%% installed) using
%%   \usepackage{import}
%% in the preamble, and then including the image with
%%   \import{<path to file>}{<filename>.pdf_tex}
%% Alternatively, one can specify
%%   \graphicspath{{<path to file>/}}
%% 
%% For more information, please see info/svg-inkscape on CTAN:
%%   http://tug.ctan.org/tex-archive/info/svg-inkscape
%%
\begingroup%
  \makeatletter%
  \providecommand\color[2][]{%
    \errmessage{(Inkscape) Color is used for the text in Inkscape, but the package 'color.sty' is not loaded}%
    \renewcommand\color[2][]{}%
  }%
  \providecommand\transparent[1]{%
    \errmessage{(Inkscape) Transparency is used (non-zero) for the text in Inkscape, but the package 'transparent.sty' is not loaded}%
    \renewcommand\transparent[1]{}%
  }%
  \providecommand\rotatebox[2]{#2}%
  \ifx\svgwidth\undefined%
    \setlength{\unitlength}{8.40000381bp}%
    \ifx\svgscale\undefined%
      \relax%
    \else%
      \setlength{\unitlength}{\unitlength * \real{\svgscale}}%
    \fi%
  \else%
    \setlength{\unitlength}{\svgwidth}%
  \fi%
  \global\let\svgwidth\undefined%
  \global\let\svgscale\undefined%
  \makeatother%
  \begin{picture}(1,1.61712173)%
    \put(0,0){\includegraphics[width=\unitlength]{blob.pdf}}%
  \end{picture}%
\endgroup%
\end{aligned}$ in  \eqref{eq:blob} equals the identity matrix and so equation  \eqref{eq:blob} holds for any special Frobenius algebra. To see this note that $\begin{aligned}%% Creator: Inkscape 0.48.2, www.inkscape.org
%% PDF/EPS/PS + LaTeX output extension by Johan Engelen, 2010
%% Accompanies image file 'blob.pdf' (pdf, eps, ps)
%%
%% To include the image in your LaTeX document, write
%%   \input{<filename>.pdf_tex}
%%  instead of
%%   \includegraphics{<filename>.pdf}
%% To scale the image, write
%%   \def\svgwidth{<desired width>}
%%   \input{<filename>.pdf_tex}
%%  instead of
%%   \includegraphics[width=<desired width>]{<filename>.pdf}
%%
%% Images with a different path to the parent latex file can
%% be accessed with the `import' package (which may need to be
%% installed) using
%%   \usepackage{import}
%% in the preamble, and then including the image with
%%   \import{<path to file>}{<filename>.pdf_tex}
%% Alternatively, one can specify
%%   \graphicspath{{<path to file>/}}
%% 
%% For more information, please see info/svg-inkscape on CTAN:
%%   http://tug.ctan.org/tex-archive/info/svg-inkscape
%%
\begingroup%
  \makeatletter%
  \providecommand\color[2][]{%
    \errmessage{(Inkscape) Color is used for the text in Inkscape, but the package 'color.sty' is not loaded}%
    \renewcommand\color[2][]{}%
  }%
  \providecommand\transparent[1]{%
    \errmessage{(Inkscape) Transparency is used (non-zero) for the text in Inkscape, but the package 'transparent.sty' is not loaded}%
    \renewcommand\transparent[1]{}%
  }%
  \providecommand\rotatebox[2]{#2}%
  \ifx\svgwidth\undefined%
    \setlength{\unitlength}{8.40000381bp}%
    \ifx\svgscale\undefined%
      \relax%
    \else%
      \setlength{\unitlength}{\unitlength * \real{\svgscale}}%
    \fi%
  \else%
    \setlength{\unitlength}{\svgwidth}%
  \fi%
  \global\let\svgwidth\undefined%
  \global\let\svgscale\undefined%
  \makeatother%
  \begin{picture}(1,1.61712173)%
    \put(0,0){\includegraphics[width=\unitlength]{blob.pdf}}%
  \end{picture}%
\endgroup%
\end{aligned}$ is always of the form
\begin{align}
\begin{aligned}
%% Creator: Inkscape 0.48.2, www.inkscape.org
%% PDF/EPS/PS + LaTeX output extension by Johan Engelen, 2010
%% Accompanies image file 'network.pdf' (pdf, eps, ps)
%%
%% To include the image in your LaTeX document, write
%%   \input{<filename>.pdf_tex}
%%  instead of
%%   \includegraphics{<filename>.pdf}
%% To scale the image, write
%%   \def\svgwidth{<desired width>}
%%   \input{<filename>.pdf_tex}
%%  instead of
%%   \includegraphics[width=<desired width>]{<filename>.pdf}
%%
%% Images with a different path to the parent latex file can
%% be accessed with the `import' package (which may need to be
%% installed) using
%%   \usepackage{import}
%% in the preamble, and then including the image with
%%   \import{<path to file>}{<filename>.pdf_tex}
%% Alternatively, one can specify
%%   \graphicspath{{<path to file>/}}
%% 
%% For more information, please see info/svg-inkscape on CTAN:
%%   http://tug.ctan.org/tex-archive/info/svg-inkscape
%%
\begingroup%
  \makeatletter%
  \providecommand\color[2][]{%
    \errmessage{(Inkscape) Color is used for the text in Inkscape, but the package 'color.sty' is not loaded}%
    \renewcommand\color[2][]{}%
  }%
  \providecommand\transparent[1]{%
    \errmessage{(Inkscape) Transparency is used (non-zero) for the text in Inkscape, but the package 'transparent.sty' is not loaded}%
    \renewcommand\transparent[1]{}%
  }%
  \providecommand\rotatebox[2]{#2}%
  \ifx\svgwidth\undefined%
    \setlength{\unitlength}{79.22917192bp}%
    \ifx\svgscale\undefined%
      \relax%
    \else%
      \setlength{\unitlength}{\unitlength * \real{\svgscale}}%
    \fi%
  \else%
    \setlength{\unitlength}{\svgwidth}%
  \fi%
  \global\let\svgwidth\undefined%
  \global\let\svgscale\undefined%
  \makeatother%
  \begin{picture}(1,0.90779951)%
    \put(0,0){\includegraphics[width=\unitlength]{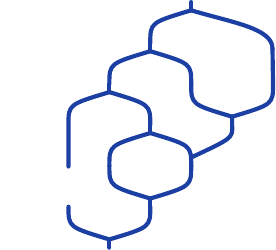}}%
    \put(0.16805466,0.19137014){\color[rgb]{0,0,0}\makebox(0,0)[lb]{\smash{$\cdots$}}}%
    \put(-0.00359913,0.40341325){\color[rgb]{0,0,0}\makebox(0,0)[lb]{\smash{$R^F$}}}%
  \end{picture}%
\endgroup%

\end{aligned}
\end{align}
where the only varying factor is the number of nodes ($F$ is the number of faces). Using associativity \eqref{fig:diagram_mult2} we can then re-write $\begin{aligned}%% Creator: Inkscape 0.48.2, www.inkscape.org
%% PDF/EPS/PS + LaTeX output extension by Johan Engelen, 2010
%% Accompanies image file 'blob.pdf' (pdf, eps, ps)
%%
%% To include the image in your LaTeX document, write
%%   \input{<filename>.pdf_tex}
%%  instead of
%%   \includegraphics{<filename>.pdf}
%% To scale the image, write
%%   \def\svgwidth{<desired width>}
%%   \input{<filename>.pdf_tex}
%%  instead of
%%   \includegraphics[width=<desired width>]{<filename>.pdf}
%%
%% Images with a different path to the parent latex file can
%% be accessed with the `import' package (which may need to be
%% installed) using
%%   \usepackage{import}
%% in the preamble, and then including the image with
%%   \import{<path to file>}{<filename>.pdf_tex}
%% Alternatively, one can specify
%%   \graphicspath{{<path to file>/}}
%% 
%% For more information, please see info/svg-inkscape on CTAN:
%%   http://tug.ctan.org/tex-archive/info/svg-inkscape
%%
\begingroup%
  \makeatletter%
  \providecommand\color[2][]{%
    \errmessage{(Inkscape) Color is used for the text in Inkscape, but the package 'color.sty' is not loaded}%
    \renewcommand\color[2][]{}%
  }%
  \providecommand\transparent[1]{%
    \errmessage{(Inkscape) Transparency is used (non-zero) for the text in Inkscape, but the package 'transparent.sty' is not loaded}%
    \renewcommand\transparent[1]{}%
  }%
  \providecommand\rotatebox[2]{#2}%
  \ifx\svgwidth\undefined%
    \setlength{\unitlength}{8.40000381bp}%
    \ifx\svgscale\undefined%
      \relax%
    \else%
      \setlength{\unitlength}{\unitlength * \real{\svgscale}}%
    \fi%
  \else%
    \setlength{\unitlength}{\svgwidth}%
  \fi%
  \global\let\svgwidth\undefined%
  \global\let\svgscale\undefined%
  \makeatother%
  \begin{picture}(1,1.61712173)%
    \put(0,0){\includegraphics[width=\unitlength]{blob.pdf}}%
  \end{picture}%
\endgroup%
\end{aligned}$ as
\begin{align}
\begin{aligned}
%% Creator: Inkscape 0.48.2, www.inkscape.org
%% PDF/EPS/PS + LaTeX output extension by Johan Engelen, 2010
%% Accompanies image file 'network2.pdf' (pdf, eps, ps)
%%
%% To include the image in your LaTeX document, write
%%   \input{<filename>.pdf_tex}
%%  instead of
%%   \includegraphics{<filename>.pdf}
%% To scale the image, write
%%   \def\svgwidth{<desired width>}
%%   \input{<filename>.pdf_tex}
%%  instead of
%%   \includegraphics[width=<desired width>]{<filename>.pdf}
%%
%% Images with a different path to the parent latex file can
%% be accessed with the `import' package (which may need to be
%% installed) using
%%   \usepackage{import}
%% in the preamble, and then including the image with
%%   \import{<path to file>}{<filename>.pdf_tex}
%% Alternatively, one can specify
%%   \graphicspath{{<path to file>/}}
%% 
%% For more information, please see info/svg-inkscape on CTAN:
%%   http://tug.ctan.org/tex-archive/info/svg-inkscape
%%
\begingroup%
  \makeatletter%
  \providecommand\color[2][]{%
    \errmessage{(Inkscape) Color is used for the text in Inkscape, but the package 'color.sty' is not loaded}%
    \renewcommand\color[2][]{}%
  }%
  \providecommand\transparent[1]{%
    \errmessage{(Inkscape) Transparency is used (non-zero) for the text in Inkscape, but the package 'transparent.sty' is not loaded}%
    \renewcommand\transparent[1]{}%
  }%
  \providecommand\rotatebox[2]{#2}%
  \ifx\svgwidth\undefined%
    \setlength{\unitlength}{166.46094971bp}%
    \ifx\svgscale\undefined%
      \relax%
    \else%
      \setlength{\unitlength}{\unitlength * \real{\svgscale}}%
    \fi%
  \else%
    \setlength{\unitlength}{\svgwidth}%
  \fi%
  \global\let\svgwidth\undefined%
  \global\let\svgscale\undefined%
  \makeatother%
  \begin{picture}(1,0.40757266)%
    \put(0,0){\includegraphics[width=\unitlength]{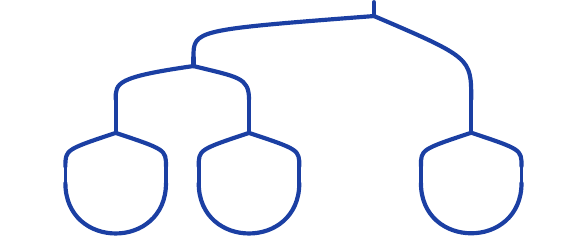}}%
    \put(-0.00171305,0.13953155){\color[rgb]{0,0,0}\makebox(0,0)[lb]{\smash{$R^F$}}}%
    \put(0.55096905,0.09147223){\color[rgb]{0,0,0}\makebox(0,0)[lb]{\smash{$\cdots$}}}%
    \put(0.93544376,0.13953155){\color[rgb]{0,0,0}\makebox(0,0)[lb]{\smash{$.$}}}%
  \end{picture}%
\endgroup%

\end{aligned}
\end{align}
Using the identity \eqref{eq:diag_1} we conclude that, as a matrix, $\begin{aligned}%% Creator: Inkscape 0.48.2, www.inkscape.org
%% PDF/EPS/PS + LaTeX output extension by Johan Engelen, 2010
%% Accompanies image file 'blob.pdf' (pdf, eps, ps)
%%
%% To include the image in your LaTeX document, write
%%   \input{<filename>.pdf_tex}
%%  instead of
%%   \includegraphics{<filename>.pdf}
%% To scale the image, write
%%   \def\svgwidth{<desired width>}
%%   \input{<filename>.pdf_tex}
%%  instead of
%%   \includegraphics[width=<desired width>]{<filename>.pdf}
%%
%% Images with a different path to the parent latex file can
%% be accessed with the `import' package (which may need to be
%% installed) using
%%   \usepackage{import}
%% in the preamble, and then including the image with
%%   \import{<path to file>}{<filename>.pdf_tex}
%% Alternatively, one can specify
%%   \graphicspath{{<path to file>/}}
%% 
%% For more information, please see info/svg-inkscape on CTAN:
%%   http://tug.ctan.org/tex-archive/info/svg-inkscape
%%
\begingroup%
  \makeatletter%
  \providecommand\color[2][]{%
    \errmessage{(Inkscape) Color is used for the text in Inkscape, but the package 'color.sty' is not loaded}%
    \renewcommand\color[2][]{}%
  }%
  \providecommand\transparent[1]{%
    \errmessage{(Inkscape) Transparency is used (non-zero) for the text in Inkscape, but the package 'transparent.sty' is not loaded}%
    \renewcommand\transparent[1]{}%
  }%
  \providecommand\rotatebox[2]{#2}%
  \ifx\svgwidth\undefined%
    \setlength{\unitlength}{8.40000381bp}%
    \ifx\svgscale\undefined%
      \relax%
    \else%
      \setlength{\unitlength}{\unitlength * \real{\svgscale}}%
    \fi%
  \else%
    \setlength{\unitlength}{\svgwidth}%
  \fi%
  \global\let\svgwidth\undefined%
  \global\let\svgscale\undefined%
  \makeatother%
  \begin{picture}(1,1.61712173)%
    \put(0,0){\includegraphics[width=\unitlength]{blob.pdf}}%
  \end{picture}%
\endgroup%
\end{aligned}$ is simply the identity. 

Through chapters \S\ref{ch:spin} to \S\ref{ch:spin_defects} only surfaces without boundary are considered and so the spherical condition \eqref{eq:spherical} is not needed. However, the status of the spherical condition is addressed in a more general framework in \S\ref{sec:cat}.

\begin{definition} A state sum model for a triangulation of $\Sigma_0$ is said to be spherical if it is determined by the data of a planar state sum model.
\end{definition}
The partition function of a sphere can be calculated from any triangulation. The result
\begin{align} \label{eq:diag_sphere_inv}
\gls{Z}(\Sigma_0)=\gls{R}\,\gls{eps}(\gls{unit})=\begin{cases}R\Tr(\gls{x}) &(\gls{k}=\Cb)\\R\Real\Tr(x)&(k=\Rb)\end{cases}
\end{align}
follows from the classification given by theorem~\ref{theo:diagram_semi_simple}. For $k=\Cb$, this result can also be written as $Z(\Sigma_0)=N\Tr(x)/\Tr(x^{-1})$. 

%Consider the recurrent example of the group algebra $\Cb H$ with $H$-valued complex functions $f$ as elements. The Frobenius form, not necessarily symmetric, reads $\varepsilon(f)=f(x)$ for some fixed $x \in H$.
%The invariant created reads
%\begin{align}
%Z(\Sigma_0)=R \sum_{i \in I} \chi_i(x^{-1})
%\end{align}
%where $\chi_i$ is the character associated with the irreducible representation $i \in I$.

%%%%%%%%%%%%%%%%%%%%%%%%%%%%%%%%%%%%%%%%%%%%%%%%%%%%%%%
%%%%%%%%%%%%%%%%%%%%%%%%%%%%%%%%%%%%%%%%%%%%%%%%%%%%%%%
\chapter{Spin state sum models}\label{ch:spin}
%%%%%%%%%%%%%%%%%%%%%%%%%%%%%%%%%%%%%%%%%%%%%%%%%%%%%%%
%%%%%%%%%%%%%%%%%%%%%%%%%%%%%%%%%%%%%%%%%%%%%%%%%%%%%%% 

%%%%%%%%%%%%%%%%%%%%%%%%%%%%%%%%%%%%%%%%%%%%%%%%%%%%%%%
%%%%%%%%%%%%%%%%%%%%%%%%%%%%%%%%%%%%%%%%%%%%%%%%%%%%%%%
\section{Spin geometry of Riemann surfaces} \label{sec:topo}
%%%%%%%%%%%%%%%%%%%%%%%%%%%%%%%%%%%%%%%%%%%%%%%%%%%%%%%
%%%%%%%%%%%%%%%%%%%%%%%%%%%%%%%%%%%%%%%%%%%%%%%%%%%%%%% 
Thus far we have been interested in discussing models concerned only with purely topological properties of a surface. The main results of this work, however, encompass a generalisation of such models concerned with further properties of two-dimensional manifolds: their spin geometry. 

In physics, the notion of spin comes associated with the construction of the spin bundle. We start from a more primitive concept, a spin structure, that is defined using the notion of DeRham cohomology \cite{MDG}. 

Consider a real surface $M$ and its space of $r$-forms, $\Omega^r(M)$, $r=0,1,2$. Using the exterior derivative $d \colon \Omega^r(M) \to \Omega^{r+1}(M)$ we can define two special classes of $r$-forms. An $r$-form $\omega$ is closed if $d\omega=0$. The set of all such forms is denoted $\mathcal{Z}^r(M)$. A special subset of $\mathcal{Z}^r(M)$ is formed by exact $r$-forms, $r$-forms $\omega$ that can be written as $\omega=d\omega'$ for some $\omega' \in \Omega^{r-1}(M)$ -- we can see they are closed because $d^2=0$. The set of exact $r$-forms is denoted $\mathcal{B}^{r}(M)$.
\begin{definition}
The DeRham cohomology spaces $H^r(M)$, $r=0,1,2$ of a real surface $M$ are the quotient vector spaces
\begin{align}
H^r(M)=\mathcal{Z}^r(M)/\mathcal{B}^r(M).
\end{align}
\end{definition}
Above, the vector spaces are $\Rb$-valued. It is an easy exercise to see $H^0(M) = \Rb$ if $M$ is connected, and that $H^{2}(M)=\Rb$. 

By Poincar\'{e} duality \cite{Nakahara}, the vector space $H^{1}(M)$ can also be regarded as the linear space $\text{Hom}(H_1(M),\Rb)$. This space $H_1(M)$ is generated by closed curves $c \in M$ under the following equivalence relation: two curves are regarded as equivalent if they differ by a boundary, a curve $c'$ that satisfies $c'=\partial N$ for some two-dimensional $N \subset M$.     

We are now ready to present the notion of a spin structure. The following definition is recovered from \cite{Cimasoni}. A $\gls{k}$-valued cohomology is denoted as $H^r(M,k)$.

\begin{definition} \label{def:spin_structure}
Consider a surface $\Sigma$ and let $P_{SO} \to \Sigma$ be the principal $SO(2)$-bundle associated to its tangent bundle -- the frame bundle. A spin structure on $\gls{Sigma}$ is a cohomology class $\gls{s} \in H^1(P_{SO},\Zb_2)$ whose restriction to each fibre $F$ gives the generator of the cyclic group $H^1(F,\Zb_2)$. 
\end{definition} 

Apart from definition \ref{def:spin_structure} there are several other equivalent ways of describing spin structures in two dimensions, the most common of them being definition \ref{def:spin-intro}. In particular, we will rely on the relation between spin structures, quadratic forms and immersion maps all of which are explored throughout this section.

Let $c$ be a framed and closed curve in $\Sigma$ -- in other words, an element $c \in H_1(P_{SO})$. If $c$ is a boundary in $P_{SO}$, i.e. if it bounds a framed disk in $\Sigma$, then according to definition \ref{def:spin_structure} we must have $s(c)=0$. On the other hand, consider $c$ to be a tangentially-framed circle. This curve is not a boundary since this framing cannot be extended to the disk bounded by $c$ \cite{Kallel}. A spin structure maps the curve $c$ to $s(c)$, in this case also the generator for the space $H^1(F,\Zb_2) \simeq \Zb_2$. Therefore, we must have $s(c)=1$.  

%%%%%%%%%%%%%%%%%%%%%%%%%%%%%%%%%%%%%%%%%%%%%%%%%%%%%%%%%%%%%%%%%%%%%%
\begin{figure}
\centering
\begin{subfigure}[t!]{0.4\textwidth}
                \centering
		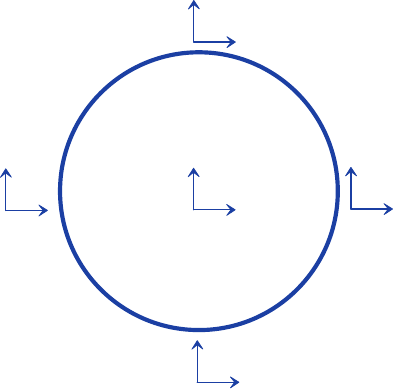
		\caption[A boundary]{If $c$ is a boundary, $s(c)=0$.}
		\label{fig:rwn}
\end{subfigure}
\hspace{5mm}
\begin{subfigure}[t!]{0.4\textwidth}
                \centering
		\vspace{4mm}
		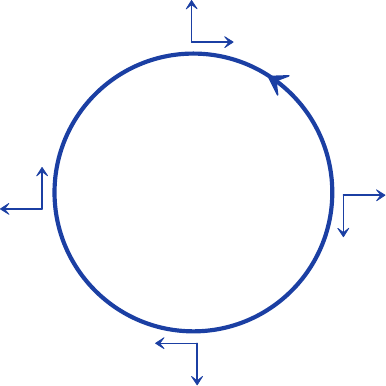
		\caption[A tangentially-framed curve]{If $c$ bounds a disk but is tangentially-framed, $s(c)=1$.}
		\label{fig:rwn2}
\end{subfigure}
\caption[Framed embedded circles in $\Sigma$]{Framed embedded circles in $\Sigma$.}
\label{fig:cycles}
\end{figure} 
%%%%%%%%%%%%%%%%%%%%%%%%%%%%%%%%%%%%%%%%%%%%%%%%%%%%%%%%%%%%%%%%%%%%%%

Consider the closed curves below that we treat as tangentially-framed. 
%%%%%%%%%%%%%%%%%%%%%%%%%%%%%%%%%%%%%%%%%%%%%%%%%%%%%%%%%%%%%%%%%%%%%%
$$
\centering
%% Creator: Inkscape 0.48.2, www.inkscape.org
%% PDF/EPS/PS + LaTeX output extension by Johan Engelen, 2010
%% Accompanies image file 'immersed_circle.pdf' (pdf, eps, ps)
%%
%% To include the image in your LaTeX document, write
%%   \input{<filename>.pdf_tex}
%%  instead of
%%   \includegraphics{<filename>.pdf}
%% To scale the image, write
%%   \def\svgwidth{<desired width>}
%%   \input{<filename>.pdf_tex}
%%  instead of
%%   \includegraphics[width=<desired width>]{<filename>.pdf}
%%
%% Images with a different path to the parent latex file can
%% be accessed with the `import' package (which may need to be
%% installed) using
%%   \usepackage{import}
%% in the preamble, and then including the image with
%%   \import{<path to file>}{<filename>.pdf_tex}
%% Alternatively, one can specify
%%   \graphicspath{{<path to file>/}}
%% 
%% For more information, please see info/svg-inkscape on CTAN:
%%   http://tug.ctan.org/tex-archive/info/svg-inkscape
%%
\begingroup%
  \makeatletter%
  \providecommand\color[2][]{%
    \errmessage{(Inkscape) Color is used for the text in Inkscape, but the package 'color.sty' is not loaded}%
    \renewcommand\color[2][]{}%
  }%
  \providecommand\transparent[1]{%
    \errmessage{(Inkscape) Transparency is used (non-zero) for the text in Inkscape, but the package 'transparent.sty' is not loaded}%
    \renewcommand\transparent[1]{}%
  }%
  \providecommand\rotatebox[2]{#2}%
  \ifx\svgwidth\undefined%
    \setlength{\unitlength}{113.2bp}%
    \ifx\svgscale\undefined%
      \relax%
    \else%
      \setlength{\unitlength}{\unitlength * \real{\svgscale}}%
    \fi%
  \else%
    \setlength{\unitlength}{\svgwidth}%
  \fi%
  \global\let\svgwidth\undefined%
  \global\let\svgscale\undefined%
  \makeatother%
  \begin{picture}(1,0.33083957)%
    \put(0,0){\includegraphics[width=\unitlength]{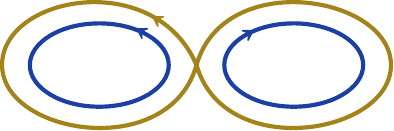}}%
  \end{picture}%
\endgroup%

$$
%%%%%%%%%%%%%%%%%%%%%%%%%%%%%%%%%%%%%%%%%%%%%%%%%%%%%%%%%%%%%%%%%%%%%%
On one hand we have an immersed circle with one intersection point; on the other we have a curve consisting of two disconnected circles. One cannot distinguish between one curve and the other as elements of $H_1(P_{SO},\Zb_2)$. In this sense, we can regard any curve as a collection of disconnected circles -- a convention we adopt for the remainder of this section and that extends to curves in $H_1(\Sigma,\Zb_2)$.

For a surface of genus $g$, $\gls{Sigma_g}$, any curve can then be seen as generated by one of $2g+1$ circles where $2g$ of these correspond to the surface generating loops, and the remaining is the tangentially-framed circle that bounds a disk that we denote as $z$. (Note that the framings of the generating loops must be tangential.) We know $\gls{s}(z)=1$ but the value of the spin structure at each of the generating loops is not a priori determined. Therefore, there are $2^{2g}$ distinct spin structures. 

%%%%%%%%%%%%%%%%%%%%%%%%%%%%%%%%%%%%%%%%%%%%%%%%%%%%%%%%%%%%%%%%%%%%%%
%\begin{figure}
%\centering
%\begin{subfigure}[t!]{0.47\textwidth}
%                \centering
%		\input{0_frontmatter/figures/distinct_spin.pdf_tex}
%		\vspace{2mm}
%		\caption[Normal $f$]{A vector field of $f$ for which $s_f(c)=1$ when $c$ is a torus generating curve.}
%		\label{fig:distinct_spin}
%\end{subfigure}
%\hspace{5mm}
%\begin{subfigure}[t!]{0.47\textwidth}
%                \centering
%		\input{0_frontmatter/figures/distinct_spin2.pdf_tex}
%		\vspace{2mm}
%		\caption[Tangential $f$]{A vector field $f$ for which $s_f(c)=0$ when $c$ is a torus generating curve.}
%		\label{fig:distinct_spin2}
%\end{subfigure}
%\caption[Torus spin structures as vector fields]{}
%\label{fig:spin_torus_f}
%\end{figure} 
%%%%%%%%%%%%%%%%%%%%%%%%%%%%%%%%%%%%%%%%%%%%%%%%%%%%%%%%%%%%%%%%%%%%%%

%The description above and definition~\ref{def:spin_structure} rely on the notion of a tangent bundle. It is, however, possible to describe a spin structure $s$ purely through skein relations on a vector space generated by curves on $\gls{Sigma}$ \cite{JBspin}. This description does not require the use of the tangent bundle and so generalises to piecewise-linear manifolds.

Instead of working with $s \in H^{1}(P_{SO},\Zb_2)$ it is possible to use an equivalent description based on quadratic forms \cite{Kirby}. The following definition is recovered from \cite{Cimasoni}.

\begin{definition}
Let $V$ be a finite dimensional vector space over the field $\Zb_2$, and let $b \colon V \times V \to \Zb_2$ be a fixed bilinear form. A function $q \colon V \to \Zb_2$ is a quadratic form on $(V,b)$ if
\begin{align} \label{eq:quad_b}
q(x+y)=q(x)+q(y)+b(x,y)
\end{align}
for all $x,y \in V$.
\end{definition}

Johnson \cite{Johnson} showed the space of spin structures $s \in H^{1}(P_{SO},\Zb_2)$ is isomorphic to the space of quadratic forms $q \in H^1(\Sigma,\Zb_2)$. This result relies on understanding how curves in $\Sigma$ can be lifted to curves in its frame bundle $P_{SO}$. 

Suppose $c_1,c_2 \in H_1(\Sigma,\Zb_2)$ are a disjoint union of $m$ and $n$ cycles: $c_1=\sum_{i=1}^m \alpha_i$ and $c_2=\sum_{i=1}^n \beta_i$. Let $c_1 \sim c_2$; then, the equivalence relation extends to the lifts of $c_1, c_2$ to $H_1(P_{SO},\Zb_2)$, denoted as $\tilde{c_1}, \tilde{c}_2$, in the following way:
\begin{align} \label{eq:lift}
\sum_{i=1}^m \tilde{\alpha}_i+m \,z \sim \sum_{i=1}^n \tilde{\beta}_i + n\,z \, ,
\end{align}
where $z$ denotes the tangentially framed circle and the $\tilde{\alpha}_i,\tilde{\beta}_i$ are also tangentially-framed. Expression \eqref{eq:lift} is intuitive when the number of disjoint cycles is the same for $c_1$ and $c_2$: the equivalence transformations between such curves lift straightforwardly from $\Sigma$ to $P_{SO}$ and we naturally have $\sum_{i=1}^m \tilde{\alpha}_i \sim \sum_{i=1}^n \tilde{\beta}_i$. In other words, for $m=n$ the boundary $c_1 - c_2 \in \Sigma$ lifts to a boundary in $P_{SO}$. However, the transformations that allow us to regard curves $c_1,c_2$ for which $m \neq n$ as equivalent, behave non-trivially under the lift. This type of transformation is shown below for a specific choice of $c_1,c_2$ with $n=m+1$ (the orientation of the curves represent their tangential frame).  
\begin{align}
\begin{aligned}
\hspace{1.5cm}
%% Creator: Inkscape 0.48.2, www.inkscape.org
%% PDF/EPS/PS + LaTeX output extension by Johan Engelen, 2010
%% Accompanies image file 'lift.pdf' (pdf, eps, ps)
%%
%% To include the image in your LaTeX document, write
%%   \input{<filename>.pdf_tex}
%%  instead of
%%   \includegraphics{<filename>.pdf}
%% To scale the image, write
%%   \def\svgwidth{<desired width>}
%%   \input{<filename>.pdf_tex}
%%  instead of
%%   \includegraphics[width=<desired width>]{<filename>.pdf}
%%
%% Images with a different path to the parent latex file can
%% be accessed with the `import' package (which may need to be
%% installed) using
%%   \usepackage{import}
%% in the preamble, and then including the image with
%%   \import{<path to file>}{<filename>.pdf_tex}
%% Alternatively, one can specify
%%   \graphicspath{{<path to file>/}}
%% 
%% For more information, please see info/svg-inkscape on CTAN:
%%   http://tug.ctan.org/tex-archive/info/svg-inkscape
%%
\begingroup%
  \makeatletter%
  \providecommand\color[2][]{%
    \errmessage{(Inkscape) Color is used for the text in Inkscape, but the package 'color.sty' is not loaded}%
    \renewcommand\color[2][]{}%
  }%
  \providecommand\transparent[1]{%
    \errmessage{(Inkscape) Transparency is used (non-zero) for the text in Inkscape, but the package 'transparent.sty' is not loaded}%
    \renewcommand\transparent[1]{}%
  }%
  \providecommand\rotatebox[2]{#2}%
  \ifx\svgwidth\undefined%
    \setlength{\unitlength}{412.85774246bp}%
    \ifx\svgscale\undefined%
      \relax%
    \else%
      \setlength{\unitlength}{\unitlength * \real{\svgscale}}%
    \fi%
  \else%
    \setlength{\unitlength}{\svgwidth}%
  \fi%
  \global\let\svgwidth\undefined%
  \global\let\svgscale\undefined%
  \makeatother%
  \begin{picture}(1,0.22060416)%
    \put(0,0){\includegraphics[width=\unitlength]{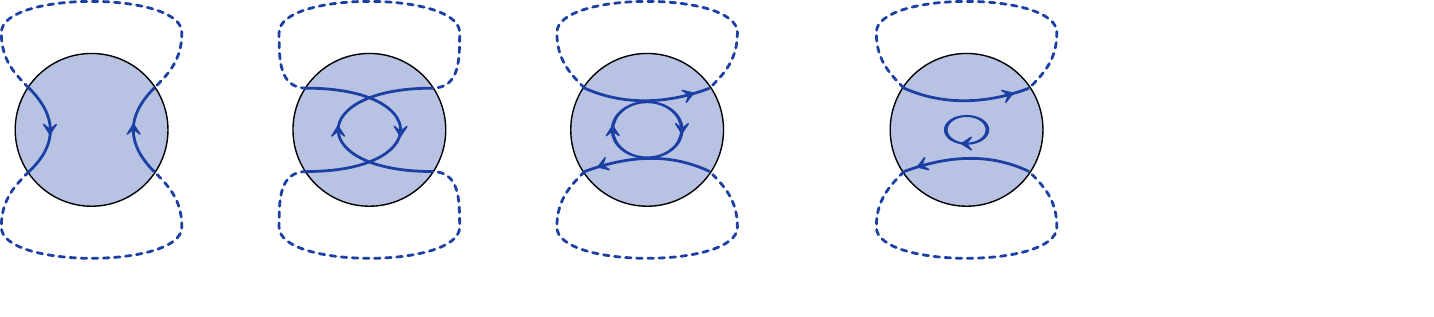}}%
    \put(0.02997954,0.0040779){\color[rgb]{0,0,0}\makebox(0,0)[lb]{\smash{$\tilde{c}_1=\tilde{\alpha}_1$}}}%
    \put(0.59191641,0.0040779){\color[rgb]{0,0,0}\makebox(0,0)[lb]{\smash{$\tilde{c}_2=\tilde{\beta}_1 + \tilde{\beta}_2 + z$}}}%
    \put(0.14624235,0.12034071){\color[rgb]{0,0,0}\makebox(0,0)[lb]{\smash{$\sim$}}}%
    \put(0.34001366,0.12034071){\color[rgb]{0,0,0}\makebox(0,0)[lb]{\smash{$\sim$}}}%
    \put(0.56285071,0.12034071){\color[rgb]{0,0,0}\makebox(0,0)[lb]{\smash{$\sim$}}}%
  \end{picture}%
\endgroup%

\end{aligned}
\end{align}
As we can see the identity $\tilde{\alpha}_1 + z = \tilde{\beta}_1 + \tilde{\beta}_2 + 2z$ holds since the coefficients are taken $\mo 2$. Crucially in this case the boundary $c_1 - c_2 \in \Sigma$ lifts to $z \in P_{SO}$ -- not to a boundary. It is shown in \cite{Johnson} that iterating this type of equivalence we can link $\tilde{c}_1$ and $\tilde{c}_2$ when $m \neq n$. 

Let $c$ be an embedded curve composed of $m$ cycles. Given a spin structure $s$, it is natural to define a map $q_{s} \colon H_1(\Sigma,\Zb_2) \to \Zb_2$ as $q_s(c)=s(\tilde{c})+m \mo 2$ -- the equivalence relation \eqref{eq:lift} determines $q_s$ is well-defined in $H^1(\Sigma,\Zb_2)$. Then, a theorem due to Johnson \cite{Johnson} asserts that $q_{s}$ is a well-defined quadratic form satisfying \eqref{eq:quad_b} for $b(x,y)=x \cdot y$, the intersection form. Moreover, $q_s$ and $s$ are in one-to-one correspondence.

To classify $q_s \in H^1(\Sigma,\Zb_2)$ we will use the results of Pinkall, where distinct spin structures are associated with equivalence classes of immersions \cite{Pinkall}. Recall that an immersion is a map $i \colon M \to M'$ having a derivative that is injective at every point. Thus an immersion is locally an embedding. The equivalence relation of interest is regular homotopy \cite{Kallel}.

\begin{definition}
A regular homotopy from $i_0$ to $i_1$ is a family of smooth immersions $i_t$, $t\in[0,1]$, that defines a smooth map $H(x,t)=i_t(x)\colon M\times[0,1]\to M'$.  
\end{definition}

It is shown in \cite{Pinkall} that the immersions of a smooth surface $\gls{Sigma_g}$ into $\Rb^3$ fall into $2^{2g}$ regular homotopy equivalence classes. Each immersion $i\colon\Sigma\to\Rb^3$ determines an induced spin structure on $\gls{Sigma}$ by pulling-back the unique spin structure on $\Rb^3$. The induced spin structure is invariant under a regular homotopy (since the homotopy is differentiable). As we have seen, there are $2^{2g}$ spin structures on an oriented surface $\Sigma_g$ and these classify the equivalence classes of immersions uniquely. 

It is important to note that the notion of spin structure as defined through \ref{def:spin_structure} is not preserved by the action of diffeomorphisms. In other words, even if there exists no regular homotopy $i(\Sigma) \to i(\Sigma')$ there might exist a diffeomorphism $f \colon \Sigma \to \Sigma$ such that $i \circ f (\Sigma)$ and $i(\Sigma')$ are regularly homotopic \cite{Pinkall}. It will be appropriate therefore to work with a structure that is differomorphism-invariant: the Arf invariant.  

To any $\Zb_2$-valued quadratic form $q_s$ we can associate the Arf invariant, here defined as the integer \cite{Cimasoni}
\begin{align} \label{eq:Arf_inv}
\text{Arf}(q_{s})=\frac{1}{|H_1(\Sigma_g,\Zb_2)|^{\frac{1}{2}}}\sum_{\alpha \in H_1(\Sigma_g,\Zb_2)}(-1)^{q_s(\alpha)}.
\end{align}
Conveniently, this form of the Arf invariant takes the values $\pm 1$ and will therefore be referred to as well as the parity of a spin structure: $P(\gls{s})=\text{Arf}(q_{s})$. To show that indeed $\text{Arf}(q_s)=\pm 1$ one would show $\text{Arf}^2(q_s)=1$ by using the properties of a quadratic form \cite[Lecture~9]{Saveliev}. If $P(s)=1$ the spin structure is said to be even, otherwise it is referred to as odd. There is an alternative version of the Arf invariant that takes values $0,1$ and that we will denote as $\text{arf}(q_s)$: $\text{Arf}(q_s)=(-1)^{\text{arf}(q_s)}$. 

The description of spin structures we have just made will be used to understand how the planar models of chaper~\S\ref{sec:diagram} can be extended from $M \subset \Rb^2$ to $\Sigma \subset \Rb^3$. In particular, the role of regular homotopy and its interplay with diffeomorphism-invariance are explored.

%%%%%%%%%%%%%%%%%%%%%%%%%%%%%%%%%%%%%%%%%%%%%%%%%%%%%%%
%%%%%%%%%%%%%%%%%%%%%%%%%%%%%%%%%%%%%%%%%%%%%%%%%%%%%%%
\section{Models with crossings} \label{sec:crossing}
%%%%%%%%%%%%%%%%%%%%%%%%%%%%%%%%%%%%%%%%%%%%%%%%%%%%%%%
%%%%%%%%%%%%%%%%%%%%%%%%%%%%%%%%%%%%%%%%%%%%%%%%%%%%%%% 

As with the spherical case, we face a challenge on trying to extend the planar calculus formalism to closed surfaces $\Sigma$. To accomplish this generalisation we will regard surfaces as subsets of $\Rb^3$. The maps used to perform such an inclusion will be immersions $\Sigma \looparrowright \Rb^3$. 

The diagrammatic method is extended to these surfaces in the following way. The graph $\gls{G}$ constructed from the dual of a triangulation $\gls{T}$ of an oriented surface $\gls{Sigma}$ can be considered as a ribbon graph by taking the ribbon to be a suitable neighbourhood of the graph (called a regular neighbourhood \cite{hirsch-srn}) in the surface. This ribbon graph is therefore immersed in $\Rb^3$ (see figure~\ref{fig:torus-embedding}). The state sum model partition function is given by a suitable evaluation of this ribbon graph. Recall that for planar models the evaluation $|G|$ was subject to an equivalence relation: invariance under Pachner moves. This notion of equivalence is extended to ribbon graphs by assuming that two graphs must be regarded as equivalent if they are related by regular homotopy.  

%%%%%%%%%%%%%%%%%%%%%%%%%%%%%%%%%%%%%%%%%%%%%%%%%%%%%
\begin{figure}[t!]
                \centering		
		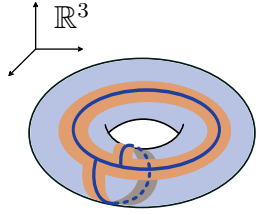
		\caption[Torus embedding]{\emph{Torus embedding.} An example of an immersion is an embedding. Above we have one such map $\Sigma_1 \to \Rb^3$ where a regular neighbourhood of a possible dual diagram $G$ has been included to generate a ribbon graph. The part of the graph that lies in the hidden side of the torus has been dashed whilst the corresponding regular neighbourhood has been shaded.}
		\label{fig:torus-embedding}
\end{figure} 
%%%%%%%%%%%%%%%%%%%%%%%%%%%%%%%%%%%%%%%%%%%%%%%%%%%%%%

The concept of regular homotopy is explored here for the case of smooth surfaces and immersions, for which there is a well-developed literature. As is standard in knot theory, the graphs can be described by the diagrams that result from a projection of $\Rb^3$ to $\Rb^2$ and the equivalence is a set of Reidemeister-like moves on diagrams. Then it is noted that the diagrams and their moves in fact also make sense as piecewise-linear diagrams, which is  more natural for triangulations. A future challenge is to develop the theory using the piecewise-linear formulation of regular homotopy \cite{HaefligerPoenaru} from the beginning. 

Surfaces and curves immersed in $\Rb^3$ are studied in \cite{Pinkall}, from which several key results are used. Let $i\colon\Sigma\to\Rb^3$ be a surface immersion and $G\subset\Sigma$ the graph constructed from the dual to a triangulation of $\Sigma$. Then $\gamma=i_{|G}\colon G\to\Rb^3$ is an immersion of the graph $G$ and in the generic case this is an embedding -- this means there is an arbitrarily small regular homotopy to an embedding.  If there is a regular homotopy $\gamma_t$ between two embedded graphs $\gamma_0$ and $\gamma_1$, then the regular homotopy can be adjusted so that $\gamma_t$ is an embedding except at a finite set of values of $t$, where there is one intersection point. As $t$ varies through one of these values, one segment of an edge of the graph passes through another (see figure \ref{fig:regular-homotopy}).

%%%%%%%%%%%%%%%%%%%%%%%%%%%%%%%%%%%%%%%%%%%%%%%%%%%%%
\begin{figure}[t!]
                	\centering		
		\includegraphics{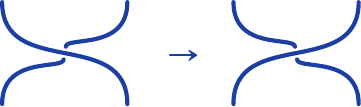}
		\caption[Regular homotopy]{\emph{Regular homotopy.} Immersions of graphs in $\Rb^3$ allow for intersections. Regular homotopy thus allows a diagram undercrossing to be transformed into an overcrossing.}
		\label{fig:regular-homotopy}
\end{figure} 
%%%%%%%%%%%%%%%%%%%%%%%%%%%%%%%%%%%%%%%%%%%%%%%%%%%%%%

The graph $\gamma$ will be described by a diagram obtained by projecting $\Rb^3$ to $\Rb^2$. It is assumed that this projection is generic, so that the graph is immersed in $\Rb^2$ with transverse self-intersections of edges. Since regular homotopy allows the edges to pass through each other, there is no need to record whether the crossings are over- or undercrossings. Diagrams are thus obtained from the usual diagrams of knot theory by setting over- and undercrossings equal.

The graph $\gamma$ has a ribbon structure obtained by taking a suitably small regular neighbourhood $K$ of $\gamma$ in $\Sigma$, thus $\gamma\subset K\subset \Sigma$.
The formalism is simplified if the projection to $\Rb^2$ preserves the ribbon structure of the graph. As is standard in knot theory \cite{kauffman-regular-isotopy}, an embedded ribbon graph can be adjusted by a regular homotopy so that the projection of the ribbon to $\Rb^2$ is an orientation-preserving immersion. This is called `blackboard framing'. Then using blackboard-framed knots throughout, it is not necessary to include the ribbon in the planar diagrams (see figure~\ref{fig:torus-diagrams}). 

The state sum model is defined from the diagram in the plane by augmenting the formalism for a planar state sum with a crossing map $\gls{lambda} \colon \gls{A} \otimes A \to A \otimes A$ where one edge of the graph crosses another as shown.  
%%%%%%%%%%%%%%%%%%%%%%%%%%%%%%%%%%%%%%%%%%%%%%%%%%%%%
\begin{align}
\begin{aligned}
%% Creator: Inkscape 0.48.2, www.inkscape.org
%% PDF/EPS/PS + LaTeX output extension by Johan Engelen, 2010
%% Accompanies image file 'crossing.pdf' (pdf, eps, ps)
%%
%% To include the image in your LaTeX document, write
%%   \input{<filename>.pdf_tex}
%%  instead of
%%   \includegraphics{<filename>.pdf}
%% To scale the image, write
%%   \def\svgwidth{<desired width>}
%%   \input{<filename>.pdf_tex}
%%  instead of
%%   \includegraphics[width=<desired width>]{<filename>.pdf}
%%
%% Images with a different path to the parent latex file can
%% be accessed with the `import' package (which may need to be
%% installed) using
%%   \usepackage{import}
%% in the preamble, and then including the image with
%%   \import{<path to file>}{<filename>.pdf_tex}
%% Alternatively, one can specify
%%   \graphicspath{{<path to file>/}}
%% 
%% For more information, please see info/svg-inkscape on CTAN:
%%   http://tug.ctan.org/tex-archive/info/svg-inkscape
%%
\begingroup%
  \makeatletter%
  \providecommand\color[2][]{%
    \errmessage{(Inkscape) Color is used for the text in Inkscape, but the package 'color.sty' is not loaded}%
    \renewcommand\color[2][]{}%
  }%
  \providecommand\transparent[1]{%
    \errmessage{(Inkscape) Transparency is used (non-zero) for the text in Inkscape, but the package 'transparent.sty' is not loaded}%
    \renewcommand\transparent[1]{}%
  }%
  \providecommand\rotatebox[2]{#2}%
  \ifx\svgwidth\undefined%
    \setlength{\unitlength}{106.83815918bp}%
    \ifx\svgscale\undefined%
      \relax%
    \else%
      \setlength{\unitlength}{\unitlength * \real{\svgscale}}%
    \fi%
  \else%
    \setlength{\unitlength}{\svgwidth}%
  \fi%
  \global\let\svgwidth\undefined%
  \global\let\svgscale\undefined%
  \makeatother%
  \begin{picture}(1,0.53023692)%
    \put(0,0){\includegraphics[width=\unitlength]{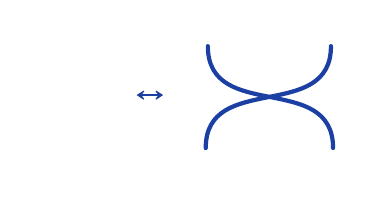}}%
    \put(-0.00266907,0.25241022){\color[rgb]{0,0,0}\makebox(0,0)[lb]{\smash{$\lambda_{ab}{}^{cd}$}}}%
    \put(0.53352247,0.00771467){\color[rgb]{0,0,0}\makebox(0,0)[lb]{\smash{$a$}}}%
    \put(0.87850342,0.00905148){\color[rgb]{0,0,0}\makebox(0,0)[lb]{\smash{$b$}}}%
    \put(0.53753359,0.47170068){\color[rgb]{0,0,0}\makebox(0,0)[lb]{\smash{$c$}}}%
    \put(0.86914347,0.47036364){\color[rgb]{0,0,0}\makebox(0,0)[lb]{\smash{$d$}}}%
  \end{picture}%
\endgroup%

\end{aligned}
\end{align}
%%%%%%%%%%%%%%%%%%%%%%%%%%%%%%%%%%%%%%%%%%%%%%%%%%%%%%

An example of a planar diagram for a torus to which a disk has been removed, triangulated using two triangles, is shown in figure \ref{fig:torus-diagrams}. The middle diagram shows a projection of the graph that is blackboard-framed and the final diagram reflects this.

%%%%%%%%%%%%%%%%%%%%%%%%%%%%%%%%%%%%%%%%%%%%%%%%%%%%%
\begin{figure}[t!]
                \centering		
		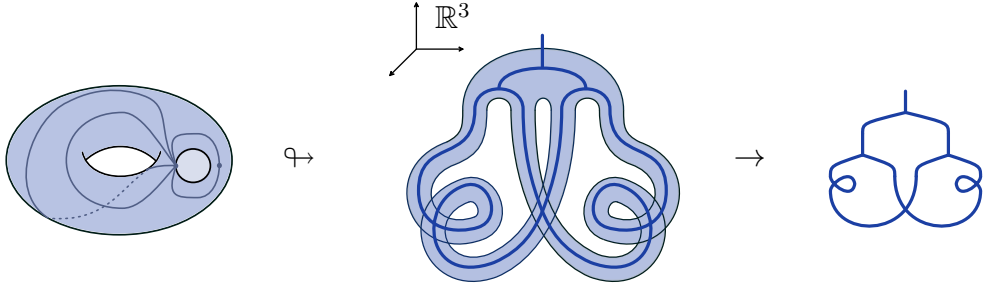
		\vspace{2mm}
		\caption[An immersion $\Sigma_1-D \looparrowright \Rb^3$]{\emph{An immersion $\Sigma_1-D \looparrowright \Rb^3$.} The surface $\Sigma_1-D$ is depicted together with a possible (degenerate) triangulation. A diagrammatic state sum model for an immersed $\Sigma_1-D$ is then created from the dual diagram in $\Rb^2$. The projection $\Rb^3 \to \Rb^2$ gives rise to a blackboard framing: the ribbon information need not be recorded.}\label{fig:torus-diagrams}
\end{figure} 
%%%%%%%%%%%%%%%%%%%%%%%%%%%%%%%%%%%%%%%%%%%%%%%%%%%%%%

The ribbon structure is preserved under the equivalence relation of regular homotopy. The usual Reidemeister moves for knots do not preserve the ribbon structure, so one has to use a modified set of moves for ribbon knots, described in \cite{kauffman-regular-isotopy,freyd-yetter}. The moves for graphs are described in \cite{Yetter,Kauffman-handbook} and the extension from ribbon knots to ribbon graphs is described in \cite{Reshetikhin-Turaev}.
A planar state sum model that is invariant under these moves is called a spin state sum model. 

A diagram with $n$ downward- and $m$ upward-pointing legs defines a map $\otimes^n \gls{A} \to \otimes^m A$ where $\otimes^0 A= \gls{k}$. Therefore, diagrams should be read bottom-to-top and the use of explicit indices has been dropped.

\begin{definition} \label{def:spin-model} 
A spin state sum model is determined by planar state sum model data $(\gls{Ct},\gls{Bb},\gls{R})$ together with a crossing map $\gls{lambda}$. The additional axioms the map $\lambda$ obeys are
 \begin{enumerate}[label=B\arabic{*}), ref=(B\arabic{*})]
\setcounter{enumi}{0}
\item \label{fig:axiom_form} compatibility with $B$, \hskip3.35cm 
%%%%%%%%%%%%%%%%%%%%%%%%%%%%%%%%%%%%%%%%%%%%%%%%%%%%
		$\vcenter{\hbox{%% Creator: Inkscape 0.48.2, www.inkscape.org
%% PDF/EPS/PS + LaTeX output extension by Johan Engelen, 2010
%% Accompanies image file 'axiom_form.pdf' (pdf, eps, ps)
%%
%% To include the image in your LaTeX document, write
%%   \input{<filename>.pdf_tex}
%%  instead of
%%   \includegraphics{<filename>.pdf}
%% To scale the image, write
%%   \def\svgwidth{<desired width>}
%%   \input{<filename>.pdf_tex}
%%  instead of
%%   \includegraphics[width=<desired width>]{<filename>.pdf}
%%
%% Images with a different path to the parent latex file can
%% be accessed with the `import' package (which may need to be
%% installed) using
%%   \usepackage{import}
%% in the preamble, and then including the image with
%%   \import{<path to file>}{<filename>.pdf_tex}
%% Alternatively, one can specify
%%   \graphicspath{{<path to file>/}}
%% 
%% For more information, please see info/svg-inkscape on CTAN:
%%   http://tug.ctan.org/tex-archive/info/svg-inkscape
%%
\begingroup%
  \makeatletter%
  \providecommand\color[2][]{%
    \errmessage{(Inkscape) Color is used for the text in Inkscape, but the package 'color.sty' is not loaded}%
    \renewcommand\color[2][]{}%
  }%
  \providecommand\transparent[1]{%
    \errmessage{(Inkscape) Transparency is used (non-zero) for the text in Inkscape, but the package 'transparent.sty' is not loaded}%
    \renewcommand\transparent[1]{}%
  }%
  \providecommand\rotatebox[2]{#2}%
  \ifx\svgwidth\undefined%
    \setlength{\unitlength}{120.86936035bp}%
    \ifx\svgscale\undefined%
      \relax%
    \else%
      \setlength{\unitlength}{\unitlength * \real{\svgscale}}%
    \fi%
  \else%
    \setlength{\unitlength}{\svgwidth}%
  \fi%
  \global\let\svgwidth\undefined%
  \global\let\svgscale\undefined%
  \makeatother%
  \begin{picture}(1,0.21629127)%
    \put(0,0){\includegraphics[width=\unitlength]{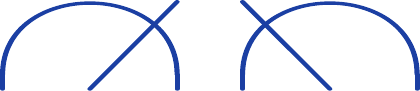}}%
    \put(0.46480483,0.06210709){\color[rgb]{0,0,0}\makebox(0,0)[lb]{\smash{$=$}}}%
  \end{picture}%
\endgroup%
}}$
%%%%%%%%%%%%%%%%%%%%%%%%%%%%%%%%%%%%%%%%%%%%%%%%%%%%%%%
\item compatibility with $C$, \hskip3.15cm \label{fig:axiom_mult} 
%%%%%%%%%%%%%%%%%%%%%%%%%%%%%%%%%%%%%%%%%%%%%%%%%%%%%
		$\vcenter{\hbox{ %% Creator: Inkscape 0.48.2, www.inkscape.org
%% PDF/EPS/PS + LaTeX output extension by Johan Engelen, 2010
%% Accompanies image file 'axiom_mult.pdf' (pdf, eps, ps)
%%
%% To include the image in your LaTeX document, write
%%   \input{<filename>.pdf_tex}
%%  instead of
%%   \includegraphics{<filename>.pdf}
%% To scale the image, write
%%   \def\svgwidth{<desired width>}
%%   \input{<filename>.pdf_tex}
%%  instead of
%%   \includegraphics[width=<desired width>]{<filename>.pdf}
%%
%% Images with a different path to the parent latex file can
%% be accessed with the `import' package (which may need to be
%% installed) using
%%   \usepackage{import}
%% in the preamble, and then including the image with
%%   \import{<path to file>}{<filename>.pdf_tex}
%% Alternatively, one can specify
%%   \graphicspath{{<path to file>/}}
%% 
%% For more information, please see info/svg-inkscape on CTAN:
%%   http://tug.ctan.org/tex-archive/info/svg-inkscape
%%
\begingroup%
  \makeatletter%
  \providecommand\color[2][]{%
    \errmessage{(Inkscape) Color is used for the text in Inkscape, but the package 'color.sty' is not loaded}%
    \renewcommand\color[2][]{}%
  }%
  \providecommand\transparent[1]{%
    \errmessage{(Inkscape) Transparency is used (non-zero) for the text in Inkscape, but the package 'transparent.sty' is not loaded}%
    \renewcommand\transparent[1]{}%
  }%
  \providecommand\rotatebox[2]{#2}%
  \ifx\svgwidth\undefined%
    \setlength{\unitlength}{124.33878174bp}%
    \ifx\svgscale\undefined%
      \relax%
    \else%
      \setlength{\unitlength}{\unitlength * \real{\svgscale}}%
    \fi%
  \else%
    \setlength{\unitlength}{\svgwidth}%
  \fi%
  \global\let\svgwidth\undefined%
  \global\let\svgscale\undefined%
  \makeatother%
  \begin{picture}(1,0.3232729)%
    \put(0,0){\includegraphics[width=\unitlength]{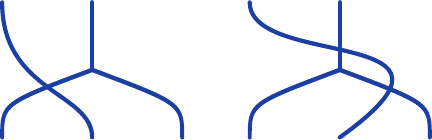}}%
    \put(0.45951805,0.05690025){\color[rgb]{0,0,0}\makebox(0,0)[lb]{\smash{$=$}}}%
  \end{picture}%
\endgroup%
}}$
%%%%%%%%%%%%%%%%%%%%%%%%%%%%%%%%%%%%%%%%%%%%%%%%%%%%%%%
\item the Reidemeister II move ,\hskip3.35cm  \label{fig:axiom_square} 
%%%%%%%%%%%%%%%%%%%%%%%%%%%%%%%%%%%%%%%%%%%%%%%%%%%%%
 		$\vcenter{\hbox{ %% Creator: Inkscape 0.48.2, www.inkscape.org
%% PDF/EPS/PS + LaTeX output extension by Johan Engelen, 2010
%% Accompanies image file 'axiom_square.pdf' (pdf, eps, ps)
%%
%% To include the image in your LaTeX document, write
%%   \input{<filename>.pdf_tex}
%%  instead of
%%   \includegraphics{<filename>.pdf}
%% To scale the image, write
%%   \def\svgwidth{<desired width>}
%%   \input{<filename>.pdf_tex}
%%  instead of
%%   \includegraphics[width=<desired width>]{<filename>.pdf}
%%
%% Images with a different path to the parent latex file can
%% be accessed with the `import' package (which may need to be
%% installed) using
%%   \usepackage{import}
%% in the preamble, and then including the image with
%%   \import{<path to file>}{<filename>.pdf_tex}
%% Alternatively, one can specify
%%   \graphicspath{{<path to file>/}}
%% 
%% For more information, please see info/svg-inkscape on CTAN:
%%   http://tug.ctan.org/tex-archive/info/svg-inkscape
%%
\begingroup%
  \makeatletter%
  \providecommand\color[2][]{%
    \errmessage{(Inkscape) Color is used for the text in Inkscape, but the package 'color.sty' is not loaded}%
    \renewcommand\color[2][]{}%
  }%
  \providecommand\transparent[1]{%
    \errmessage{(Inkscape) Transparency is used (non-zero) for the text in Inkscape, but the package 'transparent.sty' is not loaded}%
    \renewcommand\transparent[1]{}%
  }%
  \providecommand\rotatebox[2]{#2}%
  \ifx\svgwidth\undefined%
    \setlength{\unitlength}{75.75253296bp}%
    \ifx\svgscale\undefined%
      \relax%
    \else%
      \setlength{\unitlength}{\unitlength * \real{\svgscale}}%
    \fi%
  \else%
    \setlength{\unitlength}{\svgwidth}%
  \fi%
  \global\let\svgwidth\undefined%
  \global\let\svgscale\undefined%
  \makeatother%
  \begin{picture}(1,0.55428761)%
    \put(0,0){\includegraphics[width=\unitlength]{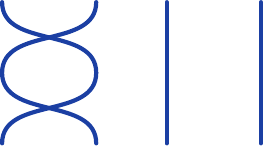}}%
    \put(0.43014569,0.2438102){\color[rgb]{0,0,0}\makebox(0,0)[lb]{\smash{$=$
}}}%
  \end{picture}%
\endgroup%
}}$
%%%%%%%%%%%%%%%%%%%%%%%%%%%%%%%%%%%%%%%%%%%%%%%%%%%%%%%
\item the Reidemeister III move ,\hskip2.2cm  \label{fig:axiom_Reid_3}
%%%%%%%%%%%%%%%%%%%%%%%%%%%%%%%%%%%%%%%%%%%%%%%%%%%%%
 		$\vcenter{\hbox{ %% Creator: Inkscape 0.48.2, www.inkscape.org
%% PDF/EPS/PS + LaTeX output extension by Johan Engelen, 2010
%% Accompanies image file 'axiom_Reid_3.pdf' (pdf, eps, ps)
%%
%% To include the image in your LaTeX document, write
%%   \input{<filename>.pdf_tex}
%%  instead of
%%   \includegraphics{<filename>.pdf}
%% To scale the image, write
%%   \def\svgwidth{<desired width>}
%%   \input{<filename>.pdf_tex}
%%  instead of
%%   \includegraphics[width=<desired width>]{<filename>.pdf}
%%
%% Images with a different path to the parent latex file can
%% be accessed with the `import' package (which may need to be
%% installed) using
%%   \usepackage{import}
%% in the preamble, and then including the image with
%%   \import{<path to file>}{<filename>.pdf_tex}
%% Alternatively, one can specify
%%   \graphicspath{{<path to file>/}}
%% 
%% For more information, please see info/svg-inkscape on CTAN:
%%   http://tug.ctan.org/tex-archive/info/svg-inkscape
%%
\begingroup%
  \makeatletter%
  \providecommand\color[2][]{%
    \errmessage{(Inkscape) Color is used for the text in Inkscape, but the package 'color.sty' is not loaded}%
    \renewcommand\color[2][]{}%
  }%
  \providecommand\transparent[1]{%
    \errmessage{(Inkscape) Transparency is used (non-zero) for the text in Inkscape, but the package 'transparent.sty' is not loaded}%
    \renewcommand\transparent[1]{}%
  }%
  \providecommand\rotatebox[2]{#2}%
  \ifx\svgwidth\undefined%
    \setlength{\unitlength}{124.40087891bp}%
    \ifx\svgscale\undefined%
      \relax%
    \else%
      \setlength{\unitlength}{\unitlength * \real{\svgscale}}%
    \fi%
  \else%
    \setlength{\unitlength}{\svgwidth}%
  \fi%
  \global\let\svgwidth\undefined%
  \global\let\svgscale\undefined%
  \makeatother%
  \begin{picture}(1,0.32278936)%
    \put(0,0){\includegraphics[width=\unitlength]{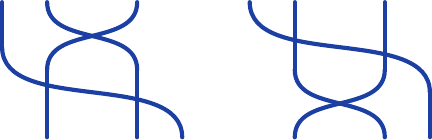}}%
    \put(0.44945526,0.14184983){\color[rgb]{0,0,0}\makebox(0,0)[lb]{\smash{$=$}}}%
  \end{picture}%
\endgroup%
}}$
 %%%%%%%%%%%%%%%%%%%%%%%%%%%%%%%%%%%%%%%%%%%%%%%%%%%%%%%
\item  \label{fig:axiom_left_right} the ribbon condition,\hskip4.2cm  
%%%%%%%%%%%%%%%%%%%%%%%%%%%%%%%%%%%%%%%%%%%%%%%%%%%%%
		$\vcenter{\hbox{ %% Creator: Inkscape 0.48.2, www.inkscape.org
%% PDF/EPS/PS + LaTeX output extension by Johan Engelen, 2010
%% Accompanies image file 'axiom_left_right.pdf' (pdf, eps, ps)
%%
%% To include the image in your LaTeX document, write
%%   \input{<filename>.pdf_tex}
%%  instead of
%%   \includegraphics{<filename>.pdf}
%% To scale the image, write
%%   \def\svgwidth{<desired width>}
%%   \input{<filename>.pdf_tex}
%%  instead of
%%   \includegraphics[width=<desired width>]{<filename>.pdf}
%%
%% Images with a different path to the parent latex file can
%% be accessed with the `import' package (which may need to be
%% installed) using
%%   \usepackage{import}
%% in the preamble, and then including the image with
%%   \import{<path to file>}{<filename>.pdf_tex}
%% Alternatively, one can specify
%%   \graphicspath{{<path to file>/}}
%% 
%% For more information, please see info/svg-inkscape on CTAN:
%%   http://tug.ctan.org/tex-archive/info/svg-inkscape
%%
\begingroup%
  \makeatletter%
  \providecommand\color[2][]{%
    \errmessage{(Inkscape) Color is used for the text in Inkscape, but the package 'color.sty' is not loaded}%
    \renewcommand\color[2][]{}%
  }%
  \providecommand\transparent[1]{%
    \errmessage{(Inkscape) Transparency is used (non-zero) for the text in Inkscape, but the package 'transparent.sty' is not loaded}%
    \renewcommand\transparent[1]{}%
  }%
  \providecommand\rotatebox[2]{#2}%
  \ifx\svgwidth\undefined%
    \setlength{\unitlength}{75.31722412bp}%
    \ifx\svgscale\undefined%
      \relax%
    \else%
      \setlength{\unitlength}{\unitlength * \real{\svgscale}}%
    \fi%
  \else%
    \setlength{\unitlength}{\svgwidth}%
  \fi%
  \global\let\svgwidth\undefined%
  \global\let\svgscale\undefined%
  \makeatother%
  \begin{picture}(1,0.59036449)%
    \put(0,0){\includegraphics[width=\unitlength]{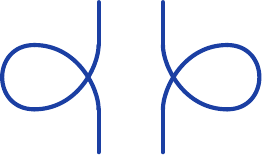}}%
    \put(0.44899274,0.26275822){\color[rgb]{0,0,0}\makebox(0,0)[lb]{\smash{$=$}}}%
  \end{picture}%
\endgroup%
}} \hspace{3mm}.$
\end{enumerate}
\end{definition}
Either side of axiom \ref{fig:axiom_left_right} defines a map, $\gls{phi} \colon \gls{A} \to A$. Either diagram in axiom  \ref{fig:axiom_left_right} is called a curl.

There are two issues to settle: the possible dependence of the state sum model on the triangulation of the surface, and on the immersion $i$. The former is the easiest to resolve: since any planar state sum model is invariant under Pachner moves the following lemma is automatically verified.

\begin{lemma}\label{lem:spin-triangulation} The partition function of a spin state sum model is independent of the triangulation of the surface.
\end{lemma}

The main issue is the dependence of the partition function on the immersion. Consider a standard immersion $i_0$ that is an embedding of the closed oriented surface of genus $g$ into $\Rb^3$. A triangulation of the surface $\gls{Sigma}$ can be constructed by identifying the edges of a $4g$-sided polygon, as in figure \ref{fig:4gpoly}, and dividing it into triangles without introducing any new vertices. Let $S\subset\Sigma$ be the subset obtained by removing a disk neighbourhood of the vertex of the polygon from $\Sigma$.  The embedding is such that $S$ projects to $\Rb^2$ by the immersion shown in figure \ref{fig:surface-diagram}. The dual graph to the triangulation is shown in the figure \ref{fig:dual-graph} with all of the graph vertices consolidated into one (associativity tells us this can be done unequivocally). 

As we have seen, the usual notion of spin structure is defined for oriented smooth manifolds using the tangent bundle. Each immersed curve $c$ on the manifold lifts to a curve in the frame bundle $P_{SO}$ and the spin structure $s\in H^1(P_{SO},\Zb_2)$ assigns to this an element $\gls{s}(c)\in\Zb_2$. A spin structure on $\Sigma$ is determined uniquely by a spin structure on the subset $S\subset\Sigma$ obtained by removing a disk. The surface $S$ can be embedded in $\Rb^3$ so that the projection to $\Rb^2$ is an immersion as in figure \ref{fig:surface-diagram}, or a modification of it by putting a curl in any of the $2g$ ribbon loops. The spin structure is read off from this diagram as a quadratic form \eqref{eq:quad_b}: $q_s(c)$ is the Whitney degree mod 2 for the projection of $c$ to $\Rb^2$.  For example, for the embedding $i_0$ each circle $c$ in figure \ref{fig:dual-graph} has no curls and so $q_s(c)=0$.  It is worth noting that this explicit construction of $q_s$ does not require a smooth structure and makes sense also for a piecewise-linear surface.

%%%%%%%%%%%%%%%%%%%%%%%%%%%%%%%%%%%%%%%%%%%%%%%%%%%
\begin{figure}
\centering
\begin{subfigure}[t!]{0.47\textwidth}
		\centering
		\includegraphics{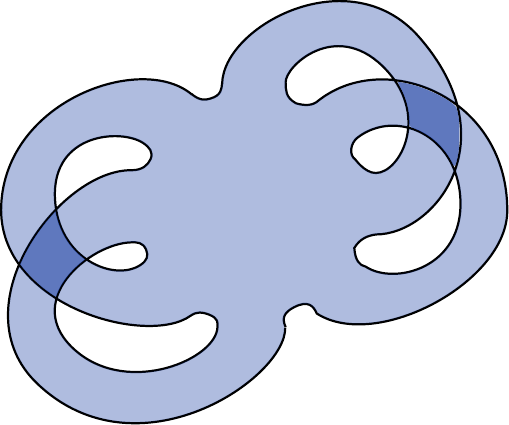}
		\vspace{4mm}
		\caption[Standard projection]{\emph{Standard projection.} The projection of $S$ into $\Rb^2$ for the standard embedding of $\Sigma_2$. Dark-shaded regions represent areas of intersection.}
                \label{fig:surface-diagram}
\end{subfigure} 
\hspace{5mm}
\begin{subfigure}[t!]{0.47\textwidth}
		\centering
		\vspace{9mm}
		\includegraphics{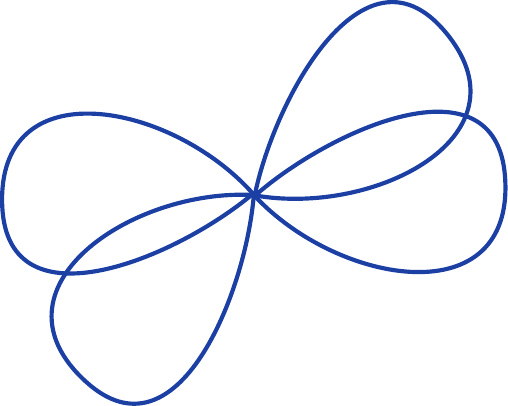}
		\vspace{4mm}
		\caption[Standard diagram]{\emph{Standard diagram} $\gamma_g$. The dual graph of $\Sigma_2$ resulting from the standard embedding in $\Rb^3$, denoted $\gamma_2$. Thickening to a ribbon graph results in figure \ref{fig:surface-diagram}. The standard diagram for a surface $\gls{Sigma_g}$ is denoted $\gamma_g$. }
                \label{fig:dual-graph}
\end{subfigure} 
\caption[Standard immersions and projections of $\Sigma_2$]{Standard immersions and projections of $\Sigma_2$}
\end{figure}
%%%%%%%%%%%%%%%%%%%%%%%%%%%%%%%%%%%%%%%%%%%%%%%%%%%%%%%%%%%%%%%

\begin{theorem} \label{spin-embedding} The partition function of a spin state sum model on $\Sigma$ depends on the immersion $i\colon \Sigma\to\Rb^3$  only via the spin structure induced on $\Sigma$.
\end{theorem}
\begin{proof} Consider  a ribbon graph $K\subset\gls{Sigma}$ and an immersion $i\colon \Sigma\to \Rb^3$. The immersion of the ribbon graph is moved by regular homotopy to $j\colon K\to\Rb^3$ that is blackboard-framed with respect to the projection $P\colon\Rb^3\to\Rb^2$, $P(x,y,z)=(x,y)$. Further, a neighbourhood of the consolidated vertex in the diagram can be moved to match a neighbourhood of the vertex in figure \ref{fig:dual-graph}.  Then each ribbon loop of $K$ can be moved independently, keeping the neighbourhood of the vertex fixed. 

According to the  Whitney-Graustein theorem \cite{kauffman-regular-isotopy}, a complete invariant of an immersed circle in $\Rb^2$ under regular homotopy (in $\Rb^2$) is the Whitney degree, which is the integer that measures the number of windings of the tangent vector to the circle. This regular homotopy extends to a regular homotopy of the ribbon graph in $\Rb^2$. Then it lifts to a regular homotopy of the ribbon graph $j(K)$ in $\Rb^3$, by keeping the $z$-coordinate constant in the homotopy. Therefore each loop of $K$ is regular-homotopic to the corresponding loop of figure \ref{fig:dual-graph} but with a number of curls. The curls can be cancelled in pairs using the ribbon condition (\ref{fig:axiom_left_right} is equivalent to the condition $\gls{phi}^2=\iden$). Each curve will have either one or zero curls; this is the data in the induced spin structure.
\end{proof}

%%%%%%%%%%%%%%%%%%%%%%%%%%%%%%%%%%%%%%%%%%%%%%%%%%%%%%%%%%%%%%%%%%%%
\begin{figure}[t!]
		\centering
		\includegraphics[width=\textwidth]{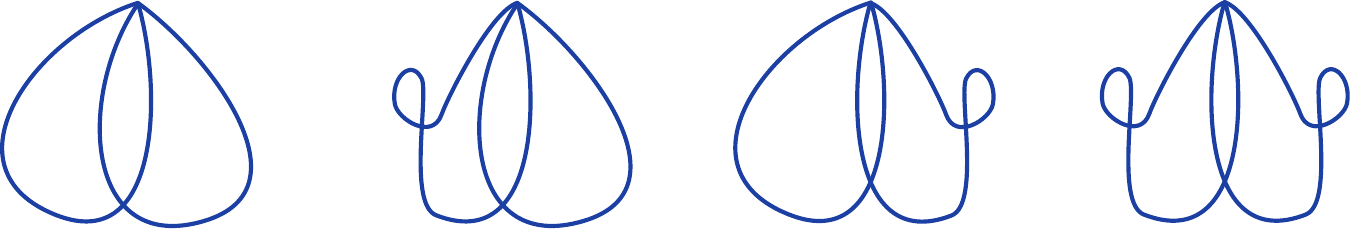}
		\caption[Torus spin models]{\emph{Torus spin models.} Dual graph diagrams for four immersions of the torus that are inequivalent under regular homotopy. The labels $c_1$ and $c_2$ are in correspondence with the left and right cycles in each diagram, respectively.}
                \label{fig:torus-spin-diagrams}
\end{figure} 
%%%%%%%%%%%%%%%%%%%%%%%%%%%%%%%%%%%%%%%%%%%%%%%%%%%%%%%%%%%%%%%%%%%%

For example, one diagram for each of the four equivalence classes for the torus are shown in figure \ref{fig:torus-spin-diagrams}. The corresponding spin structures have $(q_s(c_1),q_s(c_2))=(0,0),(1,0),(0,1),(1,1)$ for the two embedded cycles $c_1$, $c_2$ forming a basis of $H_1(\Sigma_1,\Zb_2)$.

Lemma \ref{lem:spin-triangulation} and theorem \ref{spin-embedding} imply the partition function is a diffeomorphic-invariant of a surface with spin structure. This means the invariant to be associated with a surface $\Sigma$ can depend on the spin structure $s$ only through the parity, $P(s)$. Let $f\colon\Sigma'\to\Sigma$ be a diffeomorphism and $i\colon\Sigma\to\Rb^3$ an immersion inducing a spin structure $s$. Then the immersion $i\circ f$ induces the spin structure $f^*s$ on $\Sigma'$ and $P(s)=P(f^{\ast}s)$.

There is a particular class of spin state sum models that do not depend on the immersion.  The data for these examples satisfy one additional axiom, $\varphi=\iden$. These models are called curl-free.
 Diagrammatically, this is 
\emph{
\begin{enumerate}[label=B\arabic{*}), ref=(B\arabic{*})]
\setcounter{enumi}{5}
\item  \label{ax:RI} the Reidemeister I move (RI),\hskip1cm
%%%%%%%%%%%%%%%%%%%%%%%%%%%%%%%%%%%%%%%%%%%%%%%%%%%%%
		$\vcenter{\hbox{%% Creator: Inkscape 0.48.2, www.inkscape.org
%% PDF/EPS/PS + LaTeX output extension by Johan Engelen, 2010
%% Accompanies image file 'axiom_left_right_equal.pdf' (pdf, eps, ps)
%%
%% To include the image in your LaTeX document, write
%%   \input{<filename>.pdf_tex}
%%  instead of
%%   \includegraphics{<filename>.pdf}
%% To scale the image, write
%%   \def\svgwidth{<desired width>}
%%   \input{<filename>.pdf_tex}
%%  instead of
%%   \includegraphics[width=<desired width>]{<filename>.pdf}
%%
%% Images with a different path to the parent latex file can
%% be accessed with the `import' package (which may need to be
%% installed) using
%%   \usepackage{import}
%% in the preamble, and then including the image with
%%   \import{<path to file>}{<filename>.pdf_tex}
%% Alternatively, one can specify
%%   \graphicspath{{<path to file>/}}
%% 
%% For more information, please see info/svg-inkscape on CTAN:
%%   http://tug.ctan.org/tex-archive/info/svg-inkscape
%%
\begingroup%
  \makeatletter%
  \providecommand\color[2][]{%
    \errmessage{(Inkscape) Color is used for the text in Inkscape, but the package 'color.sty' is not loaded}%
    \renewcommand\color[2][]{}%
  }%
  \providecommand\transparent[1]{%
    \errmessage{(Inkscape) Transparency is used (non-zero) for the text in Inkscape, but the package 'transparent.sty' is not loaded}%
    \renewcommand\transparent[1]{}%
  }%
  \providecommand\rotatebox[2]{#2}%
  \ifx\svgwidth\undefined%
    \setlength{\unitlength}{50.54191895bp}%
    \ifx\svgscale\undefined%
      \relax%
    \else%
      \setlength{\unitlength}{\unitlength * \real{\svgscale}}%
    \fi%
  \else%
    \setlength{\unitlength}{\svgwidth}%
  \fi%
  \global\let\svgwidth\undefined%
  \global\let\svgscale\undefined%
  \makeatother%
  \begin{picture}(1,0.93768129)%
    \put(0,0){\includegraphics[width=\unitlength]{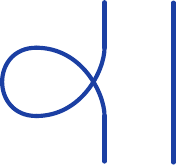}}%
    \put(0.69298967,0.40858098){\color[rgb]{0,0,0}\makebox(0,0)[lb]{\smash{$=$}}}%
  \end{picture}%
\endgroup%
}}\hspace{3mm}.$
%%%%%%%%%%%%%%%%%%%%%%%%%%%%%%%%%%%%%%%%%%%%%%%%%%%%%%%
\end{enumerate}
}

\begin{corollary}\label{lem:embedding} The partition function of a curl-free state sum model for a closed surface is independent of the immersion $i$.
\end{corollary}
\begin{proof} According to theorem~\ref{spin-embedding} a spin diagram for $\Sigma_g$ differs from the standard diagram $\gamma_g$ only by the number of curls in each curve taken $\mo 2$. Since axiom \ref{ax:RI} imposes $\varphi=\id$ the diagram cannot depend on the choice of immersion. 
\end{proof}

This result together with lemma \ref{lem:spin-triangulation} imply the partition function of a curl-free model is a topological invariant. Let $f\colon\Sigma'\to\gls{Sigma}$ be a diffeomorphism. If $\Sigma$ is a triangulated surface and $i\colon\Sigma\to\Rb^3$ is an immersion, then $f$ induces a triangulation and an immersion for $\Sigma'$ such that their dual graph diagrams in the plane coincide.

%Note that the invariance of the partition function can also be checked directly, without using the Pachner moves, by examining the effect of Dehn twists \cite{lickorish,lickorish-erratum} on the surface.

%%%%%%%%%%%%%%%%%%%%%%%%%%%%%%%%%%%%%%%%%%%%%%%%%%%%%%%
%%%%%%%%%%%%%%%%%%%%%%%%%%%%%%%%%%%%%%%%%%%%%%%%%%%%%%%
\section{Algebraic structure of spin models}\label{sec:spinssm}
%%%%%%%%%%%%%%%%%%%%%%%%%%%%%%%%%%%%%%%%%%%%%%%%%%%%%%%
%%%%%%%%%%%%%%%%%%%%%%%%%%%%%%%%%%%%%%%%%%%%%%%%%%%%%%%

To calculate examples of spin models, an explicit formula is needed for the partition functions $\gls{Z}(\gls{Sigma},\gls{s})$ that is manifestly an invariant of the surface with spin structure. To establish this non-trivial result (theorem \ref{theo:main}), the algebraic consequences of the axioms for the spin models are studied. 

The two following results involving spin diagrams will prove useful throughout this section. Obvious generalisations are also of interest in chapter \S\ref{ch:spin_defects} -- by abuse of notation we will refer to them as a consequence of lemmas  
\ref{lem:closed_below} and \ref{lem:varphi}.
\begin{lemma} \label{lem:closed_below}
Consider a diagram $\gls{G}$ of the form $\begin{aligned}\tiny%% Creator: Inkscape 0.48.2, www.inkscape.org
%% PDF/EPS/PS + LaTeX output extension by Johan Engelen, 2010
%% Accompanies image file 'closed_blob.pdf' (pdf, eps, ps)
%%
%% To include the image in your LaTeX document, write
%%   \input{<filename>.pdf_tex}
%%  instead of
%%   \includegraphics{<filename>.pdf}
%% To scale the image, write
%%   \def\svgwidth{<desired width>}
%%   \input{<filename>.pdf_tex}
%%  instead of
%%   \includegraphics[width=<desired width>]{<filename>.pdf}
%%
%% Images with a different path to the parent latex file can
%% be accessed with the `import' package (which may need to be
%% installed) using
%%   \usepackage{import}
%% in the preamble, and then including the image with
%%   \import{<path to file>}{<filename>.pdf_tex}
%% Alternatively, one can specify
%%   \graphicspath{{<path to file>/}}
%% 
%% For more information, please see info/svg-inkscape on CTAN:
%%   http://tug.ctan.org/tex-archive/info/svg-inkscape
%%
\begingroup%
  \makeatletter%
  \providecommand\color[2][]{%
    \errmessage{(Inkscape) Color is used for the text in Inkscape, but the package 'color.sty' is not loaded}%
    \renewcommand\color[2][]{}%
  }%
  \providecommand\transparent[1]{%
    \errmessage{(Inkscape) Transparency is used (non-zero) for the text in Inkscape, but the package 'transparent.sty' is not loaded}%
    \renewcommand\transparent[1]{}%
  }%
  \providecommand\rotatebox[2]{#2}%
  \ifx\svgwidth\undefined%
    \setlength{\unitlength}{14.83815481bp}%
    \ifx\svgscale\undefined%
      \relax%
    \else%
      \setlength{\unitlength}{\unitlength * \real{\svgscale}}%
    \fi%
  \else%
    \setlength{\unitlength}{\svgwidth}%
  \fi%
  \global\let\svgwidth\undefined%
  \global\let\svgscale\undefined%
  \makeatother%
  \begin{picture}(1,1.60712311)%
    \put(0,0){\includegraphics[width=\unitlength]{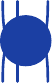}}%
    \put(0.42430097,0.15569849){\color[rgb]{0,0,0}\makebox(0,0)[lb]{\smash{...}}}%
    \put(0.42430097,1.44965992){\color[rgb]{0,0,0}\makebox(0,0)[lb]{\smash{...}}}%
  \end{picture}%
\endgroup%
\normalsize\end{aligned}$ with any number of upward and downward pointing legs, such that the shaded region does not have open edges. Then $G$ is said to be semi-closed, and it satisfies
\begin{align}
\begin{aligned}
%% Creator: Inkscape 0.48.2, www.inkscape.org
%% PDF/EPS/PS + LaTeX output extension by Johan Engelen, 2010
%% Accompanies image file 'closed_blob_identity.pdf' (pdf, eps, ps)
%%
%% To include the image in your LaTeX document, write
%%   \input{<filename>.pdf_tex}
%%  instead of
%%   \includegraphics{<filename>.pdf}
%% To scale the image, write
%%   \def\svgwidth{<desired width>}
%%   \input{<filename>.pdf_tex}
%%  instead of
%%   \includegraphics[width=<desired width>]{<filename>.pdf}
%%
%% Images with a different path to the parent latex file can
%% be accessed with the `import' package (which may need to be
%% installed) using
%%   \usepackage{import}
%% in the preamble, and then including the image with
%%   \import{<path to file>}{<filename>.pdf_tex}
%% Alternatively, one can specify
%%   \graphicspath{{<path to file>/}}
%% 
%% For more information, please see info/svg-inkscape on CTAN:
%%   http://tug.ctan.org/tex-archive/info/svg-inkscape
%%
\begingroup%
  \makeatletter%
  \providecommand\color[2][]{%
    \errmessage{(Inkscape) Color is used for the text in Inkscape, but the package 'color.sty' is not loaded}%
    \renewcommand\color[2][]{}%
  }%
  \providecommand\transparent[1]{%
    \errmessage{(Inkscape) Transparency is used (non-zero) for the text in Inkscape, but the package 'transparent.sty' is not loaded}%
    \renewcommand\transparent[1]{}%
  }%
  \providecommand\rotatebox[2]{#2}%
  \ifx\svgwidth\undefined%
    \setlength{\unitlength}{175.318046bp}%
    \ifx\svgscale\undefined%
      \relax%
    \else%
      \setlength{\unitlength}{\unitlength * \real{\svgscale}}%
    \fi%
  \else%
    \setlength{\unitlength}{\svgwidth}%
  \fi%
  \global\let\svgwidth\undefined%
  \global\let\svgscale\undefined%
  \makeatother%
  \begin{picture}(1,0.38260835)%
    \put(0,0){\includegraphics[width=\unitlength]{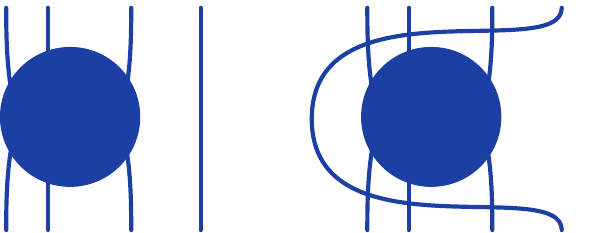}}%
    \put(0.39808054,0.16441109){\color[rgb]{0,0,0}\makebox(0,0)[lb]{\smash{$=$}}}%
    \put(0.10147673,0.34693658){\color[rgb]{0,0,0}\makebox(0,0)[lb]{\smash{$\cdots$}}}%
    \put(0.10147673,0.00470141){\color[rgb]{0,0,0}\makebox(0,0)[lb]{\smash{$\cdots$}}}%
    \put(0.69468435,0.34693644){\color[rgb]{0,0,0}\makebox(0,0)[lb]{\smash{$\cdots$}}}%
    \put(0.69468435,0.00470128){\color[rgb]{0,0,0}\makebox(0,0)[lb]{\smash{$\cdots$}}}%
    \put(0.99128815,0.16441109){\color[rgb]{0,0,0}\makebox(0,0)[lb]{\smash{.}}}%
  \end{picture}%
\endgroup%

\label{eq:closed_below}
\end{aligned}
\end{align}
\end{lemma}
\begin{proof}
Semi-closed diagrams have one key property: the shaded region of $\gls{G}$, although possibly quite complex, is composed of five types of map only: the identity, $\gls{Bb}$ and $B^{-1}$, $\gls{Ct}$ and finally $\gls{lambda}$. Furthermore, axiom \ref{fig:axiom_square} guarantees the identity, $B$ and $B^{-1}$ all satisfy \eqref{eq:closed_below}. It thus remains only to verify \eqref{eq:closed_below} is true for the cases of $C$ and $\lambda$:
%%%%%%%%%%%%%%%%%%%%%%%%%%%%%%%%%%%%%%%%%%%%%%%%%%%%%
\begin{align}
\begin{aligned}
%% Creator: Inkscape 0.48.2, www.inkscape.org
%% PDF/EPS/PS + LaTeX output extension by Johan Engelen, 2010
%% Accompanies image file 'blob_mult.pdf' (pdf, eps, ps)
%%
%% To include the image in your LaTeX document, write
%%   \input{<filename>.pdf_tex}
%%  instead of
%%   \includegraphics{<filename>.pdf}
%% To scale the image, write
%%   \def\svgwidth{<desired width>}
%%   \input{<filename>.pdf_tex}
%%  instead of
%%   \includegraphics[width=<desired width>]{<filename>.pdf}
%%
%% Images with a different path to the parent latex file can
%% be accessed with the `import' package (which may need to be
%% installed) using
%%   \usepackage{import}
%% in the preamble, and then including the image with
%%   \import{<path to file>}{<filename>.pdf_tex}
%% Alternatively, one can specify
%%   \graphicspath{{<path to file>/}}
%% 
%% For more information, please see info/svg-inkscape on CTAN:
%%   http://tug.ctan.org/tex-archive/info/svg-inkscape
%%
\begingroup%
  \makeatletter%
  \providecommand\color[2][]{%
    \errmessage{(Inkscape) Color is used for the text in Inkscape, but the package 'color.sty' is not loaded}%
    \renewcommand\color[2][]{}%
  }%
  \providecommand\transparent[1]{%
    \errmessage{(Inkscape) Transparency is used (non-zero) for the text in Inkscape, but the package 'transparent.sty' is not loaded}%
    \renewcommand\transparent[1]{}%
  }%
  \providecommand\rotatebox[2]{#2}%
  \ifx\svgwidth\undefined%
    \setlength{\unitlength}{309.73549012bp}%
    \ifx\svgscale\undefined%
      \relax%
    \else%
      \setlength{\unitlength}{\unitlength * \real{\svgscale}}%
    \fi%
  \else%
    \setlength{\unitlength}{\svgwidth}%
  \fi%
  \global\let\svgwidth\undefined%
  \global\let\svgscale\undefined%
  \makeatother%
  \begin{picture}(1,0.1331405)%
    \put(0,0){\includegraphics[width=\unitlength]{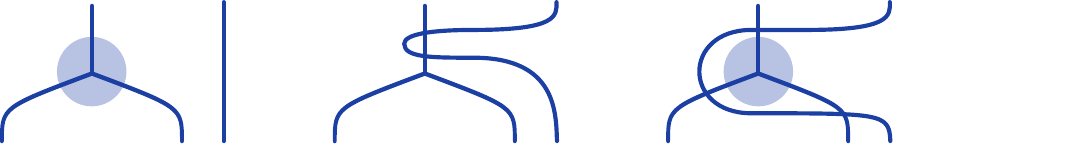}}%
    \put(0.23378422,0.05365489){\color[rgb]{0,0,0}\makebox(0,0)[lb]{\smash{$\overset{\text{\ref{fig:axiom_square}}}{=}$}}}%
    \put(0.86658223,0.05365489){\color[rgb]{0,0,0}\makebox(0,0)[lb]{\smash{,}}}%
    \put(0.5437261,0.05365489){\color[rgb]{0,0,0}\makebox(0,0)[lb]{\smash{$\overset{\text{\ref{fig:axiom_mult}}}{=}$}}}%
  \end{picture}%
\endgroup%

\end{aligned}
\end{align}
%%%%%%%%%%%%%%%%%%%%%%%%%%%%%%%%%%%%%%%%%%%%%%%%%%%%%
%%%%%%%%%%%%%%%%%%%%%%%%%%%%%%%%%%%%%%%%%%%%%%%%%%%%%
\begin{align}
\begin{aligned}
\hspace{4mm}
%% Creator: Inkscape 0.48.2, www.inkscape.org
%% PDF/EPS/PS + LaTeX output extension by Johan Engelen, 2010
%% Accompanies image file 'blob_cross.pdf' (pdf, eps, ps)
%%
%% To include the image in your LaTeX document, write
%%   \input{<filename>.pdf_tex}
%%  instead of
%%   \includegraphics{<filename>.pdf}
%% To scale the image, write
%%   \def\svgwidth{<desired width>}
%%   \input{<filename>.pdf_tex}
%%  instead of
%%   \includegraphics[width=<desired width>]{<filename>.pdf}
%%
%% Images with a different path to the parent latex file can
%% be accessed with the `import' package (which may need to be
%% installed) using
%%   \usepackage{import}
%% in the preamble, and then including the image with
%%   \import{<path to file>}{<filename>.pdf_tex}
%% Alternatively, one can specify
%%   \graphicspath{{<path to file>/}}
%% 
%% For more information, please see info/svg-inkscape on CTAN:
%%   http://tug.ctan.org/tex-archive/info/svg-inkscape
%%
\begingroup%
  \makeatletter%
  \providecommand\color[2][]{%
    \errmessage{(Inkscape) Color is used for the text in Inkscape, but the package 'color.sty' is not loaded}%
    \renewcommand\color[2][]{}%
  }%
  \providecommand\transparent[1]{%
    \errmessage{(Inkscape) Transparency is used (non-zero) for the text in Inkscape, but the package 'transparent.sty' is not loaded}%
    \renewcommand\transparent[1]{}%
  }%
  \providecommand\rotatebox[2]{#2}%
  \ifx\svgwidth\undefined%
    \setlength{\unitlength}{270.11731763bp}%
    \ifx\svgscale\undefined%
      \relax%
    \else%
      \setlength{\unitlength}{\unitlength * \real{\svgscale}}%
    \fi%
  \else%
    \setlength{\unitlength}{\svgwidth}%
  \fi%
  \global\let\svgwidth\undefined%
  \global\let\svgscale\undefined%
  \makeatother%
  \begin{picture}(1,0.1526891)%
    \put(0,0){\includegraphics[width=\unitlength]{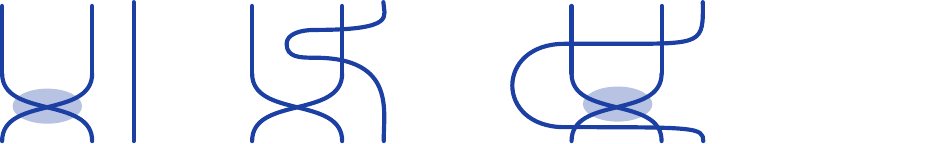}}%
    \put(0.17230698,0.06154533){\color[rgb]{0,0,0}\makebox(0,0)[lb]{\smash{$\overset{\text{\ref{fig:axiom_square}}}{=}$}}}%
    \put(0.79425896,0.06154533){\color[rgb]{0,0,0}\makebox(0,0)[lb]{\smash{.}}}%
    \put(0.43885783,0.06154533){\color[rgb]{0,0,0}\makebox(0,0)[lb]{\smash{$\overset{\text{\ref{fig:axiom_Reid_3}}}{=}$}}}%
  \end{picture}%
\endgroup%

\end{aligned}
\end{align}
%%%%%%%%%%%%%%%%%%%%%%%%%%%%%%%%%%%%%%%%%%%%%%%%%%%%%
\end{proof}
\begin{lemma} \label{lem:varphi}
If $\gls{phi}$ is the curl map associated with a spin state sum model then it satisfies:
\begin{enumerate}[label=C\arabic{*}), ref=(C\arabic{*})]
\item \label{it:phi_one} $\varphi^2=\iden$;
\item \label{it:phi_two} $\varphi$ is an automorphism of $\gls{A}$;
\item \label{it:phi_four} $\varphi$ is compatible with $B$, $\begin{aligned} %% Creator: Inkscape 0.48.2, www.inkscape.org
%% PDF/EPS/PS + LaTeX output extension by Johan Engelen, 2010
%% Accompanies image file 'curl_form.pdf' (pdf, eps, ps)
%%
%% To include the image in your LaTeX document, write
%%   \input{<filename>.pdf_tex}
%%  instead of
%%   \includegraphics{<filename>.pdf}
%% To scale the image, write
%%   \def\svgwidth{<desired width>}
%%   \input{<filename>.pdf_tex}
%%  instead of
%%   \includegraphics[width=<desired width>]{<filename>.pdf}
%%
%% Images with a different path to the parent latex file can
%% be accessed with the `import' package (which may need to be
%% installed) using
%%   \usepackage{import}
%% in the preamble, and then including the image with
%%   \import{<path to file>}{<filename>.pdf_tex}
%% Alternatively, one can specify
%%   \graphicspath{{<path to file>/}}
%% 
%% For more information, please see info/svg-inkscape on CTAN:
%%   http://tug.ctan.org/tex-archive/info/svg-inkscape
%%
\begingroup%
  \makeatletter%
  \providecommand\color[2][]{%
    \errmessage{(Inkscape) Color is used for the text in Inkscape, but the package 'color.sty' is not loaded}%
    \renewcommand\color[2][]{}%
  }%
  \providecommand\transparent[1]{%
    \errmessage{(Inkscape) Transparency is used (non-zero) for the text in Inkscape, but the package 'transparent.sty' is not loaded}%
    \renewcommand\transparent[1]{}%
  }%
  \providecommand\rotatebox[2]{#2}%
  \ifx\svgwidth\undefined%
    \setlength{\unitlength}{125.725662bp}%
    \ifx\svgscale\undefined%
      \relax%
    \else%
      \setlength{\unitlength}{\unitlength * \real{\svgscale}}%
    \fi%
  \else%
    \setlength{\unitlength}{\svgwidth}%
  \fi%
  \global\let\svgwidth\undefined%
  \global\let\svgscale\undefined%
  \makeatother%
  \begin{picture}(1,0.27169109)%
    \put(0,0){\includegraphics[width=\unitlength]{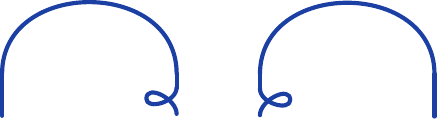}}%
    \put(0.46823697,0.13511637){\color[rgb]{0,0,0}\makebox(0,0)[lb]{\smash{$=$}}}%
  \end{picture}%
\endgroup%
 \end{aligned}\,$;
\item \label{it:phi_three} $\varphi$ is compatible with $\lambda$, $\hspace{1mm}\begin{aligned} %% Creator: Inkscape 0.48.2, www.inkscape.org
%% PDF/EPS/PS + LaTeX output extension by Johan Engelen, 2010
%% Accompanies image file 'curl_cross2.pdf' (pdf, eps, ps)
%%
%% To include the image in your LaTeX document, write
%%   \input{<filename>.pdf_tex}
%%  instead of
%%   \includegraphics{<filename>.pdf}
%% To scale the image, write
%%   \def\svgwidth{<desired width>}
%%   \input{<filename>.pdf_tex}
%%  instead of
%%   \includegraphics[width=<desired width>]{<filename>.pdf}
%%
%% Images with a different path to the parent latex file can
%% be accessed with the `import' package (which may need to be
%% installed) using
%%   \usepackage{import}
%% in the preamble, and then including the image with
%%   \import{<path to file>}{<filename>.pdf_tex}
%% Alternatively, one can specify
%%   \graphicspath{{<path to file>/}}
%% 
%% For more information, please see info/svg-inkscape on CTAN:
%%   http://tug.ctan.org/tex-archive/info/svg-inkscape
%%
\begingroup%
  \makeatletter%
  \providecommand\color[2][]{%
    \errmessage{(Inkscape) Color is used for the text in Inkscape, but the package 'color.sty' is not loaded}%
    \renewcommand\color[2][]{}%
  }%
  \providecommand\transparent[1]{%
    \errmessage{(Inkscape) Transparency is used (non-zero) for the text in Inkscape, but the package 'transparent.sty' is not loaded}%
    \renewcommand\transparent[1]{}%
  }%
  \providecommand\rotatebox[2]{#2}%
  \ifx\svgwidth\undefined%
    \setlength{\unitlength}{138.12807236bp}%
    \ifx\svgscale\undefined%
      \relax%
    \else%
      \setlength{\unitlength}{\unitlength * \real{\svgscale}}%
    \fi%
  \else%
    \setlength{\unitlength}{\svgwidth}%
  \fi%
  \global\let\svgwidth\undefined%
  \global\let\svgscale\undefined%
  \makeatother%
  \begin{picture}(1,0.34489058)%
    \put(0,0){\includegraphics[width=\unitlength]{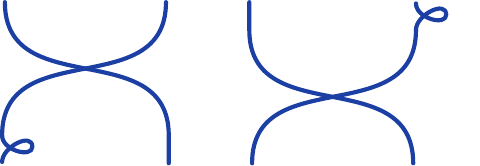}}%
    \put(0.40976992,0.14927837){\color[rgb]{0,0,0}\makebox(0,0)[lb]{\smash{$=$}}}%
    \put(0.98894255,0.14927837){\color[rgb]{0,0,0}\makebox(0,0)[lb]{\smash{.}}}%
  \end{picture}%
\endgroup%
 \end{aligned}$
\end{enumerate}
\end{lemma}
\begin{proof}
The axioms \ref{fig:axiom_form}, \ref{fig:axiom_square} and \ref{fig:axiom_Reid_3} imply, via the Whitney trick \cite{kauffman-regular-isotopy}, that \ref{it:phi_one} holds.
%%%%%%%%%%%%%%%%%%%%%%%%%%%%%%%%%%%%%%%%%%%%%%%%%%%%%
\begin{align}
\begin{aligned}
\hspace{15mm}
%% Creator: Inkscape 0.48.2, www.inkscape.org
%% PDF/EPS/PS + LaTeX output extension by Johan Engelen, 2010
%% Accompanies image file 'Whitney_trick.pdf' (pdf, eps, ps)
%%
%% To include the image in your LaTeX document, write
%%   \input{<filename>.pdf_tex}
%%  instead of
%%   \includegraphics{<filename>.pdf}
%% To scale the image, write
%%   \def\svgwidth{<desired width>}
%%   \input{<filename>.pdf_tex}
%%  instead of
%%   \includegraphics[width=<desired width>]{<filename>.pdf}
%%
%% Images with a different path to the parent latex file can
%% be accessed with the `import' package (which may need to be
%% installed) using
%%   \usepackage{import}
%% in the preamble, and then including the image with
%%   \import{<path to file>}{<filename>.pdf_tex}
%% Alternatively, one can specify
%%   \graphicspath{{<path to file>/}}
%% 
%% For more information, please see info/svg-inkscape on CTAN:
%%   http://tug.ctan.org/tex-archive/info/svg-inkscape
%%
\begingroup%
  \makeatletter%
  \providecommand\color[2][]{%
    \errmessage{(Inkscape) Color is used for the text in Inkscape, but the package 'color.sty' is not loaded}%
    \renewcommand\color[2][]{}%
  }%
  \providecommand\transparent[1]{%
    \errmessage{(Inkscape) Transparency is used (non-zero) for the text in Inkscape, but the package 'transparent.sty' is not loaded}%
    \renewcommand\transparent[1]{}%
  }%
  \providecommand\rotatebox[2]{#2}%
  \ifx\svgwidth\undefined%
    \setlength{\unitlength}{399.12283229bp}%
    \ifx\svgscale\undefined%
      \relax%
    \else%
      \setlength{\unitlength}{\unitlength * \real{\svgscale}}%
    \fi%
  \else%
    \setlength{\unitlength}{\svgwidth}%
  \fi%
  \global\let\svgwidth\undefined%
  \global\let\svgscale\undefined%
  \makeatother%
  \begin{picture}(1,0.21520661)%
    \put(0,0){\includegraphics[width=\unitlength]{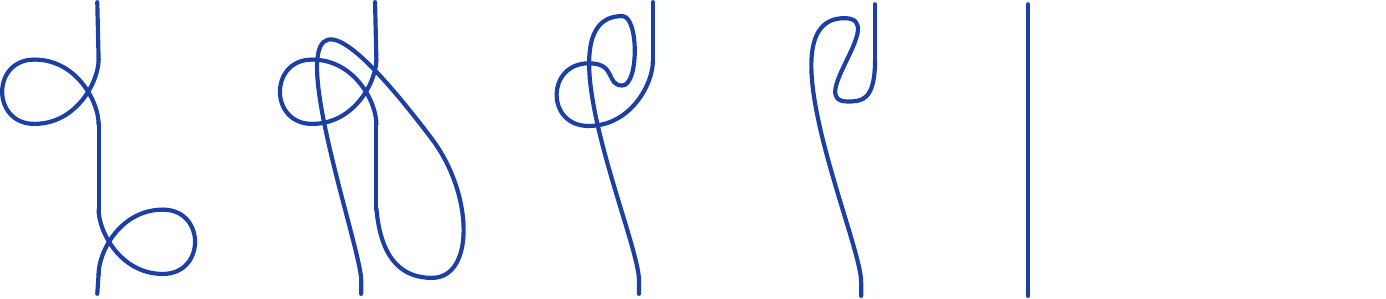}}%
    \put(0.16134951,0.09486876){\color[rgb]{0,0,0}\makebox(0,0)[lb]{\smash{$\overset{\text{\ref{fig:axiom_Reid_3}}}{=}$}}}%
    \put(0.36083463,0.09336553){\color[rgb]{0,0,0}\makebox(0,0)[lb]{\smash{$\overset{\text{\ref{fig:axiom_square}}}{=}$}}}%
    \put(0.51116435,0.09186211){\color[rgb]{0,0,0}\makebox(0,0)[lb]{\smash{$\overset{\text{\ref{fig:axiom_square}}}{=}$}}}%
    \put(0.67151599,0.09186211){\color[rgb]{0,0,0}\makebox(0,0)[lb]{\smash{$\overset{\text{\eqref{eq:snake}}}{=}$}}}%
  \end{picture}%
\endgroup%

\end{aligned}
\end{align}
%%%%%%%%%%%%%%%%%%%%%%%%%%%%%%%%%%%%%%%%%%%%%%%%%%%%%
The diagrammatic proof of \ref{it:phi_two} can be found below. First it is established that $\varphi (a \cdot b) = \varphi (a) \cdot \varphi (b)$,
%%%%%%%%%%%%%%%%%%%%%%%%%%%%%%%%%%%%%%%%%%%%%%%%%%%%%
\begin{align}	
\begin{aligned}
\hspace{15mm}	
%% Creator: Inkscape 0.48.2, www.inkscape.org
%% PDF/EPS/PS + LaTeX output extension by Johan Engelen, 2010
%% Accompanies image file 'phi_homomorphism.pdf' (pdf, eps, ps)
%%
%% To include the image in your LaTeX document, write
%%   \input{<filename>.pdf_tex}
%%  instead of
%%   \includegraphics{<filename>.pdf}
%% To scale the image, write
%%   \def\svgwidth{<desired width>}
%%   \input{<filename>.pdf_tex}
%%  instead of
%%   \includegraphics[width=<desired width>]{<filename>.pdf}
%%
%% Images with a different path to the parent latex file can
%% be accessed with the `import' package (which may need to be
%% installed) using
%%   \usepackage{import}
%% in the preamble, and then including the image with
%%   \import{<path to file>}{<filename>.pdf_tex}
%% Alternatively, one can specify
%%   \graphicspath{{<path to file>/}}
%% 
%% For more information, please see info/svg-inkscape on CTAN:
%%   http://tug.ctan.org/tex-archive/info/svg-inkscape
%%
\begingroup%
  \makeatletter%
  \providecommand\color[2][]{%
    \errmessage{(Inkscape) Color is used for the text in Inkscape, but the package 'color.sty' is not loaded}%
    \renewcommand\color[2][]{}%
  }%
  \providecommand\transparent[1]{%
    \errmessage{(Inkscape) Transparency is used (non-zero) for the text in Inkscape, but the package 'transparent.sty' is not loaded}%
    \renewcommand\transparent[1]{}%
  }%
  \providecommand\rotatebox[2]{#2}%
  \ifx\svgwidth\undefined%
    \setlength{\unitlength}{457.16326676bp}%
    \ifx\svgscale\undefined%
      \relax%
    \else%
      \setlength{\unitlength}{\unitlength * \real{\svgscale}}%
    \fi%
  \else%
    \setlength{\unitlength}{\svgwidth}%
  \fi%
  \global\let\svgwidth\undefined%
  \global\let\svgscale\undefined%
  \makeatother%
  \begin{picture}(1,0.19025318)%
    \put(0,0){\includegraphics[width=\unitlength]{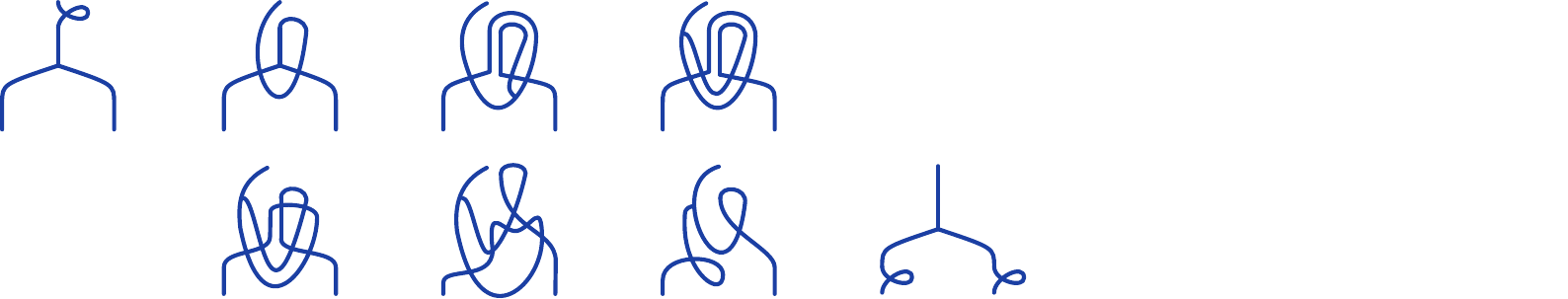}}%
    \put(0.09032261,0.15136911){\color[rgb]{0,0,0}\makebox(0,0)[lb]{\smash{$\overset{\text{\ref{fig:axiom_mult}}}{=}$}}}%
    \put(0.23906596,0.15136911){\color[rgb]{0,0,0}\makebox(0,0)[lb]{\smash{$\overset{\substack{\text{\ref{fig:axiom_form}}\\+\\\text{\ref{fig:axiom_mult}}}}{=}$}}}%
    \put(0.3790597,0.15136911){\color[rgb]{0,0,0}\makebox(0,0)[lb]{\smash{$\overset{\substack{\text{\ref{fig:axiom_form}}\\+\\\text{\ref{fig:axiom_mult}}}}{=}$}}}%
    \put(0.51905345,0.04719354){\color[rgb]{0,0,0}\makebox(0,0)[lb]{\smash{$=$}}}%
    \put(0.23906598,0.0463738){\color[rgb]{0,0,0}\makebox(0,0)[lb]{\smash{$\overset{\text{\ref{fig:axiom_Reid_3}}}{=}$}}}%
    \put(0.37905965,0.0463738){\color[rgb]{0,0,0}\makebox(0,0)[lb]{\smash{$\overset{\text{\ref{fig:axiom_square}}}{=}$}}}%
    \put(0.09907222,0.0463738){\color[rgb]{0,0,0}\makebox(0,0)[lb]{\smash{$\overset{\text{\ref{fig:axiom_square}}}{=}$}}}%
  \end{picture}%
\endgroup%

\end{aligned}
\end{align} 
and second one shows $\gls{phi}$ preserves the algebra unit. Recall that $\gls{unit}=\gls{R}.\gls{m}(\gls{Bb})$:
%%%%%%%%%%%%%%%%%%%%%%%%%%%%%%%%%%%%%%%%%%%%%%%%%%%%%
\begin{align}
\begin{aligned}
\hspace{15mm}
%% Creator: Inkscape 0.48.2, www.inkscape.org
%% PDF/EPS/PS + LaTeX output extension by Johan Engelen, 2010
%% Accompanies image file 'curl_unit.pdf' (pdf, eps, ps)
%%
%% To include the image in your LaTeX document, write
%%   \input{<filename>.pdf_tex}
%%  instead of
%%   \includegraphics{<filename>.pdf}
%% To scale the image, write
%%   \def\svgwidth{<desired width>}
%%   \input{<filename>.pdf_tex}
%%  instead of
%%   \includegraphics[width=<desired width>]{<filename>.pdf}
%%
%% Images with a different path to the parent latex file can
%% be accessed with the `import' package (which may need to be
%% installed) using
%%   \usepackage{import}
%% in the preamble, and then including the image with
%%   \import{<path to file>}{<filename>.pdf_tex}
%% Alternatively, one can specify
%%   \graphicspath{{<path to file>/}}
%% 
%% For more information, please see info/svg-inkscape on CTAN:
%%   http://tug.ctan.org/tex-archive/info/svg-inkscape
%%
\begingroup%
  \makeatletter%
  \providecommand\color[2][]{%
    \errmessage{(Inkscape) Color is used for the text in Inkscape, but the package 'color.sty' is not loaded}%
    \renewcommand\color[2][]{}%
  }%
  \providecommand\transparent[1]{%
    \errmessage{(Inkscape) Transparency is used (non-zero) for the text in Inkscape, but the package 'transparent.sty' is not loaded}%
    \renewcommand\transparent[1]{}%
  }%
  \providecommand\rotatebox[2]{#2}%
  \ifx\svgwidth\undefined%
    \setlength{\unitlength}{230.65449808bp}%
    \ifx\svgscale\undefined%
      \relax%
    \else%
      \setlength{\unitlength}{\unitlength * \real{\svgscale}}%
    \fi%
  \else%
    \setlength{\unitlength}{\svgwidth}%
  \fi%
  \global\let\svgwidth\undefined%
  \global\let\svgscale\undefined%
  \makeatother%
  \begin{picture}(1,0.21473586)%
    \put(0,0){\includegraphics[width=\unitlength]{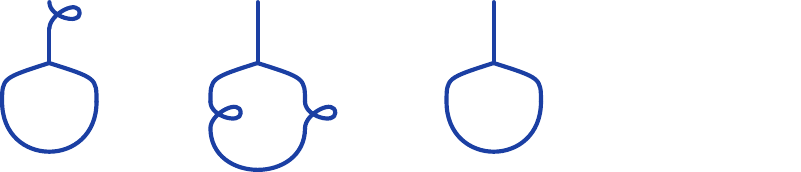}}%
    \put(0.1657063,0.09065617){\color[rgb]{0,0,0}\makebox(0,0)[lb]{\smash{$\overset{\text{\ref{it:phi_two}}}{=}$}}}%
    \put(0.46051954,0.09065617){\color[rgb]{0,0,0}\makebox(0,0)[lb]{\smash{$\overset{\text{\ref{it:phi_one}}}{=}$}}}%
  \end{picture}%
\endgroup%
 
\end{aligned}.
\end{align}
%%%%%%%%%%%%%%%%%%%%%%%%%%%%%%%%%%%%%%%%%%%%%%%%%%%%%
Since $\varphi$ is an invertible linear map, having established it is also a homomorphism is enough to guarantee $\varphi$ is indeed an automorphism of $A$. Property~\ref{it:phi_four} on the other hand, stems directly from axiom~\ref{fig:axiom_form}. 
%%%%%%%%%%%%%%%%%%%%%%%%%%%%%%%%%%%%%%%%%%%%%%%%%%%%%
\begin{align}
\begin{aligned}
\centering
%% Creator: Inkscape 0.48.2, www.inkscape.org
%% PDF/EPS/PS + LaTeX output extension by Johan Engelen, 2010
%% Accompanies image file 'curl_form_proof.pdf' (pdf, eps, ps)
%%
%% To include the image in your LaTeX document, write
%%   \input{<filename>.pdf_tex}
%%  instead of
%%   \includegraphics{<filename>.pdf}
%% To scale the image, write
%%   \def\svgwidth{<desired width>}
%%   \input{<filename>.pdf_tex}
%%  instead of
%%   \includegraphics[width=<desired width>]{<filename>.pdf}
%%
%% Images with a different path to the parent latex file can
%% be accessed with the `import' package (which may need to be
%% installed) using
%%   \usepackage{import}
%% in the preamble, and then including the image with
%%   \import{<path to file>}{<filename>.pdf_tex}
%% Alternatively, one can specify
%%   \graphicspath{{<path to file>/}}
%% 
%% For more information, please see info/svg-inkscape on CTAN:
%%   http://tug.ctan.org/tex-archive/info/svg-inkscape
%%
\begingroup%
  \makeatletter%
  \providecommand\color[2][]{%
    \errmessage{(Inkscape) Color is used for the text in Inkscape, but the package 'color.sty' is not loaded}%
    \renewcommand\color[2][]{}%
  }%
  \providecommand\transparent[1]{%
    \errmessage{(Inkscape) Transparency is used (non-zero) for the text in Inkscape, but the package 'transparent.sty' is not loaded}%
    \renewcommand\transparent[1]{}%
  }%
  \providecommand\rotatebox[2]{#2}%
  \ifx\svgwidth\undefined%
    \setlength{\unitlength}{305.725662bp}%
    \ifx\svgscale\undefined%
      \relax%
    \else%
      \setlength{\unitlength}{\unitlength * \real{\svgscale}}%
    \fi%
  \else%
    \setlength{\unitlength}{\svgwidth}%
  \fi%
  \global\let\svgwidth\undefined%
  \global\let\svgscale\undefined%
  \makeatother%
  \begin{picture}(1,0.14280989)%
    \put(0,0){\includegraphics[width=\unitlength]{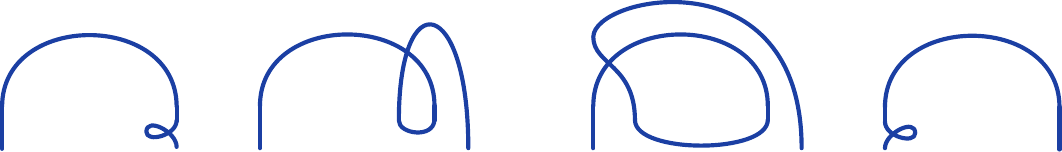}}%
    \put(0.1925563,0.05556483){\color[rgb]{0,0,0}\makebox(0,0)[lb]{\smash{$=$}}}%
    \put(0.46731247,0.05556483){\color[rgb]{0,0,0}\makebox(0,0)[lb]{\smash{$\overset{\text{\ref{fig:axiom_form}}}{=}$}}}%
    \put(0.78131948,0.05556483){\color[rgb]{0,0,0}\makebox(0,0)[lb]{\smash{$=$}}}%
  \end{picture}%
\endgroup%
 
\end{aligned}
\end{align}
%%%%%%%%%%%%%%%%%%%%%%%%%%%%%%%%%%%%%%%%%%%%%%%%%%%%%
Finally, one shows that \ref{it:phi_three} holds.
%%%%%%%%%%%%%%%%%%%%%%%%%%%%%%%%%%%%%%%%%%%%%%%%%%%%%
\begin{align}
\begin{aligned}
\hspace{20mm}
%% Creator: Inkscape 0.48.2, www.inkscape.org
%% PDF/EPS/PS + LaTeX output extension by Johan Engelen, 2010
%% Accompanies image file 'curl_cross.pdf' (pdf, eps, ps)
%%
%% To include the image in your LaTeX document, write
%%   \input{<filename>.pdf_tex}
%%  instead of
%%   \includegraphics{<filename>.pdf}
%% To scale the image, write
%%   \def\svgwidth{<desired width>}
%%   \input{<filename>.pdf_tex}
%%  instead of
%%   \includegraphics[width=<desired width>]{<filename>.pdf}
%%
%% Images with a different path to the parent latex file can
%% be accessed with the `import' package (which may need to be
%% installed) using
%%   \usepackage{import}
%% in the preamble, and then including the image with
%%   \import{<path to file>}{<filename>.pdf_tex}
%% Alternatively, one can specify
%%   \graphicspath{{<path to file>/}}
%% 
%% For more information, please see info/svg-inkscape on CTAN:
%%   http://tug.ctan.org/tex-archive/info/svg-inkscape
%%
\begingroup%
  \makeatletter%
  \providecommand\color[2][]{%
    \errmessage{(Inkscape) Color is used for the text in Inkscape, but the package 'color.sty' is not loaded}%
    \renewcommand\color[2][]{}%
  }%
  \providecommand\transparent[1]{%
    \errmessage{(Inkscape) Transparency is used (non-zero) for the text in Inkscape, but the package 'transparent.sty' is not loaded}%
    \renewcommand\transparent[1]{}%
  }%
  \providecommand\rotatebox[2]{#2}%
  \ifx\svgwidth\undefined%
    \setlength{\unitlength}{387.28432236bp}%
    \ifx\svgscale\undefined%
      \relax%
    \else%
      \setlength{\unitlength}{\unitlength * \real{\svgscale}}%
    \fi%
  \else%
    \setlength{\unitlength}{\svgwidth}%
  \fi%
  \global\let\svgwidth\undefined%
  \global\let\svgscale\undefined%
  \makeatother%
  \begin{picture}(1,0.12397439)%
    \put(0,0){\includegraphics[width=\unitlength]{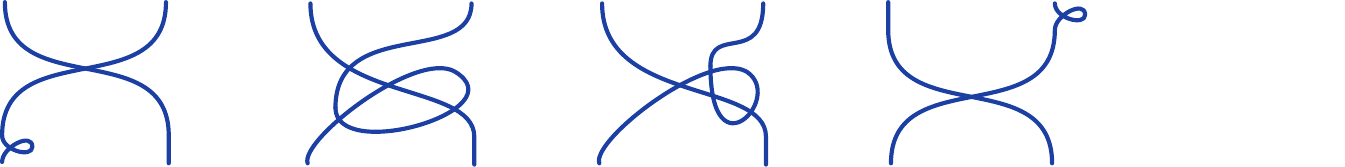}}%
    \put(0.1668044,0.05420772){\color[rgb]{0,0,0}\makebox(0,0)[lb]{\smash{$\overset{\text{\ref{fig:axiom_square}}}{=}$}}}%
    \put(0.61092256,0.05420772){\color[rgb]{0,0,0}\makebox(0,0)[lb]{\smash{$\overset{\text{\ref{fig:axiom_square}}}{=}$}}}%
    \put(0.39402764,0.05420772){\color[rgb]{0,0,0}\makebox(0,0)[lb]{\smash{$\overset{\text{\ref{fig:axiom_Reid_3}}}{=}$}}}%
  \end{picture}%
\endgroup%
 
\end{aligned}
\end{align}
%%%%%%%%%%%%%%%%%%%%%%%%%%%%%%%%%%%%%%%%%%%%%%%%%%%%%
\end{proof}
The first objective is to build the diagrammatic counterpart of expression \eqref{eq:invariant}, assigning $\gls{Z}(\gls{Sigma_g},\gls{s})$ to an orientable surface with spin structure. As explained in section \S\ref{sec:crossing} each diagram can be brought to a standard format (see figure~\ref{fig:flower_diagram}) where each $\begin{aligned}%% Creator: Inkscape 0.48.2, www.inkscape.org
%% PDF/EPS/PS + LaTeX output extension by Johan Engelen, 2010
%% Accompanies image file 'blob.pdf' (pdf, eps, ps)
%%
%% To include the image in your LaTeX document, write
%%   \input{<filename>.pdf_tex}
%%  instead of
%%   \includegraphics{<filename>.pdf}
%% To scale the image, write
%%   \def\svgwidth{<desired width>}
%%   \input{<filename>.pdf_tex}
%%  instead of
%%   \includegraphics[width=<desired width>]{<filename>.pdf}
%%
%% Images with a different path to the parent latex file can
%% be accessed with the `import' package (which may need to be
%% installed) using
%%   \usepackage{import}
%% in the preamble, and then including the image with
%%   \import{<path to file>}{<filename>.pdf_tex}
%% Alternatively, one can specify
%%   \graphicspath{{<path to file>/}}
%% 
%% For more information, please see info/svg-inkscape on CTAN:
%%   http://tug.ctan.org/tex-archive/info/svg-inkscape
%%
\begingroup%
  \makeatletter%
  \providecommand\color[2][]{%
    \errmessage{(Inkscape) Color is used for the text in Inkscape, but the package 'color.sty' is not loaded}%
    \renewcommand\color[2][]{}%
  }%
  \providecommand\transparent[1]{%
    \errmessage{(Inkscape) Transparency is used (non-zero) for the text in Inkscape, but the package 'transparent.sty' is not loaded}%
    \renewcommand\transparent[1]{}%
  }%
  \providecommand\rotatebox[2]{#2}%
  \ifx\svgwidth\undefined%
    \setlength{\unitlength}{8.40000381bp}%
    \ifx\svgscale\undefined%
      \relax%
    \else%
      \setlength{\unitlength}{\unitlength * \real{\svgscale}}%
    \fi%
  \else%
    \setlength{\unitlength}{\svgwidth}%
  \fi%
  \global\let\svgwidth\undefined%
  \global\let\svgscale\undefined%
  \makeatother%
  \begin{picture}(1,1.61712173)%
    \put(0,0){\includegraphics[width=\unitlength]{blob.pdf}}%
  \end{picture}%
\endgroup%
\end{aligned}$ stands for $\gls{phi}$ or the identity according to what immersion $\Sigma_g \looparrowright \Rb^3$ was chosen. 

%%%%%%%%%%%%%%%%%%%%%%%%%%%%%%%%%%%%%%%%%%%%%%%%%%%%%%%%%%%%%%%%%%%%
\begin{figure}[t!]
		\centering
		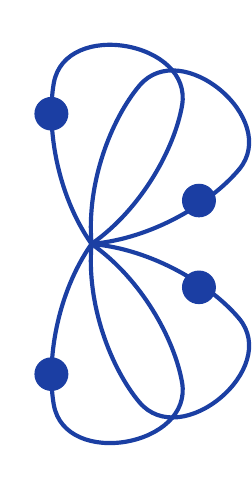
		\caption[Spin diagrams]{\emph{Spin diagrams.} A diagram for a closed surface $\gls{Sigma_g}$ of genus 1 or higher can be brought to the standard form above where $\begin{aligned}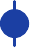\end{aligned}$ stands for either $\gls{phi}$ or the identity map.}
                \label{fig:flower_diagram}
\end{figure} 
%%%%%%%%%%%%%%%%%%%%%%%%%%%%%%%%%%%%%%%%%%%%%%%%%%%%%%%%%%%%%%%%%%%%

It is necessary to construct the analogue of the element $\gls{z}=e_a\cdot e_b\cdot e_c \cdot e_d\,B^{ac}\,B^{bd}$, introduced in equation~\eqref{eq:torus}, in the spin model. Note that the geometrical interpretation we gave $z$ was that of the algebra element to associate with $\Sigma_1-D$. Therefore, its correct generalisation to the spin setting must take into account all possible diagrams associated with $\Sigma_1-D$ with a spin structure. Denoted $\eta_1,\eta_2,\eta_3$ and $\gls{chi}$ they are preferred elements of the algebra -- the building blocks of the spin partition functions:
%%%%%%%%%%%%%%%%%%%%%%%%%%%%%%%
\begin{align}
\begin{aligned}
%% Creator: Inkscape 0.48.2, www.inkscape.org
%% PDF/EPS/PS + LaTeX output extension by Johan Engelen, 2010
%% Accompanies image file 'element1.pdf' (pdf, eps, ps)
%%
%% To include the image in your LaTeX document, write
%%   \input{<filename>.pdf_tex}
%%  instead of
%%   \includegraphics{<filename>.pdf}
%% To scale the image, write
%%   \def\svgwidth{<desired width>}
%%   \input{<filename>.pdf_tex}
%%  instead of
%%   \includegraphics[width=<desired width>]{<filename>.pdf}
%%
%% Images with a different path to the parent latex file can
%% be accessed with the `import' package (which may need to be
%% installed) using
%%   \usepackage{import}
%% in the preamble, and then including the image with
%%   \import{<path to file>}{<filename>.pdf_tex}
%% Alternatively, one can specify
%%   \graphicspath{{<path to file>/}}
%% 
%% For more information, please see info/svg-inkscape on CTAN:
%%   http://tug.ctan.org/tex-archive/info/svg-inkscape
%%
\begingroup%
  \makeatletter%
  \providecommand\color[2][]{%
    \errmessage{(Inkscape) Color is used for the text in Inkscape, but the package 'color.sty' is not loaded}%
    \renewcommand\color[2][]{}%
  }%
  \providecommand\transparent[1]{%
    \errmessage{(Inkscape) Transparency is used (non-zero) for the text in Inkscape, but the package 'transparent.sty' is not loaded}%
    \renewcommand\transparent[1]{}%
  }%
  \providecommand\rotatebox[2]{#2}%
  \ifx\svgwidth\undefined%
    \setlength{\unitlength}{362.65819022bp}%
    \ifx\svgscale\undefined%
      \relax%
    \else%
      \setlength{\unitlength}{\unitlength * \real{\svgscale}}%
    \fi%
  \else%
    \setlength{\unitlength}{\svgwidth}%
  \fi%
  \global\let\svgwidth\undefined%
  \global\let\svgscale\undefined%
  \makeatother%
  \begin{picture}(1,0.18258887)%
    \put(0,0){\includegraphics[width=\unitlength]{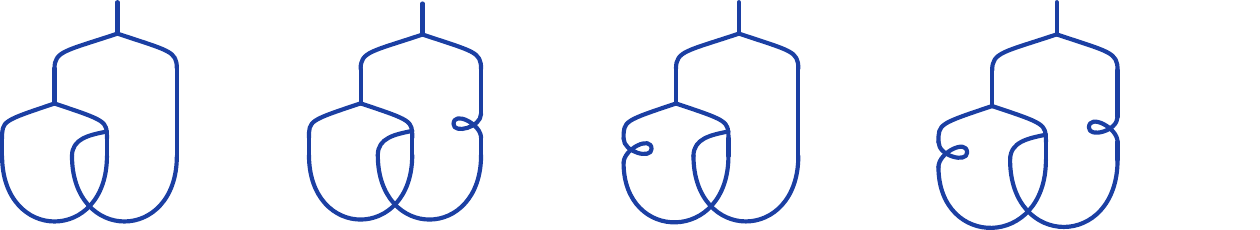}}%
    \put(0.15278445,0.06919336){\color[rgb]{0,0,0}\makebox(0,0)[lb]{\smash{$=\eta_1\, ,$}}}%
    \put(0.39510125,0.06807921){\color[rgb]{0,0,0}\makebox(0,0)[lb]{\smash{$=\eta_2 \, ,$ }}}%
    \put(0.6463314,0.06919336){\color[rgb]{0,0,0}\makebox(0,0)[lb]{\smash{$=\eta_3 \, ,$}}}%
    \put(0.89867568,0.06919336){\color[rgb]{0,0,0}\makebox(0,0)[lb]{\smash{$=\chi \, .$}}}%
  \end{picture}%
\endgroup%

\end{aligned}
\end{align}
%%%%%%%%%%%%%%%%%%%%%%%%%%%%%%%
A useful preliminary is the study of the two possible diagrams we can associate with the cylinder topology and which would correspond to its two different spin structures. These maps $\gls{A} \to A$ are depicted below and denoted $p$ and $n$ respectively. 
%%%%%%%%%%%%%%%%%%%%%%%%%%%%%%%%%%%%%%%%%%%%%%%%%%%%%
\begin{align}
\begin{aligned}
%% Creator: Inkscape 0.48.2, www.inkscape.org
%% PDF/EPS/PS + LaTeX output extension by Johan Engelen, 2010
%% Accompanies image file 'cyl_no_dehn.pdf' (pdf, eps, ps)
%%
%% To include the image in your LaTeX document, write
%%   \input{<filename>.pdf_tex}
%%  instead of
%%   \includegraphics{<filename>.pdf}
%% To scale the image, write
%%   \def\svgwidth{<desired width>}
%%   \input{<filename>.pdf_tex}
%%  instead of
%%   \includegraphics[width=<desired width>]{<filename>.pdf}
%%
%% Images with a different path to the parent latex file can
%% be accessed with the `import' package (which may need to be
%% installed) using
%%   \usepackage{import}
%% in the preamble, and then including the image with
%%   \import{<path to file>}{<filename>.pdf_tex}
%% Alternatively, one can specify
%%   \graphicspath{{<path to file>/}}
%% 
%% For more information, please see info/svg-inkscape on CTAN:
%%   http://tug.ctan.org/tex-archive/info/svg-inkscape
%%
\begingroup%
  \makeatletter%
  \providecommand\color[2][]{%
    \errmessage{(Inkscape) Color is used for the text in Inkscape, but the package 'color.sty' is not loaded}%
    \renewcommand\color[2][]{}%
  }%
  \providecommand\transparent[1]{%
    \errmessage{(Inkscape) Transparency is used (non-zero) for the text in Inkscape, but the package 'transparent.sty' is not loaded}%
    \renewcommand\transparent[1]{}%
  }%
  \providecommand\rotatebox[2]{#2}%
  \ifx\svgwidth\undefined%
    \setlength{\unitlength}{178.94314151bp}%
    \ifx\svgscale\undefined%
      \relax%
    \else%
      \setlength{\unitlength}{\unitlength * \real{\svgscale}}%
    \fi%
  \else%
    \setlength{\unitlength}{\svgwidth}%
  \fi%
  \global\let\svgwidth\undefined%
  \global\let\svgscale\undefined%
  \makeatother%
  \begin{picture}(1,0.4009363)%
    \put(0,0){\includegraphics[width=\unitlength]{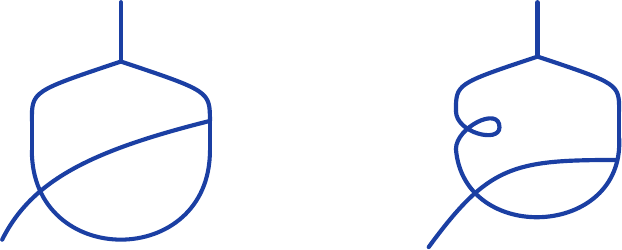}}%
  \end{picture}%
\endgroup%

\end{aligned}
\end{align} 
%%%%%%%%%%%%%%%%%%%%%%%%%%%%%%%%%%%%%%%%%%%%%%%%%%%%%%
We now define two linear subspaces of $\gls{A}$: $\mathcal{Z}_{\lambda}(A)$, the set of elements $a \in A$ satisfying $\gls{m}(b \otimes a)=m \circ \gls{lambda}(b \otimes a)$ for all $b \in A$, and $\overline{\mathcal{Z}}_{\lambda}(A)$, the set of elements $a \in A$ obeying equation $m(b \otimes a)=m \circ \lambda(\varphi(b) \otimes a)$ instead. In diagrammatic terms, the elements $a$ of $\mathcal{Z}_{\lambda}(A)$ or $\overline{\mathcal{Z}}_{\lambda}(A)$ satisfy either
\begin{align} \label{eq:Z_bar_lambda}
\begin{aligned}%% Creator: Inkscape 0.48.2, www.inkscape.org
%% PDF/EPS/PS + LaTeX output extension by Johan Engelen, 2010
%% Accompanies image file 'Z_lambda.pdf' (pdf, eps, ps)
%%
%% To include the image in your LaTeX document, write
%%   \input{<filename>.pdf_tex}
%%  instead of
%%   \includegraphics{<filename>.pdf}
%% To scale the image, write
%%   \def\svgwidth{<desired width>}
%%   \input{<filename>.pdf_tex}
%%  instead of
%%   \includegraphics[width=<desired width>]{<filename>.pdf}
%%
%% Images with a different path to the parent latex file can
%% be accessed with the `import' package (which may need to be
%% installed) using
%%   \usepackage{import}
%% in the preamble, and then including the image with
%%   \import{<path to file>}{<filename>.pdf_tex}
%% Alternatively, one can specify
%%   \graphicspath{{<path to file>/}}
%% 
%% For more information, please see info/svg-inkscape on CTAN:
%%   http://tug.ctan.org/tex-archive/info/svg-inkscape
%%
\begingroup%
  \makeatletter%
  \providecommand\color[2][]{%
    \errmessage{(Inkscape) Color is used for the text in Inkscape, but the package 'color.sty' is not loaded}%
    \renewcommand\color[2][]{}%
  }%
  \providecommand\transparent[1]{%
    \errmessage{(Inkscape) Transparency is used (non-zero) for the text in Inkscape, but the package 'transparent.sty' is not loaded}%
    \renewcommand\transparent[1]{}%
  }%
  \providecommand\rotatebox[2]{#2}%
  \ifx\svgwidth\undefined%
    \setlength{\unitlength}{120.6953125bp}%
    \ifx\svgscale\undefined%
      \relax%
    \else%
      \setlength{\unitlength}{\unitlength * \real{\svgscale}}%
    \fi%
  \else%
    \setlength{\unitlength}{\svgwidth}%
  \fi%
  \global\let\svgwidth\undefined%
  \global\let\svgscale\undefined%
  \makeatother%
  \begin{picture}(1,0.46966653)%
    \put(0,0){\includegraphics[width=\unitlength]{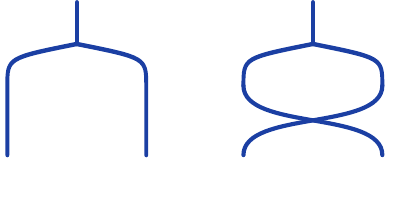}}%
    \put(0.42847433,0.23218979){\color[rgb]{0,0,0}\makebox(0,0)[lb]{\smash{$=$}}}%
    \put(-0.00236261,0.00682892){\color[rgb]{0,0,0}\makebox(0,0)[lb]{\smash{$b$}}}%
    \put(0.32905042,0.00682892){\color[rgb]{0,0,0}\makebox(0,0)[lb]{\smash{$a$}}}%
    \put(0.56103955,0.00682892){\color[rgb]{0,0,0}\makebox(0,0)[lb]{\smash{$b$}}}%
    \put(0.89245259,0.00682892){\color[rgb]{0,0,0}\makebox(0,0)[lb]{\smash{$a$}}}%
  \end{picture}%
\endgroup%
\end{aligned} \text{ , or} \hspace{5mm} 
\begin{aligned}%% Creator: Inkscape 0.48.2, www.inkscape.org
%% PDF/EPS/PS + LaTeX output extension by Johan Engelen, 2010
%% Accompanies image file 'Z_bar_lambda.pdf' (pdf, eps, ps)
%%
%% To include the image in your LaTeX document, write
%%   \input{<filename>.pdf_tex}
%%  instead of
%%   \includegraphics{<filename>.pdf}
%% To scale the image, write
%%   \def\svgwidth{<desired width>}
%%   \input{<filename>.pdf_tex}
%%  instead of
%%   \includegraphics[width=<desired width>]{<filename>.pdf}
%%
%% Images with a different path to the parent latex file can
%% be accessed with the `import' package (which may need to be
%% installed) using
%%   \usepackage{import}
%% in the preamble, and then including the image with
%%   \import{<path to file>}{<filename>.pdf_tex}
%% Alternatively, one can specify
%%   \graphicspath{{<path to file>/}}
%% 
%% For more information, please see info/svg-inkscape on CTAN:
%%   http://tug.ctan.org/tex-archive/info/svg-inkscape
%%
\begingroup%
  \makeatletter%
  \providecommand\color[2][]{%
    \errmessage{(Inkscape) Color is used for the text in Inkscape, but the package 'color.sty' is not loaded}%
    \renewcommand\color[2][]{}%
  }%
  \providecommand\transparent[1]{%
    \errmessage{(Inkscape) Transparency is used (non-zero) for the text in Inkscape, but the package 'transparent.sty' is not loaded}%
    \renewcommand\transparent[1]{}%
  }%
  \providecommand\rotatebox[2]{#2}%
  \ifx\svgwidth\undefined%
    \setlength{\unitlength}{120.6953125bp}%
    \ifx\svgscale\undefined%
      \relax%
    \else%
      \setlength{\unitlength}{\unitlength * \real{\svgscale}}%
    \fi%
  \else%
    \setlength{\unitlength}{\svgwidth}%
  \fi%
  \global\let\svgwidth\undefined%
  \global\let\svgscale\undefined%
  \makeatother%
  \begin{picture}(1,0.53594914)%
    \put(0,0){\includegraphics[width=\unitlength]{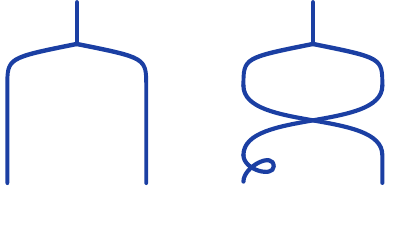}}%
    \put(0.42847433,0.29847239){\color[rgb]{0,0,0}\makebox(0,0)[lb]{\smash{$=$}}}%
    \put(-0.00236261,0.00682892){\color[rgb]{0,0,0}\makebox(0,0)[lb]{\smash{$b$}}}%
    \put(0.32905042,0.00682892){\color[rgb]{0,0,0}\makebox(0,0)[lb]{\smash{$a$}}}%
    \put(0.56103955,0.00682892){\color[rgb]{0,0,0}\makebox(0,0)[lb]{\smash{$b$}}}%
    \put(0.89245259,0.00682892){\color[rgb]{0,0,0}\makebox(0,0)[lb]{\smash{$a$}}}%
  \end{picture}%
\endgroup%
\end{aligned} .
\end{align}
\begin{lemma} \label{lem:projector}
The map $\gls{R}.p$ is a projector $A \to A$ with image $\mathcal{Z}_{\lambda}(A)$. The map $R.n$ is a projector $A \to A$ with image $\overline{\mathcal{Z}}_{\lambda}(A)$. Further, $\gls{phi} \circ p = p \circ \varphi = p$ and $\varphi \circ  n = n \circ \varphi$.  
\end{lemma} 
\begin{proof}
First, one must note that for all $a \in A$, $p(a) \in \mathcal{Z}_{\lambda}(A)$. 
%%%%%%%%%%%%%%%%%%%%%%%%%%%%%%
\begin{align}
\hspace{8mm}
%% Creator: Inkscape 0.48.2, www.inkscape.org
%% PDF/EPS/PS + LaTeX output extension by Johan Engelen, 2010
%% Accompanies image file 'cylproof.pdf' (pdf, eps, ps)
%%
%% To include the image in your LaTeX document, write
%%   \input{<filename>.pdf_tex}
%%  instead of
%%   \includegraphics{<filename>.pdf}
%% To scale the image, write
%%   \def\svgwidth{<desired width>}
%%   \input{<filename>.pdf_tex}
%%  instead of
%%   \includegraphics[width=<desired width>]{<filename>.pdf}
%%
%% Images with a different path to the parent latex file can
%% be accessed with the `import' package (which may need to be
%% installed) using
%%   \usepackage{import}
%% in the preamble, and then including the image with
%%   \import{<path to file>}{<filename>.pdf_tex}
%% Alternatively, one can specify
%%   \graphicspath{{<path to file>/}}
%% 
%% For more information, please see info/svg-inkscape on CTAN:
%%   http://tug.ctan.org/tex-archive/info/svg-inkscape
%%
\begingroup%
  \makeatletter%
  \providecommand\color[2][]{%
    \errmessage{(Inkscape) Color is used for the text in Inkscape, but the package 'color.sty' is not loaded}%
    \renewcommand\color[2][]{}%
  }%
  \providecommand\transparent[1]{%
    \errmessage{(Inkscape) Transparency is used (non-zero) for the text in Inkscape, but the package 'transparent.sty' is not loaded}%
    \renewcommand\transparent[1]{}%
  }%
  \providecommand\rotatebox[2]{#2}%
  \ifx\svgwidth\undefined%
    \setlength{\unitlength}{457.28384906bp}%
    \ifx\svgscale\undefined%
      \relax%
    \else%
      \setlength{\unitlength}{\unitlength * \real{\svgscale}}%
    \fi%
  \else%
    \setlength{\unitlength}{\svgwidth}%
  \fi%
  \global\let\svgwidth\undefined%
  \global\let\svgscale\undefined%
  \makeatother%
  \begin{picture}(1,0.29939582)%
    \put(0,0){\includegraphics[width=\unitlength]{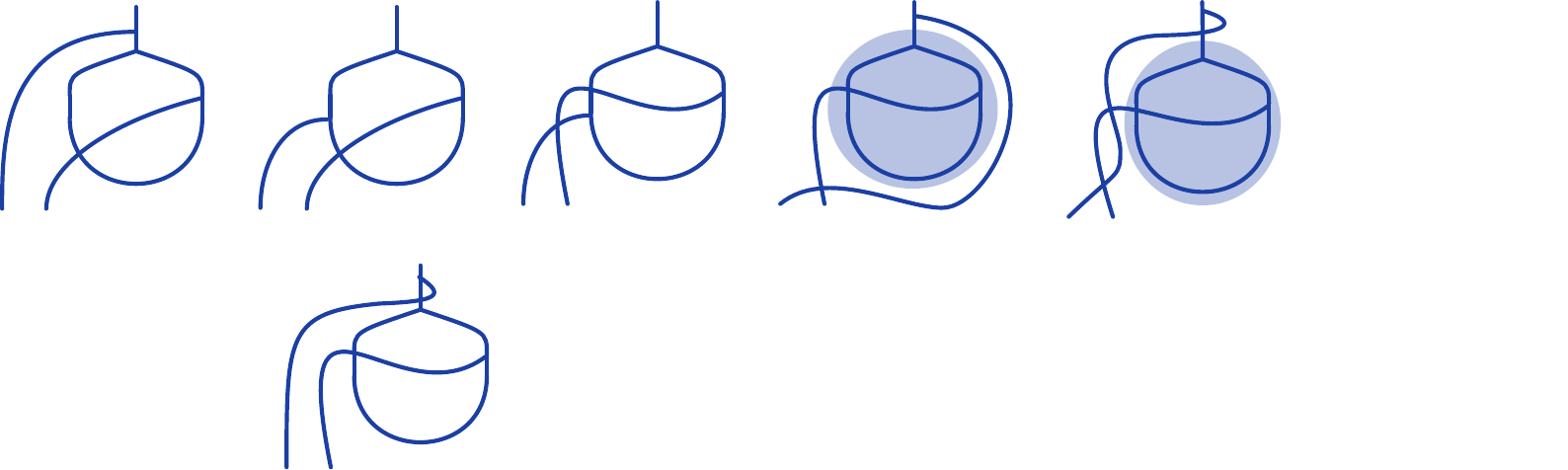}}%
    \put(0.1378184,0.21229814){\color[rgb]{0,0,0}\makebox(0,0)[lb]{\smash{$\overset{\text{\eqref{fig:diagram_mult2}}}{=}$}}}%
    \put(0.30401715,0.21229814){\color[rgb]{0,0,0}\makebox(0,0)[lb]{\smash{$\overset{\text{\ref{fig:axiom_mult}}}{=}$}}}%
    \put(0.47021589,0.21229814){\color[rgb]{0,0,0}\makebox(0,0)[lb]{\smash{$\overset{\text{\eqref{fig:diagram_mult2}}}{=}$}}}%
    \put(0.65390923,0.21229814){\color[rgb]{0,0,0}\makebox(0,0)[lb]{\smash{$\overset{\text{\eqref{eq:closed_below}}}{=}$}}}%
    \put(0.13700807,0.04434989){\color[rgb]{0,0,0}\makebox(0,0)[lb]{\smash{$\overset{\text{\ref{fig:axiom_square}}}{=}$}}}%
  \end{picture}%
\endgroup%

\end{align}
%%%%%%%%%%%%%%%%%%%%%%%%%%%%%%%
One can then further conclude that if $a \in \mathcal{Z}_{\lambda}(A)$ then $R.p(a)=a$. 
%%%%%%%%%%%%%%%%%%%%%%%%%%%%%%%
\begin{align}
\hspace{13mm}      
%% Creator: Inkscape 0.48.2, www.inkscape.org
%% PDF/EPS/PS + LaTeX output extension by Johan Engelen, 2010
%% Accompanies image file 'cylproof2.pdf' (pdf, eps, ps)
%%
%% To include the image in your LaTeX document, write
%%   \input{<filename>.pdf_tex}
%%  instead of
%%   \includegraphics{<filename>.pdf}
%% To scale the image, write
%%   \def\svgwidth{<desired width>}
%%   \input{<filename>.pdf_tex}
%%  instead of
%%   \includegraphics[width=<desired width>]{<filename>.pdf}
%%
%% Images with a different path to the parent latex file can
%% be accessed with the `import' package (which may need to be
%% installed) using
%%   \usepackage{import}
%% in the preamble, and then including the image with
%%   \import{<path to file>}{<filename>.pdf_tex}
%% Alternatively, one can specify
%%   \graphicspath{{<path to file>/}}
%% 
%% For more information, please see info/svg-inkscape on CTAN:
%%   http://tug.ctan.org/tex-archive/info/svg-inkscape
%%
\begingroup%
  \makeatletter%
  \providecommand\color[2][]{%
    \errmessage{(Inkscape) Color is used for the text in Inkscape, but the package 'color.sty' is not loaded}%
    \renewcommand\color[2][]{}%
  }%
  \providecommand\transparent[1]{%
    \errmessage{(Inkscape) Transparency is used (non-zero) for the text in Inkscape, but the package 'transparent.sty' is not loaded}%
    \renewcommand\transparent[1]{}%
  }%
  \providecommand\rotatebox[2]{#2}%
  \ifx\svgwidth\undefined%
    \setlength{\unitlength}{434.92403564bp}%
    \ifx\svgscale\undefined%
      \relax%
    \else%
      \setlength{\unitlength}{\unitlength * \real{\svgscale}}%
    \fi%
  \else%
    \setlength{\unitlength}{\svgwidth}%
  \fi%
  \global\let\svgwidth\undefined%
  \global\let\svgscale\undefined%
  \makeatother%
  \begin{picture}(1,0.1455194)%
    \put(0,0){\includegraphics[width=\unitlength]{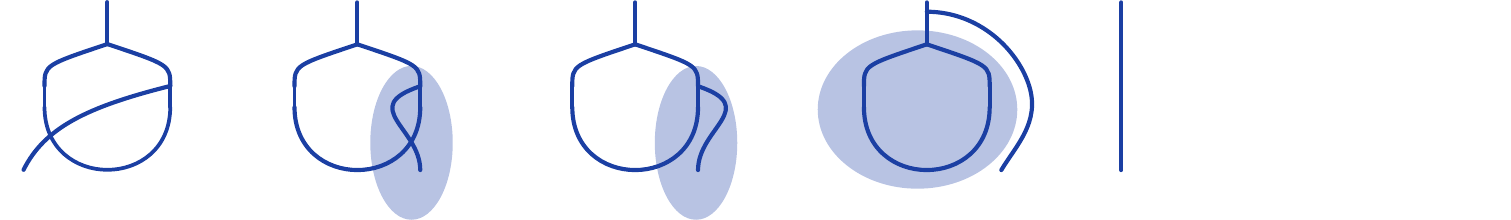}}%
    \put(-0.00065565,0.07557731){\color[rgb]{0,0,0}\makebox(0,0)[lb]{\smash{$R$}}}%
    \put(0.16770678,0.07541553){\color[rgb]{0,0,0}\makebox(0,0)[lb]{\smash{$R$}}}%
    \put(0.3506765,0.07541553){\color[rgb]{0,0,0}\makebox(0,0)[lb]{\smash{$R$}}}%
    \put(0.49782864,0.07541553){\color[rgb]{0,0,0}\makebox(0,0)[lb]{\smash{$\overset{\text{\eqref{fig:diagram_mult2}}}{=}$}}}%
    \put(0.54706334,0.07541553){\color[rgb]{0,0,0}\makebox(0,0)[lb]{\smash{$R$}}}%
    \put(0.01038074,0.00751945){\color[rgb]{0,0,0}\makebox(0,0)[lb]{\smash{$a$}}}%
    \put(0.27255262,0.00735767){\color[rgb]{0,0,0}\makebox(0,0)[lb]{\smash{$a$}}}%
    \put(0.4573618,0.00735767){\color[rgb]{0,0,0}\makebox(0,0)[lb]{\smash{$a$}}}%
    \put(0.65742749,0.00735767){\color[rgb]{0,0,0}\makebox(0,0)[lb]{\smash{$a$}}}%
    \put(0.73771439,0.00735767){\color[rgb]{0,0,0}\makebox(0,0)[lb]{\smash{$a$}}}%
    \put(0.69237624,0.07541553){\color[rgb]{0,0,0}\makebox(0,0)[lb]{\smash{$\overset{\text{\eqref{eq:diag_1}}}{=}$}}}%
    \put(0.30566185,0.07541553){\color[rgb]{0,0,0}\makebox(0,0)[lb]{\smash{$\overset{\eqref{eq:Z_bar_lambda}}{=}$}}}%
    \put(0.12258427,0.07557731){\color[rgb]{0,0,0}\makebox(0,0)[lb]{\smash{$\overset{\text{\ref{fig:axiom_form}}}{=}$}}}%
  \end{picture}%
\endgroup%

\end{align}
%%%%%%%%%%%%%%%%%%%%%%%%%%%%%%%
These two properties are enough to establish $R.p$ as a projector onto $\mathcal{Z}_{\lambda}(A)$. Notice that since $p(a) \in \mathcal{Z}_{\lambda}(A)$ for all $a \in A$ and $R.p(b)=b$ for all $b \in \mathcal{Z}_{\lambda}(A) \subset A$ by setting $b=R.p(a)$ we learn $R^2.p^2(a)=R.p(a)$. In other words, the map $R.p$ is idempotent and, therefore, a projector. Also, recall that if $\text{Im}(R.p)=\mathcal{Z}_{\lambda}(A)$ as claimed, for all $a \in \mathcal{Z}_{\lambda}(A)$ there should exist $b \in \gls{A}$ such that $\gls{R}.p(b)=a$. Since for all $a \in \mathcal{Z}_{\lambda}(A)$ the identity $R.p(a)=a$ holds the choice $b=a$ concludes the proof. It is now shown that for all $a \in A$ the element $n(a)$ belongs to $\overline{\mathcal{Z}}_{\lambda}(A)$. 
%%%%%%%%%%%%%%%%%%%%%%%%%%%%%%%
\begin{align}
%% Creator: Inkscape 0.48.2, www.inkscape.org
%% PDF/EPS/PS + LaTeX output extension by Johan Engelen, 2010
%% Accompanies image file 'nproof.pdf' (pdf, eps, ps)
%%
%% To include the image in your LaTeX document, write
%%   \input{<filename>.pdf_tex}
%%  instead of
%%   \includegraphics{<filename>.pdf}
%% To scale the image, write
%%   \def\svgwidth{<desired width>}
%%   \input{<filename>.pdf_tex}
%%  instead of
%%   \includegraphics[width=<desired width>]{<filename>.pdf}
%%
%% Images with a different path to the parent latex file can
%% be accessed with the `import' package (which may need to be
%% installed) using
%%   \usepackage{import}
%% in the preamble, and then including the image with
%%   \import{<path to file>}{<filename>.pdf_tex}
%% Alternatively, one can specify
%%   \graphicspath{{<path to file>/}}
%% 
%% For more information, please see info/svg-inkscape on CTAN:
%%   http://tug.ctan.org/tex-archive/info/svg-inkscape
%%
\begingroup%
  \makeatletter%
  \providecommand\color[2][]{%
    \errmessage{(Inkscape) Color is used for the text in Inkscape, but the package 'color.sty' is not loaded}%
    \renewcommand\color[2][]{}%
  }%
  \providecommand\transparent[1]{%
    \errmessage{(Inkscape) Transparency is used (non-zero) for the text in Inkscape, but the package 'transparent.sty' is not loaded}%
    \renewcommand\transparent[1]{}%
  }%
  \providecommand\rotatebox[2]{#2}%
  \ifx\svgwidth\undefined%
    \setlength{\unitlength}{363.49643232bp}%
    \ifx\svgscale\undefined%
      \relax%
    \else%
      \setlength{\unitlength}{\unitlength * \real{\svgscale}}%
    \fi%
  \else%
    \setlength{\unitlength}{\svgwidth}%
  \fi%
  \global\let\svgwidth\undefined%
  \global\let\svgscale\undefined%
  \makeatother%
  \begin{picture}(1,0.32357305)%
    \put(0,0){\includegraphics[width=\unitlength]{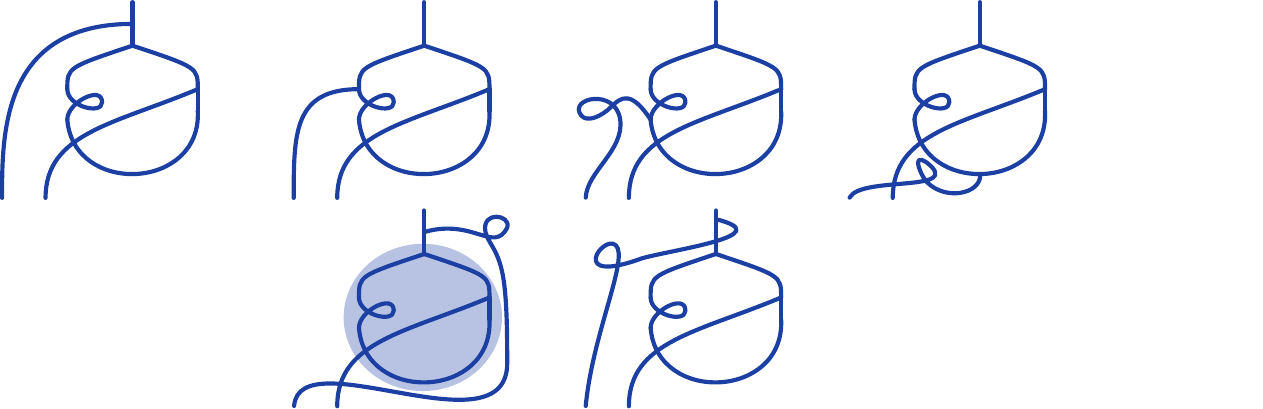}}%
    \put(0.17103941,0.23383433){\color[rgb]{0,0,0}\makebox(0,0)[lb]{\smash{$\overset{\text{\eqref{fig:diagram_mult2}}}{=}$}}}%
    \put(0.40212833,0.23383433){\color[rgb]{0,0,0}\makebox(0,0)[lb]{\smash{$\overset{\text{\ref{it:phi_two}}}{=}$}}}%
    \put(0.64422149,0.23383433){\color[rgb]{0,0,0}\makebox(0,0)[lb]{\smash{$\overset{\text{\ref{it:phi_three}}}{=}$}}}%
    \put(0.4131325,0.06877088){\color[rgb]{0,0,0}\makebox(0,0)[lb]{\smash{$\overset{\text{\eqref{eq:closed_below}}}{=}$}}}%
    \put(0.17103941,0.06877088){\color[rgb]{0,0,0}\makebox(0,0)[lb]{\smash{$\overset{\text{\eqref{fig:diagram_mult2}}}{=}$}}}%
  \end{picture}%
\endgroup%

\label{fig:nproof}	
\end{align}
%%%%%%%%%%%%%%%%%%%%%%%%%%%%%%%
Furthermore, it is established that if $a \in \overline{\mathcal{Z}}_{\lambda}(A)$ then $R.n(a)=a$. 
%%%%%%%%%%%%%%%%%%%%%%%%%%%%%%%
\begin{align}
\hspace{10mm}
%% Creator: Inkscape 0.48.2, www.inkscape.org
%% PDF/EPS/PS + LaTeX output extension by Johan Engelen, 2010
%% Accompanies image file 'nproof2.pdf' (pdf, eps, ps)
%%
%% To include the image in your LaTeX document, write
%%   \input{<filename>.pdf_tex}
%%  instead of
%%   \includegraphics{<filename>.pdf}
%% To scale the image, write
%%   \def\svgwidth{<desired width>}
%%   \input{<filename>.pdf_tex}
%%  instead of
%%   \includegraphics[width=<desired width>]{<filename>.pdf}
%%
%% Images with a different path to the parent latex file can
%% be accessed with the `import' package (which may need to be
%% installed) using
%%   \usepackage{import}
%% in the preamble, and then including the image with
%%   \import{<path to file>}{<filename>.pdf_tex}
%% Alternatively, one can specify
%%   \graphicspath{{<path to file>/}}
%% 
%% For more information, please see info/svg-inkscape on CTAN:
%%   http://tug.ctan.org/tex-archive/info/svg-inkscape
%%
\begingroup%
  \makeatletter%
  \providecommand\color[2][]{%
    \errmessage{(Inkscape) Color is used for the text in Inkscape, but the package 'color.sty' is not loaded}%
    \renewcommand\color[2][]{}%
  }%
  \providecommand\transparent[1]{%
    \errmessage{(Inkscape) Transparency is used (non-zero) for the text in Inkscape, but the package 'transparent.sty' is not loaded}%
    \renewcommand\transparent[1]{}%
  }%
  \providecommand\rotatebox[2]{#2}%
  \ifx\svgwidth\undefined%
    \setlength{\unitlength}{443.6767334bp}%
    \ifx\svgscale\undefined%
      \relax%
    \else%
      \setlength{\unitlength}{\unitlength * \real{\svgscale}}%
    \fi%
  \else%
    \setlength{\unitlength}{\svgwidth}%
  \fi%
  \global\let\svgwidth\undefined%
  \global\let\svgscale\undefined%
  \makeatother%
  \begin{picture}(1,0.16922011)%
    \put(0,0){\includegraphics[width=\unitlength]{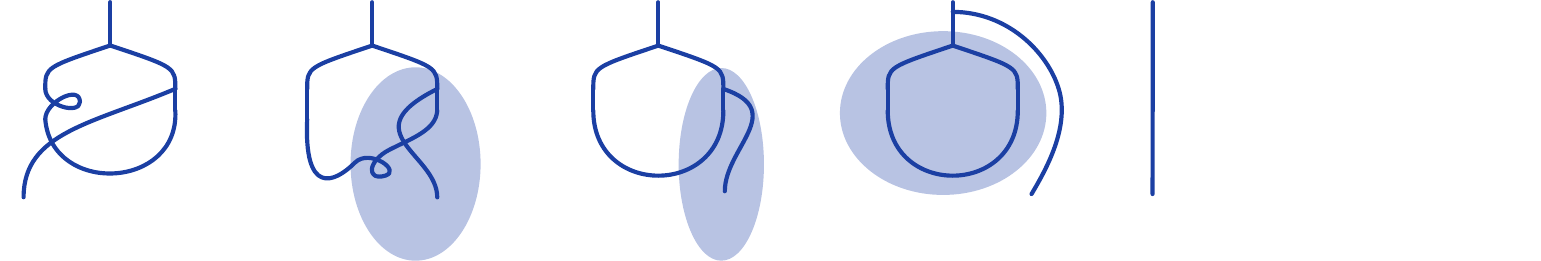}}%
    \put(-0.00064271,0.09769345){\color[rgb]{0,0,0}\makebox(0,0)[lb]{\smash{$R$}}}%
    \put(0.17391667,0.09807918){\color[rgb]{0,0,0}\makebox(0,0)[lb]{\smash{$R$}}}%
    \put(0.35946736,0.09853453){\color[rgb]{0,0,0}\makebox(0,0)[lb]{\smash{$R$}}}%
    \put(0.31744249,0.09015578){\color[rgb]{0,0,0}\makebox(0,0)[lb]{\smash{$\overset{\text{\eqref{eq:Z_bar_lambda}}}{=}$}}}%
    \put(0.55045692,0.09807918){\color[rgb]{0,0,0}\makebox(0,0)[lb]{\smash{$R$}}}%
    \put(0.49935033,0.09015578){\color[rgb]{0,0,0}\makebox(0,0)[lb]{\smash{$\overset{\text{\eqref{fig:diagram_mult2}}}{=}$}}}%
    \put(0.69844493,0.09015578){\color[rgb]{0,0,0}\makebox(0,0)[lb]{\smash{$\overset{\text{\eqref{eq:diag_1}}}{=}$}}}%
    \put(0.01081869,0.01269146){\color[rgb]{0,0,0}\makebox(0,0)[lb]{\smash{$a$}}}%
    \put(0.27777397,0.01262191){\color[rgb]{0,0,0}\makebox(0,0)[lb]{\smash{$a$}}}%
    \put(0.46509115,0.01262191){\color[rgb]{0,0,0}\makebox(0,0)[lb]{\smash{$a$}}}%
    \put(0.66418564,0.01262191){\color[rgb]{0,0,0}\makebox(0,0)[lb]{\smash{$a$}}}%
    \put(0.74274593,0.01262191){\color[rgb]{0,0,0}\makebox(0,0)[lb]{\smash{$a$}}}%
    \put(0.12621805,0.09022539){\color[rgb]{0,0,0}\makebox(0,0)[lb]{\smash{$\overset{\substack{\text{\ref{fig:axiom_form}}\\+\\\text{\ref{it:phi_four}}}}{=}$}}}%
  \end{picture}%
\endgroup%

\label{fig:nproof2}	
\end{align}
%%%%%%%%%%%%%%%%%%%%%%%%%%%%%%%
An analysis similar to the one conducted for $p$ allows us to conclude that $R.n$ is a projector onto $\overline{\mathcal{Z}}_{\lambda}(A)$. The proofs $\gls{phi} \circ p = p \circ \varphi$ and $p \circ \varphi = p$ are accomplished by direct composition:
%%%%%%%%%%%%%%%%%%%%%%%%%%%%%%%
\begin{align}
\hspace{15mm}
&%% Creator: Inkscape 0.48.2, www.inkscape.org
%% PDF/EPS/PS + LaTeX output extension by Johan Engelen, 2010
%% Accompanies image file 'p_phi.pdf' (pdf, eps, ps)
%%
%% To include the image in your LaTeX document, write
%%   \input{<filename>.pdf_tex}
%%  instead of
%%   \includegraphics{<filename>.pdf}
%% To scale the image, write
%%   \def\svgwidth{<desired width>}
%%   \input{<filename>.pdf_tex}
%%  instead of
%%   \includegraphics[width=<desired width>]{<filename>.pdf}
%%
%% Images with a different path to the parent latex file can
%% be accessed with the `import' package (which may need to be
%% installed) using
%%   \usepackage{import}
%% in the preamble, and then including the image with
%%   \import{<path to file>}{<filename>.pdf_tex}
%% Alternatively, one can specify
%%   \graphicspath{{<path to file>/}}
%% 
%% For more information, please see info/svg-inkscape on CTAN:
%%   http://tug.ctan.org/tex-archive/info/svg-inkscape
%%
\begingroup%
  \makeatletter%
  \providecommand\color[2][]{%
    \errmessage{(Inkscape) Color is used for the text in Inkscape, but the package 'color.sty' is not loaded}%
    \renewcommand\color[2][]{}%
  }%
  \providecommand\transparent[1]{%
    \errmessage{(Inkscape) Transparency is used (non-zero) for the text in Inkscape, but the package 'transparent.sty' is not loaded}%
    \renewcommand\transparent[1]{}%
  }%
  \providecommand\rotatebox[2]{#2}%
  \ifx\svgwidth\undefined%
    \setlength{\unitlength}{382.07838031bp}%
    \ifx\svgscale\undefined%
      \relax%
    \else%
      \setlength{\unitlength}{\unitlength * \real{\svgscale}}%
    \fi%
  \else%
    \setlength{\unitlength}{\svgwidth}%
  \fi%
  \global\let\svgwidth\undefined%
  \global\let\svgscale\undefined%
  \makeatother%
  \begin{picture}(1,0.17574442)%
    \put(0,0){\includegraphics[width=\unitlength]{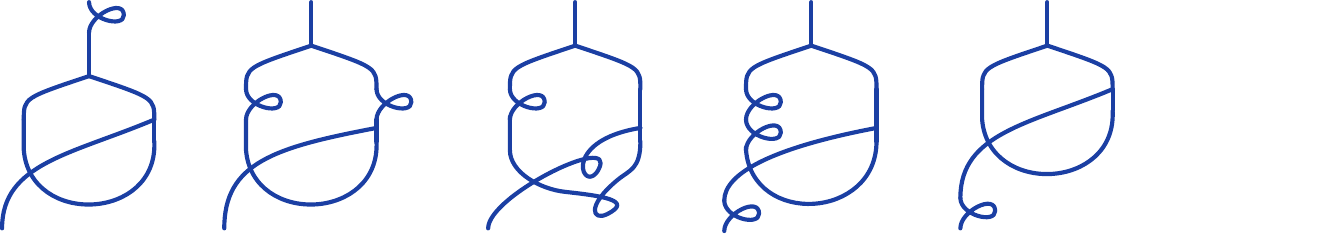}}%
    \put(0.12993345,0.06943194){\color[rgb]{0,0,0}\makebox(0,0)[lb]{\smash{$\overset{\text{\ref{it:phi_two}}}{=}$}}}%
    \put(0.31837646,0.06943194){\color[rgb]{0,0,0}\makebox(0,0)[lb]{\smash{$\overset{\text{\ref{it:phi_two}}}{=}$}}}%
    \put(0.49635041,0.06943194){\color[rgb]{0,0,0}\makebox(0,0)[lb]{\smash{$\overset{\text{\ref{it:phi_three}}}{=}$}}}%
    \put(0.67432443,0.06943194){\color[rgb]{0,0,0}\makebox(0,0)[lb]{\smash{$\overset{\text{\ref{it:phi_one}}}{=}$}}}%
    \put(0.8732365,0.06943194){\color[rgb]{0,0,0}\makebox(0,0)[lb]{\smash{,}}}%
  \end{picture}%
\endgroup%
\\
\hspace{3mm}
&%% Creator: Inkscape 0.48.2, www.inkscape.org
%% PDF/EPS/PS + LaTeX output extension by Johan Engelen, 2010
%% Accompanies image file 'cylcomp.pdf' (pdf, eps, ps)
%%
%% To include the image in your LaTeX document, write
%%   \input{<filename>.pdf_tex}
%%  instead of
%%   \includegraphics{<filename>.pdf}
%% To scale the image, write
%%   \def\svgwidth{<desired width>}
%%   \input{<filename>.pdf_tex}
%%  instead of
%%   \includegraphics[width=<desired width>]{<filename>.pdf}
%%
%% Images with a different path to the parent latex file can
%% be accessed with the `import' package (which may need to be
%% installed) using
%%   \usepackage{import}
%% in the preamble, and then including the image with
%%   \import{<path to file>}{<filename>.pdf_tex}
%% Alternatively, one can specify
%%   \graphicspath{{<path to file>/}}
%% 
%% For more information, please see info/svg-inkscape on CTAN:
%%   http://tug.ctan.org/tex-archive/info/svg-inkscape
%%
\begingroup%
  \makeatletter%
  \providecommand\color[2][]{%
    \errmessage{(Inkscape) Color is used for the text in Inkscape, but the package 'color.sty' is not loaded}%
    \renewcommand\color[2][]{}%
  }%
  \providecommand\transparent[1]{%
    \errmessage{(Inkscape) Transparency is used (non-zero) for the text in Inkscape, but the package 'transparent.sty' is not loaded}%
    \renewcommand\transparent[1]{}%
  }%
  \providecommand\rotatebox[2]{#2}%
  \ifx\svgwidth\undefined%
    \setlength{\unitlength}{336.79296502bp}%
    \ifx\svgscale\undefined%
      \relax%
    \else%
      \setlength{\unitlength}{\unitlength * \real{\svgscale}}%
    \fi%
  \else%
    \setlength{\unitlength}{\svgwidth}%
  \fi%
  \global\let\svgwidth\undefined%
  \global\let\svgscale\undefined%
  \makeatother%
  \begin{picture}(1,0.26184664)%
    \put(0,0){\includegraphics[width=\unitlength]{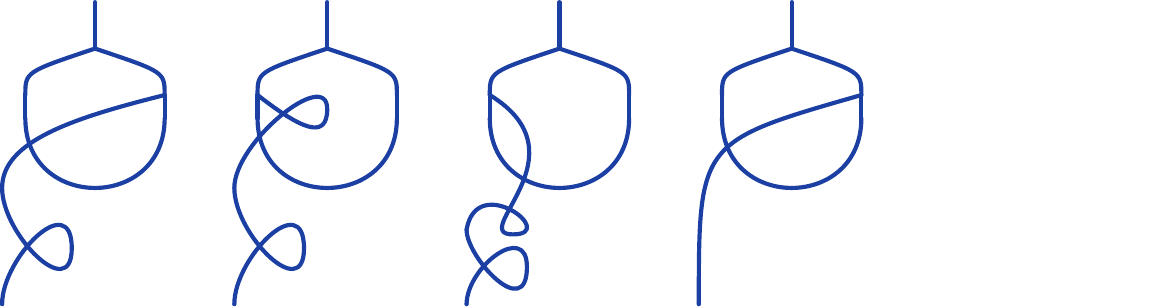}}%
    \put(0.15439751,0.15439746){\color[rgb]{0,0,0}\makebox(0,0)[lb]{\smash{$\overset{\substack{\text{\eqref{eq:prop_sep}}\\+\\\text{\ref{fig:axiom_mult}}}}{=}$}}}%
    \put(0.35630197,0.15439746){\color[rgb]{0,0,0}\makebox(0,0)[lb]{\smash{$\overset{\text{\ref{it:phi_three}}}{=}$}}}%
    \put(0.55820642,0.15439746){\color[rgb]{0,0,0}\makebox(0,0)[lb]{\smash{$\overset{\text{\ref{it:phi_one}}}{=}$}}}%
    \put(0.77198761,0.15439746){\color[rgb]{0,0,0}\makebox(0,0)[lb]{\smash{.}}}%
  \end{picture}%
\endgroup%
	
\end{align}
%%%%%%%%%%%%%%%%%%%%%%%%%%%%%%%
Showing the identity $\gls{phi} \circ n = n \circ \varphi$ holds concludes the proof of the lemma.
%%%%%%%%%%%%%%%%%%%%%%%%%%%%%%%
\begin{align}
\hspace{15mm}
%% Creator: Inkscape 0.48.2, www.inkscape.org
%% PDF/EPS/PS + LaTeX output extension by Johan Engelen, 2010
%% Accompanies image file 'proof_n1_n2.pdf' (pdf, eps, ps)
%%
%% To include the image in your LaTeX document, write
%%   \input{<filename>.pdf_tex}
%%  instead of
%%   \includegraphics{<filename>.pdf}
%% To scale the image, write
%%   \def\svgwidth{<desired width>}
%%   \input{<filename>.pdf_tex}
%%  instead of
%%   \includegraphics[width=<desired width>]{<filename>.pdf}
%%
%% Images with a different path to the parent latex file can
%% be accessed with the `import' package (which may need to be
%% installed) using
%%   \usepackage{import}
%% in the preamble, and then including the image with
%%   \import{<path to file>}{<filename>.pdf_tex}
%% Alternatively, one can specify
%%   \graphicspath{{<path to file>/}}
%% 
%% For more information, please see info/svg-inkscape on CTAN:
%%   http://tug.ctan.org/tex-archive/info/svg-inkscape
%%
\begingroup%
  \makeatletter%
  \providecommand\color[2][]{%
    \errmessage{(Inkscape) Color is used for the text in Inkscape, but the package 'color.sty' is not loaded}%
    \renewcommand\color[2][]{}%
  }%
  \providecommand\transparent[1]{%
    \errmessage{(Inkscape) Transparency is used (non-zero) for the text in Inkscape, but the package 'transparent.sty' is not loaded}%
    \renewcommand\transparent[1]{}%
  }%
  \providecommand\rotatebox[2]{#2}%
  \ifx\svgwidth\undefined%
    \setlength{\unitlength}{369.27670313bp}%
    \ifx\svgscale\undefined%
      \relax%
    \else%
      \setlength{\unitlength}{\unitlength * \real{\svgscale}}%
    \fi%
  \else%
    \setlength{\unitlength}{\svgwidth}%
  \fi%
  \global\let\svgwidth\undefined%
  \global\let\svgscale\undefined%
  \makeatother%
  \begin{picture}(1,0.18017007)%
    \put(0,0){\includegraphics[width=\unitlength]{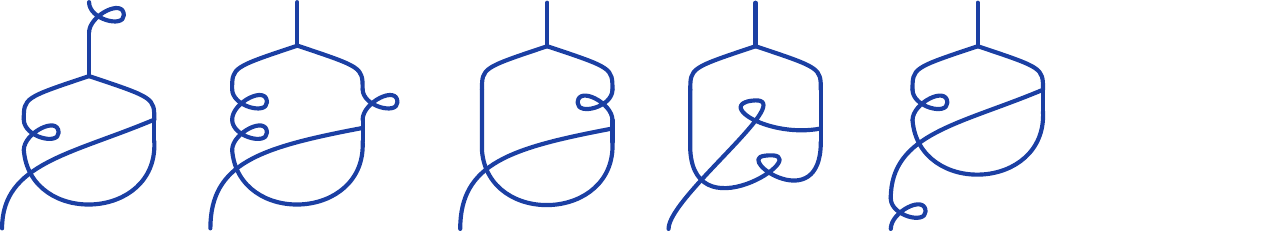}}%
    \put(0.32495956,0.08705956){\color[rgb]{0,0,0}\makebox(0,0)[lb]{\smash{$\overset{\text{\ref{it:phi_one}}}{=}$}}}%
    \put(0.12998381,0.08705956){\color[rgb]{0,0,0}\makebox(0,0)[lb]{\smash{$\overset{\text{\ref{it:phi_two}}}{=}$}}}%
    \put(0.48731069,0.08661759){\color[rgb]{0,0,0}\makebox(0,0)[lb]{\smash{$\overset{\text{\ref{it:phi_two}}}{=}$}}}%
    \put(0.64979048,0.08661759){\color[rgb]{0,0,0}\makebox(0,0)[lb]{\smash{$\overset{\text{\ref{it:phi_three}}}{=}$}}}%
  \end{picture}%
\endgroup%

\end{align}
%%%%%%%%%%%%%%%%%%%%%%%%%%%%%%%
\end{proof}
Curl-free models are in one-to-one correspondence with models for which $\mathcal{Z}_{\lambda}(A)=\overline{\mathcal{Z}}_{\lambda}(A)$.
\begin{lemma}
The spaces $\mathcal{Z}_{\lambda}(A)$, $\overline{\mathcal{Z}}_{\lambda}(A)$ coincide if and only if $\varphi=\iden$.
\end{lemma}
\begin{proof}
Showing that $\varphi=\text{id}$ implies $\mathcal{Z}_{\lambda}(A)=\overline{\mathcal{Z}}_{\lambda}(A)$ is straightforward; it follows from the identity $p=n$ for curl-free models. The reverse is more delicate. First note that independently of the choice of crossing $\gls{lambda}(a \otimes 1)= \gls{unit} \otimes a$ for all $a \in \gls{A}$; this follows from axioms \ref{fig:axiom_mult} and \ref{fig:axiom_square} and the description of the algebra unit as $1=\gls{R}.\gls{m}(\gls{Bb})$. Notice, therefore, that $m(a \otimes 1)=a$ and $m \circ  \lambda (a \otimes 1)=a$ for all $a \in A$ -- in other words, $1 \in \gls{ZA}$. A necessary condition to have $\mathcal{Z}_{\lambda}(A)=\overline{\mathcal{Z}}_{\lambda}(A)$ is then $1 \in \overline{\mathcal{Z}}_{\lambda}(A)$. However, $m \circ \lambda (\varphi(a) \otimes 1) = \varphi(a)$ and having $1 \in \overline{\mathcal{Z}}_{\lambda}(A)$ would imply $\varphi(a)=a$ for all $a \in A$. 
\end{proof}
It will now be easy to verify the identity $\eta_1=\eta_2=\eta_3$ is satisfied; the notation $\gls{eta}$ is used for any of these maps. To see how the result holds note that one of the relations, $\eta_1=\eta_2$, is trivial -- it follows from $p \circ \varphi=p$. The proof for the remaining equation, $\eta_3=\eta_1$ is depicted below. 
%%%%%%%%%%%%%%%%%%%%%%%%%%%%%%%
\begin{align}
\hspace{15mm}
%% Creator: Inkscape 0.48.2, www.inkscape.org
%% PDF/EPS/PS + LaTeX output extension by Johan Engelen, 2010
%% Accompanies image file 'elementproof.pdf' (pdf, eps, ps)
%%
%% To include the image in your LaTeX document, write
%%   \input{<filename>.pdf_tex}
%%  instead of
%%   \includegraphics{<filename>.pdf}
%% To scale the image, write
%%   \def\svgwidth{<desired width>}
%%   \input{<filename>.pdf_tex}
%%  instead of
%%   \includegraphics[width=<desired width>]{<filename>.pdf}
%%
%% Images with a different path to the parent latex file can
%% be accessed with the `import' package (which may need to be
%% installed) using
%%   \usepackage{import}
%% in the preamble, and then including the image with
%%   \import{<path to file>}{<filename>.pdf_tex}
%% Alternatively, one can specify
%%   \graphicspath{{<path to file>/}}
%% 
%% For more information, please see info/svg-inkscape on CTAN:
%%   http://tug.ctan.org/tex-archive/info/svg-inkscape
%%
\begingroup%
  \makeatletter%
  \providecommand\color[2][]{%
    \errmessage{(Inkscape) Color is used for the text in Inkscape, but the package 'color.sty' is not loaded}%
    \renewcommand\color[2][]{}%
  }%
  \providecommand\transparent[1]{%
    \errmessage{(Inkscape) Transparency is used (non-zero) for the text in Inkscape, but the package 'transparent.sty' is not loaded}%
    \renewcommand\transparent[1]{}%
  }%
  \providecommand\rotatebox[2]{#2}%
  \ifx\svgwidth\undefined%
    \setlength{\unitlength}{399.58114557bp}%
    \ifx\svgscale\undefined%
      \relax%
    \else%
      \setlength{\unitlength}{\unitlength * \real{\svgscale}}%
    \fi%
  \else%
    \setlength{\unitlength}{\svgwidth}%
  \fi%
  \global\let\svgwidth\undefined%
  \global\let\svgscale\undefined%
  \makeatother%
  \begin{picture}(1,0.17869774)%
    \put(0,0){\includegraphics[width=\unitlength]{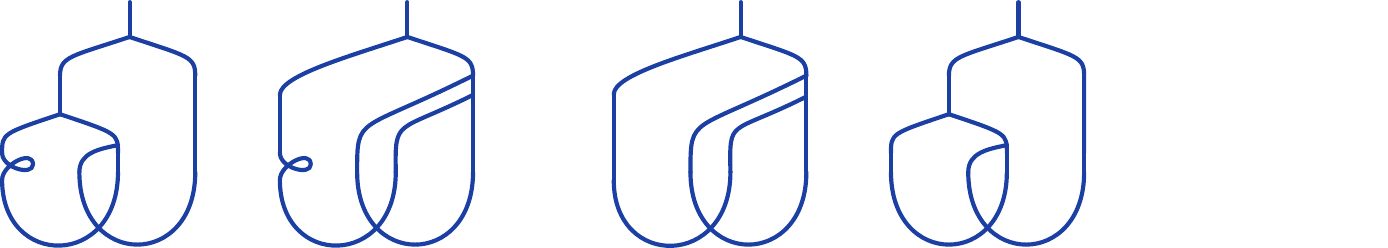}}%
    \put(0.59392048,0.06701521){\color[rgb]{0,0,0}\makebox(0,0)[lb]{\smash{$\overset{\text{\eqref{fig:diagram_mult2}}}{=}$}}}%
    \put(0.15345923,0.06701521){\color[rgb]{0,0,0}\makebox(0,0)[lb]{\smash{$\overset{\text{\eqref{fig:diagram_mult2}}}{=}$}}}%
    \put(0.35366891,0.06701521){\color[rgb]{0,0,0}\makebox(0,0)[lb]{\smash{$\overset{(p \circ \varphi = p)}{=}$}}}%
  \end{picture}%
\endgroup%

\end{align}
%%%%%%%%%%%%%%%%%%%%%%%%%%%%%%%
The two non-equivalent generalisations of $\gls{z}$ have the following properties. 
\begin{proposition} \label{lem:properties_eta_chi}
The elements $\eta$ and $\gls{chi}$ belong to $\mathcal{Z}(A) \cap \mathcal{Z}_{\lambda}(A)$ and satisfy $\eta^2=\chi^2$.
\end{proposition}
\begin{proof}
One is able to easily conclude that $\gls{eta}$ and $\gls{chi}$ are central elements. In the diagrams below $\begin{aligned} %% Creator: Inkscape 0.48.2, www.inkscape.org
%% PDF/EPS/PS + LaTeX output extension by Johan Engelen, 2010
%% Accompanies image file 'blob.pdf' (pdf, eps, ps)
%%
%% To include the image in your LaTeX document, write
%%   \input{<filename>.pdf_tex}
%%  instead of
%%   \includegraphics{<filename>.pdf}
%% To scale the image, write
%%   \def\svgwidth{<desired width>}
%%   \input{<filename>.pdf_tex}
%%  instead of
%%   \includegraphics[width=<desired width>]{<filename>.pdf}
%%
%% Images with a different path to the parent latex file can
%% be accessed with the `import' package (which may need to be
%% installed) using
%%   \usepackage{import}
%% in the preamble, and then including the image with
%%   \import{<path to file>}{<filename>.pdf_tex}
%% Alternatively, one can specify
%%   \graphicspath{{<path to file>/}}
%% 
%% For more information, please see info/svg-inkscape on CTAN:
%%   http://tug.ctan.org/tex-archive/info/svg-inkscape
%%
\begingroup%
  \makeatletter%
  \providecommand\color[2][]{%
    \errmessage{(Inkscape) Color is used for the text in Inkscape, but the package 'color.sty' is not loaded}%
    \renewcommand\color[2][]{}%
  }%
  \providecommand\transparent[1]{%
    \errmessage{(Inkscape) Transparency is used (non-zero) for the text in Inkscape, but the package 'transparent.sty' is not loaded}%
    \renewcommand\transparent[1]{}%
  }%
  \providecommand\rotatebox[2]{#2}%
  \ifx\svgwidth\undefined%
    \setlength{\unitlength}{8.40000381bp}%
    \ifx\svgscale\undefined%
      \relax%
    \else%
      \setlength{\unitlength}{\unitlength * \real{\svgscale}}%
    \fi%
  \else%
    \setlength{\unitlength}{\svgwidth}%
  \fi%
  \global\let\svgwidth\undefined%
  \global\let\svgscale\undefined%
  \makeatother%
  \begin{picture}(1,1.61712173)%
    \put(0,0){\includegraphics[width=\unitlength]{blob.pdf}}%
  \end{picture}%
\endgroup%
 \end{aligned}$ stands for either the identity map or $\gls{phi}$, depending on which element, $\eta$ or $\chi$, we are interested in. Since the identity map also satisfies properties \ref{it:phi_one}--\ref{it:phi_three}, we will use this notation to signify the corresponding properties of $\begin{aligned} %% Creator: Inkscape 0.48.2, www.inkscape.org
%% PDF/EPS/PS + LaTeX output extension by Johan Engelen, 2010
%% Accompanies image file 'blob.pdf' (pdf, eps, ps)
%%
%% To include the image in your LaTeX document, write
%%   \input{<filename>.pdf_tex}
%%  instead of
%%   \includegraphics{<filename>.pdf}
%% To scale the image, write
%%   \def\svgwidth{<desired width>}
%%   \input{<filename>.pdf_tex}
%%  instead of
%%   \includegraphics[width=<desired width>]{<filename>.pdf}
%%
%% Images with a different path to the parent latex file can
%% be accessed with the `import' package (which may need to be
%% installed) using
%%   \usepackage{import}
%% in the preamble, and then including the image with
%%   \import{<path to file>}{<filename>.pdf_tex}
%% Alternatively, one can specify
%%   \graphicspath{{<path to file>/}}
%% 
%% For more information, please see info/svg-inkscape on CTAN:
%%   http://tug.ctan.org/tex-archive/info/svg-inkscape
%%
\begingroup%
  \makeatletter%
  \providecommand\color[2][]{%
    \errmessage{(Inkscape) Color is used for the text in Inkscape, but the package 'color.sty' is not loaded}%
    \renewcommand\color[2][]{}%
  }%
  \providecommand\transparent[1]{%
    \errmessage{(Inkscape) Transparency is used (non-zero) for the text in Inkscape, but the package 'transparent.sty' is not loaded}%
    \renewcommand\transparent[1]{}%
  }%
  \providecommand\rotatebox[2]{#2}%
  \ifx\svgwidth\undefined%
    \setlength{\unitlength}{8.40000381bp}%
    \ifx\svgscale\undefined%
      \relax%
    \else%
      \setlength{\unitlength}{\unitlength * \real{\svgscale}}%
    \fi%
  \else%
    \setlength{\unitlength}{\svgwidth}%
  \fi%
  \global\let\svgwidth\undefined%
  \global\let\svgscale\undefined%
  \makeatother%
  \begin{picture}(1,1.61712173)%
    \put(0,0){\includegraphics[width=\unitlength]{blob.pdf}}%
  \end{picture}%
\endgroup%
 \end{aligned}$.
%%%%%%%%%%%%%%%%%%%%%%%%%%%%%%%
$$
%% Creator: Inkscape 0.48.2, www.inkscape.org
%% PDF/EPS/PS + LaTeX output extension by Johan Engelen, 2010
%% Accompanies image file 'etacommutativity_details_0.pdf' (pdf, eps, ps)
%%
%% To include the image in your LaTeX document, write
%%   \input{<filename>.pdf_tex}
%%  instead of
%%   \includegraphics{<filename>.pdf}
%% To scale the image, write
%%   \def\svgwidth{<desired width>}
%%   \input{<filename>.pdf_tex}
%%  instead of
%%   \includegraphics[width=<desired width>]{<filename>.pdf}
%%
%% Images with a different path to the parent latex file can
%% be accessed with the `import' package (which may need to be
%% installed) using
%%   \usepackage{import}
%% in the preamble, and then including the image with
%%   \import{<path to file>}{<filename>.pdf_tex}
%% Alternatively, one can specify
%%   \graphicspath{{<path to file>/}}
%% 
%% For more information, please see info/svg-inkscape on CTAN:
%%   http://tug.ctan.org/tex-archive/info/svg-inkscape
%%
\begingroup%
  \makeatletter%
  \providecommand\color[2][]{%
    \errmessage{(Inkscape) Color is used for the text in Inkscape, but the package 'color.sty' is not loaded}%
    \renewcommand\color[2][]{}%
  }%
  \providecommand\transparent[1]{%
    \errmessage{(Inkscape) Transparency is used (non-zero) for the text in Inkscape, but the package 'transparent.sty' is not loaded}%
    \renewcommand\transparent[1]{}%
  }%
  \providecommand\rotatebox[2]{#2}%
  \ifx\svgwidth\undefined%
    \setlength{\unitlength}{597.34615375bp}%
    \ifx\svgscale\undefined%
      \relax%
    \else%
      \setlength{\unitlength}{\unitlength * \real{\svgscale}}%
    \fi%
  \else%
    \setlength{\unitlength}{\svgwidth}%
  \fi%
  \global\let\svgwidth\undefined%
  \global\let\svgscale\undefined%
  \makeatother%
  \begin{picture}(1,0.14285221)%
    \put(0,0){\includegraphics[width=\unitlength]{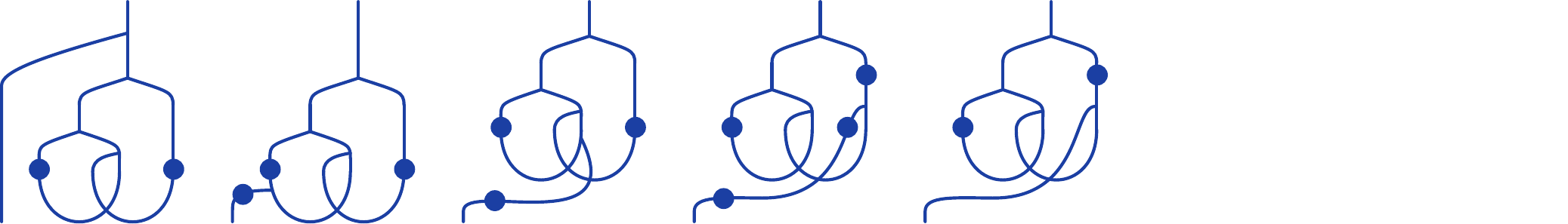}}%
    \put(0.12153767,0.06145372){\color[rgb]{0,0,0}\makebox(0,0)[lb]{\smash{$\overset{\substack{\text{\eqref{fig:diagram_mult2}}\\ + \\ \text{\ref{it:phi_two}}}}{=}$}}}%
    \put(0.26885594,0.06145372){\color[rgb]{0,0,0}\makebox(0,0)[lb]{\smash{$\overset{\text{\ref{fig:axiom_mult}}}{=}$}}}%
    \put(0.4161742,0.06145372){\color[rgb]{0,0,0}\makebox(0,0)[lb]{\smash{$\overset{\substack{\text{\eqref{fig:diagram_mult2}}\\ + \\ \text{\ref{it:phi_two}}}}{=}$}}}%
    \put(0.56349247,0.06145372){\color[rgb]{0,0,0}\makebox(0,0)[lb]{\smash{$\overset{\substack{\text{\ref{it:phi_three}}\\ + \\ \text{\ref{it:phi_one}}}}{=}$}}}%
  \end{picture}%
\endgroup%

$$
$$
\hspace{10mm}
%% Creator: Inkscape 0.48.2, www.inkscape.org
%% PDF/EPS/PS + LaTeX output extension by Johan Engelen, 2010
%% Accompanies image file 'etacommutativity_details_1.pdf' (pdf, eps, ps)
%%
%% To include the image in your LaTeX document, write
%%   \input{<filename>.pdf_tex}
%%  instead of
%%   \includegraphics{<filename>.pdf}
%% To scale the image, write
%%   \def\svgwidth{<desired width>}
%%   \input{<filename>.pdf_tex}
%%  instead of
%%   \includegraphics[width=<desired width>]{<filename>.pdf}
%%
%% Images with a different path to the parent latex file can
%% be accessed with the `import' package (which may need to be
%% installed) using
%%   \usepackage{import}
%% in the preamble, and then including the image with
%%   \import{<path to file>}{<filename>.pdf_tex}
%% Alternatively, one can specify
%%   \graphicspath{{<path to file>/}}
%% 
%% For more information, please see info/svg-inkscape on CTAN:
%%   http://tug.ctan.org/tex-archive/info/svg-inkscape
%%
\begingroup%
  \makeatletter%
  \providecommand\color[2][]{%
    \errmessage{(Inkscape) Color is used for the text in Inkscape, but the package 'color.sty' is not loaded}%
    \renewcommand\color[2][]{}%
  }%
  \providecommand\transparent[1]{%
    \errmessage{(Inkscape) Transparency is used (non-zero) for the text in Inkscape, but the package 'transparent.sty' is not loaded}%
    \renewcommand\transparent[1]{}%
  }%
  \providecommand\rotatebox[2]{#2}%
  \ifx\svgwidth\undefined%
    \setlength{\unitlength}{393.8515625bp}%
    \ifx\svgscale\undefined%
      \relax%
    \else%
      \setlength{\unitlength}{\unitlength * \real{\svgscale}}%
    \fi%
  \else%
    \setlength{\unitlength}{\svgwidth}%
  \fi%
  \global\let\svgwidth\undefined%
  \global\let\svgscale\undefined%
  \makeatother%
  \begin{picture}(1,0.21639661)%
    \put(0,0){\includegraphics[width=\unitlength]{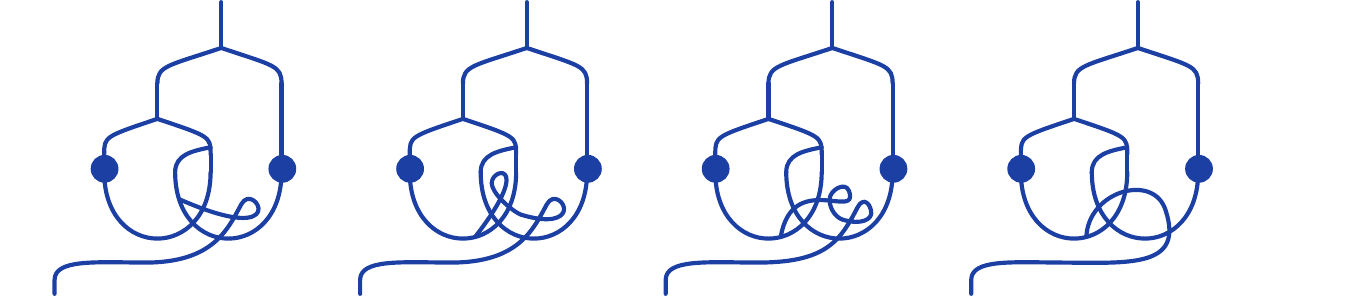}}%
    \put(0.22271041,0.09295011){\color[rgb]{0,0,0}\makebox(0,0)[lb]{\smash{$\overset{\substack{\text{\eqref{fig:diagram_mult2}}\\+\\ \text{\ref{fig:axiom_mult}}}}{=}$}}}%
    \put(0.44614484,0.09295011){\color[rgb]{0,0,0}\makebox(0,0)[lb]{\smash{$\overset{\text{\ref{it:phi_three}}}{=}$}}}%
    \put(0.66957928,0.09295011){\color[rgb]{0,0,0}\makebox(0,0)[lb]{\smash{$\overset{\text{\ref{it:phi_one}}}{=}$}}}%
    \put(-0.00072402,0.09295011){\color[rgb]{0,0,0}\makebox(0,0)[lb]{\smash{$\overset{\substack{\text{\eqref{eq:diag_1}}\\+\\ \text{\ref{fig:axiom_mult}}}}{=}$}}}%
  \end{picture}%
\endgroup%
 
$$
$$
\hspace{10mm}
%% Creator: Inkscape 0.48.2, www.inkscape.org
%% PDF/EPS/PS + LaTeX output extension by Johan Engelen, 2010
%% Accompanies image file 'etacommutativity_details_2.pdf' (pdf, eps, ps)
%%
%% To include the image in your LaTeX document, write
%%   \input{<filename>.pdf_tex}
%%  instead of
%%   \includegraphics{<filename>.pdf}
%% To scale the image, write
%%   \def\svgwidth{<desired width>}
%%   \input{<filename>.pdf_tex}
%%  instead of
%%   \includegraphics[width=<desired width>]{<filename>.pdf}
%%
%% Images with a different path to the parent latex file can
%% be accessed with the `import' package (which may need to be
%% installed) using
%%   \usepackage{import}
%% in the preamble, and then including the image with
%%   \import{<path to file>}{<filename>.pdf_tex}
%% Alternatively, one can specify
%%   \graphicspath{{<path to file>/}}
%% 
%% For more information, please see info/svg-inkscape on CTAN:
%%   http://tug.ctan.org/tex-archive/info/svg-inkscape
%%
\begingroup%
  \makeatletter%
  \providecommand\color[2][]{%
    \errmessage{(Inkscape) Color is used for the text in Inkscape, but the package 'color.sty' is not loaded}%
    \renewcommand\color[2][]{}%
  }%
  \providecommand\transparent[1]{%
    \errmessage{(Inkscape) Transparency is used (non-zero) for the text in Inkscape, but the package 'transparent.sty' is not loaded}%
    \renewcommand\transparent[1]{}%
  }%
  \providecommand\rotatebox[2]{#2}%
  \ifx\svgwidth\undefined%
    \setlength{\unitlength}{559.05657163bp}%
    \ifx\svgscale\undefined%
      \relax%
    \else%
      \setlength{\unitlength}{\unitlength * \real{\svgscale}}%
    \fi%
  \else%
    \setlength{\unitlength}{\svgwidth}%
  \fi%
  \global\let\svgwidth\undefined%
  \global\let\svgscale\undefined%
  \makeatother%
  \begin{picture}(1,0.15244995)%
    \put(0,0){\includegraphics[width=\unitlength]{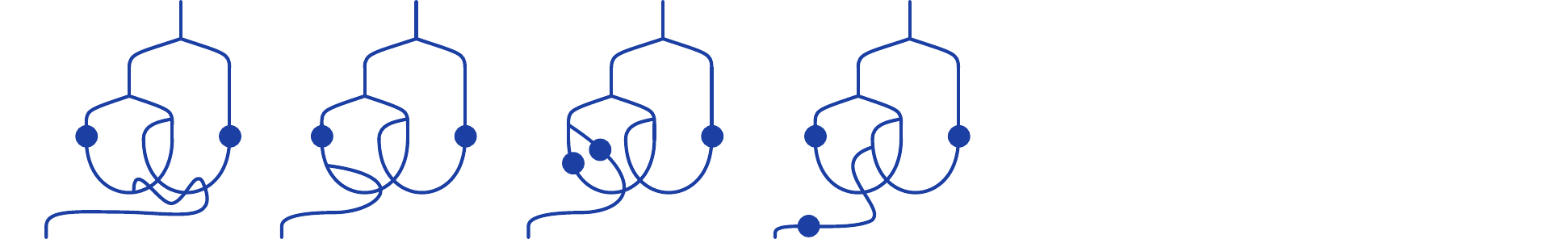}}%
    \put(0.15827084,0.05832781){\color[rgb]{0,0,0}\makebox(0,0)[lb]{\smash{$\overset{\substack{\text{\eqref{eq:diag_1}}\\ + \\ \text{\ref{fig:axiom_square}}}}{=}$}}}%
    \put(0.46593201,0.05832781){\color[rgb]{0,0,0}\makebox(0,0)[lb]{\smash{$\overset{\substack{\text{\eqref{fig:diagram_mult2}}\\ + \\ \text{\ref{it:phi_three}}}}{=}$}}}%
    \put(-0.00049917,0.0583278){\color[rgb]{0,0,0}\makebox(0,0)[lb]{\smash{$\overset{\text{\ref{fig:axiom_Reid_3}}}{=}$}}}%
    \put(0.30852397,0.05832781){\color[rgb]{0,0,0}\makebox(0,0)[lb]{\smash{$\overset{\text{\ref{it:phi_two}}}{=}$}}}%
  \end{picture}%
\endgroup%

$$
$$
\hspace{10mm}
%% Creator: Inkscape 0.48.2, www.inkscape.org
%% PDF/EPS/PS + LaTeX output extension by Johan Engelen, 2010
%% Accompanies image file 'etacommutativity_details_3.pdf' (pdf, eps, ps)
%%
%% To include the image in your LaTeX document, write
%%   \input{<filename>.pdf_tex}
%%  instead of
%%   \includegraphics{<filename>.pdf}
%% To scale the image, write
%%   \def\svgwidth{<desired width>}
%%   \input{<filename>.pdf_tex}
%%  instead of
%%   \includegraphics[width=<desired width>]{<filename>.pdf}
%%
%% Images with a different path to the parent latex file can
%% be accessed with the `import' package (which may need to be
%% installed) using
%%   \usepackage{import}
%% in the preamble, and then including the image with
%%   \import{<path to file>}{<filename>.pdf_tex}
%% Alternatively, one can specify
%%   \graphicspath{{<path to file>/}}
%% 
%% For more information, please see info/svg-inkscape on CTAN:
%%   http://tug.ctan.org/tex-archive/info/svg-inkscape
%%
\begingroup%
  \makeatletter%
  \providecommand\color[2][]{%
    \errmessage{(Inkscape) Color is used for the text in Inkscape, but the package 'color.sty' is not loaded}%
    \renewcommand\color[2][]{}%
  }%
  \providecommand\transparent[1]{%
    \errmessage{(Inkscape) Transparency is used (non-zero) for the text in Inkscape, but the package 'transparent.sty' is not loaded}%
    \renewcommand\transparent[1]{}%
  }%
  \providecommand\rotatebox[2]{#2}%
  \ifx\svgwidth\undefined%
    \setlength{\unitlength}{485.3984375bp}%
    \ifx\svgscale\undefined%
      \relax%
    \else%
      \setlength{\unitlength}{\unitlength * \real{\svgscale}}%
    \fi%
  \else%
    \setlength{\unitlength}{\svgwidth}%
  \fi%
  \global\let\svgwidth\undefined%
  \global\let\svgscale\undefined%
  \makeatother%
  \begin{picture}(1,0.17559313)%
    \put(0,0){\includegraphics[width=\unitlength]{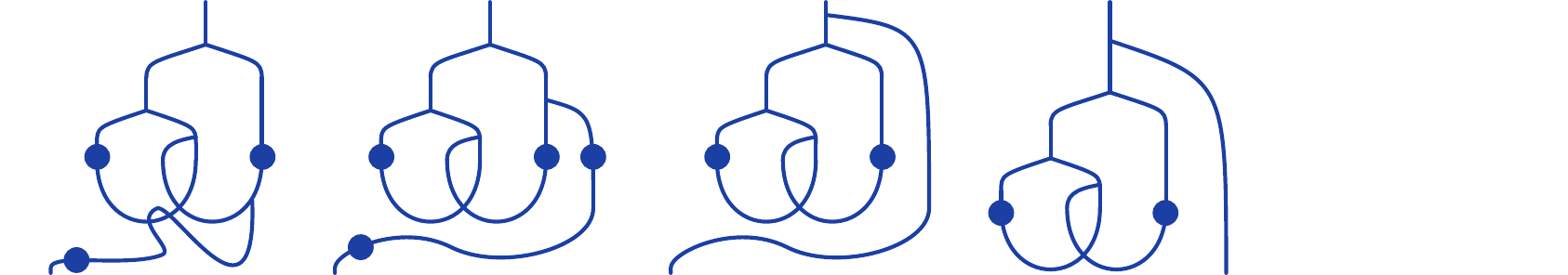}}%
    \put(0.60098019,0.06717893){\color[rgb]{0,0,0}\makebox(0,0)[lb]{\smash{$\overset{\text{\eqref{eq:snake}}}{=}$}}}%
    \put(-0.00058747,0.06717893){\color[rgb]{0,0,0}\makebox(0,0)[lb]{\smash{$\overset{\substack{\text{\ref{fig:axiom_mult}} \\ + \\ \text{\eqref{eq:prop_sep}}}}{=}$}}}%
    \put(0.18070689,0.06717893){\color[rgb]{0,0,0}\makebox(0,0)[lb]{\smash{$\overset{\substack{\text{\ref{fig:axiom_square}}\\ + \\ \text{\ref{it:phi_two}}}}{=}$}}}%
    \put(0.39496387,0.06717893){\color[rgb]{0,0,0}\makebox(0,0)[lb]{\smash{$\overset{\substack{\text{\ref{it:phi_one}}\\ + \\ \text{\eqref{fig:diagram_mult2}}}}{=}$}}}%
  \end{picture}%
\endgroup%
 
$$
%%%%%%%%%%%%%%%%%%%%%%%%%%%%%%%
In addition, because $\eta$ and $\chi$ are both determined by a semi-closed diagram we are able to conclude $\eta, \chi \in \mathcal{Z}_{\lambda}(A)$.
%%%%%%%%%%%%%%%%%%%%%%%%%%%%%%%
\begin{align}
\begin{aligned}
%% Creator: Inkscape 0.48.2, www.inkscape.org
%% PDF/EPS/PS + LaTeX output extension by Johan Engelen, 2010
%% Accompanies image file 'eta_central.pdf' (pdf, eps, ps)
%%
%% To include the image in your LaTeX document, write
%%   \input{<filename>.pdf_tex}
%%  instead of
%%   \includegraphics{<filename>.pdf}
%% To scale the image, write
%%   \def\svgwidth{<desired width>}
%%   \input{<filename>.pdf_tex}
%%  instead of
%%   \includegraphics[width=<desired width>]{<filename>.pdf}
%%
%% Images with a different path to the parent latex file can
%% be accessed with the `import' package (which may need to be
%% installed) using
%%   \usepackage{import}
%% in the preamble, and then including the image with
%%   \import{<path to file>}{<filename>.pdf_tex}
%% Alternatively, one can specify
%%   \graphicspath{{<path to file>/}}
%% 
%% For more information, please see info/svg-inkscape on CTAN:
%%   http://tug.ctan.org/tex-archive/info/svg-inkscape
%%
\begingroup%
  \makeatletter%
  \providecommand\color[2][]{%
    \errmessage{(Inkscape) Color is used for the text in Inkscape, but the package 'color.sty' is not loaded}%
    \renewcommand\color[2][]{}%
  }%
  \providecommand\transparent[1]{%
    \errmessage{(Inkscape) Transparency is used (non-zero) for the text in Inkscape, but the package 'transparent.sty' is not loaded}%
    \renewcommand\transparent[1]{}%
  }%
  \providecommand\rotatebox[2]{#2}%
  \ifx\svgwidth\undefined%
    \setlength{\unitlength}{248.88428188bp}%
    \ifx\svgscale\undefined%
      \relax%
    \else%
      \setlength{\unitlength}{\unitlength * \real{\svgscale}}%
    \fi%
  \else%
    \setlength{\unitlength}{\svgwidth}%
  \fi%
  \global\let\svgwidth\undefined%
  \global\let\svgscale\undefined%
  \makeatother%
  \begin{picture}(1,0.30387842)%
    \put(0,0){\includegraphics[width=\unitlength]{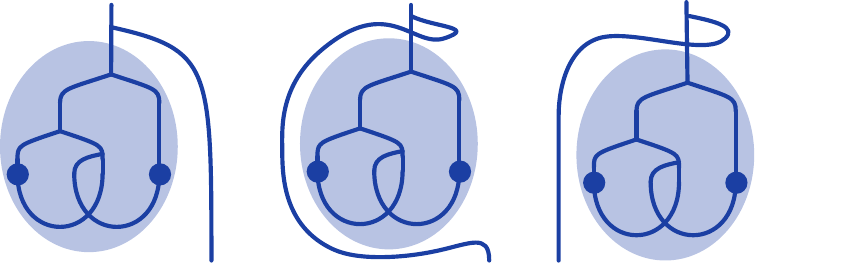}}%
    \put(0.25714724,0.10530414){\color[rgb]{0,0,0}\makebox(0,0)[lb]{\smash{$\overset{\text{\eqref{eq:closed_below}}}{=}$}}}%
    \put(0.57858176,0.10530414){\color[rgb]{0,0,0}\makebox(0,0)[lb]{\smash{$\overset{\text{\eqref{eq:snake}}}{=}$}}}%
  \end{picture}%
\endgroup%

\end{aligned}
\end{align}
%%%%%%%%%%%%%%%%%%%%%%%%%%%%%%%
In other words, $\gls{eta}$ and $\gls{chi}$ belong to $\gls{ZA} \cap \mathcal{Z}_{\lambda}(A)$. 

The last and most lengthy part of the proof comes from determining the non-trivial identity $\chi^2=\eta^2$. Behind this proof is the idea that we should commute the two factors of $\chi$ in a non-trivial way. To highlight this permutation the two factors are presented in different colours. Note that the properties of a planar model guarantee that any map $\begin{aligned}%% Creator: Inkscape 0.48.2, www.inkscape.org
%% PDF/EPS/PS + LaTeX output extension by Johan Engelen, 2010
%% Accompanies image file 'blob.pdf' (pdf, eps, ps)
%%
%% To include the image in your LaTeX document, write
%%   \input{<filename>.pdf_tex}
%%  instead of
%%   \includegraphics{<filename>.pdf}
%% To scale the image, write
%%   \def\svgwidth{<desired width>}
%%   \input{<filename>.pdf_tex}
%%  instead of
%%   \includegraphics[width=<desired width>]{<filename>.pdf}
%%
%% Images with a different path to the parent latex file can
%% be accessed with the `import' package (which may need to be
%% installed) using
%%   \usepackage{import}
%% in the preamble, and then including the image with
%%   \import{<path to file>}{<filename>.pdf_tex}
%% Alternatively, one can specify
%%   \graphicspath{{<path to file>/}}
%% 
%% For more information, please see info/svg-inkscape on CTAN:
%%   http://tug.ctan.org/tex-archive/info/svg-inkscape
%%
\begingroup%
  \makeatletter%
  \providecommand\color[2][]{%
    \errmessage{(Inkscape) Color is used for the text in Inkscape, but the package 'color.sty' is not loaded}%
    \renewcommand\color[2][]{}%
  }%
  \providecommand\transparent[1]{%
    \errmessage{(Inkscape) Transparency is used (non-zero) for the text in Inkscape, but the package 'transparent.sty' is not loaded}%
    \renewcommand\transparent[1]{}%
  }%
  \providecommand\rotatebox[2]{#2}%
  \ifx\svgwidth\undefined%
    \setlength{\unitlength}{8.40000381bp}%
    \ifx\svgscale\undefined%
      \relax%
    \else%
      \setlength{\unitlength}{\unitlength * \real{\svgscale}}%
    \fi%
  \else%
    \setlength{\unitlength}{\svgwidth}%
  \fi%
  \global\let\svgwidth\undefined%
  \global\let\svgscale\undefined%
  \makeatother%
  \begin{picture}(1,1.61712173)%
    \put(0,0){\includegraphics[width=\unitlength]{blob.pdf}}%
  \end{picture}%
\endgroup%
\end{aligned} \colon \gls{A} \to A$ satisfies $\begin{aligned}%% Creator: Inkscape 0.48.2, www.inkscape.org
%% PDF/EPS/PS + LaTeX output extension by Johan Engelen, 2010
%% Accompanies image file 'blob.pdf' (pdf, eps, ps)
%%
%% To include the image in your LaTeX document, write
%%   \input{<filename>.pdf_tex}
%%  instead of
%%   \includegraphics{<filename>.pdf}
%% To scale the image, write
%%   \def\svgwidth{<desired width>}
%%   \input{<filename>.pdf_tex}
%%  instead of
%%   \includegraphics[width=<desired width>]{<filename>.pdf}
%%
%% Images with a different path to the parent latex file can
%% be accessed with the `import' package (which may need to be
%% installed) using
%%   \usepackage{import}
%% in the preamble, and then including the image with
%%   \import{<path to file>}{<filename>.pdf_tex}
%% Alternatively, one can specify
%%   \graphicspath{{<path to file>/}}
%% 
%% For more information, please see info/svg-inkscape on CTAN:
%%   http://tug.ctan.org/tex-archive/info/svg-inkscape
%%
\begingroup%
  \makeatletter%
  \providecommand\color[2][]{%
    \errmessage{(Inkscape) Color is used for the text in Inkscape, but the package 'color.sty' is not loaded}%
    \renewcommand\color[2][]{}%
  }%
  \providecommand\transparent[1]{%
    \errmessage{(Inkscape) Transparency is used (non-zero) for the text in Inkscape, but the package 'transparent.sty' is not loaded}%
    \renewcommand\transparent[1]{}%
  }%
  \providecommand\rotatebox[2]{#2}%
  \ifx\svgwidth\undefined%
    \setlength{\unitlength}{8.40000381bp}%
    \ifx\svgscale\undefined%
      \relax%
    \else%
      \setlength{\unitlength}{\unitlength * \real{\svgscale}}%
    \fi%
  \else%
    \setlength{\unitlength}{\svgwidth}%
  \fi%
  \global\let\svgwidth\undefined%
  \global\let\svgscale\undefined%
  \makeatother%
  \begin{picture}(1,1.61712173)%
    \put(0,0){\includegraphics[width=\unitlength]{blob.pdf}}%
  \end{picture}%
\endgroup%
\end{aligned} = \gls{sigma} \circ \begin{aligned}%% Creator: Inkscape 0.48.2, www.inkscape.org
%% PDF/EPS/PS + LaTeX output extension by Johan Engelen, 2010
%% Accompanies image file 'blob.pdf' (pdf, eps, ps)
%%
%% To include the image in your LaTeX document, write
%%   \input{<filename>.pdf_tex}
%%  instead of
%%   \includegraphics{<filename>.pdf}
%% To scale the image, write
%%   \def\svgwidth{<desired width>}
%%   \input{<filename>.pdf_tex}
%%  instead of
%%   \includegraphics[width=<desired width>]{<filename>.pdf}
%%
%% Images with a different path to the parent latex file can
%% be accessed with the `import' package (which may need to be
%% installed) using
%%   \usepackage{import}
%% in the preamble, and then including the image with
%%   \import{<path to file>}{<filename>.pdf_tex}
%% Alternatively, one can specify
%%   \graphicspath{{<path to file>/}}
%% 
%% For more information, please see info/svg-inkscape on CTAN:
%%   http://tug.ctan.org/tex-archive/info/svg-inkscape
%%
\begingroup%
  \makeatletter%
  \providecommand\color[2][]{%
    \errmessage{(Inkscape) Color is used for the text in Inkscape, but the package 'color.sty' is not loaded}%
    \renewcommand\color[2][]{}%
  }%
  \providecommand\transparent[1]{%
    \errmessage{(Inkscape) Transparency is used (non-zero) for the text in Inkscape, but the package 'transparent.sty' is not loaded}%
    \renewcommand\transparent[1]{}%
  }%
  \providecommand\rotatebox[2]{#2}%
  \ifx\svgwidth\undefined%
    \setlength{\unitlength}{8.40000381bp}%
    \ifx\svgscale\undefined%
      \relax%
    \else%
      \setlength{\unitlength}{\unitlength * \real{\svgscale}}%
    \fi%
  \else%
    \setlength{\unitlength}{\svgwidth}%
  \fi%
  \global\let\svgwidth\undefined%
  \global\let\svgscale\undefined%
  \makeatother%
  \begin{picture}(1,1.61712173)%
    \put(0,0){\includegraphics[width=\unitlength]{blob.pdf}}%
  \end{picture}%
\endgroup%
\end{aligned} \circ \sigma$.
%%%%%%%%%%%%%%%%%%%%%%%%%%%%%%%
$$
%% Creator: Inkscape 0.48.2, www.inkscape.org
%% PDF/EPS/PS + LaTeX output extension by Johan Engelen, 2010
%% Accompanies image file 'bigproof_1.pdf' (pdf, eps, ps)
%%
%% To include the image in your LaTeX document, write
%%   \input{<filename>.pdf_tex}
%%  instead of
%%   \includegraphics{<filename>.pdf}
%% To scale the image, write
%%   \def\svgwidth{<desired width>}
%%   \input{<filename>.pdf_tex}
%%  instead of
%%   \includegraphics[width=<desired width>]{<filename>.pdf}
%%
%% Images with a different path to the parent latex file can
%% be accessed with the `import' package (which may need to be
%% installed) using
%%   \usepackage{import}
%% in the preamble, and then including the image with
%%   \import{<path to file>}{<filename>.pdf_tex}
%% Alternatively, one can specify
%%   \graphicspath{{<path to file>/}}
%% 
%% For more information, please see info/svg-inkscape on CTAN:
%%   http://tug.ctan.org/tex-archive/info/svg-inkscape
%%
\begingroup%
  \makeatletter%
  \providecommand\color[2][]{%
    \errmessage{(Inkscape) Color is used for the text in Inkscape, but the package 'color.sty' is not loaded}%
    \renewcommand\color[2][]{}%
  }%
  \providecommand\transparent[1]{%
    \errmessage{(Inkscape) Transparency is used (non-zero) for the text in Inkscape, but the package 'transparent.sty' is not loaded}%
    \renewcommand\transparent[1]{}%
  }%
  \providecommand\rotatebox[2]{#2}%
  \ifx\svgwidth\undefined%
    \setlength{\unitlength}{405.5234375bp}%
    \ifx\svgscale\undefined%
      \relax%
    \else%
      \setlength{\unitlength}{\unitlength * \real{\svgscale}}%
    \fi%
  \else%
    \setlength{\unitlength}{\svgwidth}%
  \fi%
  \global\let\svgwidth\undefined%
  \global\let\svgscale\undefined%
  \makeatother%
  \begin{picture}(1,0.16576165)%
    \put(0,0){\includegraphics[width=\unitlength]{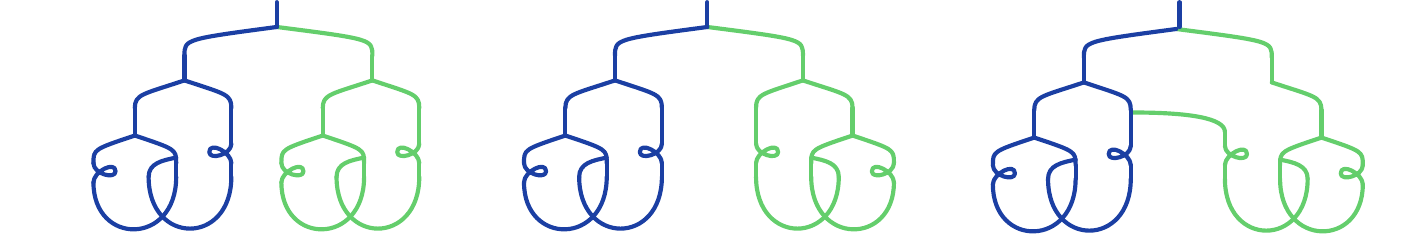}}%
    \put(-0.00070318,0.04796858){\color[rgb]{0,0,0}\makebox(0,0)[lb]{\smash{$\chi^2=$}}}%
    \put(0.31493825,0.04796858){\color[rgb]{0,0,0}\makebox(0,0)[lb]{\smash{$\overset{\eqref{fig:diagram_mult2}}{=}$}}}%
    \put(0.65030728,0.04796858){\color[rgb]{0,0,0}\makebox(0,0)[lb]{\smash{$\overset{\eqref{fig:diagram_mult2}}{=}$}}}%
  \end{picture}%
\endgroup%
	
$$
%%%%%%%%%%%%%%%%%%%%%%%%%%%%%%%
%%%%%%%%%%%%%%%%%%%%%%%%%%%%%%%
$$
\hspace{5mm}
%% Creator: Inkscape 0.48.2, www.inkscape.org
%% PDF/EPS/PS + LaTeX output extension by Johan Engelen, 2010
%% Accompanies image file 'bigproof11.pdf' (pdf, eps, ps)
%%
%% To include the image in your LaTeX document, write
%%   \input{<filename>.pdf_tex}
%%  instead of
%%   \includegraphics{<filename>.pdf}
%% To scale the image, write
%%   \def\svgwidth{<desired width>}
%%   \input{<filename>.pdf_tex}
%%  instead of
%%   \includegraphics[width=<desired width>]{<filename>.pdf}
%%
%% Images with a different path to the parent latex file can
%% be accessed with the `import' package (which may need to be
%% installed) using
%%   \usepackage{import}
%% in the preamble, and then including the image with
%%   \import{<path to file>}{<filename>.pdf_tex}
%% Alternatively, one can specify
%%   \graphicspath{{<path to file>/}}
%% 
%% For more information, please see info/svg-inkscape on CTAN:
%%   http://tug.ctan.org/tex-archive/info/svg-inkscape
%%
\begingroup%
  \makeatletter%
  \providecommand\color[2][]{%
    \errmessage{(Inkscape) Color is used for the text in Inkscape, but the package 'color.sty' is not loaded}%
    \renewcommand\color[2][]{}%
  }%
  \providecommand\transparent[1]{%
    \errmessage{(Inkscape) Transparency is used (non-zero) for the text in Inkscape, but the package 'transparent.sty' is not loaded}%
    \renewcommand\transparent[1]{}%
  }%
  \providecommand\rotatebox[2]{#2}%
  \ifx\svgwidth\undefined%
    \setlength{\unitlength}{413.9765625bp}%
    \ifx\svgscale\undefined%
      \relax%
    \else%
      \setlength{\unitlength}{\unitlength * \real{\svgscale}}%
    \fi%
  \else%
    \setlength{\unitlength}{\svgwidth}%
  \fi%
  \global\let\svgwidth\undefined%
  \global\let\svgscale\undefined%
  \makeatother%
  \begin{picture}(1,0.18754548)%
    \put(0,0){\includegraphics[width=\unitlength]{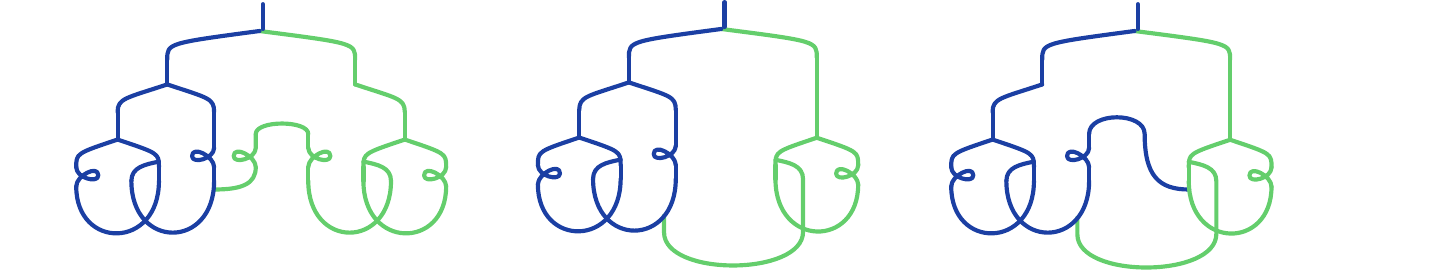}}%
    \put(-0.00068882,0.07076484){\color[rgb]{0,0,0}\makebox(0,0)[lb]{\smash{$\overset{\text{\ref{it:phi_two}}}{=}$}}}%
    \put(0.31816981,0.07076484){\color[rgb]{0,0,0}\makebox(0,0)[lb]{\smash{$\overset{\text{\ref{it:phi_one}}}{=}$}}}%
    \put(0.60804129,0.07076484){\color[rgb]{0,0,0}\makebox(0,0)[lb]{\smash{$\overset{\text{\eqref{fig:diagram_mult2}}}{=}$}}}%
  \end{picture}%
\endgroup%
	
$$
%%%%%%%%%%%%%%%%%%%%%%%%%%%%%%%
%%%%%%%%%%%%%%%%%%%%%%%%%%%%%%%
$$
\hspace{5mm}
%% Creator: Inkscape 0.48.2, www.inkscape.org
%% PDF/EPS/PS + LaTeX output extension by Johan Engelen, 2010
%% Accompanies image file 'bigproof13.pdf' (pdf, eps, ps)
%%
%% To include the image in your LaTeX document, write
%%   \input{<filename>.pdf_tex}
%%  instead of
%%   \includegraphics{<filename>.pdf}
%% To scale the image, write
%%   \def\svgwidth{<desired width>}
%%   \input{<filename>.pdf_tex}
%%  instead of
%%   \includegraphics[width=<desired width>]{<filename>.pdf}
%%
%% Images with a different path to the parent latex file can
%% be accessed with the `import' package (which may need to be
%% installed) using
%%   \usepackage{import}
%% in the preamble, and then including the image with
%%   \import{<path to file>}{<filename>.pdf_tex}
%% Alternatively, one can specify
%%   \graphicspath{{<path to file>/}}
%% 
%% For more information, please see info/svg-inkscape on CTAN:
%%   http://tug.ctan.org/tex-archive/info/svg-inkscape
%%
\begingroup%
  \makeatletter%
  \providecommand\color[2][]{%
    \errmessage{(Inkscape) Color is used for the text in Inkscape, but the package 'color.sty' is not loaded}%
    \renewcommand\color[2][]{}%
  }%
  \providecommand\transparent[1]{%
    \errmessage{(Inkscape) Transparency is used (non-zero) for the text in Inkscape, but the package 'transparent.sty' is not loaded}%
    \renewcommand\transparent[1]{}%
  }%
  \providecommand\rotatebox[2]{#2}%
  \ifx\svgwidth\undefined%
    \setlength{\unitlength}{593.4453125bp}%
    \ifx\svgscale\undefined%
      \relax%
    \else%
      \setlength{\unitlength}{\unitlength * \real{\svgscale}}%
    \fi%
  \else%
    \setlength{\unitlength}{\svgwidth}%
  \fi%
  \global\let\svgwidth\undefined%
  \global\let\svgscale\undefined%
  \makeatother%
  \begin{picture}(1,0.13882929)%
    \put(0,0){\includegraphics[width=\unitlength]{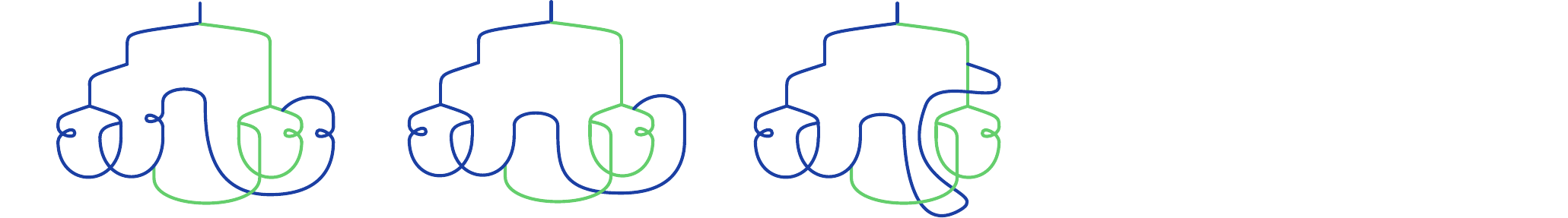}}%
    \put(-0.00048051,0.05736527){\color[rgb]{0,0,0}\makebox(0,0)[lb]{\smash{$\overset{\substack{\text{\ref{fig:axiom_mult}}\\ + \\ \text{\ref{it:phi_two}}}}{=}$}}}%
    \put(0.22194942,0.05736527){\color[rgb]{0,0,0}\makebox(0,0)[lb]{\smash{$\overset{\substack{\text{\ref{it:phi_three}}\\ + \\ \text{\ref{it:phi_one}}}}{=}$}}}%
    \put(0.44437935,0.05736527){\color[rgb]{0,0,0}\makebox(0,0)[lb]{\smash{$\overset{\substack{\text{\eqref{fig:diagram_mult2}} \\ + \\ \text{\eqref{eq:closed_below}}}}{=}$}}}%
  \end{picture}%
\endgroup%
	
$$
%%%%%%%%%%%%%%%%%%%%%%%%%%%%%%%
%%%%%%%%%%%%%%%%%%%%%%%%%%%%%%%
$$
\hspace{5mm}
%% Creator: Inkscape 0.48.2, www.inkscape.org
%% PDF/EPS/PS + LaTeX output extension by Johan Engelen, 2010
%% Accompanies image file 'bigproof14.pdf' (pdf, eps, ps)
%%
%% To include the image in your LaTeX document, write
%%   \input{<filename>.pdf_tex}
%%  instead of
%%   \includegraphics{<filename>.pdf}
%% To scale the image, write
%%   \def\svgwidth{<desired width>}
%%   \input{<filename>.pdf_tex}
%%  instead of
%%   \includegraphics[width=<desired width>]{<filename>.pdf}
%%
%% Images with a different path to the parent latex file can
%% be accessed with the `import' package (which may need to be
%% installed) using
%%   \usepackage{import}
%% in the preamble, and then including the image with
%%   \import{<path to file>}{<filename>.pdf_tex}
%% Alternatively, one can specify
%%   \graphicspath{{<path to file>/}}
%% 
%% For more information, please see info/svg-inkscape on CTAN:
%%   http://tug.ctan.org/tex-archive/info/svg-inkscape
%%
\begingroup%
  \makeatletter%
  \providecommand\color[2][]{%
    \errmessage{(Inkscape) Color is used for the text in Inkscape, but the package 'color.sty' is not loaded}%
    \renewcommand\color[2][]{}%
  }%
  \providecommand\transparent[1]{%
    \errmessage{(Inkscape) Transparency is used (non-zero) for the text in Inkscape, but the package 'transparent.sty' is not loaded}%
    \renewcommand\transparent[1]{}%
  }%
  \providecommand\rotatebox[2]{#2}%
  \ifx\svgwidth\undefined%
    \setlength{\unitlength}{452.2734375bp}%
    \ifx\svgscale\undefined%
      \relax%
    \else%
      \setlength{\unitlength}{\unitlength * \real{\svgscale}}%
    \fi%
  \else%
    \setlength{\unitlength}{\svgwidth}%
  \fi%
  \global\let\svgwidth\undefined%
  \global\let\svgscale\undefined%
  \makeatother%
  \begin{picture}(1,0.18111726)%
    \put(0,0){\includegraphics[width=\unitlength]{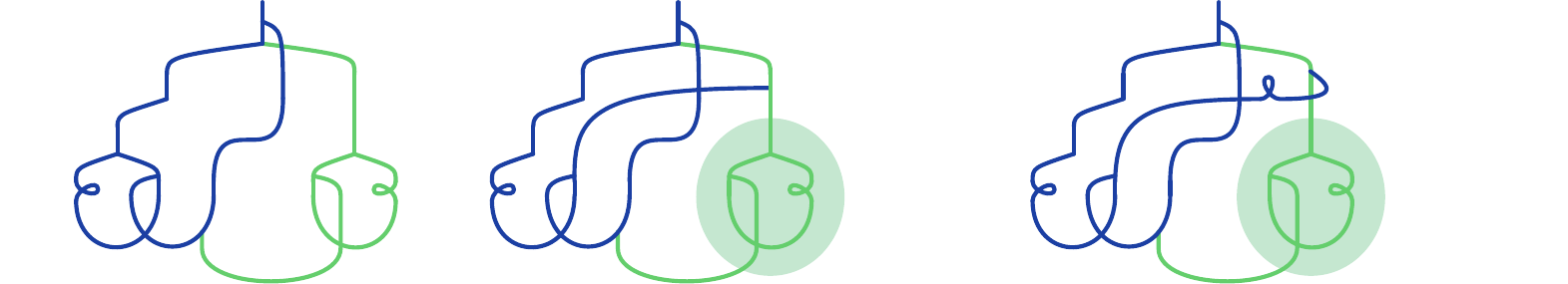}}%
    \put(-0.0006305,0.06477273){\color[rgb]{0,0,0}\makebox(0,0)[lb]{\smash{$\overset{\substack{\text{\eqref{fig:diagram_mult2}} \\ + \\ \text{\ref{fig:axiom_square}}}}{=}$}}}%
    \put(0.26469572,0.06477273){\color[rgb]{0,0,0}\makebox(0,0)[lb]{\smash{$\overset{\text{\eqref{fig:diagram_mult2}}}{=}$}}}%
    \put(0.54771035,0.06477273){\color[rgb]{0,0,0}\makebox(0,0)[lb]{\smash{$\overset{n(a) \in \overline{\mathcal{Z}}_{\lambda}(A)}{=}$}}}%
  \end{picture}%
\endgroup%
	
$$
%%%%%%%%%%%%%%%%%%%%%%%%%%%%%%%
%%%%%%%%%%%%%%%%%%%%%%%%%%%%%%%
$$
\hspace{5mm}
%% Creator: Inkscape 0.48.2, www.inkscape.org
%% PDF/EPS/PS + LaTeX output extension by Johan Engelen, 2010
%% Accompanies image file 'bigproof15.pdf' (pdf, eps, ps)
%%
%% To include the image in your LaTeX document, write
%%   \input{<filename>.pdf_tex}
%%  instead of
%%   \includegraphics{<filename>.pdf}
%% To scale the image, write
%%   \def\svgwidth{<desired width>}
%%   \input{<filename>.pdf_tex}
%%  instead of
%%   \includegraphics[width=<desired width>]{<filename>.pdf}
%%
%% Images with a different path to the parent latex file can
%% be accessed with the `import' package (which may need to be
%% installed) using
%%   \usepackage{import}
%% in the preamble, and then including the image with
%%   \import{<path to file>}{<filename>.pdf_tex}
%% Alternatively, one can specify
%%   \graphicspath{{<path to file>/}}
%% 
%% For more information, please see info/svg-inkscape on CTAN:
%%   http://tug.ctan.org/tex-archive/info/svg-inkscape
%%
\begingroup%
  \makeatletter%
  \providecommand\color[2][]{%
    \errmessage{(Inkscape) Color is used for the text in Inkscape, but the package 'color.sty' is not loaded}%
    \renewcommand\color[2][]{}%
  }%
  \providecommand\transparent[1]{%
    \errmessage{(Inkscape) Transparency is used (non-zero) for the text in Inkscape, but the package 'transparent.sty' is not loaded}%
    \renewcommand\transparent[1]{}%
  }%
  \providecommand\rotatebox[2]{#2}%
  \ifx\svgwidth\undefined%
    \setlength{\unitlength}{431.1796875bp}%
    \ifx\svgscale\undefined%
      \relax%
    \else%
      \setlength{\unitlength}{\unitlength * \real{\svgscale}}%
    \fi%
  \else%
    \setlength{\unitlength}{\svgwidth}%
  \fi%
  \global\let\svgwidth\undefined%
  \global\let\svgscale\undefined%
  \makeatother%
  \begin{picture}(1,0.20853145)%
    \put(0,0){\includegraphics[width=\unitlength]{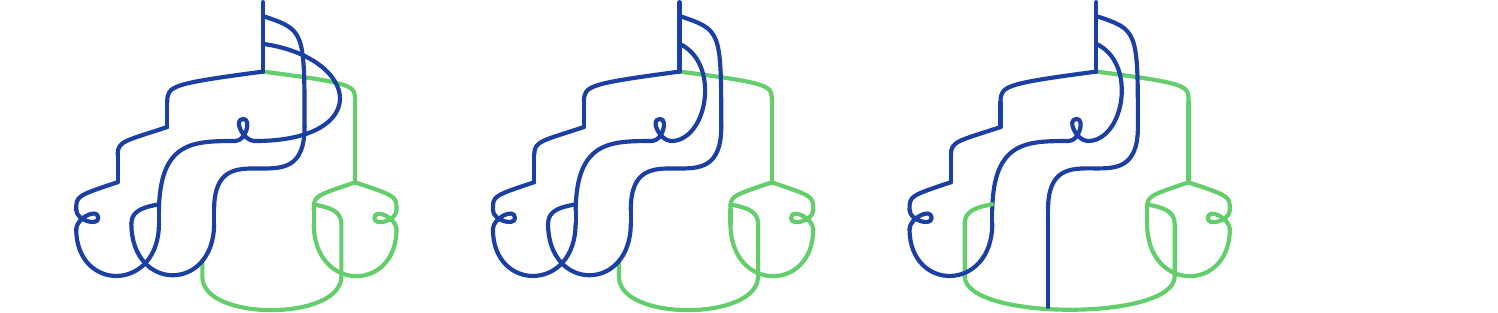}}%
    \put(-0.00066134,0.06794148){\color[rgb]{0,0,0}\makebox(0,0)[lb]{\smash{$\overset{\substack{\text{\eqref{fig:diagram_mult2}} \\ + \\ \text{\ref{it:phi_three}}}}{=}$}}}%
    \put(0.27764491,0.06794148){\color[rgb]{0,0,0}\makebox(0,0)[lb]{\smash{$\overset{\substack{\text{\ref{fig:axiom_Reid_3}} \\ + \\ \text{\ref{fig:axiom_square}}}}{=}$}}}%
    \put(0.55595115,0.06794148){\color[rgb]{0,0,0}\makebox(0,0)[lb]{\smash{$\overset{\text{\eqref{fig:diagram_mult2}}}{=}$}}}%
  \end{picture}%
\endgroup%
	
$$
%%%%%%%%%%%%%%%%%%%%%%%%%%%%%%%
%%%%%%%%%%%%%%%%%%%%%%%%%%%%%%%
$$
%% Creator: Inkscape 0.48.2, www.inkscape.org
%% PDF/EPS/PS + LaTeX output extension by Johan Engelen, 2010
%% Accompanies image file 'bigproof16.pdf' (pdf, eps, ps)
%%
%% To include the image in your LaTeX document, write
%%   \input{<filename>.pdf_tex}
%%  instead of
%%   \includegraphics{<filename>.pdf}
%% To scale the image, write
%%   \def\svgwidth{<desired width>}
%%   \input{<filename>.pdf_tex}
%%  instead of
%%   \includegraphics[width=<desired width>]{<filename>.pdf}
%%
%% Images with a different path to the parent latex file can
%% be accessed with the `import' package (which may need to be
%% installed) using
%%   \usepackage{import}
%% in the preamble, and then including the image with
%%   \import{<path to file>}{<filename>.pdf_tex}
%% Alternatively, one can specify
%%   \graphicspath{{<path to file>/}}
%% 
%% For more information, please see info/svg-inkscape on CTAN:
%%   http://tug.ctan.org/tex-archive/info/svg-inkscape
%%
\begingroup%
  \makeatletter%
  \providecommand\color[2][]{%
    \errmessage{(Inkscape) Color is used for the text in Inkscape, but the package 'color.sty' is not loaded}%
    \renewcommand\color[2][]{}%
  }%
  \providecommand\transparent[1]{%
    \errmessage{(Inkscape) Transparency is used (non-zero) for the text in Inkscape, but the package 'transparent.sty' is not loaded}%
    \renewcommand\transparent[1]{}%
  }%
  \providecommand\rotatebox[2]{#2}%
  \ifx\svgwidth\undefined%
    \setlength{\unitlength}{521.3828125bp}%
    \ifx\svgscale\undefined%
      \relax%
    \else%
      \setlength{\unitlength}{\unitlength * \real{\svgscale}}%
    \fi%
  \else%
    \setlength{\unitlength}{\svgwidth}%
  \fi%
  \global\let\svgwidth\undefined%
  \global\let\svgscale\undefined%
  \makeatother%
  \begin{picture}(1,0.17225598)%
    \put(0,0){\includegraphics[width=\unitlength]{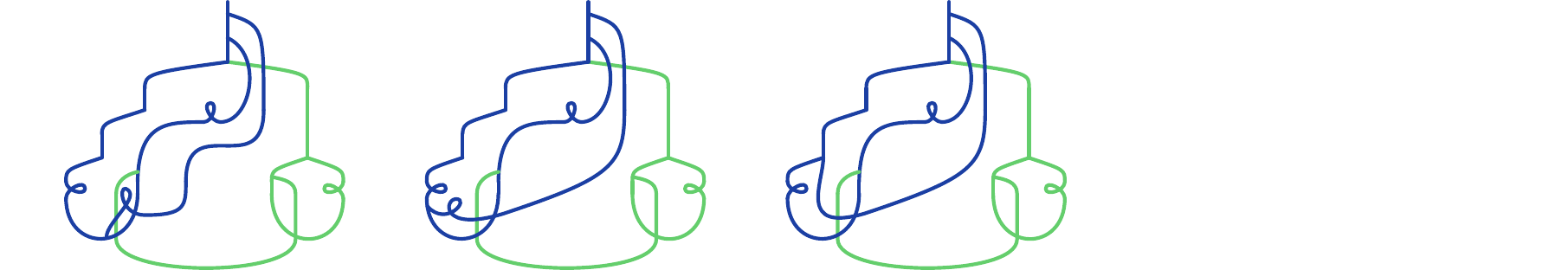}}%
    \put(0.45976745,0.05598913){\color[rgb]{0,0,0}\makebox(0,0)[lb]{\smash{$\overset{\substack{\text{\ref{it:phi_three}} \\ + \\ \text{\eqref{eq:prop_sep}}}}{=}$}}}%
    \put(0.22961026,0.05598913){\color[rgb]{0,0,0}\makebox(0,0)[lb]{\smash{$\overset{\substack{\text{\ref{it:phi_two}} \\ + \\ \text{\ref{it:phi_one}}}}{=}$}}}%
    \put(-0.00054692,0.05598913){\color[rgb]{0,0,0}\makebox(0,0)[lb]{\smash{$\overset{\substack{\text{\ref{fig:axiom_mult}} \\ + \\ \text{\eqref{fig:diagram_mult2}}}}{=}$}}}%
  \end{picture}%
\endgroup%
	
\hspace{20mm}
$$
%%%%%%%%%%%%%%%%%%%%%%%%%%%%%%%
%%%%%%%%%%%%%%%%%%%%%%%%%%%%%%%
$$
%% Creator: Inkscape 0.48.2, www.inkscape.org
%% PDF/EPS/PS + LaTeX output extension by Johan Engelen, 2010
%% Accompanies image file 'bigproof17.pdf' (pdf, eps, ps)
%%
%% To include the image in your LaTeX document, write
%%   \input{<filename>.pdf_tex}
%%  instead of
%%   \includegraphics{<filename>.pdf}
%% To scale the image, write
%%   \def\svgwidth{<desired width>}
%%   \input{<filename>.pdf_tex}
%%  instead of
%%   \includegraphics[width=<desired width>]{<filename>.pdf}
%%
%% Images with a different path to the parent latex file can
%% be accessed with the `import' package (which may need to be
%% installed) using
%%   \usepackage{import}
%% in the preamble, and then including the image with
%%   \import{<path to file>}{<filename>.pdf_tex}
%% Alternatively, one can specify
%%   \graphicspath{{<path to file>/}}
%% 
%% For more information, please see info/svg-inkscape on CTAN:
%%   http://tug.ctan.org/tex-archive/info/svg-inkscape
%%
\begingroup%
  \makeatletter%
  \providecommand\color[2][]{%
    \errmessage{(Inkscape) Color is used for the text in Inkscape, but the package 'color.sty' is not loaded}%
    \renewcommand\color[2][]{}%
  }%
  \providecommand\transparent[1]{%
    \errmessage{(Inkscape) Transparency is used (non-zero) for the text in Inkscape, but the package 'transparent.sty' is not loaded}%
    \renewcommand\transparent[1]{}%
  }%
  \providecommand\rotatebox[2]{#2}%
  \ifx\svgwidth\undefined%
    \setlength{\unitlength}{430.2890625bp}%
    \ifx\svgscale\undefined%
      \relax%
    \else%
      \setlength{\unitlength}{\unitlength * \real{\svgscale}}%
    \fi%
  \else%
    \setlength{\unitlength}{\svgwidth}%
  \fi%
  \global\let\svgwidth\undefined%
  \global\let\svgscale\undefined%
  \makeatother%
  \begin{picture}(1,0.20872319)%
    \put(0,0){\includegraphics[width=\unitlength]{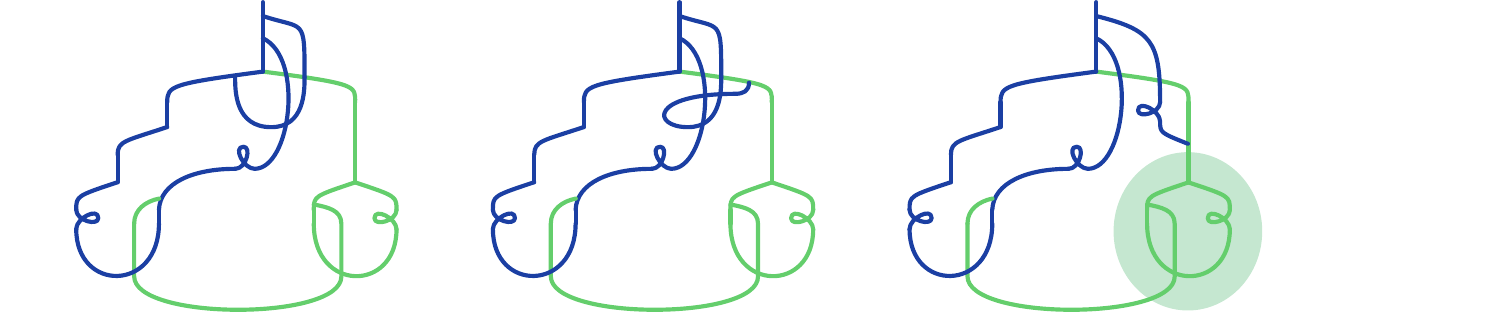}}%
    \put(-0.00066271,0.06784223){\color[rgb]{0,0,0}\makebox(0,0)[lb]{\smash{$\overset{\substack{\text{\ref{fig:axiom_mult}}\\ + \\ \text{\ref{it:phi_three}}}}{=}$}}}%
    \put(0.27821958,0.06784223){\color[rgb]{0,0,0}\makebox(0,0)[lb]{\smash{$\overset{\substack{\text{\eqref{fig:diagram_mult2}}\\ + \\ \text{\ref{fig:axiom_mult}}}}{=}$}}}%
    \put(0.55710188,0.06784223){\color[rgb]{0,0,0}\makebox(0,0)[lb]{\smash{$\overset{\text{\ref{fig:axiom_square}}}{=}$}}}%
  \end{picture}%
\endgroup%
	
\hspace{20mm}
$$
%%%%%%%%%%%%%%%%%%%%%%%%%%%%%%%
%%%%%%%%%%%%%%%%%%%%%%%%%%%%%%%
$$
%% Creator: Inkscape 0.48.2, www.inkscape.org
%% PDF/EPS/PS + LaTeX output extension by Johan Engelen, 2010
%% Accompanies image file 'bigproof18.pdf' (pdf, eps, ps)
%%
%% To include the image in your LaTeX document, write
%%   \input{<filename>.pdf_tex}
%%  instead of
%%   \includegraphics{<filename>.pdf}
%% To scale the image, write
%%   \def\svgwidth{<desired width>}
%%   \input{<filename>.pdf_tex}
%%  instead of
%%   \includegraphics[width=<desired width>]{<filename>.pdf}
%%
%% Images with a different path to the parent latex file can
%% be accessed with the `import' package (which may need to be
%% installed) using
%%   \usepackage{import}
%% in the preamble, and then including the image with
%%   \import{<path to file>}{<filename>.pdf_tex}
%% Alternatively, one can specify
%%   \graphicspath{{<path to file>/}}
%% 
%% For more information, please see info/svg-inkscape on CTAN:
%%   http://tug.ctan.org/tex-archive/info/svg-inkscape
%%
\begingroup%
  \makeatletter%
  \providecommand\color[2][]{%
    \errmessage{(Inkscape) Color is used for the text in Inkscape, but the package 'color.sty' is not loaded}%
    \renewcommand\color[2][]{}%
  }%
  \providecommand\transparent[1]{%
    \errmessage{(Inkscape) Transparency is used (non-zero) for the text in Inkscape, but the package 'transparent.sty' is not loaded}%
    \renewcommand\transparent[1]{}%
  }%
  \providecommand\rotatebox[2]{#2}%
  \ifx\svgwidth\undefined%
    \setlength{\unitlength}{553.8203125bp}%
    \ifx\svgscale\undefined%
      \relax%
    \else%
      \setlength{\unitlength}{\unitlength * \real{\svgscale}}%
    \fi%
  \else%
    \setlength{\unitlength}{\svgwidth}%
  \fi%
  \global\let\svgwidth\undefined%
  \global\let\svgscale\undefined%
  \makeatother%
  \begin{picture}(1,0.16216687)%
    \put(0,0){\includegraphics[width=\unitlength]{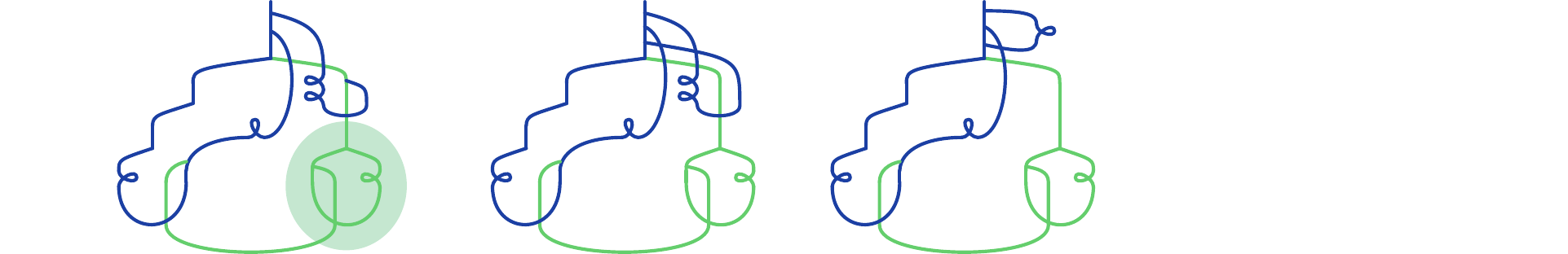}}%
    \put(-0.00051489,0.05270982){\color[rgb]{0,0,0}\makebox(0,0)[lb]{\smash{$\overset{n(a) \in \overline{\mathcal{Z}}_{\lambda}(A)}{=}$}}}%
    \put(0.2667198,0.05270982){\color[rgb]{0,0,0}\makebox(0,0)[lb]{\smash{$\overset{\substack{\text{\ref{fig:axiom_mult}}\\ + \\ \text{\eqref{fig:diagram_mult2}}}}{=}$}}}%
    \put(0.49061914,0.05270982){\color[rgb]{0,0,0}\makebox(0,0)[lb]{\smash{$\overset{\substack{\text{\ref{it:phi_one}}\\ + \\ \text{\ref{fig:axiom_square}}}}{=}$}}}%
  \end{picture}%
\endgroup%
	
\hspace{20mm}
$$
%%%%%%%%%%%%%%%%%%%%%%%%%%%%%%%
%%%%%%%%%%%%%%%%%%%%%%%%%%%%%%%
$$
%% Creator: Inkscape 0.48.2, www.inkscape.org
%% PDF/EPS/PS + LaTeX output extension by Johan Engelen, 2010
%% Accompanies image file 'bigproof19.pdf' (pdf, eps, ps)
%%
%% To include the image in your LaTeX document, write
%%   \input{<filename>.pdf_tex}
%%  instead of
%%   \includegraphics{<filename>.pdf}
%% To scale the image, write
%%   \def\svgwidth{<desired width>}
%%   \input{<filename>.pdf_tex}
%%  instead of
%%   \includegraphics[width=<desired width>]{<filename>.pdf}
%%
%% Images with a different path to the parent latex file can
%% be accessed with the `import' package (which may need to be
%% installed) using
%%   \usepackage{import}
%% in the preamble, and then including the image with
%%   \import{<path to file>}{<filename>.pdf_tex}
%% Alternatively, one can specify
%%   \graphicspath{{<path to file>/}}
%% 
%% For more information, please see info/svg-inkscape on CTAN:
%%   http://tug.ctan.org/tex-archive/info/svg-inkscape
%%
\begingroup%
  \makeatletter%
  \providecommand\color[2][]{%
    \errmessage{(Inkscape) Color is used for the text in Inkscape, but the package 'color.sty' is not loaded}%
    \renewcommand\color[2][]{}%
  }%
  \providecommand\transparent[1]{%
    \errmessage{(Inkscape) Transparency is used (non-zero) for the text in Inkscape, but the package 'transparent.sty' is not loaded}%
    \renewcommand\transparent[1]{}%
  }%
  \providecommand\rotatebox[2]{#2}%
  \ifx\svgwidth\undefined%
    \setlength{\unitlength}{607.0078125bp}%
    \ifx\svgscale\undefined%
      \relax%
    \else%
      \setlength{\unitlength}{\unitlength * \real{\svgscale}}%
    \fi%
  \else%
    \setlength{\unitlength}{\svgwidth}%
  \fi%
  \global\let\svgwidth\undefined%
  \global\let\svgscale\undefined%
  \makeatother%
  \begin{picture}(1,0.21404896)%
    \put(0,0){\includegraphics[width=\unitlength]{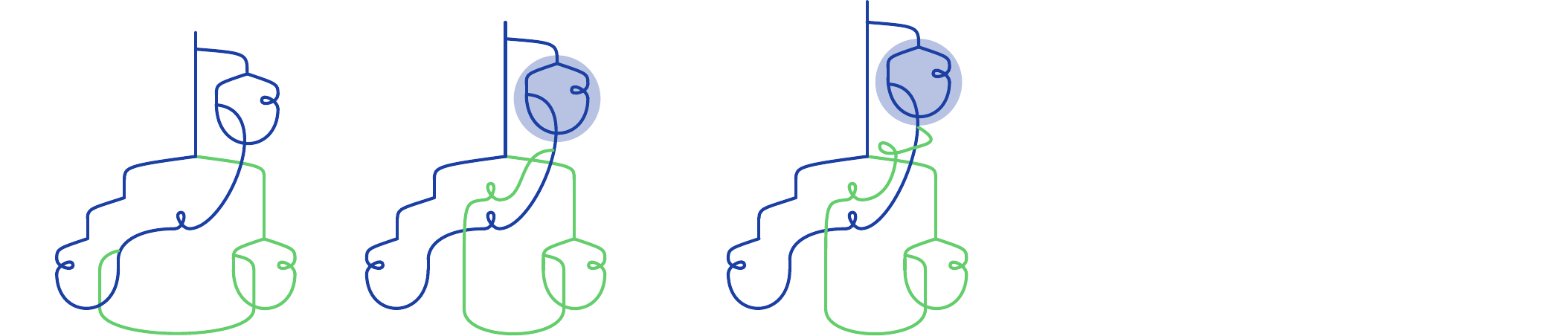}}%
    \put(-0.00046977,0.04828579){\color[rgb]{0,0,0}\makebox(0,0)[lb]{\smash{$\overset{\substack{\text{\ref{fig:axiom_left_right}}\\ + \\ \text{\eqref{fig:diagram_mult2}}}}{=}$}}}%
    \put(0.19722126,0.04828579){\color[rgb]{0,0,0}\makebox(0,0)[lb]{\smash{$\overset{\substack{\text{\ref{it:phi_two}}\\ +\\ \text{\ref{fig:axiom_mult}}}}{=}$}}}%
    \put(0.39491229,0.04828579){\color[rgb]{0,0,0}\makebox(0,0)[lb]{\smash{$\overset{\substack{ n=\sigma \circ n \circ \sigma \\ + \\ n(a) \in \overline{\mathcal{Z}}_{\lambda}(A)}}{=}$}}}%
  \end{picture}%
\endgroup%
	
\hspace{20mm}
$$
%%%%%%%%%%%%%%%%%%%%%%%%%%%%%%%
%%%%%%%%%%%%%%%%%%%%%%%%%%%%%%%
$$
%% Creator: Inkscape 0.48.2, www.inkscape.org
%% PDF/EPS/PS + LaTeX output extension by Johan Engelen, 2010
%% Accompanies image file 'bigproof20.pdf' (pdf, eps, ps)
%%
%% To include the image in your LaTeX document, write
%%   \input{<filename>.pdf_tex}
%%  instead of
%%   \includegraphics{<filename>.pdf}
%% To scale the image, write
%%   \def\svgwidth{<desired width>}
%%   \input{<filename>.pdf_tex}
%%  instead of
%%   \includegraphics[width=<desired width>]{<filename>.pdf}
%%
%% Images with a different path to the parent latex file can
%% be accessed with the `import' package (which may need to be
%% installed) using
%%   \usepackage{import}
%% in the preamble, and then including the image with
%%   \import{<path to file>}{<filename>.pdf_tex}
%% Alternatively, one can specify
%%   \graphicspath{{<path to file>/}}
%% 
%% For more information, please see info/svg-inkscape on CTAN:
%%   http://tug.ctan.org/tex-archive/info/svg-inkscape
%%
\begingroup%
  \makeatletter%
  \providecommand\color[2][]{%
    \errmessage{(Inkscape) Color is used for the text in Inkscape, but the package 'color.sty' is not loaded}%
    \renewcommand\color[2][]{}%
  }%
  \providecommand\transparent[1]{%
    \errmessage{(Inkscape) Transparency is used (non-zero) for the text in Inkscape, but the package 'transparent.sty' is not loaded}%
    \renewcommand\transparent[1]{}%
  }%
  \providecommand\rotatebox[2]{#2}%
  \ifx\svgwidth\undefined%
    \setlength{\unitlength}{517.3515625bp}%
    \ifx\svgscale\undefined%
      \relax%
    \else%
      \setlength{\unitlength}{\unitlength * \real{\svgscale}}%
    \fi%
  \else%
    \setlength{\unitlength}{\svgwidth}%
  \fi%
  \global\let\svgwidth\undefined%
  \global\let\svgscale\undefined%
  \makeatother%
  \begin{picture}(1,0.25114332)%
    \put(0,0){\includegraphics[width=\unitlength]{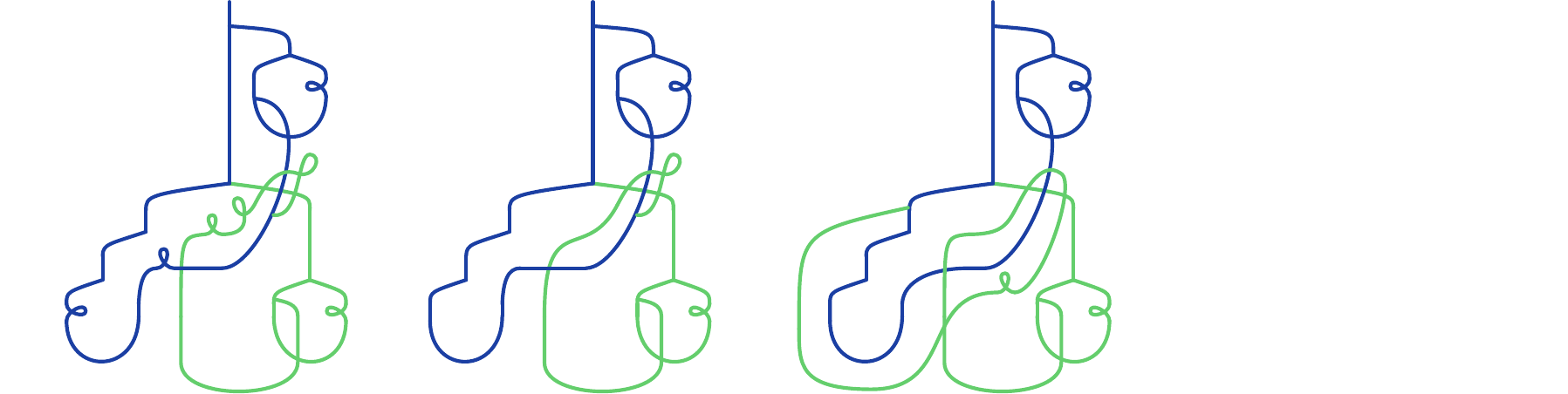}}%
    \put(-0.00055118,0.05665364){\color[rgb]{0,0,0}\makebox(0,0)[lb]{\smash{$\overset{\substack{\text{\ref{it:phi_three}} \\ + \\ \text{\ref{fig:axiom_mult}}}}{=}$}}}%
    \put(0.23139941,0.05665364){\color[rgb]{0,0,0}\makebox(0,0)[lb]{\smash{$\overset{\substack{\text{\ref{it:phi_four}}\\ + \\ \text{\ref{it:phi_one}}}}{=}$}}}%
    \put(0.46334999,0.05665364){\color[rgb]{0,0,0}\makebox(0,0)[lb]{\smash{$\overset{\substack{\text{\ref{fig:axiom_mult}}\\ + \\ \text{\ref{it:phi_three}}}}{=}$}}}%
  \end{picture}%
\endgroup%
	
\hspace{20mm}
$$
%%%%%%%%%%%%%%%%%%%%%%%%%%%%%%%
%%%%%%%%%%%%%%%%%%%%%%%%%%%%%%%
$$
%% Creator: Inkscape 0.48.2, www.inkscape.org
%% PDF/EPS/PS + LaTeX output extension by Johan Engelen, 2010
%% Accompanies image file 'bigproof21.pdf' (pdf, eps, ps)
%%
%% To include the image in your LaTeX document, write
%%   \input{<filename>.pdf_tex}
%%  instead of
%%   \includegraphics{<filename>.pdf}
%% To scale the image, write
%%   \def\svgwidth{<desired width>}
%%   \input{<filename>.pdf_tex}
%%  instead of
%%   \includegraphics[width=<desired width>]{<filename>.pdf}
%%
%% Images with a different path to the parent latex file can
%% be accessed with the `import' package (which may need to be
%% installed) using
%%   \usepackage{import}
%% in the preamble, and then including the image with
%%   \import{<path to file>}{<filename>.pdf_tex}
%% Alternatively, one can specify
%%   \graphicspath{{<path to file>/}}
%% 
%% For more information, please see info/svg-inkscape on CTAN:
%%   http://tug.ctan.org/tex-archive/info/svg-inkscape
%%
\begingroup%
  \makeatletter%
  \providecommand\color[2][]{%
    \errmessage{(Inkscape) Color is used for the text in Inkscape, but the package 'color.sty' is not loaded}%
    \renewcommand\color[2][]{}%
  }%
  \providecommand\transparent[1]{%
    \errmessage{(Inkscape) Transparency is used (non-zero) for the text in Inkscape, but the package 'transparent.sty' is not loaded}%
    \renewcommand\transparent[1]{}%
  }%
  \providecommand\rotatebox[2]{#2}%
  \ifx\svgwidth\undefined%
    \setlength{\unitlength}{576.1015625bp}%
    \ifx\svgscale\undefined%
      \relax%
    \else%
      \setlength{\unitlength}{\unitlength * \real{\svgscale}}%
    \fi%
  \else%
    \setlength{\unitlength}{\svgwidth}%
  \fi%
  \global\let\svgwidth\undefined%
  \global\let\svgscale\undefined%
  \makeatother%
  \begin{picture}(1,0.22553208)%
    \put(0,0){\includegraphics[width=\unitlength]{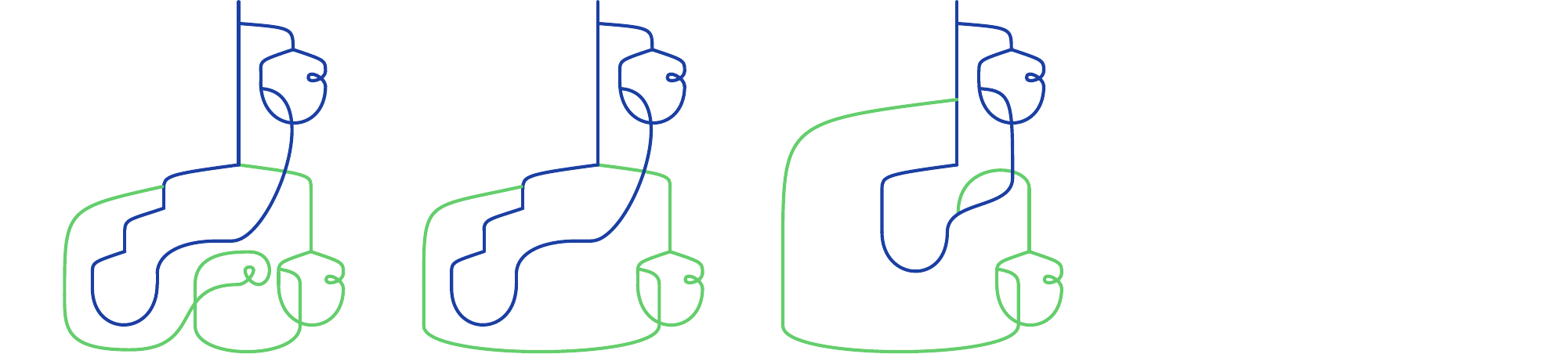}}%
    \put(-0.00049498,0.05087617){\color[rgb]{0,0,0}\makebox(0,0)[lb]{\smash{$\overset{\substack{\text{\ref{fig:axiom_Reid_3}}\\ + \\ \text{\ref{fig:axiom_square}}}}{=}$}}}%
    \put(0.22863129,0.05087617){\color[rgb]{0,0,0}\makebox(0,0)[lb]{\smash{$\overset{\substack{\text{\ref{it:phi_three}}\\ + \\ \text{\ref{it:phi_one}}}}{=}$}}}%
    \put(0.45775756,0.05087617){\color[rgb]{0,0,0}\makebox(0,0)[lb]{\smash{$\overset{\substack{\text{\eqref{fig:diagram_mult2}}\\ + \\ \text{\eqref{eq:prop_sep}}}}{=}$}}}%
  \end{picture}%
\endgroup%
	
\hspace{20mm}
$$
%%%%%%%%%%%%%%%%%%%%%%%%%%%%%%%
%%%%%%%%%%%%%%%%%%%%%%%%%%%%%%%
$$
%% Creator: Inkscape 0.48.2, www.inkscape.org
%% PDF/EPS/PS + LaTeX output extension by Johan Engelen, 2010
%% Accompanies image file 'bigproof22.pdf' (pdf, eps, ps)
%%
%% To include the image in your LaTeX document, write
%%   \input{<filename>.pdf_tex}
%%  instead of
%%   \includegraphics{<filename>.pdf}
%% To scale the image, write
%%   \def\svgwidth{<desired width>}
%%   \input{<filename>.pdf_tex}
%%  instead of
%%   \includegraphics[width=<desired width>]{<filename>.pdf}
%%
%% Images with a different path to the parent latex file can
%% be accessed with the `import' package (which may need to be
%% installed) using
%%   \usepackage{import}
%% in the preamble, and then including the image with
%%   \import{<path to file>}{<filename>.pdf_tex}
%% Alternatively, one can specify
%%   \graphicspath{{<path to file>/}}
%% 
%% For more information, please see info/svg-inkscape on CTAN:
%%   http://tug.ctan.org/tex-archive/info/svg-inkscape
%%
\begingroup%
  \makeatletter%
  \providecommand\color[2][]{%
    \errmessage{(Inkscape) Color is used for the text in Inkscape, but the package 'color.sty' is not loaded}%
    \renewcommand\color[2][]{}%
  }%
  \providecommand\transparent[1]{%
    \errmessage{(Inkscape) Transparency is used (non-zero) for the text in Inkscape, but the package 'transparent.sty' is not loaded}%
    \renewcommand\transparent[1]{}%
  }%
  \providecommand\rotatebox[2]{#2}%
  \ifx\svgwidth\undefined%
    \setlength{\unitlength}{573.8203125bp}%
    \ifx\svgscale\undefined%
      \relax%
    \else%
      \setlength{\unitlength}{\unitlength * \real{\svgscale}}%
    \fi%
  \else%
    \setlength{\unitlength}{\svgwidth}%
  \fi%
  \global\let\svgwidth\undefined%
  \global\let\svgscale\undefined%
  \makeatother%
  \begin{picture}(1,0.22784865)%
    \put(0,0){\includegraphics[width=\unitlength]{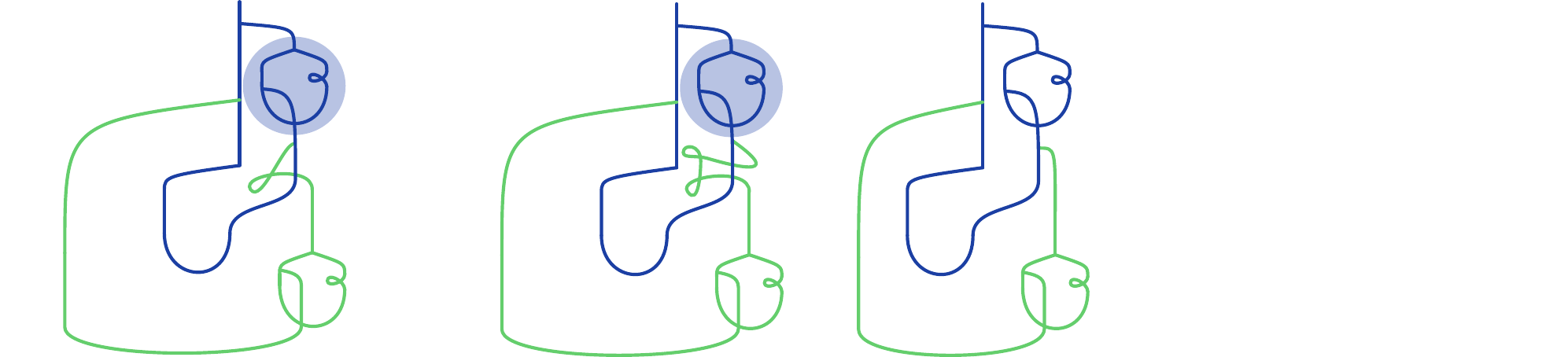}}%
    \put(0.22954023,0.05110423){\color[rgb]{0,0,0}\makebox(0,0)[lb]{\smash{$\overset{\substack{n= \sigma \circ n \circ \sigma\\ + \\ n(a) \in \overline{\mathcal{Z}}_{\lambda}(A)}}{=}$}}}%
    \put(-0.00049694,0.05110423){\color[rgb]{0,0,0}\makebox(0,0)[lb]{\smash{$\overset{\text{\ref{fig:axiom_mult}}}{=}$}}}%
    \put(0.50837316,0.05110423){\color[rgb]{0,0,0}\makebox(0,0)[lb]{\smash{$\overset{\substack{\text{\ref{it:phi_one}}\\ + \\ \text{\ref{fig:axiom_square}}}}{=}$}}}%
  \end{picture}%
\endgroup%
	
\hspace{20mm}
$$
%%%%%%%%%%%%%%%%%%%%%%%%%%%%%%%
%%%%%%%%%%%%%%%%%%%%%%%%%%%%%%%
$$
%% Creator: Inkscape 0.48.2, www.inkscape.org
%% PDF/EPS/PS + LaTeX output extension by Johan Engelen, 2010
%% Accompanies image file 'bigproof23.pdf' (pdf, eps, ps)
%%
%% To include the image in your LaTeX document, write
%%   \input{<filename>.pdf_tex}
%%  instead of
%%   \includegraphics{<filename>.pdf}
%% To scale the image, write
%%   \def\svgwidth{<desired width>}
%%   \input{<filename>.pdf_tex}
%%  instead of
%%   \includegraphics[width=<desired width>]{<filename>.pdf}
%%
%% Images with a different path to the parent latex file can
%% be accessed with the `import' package (which may need to be
%% installed) using
%%   \usepackage{import}
%% in the preamble, and then including the image with
%%   \import{<path to file>}{<filename>.pdf_tex}
%% Alternatively, one can specify
%%   \graphicspath{{<path to file>/}}
%% 
%% For more information, please see info/svg-inkscape on CTAN:
%%   http://tug.ctan.org/tex-archive/info/svg-inkscape
%%
\begingroup%
  \makeatletter%
  \providecommand\color[2][]{%
    \errmessage{(Inkscape) Color is used for the text in Inkscape, but the package 'color.sty' is not loaded}%
    \renewcommand\color[2][]{}%
  }%
  \providecommand\transparent[1]{%
    \errmessage{(Inkscape) Transparency is used (non-zero) for the text in Inkscape, but the package 'transparent.sty' is not loaded}%
    \renewcommand\transparent[1]{}%
  }%
  \providecommand\rotatebox[2]{#2}%
  \ifx\svgwidth\undefined%
    \setlength{\unitlength}{434.4609375bp}%
    \ifx\svgscale\undefined%
      \relax%
    \else%
      \setlength{\unitlength}{\unitlength * \real{\svgscale}}%
    \fi%
  \else%
    \setlength{\unitlength}{\svgwidth}%
  \fi%
  \global\let\svgwidth\undefined%
  \global\let\svgscale\undefined%
  \makeatother%
  \begin{picture}(1,0.16562412)%
    \put(0,0){\includegraphics[width=\unitlength]{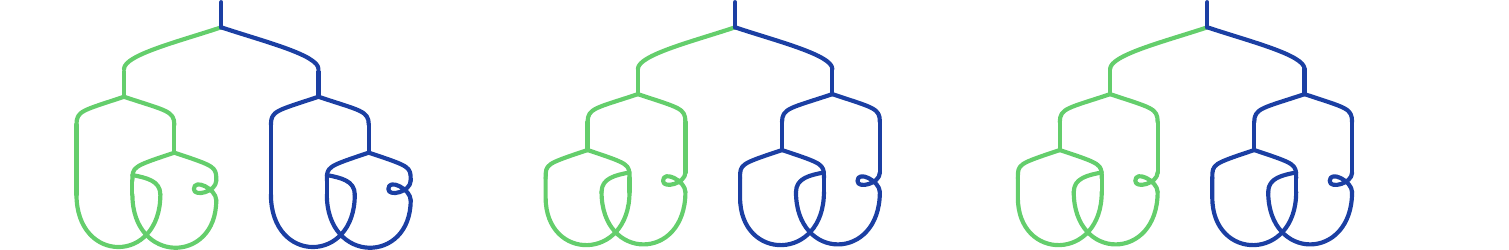}}%
    \put(-0.00065634,0.0462452){\color[rgb]{0,0,0}\makebox(0,0)[lb]{\smash{$\overset{\text{\eqref{fig:diagram_mult2}}}{=} $}}}%
    \put(0.28475481,0.0462452){\color[rgb]{0,0,0}\makebox(0,0)[lb]{\smash{$\overset{\text{\eqref{fig:diagram_mult2}}}{=}$}}}%
    \put(0.59778641,0.0462452){\color[rgb]{0,0,0}\makebox(0,0)[lb]{\smash{$\overset{p = p \circ \varphi}{=}$}}}%
    \put(0.92002482,0.0462452){\color[rgb]{0,0,0}\makebox(0,0)[lb]{\smash{$=\eta^2$}}}%
  \end{picture}%
\endgroup%
	
\hspace{20mm}
$$
%%%%%%%%%%%%%%%%%%%%%%%%%%%%%%%
\end{proof}
The partition function for a surface with spin structure can now be presented. Each handle contributes the element $\gls{chi}$ or $\gls{eta}$ depending on the spin structure; the partition function is thus
\begin{align}\label{eq:Z_general}
\gls{Z}(\gls{Sigma_g},\gls{s})=\gls{R}\,\gls{eps}(\gls{eta}^{g-l}\gls{chi}^{l}).
\end{align}
However, the properties described in proposition~\ref{lem:properties_eta_chi} mean all that matters is $l\,\mo 2$ which reflects the fact the partition function is a homeomorphism invariant. The only homeomorphism invariant of a spin structure is the Arf invariant of the quadratic form, in this case $\text{arf}(q)=l\,\mo 2\in\Zb_2$. It is most convenient to express this invariant of the spin structure as the parity $P(s)=(-1)^{\text{arf}(q)}$.   

 These results are collected together to give the main result for this section.
\begin{theorem} \label{theo:main}
Let $(\gls{Ct},\gls{Bb},R,\gls{lambda})$ be a spin state sum model. Then,
\begin{align} \label{eq:spin_partition}
Z(\Sigma_g,s)=\left\lbrace
\begin{array}{lcl}
R\,\varepsilon(\eta^g)&& (s \text{ is even})\\
R\,\varepsilon(\chi\eta^{g-1})&& (s \text{ is odd}).
\end{array}
\right. 
\end{align}
\end{theorem}
Note that $Z(\Sigma_0)$ is independent of the choice of $\lambda$, as is to be expected. According to the classification of planar state sum models given in theorem~\ref{theo:diagram_semi_simple}, $\eta \in \gls{ZA}$ implies that $\eta=\oplus_i \,\eta_i\,1_i$ for some constants $\eta_i \in \Rb, \Cb_{\Rb}$ or $\Cb$. The expression for $\chi$ will therefore be $\chi=\oplus_i \,\text{sgn}_i\,\eta_i\,1_i$, where $\text{sgn}_i=\pm 1$, since $\chi$ is also a central element and $\eta^2=\chi^2$. In particular this means simple matrix algebras can at most attribute different signs to spin structures of different parity.

An algebraic condition that guarantees topologically-inequivalent spin structures cannot be distinguished is $\eta=\chi$. It is now shown that the canonical crossing map gives rise to spin state sum models that fall into this class.
\begin{corollary} \label{cor:canonical_crossing}
Let $\lambda \colon \gls{A} \otimes A \to A \otimes A$ be such that $a \otimes b \mapsto b \otimes a$. Then $\chi=\eta$, implying the partition function does not depend on the spin structure.
\end{corollary}
\begin{proof}
For a crossing of the form above it is easy to conclude $\gls{phi}=\gls{sigma}$ where $\sigma$ represents the Nakayama automorphism associated with the Frobenius form $\varepsilon$ as in definition~\ref{Nakayama}. The set $\overline{\mathcal{Z}}_{\lambda}(A)$ coincides in this case with the set of elements $a \in A$ satisfying $a\cdot b=\sigma(b) \cdot a$ for all $b \in A$. Recall that if an algebra $A$ satisfies the conditions of theorem~\ref{theo:diagram} then $\sigma$ is an inner automorphism: $\sigma(a)=\gls{x}\cdot a \cdot x^{-1}$. Then it is possible to conclude $\varphi \circ n=n$:
\begin{align} 
\gls{phi} \circ n(a)=\gls{sigma} \circ n (a)=\gls{x} \cdot n(a) \cdot x^{-1}=x \sigma(x^{-1}) \cdot n(a)=n(a).
\end{align}
The diagrammatic form of $\gls{eta}$ and $\gls{chi}$ implies that $\eta=\chi$ if the maps $n$ and $\varphi \circ n$ coincide. Then theorem \ref{theo:main} implies that the partition function does not distinguish spin parity. 
\end{proof}

We are now ready to start discussing examples. Below we finally understand how FHK models fit within the more general framework.  

\begin{example}
An FHK state sum model as defined in \S\ref{sec:lattice_tft} is a curl-free state sum model where the choice of crossing is canonical. In other words, the map $\gls{lambda} \colon \gls{A} \otimes A \to A \otimes A$ takes $a \otimes b \mapsto b \otimes a$. 
\end{example}

We also introduce the most general form of an invariant constructed from a planar model and the canonical crossing. Note that for such a spin model the fact $\varphi^2=\iden$ is equivalent to having the Nakayama automorphism satisfy $\sigma^2=\iden$.

\begin{example} \label{ex:semi_simple_canonical}
Let $A$ be a complex algebra in the conditions of theorem~\ref{theo:diagram_semi_simple}. Then a spin model determined by $A$ and the canonical crossing generates the partition function
\begin{align}
\label{eq:Z_canonical_complex}
\gls{Z}(\gls{Sigma_g})= \gls{R}^{1-g}\sum_{i=1}^N \Tr(x_i)^{1-g}, \hspace{5mm} x_i^2=R\Tr(x_i)1_i.
\end{align}
To see this let us first recall the bilinear form $B$ can be written as in equation \eqref{B_form_complex}. The associated expression for $\eta$ will then be
\begin{align}
\eta = \sum_i \sum_{lmnp} e^i_{lm}x_i^{-1}e^i_{np}x^{-1}_ie^i_{ml}e^i_{pn}.
\end{align}
Using the identities $\sum_{lm}e^{i}_{lm}y_ie^{i}_{ml}=\Tr_i(y_i)1_i$ and $\sum_{np}\Tr_i(y_ie^i_{np})e^i_{pn}=y_i$ valid for all $y \in A_i$ we can conclude the expression above simplifies to $\eta=\sum_i x_i^{-2}$. Since $\sigma^2=\iden$ and $R\Tr(x_i^{-1})=1_i$, the $x_i$ must satisfy $x_i^2=R\Tr(x_i)1$. Therefore, $\eta=\frac{1}{R}\oplus_i \frac{1_i}{\Tr(x_i)}$. Expression \eqref{eq:spin_partition} for the partition function delivers the claimed result. 

If we chose $\gls{A}$ to be a real algebra, the invariant would read instead
\begin{align}
\label{eq:Z_canonical_real}
\gls{Z}(\gls{Sigma_g})= \gls{R}^{1-g}\sum_{i=1}^N f(i,g)\Real\Tr(x_i)^{1-g}, \hspace{5mm} x_i^2=|D_i|R\Real\Tr(x_i)1_i
\end{align}
where $f(i,g)=1$ if $D_i=\Rb$, $\Hb_{\Rb}$ and $f(i,g)=2^g$ if $D_i=\Cb_{\Rb}$.

To see this we use the form of $B$ given by expression \eqref{eq:B_form} to determine $\eta$:
\begin{align}
\gls{eta} = \bigoplus_i \eta_i = \bigoplus_{i} \sum_{w_i t_i} \sum_{lmnp} w_ie^i_{lm}x_i^{-1}t_ie^i_{np}x^{-1}_ie^i_{ml}e^i_{pn}\overline{w_it_i}.
\end{align}
It is easier to treat each $\eta_i$ separately. If $D_i = \Rb$, $\Cb_{\Rb}$ the $w_i,t_i$ commute with all the elements of the algebra and we conclude $\eta_i = |D_i|^2 x_i^{-2}$ following the techniques used in the complex case.

If $D_i=\Hb_{\Rb}$, however, a little more work is needed -- the index $i$ is suppressed on algebra elements, for simplicity. We can use the identity $\sum_{lm}e_{lm}y_ie_ml=\Tr_i(y)1_i$ to learn 
\begin{align}
\eta_i = \sum_{w t} \sum_{np} w \Tr_i\left(x^{-1}t e_{np}x^{-1}\right) e_{pn}\overline{wt}.
\end{align}
By noticing $\sum_{w} wy_i\overline{w}=|D_i|\Real(y_i)$ we have 
\begin{align}
\eta_i = \sum_{t} |D_i| \sum_{np}\Real \left(\Tr_i\left( x_i^{-1}te_{np}x_i^{-1} \right)e_{pn} \right) \overline{t}.
\end{align}
If we rewrite $x_i^{-1}=\sum_{s} sx_i^{s}$, we can then replace $\sum_{np}\Tr_i\left( x_i^{s}x_i^{-1}tse_{np} \right)e_{pn}$ with $x_i^{s}x_i^{-1}ts$. In other words,
\begin{align}
\eta_i = \sum_{ts} |D_i| \Real \left( x_i^{s}x_i^{-1}ts \right) \overline{t}.
\end{align}
If we rewrite $ts=w \Leftrightarrow  t=\sign(s)ws$ where $\overline{s}=\sign(s)s$, the sum in $t$ can be re-labeled using $w$. Note as well that $\Real \left( x_i^{s}x_i^{-1}w \right)= \sign(w)x_i^{s}x_i^{w}$:
\begin{align}
\eta_i = \sum_{ws} |D_i| x_i^{s}x_i^{w}\sign(s)\sign(w)\overline{sw}= |D_i|x_i^{-2}.
\end{align}
As before, the identity $x_i^2=|D_i|R\Real\Tr(x_i)1_i$ is a consequence of $\gls{sigma}^2=\iden$. This means $\eta= \frac{1}{R}\oplus_i\frac{|D_i|^{\alpha_i}}{\Real\Tr(x_i)}1_i$ with $\alpha_i=1$ for $D_i=\Rb$, $\Cb_{\Rb}$ and $\alpha_i=0$ for $D_i=\Hb_{\Rb}$. Expression \eqref{eq:spin_partition} delivers the claimed result. 
\end{example}

It is an easy exercise to verify that expressions \eqref{eq:Z_canonical_complex} and \eqref{eq:Z_canonical_real} reduce to \eqref{eq:gen_inv} and \eqref{eq:gen_inv-real} respectively if the Frobenius algebra is symmetric. 

%To finalise this section we present the invariants created by a spin model with canonical crossing for the group algebra $\Cb H$. The partition function reads
%\begin{align}
%Z(\Sigma_g)=R^{2-2g}\sum_{i} \chi_i(x)^{2-2g}.
%\end{align}
%where $x \in H$ satisfies $x^2=1$ and $\chi_i$ represents the characters associated with the representation $i$. 

%%%%%%%%%%%%%%%%%%%%%%%%%%%%%%%%%%%%%%%%%%%%%%%%%%%%%%%%%
%%%%%%%%%%%%%%%%%%%%%%%%%%%%%%%%%%%%%%%%%%%%%%%%%%%%%%%%%
\section{Group graded algebras} \label{sec:graded_algebras}
%%%%%%%%%%%%%%%%%%%%%%%%%%%%%%%%%%%%%%%%%%%%%%%%%%%%%%%%%
%%%%%%%%%%%%%%%%%%%%%%%%%%%%%%%%%%%%%%%%%%%%%%%%%%%%%%%%%
The treatment thus far introduced does not guarantee the existence of true spin models -- models that are not simply topological. In this section we will address this issue by explicitly constructing such spin models. This construction is rather extensive: we aim to display how the inclusion of a crossing allows us to create a large amount of new invariants.  

Examples~\ref{ex:non_symmetric_Frobenius}-\ref{ex:non_abelian} cover a varied range of algebras that give rise to planar models (see theorem~\ref{theo:diagram_semi_simple}) and that can be equipped with a crossing. In particular, we highlight example ~\ref{ex:non_symmetric_Frobenius} of a spin model constructed from a $\Zb_2$-graded algebra (or superalgebra), and examples~\ref{ex:H_awesome}-\ref{ex:non_abelian} where spin parity can be distinguished for group algebras $\Cb G$ when $G$ is non-trivial.

Instead of tackling the classification of crossing maps in its full generality we will confine our study to a special class of crossings. The reasoning behind this approach is as follows. One, the general properties of spin models have already been established through our diagrammatic construction and thus a complete classification of crossings would not bring additional abstract insight. Two, our main goal is to work with a rich source of new examples so that the question of existence can be settled.    

The crossings we will be dealing with arise from a particular class of (semi-simple) algebras: $\gls{H}$-graded algebras $\gls{A}=\bigoplus_{h \in H} A_h$ where $H$ is a finite group. Crossing maps can then be constructed from bicharacters \cite{Bahturin}. A bicharacter $\gls{bi} \colon H \times H \to \gls{k}$ is defined by 
\begin{align} 
\tilde{\lambda}(h,jl)&=\tilde{\lambda}(h,j)\tilde{\lambda}(h,l), \label{eq:graded_mult}\\
\gls{field_unit}&=\tilde{\lambda}(h,j)\tilde{\lambda}(j,h). \label{eq:graded_inverse}
\end{align}   
The candidate for a crossing map $\gls{lambda}$ is then determined by setting 
\begin{equation}
\label{eq:bi_def}
\lambda(a_h \otimes b_j)=\tilde{\lambda}(h,j)\; b_j \otimes a_h \in A_j \otimes A_h.
\end{equation}
The properties of $\tilde{\lambda}$ guarantee that $\tilde{\lambda}(h,j)=\tilde{\lambda}(lhl^{-1},j)$ for all $h,j,l \in H$. This means the arguments of $\tilde{\lambda}$ takes values in $H/\Inn(H)$ where $\Inn(H)$ denotes the group of inner automorphisms of $H$. In other words, $\tilde{\lambda}(h,l)=1_k$ whenever $h \in \Inn(H)$.  Since $H/\Inn(H)$ is isomorphic to the centre of $H$, $\mathcal{Z}(H)$, there is no loss of generality in choosing $H$ to be abelian -- a choice we will resort to from now on. 

From definition \eqref{eq:bi_def} it is straightforward to conclude properties \eqref{eq:graded_mult} and \eqref{eq:graded_inverse} of a bicharacter $\gls{bi}$ are in correspondence with the crossing axioms \ref{fig:axiom_mult} and \ref{fig:axiom_square} of definition \ref{def:spin-model}. On the other hand, axiom \ref{fig:axiom_Reid_3} is automatically verified since $\tilde{\lambda}$ is $k$-valued. The remaining conditions, however, impose new constraints on a bicharacter. Write $A_h\perp A_j$ if $\gls{eps}(a_h \cdot b_j)=0$ for all $a_h\in A_h$, $b_j\in A_j$. 
\begin{proposition} \label{lem:graded} A graded Frobenius algebra with a bicharacter  $\tilde\lambda$ determines a spin state sum model if and only if
\begin{enumerate}[label=D\arabic{*}), ref=(D\arabic{*})]
\item\label{it:cr-axiom1} For each $h,j\in \gls{H}$, either  $A_h\perp A_j$ or $\tilde{\lambda}(h,l)=\tilde{\lambda}(l,j)$  for all $l \in H$.
\item\label{it:cr-axiom2} The Nakayama automorphism $\gls{sigma}$ obeys $\sigma^2=\iden$.
\end{enumerate}
\end{proposition}
\begin{proof} 

Applying the maps in axiom \ref{fig:axiom_form} of definition \ref{def:spin-model} to $a_h\otimes c_l\otimes b_j$ gives
\begin{equation}
\tilde\lambda(l,j)\,\varepsilon(a_h\cdot b_j)\,c_l=\tilde\lambda(h,l)\,\varepsilon(a_h\cdot b_j)\,c_l,
\end{equation}
which is equivalent to condition \ref{it:cr-axiom1}.

The element $\gls{Bb}$ can be written as a sum of linearly independent terms as $B=\sum y_m\otimes z_n$, in which the gradings $m$ and $n$ may vary. An equivalent relation to condition \ref{it:cr-axiom1} is that for each term $y_m\otimes z_n$ in the sum,
\begin{equation}
\tilde\lambda(n,l)=\tilde\lambda(l,m) \text{ for all } l\in H.
\label{eq:bichar-identity}
\end{equation}
This can be proved by using an equivalent form of axiom \ref{fig:axiom_form} given by rotating both diagrams in the expression by $\pi$. Then applying the maps on both sides of the equation to $a_l$ gives the identity
\begin{equation}\sum\tilde\lambda(l,m)\, y_m\otimes a_l\otimes z_n=\sum\tilde\lambda(n,l)\, y_m\otimes a_l\otimes z_n.\end{equation}

The curl on the right-hand side of axiom \ref{fig:axiom_left_right} is the map 
\begin{equation}
\label{eq:bichar-curl} 
a_l\mapsto\sum \varepsilon(a_l \cdot z_n)\tilde\lambda(l,m)\,y_m.
\end{equation}
However, from \eqref{eq:bichar-identity}, $\tilde\lambda(l,m)=\tilde\lambda(n,l)$ and for the non-zero terms in \eqref{eq:bichar-curl}, $\tilde\lambda(l,l)=\tilde\lambda(l,n)$. Together these imply $\tilde\lambda(l,m)=\tilde\lambda(l,l)$. Hence the curl is
\begin{equation}
\gls{phi}(a_l)=\tilde\lambda(l,l)\sum\varepsilon(z_n \cdot \sigma^{-1}(a_l))\,y_m=\tilde\lambda(l,l)\sigma^{-1}(a_l). \label{eq:graded_group_Naka}
\end{equation}
Since axiom \ref{fig:axiom_left_right} is equivalent to $\gls{phi}^2=\iden$, and \eqref{eq:graded_inverse} implies $\gls{bi}(l,l)^2=\gls{field_unit}$,  the axioms of definition \ref{def:spin-model} imply that $\gls{sigma}^2=\iden$. Conversely, $\sigma^2=\iden$ together with axioms \ref{fig:axiom_form} to \ref{fig:axiom_Reid_3} imply that axiom \ref{fig:axiom_left_right} is satisfied.
\end{proof}

It is possible to regard crossings from bicharacters as the natural generalisation of the canonical choice $a \otimes b \mapsto b \otimes a$. Although non-trivial, they retain many of the simple features of the canonical choice. In particular, they share with the planar models with canonical crossing of example \ref{ex:semi_simple_canonical} one key feature: partition functions $\gls{Z}(\gls{Sigma_g}, \gls{s}, A_1 \oplus A_2)$ can be naturally defined through the partition functions of their simple components, $Z(\Sigma_g,s,A_1)$ and $Z(\Sigma_g,s,A_2)$. To understand this we must first show how one can define a bicharacter for $A_1 \oplus A_2$ if the $A_1,A_2$ come equipped with one \cite{Bahturin}.

\begin{lemma} \label{lem:bicharacters_semi}
Let $\tilde{\lambda}_i \colon H_i \times H_i \to k$ be bicharacters for $H_i$-graded algebras $A_i$, $i=1,2$. Then $\gls{A}=A_1 \oplus A_2$ is $H$-graded with $\gls{H}=H_1 \times H_2$ and it comes naturally equipped with a bicharacter $\tilde{\lambda}$ defined as
\begin{align}
\tilde{\lambda}\left(g_1g_2,h_1h_2\right)\equiv\tilde{\lambda}_1(g_1,h_1)\tilde{\lambda}_2(g_2,h_2)
\label{eq:lambda_semi}
\end{align}
where $g_1g_2$ and $h_1h_2$ belong to $H_1 \times H_2$.
\end{lemma}
\begin{proof}
Let the gradings for the $A_i$ be displayed as $A_i = \bigoplus_{h \in H_i} (A_i)_h$. Then $A$ is naturally $H$-graded as 
\begin{align}\label{eq:grading}
A=\bigoplus_{h_1h_2\in H_1 \times H_2} (A_1)_{h_1} \oplus (A_2)_{h_2}
\end{align}
where the group $H_1 \times H_2$ comes with the product $h_1h_2 \cdot l_1l_2=(h_1\cdot l_1)(h_2 \cdot l_2)$. Using this definition of group product we know $\tilde{\lambda}(g_1g_2 \cdot h_1h_2,l_1l_2)=\tilde{\lambda}((g_1 \cdot h_1)(g_2 \cdot h_2),l_1l_2)$. Then expression \eqref{eq:lambda_semi} tells us $\tilde{\lambda}((g_1\cdot h_1)(g_2 \cdot h_2),l_1l_2))=\tilde{\lambda}_1(g_1 \cdot h_1,l_1)\tilde{\lambda}_2(g_2 \cdot h_2,l_2)$. Using property \eqref{eq:graded_mult} for the $\tilde{\lambda}_i$ and rearranging terms we find $\tilde{\lambda}_1(g_1 \cdot h_1,l_1)\tilde{\lambda}_2(g_2 \cdot h_2,l_2)=\tilde{\lambda}_1(g_1,l_1)\tilde{\lambda}_2(g_2,l_2)\tilde{\lambda}_1(h_1,l_1)\tilde{\lambda}_2(h_2,l_2)$. Again using the definition \eqref{eq:lambda_semi} we can finally rewrite 
\begin{align}
\tilde{\lambda}(g_1g_2 \cdot h_1h_2,l_1l_2)=\tilde{\lambda}(g_1g_2,l_1l_2)\tilde{\lambda}(h_1h_2,l_1l_2).
\end{align}
Similarly to establish property \eqref{eq:graded_inverse} we use the corresponding properties for the $\tilde{\lambda}_i$.
\begin{align}
\gls{bi}(g_1g_2,h_1h_2)
&=\tilde{\lambda}_1(g_1,h_1)\tilde{\lambda}_2(g_2,h_2)
=(\tilde{\lambda}_1(h_1,g_1)\tilde{\lambda}_2(h_2,g_2))^{-1} \notag \\
&=\tilde{\lambda}(h_1h_2,g_1g_2)^{-1}
\end{align}
\end{proof}

Moreover, if the $\lambda_i$ also satisfy the conditions of proposition \ref{lem:graded} then also will $\gls{lambda}$ as defined in \eqref{eq:lambda_semi}. Note that since the Frobenius form $\gls{eps}$ for $\gls{A}$ decomposes as $\varepsilon=\varepsilon_1 \oplus \varepsilon_2$ we know that $A_1 \perp A_2$. Therefore determining property \ref{it:cr-axiom1} for $\lambda$ is equivalent to determining it for each $\lambda_i$ which is true by construction. It is also an easy exercise to verify that the Nakayama automorphism associated with $\varepsilon$, $\gls{sigma}$, decomposes as $\sigma = \sigma_1 \oplus \sigma_2$. Therefore, since $\sigma_i^2=\iden$ we know $\sigma^2=\iden$. 

We will finally show that if the preferred elements of $A_i$ are $\eta_i$ and $\chi_i$ then the preferred elements for $A_1 \oplus A_2$ are given by $\gls{eta}=\eta_1 \oplus \eta_2$ and $\gls{chi}=\chi_1 \oplus \chi_2$. Let us use the expression $B_i=\sum y_{m_i} \otimes z_{n_i}$ to denote the $A_i$ bilinear maps, subject to condition \eqref{eq:bichar-identity}. Then, using the grading of equation \eqref{eq:grading}, we conclude 
\begin{align}
\gls{Bb}= \sum \left( y_{m_1} \oplus  y_{m_2} \right) \otimes \left( z_{n_1} \oplus  z_{n_2} \right).
\end{align}
The algebraic expression for $\eta$ will then read:
\begin{align}
\label{eq:eta_graded}
\eta = \sum \tilde{\lambda}(n_1n_2,h_1h_2) \left( y_{m_1} \oplus  y_{m_2} \right) \cdot \left( y_{h_1} \oplus  y_{h_2} \right) \cdot \left( z_{n_1} \oplus  z_{n_2} \right) \cdot \left( z_{l_1} \oplus  z_{l_2} \right).
\end{align}
Using the multiplication property $(a \oplus b) \cdot (c \oplus d)= a\cdot c \oplus b\cdot d$ along with the defining equation for $\tilde{\lambda}$, expression \eqref{eq:lambda_semi}, we learn \eqref{eq:eta_graded} can be rewritten as
\begin{align}
\eta &= \sum \tilde{\lambda}_1(n_1,h_1)\, y_{m_1} \cdot y_{h_1} \cdot z_{n_1} \cdot z_{l_1} \oplus \sum  \tilde{\lambda}_2(n_2,h_2)\, y_{m_2} \cdot y_{h_2} \cdot z_{n_2} \cdot z_{l_2} \notag \\
&= \eta_1 \oplus \eta_2.
\end{align}
The expression for $\chi$ is entirely similar apart from the action of $\varphi$ on $y_{m_1} \oplus  y_{m_2}$ and $z_{l_1} \oplus  z_{l_2}$. Since $\varphi$ is an automorphism it will also preserve the algebra decomposition. Therefore the conclusion $\chi = \chi_1 \oplus \chi_2$ also holds. 

This is enough to guarantee the partition function $\gls{Z}(\gls{Sigma_g}, \gls{s}, A_1 \oplus A_2)$ decomposes as a sum of the partition functions $Z(\Sigma_g, s, A_i)$. To see this note the general expression for the partition function, equation \eqref{eq:Z_general}, can be rewritten for $A_1 \oplus A_2$ as $\gls{Z}(\gls{Sigma_g},\gls{s})=\gls{R}\,\gls{eps} ((\eta_1 \oplus \eta_2)^{g-l} \cdot (\chi_1 \oplus \chi_2)^l)$. Again using the product rule $(a \oplus b) \cdot (c \oplus d)= a\cdot c \oplus b\cdot d$, the linearity of $\varepsilon$ and its decomposition $\varepsilon = \varepsilon_1 \oplus \varepsilon_2$ we can further rewrite 
\begin{align}
Z(\Sigma_g,s)=R\,\varepsilon_1(\eta_1^{g-l}\cdot \chi_1^l) + R\,\varepsilon_2(\eta_2^{g-l} \cdot \chi_2^l)
\end{align}
which establishes the intended result.

To make our knowledge of bicharacters more concrete we will use the fundamental theorem of finite abelian groups. According to this classification, any finite abelian group $\gls{H}$ is isomorphic to $\mathbb{Z}_{n_1} \times \cdots \times \mathbb{Z}_{n_p}$ for some natural numbers $n_i$ \cite{Scott} where $\Zb_n$ is the cyclic group of order $n$. Without loss of generality we will consider abelian groups of this type only, from now on. This will allow us to determine more precisely what the constants $\gls{bi}(h,l)$ look like.

\begin{proposition} \label{prop:graded_roots}
Suppose $A$ is an $H$-graded algebra with $H=\Zb_{n_1} \times \cdots \times \Zb_{n_p}$ where $\Zb_{n_k}$ is generated by $r_k$. Let $r_1^{\alpha_1}\cdots r_p^{\alpha_p}$, $\alpha_k \in \mathbb{N} \mod n_k$ denote an element of $H$. Then, a bicharacter $\tilde{\lambda} \colon H \times H \to \gls{k}$ must satisfy the following conditions. 
\begin{enumerate}[label=E\arabic{*}), ref=(E\arabic{*})]
\item \label{it:lambda_2} $\tilde{\lambda} \left(r_1^{\alpha_1}\cdots r_p^{\alpha_p}, r_1^{\beta_1} \cdots r_p^{\beta_p}\right)= \prod\limits_{i,j=1}^{p} \tilde{\lambda}(r_i,r_j)^{\alpha_i\beta_j}$
\item \label{it:lambda_3} $\tilde{\lambda}(r_i,r_j)^{n_i}=\tilde{\lambda}(r_i,r_j)^{n_j}=\gls{field_unit}$
\item \label{it:lambda_1} $\tilde{\lambda}(r_i,r_j)\tilde{\lambda}(r_j,r_i)=1_k$
\end{enumerate}
\end{proposition}
\begin{proof}
Identity \ref{it:lambda_2} will follow from equation \eqref{eq:graded_mult}. Using this property we first rewrite 
\begin{align}
\tilde{\lambda} (r_1^{\alpha_1}\cdots r_p^{\alpha_p}, r_1^{\beta_1} \cdots r_p^{\beta_p})
=\tilde{\lambda} (r_1, r_1^{\beta_1} \cdots r_p^{\beta_p})\tilde{\lambda} (r_1^{\alpha_1-1}r_2^{\alpha_2}\cdots r_p^{\alpha_p}, r_1^{\beta_1} \cdots r_p^{\beta_p})
\end{align}
which is valid for every $\alpha_i\in \mathbb{N} \mod n_i$. Iterating this process a total of $\alpha_i$ times we find 
\begin{align}
\tilde{\lambda} (r_1^{\alpha_1}\cdots r_p^{\alpha_p}, r_1^{\beta_1} \cdots r_p^{\beta_p})
=\tilde{\lambda} (r_1, r_1^{\beta_1} \cdots r_p^{\beta_p})^{\alpha_1}\tilde{\lambda} (r_2^{\alpha_2}\cdots r_p^{\alpha_p}, r_1^{\beta_1} \cdots r_p^{\beta_p}).
\end{align}
Repeating the process for each remaining $r_k$ will then deliver identity \ref{it:lambda_2}.

The second property is a consequence of \ref{it:lambda_2}. Consider the elements $r_i^{n_i}$ and $r_j^{\beta_j}$. Then identity \ref{it:lambda_2} tells us $\tilde{\lambda}(r_i^{n_i},r_j^{\beta_j})=\tilde{\lambda}(r_i,r_j)^{n_i\beta_j}$.
However, since $r_i^{n_i}=1$ and $\gls{bi}(1,h)=\gls{field_unit}$ for all $h \in \gls{H}$ we must have $(\tilde{\lambda}(r_i,r_j)^{n_i})^{\beta_j}=1_k$ for all $\beta_j$. This means the identity $\tilde{\lambda}(r_i,r_j)^{n_i}=1_k$ must hold. An entirely analogous treatment would allows us to conclude $\tilde{\lambda}(r_i,r_j)^{n_j}=1_k$.

Finally, the last statement is trivial: it follows directly from identity \eqref{eq:graded_inverse}. 
\end{proof}

The first non-trivial consequence of proposition \ref{prop:graded_roots} is understanding what condition \eqref{eq:bichar-identity}, $\tilde{\lambda}(n,l)=\tilde{\lambda}(l,m)$ for all $l \in H$, implies. It is equivalent to having either $\tilde{\lambda}(n,l)=1_k$ for all $n,l \in H$ or $n=m^{-1}$. To see this note that equation \eqref{eq:bichar-identity} can be rewritten as 
\begin{align}
\tilde{\lambda}(n,l)\tilde{\lambda}(m,l)=\tilde{\lambda}(r_1^{\alpha_1} \cdots r_p^{\alpha_p},r_1^{\beta_1} \cdots r_p^{\beta_p})\tilde{\lambda}(r_1^{\gamma_1} \cdots r_p^{\gamma_p},r_1^{\beta_1} \cdots r_p^{\beta_p})=1_k.
\end{align}
Property \ref{it:lambda_2} then tells us this equation is equivalent to $\prod_{i,j=1}^p (\tilde{\lambda}(r_i,r_j)^{\alpha_i + \gamma_i})^{\beta_j}=1_k$. Since this must be true for all $\beta_j$ we conclude that either $\tilde{\lambda}(r_i,r_j)=1_k$ for all $i,j=1,p$ or $\alpha_i + \gamma_i =0$. In other words either the bicharacter is trivial or $n=m^{-1}$.

The equivalence we have just discussed also allows us to conclude that $\gls{phi}$ preserves the algebra grading when $\tilde{\lambda}$ is not trivial. (When $\tilde{\lambda}(h,l)=1_k$ for all $h,l \in H$ then $\varphi=\gls{sigma}$ and the grading is not relevant for constructing partition functions.) Note it implies that if $\gls{eps} (a_h \cdot b_l) \neq 0$ we must have $l=h^{-1}$ . However, the definition of the Nakayama automorphism tells us that $\varepsilon (a_h \cdot b_{h^{-1}})=\varepsilon (b_{h^{-1}} \cdot \sigma(a_h))$. Since we assumed the left hand-side does not vanish we must have $\sigma(a_{h}) \in A_h$. On the other hand, $\varepsilon (a_h \cdot b_l) = 0$ cannot be true for all $a_h, h \in H$ since $\varepsilon$ is non-degenerate. Hence, $\sigma$ preserves the algebra grading. Since $\varphi(a_h)=\tilde{\lambda}(h,h)\sigma(a_h)$ we also conclude $\varphi (A_h) \subset A_h$. 

Since we now know how we can construct partition functions for semi-simple algebras from their simple blocks we will be mainly concerned with generating examples for $\gls{A}=\Mb_{n}(D)$ with $D=\Cb, \Rb, \Cb_{\Rb}, \Hb_{\Rb}$. As we understand most of the properties of crossings from bicharacters already, our freedom to create invariants relies on our choice of grading.    

\begin{example}[Algebras $A=\Mb_n(k)$, $k=\Rb,\Cb$] \label{ex:non_symmetric_Frobenius}
\noindent Let $\varepsilon(a)=\Tr(\gls{x}a)$ with $x^2=\gls{R}\Tr(x) \gls{unit}$ and $\Tr(x) \neq 0$. The algebra $A=\Mb_n(k)$ can be seen as having a natural $\Zb_2$-grading $A_{r^0} \oplus A_{r^1}$ where each matrix splits into a block-diagonal and a block-anti-diagonal part ($p+q=n$).
\begin{align}
\left(\begin{array}{cc}
a_{p\times p} & b_{p \times q}\\ c_{q \times p} & d_{q \times q}
\end{array}\right)
=
\left(\begin{array}{cc}
a_{p\times p} & 0\\ 0 & d_{q \times q}
\end{array}\right)
\oplus
\left(\begin{array}{cc}
0 & b_{p \times q}\\ c_{q \times p} & 0
\end{array}\right)
\in A_{r^0} \oplus A_{r^1}.
\end{align}
According to proposition \ref{prop:graded_roots} the components of $\gls{bi}$ are determined by the relation $\tilde{\lambda}(r^{\alpha},r^{\beta})=\tilde{\lambda}(r,r)^{\alpha,\beta}$ with $\alpha,\beta =0,1$ and $\tilde{\lambda}(r,r)=\pm \gls{field_unit}$. 

If we choose $\tilde{\lambda}(r,r)=+1_k$ we are in the conditions of example \ref{ex:semi_simple_canonical}. Therefore we will concentrate on the case $\tilde{\lambda}(r,r)=-1_k$. Condition \ref{it:cr-axiom1} tells us the identity $\varepsilon(a_{\alpha}b_{\beta})=0$ must hold whenever $\alpha \neq \beta$. This means we must have $\gls{eps} (a_0a_1)=0$ for all $a_i \in A_{r^{i}}$. However, we know $a_0a_1 \in A_1$ so it is just as valid to show $\varepsilon(a_1)=0$. Since $\varepsilon(a_1)=\Tr(\gls{x}a_1)$ we learn that $x$ must belong to $A_{r^0}$. 

The bilinear form can be written as
\begin{align}
\gls{Bb} = \frac{1}{\gls{R}\Tr(x)}\sum_{\alpha}\sum_{lm}e_{lm}^{\alpha}x \otimes e_{ml}^{\alpha} 
\end{align}
where the label $\alpha$ identifies whether the elementary matrix belongs to $A_{r^0}$ or $A_{r^1}$. The expressions for both $\gls{eta}$ and $\gls{chi}$ read
\begin{align}
\eta&=\frac{1}{R^2\Tr(x)^2} \sum_{\alpha \beta} \tilde{\lambda}(r,r)^{\alpha\beta}\sum_{lmso}e_{lm}^{\alpha}xe_{so}^{\beta}e_{ml}^{\alpha}xe_{os}^{\beta} \label{eq:eta_Z2_graded},\\
\chi&=\frac{1}{R^2\Tr(x)^2} \sum_{\alpha \beta} \tilde{\lambda}(r,r)^{\alpha\beta + \alpha + \beta}\sum_{lmso}xe_{lm}^{\alpha}e_{so}^{\beta}e_{ml}^{\alpha}e_{os}^{\beta}x. \label{eq:chi_Z2_graded}
\end{align}
We can rewrite the sum terms in \eqref{eq:eta_Z2_graded} as
\begin{align} \label{eq:eta_sum_separate}
\sum_{\alpha} \sum_{lmso}e_{lm}^{\alpha}xe_{so}^{0}e_{ml}^{\alpha}xe_{so}^{0} +
\sum_{\alpha} \tilde{\lambda}(r,r)^{\alpha}\sum_{lmso}e_{lm}^{\alpha}xe_{so}^{1}e_{ml}^{\alpha}xe_{so}^{1}.
\end{align}
The simplification $\sum_{lm}e_{lm}xe_{so}^{0}e_{ml}=\Tr(xe_{so}^{0})\gls{unit}$ holds. Since $x \in A_{r^0}$ we can rewrite the first sum term in equation \eqref{eq:eta_sum_separate} as $\sum_{so} \Tr(xe_{so}^{0}) xe_{so}^{0}=x^2=R\Tr(x)1$. On the other hand, $\sum_{so}e^1_{so}ye^1_{os}=0$ holds for all $y \in A$. Therefore the second term in equation \eqref{eq:eta_sum_separate} vanishes. This means that we have $\eta=\frac{1}{R\Tr(x)}$.

Let us now similarly decompose the sum terms in \eqref{eq:chi_Z2_graded}.
\begin{align} \label{eq:eta_sum_separate}
\sum_{\alpha} \gls{bi}(r,r)^{\alpha}\sum_{lmso}\gls{x}e_{lm}^{\alpha}e_{so}^{0}e_{ml}^{\alpha}e_{os}^{0}x - 
\sum_{\alpha} \sum_{lmso}xe_{lm}^{\alpha}e_{so}^{1}e_{ml}^{\alpha}e_{os}^{1}x
\end{align}
The property $\sum_{so}e^1_{so}ye^1_{os}=0$ for all $y \in \gls{A}$ again renders the second term in the sum zero. It also means that only $\alpha=0$ contributes to the first sum. One can finally easily verify that $\sum_{lmso}e_{lm}^{0}e_{so}^{0}e_{ml}^{0}e_{os}^{0}=\gls{unit}$. Therefore, $\gls{chi}=\frac{1}{R\Tr(x)}=\gls{eta}$ which means this choice of crossing also does not allow for spin structures to be distinguished. The partition function reads
\begin{align}\label{eq:Z2_partition}
\gls{Z}(\gls{Sigma_g})=\gls{R}^{1-g}\Tr(x)^{1-g},
\end{align}
which coincides with the canonical crossing one.
\end{example}

Our conclusions so far do not guarantee the existence of crossing maps satisfying $\eta \neq \chi$. The last efforts in this section therefore concentrate on presenting various examples of such algebras.

\begin{example}[Algebras $A=\Mb_n(\Cb_{\Rb})$] \label{ex:complex_crossing} These algebras are naturally $\Zb_2$-graded: $\Mb_n(\Cb_{\Rb}) =A_{r^0} \bigoplus A_{r^1}$ with $A_{r^0}=\Mb_{n}(\Rb)$ and $A_{r^1}=\hat{\imath}\Mb_{n}(\Rb)$. The components of the crossing are then determined by the relation $\tilde{\lambda}(r^\alpha,r^{\beta})=\tilde{\lambda}(r,r)^{\alpha\beta}$ for $\tilde{\lambda}(r,r)=-1_k$. For convenience we will also use the alternative notation $r=\hat{\imath}$.

Let us take the Frobenius form to be as general as possible: $\gls{eps}(a)=\Real\Tr(xa)$ with $\gls{x}^2=2\gls{R}\Tr(x)\gls{unit}$ and $\Tr(x) \neq 0$. Then the analysis of the previous example tells us that $x \in A_{r^0}$.  

Then the bilinear form reads as
\begin{align}
\gls{Bb}=\frac{1}{2R\Tr(x)}\sum_{w} \sum_{lm} w e_{lm}x \otimes w ^{\ast} e_{ml}. 
\end{align}
This information can be used to determine the relation 
\begin{align}\label{eq:eta_com}
\eta=\frac{1}{4R^2\Tr(x)^2}\sum_{wt} \gls{bi}(w,t) \sum_{lmpq} e_{lm} x e_{pq} x e_{ml} e_{qp}.
\end{align}
Recalling that $\sum_{wt} \tilde{\lambda}(w,t)=\sum_{\alpha,\beta}\tilde{\lambda}(r,r)^{\alpha\beta}$ and that $\tilde{\lambda}(r,r)=-\gls{field_unit}$ we can identify the first sum as the constant $2$. The second sum is just the element $\gls{x}^2=2\gls{R}\Tr(x)\gls{unit}$, so one concludes that $\gls{eta}=\frac{1}{2R\Tr(x)}$. The element $\gls{chi}$ is constructed in an analogous fashion: 
\begin{align}\label{eq:chi_com}
\eta=\frac{1}{4R^2\Tr(x)^2}\sum_{wt} \tilde{\lambda}(w,w)\tilde{\lambda}(w,t)\tilde{\lambda}(t,t) \sum_{lmpq} xe_{lm} e_{pq} e_{ml} e_{qp}x.
\end{align}

As before, the second sum term gives rise to the term $x^2=2R\Tr(x)1$. However, the first sum is now equal to $\sum_{\alpha\beta}(-1_k)^{\alpha\beta + \alpha + \beta}=-2$. Hence, $\chi=-\eta$. The invariant produced distinguishes spin structures of different parity and takes the form
\begin{align}\label{eq:first_spin_inv}
\gls{Z}(\gls{Sigma_g},\gls{s})=P(s)2^{-g}R^{1-g}\Tr(x)^{1-g}.
\end{align}
\end{example}

\begin{example}[Algebras $\gls{A}=\Mb_n(\Hb_{\Rb})$] \label{ex:quaternions_crossing} \hspace{1mm}\\
\noindent Consider the group of the quaternions $\widehat{\mathcal{K}}=\left\lbrace 1,\hat{\imath},\hat{\jmath},\hat{k}, -1,-\hat{\imath},-\hat{\jmath},-\hat{k}\right\rbrace  \ni w$ and define $A_w=w\Mb_{n}(\Rb)$. Then $A_w=A_{-w}$ and $A_wA_t=A_{wt}$, so that the algebra  $A=\Mb_n(\Hb_{\Rb})$ is graded by the quotient group $\mathcal{K}=\widehat{\mathcal{K}}/\{\pm1\}$, which is isomorphic to the Klein group $\Zb_2^{(1)} \times \Zb_2^{(2)}$. The grading is conveniently written $\Mb_n(\Hb_{\Rb}) = \oplus_{w}A_w$ with $w\in\left\lbrace 1,\hat{\imath},\hat{\jmath},\hat{k}\right\rbrace$. 

Let us denote the generators of $\Zb_2^{(1)} \times \Zb_2^{(2)}$ as $r_1=\hat{\imath}$ and $r_2=\hat{\jmath}$. Then proposition \ref{prop:graded_roots} guarantees the bicharacter will be determined by the terms $\tilde{\lambda}(\hat{\imath},\hat{\imath})$, $\tilde{\lambda}(\hat{\jmath},\hat{\jmath})$ and $\tilde{\lambda}(\hat{\imath},\hat{\jmath})$ all of which are roots of unit. We exclude the case $\tilde{\lambda}(\hat{\imath},\hat{\imath})=\tilde{\lambda}(\hat{\jmath},\hat{\jmath})=\tilde{\lambda}(\hat{\imath},\hat{\jmath})=1_k$ as this gives rises to the canonical crossing.

Let us take the Frobenius form to be as general as possible: $\varepsilon(a)=\Real\Tr(xa)$ with $x^2=4R\Tr(x)1$ and $\Tr(x) \neq 0$. The condition $\varepsilon(a_{h}b_{l})=0$ whenever $l\neq h^{-1}$ translates into the identities $\gls{eps}(a_{\hat{\imath}})=\varepsilon(a_{\hat{\jmath}})=\varepsilon(a_{\hat{k}})=0$. This means the element $x \in A_{1}$.

The bilinear form satisfies
\begin{align}
\gls{Bb}=\frac{1}{4R\Tr(x)}\sum_{w,lm}\left(w \,e_{lm} x \otimes w^{\ast}\,e_{ml}\right).
\end{align}
This information can be used to determine the identity 
\begin{align}
\gls{eta}=\frac{1}{16\gls{R}^2\Tr(\gls{x})^2}\left(\sum_{lmpq}e_{lm}xe_{pq}xe_{ml}e_{qp}\right)\left(\sum_{wt}\gls{bi}(w,t)wtw^{\ast}t^{\ast}\right).
\end{align}
The first sum is simply the element $4R\Tr(x)\gls{unit}$. The second, with some algebraic manipulation, can be seen to satisfy
\begin{align}
\sum_{wt}\tilde{\lambda}(w,t)wtw^{\ast}t^{\ast}=\frac{11}{2}-2\Lambda+ \frac{\Lambda^2}{2}
\end{align}
with $\Lambda = \tilde{\lambda}(\hat{\imath},\hat{\imath}) + \tilde{\lambda}(\hat{\jmath},\hat{\jmath}) + \tilde{\lambda}(\hat{k},\hat{k})$. The element $\gls{chi}$ is constructed in an analogous fashion:
\begin{align}
\chi=\frac{1}{16R^2\Tr(x)^2}\sum_{lmpq}xe_{lm}e_{pq}e_{ml}e_{qp}x\sum_{wt}\tilde{\lambda}(w,w)\tilde{\lambda}(w,t)\tilde{\lambda}(t,t)wtw^{\ast}t^{\ast}.
\end{align}
and again the first sum equates to $4R\Tr(x)1$. We are then able to verify
\begin{equation}
\sum_{wt}\tilde{\lambda}(w,w)\tilde{\lambda}(w,t)\tilde{\lambda}(t,t)wtw^{\ast}t^{\ast}
=\sum_{wt}\tilde{\lambda}(w,t)wtw^{\ast}t^{\ast}-\left(\Lambda -1\right)\left(\Lambda -3\right)\left(\Lambda +3\right).
\end{equation}
It is easy to see $\Lambda \in \lbrace -3,-1,1 \rbrace$ (the canonical crossing would correspond to the choice $\Lambda=3$). Therefore only the crossings satisfying $\Lambda=-1$ distinguish spin structures. 
The partition function reads  
\begin{align} \label{eq:quaternion_spin}
\gls{Z}(\gls{Sigma_g},\gls{s})=f(\Lambda,s)R^{1-g}\Tr(x)^{1-g},
\hspace{5mm}
f(\Lambda,s)=
\begin{cases}
2^{2g} &(\Lambda=-3) \\ 
P(s)2^{g}&(\Lambda=-1) \\ 
1 &(\Lambda=+1) 
\end{cases}.
\end{align}
\end{example}

A natural question is whether algebras with symmetric Frobenius forms, and a crossing respecting the conditions of definition~\ref{def:spin-model} and the curl-free condition always give rise to an FHK state sum model. This is not, however, the case as it can be seen from the example below.

\begin{example}[Algebras $\gls{A}=\Mb_n(\Cb)$] \label{ex:G_general}
As studied in \cite{Bahturin} $\Mb_n(\Cb)$ can be regarded as an $\gls{H}$-graded algebra where the group is $H = \Zb_n^{(1)} \times \Zb_n^{(2)}$ and we denote the generators as $r_1=a$ and $r_2=b$. $\Mb_{n}(\Cb)$ is decomposed into $n^2$ components as follows. Let $\xi \in \Cb$ be a  primitive $n$-th root of unit and define the matrices $X_a=\text{diag}(\xi^{n-1},\cdots,\xi,1)$ and $Y_b=e_{n1}+\sum_{m=1}^{n-1}e_{m(m+1)}$ \cite{Bahturin}. Then, the $X_a^{\alpha}Y_b^{\beta}$ are linearly independent. Let us denote the one-dimensional subspace generated by each $X_a^{\alpha}Y_b^{\beta}$ as $A_h$, $h=a^{\alpha}b^{\beta}$. The fact $X_a^iY_b^jX_a^{i'}Y_b^{j'}=\xi^{-ji'}X_a^{i+i'}Y_b^{j+j'}$ implies the decomposition $\oplus_{h \in H} A_h$ is indeed a group grading.

The expression
\begin{align}
\gls{bi}(a^ib^j,a^{i'}b^{j'})=\tilde{\lambda}(a,a)^{ii'}\tilde{\lambda}(b,b)^{jj'}\tilde{\lambda}(a,b)^{ij'-ji'}
\end{align}
is the most general we can have for a grading. The constants $\tilde{\lambda}(a,a)$, $\tilde{\lambda}(b,b)$ and $\tilde{\lambda}(a,b)$ are $n$-roots of unit. In addition we must have $\tilde{\lambda}(a,a)^2=\tilde{\lambda}(b,b)^2=\gls{field_unit}$. Therefore, if $n$ is odd we must have $\tilde{\lambda}(a,a)=\tilde{\lambda}(b,b)=1_k$.
 
It is necessary to verify condition \ref{it:cr-axiom1}, which is to say that $\gls{eps} (a_h)=0$ for all $h \neq 1$. This means the Frobenius form $\varepsilon(a)=\Tr(xa)$ must be such that $\gls{x} \in A_1$. In other words, the Frobenius form is required to be symmetric: $x=\gls{R}\,n.\gls{unit}$. To show that $\varepsilon (a_h)=0$ is satisfied under these conditions, note that $\varepsilon (Y_b^jX_a^i)=\varepsilon (X_a^iY_b^j)=\xi^{ij}\varepsilon (Y_b^jX_a^i)$. Since $\xi$ is a primitive root of unity $\varepsilon(a_h)=0$ is verified whenever $i,j$ are both non-zero. On the other hand, we have $X_a^iY_b^jX_a^{-1}=\xi^{ij}Y_b^j$. This implies $\varepsilon(X_a^i)=\varepsilon(Y_b^j)=0$ whenever $i,j \neq 0$. 

The bilinear form reads
\begin{align}
\gls{Bb}= \frac{1}{Rn^2}\sum_{ij} \xi^{-ij} X_a^iY_b^j \otimes X_a^{-i}Y_b^{-j}
\end{align}
and some algebraic manipulations allow us to conclude:
\begin{align}
\gls{eta} &= \frac{1}{R^2n^4}\sum_{ijkl} \left(\xi^{-1} \tilde{\lambda}(a,b)\right)^{il-jk} \tilde{\lambda}(a,a)^{ik}\tilde{\lambda}(b,b)^{jl}, \\
\gls{chi} &= \frac{1}{R^2n^4}\sum_{ijkl} \left(\xi^{-1} \tilde{\lambda}(a,b)\right)^{il-jk} \tilde{\lambda}(a,a)^{i+ik+k}\tilde{\lambda}(b,b)^{j+jl+l}.
\end{align}
We can see the constant $\tilde{\lambda}(a,b)$ does not play a role in determining whether the model distinguishes spin structures or not. Therefore, we will set $\tilde{\lambda}(a,b)=\xi$ for simplicity. Noting that $\sum_{ij=1}^{2p} (-1)^{ij}=2p$ and $\sum_{ij=1}^{2p} (-1)^{i+ij+j}=-2p$, and defining $\Lambda = \gls{bi}(a,a)+ \tilde{\lambda}(b,b)$ we conclude 
\begin{align}
\gls{Z}(\gls{Sigma_g},\gls{s})=f(\Lambda,s)R^{2-2g}n^2,
\hspace{5mm}
f(\Lambda,s)=
\begin{cases}
n^{-2g}&\left(\Lambda=-2\right)  \\
P(s)n^{-g}&\left(\Lambda=0\right)  \\
1&\left(\Lambda=+2\right)  
\end{cases}.
\end{align}
\end{example}

The techniques introduced thus far allowed us to create a rich class of examples. In particular, we now know that any algebra $\gls{A}$ with a simple component isomorphic to $\Mb_{2n}(\Rb)$, $\Mb_{n}(\Cb_{\Rb})$, $\Mb_{n}(\Hb_{\Rb})$ or $\Mb_{2n}(\Cb)$ can be equipped with a crossing giving rise to a non-topological spin model. However, there are some algebras for which we can never find a non-trivial crossing by constructing one from the simple components of the algebra. These are commutative algebras $A=\bigoplus_{i=1}^N A_i$ with a trivial centre, and are isomorphic to a number of copies over either $\Cb$ or $\Rb$. Each simple component $A_i$ is one-dimensional which means the space $A_i \otimes A_i$ is also one-dimensional. Consequently, $\lambda_i \colon A_i \otimes A_i \to A_i \otimes A_i$ has to be the canonical crossing. The commutative complex algebras are also isomorphic to $\Cb H$  for some abelian group $\gls{H}$ with $|H|=N$. We can explore this isomorphism to generate crossings from bicharacters. 

\begin{example}\label{ex:H_awesome}
The algebra $\Cb H$ for abelian $H$ can naturally be seen as $H$-graded with each element $h \in H$ generating the linear subspace $A_h$. If we take the group elements as the algebra basis it is easy to conclude the Frobenius form $\gls{eps}(h)=\gls{R}|H|\,\delta_{h,1}$ and $\gls{Bb}=R^{-1}|H|^{-1}\sum_{h \in H} h \otimes h^{-1}$ (the Frobenius form must be symmetric because the algebra is commutative). If we choose the crossing to be given by equation \eqref{eq:bi_def} the preferred elements of the algebra satisfy 
\begin{align}
\gls{eta}&=\frac{\gls{unit}}{R^2|H|^2}\sum_{h,l}\tilde{\lambda}(h,l), \label{eq:eta_com_group}\\
\gls{chi}&=\frac{1}{R^2|H|^2}\sum_{h,l}\tilde{\lambda}(h,l)\tilde{\lambda}(h,h)\tilde{\lambda}(l,l) \label{eq:chi_com_group}.
\end{align}  
Using the classification of bicharacters of proposition~\ref{prop:graded_roots} where $H=\Zb_{n_1} \times \cdots \times \Zb_{n_p}$, $|H|=n_1\cdots n_p$ we can rewrite  the identities above as
\begin{align}
\gls{eta}&=\frac{\gls{unit}}{\gls{R}^2|H|^2}\prod_{ij=1}^p\sum_{\alpha_i=1}^{n_i}\sum_{\beta_j=1}^{n_j}\tilde{\lambda}(r_i,r_j)^{\alpha_i\beta_j},\\
\gls{chi}&=\frac{1}{R^2|H|^2}\prod_{ij=1}^p\sum_{\alpha_i=1}^{n_i}\sum_{\beta_j=1}^{n_j}\tilde{\lambda}(r_i,r_j)^{\alpha_i\beta_j}\tilde{\lambda}(r_i,r_i)^{\alpha_i}\tilde{\lambda}(r_j,r_j)^{\beta_j}.
\end{align}  
We can see the factors $\tilde{\lambda}(r_i,r_j)$ for $i \neq j$ do not contribute for $\eta$ and $\chi$ to be distinct. Therefore, we set $\tilde{\lambda}(r_i,r_j)=1_k$ whenever $i \neq j$. The expressions for $\eta$ and $\chi$ are simplified to
\begin{align}
\eta&=\frac{1}{R^2|H|^2}\prod_{i=1}^p\sum_{\alpha_i,\beta_i=1}^{n_i}\tilde{\lambda}(r_i,r_i)^{\alpha_i\beta_i} = \frac{1}{R^2|H|^2}\prod_{i=1}^p \eta_i,\\
\chi&=\frac{1}{R^2|H|^2}\prod_{i=1}^p\sum_{\alpha_i,\beta_i=1}^{n_i}\tilde{\lambda}(r_i,r_i)^{\alpha_i\beta_i+ \alpha_i +\beta_i}= \frac{1}{R^2|H|^2}\prod_{i=1}^p \chi_i.
\end{align}  
Each term in the sum over $i$ can be independently calculated. If $n_i$ is odd then $\tilde{\lambda}(r_i,r_i)=1_k$ and $\eta_i=\chi_i=n^2_i$. If $n_i$ is even it is possible to choose $\tilde{\lambda}(r_i,r_i)=-1_k$. For this choice, as discussed in example~\ref{ex:G_general}, we have $\eta_i=-\chi_i=n_i$. Let $I$ be the set of indices $i$ for which $\tilde{\lambda}(r_i,r_i)=-1_k$ and denote $\prod_{i \not\in I} n_i=N_{I}$. Then,
\begin{align}
\eta=1.\frac{N_I}{R^2|H|}, \hspace{5mm}\chi=(-1_k)^{|I|}\eta
\end{align}
where $|I|$ denotes the cardinality of $I$. The partition function reads
\begin{align} \label{eq:Z_p}
\gls{Z}(\gls{Sigma_g},\gls{s})=P(s)^{|I|}R^{2-2g}N_I^g|H|^{1-g}.
\end{align}
When $I=\empty$ and $p=0,1$ the equation above should be compared with expression \eqref{eq:group_inv}: it is easy to verify \eqref{eq:Z_p} does reduce to $R^{2-2g}|H|$ since $N_I=|H|$ under the quoted circumstances. 

The construction we have just described cannot be used to differentiate spin structures if the algebra is odd-dimensional. However, for $|\gls{H}|$ odd we can always regard $\Cb H \simeq \Cb \oplus \Cb N$ as a $\Zb_1 \times N$ algebra ($|H|=1+|N|$ ). Then we know
\begin{align}
Z(\Sigma_g,s,\Cb H)=R^{2-2g}+Z(\Sigma_g,s,\Cb N).
\end{align} 
Of course, this process can be generalised to include the gradings generated by all abelian  $M \times N$ such that $|M|+|N|=|H|$. Currently, we conjecture that all distinct partition functions for algebras $\Cb \gls{H}$ can be created from crossings generated in this way. The first few iterations of this conjecture (for $|H|=2,3,4$) have been established using the program in appendix~\ref{app:code}.
\end{example}

We finish this section by presenting crossings that can be used for group algebras $\Cb H$ when $H$ is not abelian.

\begin{example} \label{ex:non_abelian}
Let $\Cb H$ be a group algebra of dimension $|H|$. Let the Frobenius form be chosen to be symmetric for simplicity: $\gls{eps}(h)=\gls{R}|H|\delta_{h,1}$. Remember that a bicharater $\gls{bi} \colon H \times H \to \Cb$ must satisfy $\tilde{\lambda}(g,h)=\tilde{\lambda}(lgl^{-1},h)$. Therefore, $\tilde{\lambda} \colon H/\Inn(H) \times H/\Inn(H) \to \gls{k}$ , where $\Inn(H)$ denotes the group of innner $H$-automorphisms and satisfies $H/\Inn(H)=\mathcal{Z}(H)$. In other words, $\tilde{\lambda}(h,1)=\gls{field_unit}$ if $h \in \Inn(H)$.  

Note $\gls{Bb}=R^{-1}|H|^{-1}\sum_{h \in H} h \otimes h^{-1}$. The expression for $\eta$ can then be decomposed as follows, if we use the elements $h \in H$ as a basis. 
\begin{align}
R^2|H|^2\, \gls{eta}&= \sum_{h,l} hlh^{-1}l^{-1}\tilde{\lambda}(h,l)\notag \\
&=\sum_{h,l \in \mathcal{Z}(H)} \tilde{\lambda}(h,g)\gls{unit} + 2\sum_{\substack{h \in \mathcal{Z}(H)\\ l \in \Inn(H)}}1 + \sum_{h,l \in \Inn(H)}lhl^{-1}h^{-1}
\end{align}
On one hand, the first sum can be identified with the element $\eta$ created by a $\Cb\mathcal{Z}(H)$ algebra for some bicharacter $\tilde{\lambda}$ -- see equation \eqref{eq:eta_com_group}. On the other hand, the last sum can be identified with the preferred element created for a $\Cb\Inn(H)$ algebra when using the canonical crossing. These elements shall be denoted $\eta^{\tilde{\lambda}}_{\mathcal{Z}(H)}$ and $\eta_{\,\Inn(H)}$ respectively. Note as well the term $\sum_{h \in \mathcal{Z}(H),\, l \in \Inn(H)}1$ reduces to $|\mathcal{Z}(H)||\Inn(H)|1=|H|1$:
\begin{align}
\eta=\frac{|\mathcal{Z}(H)|^2}{|H|^2}\eta^{\lambda}_{\mathcal{Z}(H)} + \frac{2}{R^2|H|}1 + \frac{1}{|\mathcal{Z}(H)|^2}\eta_{\,\Inn(H)}.
\end{align}
The element $\gls{chi}$ decomposes as follows.
\begin{align}
R^2|H|^2\, \chi = \sum_{h,l \in \mathcal{Z}(H)} \tilde{\lambda}(h,g)\tilde{\lambda}(h,h)\tilde{\lambda}(g,g)1 + 2\sum_{\substack{h \in \mathcal{Z}(H)\\ l \in \Inn(H)}}\tilde{\lambda}(h,h)1 + \sum_{h,l \in \Inn(H)}lhl^{-1}h^{-1} 
\end{align}
To compute the sum $\sum_{h \in \mathcal{Z}(H), \,l \in \Inn(H)}\tilde{\lambda}(h,h)1=|\Inn(H)|\sum_{h \in \mathcal{Z}(H)}\tilde{\lambda}(h,h)1$ we use the decomposition of proposition~\ref{prop:graded_roots}:
\begin{align}
\sum_{h \in \mathcal{Z}(H)}\gls{bi}(h,h)\gls{unit}=\prod_{i=1}^p\sum_{\alpha_i=1}^{n_i}\tilde{\lambda}(r_i,r_i)^{\alpha_i}=|\mathcal{Z}(H)|\delta_{|I|,0}.
\end{align}
Note the last step uses the fact $\sum_{\alpha=1}^{2n}(-\gls{field_unit})^{\alpha}=0$. With a definition for $\chi^{\tilde{\lambda}}_{\mathcal{Z}(H)}$ similar to that of $\eta^{\tilde{\lambda}}_{\mathcal{Z}(H)}$ we find
\begin{align}
\gls{chi}= \frac{|\mathcal{Z}(H)|^2}{|H|^2}\chi^{\lambda}_{\mathcal{Z}(H)} + \frac{2\delta_{|I|,0}}{\gls{R}^2|H|}1 + \frac{1}{|\mathcal{Z}(H)|^2}\eta_{\,\Inn(H)}.
\end{align}
For simplicity, we will calculate partition functions for bicharacters satisfying $\tilde{\lambda}(r_i,r_j)=1_k$ whenever $i \neq j$ as in the previous example. It is easy to see that $\gls{eta}=\chi$ when $|I|=0$ (which is to say the bicharacter is trivial). Therefore, we will concentrate on the case $|I|\neq 0$. Recall that under these restrictions $\eta^{\tilde{\lambda}}_{\mathcal{Z}(H)}=R^{-2}N_I|\mathcal{Z}(H)|^{-1}1$ and $\chi^{\tilde{\lambda}}_{\mathcal{Z}(H)}=(-1)^{|I|}\eta^{\tilde{\lambda}}_{\mathcal{Z}(H)}$.
\begin{align}
\eta&=\frac{2|H|+N_I|\mathcal{Z}(H)|}{R^2|H|^2}1 + \frac{1}{|\mathcal{Z}(H)|^2}\eta_{\,\Inn(H)}\\
\chi&= \frac{(-1)^{|I|}N_I|\mathcal{Z}(H)|}{R^2|H|^2}1 + \frac{1}{|\mathcal{Z}(H)|^2}\eta_{\,\Inn(H)}
\end{align}
We start with case $\gls{s}=\text{even}$. We note that $\gls{eps}_{H}=|\mathcal{Z}(H)|\varepsilon_{\Inn(H)}$. Therefore, using the binomial decomposition of $\eta^g$ we can relate the partition function for $\Cb \gls{H}$ with the partition function for $\Cb \Inn(H)$. Moreover, since the model for $\Cb \Inn(H)$ is an FHK model we can use expression \eqref{eq:group_inv} for further simplification. The representations of $\Inn(H)$ have labels $j \in J$ and we denote $|\Inn(H)|^{-1}N_I=d_{H,I}$.
\begin{align}
\gls{Z}(\gls{Sigma_g},\text{even})= R^{2-2g}\mathcal{Z}(H)|^{1-g}\sum_{k=0}^g\sum_{j \in J}({}_{k}^{g})\left(d_{H,I}+2\right)^k(\dim j)^{2-2(g-k)}
\end{align}
Reordering the sums and performing the summation over $k$ the expression can be simplified to:
\begin{align}
Z(\Sigma_g,\text{even})= R^{2-2g}|H|^{-g}|\mathcal{Z}(H)|^{1-g}\sum_{j \in J}\left(d_{H,I}+2+(\dim j)^{-2}\right)^{g}(\dim j)^{2}.
\end{align}
A similar treatment for the case $s=\text{odd}$ allows us to write the partition function as
\begin{align}
\gls{Z}(\gls{Sigma_g},\gls{s})&=\gls{R}^{2-2g}|H|^{-g}|\mathcal{Z}(H)|^{1-g}\sum_{j \in J}\left(d_{H,I}+2+(\dim j)^{-2}\right)^{g-1}(\dim j)^{2}f(s,j), \notag \\
f(s,j) &= 
\begin{cases}
d_{H,I}+2+(\dim j)^{-2} & (s \text{ even})\\
P(s)^{|I|}d_{H,I}+(\dim j)^{-2} & (s \text{ odd})
\end{cases}.
\end{align}
\end{example}

%%%%%%%%%%%%%%%%%%%%%%%%%%%%%%%%%%%%%%%%%%%%%%%%%%%%%%%%%
%%%%%%%%%%%%%%%%%%%%%%%%%%%%%%%%%%%%%%%%%%%%%%%%%%%%%%%%%
\chapter{Planar and spherical defect models}\label{ch:defects}
%%%%%%%%%%%%%%%%%%%%%%%%%%%%%%%%%%%%%%%%%%%%%%%%%%%%%%%%%
%%%%%%%%%%%%%%%%%%%%%%%%%%%%%%%%%%%%%%%%%%%%%%%%%%%%%%%%%

%%%%%%%%%%%%%%%%%%%%%%%%%%%%%%%%%%%%%%%%%%%%%%%%%%%%%%%%%
%%%%%%%%%%%%%%%%%%%%%%%%%%%%%%%%%%%%%%%%%%%%%%%%%%%%%%%%%
\section{Planar models with defects}
%%%%%%%%%%%%%%%%%%%%%%%%%%%%%%%%%%%%%%%%%%%%%%%%%%%%%%%%%
%%%%%%%%%%%%%%%%%%%%%%%%%%%%%%%%%%%%%%%%%%%%%%%%%%%%%%%%%

%%%%%%%%%%%%%%%%%%%%%%%%%%%%%%%%%%%%%%%%%%%%%
\begin{figure}
     \begin{subfigure}[t!]{0.32\textwidth}
                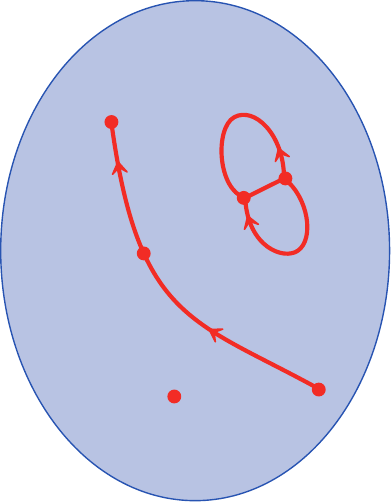
                \caption{}
                \label{fig:zoom1}
      \end{subfigure}
      \begin{subfigure}[t!]{0.32\textwidth}
                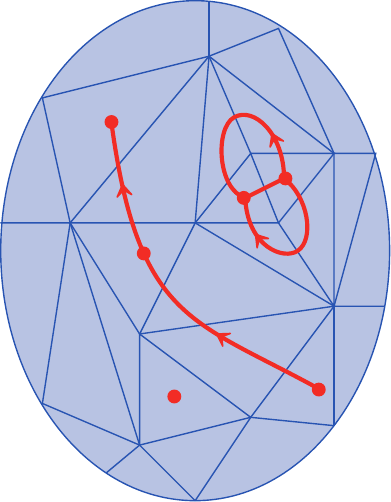
                \caption{}
                \label{fig:zoom2}
      \end{subfigure}
     \begin{subfigure}[t!]{0.32\textwidth}
                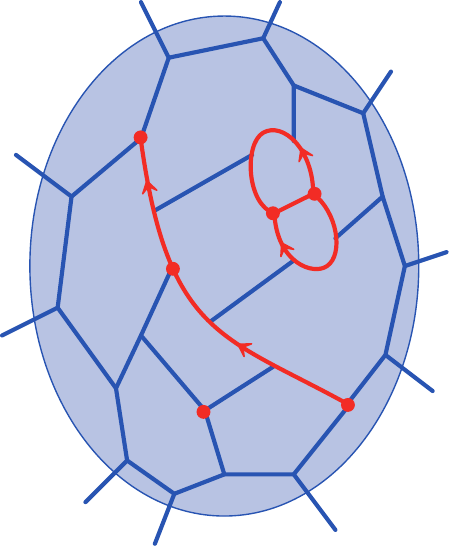
                \caption{}
                \label{fig:zoom3}
      \end{subfigure}
\caption[Defect graphs]{Graphs are embedded in the surface and referred to as defects. We introduce a triangulation respecting conditions \ref{it:defect1} to \ref{it:defect4}. To construct a diagram we regard the defects as elements of the dual diagram of the surface.}
\label{fig:graph}
\end{figure} 
%%%%%%%%%%%%%%%%%%%%%%%%%%%%%%%%%%%%%%%%%%%%%

The introduction of defects in FHK models has been studied by Davydov, Kong and Runkel in the context of topological quantum field theory \cite{Runkel}. In this section we will discern how their framework can be understood within the more general theory of planar state sum models. 

A diagrammatic state sum model is defined for a triangulated compact subset $M \subset \Rb^2$ (see section \S\ref{sec:planar}). Our new model is constructed for an extension of this type of space: a compact subset $M$ of $\Rb^2$ that comes equipped with a collection of embedded graphs. A graph $\gls{Gamma}$ is defined as a pair $(\Gamma_0,\Gamma_1)$ where $\Gamma_0$ is a finite set of labelled nodes $\mathfrak{i}$ and $\Gamma_1$ is a finite set of pairs of ordered nodes $(\mathfrak{i},\mathfrak{j})$ in $\Gamma_0$ called arrows. These come labelled as $\mathfrak{b}_{\mathfrak{i}\mathfrak{j}}$ where $\mathfrak{b}$ indexes the number of arrows $(\mathfrak{i},\mathfrak{j})$. 
$$
\hspace{40mm}
%% Creator: Inkscape 0.48.2, www.inkscape.org
%% PDF/EPS/PS + LaTeX output extension by Johan Engelen, 2010
%% Accompanies image file 'defect_index.pdf' (pdf, eps, ps)
%%
%% To include the image in your LaTeX document, write
%%   \input{<filename>.pdf_tex}
%%  instead of
%%   \includegraphics{<filename>.pdf}
%% To scale the image, write
%%   \def\svgwidth{<desired width>}
%%   \input{<filename>.pdf_tex}
%%  instead of
%%   \includegraphics[width=<desired width>]{<filename>.pdf}
%%
%% Images with a different path to the parent latex file can
%% be accessed with the `import' package (which may need to be
%% installed) using
%%   \usepackage{import}
%% in the preamble, and then including the image with
%%   \import{<path to file>}{<filename>.pdf_tex}
%% Alternatively, one can specify
%%   \graphicspath{{<path to file>/}}
%% 
%% For more information, please see info/svg-inkscape on CTAN:
%%   http://tug.ctan.org/tex-archive/info/svg-inkscape
%%
\begingroup%
  \makeatletter%
  \providecommand\color[2][]{%
    \errmessage{(Inkscape) Color is used for the text in Inkscape, but the package 'color.sty' is not loaded}%
    \renewcommand\color[2][]{}%
  }%
  \providecommand\transparent[1]{%
    \errmessage{(Inkscape) Transparency is used (non-zero) for the text in Inkscape, but the package 'transparent.sty' is not loaded}%
    \renewcommand\transparent[1]{}%
  }%
  \providecommand\rotatebox[2]{#2}%
  \ifx\svgwidth\undefined%
    \setlength{\unitlength}{171.45849609bp}%
    \ifx\svgscale\undefined%
      \relax%
    \else%
      \setlength{\unitlength}{\unitlength * \real{\svgscale}}%
    \fi%
  \else%
    \setlength{\unitlength}{\svgwidth}%
  \fi%
  \global\let\svgwidth\undefined%
  \global\let\svgscale\undefined%
  \makeatother%
  \begin{picture}(1,0.29763945)%
    \put(0,0){\includegraphics[width=\unitlength]{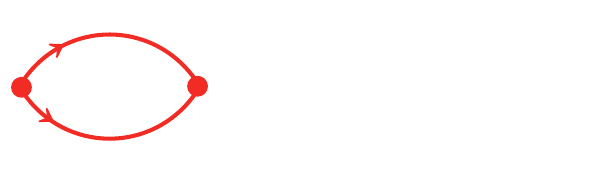}}%
    \put(0.16919125,0.2611647){\color[rgb]{0,0,0}\makebox(0,0)[lb]{\smash{$\mathfrak{1}_{\mathfrak{i}\mathfrak{j}}$}}}%
    \put(0.16536223,0.00981925){\color[rgb]{0,0,0}\makebox(0,0)[lb]{\smash{$\mathfrak{2}_{\mathfrak{i}\mathfrak{j}}$}}}%
    \put(-0.00166312,0.17632672){\color[rgb]{0,0,0}\makebox(0,0)[lb]{\smash{$\mathfrak{i}$}}}%
    \put(0.35064617,0.17124708){\color[rgb]{0,0,0}\makebox(0,0)[lb]{\smash{$\mathfrak{j}$}}}%
  \end{picture}%
\endgroup%

$$
We also restrict nodes to be at most trivalent: this means there are at most three $\mathfrak{b}_{\mathfrak{i}\mathfrak{j}}$ for each $\mathfrak{i}$. An example of this new type of space is illustrated in figure~\ref{fig:graph} and embedded graphs will also be referred to as defects. Note the embeddings are restricted to the plane which is to say that we do not allow for graphs to intersect each other.

The discrete structure of the new model is again given by a triangulation of $(M,\gls{Gamma})$. In this case, however, we restrict the possible choices of triangulation. 

\begin{definition} \label{def:graph_like}
A triangulation $\gls{T}$ of $(M,\gls{Gamma})$, $M \subset \Rb^2$ is called graph-like if the following conditions are satisfied.
\begin{enumerate}[label=F\arabic{*}), ref=(F\arabic{*})]
\item \label{it:defect1} Each defect arrow must intersect only edges and triangles of $T$ and it must intersect an edge at least once. 
\item \label{it:defect2} Edges may be intersected by defect arrows at most once.
\item \label{it:defect3} Each defect node must intersect exactly one triangle.
\item \label{it:defect4} Each triangle contains at most one defect node; if a node is present only defect arrows starting or ending at that node are allowed to cross the triangle.
\end{enumerate}
\end{definition}

The point of choosing a graph-like triangulation is the unambiguous construction of the dual diagram $\gls{G}$ -- this is achieved by regarding the graph components as part of $G$ and creating the entire dual diagram by simply adding the counterparts of surface vertices, edges and triangles (see figure~\ref{fig:graph}). Note that if a triangulation is not graph-like we can always generate a new one from it that is, by a (finite) series of Pachner moves. A straightforward consequence of definition~\ref{def:graph_like} is that the boundary of $M$ is not intersected by defects -- arrows begin and end at nodes and nodes cannot be placed at edges or vertices of the triangulation, the only components of a boundary.   

We assume that a planar model with defects must reduce to a planar model when all graphs present are the empty graph. Therefore, for the parts of the diagram that do not involve defects we will use the maps $\gls{Ct}$, $\gls{Bb}$ and $\gls{R}$ for the duals of triangles, pairs of edges and vertices respectively. However, the introduction of defects presents us with a collection of new variables that we must fix. 

We start with the elements of the dual diagram that involve defect arrows $\mathfrak{b}_{\mathfrak{ij}}$ but not nodes. Fix the pair of nodes $(\mathfrak{i},\mathfrak{j})$. To each arrow $\mathfrak{b}$ we associate a set of states $\alpha, \beta, \gamma \in \gls{Tb}$ as we did for edges in the surface, here labeled with states $a,b,c \in \gls{S}$. We will denote the corresponding arrow $\mathfrak{b}_{\mathfrak{ji}}$, $\mathfrak{b}_{\mathfrak{ij}}$ with reversed orientation, as $\mathfrak{b}^{\times}$ and the new set of states as $\mathfrak{T}(\mathfrak{b}^{\times})$. Then the counterparts of the constants $\gls{C}$ and $\gls{B}$ for defect lines are as follows.
\begin{align}
\begin{aligned}
\hspace{25mm}
&%% Creator: Inkscape 0.48.2, www.inkscape.org
%% PDF/EPS/PS + LaTeX output extension by Johan Engelen, 2010
%% Accompanies image file 'diagC_left.pdf' (pdf, eps, ps)
%%
%% To include the image in your LaTeX document, write
%%   \input{<filename>.pdf_tex}
%%  instead of
%%   \includegraphics{<filename>.pdf}
%% To scale the image, write
%%   \def\svgwidth{<desired width>}
%%   \input{<filename>.pdf_tex}
%%  instead of
%%   \includegraphics[width=<desired width>]{<filename>.pdf}
%%
%% Images with a different path to the parent latex file can
%% be accessed with the `import' package (which may need to be
%% installed) using
%%   \usepackage{import}
%% in the preamble, and then including the image with
%%   \import{<path to file>}{<filename>.pdf_tex}
%% Alternatively, one can specify
%%   \graphicspath{{<path to file>/}}
%% 
%% For more information, please see info/svg-inkscape on CTAN:
%%   http://tug.ctan.org/tex-archive/info/svg-inkscape
%%
\begingroup%
  \makeatletter%
  \providecommand\color[2][]{%
    \errmessage{(Inkscape) Color is used for the text in Inkscape, but the package 'color.sty' is not loaded}%
    \renewcommand\color[2][]{}%
  }%
  \providecommand\transparent[1]{%
    \errmessage{(Inkscape) Transparency is used (non-zero) for the text in Inkscape, but the package 'transparent.sty' is not loaded}%
    \renewcommand\transparent[1]{}%
  }%
  \providecommand\rotatebox[2]{#2}%
  \ifx\svgwidth\undefined%
    \setlength{\unitlength}{282.16124268bp}%
    \ifx\svgscale\undefined%
      \relax%
    \else%
      \setlength{\unitlength}{\unitlength * \real{\svgscale}}%
    \fi%
  \else%
    \setlength{\unitlength}{\svgwidth}%
  \fi%
  \global\let\svgwidth\undefined%
  \global\let\svgscale\undefined%
  \makeatother%
  \begin{picture}(1,0.15358301)%
    \put(0,0){\includegraphics[width=\unitlength]{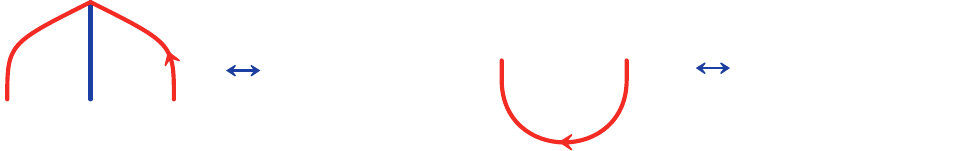}}%
    \put(-0.00101061,0.01310383){\color[rgb]{0,0,0}\makebox(0,0)[lb]{\smash{$\alpha$}}}%
    \put(0.08515608,0.01289478){\color[rgb]{0,0,0}\makebox(0,0)[lb]{\smash{$a$}}}%
    \put(0.1699187,0.01356795){\color[rgb]{0,0,0}\makebox(0,0)[lb]{\smash{$\beta$}}}%
    \put(0.29709519,0.07173184){\color[rgb]{0,0,0}\makebox(0,0)[lb]{\smash{$L_{\alpha a \beta}$}}}%
    \put(0.50404758,0.11095987){\color[rgb]{0,0,0}\makebox(0,0)[lb]{\smash{$\alpha$}}}%
    \put(0.62849871,0.11271252){\color[rgb]{0,0,0}\makebox(0,0)[lb]{\smash{$\beta$}}}%
    \put(0.77823241,0.07487037){\color[rgb]{0,0,0}\makebox(0,0)[lb]{\smash{$P^{\alpha\beta}$}}}%
  \end{picture}%
\endgroup%
 \\
\hspace{25mm}
&%% Creator: Inkscape 0.48.2, www.inkscape.org
%% PDF/EPS/PS + LaTeX output extension by Johan Engelen, 2010
%% Accompanies image file 'diagC_right.pdf' (pdf, eps, ps)
%%
%% To include the image in your LaTeX document, write
%%   \input{<filename>.pdf_tex}
%%  instead of
%%   \includegraphics{<filename>.pdf}
%% To scale the image, write
%%   \def\svgwidth{<desired width>}
%%   \input{<filename>.pdf_tex}
%%  instead of
%%   \includegraphics[width=<desired width>]{<filename>.pdf}
%%
%% Images with a different path to the parent latex file can
%% be accessed with the `import' package (which may need to be
%% installed) using
%%   \usepackage{import}
%% in the preamble, and then including the image with
%%   \import{<path to file>}{<filename>.pdf_tex}
%% Alternatively, one can specify
%%   \graphicspath{{<path to file>/}}
%% 
%% For more information, please see info/svg-inkscape on CTAN:
%%   http://tug.ctan.org/tex-archive/info/svg-inkscape
%%
\begingroup%
  \makeatletter%
  \providecommand\color[2][]{%
    \errmessage{(Inkscape) Color is used for the text in Inkscape, but the package 'color.sty' is not loaded}%
    \renewcommand\color[2][]{}%
  }%
  \providecommand\transparent[1]{%
    \errmessage{(Inkscape) Transparency is used (non-zero) for the text in Inkscape, but the package 'transparent.sty' is not loaded}%
    \renewcommand\transparent[1]{}%
  }%
  \providecommand\rotatebox[2]{#2}%
  \ifx\svgwidth\undefined%
    \setlength{\unitlength}{283.03624268bp}%
    \ifx\svgscale\undefined%
      \relax%
    \else%
      \setlength{\unitlength}{\unitlength * \real{\svgscale}}%
    \fi%
  \else%
    \setlength{\unitlength}{\svgwidth}%
  \fi%
  \global\let\svgwidth\undefined%
  \global\let\svgscale\undefined%
  \makeatother%
  \begin{picture}(1,0.15308475)%
    \put(0,0){\includegraphics[width=\unitlength]{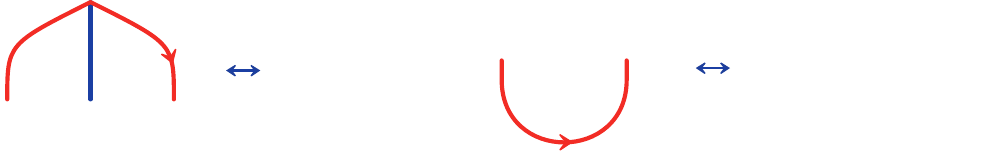}}%
    \put(-0.00100749,0.01303985){\color[rgb]{0,0,0}\makebox(0,0)[lb]{\smash{$\alpha$}}}%
    \put(0.08489283,0.01283145){\color[rgb]{0,0,0}\makebox(0,0)[lb]{\smash{$a$}}}%
    \put(0.1693934,0.01350254){\color[rgb]{0,0,0}\makebox(0,0)[lb]{\smash{$\beta$}}}%
    \put(0.29617673,0.07148662){\color[rgb]{0,0,0}\makebox(0,0)[lb]{\smash{$R_{\alpha a \beta}$}}}%
    \put(0.50248933,0.11059337){\color[rgb]{0,0,0}\makebox(0,0)[lb]{\smash{$\alpha$}}}%
    \put(0.62655572,0.11234061){\color[rgb]{0,0,0}\makebox(0,0)[lb]{\smash{$\beta$}}}%
    \put(0.77582652,0.07461545){\color[rgb]{0,0,0}\makebox(0,0)[lb]{\smash{$Q^{\alpha\beta}$}}}%
  \end{picture}%
\endgroup%
 
\end{aligned}
\end{align}
Note that the constants $L_{\alpha a\beta}$, $R_{\alpha a\beta}$, $P^{\alpha\beta}$, $Q^{\alpha\beta}$ are specific to each defect arrow, a dependence we leave implicit. Combinations of these maps are used to define diagrams with one leg pointing upwards.
\begin{align}
%% Creator: Inkscape 0.48.2, www.inkscape.org
%% PDF/EPS/PS + LaTeX output extension by Johan Engelen, 2010
%% Accompanies image file 'diagleft.pdf' (pdf, eps, ps)
%%
%% To include the image in your LaTeX document, write
%%   \input{<filename>.pdf_tex}
%%  instead of
%%   \includegraphics{<filename>.pdf}
%% To scale the image, write
%%   \def\svgwidth{<desired width>}
%%   \input{<filename>.pdf_tex}
%%  instead of
%%   \includegraphics[width=<desired width>]{<filename>.pdf}
%%
%% Images with a different path to the parent latex file can
%% be accessed with the `import' package (which may need to be
%% installed) using
%%   \usepackage{import}
%% in the preamble, and then including the image with
%%   \import{<path to file>}{<filename>.pdf_tex}
%% Alternatively, one can specify
%%   \graphicspath{{<path to file>/}}
%% 
%% For more information, please see info/svg-inkscape on CTAN:
%%   http://tug.ctan.org/tex-archive/info/svg-inkscape
%%
\begingroup%
  \makeatletter%
  \providecommand\color[2][]{%
    \errmessage{(Inkscape) Color is used for the text in Inkscape, but the package 'color.sty' is not loaded}%
    \renewcommand\color[2][]{}%
  }%
  \providecommand\transparent[1]{%
    \errmessage{(Inkscape) Transparency is used (non-zero) for the text in Inkscape, but the package 'transparent.sty' is not loaded}%
    \renewcommand\transparent[1]{}%
  }%
  \providecommand\rotatebox[2]{#2}%
  \ifx\svgwidth\undefined%
    \setlength{\unitlength}{559.25952148bp}%
    \ifx\svgscale\undefined%
      \relax%
    \else%
      \setlength{\unitlength}{\unitlength * \real{\svgscale}}%
    \fi%
  \else%
    \setlength{\unitlength}{\svgwidth}%
  \fi%
  \global\let\svgwidth\undefined%
  \global\let\svgscale\undefined%
  \makeatother%
  \begin{picture}(1,0.10109027)%
    \put(0,0){\includegraphics[width=\unitlength]{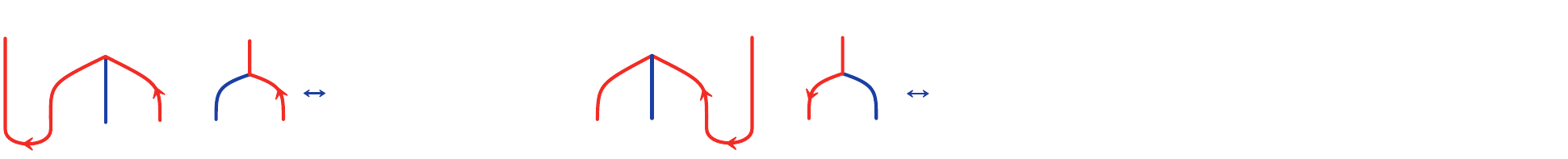}}%
    \put(-0.00050988,0.08865408){\color[rgb]{0,0,0}\makebox(0,0)[lb]{\smash{$\beta$}}}%
    \put(0.06384859,0.00255286){\color[rgb]{0,0,0}\makebox(0,0)[lb]{\smash{$a$}}}%
    \put(0.09834707,0.00284277){\color[rgb]{0,0,0}\makebox(0,0)[lb]{\smash{$\alpha$}}}%
    \put(0.13342537,0.00284277){\color[rgb]{0,0,0}\makebox(0,0)[lb]{\smash{$a$}}}%
    \put(0.17643333,0.00349112){\color[rgb]{0,0,0}\makebox(0,0)[lb]{\smash{$\alpha$}}}%
    \put(0.15497639,0.08931889){\color[rgb]{0,0,0}\makebox(0,0)[lb]{\smash{$\beta$}}}%
    \put(0.21510512,0.03675021){\color[rgb]{0,0,0}\makebox(0,0)[lb]{\smash{$P^{\beta\gamma}L_{\gamma a \alpha}=L_{a\alpha}{}^{\beta},$}}}%
    \put(0.11304455,0.03636265){\color[rgb]{0,0,0}\makebox(0,0)[lb]{\smash{$=$}}}%
    \put(0.3767349,0.00314175){\color[rgb]{0,0,0}\makebox(0,0)[lb]{\smash{$\alpha$}}}%
    \put(0.41152337,0.00343171){\color[rgb]{0,0,0}\makebox(0,0)[lb]{\smash{$a$}}}%
    \put(0.47588176,0.08953293){\color[rgb]{0,0,0}\makebox(0,0)[lb]{\smash{$\beta$}}}%
    \put(0.50977825,0.00359995){\color[rgb]{0,0,0}\makebox(0,0)[lb]{\smash{$\alpha$}}}%
    \put(0.55458036,0.00408001){\color[rgb]{0,0,0}\makebox(0,0)[lb]{\smash{$a$}}}%
    \put(0.53312341,0.08990779){\color[rgb]{0,0,0}\makebox(0,0)[lb]{\smash{$\beta$}}}%
    \put(0.48877906,0.03698066){\color[rgb]{0,0,0}\makebox(0,0)[lb]{\smash{$=$}}}%
    \put(0.59933768,0.03717086){\color[rgb]{0,0,0}\makebox(0,0)[lb]{\smash{$L_{\alpha a \gamma}Q^{\gamma\beta}=L_{\alpha a}{}^{\beta}$}}}%
  \end{picture}%
\endgroup%

\end{align}
\begin{align}
%% Creator: Inkscape 0.48.2, www.inkscape.org
%% PDF/EPS/PS + LaTeX output extension by Johan Engelen, 2010
%% Accompanies image file 'diagright.pdf' (pdf, eps, ps)
%%
%% To include the image in your LaTeX document, write
%%   \input{<filename>.pdf_tex}
%%  instead of
%%   \includegraphics{<filename>.pdf}
%% To scale the image, write
%%   \def\svgwidth{<desired width>}
%%   \input{<filename>.pdf_tex}
%%  instead of
%%   \includegraphics[width=<desired width>]{<filename>.pdf}
%%
%% Images with a different path to the parent latex file can
%% be accessed with the `import' package (which may need to be
%% installed) using
%%   \usepackage{import}
%% in the preamble, and then including the image with
%%   \import{<path to file>}{<filename>.pdf_tex}
%% Alternatively, one can specify
%%   \graphicspath{{<path to file>/}}
%% 
%% For more information, please see info/svg-inkscape on CTAN:
%%   http://tug.ctan.org/tex-archive/info/svg-inkscape
%%
\begingroup%
  \makeatletter%
  \providecommand\color[2][]{%
    \errmessage{(Inkscape) Color is used for the text in Inkscape, but the package 'color.sty' is not loaded}%
    \renewcommand\color[2][]{}%
  }%
  \providecommand\transparent[1]{%
    \errmessage{(Inkscape) Transparency is used (non-zero) for the text in Inkscape, but the package 'transparent.sty' is not loaded}%
    \renewcommand\transparent[1]{}%
  }%
  \providecommand\rotatebox[2]{#2}%
  \ifx\svgwidth\undefined%
    \setlength{\unitlength}{561.04077148bp}%
    \ifx\svgscale\undefined%
      \relax%
    \else%
      \setlength{\unitlength}{\unitlength * \real{\svgscale}}%
    \fi%
  \else%
    \setlength{\unitlength}{\svgwidth}%
  \fi%
  \global\let\svgwidth\undefined%
  \global\let\svgscale\undefined%
  \makeatother%
  \begin{picture}(1,0.10076931)%
    \put(0,0){\includegraphics[width=\unitlength]{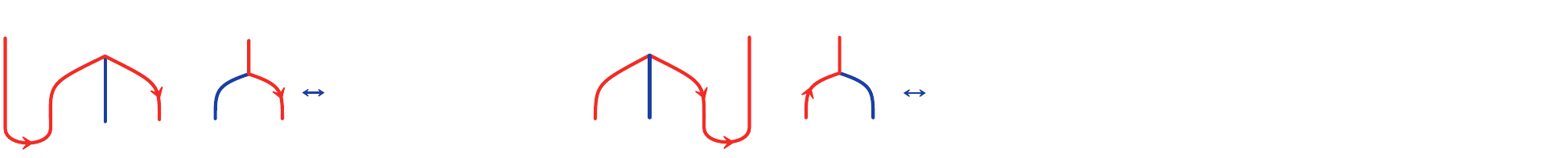}}%
    \put(-0.00050826,0.08837261){\color[rgb]{0,0,0}\makebox(0,0)[lb]{\smash{$\beta$}}}%
    \put(0.06364588,0.00254475){\color[rgb]{0,0,0}\makebox(0,0)[lb]{\smash{$a$}}}%
    \put(0.09803483,0.00283374){\color[rgb]{0,0,0}\makebox(0,0)[lb]{\smash{$\alpha$}}}%
    \put(0.13300176,0.00283374){\color[rgb]{0,0,0}\makebox(0,0)[lb]{\smash{$a$}}}%
    \put(0.17587317,0.00348003){\color[rgb]{0,0,0}\makebox(0,0)[lb]{\smash{$\alpha$}}}%
    \put(0.15448435,0.08903531){\color[rgb]{0,0,0}\makebox(0,0)[lb]{\smash{$\beta$}}}%
    \put(0.21442218,0.03663353){\color[rgb]{0,0,0}\makebox(0,0)[lb]{\smash{$Q^{\beta\gamma}R_{\gamma a \alpha}=R_{a\alpha}{}^{\beta},$}}}%
    \put(0.11268564,0.0362472){\color[rgb]{0,0,0}\makebox(0,0)[lb]{\smash{$=$}}}%
    \put(0.3755388,0.00313178){\color[rgb]{0,0,0}\makebox(0,0)[lb]{\smash{$\alpha$}}}%
    \put(0.41021683,0.00342081){\color[rgb]{0,0,0}\makebox(0,0)[lb]{\smash{$a$}}}%
    \put(0.47437088,0.08924867){\color[rgb]{0,0,0}\makebox(0,0)[lb]{\smash{$\beta$}}}%
    \put(0.50815975,0.00358852){\color[rgb]{0,0,0}\makebox(0,0)[lb]{\smash{$\alpha$}}}%
    \put(0.55281962,0.00406706){\color[rgb]{0,0,0}\makebox(0,0)[lb]{\smash{$a$}}}%
    \put(0.5314308,0.08962234){\color[rgb]{0,0,0}\makebox(0,0)[lb]{\smash{$\beta$}}}%
    \put(0.48722724,0.03686325){\color[rgb]{0,0,0}\makebox(0,0)[lb]{\smash{$=$}}}%
    \put(0.59743484,0.03705285){\color[rgb]{0,0,0}\makebox(0,0)[lb]{\smash{$R_{\alpha a \gamma}P^{\gamma\beta}=R_{\alpha a}{}^{\beta}$}}}%
  \end{picture}%
\endgroup%

\end{align}
Similar combinations allow us to define defect-line diagrams with two or three legs pointing upwards. We must now include the data associated with defect nodes. Allow us to represent by a solid black line a leg that might belong to the surface or a defect, $\begin{aligned} %% Creator: Inkscape 0.48.2, www.inkscape.org
%% PDF/EPS/PS + LaTeX output extension by Johan Engelen, 2010
%% Accompanies image file 'solid_line_1.pdf' (pdf, eps, ps)
%%
%% To include the image in your LaTeX document, write
%%   \input{<filename>.pdf_tex}
%%  instead of
%%   \includegraphics{<filename>.pdf}
%% To scale the image, write
%%   \def\svgwidth{<desired width>}
%%   \input{<filename>.pdf_tex}
%%  instead of
%%   \includegraphics[width=<desired width>]{<filename>.pdf}
%%
%% Images with a different path to the parent latex file can
%% be accessed with the `import' package (which may need to be
%% installed) using
%%   \usepackage{import}
%% in the preamble, and then including the image with
%%   \import{<path to file>}{<filename>.pdf_tex}
%% Alternatively, one can specify
%%   \graphicspath{{<path to file>/}}
%% 
%% For more information, please see info/svg-inkscape on CTAN:
%%   http://tug.ctan.org/tex-archive/info/svg-inkscape
%%
\begingroup%
  \makeatletter%
  \providecommand\color[2][]{%
    \errmessage{(Inkscape) Color is used for the text in Inkscape, but the package 'color.sty' is not loaded}%
    \renewcommand\color[2][]{}%
  }%
  \providecommand\transparent[1]{%
    \errmessage{(Inkscape) Transparency is used (non-zero) for the text in Inkscape, but the package 'transparent.sty' is not loaded}%
    \renewcommand\transparent[1]{}%
  }%
  \providecommand\rotatebox[2]{#2}%
  \ifx\svgwidth\undefined%
    \setlength{\unitlength}{1.20146561bp}%
    \ifx\svgscale\undefined%
      \relax%
    \else%
      \setlength{\unitlength}{\unitlength * \real{\svgscale}}%
    \fi%
  \else%
    \setlength{\unitlength}{\svgwidth}%
  \fi%
  \global\let\svgwidth\undefined%
  \global\let\svgscale\undefined%
  \makeatother%
  \begin{picture}(1,11.01848951)%
    \put(0,0){\includegraphics[width=\unitlength]{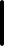}}%
  \end{picture}%
\endgroup%
 = %% Creator: Inkscape 0.48.2, www.inkscape.org
%% PDF/EPS/PS + LaTeX output extension by Johan Engelen, 2010
%% Accompanies image file 'solid_line_2.pdf' (pdf, eps, ps)
%%
%% To include the image in your LaTeX document, write
%%   \input{<filename>.pdf_tex}
%%  instead of
%%   \includegraphics{<filename>.pdf}
%% To scale the image, write
%%   \def\svgwidth{<desired width>}
%%   \input{<filename>.pdf_tex}
%%  instead of
%%   \includegraphics[width=<desired width>]{<filename>.pdf}
%%
%% Images with a different path to the parent latex file can
%% be accessed with the `import' package (which may need to be
%% installed) using
%%   \usepackage{import}
%% in the preamble, and then including the image with
%%   \import{<path to file>}{<filename>.pdf_tex}
%% Alternatively, one can specify
%%   \graphicspath{{<path to file>/}}
%% 
%% For more information, please see info/svg-inkscape on CTAN:
%%   http://tug.ctan.org/tex-archive/info/svg-inkscape
%%
\begingroup%
  \makeatletter%
  \providecommand\color[2][]{%
    \errmessage{(Inkscape) Color is used for the text in Inkscape, but the package 'color.sty' is not loaded}%
    \renewcommand\color[2][]{}%
  }%
  \providecommand\transparent[1]{%
    \errmessage{(Inkscape) Transparency is used (non-zero) for the text in Inkscape, but the package 'transparent.sty' is not loaded}%
    \renewcommand\transparent[1]{}%
  }%
  \providecommand\rotatebox[2]{#2}%
  \ifx\svgwidth\undefined%
    \setlength{\unitlength}{1.20146561bp}%
    \ifx\svgscale\undefined%
      \relax%
    \else%
      \setlength{\unitlength}{\unitlength * \real{\svgscale}}%
    \fi%
  \else%
    \setlength{\unitlength}{\svgwidth}%
  \fi%
  \global\let\svgwidth\undefined%
  \global\let\svgscale\undefined%
  \makeatother%
  \begin{picture}(1,11.01848951)%
    \put(0,0){\includegraphics[width=\unitlength]{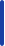}}%
  \end{picture}%
\endgroup%
, %% Creator: Inkscape 0.48.2, www.inkscape.org
%% PDF/EPS/PS + LaTeX output extension by Johan Engelen, 2010
%% Accompanies image file 'solid_line_3.pdf' (pdf, eps, ps)
%%
%% To include the image in your LaTeX document, write
%%   \input{<filename>.pdf_tex}
%%  instead of
%%   \includegraphics{<filename>.pdf}
%% To scale the image, write
%%   \def\svgwidth{<desired width>}
%%   \input{<filename>.pdf_tex}
%%  instead of
%%   \includegraphics[width=<desired width>]{<filename>.pdf}
%%
%% Images with a different path to the parent latex file can
%% be accessed with the `import' package (which may need to be
%% installed) using
%%   \usepackage{import}
%% in the preamble, and then including the image with
%%   \import{<path to file>}{<filename>.pdf_tex}
%% Alternatively, one can specify
%%   \graphicspath{{<path to file>/}}
%% 
%% For more information, please see info/svg-inkscape on CTAN:
%%   http://tug.ctan.org/tex-archive/info/svg-inkscape
%%
\begingroup%
  \makeatletter%
  \providecommand\color[2][]{%
    \errmessage{(Inkscape) Color is used for the text in Inkscape, but the package 'color.sty' is not loaded}%
    \renewcommand\color[2][]{}%
  }%
  \providecommand\transparent[1]{%
    \errmessage{(Inkscape) Transparency is used (non-zero) for the text in Inkscape, but the package 'transparent.sty' is not loaded}%
    \renewcommand\transparent[1]{}%
  }%
  \providecommand\rotatebox[2]{#2}%
  \ifx\svgwidth\undefined%
    \setlength{\unitlength}{4.70028763bp}%
    \ifx\svgscale\undefined%
      \relax%
    \else%
      \setlength{\unitlength}{\unitlength * \real{\svgscale}}%
    \fi%
  \else%
    \setlength{\unitlength}{\svgwidth}%
  \fi%
  \global\let\svgwidth\undefined%
  \global\let\svgscale\undefined%
  \makeatother%
  \begin{picture}(1,2.81649491)%
    \put(0,0){\includegraphics[width=\unitlength]{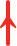}}%
  \end{picture}%
\endgroup%
, %% Creator: Inkscape 0.48.2, www.inkscape.org
%% PDF/EPS/PS + LaTeX output extension by Johan Engelen, 2010
%% Accompanies image file 'solid_line_4.pdf' (pdf, eps, ps)
%%
%% To include the image in your LaTeX document, write
%%   \input{<filename>.pdf_tex}
%%  instead of
%%   \includegraphics{<filename>.pdf}
%% To scale the image, write
%%   \def\svgwidth{<desired width>}
%%   \input{<filename>.pdf_tex}
%%  instead of
%%   \includegraphics[width=<desired width>]{<filename>.pdf}
%%
%% Images with a different path to the parent latex file can
%% be accessed with the `import' package (which may need to be
%% installed) using
%%   \usepackage{import}
%% in the preamble, and then including the image with
%%   \import{<path to file>}{<filename>.pdf_tex}
%% Alternatively, one can specify
%%   \graphicspath{{<path to file>/}}
%% 
%% For more information, please see info/svg-inkscape on CTAN:
%%   http://tug.ctan.org/tex-archive/info/svg-inkscape
%%
\begingroup%
  \makeatletter%
  \providecommand\color[2][]{%
    \errmessage{(Inkscape) Color is used for the text in Inkscape, but the package 'color.sty' is not loaded}%
    \renewcommand\color[2][]{}%
  }%
  \providecommand\transparent[1]{%
    \errmessage{(Inkscape) Transparency is used (non-zero) for the text in Inkscape, but the package 'transparent.sty' is not loaded}%
    \renewcommand\transparent[1]{}%
  }%
  \providecommand\rotatebox[2]{#2}%
  \ifx\svgwidth\undefined%
    \setlength{\unitlength}{4.69499741bp}%
    \ifx\svgscale\undefined%
      \relax%
    \else%
      \setlength{\unitlength}{\unitlength * \real{\svgscale}}%
    \fi%
  \else%
    \setlength{\unitlength}{\svgwidth}%
  \fi%
  \global\let\svgwidth\undefined%
  \global\let\svgscale\undefined%
  \makeatother%
  \begin{picture}(1,2.81966848)%
    \put(0,0){\includegraphics[width=\unitlength]{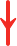}}%
  \end{picture}%
\endgroup%
\end{aligned}$. Then the constant associated with a node is represented as follows, where $I,J,K$ are labels in either $\gls{S}$, $\gls{Tb}$ or $\mathfrak{T}(\mathfrak{b}^{\times})$.
\begin{align}
\begin{aligned}
\label{eq:diag_node}
\hspace{25mm}
%% Creator: Inkscape 0.48.2, www.inkscape.org
%% PDF/EPS/PS + LaTeX output extension by Johan Engelen, 2010
%% Accompanies image file 'diagnode.pdf' (pdf, eps, ps)
%%
%% To include the image in your LaTeX document, write
%%   \input{<filename>.pdf_tex}
%%  instead of
%%   \includegraphics{<filename>.pdf}
%% To scale the image, write
%%   \def\svgwidth{<desired width>}
%%   \input{<filename>.pdf_tex}
%%  instead of
%%   \includegraphics[width=<desired width>]{<filename>.pdf}
%%
%% Images with a different path to the parent latex file can
%% be accessed with the `import' package (which may need to be
%% installed) using
%%   \usepackage{import}
%% in the preamble, and then including the image with
%%   \import{<path to file>}{<filename>.pdf_tex}
%% Alternatively, one can specify
%%   \graphicspath{{<path to file>/}}
%% 
%% For more information, please see info/svg-inkscape on CTAN:
%%   http://tug.ctan.org/tex-archive/info/svg-inkscape
%%
\begingroup%
  \makeatletter%
  \providecommand\color[2][]{%
    \errmessage{(Inkscape) Color is used for the text in Inkscape, but the package 'color.sty' is not loaded}%
    \renewcommand\color[2][]{}%
  }%
  \providecommand\transparent[1]{%
    \errmessage{(Inkscape) Transparency is used (non-zero) for the text in Inkscape, but the package 'transparent.sty' is not loaded}%
    \renewcommand\transparent[1]{}%
  }%
  \providecommand\rotatebox[2]{#2}%
  \ifx\svgwidth\undefined%
    \setlength{\unitlength}{120.9078125bp}%
    \ifx\svgscale\undefined%
      \relax%
    \else%
      \setlength{\unitlength}{\unitlength * \real{\svgscale}}%
    \fi%
  \else%
    \setlength{\unitlength}{\svgwidth}%
  \fi%
  \global\let\svgwidth\undefined%
  \global\let\svgscale\undefined%
  \makeatother%
  \begin{picture}(1,0.30411227)%
    \put(0,0){\includegraphics[width=\unitlength]{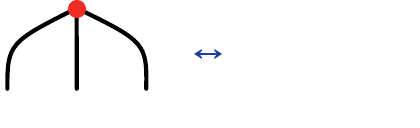}}%
    \put(-0.00235846,0.00681692){\color[rgb]{0,0,0}\makebox(0,0)[lb]{\smash{$I$}}}%
    \put(0.16305682,0.00681692){\color[rgb]{0,0,0}\makebox(0,0)[lb]{\smash{$J$}}}%
    \put(0.32847211,0.00681692){\color[rgb]{0,0,0}\makebox(0,0)[lb]{\smash{$K$}}}%
    \put(0.64606945,0.15238237){\color[rgb]{0,0,0}\makebox(0,0)[lb]{\smash{$\phi_{IJK}$}}}%
  \end{picture}%
\endgroup%

\end{aligned}
\end{align}
In chapter \S\ref{sec:diagram} we discussed how the cyclic symmetry of $\gls{Ct}$ and the identity $\gls{B}=B^{ba}$ could be replaced with a more general requirement, equation \eqref{eq:BCequation}. For FHK models with defects, it is discussed in \cite{Runkel} how the former requirements apply to defect graphs: $P^{\alpha\beta}=P^{\beta\alpha}$, $Q^{\alpha\beta}=Q^{\beta\alpha}$ and $\phi_{IJK}=\phi_{JKI}=\phi_{KIJ}$ whenever $I,J,K$ are labels in the same space guarantee rotational symmetry. Diagrammatic models with defects include a refinement of these conditions. One, we set $B^{ab}P^{\alpha\gamma}L_{b\gamma}{}^{\beta}=B^{ba}Q^{\gamma\beta}L_{\gamma b}{}^{\beta}$ as depicted below,
\begin{align}
\begin{aligned}
\label{eq:defects_cyclicity}
%% Creator: Inkscape 0.48.2, www.inkscape.org
%% PDF/EPS/PS + LaTeX output extension by Johan Engelen, 2010
%% Accompanies image file 'defects_cyclicity.pdf' (pdf, eps, ps)
%%
%% To include the image in your LaTeX document, write
%%   \input{<filename>.pdf_tex}
%%  instead of
%%   \includegraphics{<filename>.pdf}
%% To scale the image, write
%%   \def\svgwidth{<desired width>}
%%   \input{<filename>.pdf_tex}
%%  instead of
%%   \includegraphics[width=<desired width>]{<filename>.pdf}
%%
%% Images with a different path to the parent latex file can
%% be accessed with the `import' package (which may need to be
%% installed) using
%%   \usepackage{import}
%% in the preamble, and then including the image with
%%   \import{<path to file>}{<filename>.pdf_tex}
%% Alternatively, one can specify
%%   \graphicspath{{<path to file>/}}
%% 
%% For more information, please see info/svg-inkscape on CTAN:
%%   http://tug.ctan.org/tex-archive/info/svg-inkscape
%%
\begingroup%
  \makeatletter%
  \providecommand\color[2][]{%
    \errmessage{(Inkscape) Color is used for the text in Inkscape, but the package 'color.sty' is not loaded}%
    \renewcommand\color[2][]{}%
  }%
  \providecommand\transparent[1]{%
    \errmessage{(Inkscape) Transparency is used (non-zero) for the text in Inkscape, but the package 'transparent.sty' is not loaded}%
    \renewcommand\transparent[1]{}%
  }%
  \providecommand\rotatebox[2]{#2}%
  \ifx\svgwidth\undefined%
    \setlength{\unitlength}{222.0390625bp}%
    \ifx\svgscale\undefined%
      \relax%
    \else%
      \setlength{\unitlength}{\unitlength * \real{\svgscale}}%
    \fi%
  \else%
    \setlength{\unitlength}{\svgwidth}%
  \fi%
  \global\let\svgwidth\undefined%
  \global\let\svgscale\undefined%
  \makeatother%
  \begin{picture}(1,0.34459832)%
    \put(0,0){\includegraphics[width=\unitlength]{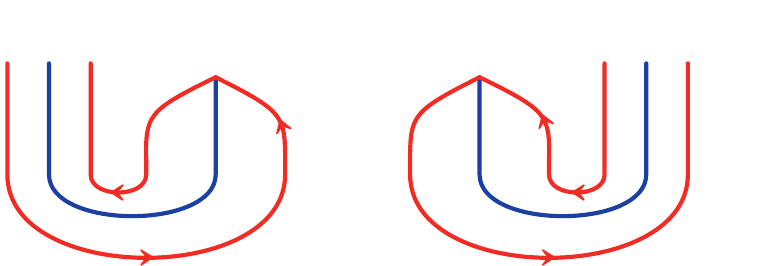}}%
    \put(0.441881,0.17231374){\color[rgb]{0,0,0}\makebox(0,0)[lb]{\smash{$=$}}}%
    \put(-0.00128426,0.31643253){\color[rgb]{0,0,0}\makebox(0,0)[lb]{\smash{$\alpha$}}}%
    \put(0.05276028,0.31643253){\color[rgb]{0,0,0}\makebox(0,0)[lb]{\smash{$a$}}}%
    \put(0.10680483,0.31643253){\color[rgb]{0,0,0}\makebox(0,0)[lb]{\smash{$\beta$}}}%
    \put(0.77335421,0.31643253){\color[rgb]{0,0,0}\makebox(0,0)[lb]{\smash{$\alpha$}}}%
    \put(0.82739875,0.31643253){\color[rgb]{0,0,0}\makebox(0,0)[lb]{\smash{$a$}}}%
    \put(0.8814433,0.31643253){\color[rgb]{0,0,0}\makebox(0,0)[lb]{\smash{$\beta$}}}%
  \end{picture}%
\endgroup%

\end{aligned}
\end{align}
and the analogous relation for $R_{\alpha a \beta}$ that is given by reversing all arrow orientations on equation~\eqref{eq:defects_cyclicity}. Two, whenever the sets of states $I,J,K$ belong to coincide the following equations must be satisfied. 
\begin{align} \label{eq:phi_samespace}
\begin{aligned}
\hspace{15mm}
%% Creator: Inkscape 0.48.2, www.inkscape.org
%% PDF/EPS/PS + LaTeX output extension by Johan Engelen, 2010
%% Accompanies image file 'diag_general_cyclic.pdf' (pdf, eps, ps)
%%
%% To include the image in your LaTeX document, write
%%   \input{<filename>.pdf_tex}
%%  instead of
%%   \includegraphics{<filename>.pdf}
%% To scale the image, write
%%   \def\svgwidth{<desired width>}
%%   \input{<filename>.pdf_tex}
%%  instead of
%%   \includegraphics[width=<desired width>]{<filename>.pdf}
%%
%% Images with a different path to the parent latex file can
%% be accessed with the `import' package (which may need to be
%% installed) using
%%   \usepackage{import}
%% in the preamble, and then including the image with
%%   \import{<path to file>}{<filename>.pdf_tex}
%% Alternatively, one can specify
%%   \graphicspath{{<path to file>/}}
%% 
%% For more information, please see info/svg-inkscape on CTAN:
%%   http://tug.ctan.org/tex-archive/info/svg-inkscape
%%
\begingroup%
  \makeatletter%
  \providecommand\color[2][]{%
    \errmessage{(Inkscape) Color is used for the text in Inkscape, but the package 'color.sty' is not loaded}%
    \renewcommand\color[2][]{}%
  }%
  \providecommand\transparent[1]{%
    \errmessage{(Inkscape) Transparency is used (non-zero) for the text in Inkscape, but the package 'transparent.sty' is not loaded}%
    \renewcommand\transparent[1]{}%
  }%
  \providecommand\rotatebox[2]{#2}%
  \ifx\svgwidth\undefined%
    \setlength{\unitlength}{212.0171875bp}%
    \ifx\svgscale\undefined%
      \relax%
    \else%
      \setlength{\unitlength}{\unitlength * \real{\svgscale}}%
    \fi%
  \else%
    \setlength{\unitlength}{\svgwidth}%
  \fi%
  \global\let\svgwidth\undefined%
  \global\let\svgscale\undefined%
  \makeatother%
  \begin{picture}(1,0.25223486)%
    \put(0,0){\includegraphics[width=\unitlength]{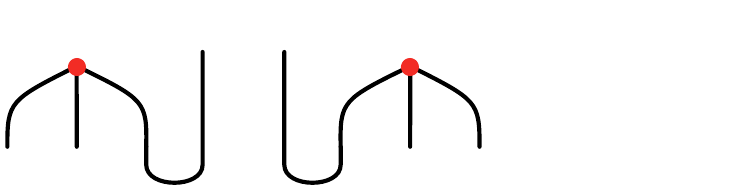}}%
    \put(-0.00134497,0.00388751){\color[rgb]{0,0,0}\makebox(0,0)[lb]{\smash{$I$}}}%
    \put(0.09298701,0.00388751){\color[rgb]{0,0,0}\makebox(0,0)[lb]{\smash{$J$}}}%
    \put(0.70803148,0.07180653){\color[rgb]{0,0,0}\makebox(0,0)[lb]{\smash{$=\phi_{IJ}{}^{K}$}}}%
    \put(0.54578049,0.00388751){\color[rgb]{0,0,0}\makebox(0,0)[lb]{\smash{$I$}}}%
    \put(0.64011246,0.00388751){\color[rgb]{0,0,0}\makebox(0,0)[lb]{\smash{$J$}}}%
    \put(0.31183719,0.07180653){\color[rgb]{0,0,0}\makebox(0,0)[lb]{\smash{$=$}}}%
    \put(0.26278456,0.22273769){\color[rgb]{0,0,0}\makebox(0,0)[lb]{\smash{$K$}}}%
    \put(0.37598293,0.22273769){\color[rgb]{0,0,0}\makebox(0,0)[lb]{\smash{$K$}}}%
  \end{picture}%
\endgroup%

\end{aligned}
\end{align}
Three, the definition of $\phi^{IJK}$ must be unambiguous.
\begin{align} \label{eq:phi_neqspace}
\begin{aligned}
\centering
%% Creator: Inkscape 0.48.2, www.inkscape.org
%% PDF/EPS/PS + LaTeX output extension by Johan Engelen, 2010
%% Accompanies image file 'defects_cyclicity3.pdf' (pdf, eps, ps)
%%
%% To include the image in your LaTeX document, write
%%   \input{<filename>.pdf_tex}
%%  instead of
%%   \includegraphics{<filename>.pdf}
%% To scale the image, write
%%   \def\svgwidth{<desired width>}
%%   \input{<filename>.pdf_tex}
%%  instead of
%%   \includegraphics[width=<desired width>]{<filename>.pdf}
%%
%% Images with a different path to the parent latex file can
%% be accessed with the `import' package (which may need to be
%% installed) using
%%   \usepackage{import}
%% in the preamble, and then including the image with
%%   \import{<path to file>}{<filename>.pdf_tex}
%% Alternatively, one can specify
%%   \graphicspath{{<path to file>/}}
%% 
%% For more information, please see info/svg-inkscape on CTAN:
%%   http://tug.ctan.org/tex-archive/info/svg-inkscape
%%
\begingroup%
  \makeatletter%
  \providecommand\color[2][]{%
    \errmessage{(Inkscape) Color is used for the text in Inkscape, but the package 'color.sty' is not loaded}%
    \renewcommand\color[2][]{}%
  }%
  \providecommand\transparent[1]{%
    \errmessage{(Inkscape) Transparency is used (non-zero) for the text in Inkscape, but the package 'transparent.sty' is not loaded}%
    \renewcommand\transparent[1]{}%
  }%
  \providecommand\rotatebox[2]{#2}%
  \ifx\svgwidth\undefined%
    \setlength{\unitlength}{209.5859375bp}%
    \ifx\svgscale\undefined%
      \relax%
    \else%
      \setlength{\unitlength}{\unitlength * \real{\svgscale}}%
    \fi%
  \else%
    \setlength{\unitlength}{\svgwidth}%
  \fi%
  \global\let\svgwidth\undefined%
  \global\let\svgscale\undefined%
  \makeatother%
  \begin{picture}(1,0.3572003)%
    \put(0,0){\includegraphics[width=\unitlength]{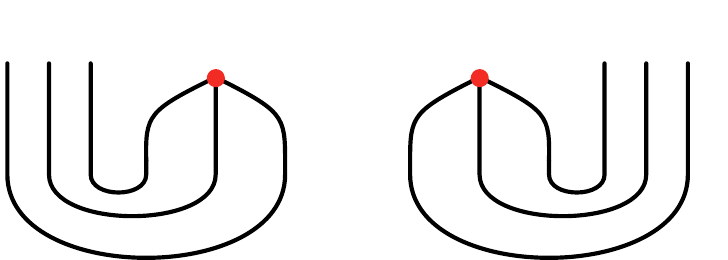}}%
    \put(0.46813658,0.17467896){\color[rgb]{0,0,0}\makebox(0,0)[lb]{\smash{$=$}}}%
    \put(-0.00136057,0.32736096){\color[rgb]{0,0,0}\makebox(0,0)[lb]{\smash{$I$}}}%
    \put(0.05589518,0.32736096){\color[rgb]{0,0,0}\makebox(0,0)[lb]{\smash{$J$}}}%
    \put(0.11315093,0.32736096){\color[rgb]{0,0,0}\makebox(0,0)[lb]{\smash{$K$}}}%
    \put(0.81930518,0.32736096){\color[rgb]{0,0,0}\makebox(0,0)[lb]{\smash{$I$}}}%
    \put(0.87656093,0.32736096){\color[rgb]{0,0,0}\makebox(0,0)[lb]{\smash{$J$}}}%
    \put(0.93381668,0.32736096){\color[rgb]{0,0,0}\makebox(0,0)[lb]{\smash{$K$}}}%
  \end{picture}%
\endgroup%

\end{aligned}
\end{align}
It is easy to see both \eqref{eq:phi_samespace} and \eqref{eq:phi_neqspace} reduce to \eqref{eq:BCequation} when $I,J,K \in \gls{S}$ and $\phi=\gls{Ct}$. 

The partition function reads, again, as the evaluation of the diagram $\gls{Gg}$ associated with the surface with defects $(M,\gls{Gamma})$.
\begin{align}
\gls{Z}(M,\gls{Gamma})=|\gls{Gg}|
\end{align}
The linear algebra interpretation of section \S\ref{sec:lattice_tft} can be extended to defect data. Let $\mathfrak{b}$ be a positively-oriented arrow. We associate a state $\alpha$ of the set $\gls{Tb}$ with a basis element $v_{\alpha}$ of a vector space $\gls{V}$ -- the dependence in $\mathfrak{b}$ is left implicit from this point onwards. By convention, if we reverse the orientation of the arrow $\mathfrak{b}$ a state $\alpha \in \mathfrak{T}(\mathfrak{b}^{\times})$ is associated with a basis element $w_{\alpha}$ of a new vector space that we will denote as $V^{\times}$. The constants $L_{\alpha a \beta}$ and $R_{\alpha a \beta}$ are then interpreted as maps
\begin{align}
L \colon V \times \gls{A} \times V^{\times} \to \gls{k}, \hspace{5mm} (v_{\alpha},e_a,w_{\beta}) \mapsto L_{\alpha a \beta}, \label{eq:L}\\
R \colon V^{\times} \times A \times V \to k, \hspace{5mm} (w_{\alpha},e_a,v_{\beta}) \mapsto R_{\alpha a \beta} \label{eq:R}.
\end{align}
The matrix element $P^{\alpha\beta}$ is seen as defined through $P=v_{\alpha} \otimes w_{\beta} \, P^{\alpha\beta} \in V \otimes V^{\times}$. Similarly, $Q^{\alpha\beta}$ is defined through $Q=w_{\alpha} \otimes v_{\beta} \, Q^{\alpha\beta} \in V^{\times} \otimes V$. Furthermore, for each $V$ we will define a left and right action of $A$:
\begin{align}
\gls{l} \colon A \otimes V \to V, \hspace{5mm} e_a \otimes v_{\alpha} \mapsto  l(e_a \otimes v_{\alpha})=L_{a\alpha}{}^{\beta}v_{\beta},\\
\gls{r} \colon V \otimes A \to V, \hspace{5mm} v_{\alpha} \otimes e_a  \mapsto r(v_{\alpha} \otimes e_a)=R_{\alpha a}{}^{\beta}v_{\beta}. 
\end{align}
Note that our definition of actions of $A$ on a vector space $V$ does not imply at this stage any assumptions other than linearity. For convenience we will also denote $l(e_a \otimes v_{\alpha})=e_a \cdot v_{\alpha}$ and $r(v_{\alpha} \otimes e_a)= v_{\alpha} \cdot e_a$. The interpretation of the symbol $\cdot$ as a multiplication, left or right action is made clear by the vectors being multiplied or acted upon. Left and right actions for $V^{\times}$ also exist: 
\begin{align}
l^{\times} \colon A \otimes V^{\times} \to V^{\times}, \hspace{5mm} e_a \otimes w_{\alpha} \mapsto l^{\times}(e_a \otimes w_{\alpha})=R_{a\alpha}{}^{\beta}w_{\beta},\\
r^{\times} \colon V^{\times} \otimes A \to V^{\times}, \hspace{5mm} w_{\alpha} \otimes e_a  \mapsto  r^{\times}(w_{\alpha} \otimes e_a)=L_{\alpha a}{}^{\beta}w_{\beta}. 
\end{align}
Finally, the constants $\phi_{IJK}$ also have a linear algebra interpretation. Let the states $I$, $J$, $K$ in $U$, $U'$, $U''$ be associated with basis elements $e_I$, $e_J$, $e_K$ of vector spaces $W$, $W'$, $W''$. Then, $\phi \colon W \times W' \times W'' \to k$ is defined according to the relation $\phi(e_I,e_J,e_K)=\phi_{IJK}$. It can also be seen as a map $\phi \colon W \otimes W' \otimes W'' \to \gls{k}$.

We are now prepared to define the concept of defect data. On what follows we consider $M \subset \Rb^2$ to be anti-clockwise oriented.

\begin{definition}
Let $(M,\Gamma)$ be graph-like triangulated and let $G$ be the dual graph that comes associated with it. Let $G_0$ be the set of vertices in $G$, and $G_1$ be its set of edges. Without loss of generality consider all trivalent vertices to have only downward pointing legs.
\begin{enumerate}[label=G\arabic{*}), ref=(G\arabic{*})]
\item To each edge $\mathfrak{b} \in G_1 \cap \Gamma_1$ with an orientation opposite to that induced by $M$ we associate an element $P \in V(\mathfrak{b}) \otimes V(\mathfrak{b})^{\times}$.
\item To each edge $\mathfrak{b} \in G_1 \cap \Gamma_1$ with an orientation that matches the induced orientation from $M$ we associate an element $Q \in V(\mathfrak{b}) \otimes V(\mathfrak{b})^{\times}$.
\item To each vertex in $G_0$ that coincides with a node in $\Gamma_0$ and with ordered edges $(\mathfrak{b}_1,\mathfrak{b}_2,\mathfrak{b}_3)$ we associate a map $\phi \colon W_1 \times W_2 \times W_3 \to k$ where the spaces $W_i=A,V(\mathfrak{b}_i),V(\mathfrak{b}_i)^{\times}$ are determined by the vertex edges.
\item Consider a vertex $v$ in $G_0$ that does not coincide with a node in $\Gamma_0$. If $v$ has ordered edges $(\mathfrak{b},a,\mathfrak{b})$, where $a$ denotes an $M-\Gamma$ edge, we associate with $v$ the map $L$ \eqref{eq:L} if $\mathfrak{b}$ comes with an orientation that matches the induced orientation from $M$. Otherwise, we associate with $v$ the map $R$ \eqref{eq:R}.
\end{enumerate}
The collection of associations above are referred to as defect data and denoted as $(L,R,P,Q,\phi)$. 
\end{definition}

We are interested in working with a specific subset of the collection defined above -- defect data restricted to having non-degenerate $P$ and $Q$. The maps $P$ and $Q$ are said to be non-degenerate if there exist 
\begin{align}
P^{-1} \colon & V^{\times} \otimes V \to k, \hspace{5mm} w_{\alpha} \otimes v_{\beta} \mapsto P_{\alpha\beta}, \label{eq:H_inverse}\\
Q^{-1} \colon & V \otimes V^{\times} \to k,  \hspace{5mm} v_{\alpha} \otimes w_{\beta} \mapsto Q_{\alpha\beta}. \label{eq:F_inverse}
\end{align}
such that $P_{\alpha\rho}P^{\rho\beta}=(\delta_{V^{\times}})_{\alpha}^{\beta}=Q^{\beta\rho}Q_{\rho\alpha}$ and $Q_{\alpha\rho}Q^{\rho\beta}=(\delta_{V})_{\alpha}^{\beta}=P^{\beta\rho}P_{\rho\alpha}$. These new maps are introduced diagrammatically below 
\begin{align}
\begin{aligned}
%% Creator: Inkscape 0.48.2, www.inkscape.org
%% PDF/EPS/PS + LaTeX output extension by Johan Engelen, 2010
%% Accompanies image file 'diagHF_inverse.pdf' (pdf, eps, ps)
%%
%% To include the image in your LaTeX document, write
%%   \input{<filename>.pdf_tex}
%%  instead of
%%   \includegraphics{<filename>.pdf}
%% To scale the image, write
%%   \def\svgwidth{<desired width>}
%%   \input{<filename>.pdf_tex}
%%  instead of
%%   \includegraphics[width=<desired width>]{<filename>.pdf}
%%
%% Images with a different path to the parent latex file can
%% be accessed with the `import' package (which may need to be
%% installed) using
%%   \usepackage{import}
%% in the preamble, and then including the image with
%%   \import{<path to file>}{<filename>.pdf_tex}
%% Alternatively, one can specify
%%   \graphicspath{{<path to file>/}}
%% 
%% For more information, please see info/svg-inkscape on CTAN:
%%   http://tug.ctan.org/tex-archive/info/svg-inkscape
%%
\begingroup%
  \makeatletter%
  \providecommand\color[2][]{%
    \errmessage{(Inkscape) Color is used for the text in Inkscape, but the package 'color.sty' is not loaded}%
    \renewcommand\color[2][]{}%
  }%
  \providecommand\transparent[1]{%
    \errmessage{(Inkscape) Transparency is used (non-zero) for the text in Inkscape, but the package 'transparent.sty' is not loaded}%
    \renewcommand\transparent[1]{}%
  }%
  \providecommand\rotatebox[2]{#2}%
  \ifx\svgwidth\undefined%
    \setlength{\unitlength}{262.68432617bp}%
    \ifx\svgscale\undefined%
      \relax%
    \else%
      \setlength{\unitlength}{\unitlength * \real{\svgscale}}%
    \fi%
  \else%
    \setlength{\unitlength}{\svgwidth}%
  \fi%
  \global\let\svgwidth\undefined%
  \global\let\svgscale\undefined%
  \makeatother%
  \begin{picture}(1,0.14844861)%
    \put(0,0){\includegraphics[width=\unitlength]{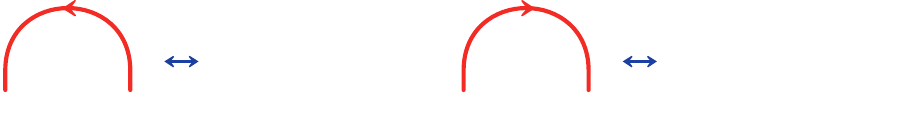}}%
    \put(-0.00108555,0.00605239){\color[rgb]{0,0,0}\makebox(0,0)[lb]{\smash{$\alpha$}}}%
    \put(0.13259309,0.007935){\color[rgb]{0,0,0}\makebox(0,0)[lb]{\smash{$\beta$}}}%
    \put(0.25327714,0.0711579){\color[rgb]{0,0,0}\makebox(0,0)[lb]{\smash{$P_{\alpha\beta},$}}}%
    \put(0.50141893,0.0060523){\color[rgb]{0,0,0}\makebox(0,0)[lb]{\smash{$\alpha$}}}%
    \put(0.63509756,0.0079349){\color[rgb]{0,0,0}\makebox(0,0)[lb]{\smash{$\beta$}}}%
    \put(0.75578161,0.07115781){\color[rgb]{0,0,0}\makebox(0,0)[lb]{\smash{$Q_{\alpha\beta}$}}}%
  \end{picture}%
\endgroup%

\end{aligned}
\end{align}
and their relation to $P$ and $Q$ is encoded through snake identities. 
\begin{align}
\begin{aligned}
%% Creator: Inkscape 0.48.2, www.inkscape.org
%% PDF/EPS/PS + LaTeX output extension by Johan Engelen, 2010
%% Accompanies image file 'InverseHFsnake.pdf' (pdf, eps, ps)
%%
%% To include the image in your LaTeX document, write
%%   \input{<filename>.pdf_tex}
%%  instead of
%%   \includegraphics{<filename>.pdf}
%% To scale the image, write
%%   \def\svgwidth{<desired width>}
%%   \input{<filename>.pdf_tex}
%%  instead of
%%   \includegraphics[width=<desired width>]{<filename>.pdf}
%%
%% Images with a different path to the parent latex file can
%% be accessed with the `import' package (which may need to be
%% installed) using
%%   \usepackage{import}
%% in the preamble, and then including the image with
%%   \import{<path to file>}{<filename>.pdf_tex}
%% Alternatively, one can specify
%%   \graphicspath{{<path to file>/}}
%% 
%% For more information, please see info/svg-inkscape on CTAN:
%%   http://tug.ctan.org/tex-archive/info/svg-inkscape
%%
\begingroup%
  \makeatletter%
  \providecommand\color[2][]{%
    \errmessage{(Inkscape) Color is used for the text in Inkscape, but the package 'color.sty' is not loaded}%
    \renewcommand\color[2][]{}%
  }%
  \providecommand\transparent[1]{%
    \errmessage{(Inkscape) Transparency is used (non-zero) for the text in Inkscape, but the package 'transparent.sty' is not loaded}%
    \renewcommand\transparent[1]{}%
  }%
  \providecommand\rotatebox[2]{#2}%
  \ifx\svgwidth\undefined%
    \setlength{\unitlength}{355.72545166bp}%
    \ifx\svgscale\undefined%
      \relax%
    \else%
      \setlength{\unitlength}{\unitlength * \real{\svgscale}}%
    \fi%
  \else%
    \setlength{\unitlength}{\svgwidth}%
  \fi%
  \global\let\svgwidth\undefined%
  \global\let\svgscale\undefined%
  \makeatother%
  \begin{picture}(1,0.15582881)%
    \put(0,0){\includegraphics[width=\unitlength]{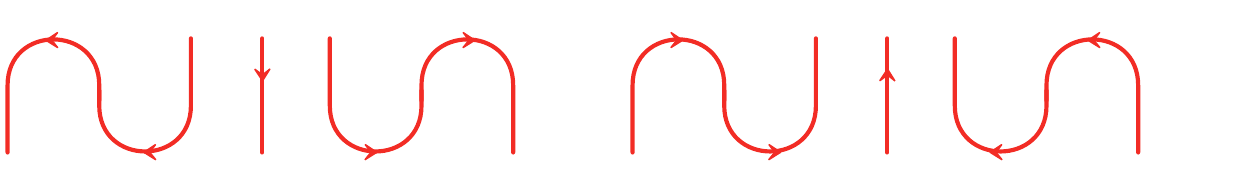}}%
    \put(-0.00080162,0.0044693){\color[rgb]{0,0,0}\makebox(0,0)[lb]{\smash{$\alpha$}}}%
    \put(0.14686966,0.13755588){\color[rgb]{0,0,0}\makebox(0,0)[lb]{\smash{$\beta$}}}%
    \put(0.2038417,0.1380116){\color[rgb]{0,0,0}\makebox(0,0)[lb]{\smash{$\beta$}}}%
    \put(0.20520905,0.00492508){\color[rgb]{0,0,0}\makebox(0,0)[lb]{\smash{$\alpha$}}}%
    \put(0.17376047,0.06508754){\color[rgb]{0,0,0}\makebox(0,0)[lb]{\smash{$=$}}}%
    \put(0.2325619,0.06573213){\color[rgb]{0,0,0}\makebox(0,0)[lb]{\smash{$=$}}}%
    \put(0.40870173,0.00556151){\color[rgb]{0,0,0}\makebox(0,0)[lb]{\smash{$\alpha$}}}%
    \put(0.26027266,0.1382481){\color[rgb]{0,0,0}\makebox(0,0)[lb]{\smash{$\beta$}}}%
    \put(0.50520656,0.0044693){\color[rgb]{0,0,0}\makebox(0,0)[lb]{\smash{$\alpha$}}}%
    \put(0.65287787,0.13755588){\color[rgb]{0,0,0}\makebox(0,0)[lb]{\smash{$\beta$}}}%
    \put(0.70984992,0.1380116){\color[rgb]{0,0,0}\makebox(0,0)[lb]{\smash{$\beta$}}}%
    \put(0.71121734,0.00492508){\color[rgb]{0,0,0}\makebox(0,0)[lb]{\smash{$\alpha$}}}%
    \put(0.67976875,0.06508754){\color[rgb]{0,0,0}\makebox(0,0)[lb]{\smash{$=$}}}%
    \put(0.73857011,0.06573213){\color[rgb]{0,0,0}\makebox(0,0)[lb]{\smash{$=$}}}%
    \put(0.91470994,0.00556151){\color[rgb]{0,0,0}\makebox(0,0)[lb]{\smash{$\alpha$}}}%
    \put(0.76628087,0.1382481){\color[rgb]{0,0,0}\makebox(0,0)[lb]{\smash{$\beta$}}}%
    \put(0.43568884,0.06628249){\color[rgb]{0,0,0}\makebox(0,0)[lb]{\smash{$,$}}}%
  \end{picture}%
\endgroup%

\end{aligned}
\end{align}
One important consequence of the existence of $P^{-1}$ and $Q^{-1}$ is that we can express more easily the relation between the actions $\gls{l}$ and $r^{\times}$, and $\gls{r}$ and $l^{\times}$ that are a consequence of \eqref{eq:defects_cyclicity}. In this sense we will regard the vector spaces $\gls{V}$ and $V^{\times}$ as dual (non-degeneracy implies the dimensions of $V$ and $V^{\times}$ match).
\begin{align} \label{eq:left_right_action_relations}
\begin{aligned}
%% Creator: Inkscape 0.48.2, www.inkscape.org
%% PDF/EPS/PS + LaTeX output extension by Johan Engelen, 2010
%% Accompanies image file 'diag_left_op_right.pdf' (pdf, eps, ps)
%%
%% To include the image in your LaTeX document, write
%%   \input{<filename>.pdf_tex}
%%  instead of
%%   \includegraphics{<filename>.pdf}
%% To scale the image, write
%%   \def\svgwidth{<desired width>}
%%   \input{<filename>.pdf_tex}
%%  instead of
%%   \includegraphics[width=<desired width>]{<filename>.pdf}
%%
%% Images with a different path to the parent latex file can
%% be accessed with the `import' package (which may need to be
%% installed) using
%%   \usepackage{import}
%% in the preamble, and then including the image with
%%   \import{<path to file>}{<filename>.pdf_tex}
%% Alternatively, one can specify
%%   \graphicspath{{<path to file>/}}
%% 
%% For more information, please see info/svg-inkscape on CTAN:
%%   http://tug.ctan.org/tex-archive/info/svg-inkscape
%%
\begingroup%
  \makeatletter%
  \providecommand\color[2][]{%
    \errmessage{(Inkscape) Color is used for the text in Inkscape, but the package 'color.sty' is not loaded}%
    \renewcommand\color[2][]{}%
  }%
  \providecommand\transparent[1]{%
    \errmessage{(Inkscape) Transparency is used (non-zero) for the text in Inkscape, but the package 'transparent.sty' is not loaded}%
    \renewcommand\transparent[1]{}%
  }%
  \providecommand\rotatebox[2]{#2}%
  \ifx\svgwidth\undefined%
    \setlength{\unitlength}{358.68919082bp}%
    \ifx\svgscale\undefined%
      \relax%
    \else%
      \setlength{\unitlength}{\unitlength * \real{\svgscale}}%
    \fi%
  \else%
    \setlength{\unitlength}{\svgwidth}%
  \fi%
  \global\let\svgwidth\undefined%
  \global\let\svgscale\undefined%
  \makeatother%
  \begin{picture}(1,0.23365143)%
    \put(0,0){\includegraphics[width=\unitlength]{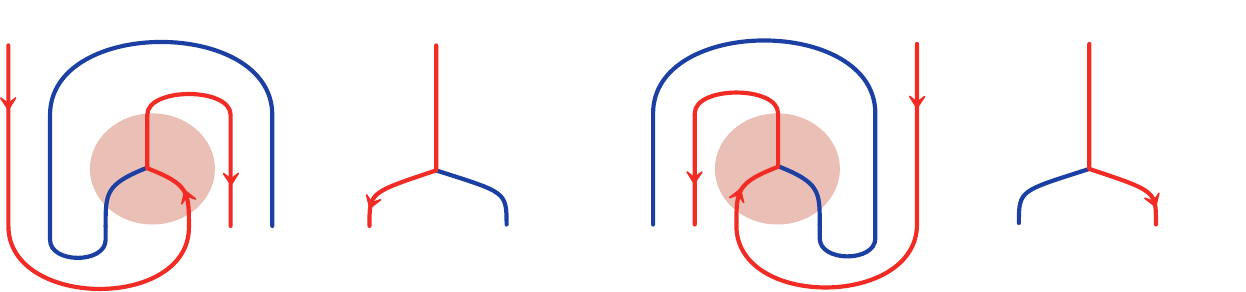}}%
    \put(0.73015298,0.2139996){\color[rgb]{0,0,0}\makebox(0,0)[lb]{\smash{$\beta$}}}%
    \put(0.80828836,0.0201154){\color[rgb]{0,0,0}\makebox(0,0)[lb]{\smash{$a$}}}%
    \put(0.91541467,0.02056735){\color[rgb]{0,0,0}\makebox(0,0)[lb]{\smash{$\alpha$}}}%
    \put(0.86620391,0.21621599){\color[rgb]{0,0,0}\makebox(0,0)[lb]{\smash{$\beta$}}}%
    \put(0.77475984,0.10915951){\color[rgb]{0,0,0}\makebox(0,0)[lb]{\smash{$=$}}}%
    \put(0.55172552,0.01994579){\color[rgb]{0,0,0}\makebox(0,0)[lb]{\smash{$\alpha$}}}%
    \put(0.51827037,0.01994579){\color[rgb]{0,0,0}\makebox(0,0)[lb]{\smash{$a$}}}%
    \put(0.64093925,0.10915951){\color[rgb]{0,0,0}\makebox(0,0)[lb]{\smash{$r$}}}%
    \put(0.886277,0.10915951){\color[rgb]{0,0,0}\makebox(0,0)[lb]{\smash{$l^{\times}$}}}%
    \put(0.17847114,0.01876527){\color[rgb]{0,0,0}\makebox(0,0)[lb]{\smash{$\alpha$}}}%
    \put(0.21192629,0.01876527){\color[rgb]{0,0,0}\makebox(0,0)[lb]{\smash{$a$}}}%
    \put(0.00004371,0.21503547){\color[rgb]{0,0,0}\makebox(0,0)[lb]{\smash{$\beta$}}}%
    \put(0.28560099,0.01893489){\color[rgb]{0,0,0}\makebox(0,0)[lb]{\smash{$\alpha$}}}%
    \put(0.39272733,0.01938684){\color[rgb]{0,0,0}\makebox(0,0)[lb]{\smash{$a$}}}%
    \put(0.34351657,0.21503547){\color[rgb]{0,0,0}\makebox(0,0)[lb]{\smash{$\beta$}}}%
    \put(0.25207246,0.107979){\color[rgb]{0,0,0}\makebox(0,0)[lb]{\smash{$=$}}}%
    \put(0.45136007,0.05340093){\color[rgb]{0,0,0}\makebox(0,0)[lb]{\smash{$,$}}}%
    \put(0.12796031,0.10915951){\color[rgb]{0,0,0}\makebox(0,0)[lb]{\smash{$l$}}}%
    \put(0.36214635,0.10915951){\color[rgb]{0,0,0}\makebox(0,0)[lb]{\smash{$r^{\times}$}}}%
  \end{picture}%
\endgroup%

\end{aligned}
\end{align}

Now the Pachner moves are introduced. Note that only transformations between graph-like triangulations are allowed. A Pachner move preserves the boundary of a graph-like triangulation and it is assumed that the corresponding dual edges do not change in a neighbourhood of the boundary, as before. A diagrammatic model with defects is called planar if it is invariant under Pachner moves. Recall that $\gls{sigma} \colon A \to A$ is the Nakayama automorphism associated with a Frobenius algebra $A$. 
\begin{theorem}
A planar state sum model $(\gls{Ct},\gls{Bb},\gls{R})$ together with defect data such that $P,Q$ are non-degenerate determine a planar state sum model with defects if and only if the following conditions are satisfied. 
\begin{enumerate}[label=H\arabic{*}), ref=(H\arabic{*})]
\item \label{it:H_one}The vector spaces $\gls{V}$, $V^{\times}$ are $\gls{A}$-$A$ bimodules.
\item \label{it:H_F} The non-degenerate maps $P^{-1} \colon V \otimes V^{\times} \to k$ and $Q^{-1} \colon V^{\times} \otimes V$ satisfy $P^{-1}(\gls{sigma}(a) \cdot v \cdot b, w)= P^{-1}(v , b \cdot w \cdot a)$ and $Q^{-1}(\sigma(a) \cdot w \cdot b, \cdot v)= Q^{-1}(w , b \cdot v \cdot a)$ for all $a,b \in A$, $v \in V$, $w \in V^{\times}$.
\item \label{it:phi} Each map $\phi$ satisfies $\phi(\sigma(a) \cdot e_I \cdot b , e_J \cdot c , e_K)=\phi(e_I , b \cdot e_J , c \cdot e_K \cdot a)$ for all $a,b,c \in A$.
\end{enumerate} 
\end{theorem}
\begin{proof}
Not all Pachner moves will be explicitly verified. As an exercise, the reader can confirm that all the omitted moves are either entirely analogous to the ones included, or automatically verified as a consequence of the presented moves being satisfied. 

One starts by showing that defect data must satisfy the conditions outlined above. Proving each $V$ must be an $A$-$A$ bimodule will consist of showing that the left and right actions $\gls{l}$ and $\gls{r}$ must be compatible with each other and the algebra multiplication $\gls{m}$. The Pachner move below is equivalent to the identity $L_{a\alpha}{}^{\gamma}R_{\gamma b}{}^{\beta}=R_{\alpha b}{}^{\gamma}L_{a\gamma}{}^{\beta}$.
\begin{align}
\begin{aligned}
%% Creator: Inkscape 0.48.2, www.inkscape.org
%% PDF/EPS/PS + LaTeX output extension by Johan Engelen, 2010
%% Accompanies image file 'defects_22_associativity.pdf' (pdf, eps, ps)
%%
%% To include the image in your LaTeX document, write
%%   \input{<filename>.pdf_tex}
%%  instead of
%%   \includegraphics{<filename>.pdf}
%% To scale the image, write
%%   \def\svgwidth{<desired width>}
%%   \input{<filename>.pdf_tex}
%%  instead of
%%   \includegraphics[width=<desired width>]{<filename>.pdf}
%%
%% Images with a different path to the parent latex file can
%% be accessed with the `import' package (which may need to be
%% installed) using
%%   \usepackage{import}
%% in the preamble, and then including the image with
%%   \import{<path to file>}{<filename>.pdf_tex}
%% Alternatively, one can specify
%%   \graphicspath{{<path to file>/}}
%% 
%% For more information, please see info/svg-inkscape on CTAN:
%%   http://tug.ctan.org/tex-archive/info/svg-inkscape
%%
\begingroup%
  \makeatletter%
  \providecommand\color[2][]{%
    \errmessage{(Inkscape) Color is used for the text in Inkscape, but the package 'color.sty' is not loaded}%
    \renewcommand\color[2][]{}%
  }%
  \providecommand\transparent[1]{%
    \errmessage{(Inkscape) Transparency is used (non-zero) for the text in Inkscape, but the package 'transparent.sty' is not loaded}%
    \renewcommand\transparent[1]{}%
  }%
  \providecommand\rotatebox[2]{#2}%
  \ifx\svgwidth\undefined%
    \setlength{\unitlength}{366.83658295bp}%
    \ifx\svgscale\undefined%
      \relax%
    \else%
      \setlength{\unitlength}{\unitlength * \real{\svgscale}}%
    \fi%
  \else%
    \setlength{\unitlength}{\svgwidth}%
  \fi%
  \global\let\svgwidth\undefined%
  \global\let\svgscale\undefined%
  \makeatother%
  \begin{picture}(1,0.18210029)%
    \put(0,0){\includegraphics[width=\unitlength]{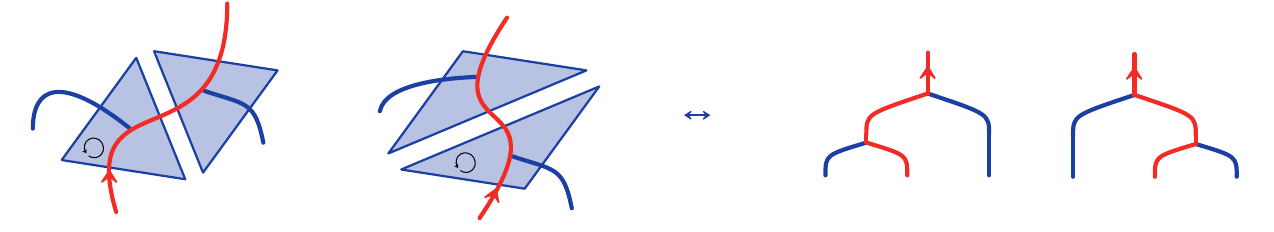}}%
    \put(0.09889674,0.01456206){\color[rgb]{0,0,0}\makebox(0,0)[lb]{\smash{$\alpha$}}}%
    \put(-0.00077734,0.08448999){\color[rgb]{0,0,0}\makebox(0,0)[lb]{\smash{$a$}}}%
    \put(0.18464448,0.16505209){\color[rgb]{0,0,0}\makebox(0,0)[lb]{\smash{$\beta$}}}%
    \put(0.2134798,0.07122319){\color[rgb]{0,0,0}\makebox(0,0)[lb]{\smash{$b$}}}%
    \put(0.34702004,0.00433393){\color[rgb]{0,0,0}\makebox(0,0)[lb]{\smash{$\alpha$}}}%
    \put(0.45717039,0.01959073){\color[rgb]{0,0,0}\makebox(0,0)[lb]{\smash{$b$}}}%
    \put(0.40394545,0.15703372){\color[rgb]{0,0,0}\makebox(0,0)[lb]{\smash{$\beta$}}}%
    \put(0.29160364,0.06835915){\color[rgb]{0,0,0}\makebox(0,0)[lb]{\smash{$a$}}}%
    \put(0.24616747,0.09323105){\color[rgb]{0,0,0}\makebox(0,0)[lb]{\smash{$=$}}}%
    \put(0.64197639,0.00939987){\color[rgb]{0,0,0}\makebox(0,0)[lb]{\smash{$a$}}}%
    \put(0.70694613,0.00895796){\color[rgb]{0,0,0}\makebox(0,0)[lb]{\smash{$\alpha$}}}%
    \put(0.77058986,0.00939987){\color[rgb]{0,0,0}\makebox(0,0)[lb]{\smash{$b$}}}%
    \put(0.83644356,0.00895796){\color[rgb]{0,0,0}\makebox(0,0)[lb]{\smash{$a$}}}%
    \put(0.90075036,0.00939987){\color[rgb]{0,0,0}\makebox(0,0)[lb]{\smash{$\alpha$}}}%
    \put(0.96461512,0.00962096){\color[rgb]{0,0,0}\makebox(0,0)[lb]{\smash{$b$}}}%
    \put(0.72198231,0.15912461){\color[rgb]{0,0,0}\makebox(0,0)[lb]{\smash{$\beta$}}}%
    \put(0.88421588,0.15860244){\color[rgb]{0,0,0}\makebox(0,0)[lb]{\smash{$\beta$}}}%
    \put(0.79678651,0.06140849){\color[rgb]{0,0,0}\makebox(0,0)[lb]{\smash{$=$}}}%
  \end{picture}%
\endgroup%

\end{aligned}
\end{align}
The definitions of $\gls{r}$ and $\gls{l}$ allow one to conclude that $L_{a\alpha}{}^{\gamma}R_{\gamma b}{}^{\beta}v_{\beta}=(e_a \cdot v_{\alpha}) \cdot e_b$ and $R_{\alpha b}{}^{\gamma}L_{a\gamma}{}^{\beta}v_{\beta}= e_a \cdot (v_{\alpha} \cdot e_b)$. Linear independence of the basis elements $v_{\beta}$ therefore imply that for the Pachner move to be respected one must have $(e_a \cdot v_{\alpha}) \cdot e_{b}= e_a \cdot (v_{\alpha} \cdot e_{b})$. By inverting the defect arrow and using the actions $l^{\times}$ and $r^{\times}$ one would also conclude that $(e_a \cdot w_{\alpha}) \cdot e_{b}= e_a \cdot (w_{\alpha} \cdot e_b)$. 
\begin{align}
\begin{aligned}
%% Creator: Inkscape 0.48.2, www.inkscape.org
%% PDF/EPS/PS + LaTeX output extension by Johan Engelen, 2010
%% Accompanies image file 'defects_22_right.pdf' (pdf, eps, ps)
%%
%% To include the image in your LaTeX document, write
%%   \input{<filename>.pdf_tex}
%%  instead of
%%   \includegraphics{<filename>.pdf}
%% To scale the image, write
%%   \def\svgwidth{<desired width>}
%%   \input{<filename>.pdf_tex}
%%  instead of
%%   \includegraphics[width=<desired width>]{<filename>.pdf}
%%
%% Images with a different path to the parent latex file can
%% be accessed with the `import' package (which may need to be
%% installed) using
%%   \usepackage{import}
%% in the preamble, and then including the image with
%%   \import{<path to file>}{<filename>.pdf_tex}
%% Alternatively, one can specify
%%   \graphicspath{{<path to file>/}}
%% 
%% For more information, please see info/svg-inkscape on CTAN:
%%   http://tug.ctan.org/tex-archive/info/svg-inkscape
%%
\begingroup%
  \makeatletter%
  \providecommand\color[2][]{%
    \errmessage{(Inkscape) Color is used for the text in Inkscape, but the package 'color.sty' is not loaded}%
    \renewcommand\color[2][]{}%
  }%
  \providecommand\transparent[1]{%
    \errmessage{(Inkscape) Transparency is used (non-zero) for the text in Inkscape, but the package 'transparent.sty' is not loaded}%
    \renewcommand\transparent[1]{}%
  }%
  \providecommand\rotatebox[2]{#2}%
  \ifx\svgwidth\undefined%
    \setlength{\unitlength}{364.42905884bp}%
    \ifx\svgscale\undefined%
      \relax%
    \else%
      \setlength{\unitlength}{\unitlength * \real{\svgscale}}%
    \fi%
  \else%
    \setlength{\unitlength}{\svgwidth}%
  \fi%
  \global\let\svgwidth\undefined%
  \global\let\svgscale\undefined%
  \makeatother%
  \begin{picture}(1,0.19416522)%
    \put(0,0){\includegraphics[width=\unitlength]{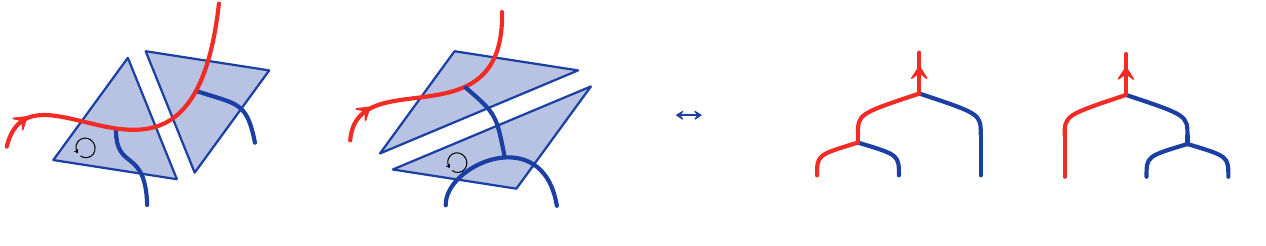}}%
    \put(0.11024598,0.00475755){\color[rgb]{0,0,0}\makebox(0,0)[lb]{\smash{$a$}}}%
    \put(-0.00078247,0.04935146){\color[rgb]{0,0,0}\makebox(0,0)[lb]{\smash{$\alpha$}}}%
    \put(0.179258,0.17700439){\color[rgb]{0,0,0}\makebox(0,0)[lb]{\smash{$\beta$}}}%
    \put(0.19601499,0.05235548){\color[rgb]{0,0,0}\makebox(0,0)[lb]{\smash{$b$}}}%
    \put(0.27063398,0.05259531){\color[rgb]{0,0,0}\makebox(0,0)[lb]{\smash{$\alpha$}}}%
    \put(0.43445401,0.00299886){\color[rgb]{0,0,0}\makebox(0,0)[lb]{\smash{$b$}}}%
    \put(0.40000774,0.16893305){\color[rgb]{0,0,0}\makebox(0,0)[lb]{\smash{$\beta$}}}%
    \put(0.34698401,0.00226167){\color[rgb]{0,0,0}\makebox(0,0)[lb]{\smash{$a$}}}%
    \put(0.24118743,0.10470888){\color[rgb]{0,0,0}\makebox(0,0)[lb]{\smash{$=$}}}%
    \put(0.63961118,0.02032389){\color[rgb]{0,0,0}\makebox(0,0)[lb]{\smash{$\alpha$}}}%
    \put(0.70501013,0.01987906){\color[rgb]{0,0,0}\makebox(0,0)[lb]{\smash{$a$}}}%
    \put(0.76907431,0.02032389){\color[rgb]{0,0,0}\makebox(0,0)[lb]{\smash{$b$}}}%
    \put(0.83536306,0.01987906){\color[rgb]{0,0,0}\makebox(0,0)[lb]{\smash{$\alpha$}}}%
    \put(0.90009469,0.02032389){\color[rgb]{0,0,0}\makebox(0,0)[lb]{\smash{$a$}}}%
    \put(0.96438136,0.02054644){\color[rgb]{0,0,0}\makebox(0,0)[lb]{\smash{$b$}}}%
    \put(0.72014564,0.17103775){\color[rgb]{0,0,0}\makebox(0,0)[lb]{\smash{$\beta$}}}%
    \put(0.88345097,0.17051213){\color[rgb]{0,0,0}\makebox(0,0)[lb]{\smash{$\beta$}}}%
    \put(0.79544402,0.07267609){\color[rgb]{0,0,0}\makebox(0,0)[lb]{\smash{$=$}}}%
  \end{picture}%
\endgroup%

\end{aligned}
\end{align}
The next move, depicted above, indicates that $R_{\alpha a}{}^{\mu}R_{\mu b}{}^{\beta} = C_{ab}{}^c R_{\alpha c}{}^{\beta}$. Note that for the $r$ action one has $(v_{\alpha} \cdot e_a) \cdot e_b = R_{\alpha a}{}^{\mu}R_{\mu b}{}^{\beta} v_{\beta}$ whilst $v_{\alpha} \cdot (e_a \cdot e_b) = C_{ab}{}^c R_{\alpha c}{}^{\beta} v_{\beta}$. Again, the linear independence of the $v_{\beta}$ allows one to conclude the move imposes the identity $(v_{\alpha} \cdot e_a) \cdot e_b=v_{\alpha} \cdot (e_a \cdot e_b)$. Furthermore, the move below tells us a similar conclusion holds for the left action $l$: $(e_a \cdot e_b) \cdot v_{\alpha}=e_a \cdot (e_b \cdot v_{\alpha})$. By inverting the direction of the arrow, the same conclusions will hold for the actions $l^{\times}$ and $r^{\times}$.
\begin{align}
\begin{aligned}
%% Creator: Inkscape 0.48.2, www.inkscape.org
%% PDF/EPS/PS + LaTeX output extension by Johan Engelen, 2010
%% Accompanies image file 'defects_22_left.pdf' (pdf, eps, ps)
%%
%% To include the image in your LaTeX document, write
%%   \input{<filename>.pdf_tex}
%%  instead of
%%   \includegraphics{<filename>.pdf}
%% To scale the image, write
%%   \def\svgwidth{<desired width>}
%%   \input{<filename>.pdf_tex}
%%  instead of
%%   \includegraphics[width=<desired width>]{<filename>.pdf}
%%
%% Images with a different path to the parent latex file can
%% be accessed with the `import' package (which may need to be
%% installed) using
%%   \usepackage{import}
%% in the preamble, and then including the image with
%%   \import{<path to file>}{<filename>.pdf_tex}
%% Alternatively, one can specify
%%   \graphicspath{{<path to file>/}}
%% 
%% For more information, please see info/svg-inkscape on CTAN:
%%   http://tug.ctan.org/tex-archive/info/svg-inkscape
%%
\begingroup%
  \makeatletter%
  \providecommand\color[2][]{%
    \errmessage{(Inkscape) Color is used for the text in Inkscape, but the package 'color.sty' is not loaded}%
    \renewcommand\color[2][]{}%
  }%
  \providecommand\transparent[1]{%
    \errmessage{(Inkscape) Transparency is used (non-zero) for the text in Inkscape, but the package 'transparent.sty' is not loaded}%
    \renewcommand\transparent[1]{}%
  }%
  \providecommand\rotatebox[2]{#2}%
  \ifx\svgwidth\undefined%
    \setlength{\unitlength}{353.04507599bp}%
    \ifx\svgscale\undefined%
      \relax%
    \else%
      \setlength{\unitlength}{\unitlength * \real{\svgscale}}%
    \fi%
  \else%
    \setlength{\unitlength}{\svgwidth}%
  \fi%
  \global\let\svgwidth\undefined%
  \global\let\svgscale\undefined%
  \makeatother%
  \begin{picture}(1,0.20930913)%
    \put(0,0){\includegraphics[width=\unitlength]{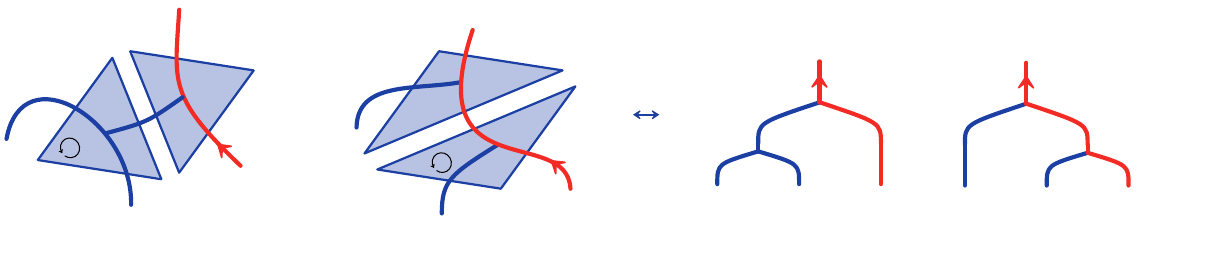}}%
    \put(-0.0008077,0.06568764){\color[rgb]{0,0,0}\makebox(0,0)[lb]{\smash{$a$}}}%
    \put(0.19089425,0.04421182){\color[rgb]{0,0,0}\makebox(0,0)[lb]{\smash{$\alpha$}}}%
    \put(0.17238006,0.19159494){\color[rgb]{0,0,0}\makebox(0,0)[lb]{\smash{$\beta$}}}%
    \put(0.10058555,0.01011457){\color[rgb]{0,0,0}\makebox(0,0)[lb]{\smash{$b$}}}%
    \put(0.45958167,0.02230227){\color[rgb]{0,0,0}\makebox(0,0)[lb]{\smash{$\alpha$}}}%
    \put(0.35406087,0.0023346){\color[rgb]{0,0,0}\makebox(0,0)[lb]{\smash{$b$}}}%
    \put(0.4002479,0.18326334){\color[rgb]{0,0,0}\makebox(0,0)[lb]{\smash{$\beta$}}}%
    \put(0.27662897,0.08102152){\color[rgb]{0,0,0}\makebox(0,0)[lb]{\smash{$a$}}}%
    \put(0.23630641,0.11696826){\color[rgb]{0,0,0}\makebox(0,0)[lb]{\smash{$=$}}}%
    \put(0.57881998,0.02270903){\color[rgb]{0,0,0}\makebox(0,0)[lb]{\smash{$a$}}}%
    \put(0.64632772,0.02224985){\color[rgb]{0,0,0}\makebox(0,0)[lb]{\smash{$b$}}}%
    \put(0.71245767,0.02270903){\color[rgb]{0,0,0}\makebox(0,0)[lb]{\smash{$\alpha$}}}%
    \put(0.7808839,0.02224985){\color[rgb]{0,0,0}\makebox(0,0)[lb]{\smash{$a$}}}%
    \put(0.84770281,0.02270903){\color[rgb]{0,0,0}\makebox(0,0)[lb]{\smash{$b$}}}%
    \put(0.91406241,0.02293875){\color[rgb]{0,0,0}\makebox(0,0)[lb]{\smash{$\alpha$}}}%
    \put(0.66195128,0.17828268){\color[rgb]{0,0,0}\makebox(0,0)[lb]{\smash{$\beta$}}}%
    \put(0.83052241,0.1777401){\color[rgb]{0,0,0}\makebox(0,0)[lb]{\smash{$\beta$}}}%
    \put(0.73967767,0.07674933){\color[rgb]{0,0,0}\makebox(0,0)[lb]{\smash{$=$}}}%
  \end{picture}%
\endgroup%

\end{aligned}
\end{align}
A 1-3 Pachner move predicts the diagram below must equal $R_{\alpha a}{}^{\beta}$. Note that it evaluates to $\gls{R}.R_{\alpha a}{}^{\gamma}L_{b\gamma}{}^{\mu}B^{cb}L_{c\mu}{}^{\beta}$. Notice as well that $B^{cb} e_c \cdot e_b \cdot e_\alpha \cdot e_a = R_{\alpha a}{}^{\gamma}L_{b\gamma}{}^{\mu}B^{cb}L_{c\mu}{}^{\beta}v_{\beta}$. The order of the actions has been omitted since one now knows they must be associative. Recalling that $\gls{m}(\gls{B})=B^{cb} e_c \cdot e_b$ and that $R.m(B)=\gls{unit}$ one learns this 1-3 Pachner move is translated algebraically as $1 \cdot v_{\alpha}=v_{\alpha}$. Similar 1-3 Pachner moves indicate $v_{\alpha} \cdot 1 = v_{\alpha}$ and $1 \cdot w_{\alpha}=w_{\alpha}= w_{\alpha} \cdot 1$. This concludes the proof that if the model is to be planar the $\gls{V}, V^{\times}$ must be $\gls{A}$-$A$ bimodules.
\begin{align}
\begin{aligned}
%% Creator: Inkscape 0.48.2, www.inkscape.org
%% PDF/EPS/PS + LaTeX output extension by Johan Engelen, 2010
%% Accompanies image file 'defetc_31_left.pdf' (pdf, eps, ps)
%%
%% To include the image in your LaTeX document, write
%%   \input{<filename>.pdf_tex}
%%  instead of
%%   \includegraphics{<filename>.pdf}
%% To scale the image, write
%%   \def\svgwidth{<desired width>}
%%   \input{<filename>.pdf_tex}
%%  instead of
%%   \includegraphics[width=<desired width>]{<filename>.pdf}
%%
%% Images with a different path to the parent latex file can
%% be accessed with the `import' package (which may need to be
%% installed) using
%%   \usepackage{import}
%% in the preamble, and then including the image with
%%   \import{<path to file>}{<filename>.pdf_tex}
%% Alternatively, one can specify
%%   \graphicspath{{<path to file>/}}
%% 
%% For more information, please see info/svg-inkscape on CTAN:
%%   http://tug.ctan.org/tex-archive/info/svg-inkscape
%%
\begingroup%
  \makeatletter%
  \providecommand\color[2][]{%
    \errmessage{(Inkscape) Color is used for the text in Inkscape, but the package 'color.sty' is not loaded}%
    \renewcommand\color[2][]{}%
  }%
  \providecommand\transparent[1]{%
    \errmessage{(Inkscape) Transparency is used (non-zero) for the text in Inkscape, but the package 'transparent.sty' is not loaded}%
    \renewcommand\transparent[1]{}%
  }%
  \providecommand\rotatebox[2]{#2}%
  \ifx\svgwidth\undefined%
    \setlength{\unitlength}{348.2165648bp}%
    \ifx\svgscale\undefined%
      \relax%
    \else%
      \setlength{\unitlength}{\unitlength * \real{\svgscale}}%
    \fi%
  \else%
    \setlength{\unitlength}{\svgwidth}%
  \fi%
  \global\let\svgwidth\undefined%
  \global\let\svgscale\undefined%
  \makeatother%
  \begin{picture}(1,0.41352755)%
    \put(0,0){\includegraphics[width=\unitlength]{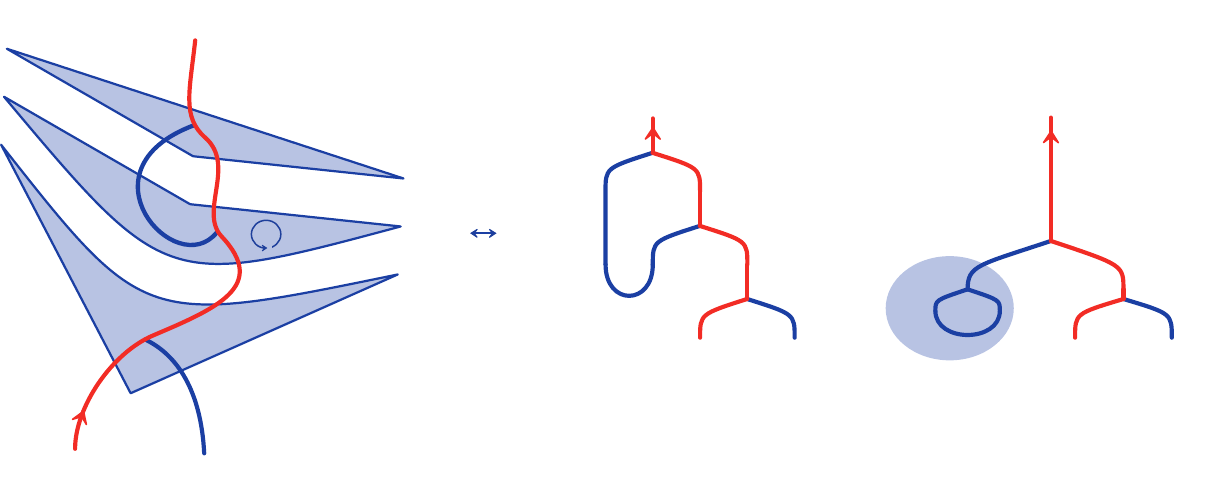}}%
    \put(0.16287955,0.00620012){\color[rgb]{0,0,0}\makebox(0,0)[lb]{\smash{$a$}}}%
    \put(0.05446135,0.00456568){\color[rgb]{0,0,0}\makebox(0,0)[lb]{\smash{$\alpha$}}}%
    \put(0.15501954,0.39556773){\color[rgb]{0,0,0}\makebox(0,0)[lb]{\smash{$\beta$}}}%
    \put(0.57180213,0.09716741){\color[rgb]{0,0,0}\makebox(0,0)[lb]{\smash{$\alpha$}}}%
    \put(0.65068819,0.09806863){\color[rgb]{0,0,0}\makebox(0,0)[lb]{\smash{$a$}}}%
    \put(0.53377157,0.33418429){\color[rgb]{0,0,0}\makebox(0,0)[lb]{\smash{$\beta$}}}%
    \put(0.45266093,0.21465201){\color[rgb]{0,0,0}\makebox(0,0)[lb]{\smash{$R$}}}%
    \put(0.86267456,0.33729641){\color[rgb]{0,0,0}\makebox(0,0)[lb]{\smash{$\beta$}}}%
    \put(0.88325971,0.09652898){\color[rgb]{0,0,0}\makebox(0,0)[lb]{\smash{$\alpha$}}}%
    \put(0.962723,0.09606329){\color[rgb]{0,0,0}\makebox(0,0)[lb]{\smash{$a$}}}%
    \put(0.68240302,0.21465201){\color[rgb]{0,0,0}\makebox(0,0)[lb]{\smash{$=$}}}%
    \put(0.73860212,0.15060972){\color[rgb]{0,0,0}\makebox(0,0)[lb]{\smash{$R$}}}%
  \end{picture}%
\endgroup%

\end{aligned}
\end{align}
Now one explores the properties of the map $P^{-1}$. The 2-2 Pachner move below is equivalent to the equation $R_{\alpha a}{}^{\mu}L_{\beta b}{}^{\nu} P_{\mu \nu}=(B^{\tr}B^{-1})_b{}^c L_{c\alpha}{}^{\mu}R_{a\beta}{}^{\nu}P_{\mu\nu}$. Recall that $(B^{\tr}\gls{Bb}^{-1})_b{}^c=\sigma_b{}^{c}$ where $\gls{sigma}$ stands for the Nakayama automorphism associated with $\gls{eps}$. Linear independence implies the equation is equivalent to $P^{-1}(v_{\alpha} \cdot e_a, w_{\beta} \cdot e_b)=P^{-1}(\sigma(e_b) \cdot v_{\alpha},e_a \cdot w_{\beta})$. By inverting the direction of the arrow one would obtain the equivalent conclusion for $Q^{-1}$.  
\begin{align}
\begin{aligned}
%% Creator: Inkscape 0.48.2, www.inkscape.org
%% PDF/EPS/PS + LaTeX output extension by Johan Engelen, 2010
%% Accompanies image file 'defects_22_H.pdf' (pdf, eps, ps)
%%
%% To include the image in your LaTeX document, write
%%   \input{<filename>.pdf_tex}
%%  instead of
%%   \includegraphics{<filename>.pdf}
%% To scale the image, write
%%   \def\svgwidth{<desired width>}
%%   \input{<filename>.pdf_tex}
%%  instead of
%%   \includegraphics[width=<desired width>]{<filename>.pdf}
%%
%% Images with a different path to the parent latex file can
%% be accessed with the `import' package (which may need to be
%% installed) using
%%   \usepackage{import}
%% in the preamble, and then including the image with
%%   \import{<path to file>}{<filename>.pdf_tex}
%% Alternatively, one can specify
%%   \graphicspath{{<path to file>/}}
%% 
%% For more information, please see info/svg-inkscape on CTAN:
%%   http://tug.ctan.org/tex-archive/info/svg-inkscape
%%
\begingroup%
  \makeatletter%
  \providecommand\color[2][]{%
    \errmessage{(Inkscape) Color is used for the text in Inkscape, but the package 'color.sty' is not loaded}%
    \renewcommand\color[2][]{}%
  }%
  \providecommand\transparent[1]{%
    \errmessage{(Inkscape) Transparency is used (non-zero) for the text in Inkscape, but the package 'transparent.sty' is not loaded}%
    \renewcommand\transparent[1]{}%
  }%
  \providecommand\rotatebox[2]{#2}%
  \ifx\svgwidth\undefined%
    \setlength{\unitlength}{394.25388947bp}%
    \ifx\svgscale\undefined%
      \relax%
    \else%
      \setlength{\unitlength}{\unitlength * \real{\svgscale}}%
    \fi%
  \else%
    \setlength{\unitlength}{\svgwidth}%
  \fi%
  \global\let\svgwidth\undefined%
  \global\let\svgscale\undefined%
  \makeatother%
  \begin{picture}(1,0.15648018)%
    \put(0,0){\includegraphics[width=\unitlength]{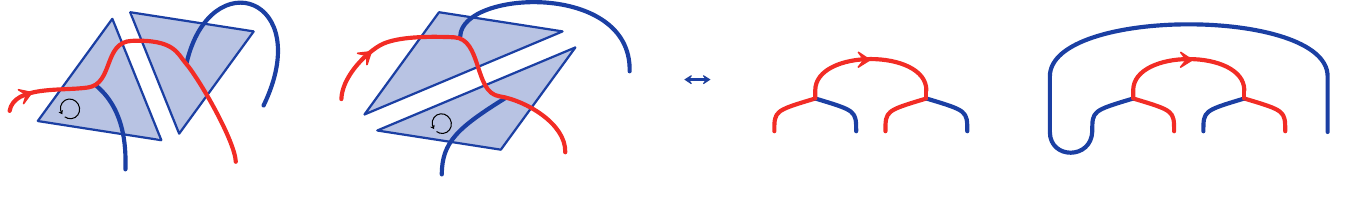}}%
    \put(-0.00072328,0.04320284){\color[rgb]{0,0,0}\makebox(0,0)[lb]{\smash{$\alpha$}}}%
    \put(0.16687027,0.00776177){\color[rgb]{0,0,0}\makebox(0,0)[lb]{\smash{$\beta$}}}%
    \put(0.18605809,0.04756029){\color[rgb]{0,0,0}\makebox(0,0)[lb]{\smash{$b$}}}%
    \put(0.08595963,0.00495299){\color[rgb]{0,0,0}\makebox(0,0)[lb]{\smash{$a$}}}%
    \put(0.24293813,0.05370044){\color[rgb]{0,0,0}\makebox(0,0)[lb]{\smash{$\alpha$}}}%
    \put(0.45358325,0.074476){\color[rgb]{0,0,0}\makebox(0,0)[lb]{\smash{$b$}}}%
    \put(0.40693833,0.01360393){\color[rgb]{0,0,0}\makebox(0,0)[lb]{\smash{$\beta$}}}%
    \put(0.31667825,0.00209058){\color[rgb]{0,0,0}\makebox(0,0)[lb]{\smash{$a$}}}%
    \put(0.21160683,0.10228292){\color[rgb]{0,0,0}\makebox(0,0)[lb]{\smash{$=$}}}%
    \put(0.62019754,0.02902102){\color[rgb]{0,0,0}\makebox(0,0)[lb]{\smash{$a$}}}%
    \put(0.7025703,0.02852296){\color[rgb]{0,0,0}\makebox(0,0)[lb]{\smash{$b$}}}%
    \put(0.55881868,0.02951574){\color[rgb]{0,0,0}\makebox(0,0)[lb]{\smash{$\alpha$}}}%
    \put(0.87307767,0.0261967){\color[rgb]{0,0,0}\makebox(0,0)[lb]{\smash{$a$}}}%
    \put(0.9631543,0.02660788){\color[rgb]{0,0,0}\makebox(0,0)[lb]{\smash{$b$}}}%
    \put(0.85101327,0.02623199){\color[rgb]{0,0,0}\makebox(0,0)[lb]{\smash{$\alpha$}}}%
    \put(0.64153779,0.02983199){\color[rgb]{0,0,0}\makebox(0,0)[lb]{\smash{$\beta$}}}%
    \put(0.93323029,0.02585667){\color[rgb]{0,0,0}\makebox(0,0)[lb]{\smash{$\beta$}}}%
    \put(0.72393457,0.07267337){\color[rgb]{0,0,0}\makebox(0,0)[lb]{\smash{$=$}}}%
  \end{picture}%
\endgroup%

\end{aligned}
\end{align}
Finally one establishes the properties that $\phi$ must obey. Recall that each leg stands for an element of $A$, $V$ or $V^{\times}$. Therefore, there is always an action of $A$ on each $e_I$. Denote such actions in general as $e_I \cdot e_a = M_{Ia}{}^J e_J$ and $e_a \cdot e_I = N_{aI}{}^Je_J$. Then the diagram below is translated as $M_{Ka}{}^{L}\phi_{IJL}=\sigma_a{}^bN_{bI}{}^L\phi_{LJK}$. In other words, one must have $\phi(e_I , e_J , e_K \cdot e_a)=\phi(\gls{sigma}(e_a) \cdot e_I , e_J , e_K)$.  
\begin{align}
\begin{aligned}
\label{eq:sigma_def}
%% Creator: Inkscape 0.48.2, www.inkscape.org
%% PDF/EPS/PS + LaTeX output extension by Johan Engelen, 2010
%% Accompanies image file 'defects_22_phi_in.pdf' (pdf, eps, ps)
%%
%% To include the image in your LaTeX document, write
%%   \input{<filename>.pdf_tex}
%%  instead of
%%   \includegraphics{<filename>.pdf}
%% To scale the image, write
%%   \def\svgwidth{<desired width>}
%%   \input{<filename>.pdf_tex}
%%  instead of
%%   \includegraphics[width=<desired width>]{<filename>.pdf}
%%
%% Images with a different path to the parent latex file can
%% be accessed with the `import' package (which may need to be
%% installed) using
%%   \usepackage{import}
%% in the preamble, and then including the image with
%%   \import{<path to file>}{<filename>.pdf_tex}
%% Alternatively, one can specify
%%   \graphicspath{{<path to file>/}}
%% 
%% For more information, please see info/svg-inkscape on CTAN:
%%   http://tug.ctan.org/tex-archive/info/svg-inkscape
%%
\begingroup%
  \makeatletter%
  \providecommand\color[2][]{%
    \errmessage{(Inkscape) Color is used for the text in Inkscape, but the package 'color.sty' is not loaded}%
    \renewcommand\color[2][]{}%
  }%
  \providecommand\transparent[1]{%
    \errmessage{(Inkscape) Transparency is used (non-zero) for the text in Inkscape, but the package 'transparent.sty' is not loaded}%
    \renewcommand\transparent[1]{}%
  }%
  \providecommand\rotatebox[2]{#2}%
  \ifx\svgwidth\undefined%
    \setlength{\unitlength}{393.68057404bp}%
    \ifx\svgscale\undefined%
      \relax%
    \else%
      \setlength{\unitlength}{\unitlength * \real{\svgscale}}%
    \fi%
  \else%
    \setlength{\unitlength}{\svgwidth}%
  \fi%
  \global\let\svgwidth\undefined%
  \global\let\svgscale\undefined%
  \makeatother%
  \begin{picture}(1,0.15670806)%
    \put(0,0){\includegraphics[width=\unitlength]{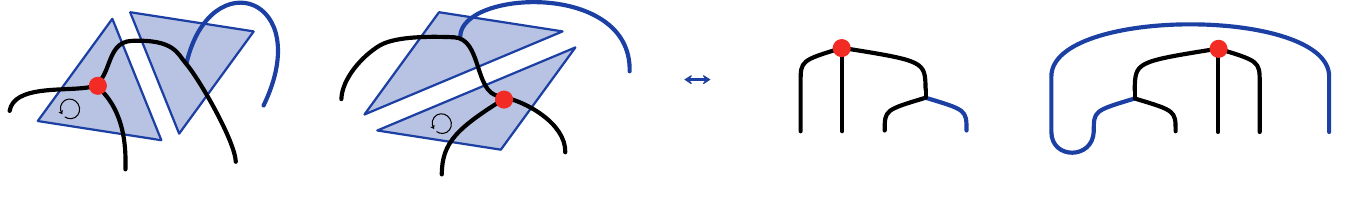}}%
    \put(-0.00072433,0.04326575){\color[rgb]{0,0,0}\makebox(0,0)[lb]{\smash{$I$}}}%
    \put(0.16711328,0.00777307){\color[rgb]{0,0,0}\makebox(0,0)[lb]{\smash{$K$}}}%
    \put(0.18632905,0.04762955){\color[rgb]{0,0,0}\makebox(0,0)[lb]{\smash{$a$}}}%
    \put(0.08608481,0.0049602){\color[rgb]{0,0,0}\makebox(0,0)[lb]{\smash{$J$}}}%
    \put(0.24329192,0.05377864){\color[rgb]{0,0,0}\makebox(0,0)[lb]{\smash{$I$}}}%
    \put(0.4542438,0.07458446){\color[rgb]{0,0,0}\makebox(0,0)[lb]{\smash{$a$}}}%
    \put(0.40753095,0.01362374){\color[rgb]{0,0,0}\makebox(0,0)[lb]{\smash{$K$}}}%
    \put(0.31713943,0.00209362){\color[rgb]{0,0,0}\makebox(0,0)[lb]{\smash{$J$}}}%
    \put(0.21191499,0.10243188){\color[rgb]{0,0,0}\makebox(0,0)[lb]{\smash{$=$}}}%
    \put(0.61121673,0.02947513){\color[rgb]{0,0,0}\makebox(0,0)[lb]{\smash{$J$}}}%
    \put(0.70112246,0.0285645){\color[rgb]{0,0,0}\makebox(0,0)[lb]{\smash{$a$}}}%
    \put(0.58104784,0.02791136){\color[rgb]{0,0,0}\makebox(0,0)[lb]{\smash{$I$}}}%
    \put(0.88588049,0.02829399){\color[rgb]{0,0,0}\makebox(0,0)[lb]{\smash{$J$}}}%
    \put(0.96702792,0.02952951){\color[rgb]{0,0,0}\makebox(0,0)[lb]{\smash{$a$}}}%
    \put(0.85385704,0.02775706){\color[rgb]{0,0,0}\makebox(0,0)[lb]{\smash{$I$}}}%
    \put(0.64123653,0.02946359){\color[rgb]{0,0,0}\makebox(0,0)[lb]{\smash{$K$}}}%
    \put(0.91646865,0.02877715){\color[rgb]{0,0,0}\makebox(0,0)[lb]{\smash{$K$}}}%
    \put(0.72498884,0.0727792){\color[rgb]{0,0,0}\makebox(0,0)[lb]{\smash{$=$}}}%
  \end{picture}%
\endgroup%

\end{aligned}
\end{align}
Similarly it is found that $\phi$ must also respect the symmetry $\phi(e_I \cdot e_a , e_J , e_K)=\phi( e_I , e_a \cdot e_J , e_K)$. Furthermore, identical moves would guarantee the condition extends to all arguments of $\phi$.
\begin{align}
\begin{aligned}
\label{eq:phi_ass}
%% Creator: Inkscape 0.48.2, www.inkscape.org
%% PDF/EPS/PS + LaTeX output extension by Johan Engelen, 2010
%% Accompanies image file 'defects_22_phi_out.pdf' (pdf, eps, ps)
%%
%% To include the image in your LaTeX document, write
%%   \input{<filename>.pdf_tex}
%%  instead of
%%   \includegraphics{<filename>.pdf}
%% To scale the image, write
%%   \def\svgwidth{<desired width>}
%%   \input{<filename>.pdf_tex}
%%  instead of
%%   \includegraphics[width=<desired width>]{<filename>.pdf}
%%
%% Images with a different path to the parent latex file can
%% be accessed with the `import' package (which may need to be
%% installed) using
%%   \usepackage{import}
%% in the preamble, and then including the image with
%%   \import{<path to file>}{<filename>.pdf_tex}
%% Alternatively, one can specify
%%   \graphicspath{{<path to file>/}}
%% 
%% For more information, please see info/svg-inkscape on CTAN:
%%   http://tug.ctan.org/tex-archive/info/svg-inkscape
%%
\begingroup%
  \makeatletter%
  \providecommand\color[2][]{%
    \errmessage{(Inkscape) Color is used for the text in Inkscape, but the package 'color.sty' is not loaded}%
    \renewcommand\color[2][]{}%
  }%
  \providecommand\transparent[1]{%
    \errmessage{(Inkscape) Transparency is used (non-zero) for the text in Inkscape, but the package 'transparent.sty' is not loaded}%
    \renewcommand\transparent[1]{}%
  }%
  \providecommand\rotatebox[2]{#2}%
  \ifx\svgwidth\undefined%
    \setlength{\unitlength}{362.1313797bp}%
    \ifx\svgscale\undefined%
      \relax%
    \else%
      \setlength{\unitlength}{\unitlength * \real{\svgscale}}%
    \fi%
  \else%
    \setlength{\unitlength}{\svgwidth}%
  \fi%
  \global\let\svgwidth\undefined%
  \global\let\svgscale\undefined%
  \makeatother%
  \begin{picture}(1,0.16736507)%
    \put(0,0){\includegraphics[width=\unitlength]{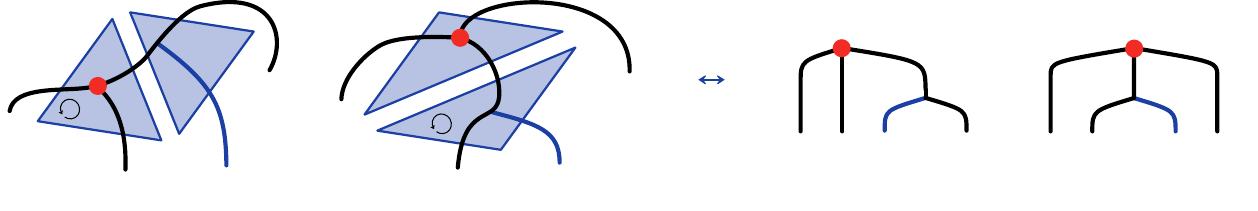}}%
    \put(-0.00078744,0.04391878){\color[rgb]{0,0,0}\makebox(0,0)[lb]{\smash{$I$}}}%
    \put(0.20763971,0.08144523){\color[rgb]{0,0,0}\makebox(0,0)[lb]{\smash{$K$}}}%
    \put(0.17346081,0.00299597){\color[rgb]{0,0,0}\makebox(0,0)[lb]{\smash{$a$}}}%
    \put(0.0935846,0.00227602){\color[rgb]{0,0,0}\makebox(0,0)[lb]{\smash{$J$}}}%
    \put(0.26448772,0.05534757){\color[rgb]{0,0,0}\makebox(0,0)[lb]{\smash{$I$}}}%
    \put(0.43951825,0.00661273){\color[rgb]{0,0,0}\makebox(0,0)[lb]{\smash{$a$}}}%
    \put(0.4940747,0.07616507){\color[rgb]{0,0,0}\makebox(0,0)[lb]{\smash{$K$}}}%
    \put(0.35778889,0.00707812){\color[rgb]{0,0,0}\makebox(0,0)[lb]{\smash{$J$}}}%
    \put(0.2303772,0.10823951){\color[rgb]{0,0,0}\makebox(0,0)[lb]{\smash{$=$}}}%
    \put(0.66446645,0.02892671){\color[rgb]{0,0,0}\makebox(0,0)[lb]{\smash{$J$}}}%
    \put(0.69762227,0.02856993){\color[rgb]{0,0,0}\makebox(0,0)[lb]{\smash{$a$}}}%
    \put(0.63166923,0.0272267){\color[rgb]{0,0,0}\makebox(0,0)[lb]{\smash{$I$}}}%
    \put(0.86259411,0.02890898){\color[rgb]{0,0,0}\makebox(0,0)[lb]{\smash{$J$}}}%
    \put(0.92928363,0.02898583){\color[rgb]{0,0,0}\makebox(0,0)[lb]{\smash{$a$}}}%
    \put(0.82904707,0.02895846){\color[rgb]{0,0,0}\makebox(0,0)[lb]{\smash{$I$}}}%
    \put(0.76295049,0.02954729){\color[rgb]{0,0,0}\makebox(0,0)[lb]{\smash{$K$}}}%
    \put(0.96169596,0.02816792){\color[rgb]{0,0,0}\makebox(0,0)[lb]{\smash{$K$}}}%
    \put(0.78815048,0.07600347){\color[rgb]{0,0,0}\makebox(0,0)[lb]{\smash{$=$}}}%
  \end{picture}%
\endgroup%

\end{aligned}
\end{align}

One must now show the reverse implication also holds. Let $\gls{V}$ and $V^{\times}$ be two $\gls{A}$-$A$ bimodules with left actions $\gls{l}$ and $l^{\times}$, and right actions $\gls{r}$ and $r^{\times}$. If they come equipped with maps $P^{-1}$ and $Q^{-1}$ satisfying conditions \ref{it:H_F} and \eqref{eq:left_right_action_relations} then the fact $l$ and $r^{\times}$, $r$ and $l^{\times}$ are related follows from the non-degeneracy of $P^{-1}$ and $Q^{-1}$. The properties of a bimodule guarantee the Pachner moves associated with defect lines are satisfied. Finally, if $\phi$ satisfies condition \ref{it:phi} all relevant Pachner moves are also satisfied.   
\end{proof}

An intuitive feature of models with symmetry under homeomorphisms is that a change of scale should leave the model unaltered. We already know that the partition function of a defect model $\gls{Z}(M,\gls{Gamma})$ is insensitive to the number of triangles, edges and vertices in the graph-like triangulation of $M$ as long as the boundary data of $M \subset \Rb^2$ remains unchanged. On the other hand, if $\Gamma$ is confined to a very small region of $M$, so as to `resemble' a defect node with no arrows, can we equate the evaluation $|\gls{Gg}|$ with that of a node map? 

Any defect graph, or collection of disjoint graphs,  $\Gamma$ is in fact equivalent to the defect graph consisting of a single node and no arrows -- $\gamma_0$ -- for a suitable choice of node map $\phi$. By equivalence, we mean they give rise to the same partition function. This statement is made precise below. 
\begin{proposition}
A planar defect model for $(M,\gls{Gamma})$ is equivalent to a model for $(M,\gamma_0)$ for some choice of $\phi$ to associate with $\gamma_0$, where $\phi \colon \gls{A} \otimes A \otimes A \to \gls{k}$. If $\phi$ is non-degenerate it satisfies $\phi(a \otimes b \otimes c)=\gls{eps} (y \cdot a \cdot b \cdot c)$, for some $y \in \gls{ZA}$.
\end{proposition}
\begin{proof}
To establish this result it will be shown that any defect graph $\Gamma$ behaves algebraically as a node map $\phi \colon A \otimes A \otimes A \to k$. To do so, one must show all properties of a node map are satisfied by a more general defect graph. Without loss of generality represent $\Gamma$ as a diagram $\begin{aligned}%% Creator: Inkscape 0.48.2, www.inkscape.org
%% PDF/EPS/PS + LaTeX output extension by Johan Engelen, 2010
%% Accompanies image file 'red_blob.pdf' (pdf, eps, ps)
%%
%% To include the image in your LaTeX document, write
%%   \input{<filename>.pdf_tex}
%%  instead of
%%   \includegraphics{<filename>.pdf}
%% To scale the image, write
%%   \def\svgwidth{<desired width>}
%%   \input{<filename>.pdf_tex}
%%  instead of
%%   \includegraphics[width=<desired width>]{<filename>.pdf}
%%
%% Images with a different path to the parent latex file can
%% be accessed with the `import' package (which may need to be
%% installed) using
%%   \usepackage{import}
%% in the preamble, and then including the image with
%%   \import{<path to file>}{<filename>.pdf_tex}
%% Alternatively, one can specify
%%   \graphicspath{{<path to file>/}}
%% 
%% For more information, please see info/svg-inkscape on CTAN:
%%   http://tug.ctan.org/tex-archive/info/svg-inkscape
%%
\begingroup%
  \makeatletter%
  \providecommand\color[2][]{%
    \errmessage{(Inkscape) Color is used for the text in Inkscape, but the package 'color.sty' is not loaded}%
    \renewcommand\color[2][]{}%
  }%
  \providecommand\transparent[1]{%
    \errmessage{(Inkscape) Transparency is used (non-zero) for the text in Inkscape, but the package 'transparent.sty' is not loaded}%
    \renewcommand\transparent[1]{}%
  }%
  \providecommand\rotatebox[2]{#2}%
  \ifx\svgwidth\undefined%
    \setlength{\unitlength}{13.86475363bp}%
    \ifx\svgscale\undefined%
      \relax%
    \else%
      \setlength{\unitlength}{\unitlength * \real{\svgscale}}%
    \fi%
  \else%
    \setlength{\unitlength}{\svgwidth}%
  \fi%
  \global\let\svgwidth\undefined%
  \global\let\svgscale\undefined%
  \makeatother%
  \begin{picture}(1,0.69249371)%
    \put(0,0){\includegraphics[width=\unitlength]{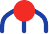}}%
  \end{picture}%
\endgroup%
\end{aligned}$ where
\begin{align}
\begin{aligned}
%% Creator: Inkscape 0.48.2, www.inkscape.org
%% PDF/EPS/PS + LaTeX output extension by Johan Engelen, 2010
%% Accompanies image file 'big_blob_def.pdf' (pdf, eps, ps)
%%
%% To include the image in your LaTeX document, write
%%   \input{<filename>.pdf_tex}
%%  instead of
%%   \includegraphics{<filename>.pdf}
%% To scale the image, write
%%   \def\svgwidth{<desired width>}
%%   \input{<filename>.pdf_tex}
%%  instead of
%%   \includegraphics[width=<desired width>]{<filename>.pdf}
%%
%% Images with a different path to the parent latex file can
%% be accessed with the `import' package (which may need to be
%% installed) using
%%   \usepackage{import}
%% in the preamble, and then including the image with
%%   \import{<path to file>}{<filename>.pdf_tex}
%% Alternatively, one can specify
%%   \graphicspath{{<path to file>/}}
%% 
%% For more information, please see info/svg-inkscape on CTAN:
%%   http://tug.ctan.org/tex-archive/info/svg-inkscape
%%
\begingroup%
  \makeatletter%
  \providecommand\color[2][]{%
    \errmessage{(Inkscape) Color is used for the text in Inkscape, but the package 'color.sty' is not loaded}%
    \renewcommand\color[2][]{}%
  }%
  \providecommand\transparent[1]{%
    \errmessage{(Inkscape) Transparency is used (non-zero) for the text in Inkscape, but the package 'transparent.sty' is not loaded}%
    \renewcommand\transparent[1]{}%
  }%
  \providecommand\rotatebox[2]{#2}%
  \ifx\svgwidth\undefined%
    \setlength{\unitlength}{167.02519019bp}%
    \ifx\svgscale\undefined%
      \relax%
    \else%
      \setlength{\unitlength}{\unitlength * \real{\svgscale}}%
    \fi%
  \else%
    \setlength{\unitlength}{\svgwidth}%
  \fi%
  \global\let\svgwidth\undefined%
  \global\let\svgscale\undefined%
  \makeatother%
  \begin{picture}(1,0.40530562)%
    \put(0,0){\includegraphics[width=\unitlength]{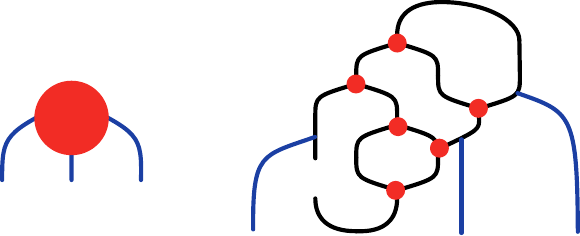}}%
    \put(0.29034437,0.16437702){\color[rgb]{0,0,0}\makebox(0,0)[lb]{\smash{$=$}}}%
    \put(0.50538947,0.07965549){\color[rgb]{0,0,0}\makebox(0,0)[lb]{\smash{$\cdots$}}}%
  \end{picture}%
\endgroup%

\end{aligned}
\end{align}
Note such a description is always possible since $\Gamma$ does not intersect the boundary of $M$. One must show that conditions \eqref{eq:phi_samespace} and \ref{it:phi} are satisfied. The identities below rely on a general notion of associativity: algebra lines can always move past vertices be these algebra vertices, bimodule vertices or defect nodes:
\begin{align}
\begin{aligned}
&%% Creator: Inkscape 0.48.2, www.inkscape.org
%% PDF/EPS/PS + LaTeX output extension by Johan Engelen, 2010
%% Accompanies image file 'phi_big_node.pdf' (pdf, eps, ps)
%%
%% To include the image in your LaTeX document, write
%%   \input{<filename>.pdf_tex}
%%  instead of
%%   \includegraphics{<filename>.pdf}
%% To scale the image, write
%%   \def\svgwidth{<desired width>}
%%   \input{<filename>.pdf_tex}
%%  instead of
%%   \includegraphics[width=<desired width>]{<filename>.pdf}
%%
%% Images with a different path to the parent latex file can
%% be accessed with the `import' package (which may need to be
%% installed) using
%%   \usepackage{import}
%% in the preamble, and then including the image with
%%   \import{<path to file>}{<filename>.pdf_tex}
%% Alternatively, one can specify
%%   \graphicspath{{<path to file>/}}
%% 
%% For more information, please see info/svg-inkscape on CTAN:
%%   http://tug.ctan.org/tex-archive/info/svg-inkscape
%%
\begingroup%
  \makeatletter%
  \providecommand\color[2][]{%
    \errmessage{(Inkscape) Color is used for the text in Inkscape, but the package 'color.sty' is not loaded}%
    \renewcommand\color[2][]{}%
  }%
  \providecommand\transparent[1]{%
    \errmessage{(Inkscape) Transparency is used (non-zero) for the text in Inkscape, but the package 'transparent.sty' is not loaded}%
    \renewcommand\transparent[1]{}%
  }%
  \providecommand\rotatebox[2]{#2}%
  \ifx\svgwidth\undefined%
    \setlength{\unitlength}{361.48677875bp}%
    \ifx\svgscale\undefined%
      \relax%
    \else%
      \setlength{\unitlength}{\unitlength * \real{\svgscale}}%
    \fi%
  \else%
    \setlength{\unitlength}{\svgwidth}%
  \fi%
  \global\let\svgwidth\undefined%
  \global\let\svgscale\undefined%
  \makeatother%
  \begin{picture}(1,0.12858389)%
    \put(0,0){\includegraphics[width=\unitlength]{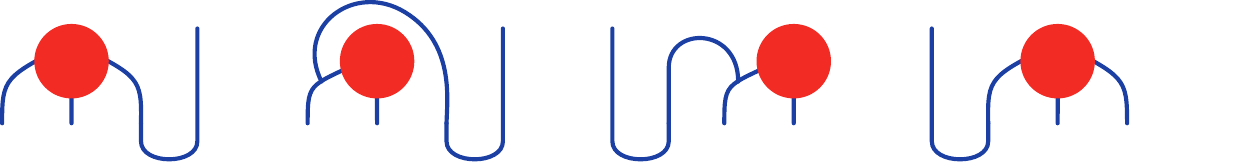}}%
    \put(0.17870656,0.04145451){\color[rgb]{0,0,0}\makebox(0,0)[lb]{\smash{$\overset{\eqref{eq:sigma_def}}{=}$}}}%
    \put(0.42214562,0.04145451){\color[rgb]{0,0,0}\makebox(0,0)[lb]{\smash{$\overset{\eqref{eq:snake}}{=}$}}}%
    \put(0.67665009,0.04145451){\color[rgb]{0,0,0}\makebox(0,0)[lb]{\smash{$\overset{\eqref{eq:phi_ass}}{=}$}}}%
  \end{picture}%
\endgroup%
 
\end{aligned}\, ,\\
\begin{aligned}
&%% Creator: Inkscape 0.48.2, www.inkscape.org
%% PDF/EPS/PS + LaTeX output extension by Johan Engelen, 2010
%% Accompanies image file 'big_node_sym.pdf' (pdf, eps, ps)
%%
%% To include the image in your LaTeX document, write
%%   \input{<filename>.pdf_tex}
%%  instead of
%%   \includegraphics{<filename>.pdf}
%% To scale the image, write
%%   \def\svgwidth{<desired width>}
%%   \input{<filename>.pdf_tex}
%%  instead of
%%   \includegraphics[width=<desired width>]{<filename>.pdf}
%%
%% Images with a different path to the parent latex file can
%% be accessed with the `import' package (which may need to be
%% installed) using
%%   \usepackage{import}
%% in the preamble, and then including the image with
%%   \import{<path to file>}{<filename>.pdf_tex}
%% Alternatively, one can specify
%%   \graphicspath{{<path to file>/}}
%% 
%% For more information, please see info/svg-inkscape on CTAN:
%%   http://tug.ctan.org/tex-archive/info/svg-inkscape
%%
\begingroup%
  \makeatletter%
  \providecommand\color[2][]{%
    \errmessage{(Inkscape) Color is used for the text in Inkscape, but the package 'color.sty' is not loaded}%
    \renewcommand\color[2][]{}%
  }%
  \providecommand\transparent[1]{%
    \errmessage{(Inkscape) Transparency is used (non-zero) for the text in Inkscape, but the package 'transparent.sty' is not loaded}%
    \renewcommand\transparent[1]{}%
  }%
  \providecommand\rotatebox[2]{#2}%
  \ifx\svgwidth\undefined%
    \setlength{\unitlength}{287.14730121bp}%
    \ifx\svgscale\undefined%
      \relax%
    \else%
      \setlength{\unitlength}{\unitlength * \real{\svgscale}}%
    \fi%
  \else%
    \setlength{\unitlength}{\svgwidth}%
  \fi%
  \global\let\svgwidth\undefined%
  \global\let\svgscale\undefined%
  \makeatother%
  \begin{picture}(1,0.15108464)%
    \put(0,0){\includegraphics[width=\unitlength]{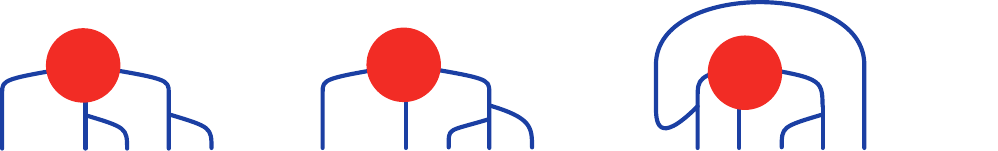}}%
    \put(0.56039821,0.0166268){\color[rgb]{0,0,0}\makebox(0,0)[lb]{\smash{$\overset{\eqref{eq:sigma_def}}{=}$}}}%
    \put(0.24000515,0.0166268){\color[rgb]{0,0,0}\makebox(0,0)[lb]{\smash{$\overset{\eqref{eq:phi_ass}}{=}$}}}%
  \end{picture}%
\endgroup%
 \, . \hspace{2.7cm}
\end{aligned}
\end{align}
The action of $\phi$ can be regarded as $a \otimes b \otimes c \mapsto f_{\phi}(a \cdot b \cdot c)$ for some $f_{\phi} \colon A \to k $. Since property \ref{it:phi} is verified this function must satisfy $f_{\phi}(\sigma(a) \cdot b)= f_{\phi} (b\cdot a)$.  
Classifying planar models with defects reduces therefore to classifying maps $f_{\phi}$. Recall the Nakayama automorphism determines a non-degenerate Frobenius form $\varepsilon$ up to a central element. Therefore, one learns that if $\phi$ is non-degenerate the map $f_{\phi}$ must obey $f_{\phi}(a)=\varepsilon(y \cdot a)$ for some $y \in \mathcal{Z}(A)$. 
\end{proof}
%%%%%%%%%%%%%%%%%%%%%%%%%%%%%%%%%%%%%%%%%%%%%%%%%%%%%%%%%
%%%%%%%%%%%%%%%%%%%%%%%%%%%%%%%%%%%%%%%%%%%%%%%%%%%%%%%%%
\section{Spherical models with defects}
%%%%%%%%%%%%%%%%%%%%%%%%%%%%%%%%%%%%%%%%%%%%%%%%%%%%%%%%%
%%%%%%%%%%%%%%%%%%%%%%%%%%%%%%%%%%%%%%%%%%%%%%%%%%%%%%%%%

Consider $M \subset \Sigma_0$, a submanifold of the sphere with a chosen orientation. As in section \S\ref{sec:spherical}, it is possible to extend the construction of models with defects to triangulated subsets of the sphere. Not only condition \eqref{eq:blob} must be satisfied but also both
\begin{align}
\begin{aligned}
\label{eq:defect_spherical_condition}
%% Creator: Inkscape 0.48.2, www.inkscape.org
%% PDF/EPS/PS + LaTeX output extension by Johan Engelen, 2010
%% Accompanies image file 'sphere_blob_defects.pdf' (pdf, eps, ps)
%%
%% To include the image in your LaTeX document, write
%%   \input{<filename>.pdf_tex}
%%  instead of
%%   \includegraphics{<filename>.pdf}
%% To scale the image, write
%%   \def\svgwidth{<desired width>}
%%   \input{<filename>.pdf_tex}
%%  instead of
%%   \includegraphics[width=<desired width>]{<filename>.pdf}
%%
%% Images with a different path to the parent latex file can
%% be accessed with the `import' package (which may need to be
%% installed) using
%%   \usepackage{import}
%% in the preamble, and then including the image with
%%   \import{<path to file>}{<filename>.pdf_tex}
%% Alternatively, one can specify
%%   \graphicspath{{<path to file>/}}
%% 
%% For more information, please see info/svg-inkscape on CTAN:
%%   http://tug.ctan.org/tex-archive/info/svg-inkscape
%%
\begingroup%
  \makeatletter%
  \providecommand\color[2][]{%
    \errmessage{(Inkscape) Color is used for the text in Inkscape, but the package 'color.sty' is not loaded}%
    \renewcommand\color[2][]{}%
  }%
  \providecommand\transparent[1]{%
    \errmessage{(Inkscape) Transparency is used (non-zero) for the text in Inkscape, but the package 'transparent.sty' is not loaded}%
    \renewcommand\transparent[1]{}%
  }%
  \providecommand\rotatebox[2]{#2}%
  \ifx\svgwidth\undefined%
    \setlength{\unitlength}{141.73971918bp}%
    \ifx\svgscale\undefined%
      \relax%
    \else%
      \setlength{\unitlength}{\unitlength * \real{\svgscale}}%
    \fi%
  \else%
    \setlength{\unitlength}{\svgwidth}%
  \fi%
  \global\let\svgwidth\undefined%
  \global\let\svgscale\undefined%
  \makeatother%
  \begin{picture}(1,0.60503771)%
    \put(0,0){\includegraphics[width=\unitlength]{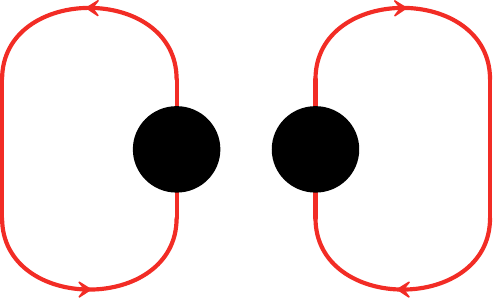}}%
    \put(0.47174086,0.27329872){\color[rgb]{0,0,0}\makebox(0,0)[lb]{\smash{$=$}}}%
  \end{picture}%
\endgroup%

\end{aligned}
\end{align}
and its counterpart (given by reversing the orientation of the arrows) must be met. Note $\begin{aligned} %% Creator: Inkscape 0.48.2, www.inkscape.org
%% PDF/EPS/PS + LaTeX output extension by Johan Engelen, 2010
%% Accompanies image file 'general_blob.pdf' (pdf, eps, ps)
%%
%% To include the image in your LaTeX document, write
%%   \input{<filename>.pdf_tex}
%%  instead of
%%   \includegraphics{<filename>.pdf}
%% To scale the image, write
%%   \def\svgwidth{<desired width>}
%%   \input{<filename>.pdf_tex}
%%  instead of
%%   \includegraphics[width=<desired width>]{<filename>.pdf}
%%
%% Images with a different path to the parent latex file can
%% be accessed with the `import' package (which may need to be
%% installed) using
%%   \usepackage{import}
%% in the preamble, and then including the image with
%%   \import{<path to file>}{<filename>.pdf_tex}
%% Alternatively, one can specify
%%   \graphicspath{{<path to file>/}}
%% 
%% For more information, please see info/svg-inkscape on CTAN:
%%   http://tug.ctan.org/tex-archive/info/svg-inkscape
%%
\begingroup%
  \makeatletter%
  \providecommand\color[2][]{%
    \errmessage{(Inkscape) Color is used for the text in Inkscape, but the package 'color.sty' is not loaded}%
    \renewcommand\color[2][]{}%
  }%
  \providecommand\transparent[1]{%
    \errmessage{(Inkscape) Transparency is used (non-zero) for the text in Inkscape, but the package 'transparent.sty' is not loaded}%
    \renewcommand\transparent[1]{}%
  }%
  \providecommand\rotatebox[2]{#2}%
  \ifx\svgwidth\undefined%
    \setlength{\unitlength}{11.50664063bp}%
    \ifx\svgscale\undefined%
      \relax%
    \else%
      \setlength{\unitlength}{\unitlength * \real{\svgscale}}%
    \fi%
  \else%
    \setlength{\unitlength}{\svgwidth}%
  \fi%
  \global\let\svgwidth\undefined%
  \global\let\svgscale\undefined%
  \makeatother%
  \begin{picture}(1,1.57272029)%
    \put(0,0){\includegraphics[width=\unitlength]{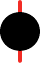}}%
  \end{picture}%
\endgroup%
\end{aligned}$ stands for a diagram, not necessarily closed, which is the same in both sides of the equation. A sufficient condition for \eqref{eq:defect_spherical_condition} to be verified is given by the identity below
\begin{align}
\label{eq:spherical_matrix}
Q_{\gamma\alpha}P^{\gamma\beta}=P_{\alpha\gamma}Q^{\beta\gamma}.
\end{align}
This condition can be seen in matrix form as $(P^{\tr})^2=Q^2$. Note that by allowing the arrows in \eqref{eq:defect_spherical_condition} to be reversed we would arrive at the condition $(Q^{\tr})^2=P^2$ which is therefore a consequence of \eqref{eq:spherical_matrix}. Another interpretation of equation \eqref{eq:spherical_matrix} is in terms of linear maps. It tells us that $\sigma_V \colon V \to V$ defined according to
\begin{align}
\sigma_V(v)=\sum_{\alpha\beta}P^{-1}(w_{\alpha},v)v_{\beta}Q^{\alpha\beta}
\end{align}
must match its inverse, $\sigma^{-1}_V \colon V \to V$ such that 
\begin{align}
\sigma^{-1}_V(v)&=\sum_{\alpha\beta}Q^{-1}(v,w_{\alpha})v_{\beta}P^{\alpha\beta}.
\end{align}
The fact $\sigma_V \circ \sigma_V^{-1}= \iden$ can be readily verified using the snake identities linking $P,P^{-1}$ and $Q,Q^{-1}$. The map $\sigma_V$ will be of use in the next chapter along with $\sigma_{V^{\times}}$, similarly defined. 

Unlike for the pure spherical model, condition \eqref{eq:defect_spherical_condition} is not automatically satisfied for $M=\Sigma_0$. Nevertheless, only surfaces without boundary will be considered in the next chapter where the explicit verification of condition \eqref{eq:defect_spherical_condition} will be used.

%%%%%%%%%%%%%%%%%%%%%%%%%%%%%%%%%%%%%%%%%%%%%%%%%%%%%%%%%
%%%%%%%%%%%%%%%%%%%%%%%%%%%%%%%%%%%%%%%%%%%%%%%%%%%%%%%%%
\chapter{Spin defect models}\label{ch:spin_defects}
%%%%%%%%%%%%%%%%%%%%%%%%%%%%%%%%%%%%%%%%%%%%%%%%%%%%%%%%%
%%%%%%%%%%%%%%%%%%%%%%%%%%%%%%%%%%%%%%%%%%%%%%%%%%%%%%%%%

%%%%%%%%%%%%%%%%%%%%%%%%%%%%%%%%%%%%%%%%%%%%%%%%%%%%%%%%%
%%%%%%%%%%%%%%%%%%%%%%%%%%%%%%%%%%%%%%%%%%%%%%%%%%%%%%%%%
\section{Defect models with crossings}
%%%%%%%%%%%%%%%%%%%%%%%%%%%%%%%%%%%%%%%%%%%%%%%%%%%%%%%%%
%%%%%%%%%%%%%%%%%%%%%%%%%%%%%%%%%%%%%%%%%%%%%%%%%%%%%%%%%

In the previous chapter we enlarged our class of planar diagrams to include defects. Accommodating this new information without losing invariance under Pachner moves was made possible by carefully translating the essential features of the old calculus to the new. The main novelty was building a description of the surface with defects $(M,\gls{Gamma})$ through a graph $\gls{Gg}$ that retained all the algebraic features of $G$, the graph dual to $M$ alone. The extension of defect diagrams to all surfaces will follow the same steps. Nonetheless this chapter will not treat the most general type of model extensively. We will discuss a restricted class of models where a single closed defect line is present. The reason for this restriction is simple: this is a class for which the role of immersions is well understood, information that is presently still lacking for the most general theory.

The new calculus is tailored to retain the invariance under regular homotopy characteristic of the spin models of chapter \S\ref{ch:spin}. We consider a surface with defects, immersed in $\Rb^3$. This surface has been triangulated through a graph-like triangulation. The construction of its dual diagram follows the rules of chapter \S\ref{ch:defects} and is extended to a ribbon graph through a regular neighbourhood, as in section \S\ref{sec:crossing}. A blackboard-framed projection from $\Rb^3$ to $\Rb^2$ is used to generate its spin diagram. This projection introduces crossings on the diagram, a structure that was not present in the planar models with defects. The algebra-algebra crossing $\lambda \colon A \otimes A \to A \otimes A$ has already been studied in detail; however, new types of crossings arise. Under and over crossings are not distinguished since they are regarded as equivalent under regular homotopy. These are denoted collectively as follows, where $\begin{aligned} %% Creator: Inkscape 0.48.2, www.inkscape.org
%% PDF/EPS/PS + LaTeX output extension by Johan Engelen, 2010
%% Accompanies image file 'solid_line_1.pdf' (pdf, eps, ps)
%%
%% To include the image in your LaTeX document, write
%%   \input{<filename>.pdf_tex}
%%  instead of
%%   \includegraphics{<filename>.pdf}
%% To scale the image, write
%%   \def\svgwidth{<desired width>}
%%   \input{<filename>.pdf_tex}
%%  instead of
%%   \includegraphics[width=<desired width>]{<filename>.pdf}
%%
%% Images with a different path to the parent latex file can
%% be accessed with the `import' package (which may need to be
%% installed) using
%%   \usepackage{import}
%% in the preamble, and then including the image with
%%   \import{<path to file>}{<filename>.pdf_tex}
%% Alternatively, one can specify
%%   \graphicspath{{<path to file>/}}
%% 
%% For more information, please see info/svg-inkscape on CTAN:
%%   http://tug.ctan.org/tex-archive/info/svg-inkscape
%%
\begingroup%
  \makeatletter%
  \providecommand\color[2][]{%
    \errmessage{(Inkscape) Color is used for the text in Inkscape, but the package 'color.sty' is not loaded}%
    \renewcommand\color[2][]{}%
  }%
  \providecommand\transparent[1]{%
    \errmessage{(Inkscape) Transparency is used (non-zero) for the text in Inkscape, but the package 'transparent.sty' is not loaded}%
    \renewcommand\transparent[1]{}%
  }%
  \providecommand\rotatebox[2]{#2}%
  \ifx\svgwidth\undefined%
    \setlength{\unitlength}{1.20146561bp}%
    \ifx\svgscale\undefined%
      \relax%
    \else%
      \setlength{\unitlength}{\unitlength * \real{\svgscale}}%
    \fi%
  \else%
    \setlength{\unitlength}{\svgwidth}%
  \fi%
  \global\let\svgwidth\undefined%
  \global\let\svgscale\undefined%
  \makeatother%
  \begin{picture}(1,11.01848951)%
    \put(0,0){\includegraphics[width=\unitlength]{solid_line_1.pdf}}%
  \end{picture}%
\endgroup%
 = %% Creator: Inkscape 0.48.2, www.inkscape.org
%% PDF/EPS/PS + LaTeX output extension by Johan Engelen, 2010
%% Accompanies image file 'solid_line_2.pdf' (pdf, eps, ps)
%%
%% To include the image in your LaTeX document, write
%%   \input{<filename>.pdf_tex}
%%  instead of
%%   \includegraphics{<filename>.pdf}
%% To scale the image, write
%%   \def\svgwidth{<desired width>}
%%   \input{<filename>.pdf_tex}
%%  instead of
%%   \includegraphics[width=<desired width>]{<filename>.pdf}
%%
%% Images with a different path to the parent latex file can
%% be accessed with the `import' package (which may need to be
%% installed) using
%%   \usepackage{import}
%% in the preamble, and then including the image with
%%   \import{<path to file>}{<filename>.pdf_tex}
%% Alternatively, one can specify
%%   \graphicspath{{<path to file>/}}
%% 
%% For more information, please see info/svg-inkscape on CTAN:
%%   http://tug.ctan.org/tex-archive/info/svg-inkscape
%%
\begingroup%
  \makeatletter%
  \providecommand\color[2][]{%
    \errmessage{(Inkscape) Color is used for the text in Inkscape, but the package 'color.sty' is not loaded}%
    \renewcommand\color[2][]{}%
  }%
  \providecommand\transparent[1]{%
    \errmessage{(Inkscape) Transparency is used (non-zero) for the text in Inkscape, but the package 'transparent.sty' is not loaded}%
    \renewcommand\transparent[1]{}%
  }%
  \providecommand\rotatebox[2]{#2}%
  \ifx\svgwidth\undefined%
    \setlength{\unitlength}{1.20146561bp}%
    \ifx\svgscale\undefined%
      \relax%
    \else%
      \setlength{\unitlength}{\unitlength * \real{\svgscale}}%
    \fi%
  \else%
    \setlength{\unitlength}{\svgwidth}%
  \fi%
  \global\let\svgwidth\undefined%
  \global\let\svgscale\undefined%
  \makeatother%
  \begin{picture}(1,11.01848951)%
    \put(0,0){\includegraphics[width=\unitlength]{solid_line_2.pdf}}%
  \end{picture}%
\endgroup%
, %% Creator: Inkscape 0.48.2, www.inkscape.org
%% PDF/EPS/PS + LaTeX output extension by Johan Engelen, 2010
%% Accompanies image file 'solid_line_3.pdf' (pdf, eps, ps)
%%
%% To include the image in your LaTeX document, write
%%   \input{<filename>.pdf_tex}
%%  instead of
%%   \includegraphics{<filename>.pdf}
%% To scale the image, write
%%   \def\svgwidth{<desired width>}
%%   \input{<filename>.pdf_tex}
%%  instead of
%%   \includegraphics[width=<desired width>]{<filename>.pdf}
%%
%% Images with a different path to the parent latex file can
%% be accessed with the `import' package (which may need to be
%% installed) using
%%   \usepackage{import}
%% in the preamble, and then including the image with
%%   \import{<path to file>}{<filename>.pdf_tex}
%% Alternatively, one can specify
%%   \graphicspath{{<path to file>/}}
%% 
%% For more information, please see info/svg-inkscape on CTAN:
%%   http://tug.ctan.org/tex-archive/info/svg-inkscape
%%
\begingroup%
  \makeatletter%
  \providecommand\color[2][]{%
    \errmessage{(Inkscape) Color is used for the text in Inkscape, but the package 'color.sty' is not loaded}%
    \renewcommand\color[2][]{}%
  }%
  \providecommand\transparent[1]{%
    \errmessage{(Inkscape) Transparency is used (non-zero) for the text in Inkscape, but the package 'transparent.sty' is not loaded}%
    \renewcommand\transparent[1]{}%
  }%
  \providecommand\rotatebox[2]{#2}%
  \ifx\svgwidth\undefined%
    \setlength{\unitlength}{4.70028763bp}%
    \ifx\svgscale\undefined%
      \relax%
    \else%
      \setlength{\unitlength}{\unitlength * \real{\svgscale}}%
    \fi%
  \else%
    \setlength{\unitlength}{\svgwidth}%
  \fi%
  \global\let\svgwidth\undefined%
  \global\let\svgscale\undefined%
  \makeatother%
  \begin{picture}(1,2.81649491)%
    \put(0,0){\includegraphics[width=\unitlength]{solid_line_3.pdf}}%
  \end{picture}%
\endgroup%
, %% Creator: Inkscape 0.48.2, www.inkscape.org
%% PDF/EPS/PS + LaTeX output extension by Johan Engelen, 2010
%% Accompanies image file 'solid_line_4.pdf' (pdf, eps, ps)
%%
%% To include the image in your LaTeX document, write
%%   \input{<filename>.pdf_tex}
%%  instead of
%%   \includegraphics{<filename>.pdf}
%% To scale the image, write
%%   \def\svgwidth{<desired width>}
%%   \input{<filename>.pdf_tex}
%%  instead of
%%   \includegraphics[width=<desired width>]{<filename>.pdf}
%%
%% Images with a different path to the parent latex file can
%% be accessed with the `import' package (which may need to be
%% installed) using
%%   \usepackage{import}
%% in the preamble, and then including the image with
%%   \import{<path to file>}{<filename>.pdf_tex}
%% Alternatively, one can specify
%%   \graphicspath{{<path to file>/}}
%% 
%% For more information, please see info/svg-inkscape on CTAN:
%%   http://tug.ctan.org/tex-archive/info/svg-inkscape
%%
\begingroup%
  \makeatletter%
  \providecommand\color[2][]{%
    \errmessage{(Inkscape) Color is used for the text in Inkscape, but the package 'color.sty' is not loaded}%
    \renewcommand\color[2][]{}%
  }%
  \providecommand\transparent[1]{%
    \errmessage{(Inkscape) Transparency is used (non-zero) for the text in Inkscape, but the package 'transparent.sty' is not loaded}%
    \renewcommand\transparent[1]{}%
  }%
  \providecommand\rotatebox[2]{#2}%
  \ifx\svgwidth\undefined%
    \setlength{\unitlength}{4.69499741bp}%
    \ifx\svgscale\undefined%
      \relax%
    \else%
      \setlength{\unitlength}{\unitlength * \real{\svgscale}}%
    \fi%
  \else%
    \setlength{\unitlength}{\svgwidth}%
  \fi%
  \global\let\svgwidth\undefined%
  \global\let\svgscale\undefined%
  \makeatother%
  \begin{picture}(1,2.81966848)%
    \put(0,0){\includegraphics[width=\unitlength]{solid_line_4.pdf}}%
  \end{picture}%
\endgroup%
\end{aligned}$ and $W,W'=A,V,V^{\times}$:
\begin{align}
\begin{aligned}
%% Creator: Inkscape 0.48.2, www.inkscape.org
%% PDF/EPS/PS + LaTeX output extension by Johan Engelen, 2010
%% Accompanies image file 'crossing_defects.pdf' (pdf, eps, ps)
%%
%% To include the image in your LaTeX document, write
%%   \input{<filename>.pdf_tex}
%%  instead of
%%   \includegraphics{<filename>.pdf}
%% To scale the image, write
%%   \def\svgwidth{<desired width>}
%%   \input{<filename>.pdf_tex}
%%  instead of
%%   \includegraphics[width=<desired width>]{<filename>.pdf}
%%
%% Images with a different path to the parent latex file can
%% be accessed with the `import' package (which may need to be
%% installed) using
%%   \usepackage{import}
%% in the preamble, and then including the image with
%%   \import{<path to file>}{<filename>.pdf_tex}
%% Alternatively, one can specify
%%   \graphicspath{{<path to file>/}}
%% 
%% For more information, please see info/svg-inkscape on CTAN:
%%   http://tug.ctan.org/tex-archive/info/svg-inkscape
%%
\begingroup%
  \makeatletter%
  \providecommand\color[2][]{%
    \errmessage{(Inkscape) Color is used for the text in Inkscape, but the package 'color.sty' is not loaded}%
    \renewcommand\color[2][]{}%
  }%
  \providecommand\transparent[1]{%
    \errmessage{(Inkscape) Transparency is used (non-zero) for the text in Inkscape, but the package 'transparent.sty' is not loaded}%
    \renewcommand\transparent[1]{}%
  }%
  \providecommand\rotatebox[2]{#2}%
  \ifx\svgwidth\undefined%
    \setlength{\unitlength}{138.38503418bp}%
    \ifx\svgscale\undefined%
      \relax%
    \else%
      \setlength{\unitlength}{\unitlength * \real{\svgscale}}%
    \fi%
  \else%
    \setlength{\unitlength}{\svgwidth}%
  \fi%
  \global\let\svgwidth\undefined%
  \global\let\svgscale\undefined%
  \makeatother%
  \begin{picture}(1,0.40936172)%
    \put(0,0){\includegraphics[width=\unitlength]{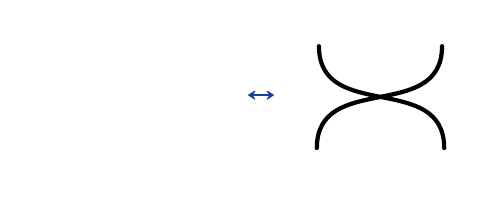}}%
    \put(-0.0020606,0.19486963){\color[rgb]{0,0,0}\makebox(0,0)[lb]{\smash{$(\lambda_{W,W'})_{IJ}{}^{KL}$}}}%
    \put(0.64313719,0.00595598){\color[rgb]{0,0,0}\makebox(0,0)[lb]{\smash{$I$}}}%
    \put(0.90947472,0.00698805){\color[rgb]{0,0,0}\makebox(0,0)[lb]{\smash{$J$}}}%
    \put(0.64623391,0.36416965){\color[rgb]{0,0,0}\makebox(0,0)[lb]{\smash{$K$}}}%
    \put(0.90224851,0.36313741){\color[rgb]{0,0,0}\makebox(0,0)[lb]{\smash{$L$}}}%
  \end{picture}%
\endgroup%

\end{aligned}.
\end{align}
Note that $I,L \in W$ and $J,K \in W'$ and that $\lambda_{A,A}=\lambda$, the algebra crossing introduced in section \S\ref{sec:crossing}. 

We are thus ready to define a spin model with defects. For simplicity, the notation $\begin{aligned}%% Creator: Inkscape 0.48.2, www.inkscape.org
%% PDF/EPS/PS + LaTeX output extension by Johan Engelen, 2010
%% Accompanies image file 'black_blob.pdf' (pdf, eps, ps)
%%
%% To include the image in your LaTeX document, write
%%   \input{<filename>.pdf_tex}
%%  instead of
%%   \includegraphics{<filename>.pdf}
%% To scale the image, write
%%   \def\svgwidth{<desired width>}
%%   \input{<filename>.pdf_tex}
%%  instead of
%%   \includegraphics[width=<desired width>]{<filename>.pdf}
%%
%% Images with a different path to the parent latex file can
%% be accessed with the `import' package (which may need to be
%% installed) using
%%   \usepackage{import}
%% in the preamble, and then including the image with
%%   \import{<path to file>}{<filename>.pdf_tex}
%% Alternatively, one can specify
%%   \graphicspath{{<path to file>/}}
%% 
%% For more information, please see info/svg-inkscape on CTAN:
%%   http://tug.ctan.org/tex-archive/info/svg-inkscape
%%
\begingroup%
  \makeatletter%
  \providecommand\color[2][]{%
    \errmessage{(Inkscape) Color is used for the text in Inkscape, but the package 'color.sty' is not loaded}%
    \renewcommand\color[2][]{}%
  }%
  \providecommand\transparent[1]{%
    \errmessage{(Inkscape) Transparency is used (non-zero) for the text in Inkscape, but the package 'transparent.sty' is not loaded}%
    \renewcommand\transparent[1]{}%
  }%
  \providecommand\rotatebox[2]{#2}%
  \ifx\svgwidth\undefined%
    \setlength{\unitlength}{16.75061919bp}%
    \ifx\svgscale\undefined%
      \relax%
    \else%
      \setlength{\unitlength}{\unitlength * \real{\svgscale}}%
    \fi%
  \else%
    \setlength{\unitlength}{\svgwidth}%
  \fi%
  \global\let\svgwidth\undefined%
  \global\let\svgscale\undefined%
  \makeatother%
  \begin{picture}(1,0.81094489)%
    \put(0,0){\includegraphics[width=\unitlength]{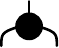}}%
  \end{picture}%
\endgroup%
\end{aligned}$ is used to denote the multiplication and action vertices $\begin{aligned}%% Creator: Inkscape 0.48.2, www.inkscape.org
%% PDF/EPS/PS + LaTeX output extension by Johan Engelen, 2010
%% Accompanies image file 'black_blob_v.pdf' (pdf, eps, ps)
%%
%% To include the image in your LaTeX document, write
%%   \input{<filename>.pdf_tex}
%%  instead of
%%   \includegraphics{<filename>.pdf}
%% To scale the image, write
%%   \def\svgwidth{<desired width>}
%%   \input{<filename>.pdf_tex}
%%  instead of
%%   \includegraphics[width=<desired width>]{<filename>.pdf}
%%
%% Images with a different path to the parent latex file can
%% be accessed with the `import' package (which may need to be
%% installed) using
%%   \usepackage{import}
%% in the preamble, and then including the image with
%%   \import{<path to file>}{<filename>.pdf_tex}
%% Alternatively, one can specify
%%   \graphicspath{{<path to file>/}}
%% 
%% For more information, please see info/svg-inkscape on CTAN:
%%   http://tug.ctan.org/tex-archive/info/svg-inkscape
%%
\begingroup%
  \makeatletter%
  \providecommand\color[2][]{%
    \errmessage{(Inkscape) Color is used for the text in Inkscape, but the package 'color.sty' is not loaded}%
    \renewcommand\color[2][]{}%
  }%
  \providecommand\transparent[1]{%
    \errmessage{(Inkscape) Transparency is used (non-zero) for the text in Inkscape, but the package 'transparent.sty' is not loaded}%
    \renewcommand\transparent[1]{}%
  }%
  \providecommand\rotatebox[2]{#2}%
  \ifx\svgwidth\undefined%
    \setlength{\unitlength}{16.75061919bp}%
    \ifx\svgscale\undefined%
      \relax%
    \else%
      \setlength{\unitlength}{\unitlength * \real{\svgscale}}%
    \fi%
  \else%
    \setlength{\unitlength}{\svgwidth}%
  \fi%
  \global\let\svgwidth\undefined%
  \global\let\svgscale\undefined%
  \makeatother%
  \begin{picture}(1,0.81094489)%
    \put(0,0){\includegraphics[width=\unitlength]{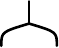}}%
  \end{picture}%
\endgroup%
\end{aligned}$ , and nodes $\begin{aligned}%% Creator: Inkscape 0.48.2, www.inkscape.org
%% PDF/EPS/PS + LaTeX output extension by Johan Engelen, 2010
%% Accompanies image file 'black_blob_node.pdf' (pdf, eps, ps)
%%
%% To include the image in your LaTeX document, write
%%   \input{<filename>.pdf_tex}
%%  instead of
%%   \includegraphics{<filename>.pdf}
%% To scale the image, write
%%   \def\svgwidth{<desired width>}
%%   \input{<filename>.pdf_tex}
%%  instead of
%%   \includegraphics[width=<desired width>]{<filename>.pdf}
%%
%% Images with a different path to the parent latex file can
%% be accessed with the `import' package (which may need to be
%% installed) using
%%   \usepackage{import}
%% in the preamble, and then including the image with
%%   \import{<path to file>}{<filename>.pdf_tex}
%% Alternatively, one can specify
%%   \graphicspath{{<path to file>/}}
%% 
%% For more information, please see info/svg-inkscape on CTAN:
%%   http://tug.ctan.org/tex-archive/info/svg-inkscape
%%
\begingroup%
  \makeatletter%
  \providecommand\color[2][]{%
    \errmessage{(Inkscape) Color is used for the text in Inkscape, but the package 'color.sty' is not loaded}%
    \renewcommand\color[2][]{}%
  }%
  \providecommand\transparent[1]{%
    \errmessage{(Inkscape) Transparency is used (non-zero) for the text in Inkscape, but the package 'transparent.sty' is not loaded}%
    \renewcommand\transparent[1]{}%
  }%
  \providecommand\rotatebox[2]{#2}%
  \ifx\svgwidth\undefined%
    \setlength{\unitlength}{16.75061919bp}%
    \ifx\svgscale\undefined%
      \relax%
    \else%
      \setlength{\unitlength}{\unitlength * \real{\svgscale}}%
    \fi%
  \else%
    \setlength{\unitlength}{\svgwidth}%
  \fi%
  \global\let\svgwidth\undefined%
  \global\let\svgscale\undefined%
  \makeatother%
  \begin{picture}(1,0.81094489)%
    \put(0,0){\includegraphics[width=\unitlength]{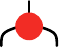}}%
  \end{picture}%
\endgroup%
\end{aligned}$ collectively. 

\begin{definition} \label{def:spin_defects}
A spin state sum model with defects is a planar state sum model with defects, together with crossing maps $\lambda _{W,W'} \colon W \otimes W' \to W' \otimes W$, $W,W'=A,V,V^{\times}$, satisfying the following axioms:
\begin{enumerate}[label=I\arabic{*}), ref=(I\arabic{*})]
\item \label{it:B_defects} compatibility with $B$, $P^{-1}$ and $Q^{-1}$, \hskip0.9cm $\begin{aligned} %% Creator: Inkscape 0.48.2, www.inkscape.org
%% PDF/EPS/PS + LaTeX output extension by Johan Engelen, 2010
%% Accompanies image file 'axiom_form_defect.pdf' (pdf, eps, ps)
%%
%% To include the image in your LaTeX document, write
%%   \input{<filename>.pdf_tex}
%%  instead of
%%   \includegraphics{<filename>.pdf}
%% To scale the image, write
%%   \def\svgwidth{<desired width>}
%%   \input{<filename>.pdf_tex}
%%  instead of
%%   \includegraphics[width=<desired width>]{<filename>.pdf}
%%
%% Images with a different path to the parent latex file can
%% be accessed with the `import' package (which may need to be
%% installed) using
%%   \usepackage{import}
%% in the preamble, and then including the image with
%%   \import{<path to file>}{<filename>.pdf_tex}
%% Alternatively, one can specify
%%   \graphicspath{{<path to file>/}}
%% 
%% For more information, please see info/svg-inkscape on CTAN:
%%   http://tug.ctan.org/tex-archive/info/svg-inkscape
%%
\begingroup%
  \makeatletter%
  \providecommand\color[2][]{%
    \errmessage{(Inkscape) Color is used for the text in Inkscape, but the package 'color.sty' is not loaded}%
    \renewcommand\color[2][]{}%
  }%
  \providecommand\transparent[1]{%
    \errmessage{(Inkscape) Transparency is used (non-zero) for the text in Inkscape, but the package 'transparent.sty' is not loaded}%
    \renewcommand\transparent[1]{}%
  }%
  \providecommand\rotatebox[2]{#2}%
  \ifx\svgwidth\undefined%
    \setlength{\unitlength}{120.57544212bp}%
    \ifx\svgscale\undefined%
      \relax%
    \else%
      \setlength{\unitlength}{\unitlength * \real{\svgscale}}%
    \fi%
  \else%
    \setlength{\unitlength}{\svgwidth}%
  \fi%
  \global\let\svgwidth\undefined%
  \global\let\svgscale\undefined%
  \makeatother%
  \begin{picture}(1,0.21909887)%
    \put(0,0){\includegraphics[width=\unitlength]{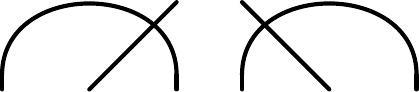}}%
    \put(0.46464559,0.06359185){\color[rgb]{0,0,0}\makebox(0,0)[lb]{\smash{$=$}}}%
  \end{picture}%
\endgroup%
\end{aligned}\,$,
\item \label{it:C_defects} compatibility with $C$, $L$, $R$ and $\phi$, \hskip1.1cm $\begin{aligned} %% Creator: Inkscape 0.48.2, www.inkscape.org
%% PDF/EPS/PS + LaTeX output extension by Johan Engelen, 2010
%% Accompanies image file 'axiom_mult_defect.pdf' (pdf, eps, ps)
%%
%% To include the image in your LaTeX document, write
%%   \input{<filename>.pdf_tex}
%%  instead of
%%   \includegraphics{<filename>.pdf}
%% To scale the image, write
%%   \def\svgwidth{<desired width>}
%%   \input{<filename>.pdf_tex}
%%  instead of
%%   \includegraphics[width=<desired width>]{<filename>.pdf}
%%
%% Images with a different path to the parent latex file can
%% be accessed with the `import' package (which may need to be
%% installed) using
%%   \usepackage{import}
%% in the preamble, and then including the image with
%%   \import{<path to file>}{<filename>.pdf_tex}
%% Alternatively, one can specify
%%   \graphicspath{{<path to file>/}}
%% 
%% For more information, please see info/svg-inkscape on CTAN:
%%   http://tug.ctan.org/tex-archive/info/svg-inkscape
%%
\begingroup%
  \makeatletter%
  \providecommand\color[2][]{%
    \errmessage{(Inkscape) Color is used for the text in Inkscape, but the package 'color.sty' is not loaded}%
    \renewcommand\color[2][]{}%
  }%
  \providecommand\transparent[1]{%
    \errmessage{(Inkscape) Transparency is used (non-zero) for the text in Inkscape, but the package 'transparent.sty' is not loaded}%
    \renewcommand\transparent[1]{}%
  }%
  \providecommand\rotatebox[2]{#2}%
  \ifx\svgwidth\undefined%
    \setlength{\unitlength}{124.33878124bp}%
    \ifx\svgscale\undefined%
      \relax%
    \else%
      \setlength{\unitlength}{\unitlength * \real{\svgscale}}%
    \fi%
  \else%
    \setlength{\unitlength}{\svgwidth}%
  \fi%
  \global\let\svgwidth\undefined%
  \global\let\svgscale\undefined%
  \makeatother%
  \begin{picture}(1,0.32354096)%
    \put(0,0){\includegraphics[width=\unitlength]{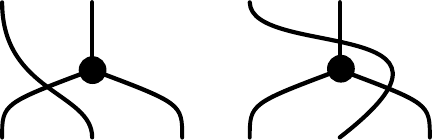}}%
    \put(0.45951808,0.05716834){\color[rgb]{0,0,0}\makebox(0,0)[lb]{\smash{$=$}}}%
  \end{picture}%
\endgroup%
\end{aligned}\,$,
\item \label{it: RII_defects} the Reidemeister II move (RII), \hskip2.35cm $\begin{aligned}%% Creator: Inkscape 0.48.2, www.inkscape.org
%% PDF/EPS/PS + LaTeX output extension by Johan Engelen, 2010
%% Accompanies image file 'axiom_square_defect.pdf' (pdf, eps, ps)
%%
%% To include the image in your LaTeX document, write
%%   \input{<filename>.pdf_tex}
%%  instead of
%%   \includegraphics{<filename>.pdf}
%% To scale the image, write
%%   \def\svgwidth{<desired width>}
%%   \input{<filename>.pdf_tex}
%%  instead of
%%   \includegraphics[width=<desired width>]{<filename>.pdf}
%%
%% Images with a different path to the parent latex file can
%% be accessed with the `import' package (which may need to be
%% installed) using
%%   \usepackage{import}
%% in the preamble, and then including the image with
%%   \import{<path to file>}{<filename>.pdf_tex}
%% Alternatively, one can specify
%%   \graphicspath{{<path to file>/}}
%% 
%% For more information, please see info/svg-inkscape on CTAN:
%%   http://tug.ctan.org/tex-archive/info/svg-inkscape
%%
\begingroup%
  \makeatletter%
  \providecommand\color[2][]{%
    \errmessage{(Inkscape) Color is used for the text in Inkscape, but the package 'color.sty' is not loaded}%
    \renewcommand\color[2][]{}%
  }%
  \providecommand\transparent[1]{%
    \errmessage{(Inkscape) Transparency is used (non-zero) for the text in Inkscape, but the package 'transparent.sty' is not loaded}%
    \renewcommand\transparent[1]{}%
  }%
  \providecommand\rotatebox[2]{#2}%
  \ifx\svgwidth\undefined%
    \setlength{\unitlength}{60.60261686bp}%
    \ifx\svgscale\undefined%
      \relax%
    \else%
      \setlength{\unitlength}{\unitlength * \real{\svgscale}}%
    \fi%
  \else%
    \setlength{\unitlength}{\svgwidth}%
  \fi%
  \global\let\svgwidth\undefined%
  \global\let\svgscale\undefined%
  \makeatother%
  \begin{picture}(1,0.6942303)%
    \put(0,0){\includegraphics[width=\unitlength]{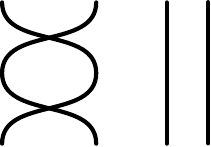}}%
    \put(0.53767689,0.30475986){\color[rgb]{0,0,0}\makebox(0,0)[lb]{\smash{$=$
}}}%
  \end{picture}%
\endgroup%
\end{aligned}\,$,
\item \label{it: RIII_defects} the Reidemeister III move (RIII), \hskip1.2cm $\begin{aligned} %% Creator: Inkscape 0.48.2, www.inkscape.org
%% PDF/EPS/PS + LaTeX output extension by Johan Engelen, 2010
%% Accompanies image file 'axiom_Reid3_defects.pdf' (pdf, eps, ps)
%%
%% To include the image in your LaTeX document, write
%%   \input{<filename>.pdf_tex}
%%  instead of
%%   \includegraphics{<filename>.pdf}
%% To scale the image, write
%%   \def\svgwidth{<desired width>}
%%   \input{<filename>.pdf_tex}
%%  instead of
%%   \includegraphics[width=<desired width>]{<filename>.pdf}
%%
%% Images with a different path to the parent latex file can
%% be accessed with the `import' package (which may need to be
%% installed) using
%%   \usepackage{import}
%% in the preamble, and then including the image with
%%   \import{<path to file>}{<filename>.pdf_tex}
%% Alternatively, one can specify
%%   \graphicspath{{<path to file>/}}
%% 
%% For more information, please see info/svg-inkscape on CTAN:
%%   http://tug.ctan.org/tex-archive/info/svg-inkscape
%%
\begingroup%
  \makeatletter%
  \providecommand\color[2][]{%
    \errmessage{(Inkscape) Color is used for the text in Inkscape, but the package 'color.sty' is not loaded}%
    \renewcommand\color[2][]{}%
  }%
  \providecommand\transparent[1]{%
    \errmessage{(Inkscape) Transparency is used (non-zero) for the text in Inkscape, but the package 'transparent.sty' is not loaded}%
    \renewcommand\transparent[1]{}%
  }%
  \providecommand\rotatebox[2]{#2}%
  \ifx\svgwidth\undefined%
    \setlength{\unitlength}{124.40082108bp}%
    \ifx\svgscale\undefined%
      \relax%
    \else%
      \setlength{\unitlength}{\unitlength * \real{\svgscale}}%
    \fi%
  \else%
    \setlength{\unitlength}{\svgwidth}%
  \fi%
  \global\let\svgwidth\undefined%
  \global\let\svgscale\undefined%
  \makeatother%
  \begin{picture}(1,0.3227895)%
    \put(0,0){\includegraphics[width=\unitlength]{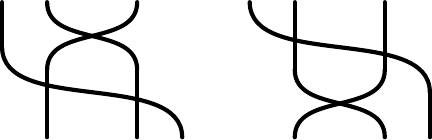}}%
    \put(0.44944427,0.14184992){\color[rgb]{0,0,0}\makebox(0,0)[lb]{\smash{$=$}}}%
  \end{picture}%
\endgroup%
\end{aligned}\,$,
\item \label{it:axiom_ribbon_defects} the ribbon condition, \hskip4.4cm  $\begin{aligned} %% Creator: Inkscape 0.48.2, www.inkscape.org
%% PDF/EPS/PS + LaTeX output extension by Johan Engelen, 2010
%% Accompanies image file 'axiom_left_right_defect.pdf' (pdf, eps, ps)
%%
%% To include the image in your LaTeX document, write
%%   \input{<filename>.pdf_tex}
%%  instead of
%%   \includegraphics{<filename>.pdf}
%% To scale the image, write
%%   \def\svgwidth{<desired width>}
%%   \input{<filename>.pdf_tex}
%%  instead of
%%   \includegraphics[width=<desired width>]{<filename>.pdf}
%%
%% Images with a different path to the parent latex file can
%% be accessed with the `import' package (which may need to be
%% installed) using
%%   \usepackage{import}
%% in the preamble, and then including the image with
%%   \import{<path to file>}{<filename>.pdf_tex}
%% Alternatively, one can specify
%%   \graphicspath{{<path to file>/}}
%% 
%% For more information, please see info/svg-inkscape on CTAN:
%%   http://tug.ctan.org/tex-archive/info/svg-inkscape
%%
\begingroup%
  \makeatletter%
  \providecommand\color[2][]{%
    \errmessage{(Inkscape) Color is used for the text in Inkscape, but the package 'color.sty' is not loaded}%
    \renewcommand\color[2][]{}%
  }%
  \providecommand\transparent[1]{%
    \errmessage{(Inkscape) Transparency is used (non-zero) for the text in Inkscape, but the package 'transparent.sty' is not loaded}%
    \renewcommand\transparent[1]{}%
  }%
  \providecommand\rotatebox[2]{#2}%
  \ifx\svgwidth\undefined%
    \setlength{\unitlength}{75.275bp}%
    \ifx\svgscale\undefined%
      \relax%
    \else%
      \setlength{\unitlength}{\unitlength * \real{\svgscale}}%
    \fi%
  \else%
    \setlength{\unitlength}{\svgwidth}%
  \fi%
  \global\let\svgwidth\undefined%
  \global\let\svgscale\undefined%
  \makeatother%
  \begin{picture}(1,0.59045802)%
    \put(0,0){\includegraphics[width=\unitlength]{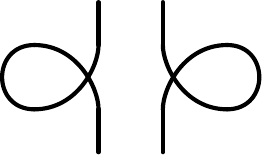}}%
    \put(0.44898783,0.26266386){\color[rgb]{0,0,0}\makebox(0,0)[lb]{\smash{$=$}}}%
  \end{picture}%
\endgroup%
\end{aligned}\,$.
\end{enumerate}
\end{definition}
Either side of equation \ref{it:axiom_ribbon_defects} defines a map $\varphi_W \colon W \to W$ where $W=A,V,V^{\times}$ -- for consistency of notation if $W=A$ we will refer to the map simply as $\varphi$. A spin model with defects is independent of the graph-like triangulation chosen, as it is spherical.
\begin{lemma}
A spin state sum model with defects is spherical.
\end{lemma}
\begin{proof}
We need only show equation \eqref{eq:defect_spherical_condition} is satisfied. Note that for any closed spin defects diagram, the natural generalisation of lemmas \ref{lem:closed_below} and \ref{lem:varphi} holds.
\begin{align}
\begin{aligned}
%% Creator: Inkscape 0.48.2, www.inkscape.org
%% PDF/EPS/PS + LaTeX output extension by Johan Engelen, 2010
%% Accompanies image file 'sphere_condition_crossing.pdf' (pdf, eps, ps)
%%
%% To include the image in your LaTeX document, write
%%   \input{<filename>.pdf_tex}
%%  instead of
%%   \includegraphics{<filename>.pdf}
%% To scale the image, write
%%   \def\svgwidth{<desired width>}
%%   \input{<filename>.pdf_tex}
%%  instead of
%%   \includegraphics[width=<desired width>]{<filename>.pdf}
%%
%% Images with a different path to the parent latex file can
%% be accessed with the `import' package (which may need to be
%% installed) using
%%   \usepackage{import}
%% in the preamble, and then including the image with
%%   \import{<path to file>}{<filename>.pdf_tex}
%% Alternatively, one can specify
%%   \graphicspath{{<path to file>/}}
%% 
%% For more information, please see info/svg-inkscape on CTAN:
%%   http://tug.ctan.org/tex-archive/info/svg-inkscape
%%
\begingroup%
  \makeatletter%
  \providecommand\color[2][]{%
    \errmessage{(Inkscape) Color is used for the text in Inkscape, but the package 'color.sty' is not loaded}%
    \renewcommand\color[2][]{}%
  }%
  \providecommand\transparent[1]{%
    \errmessage{(Inkscape) Transparency is used (non-zero) for the text in Inkscape, but the package 'transparent.sty' is not loaded}%
    \renewcommand\transparent[1]{}%
  }%
  \providecommand\rotatebox[2]{#2}%
  \ifx\svgwidth\undefined%
    \setlength{\unitlength}{359.30873646bp}%
    \ifx\svgscale\undefined%
      \relax%
    \else%
      \setlength{\unitlength}{\unitlength * \real{\svgscale}}%
    \fi%
  \else%
    \setlength{\unitlength}{\svgwidth}%
  \fi%
  \global\let\svgwidth\undefined%
  \global\let\svgscale\undefined%
  \makeatother%
  \begin{picture}(1,0.23870366)%
    \put(0,0){\includegraphics[width=\unitlength]{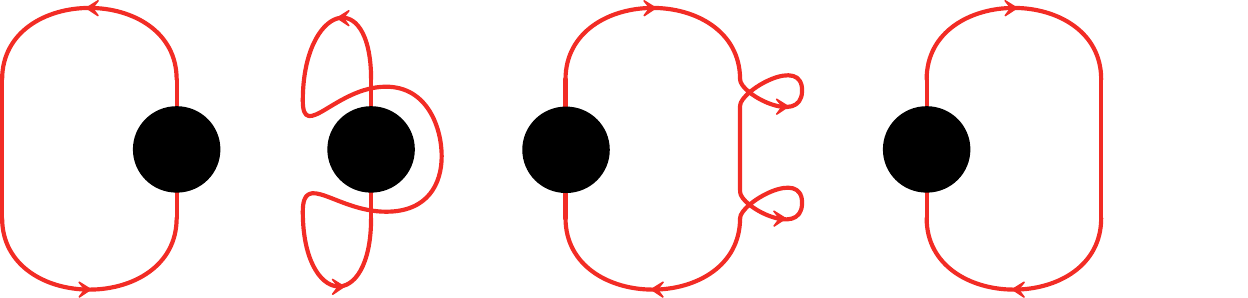}}%
    \put(0.18612167,0.10781065){\color[rgb]{0,0,0}\makebox(0,0)[lb]{\smash{$\overset{\eqref{eq:closed_below}}{=}$}}}%
    \put(0.36424147,0.10781065){\color[rgb]{0,0,0}\makebox(0,0)[lb]{\smash{$\overset{\text{\ref{it:phi_four}}}{=}$}}}%
    \put(0.65368614,0.10781065){\color[rgb]{0,0,0}\makebox(0,0)[lb]{\smash{$\overset{\text{\ref{it:phi_one}}}{=}$}}}%
  \end{picture}%
\endgroup%

\end{aligned}
\end{align}
\end{proof}

%%%%%%%%%%%%%%%%%%%%%%%%%%%%%%%%%%%%%%%%%%%%%
\begin{figure}[t!]
\center
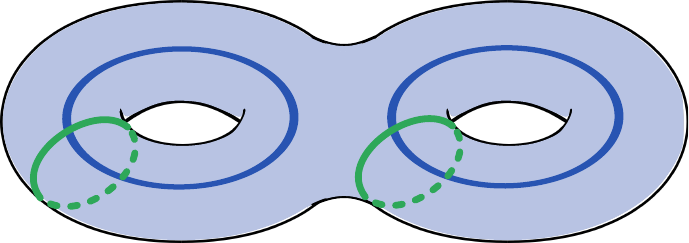
\caption[Generating loops]{A basis of curves that generate the fundamental group of $\Sigma_2$ are depicted above. This specific basis of generators will be known as `generating curves' and defects placed along them are called `generating loops'.}
\label{fig:gen}
\end{figure} 
%%%%%%%%%%%%%%%%%%%%%%%%%%%%%%%%%%%%%%%%%%%%%

The relation between a spin diagram $(\Sigma_g,\Gamma)$ and spin structures of $\Sigma_g$ is currently an open problem and one of the threads of development that are worth future investigation. The key difference between spin models and spin models with defects is the lack of a standard format any $G_{\Gamma}$ can be brought to, as opposed to the dual of $\Sigma_g$ with no defects which is always equivalent to $\gamma_g$ (see figure~\ref{fig:dual-graph}).   

To circumvent this issue we will devote our attention to a small class of $G_{\Gamma}$ that we shall denote as `generating loops'. Let us define a loop as a defect graph consisting of a single arrow and a single node where the arrow both begins and terminates. Furthermore, the map $\phi$ the node would come associated with is simply $L$ or $R$ -- we refer to this type of node as a transparent node. Topologically, a loop is therefore a simple closed curve. 

Generating loops are a very restricted subset of loops that do not bound a circle. They can be seen as defects placed along a specific basis of the fundamental group of a surface and are illustrated in figure~\ref{fig:gen}. Surfaces $\Sigma_g$ whose defect graph is a generating loop $l$ are denoted as $(\Sigma_g,l)$. We must reiterate the result below applies only to generating loops, although we conjecture it can be extended to all defect loops.

\begin{theorem}
The partition function for $(\Sigma_g,l)$, a surface with a generating loop, depends on the immersion only through the spin structure induced on $\Sigma_g$.
\end{theorem}
\begin{proof}
Recall that a spin structure in $\Sigma_g - D$ extends uniquely to a spin structure in $\Sigma_g$. Let $K_l$ be the a ribbon graph to associate with $(\Sigma_g - D,l)$ that we choose to be immersed in $\Rb^2 \subset \Rb^3$. If the immersion is an embedding $i_0$ the ribbon graph can be brought to the standard format below where the defect line could have been placed along any of the generating curves. For simplicity, the multiplication and action vertices have been consolidated into one.
$$
%% Creator: Inkscape 0.48.2, www.inkscape.org
%% PDF/EPS/PS + LaTeX output extension by Johan Engelen, 2010
%% Accompanies image file 'surface-immersion-defects.pdf' (pdf, eps, ps)
%%
%% To include the image in your LaTeX document, write
%%   \input{<filename>.pdf_tex}
%%  instead of
%%   \includegraphics{<filename>.pdf}
%% To scale the image, write
%%   \def\svgwidth{<desired width>}
%%   \input{<filename>.pdf_tex}
%%  instead of
%%   \includegraphics[width=<desired width>]{<filename>.pdf}
%%
%% Images with a different path to the parent latex file can
%% be accessed with the `import' package (which may need to be
%% installed) using
%%   \usepackage{import}
%% in the preamble, and then including the image with
%%   \import{<path to file>}{<filename>.pdf_tex}
%% Alternatively, one can specify
%%   \graphicspath{{<path to file>/}}
%% 
%% For more information, please see info/svg-inkscape on CTAN:
%%   http://tug.ctan.org/tex-archive/info/svg-inkscape
%%
\begingroup%
  \makeatletter%
  \providecommand\color[2][]{%
    \errmessage{(Inkscape) Color is used for the text in Inkscape, but the package 'color.sty' is not loaded}%
    \renewcommand\color[2][]{}%
  }%
  \providecommand\transparent[1]{%
    \errmessage{(Inkscape) Transparency is used (non-zero) for the text in Inkscape, but the package 'transparent.sty' is not loaded}%
    \renewcommand\transparent[1]{}%
  }%
  \providecommand\rotatebox[2]{#2}%
  \ifx\svgwidth\undefined%
    \setlength{\unitlength}{197.36258342bp}%
    \ifx\svgscale\undefined%
      \relax%
    \else%
      \setlength{\unitlength}{\unitlength * \real{\svgscale}}%
    \fi%
  \else%
    \setlength{\unitlength}{\svgwidth}%
  \fi%
  \global\let\svgwidth\undefined%
  \global\let\svgscale\undefined%
  \makeatother%
  \begin{picture}(1,0.83461445)%
    \put(0,0){\includegraphics[width=\unitlength]{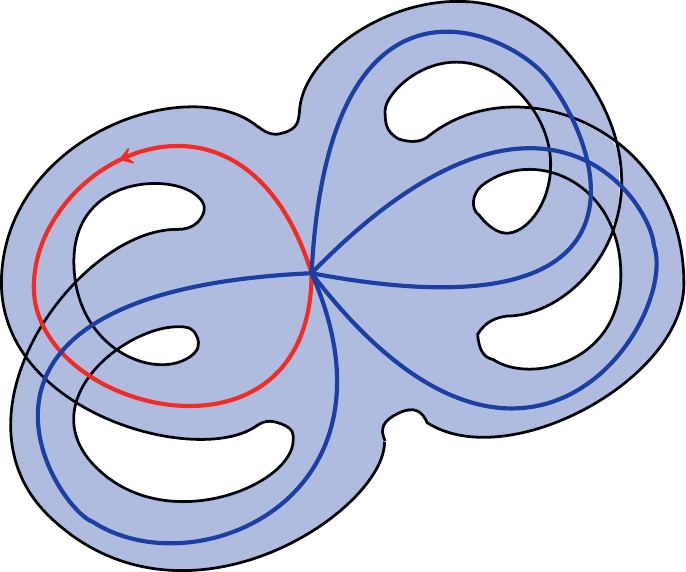}}%
  \end{picture}%
\endgroup%

$$
As in the case of spin models we can perform transformations on the ribbon loops leaving a neighbourhood of the vertex unaltered. By the same arguments presented in theorem~\ref{spin-embedding} we know each immersion of $(\Sigma_g,l)$ will be like $i_0(K_l)$ except for a number of curls in each ribbon loop. Since such loops can be cancelled pairwise for either defect or algebra lines (as a consequence of lemma \ref{lem:varphi}) the result follows.
\end{proof}

%%%%%%%%%%%%%%%%%%%%%%%%%%%%%%%%%%%%%%%%%%%%%%%%%%%%%%%%%
%%%%%%%%%%%%%%%%%%%%%%%%%%%%%%%%%%%%%%%%%%%%%%%%%%%%%%%%%
\section{Algebraic structure of defect lines}
%%%%%%%%%%%%%%%%%%%%%%%%%%%%%%%%%%%%%%%%%%%%%%%%%%%%%%%%%
%%%%%%%%%%%%%%%%%%%%%%%%%%%%%%%%%%%%%%%%%%%%%%%%%%%%%%%%%

In this section we will systematically study spin models equipped with defects $l$. Much like in section \S\ref{sec:spinssm}, we are interested in learning what consequences for distinguishing spin structures arise from the introduction of new information -- in the latter algebra-algebra crossings were the novelty, whereas now are those involving defect lines.

It will be useful to first study the four types of maps we can associate to a cylinder with a defect line. These are shown below and their corresponding maps are denoted as $p^V$, $p_V$ and $n^V$, $n_V$ respectively.   
\begin{align}
\begin{aligned}
%% Creator: Inkscape 0.48.2, www.inkscape.org
%% PDF/EPS/PS + LaTeX output extension by Johan Engelen, 2010
%% Accompanies image file 'cyl_defects.pdf' (pdf, eps, ps)
%%
%% To include the image in your LaTeX document, write
%%   \input{<filename>.pdf_tex}
%%  instead of
%%   \includegraphics{<filename>.pdf}
%% To scale the image, write
%%   \def\svgwidth{<desired width>}
%%   \input{<filename>.pdf_tex}
%%  instead of
%%   \includegraphics[width=<desired width>]{<filename>.pdf}
%%
%% Images with a different path to the parent latex file can
%% be accessed with the `import' package (which may need to be
%% installed) using
%%   \usepackage{import}
%% in the preamble, and then including the image with
%%   \import{<path to file>}{<filename>.pdf_tex}
%% Alternatively, one can specify
%%   \graphicspath{{<path to file>/}}
%% 
%% For more information, please see info/svg-inkscape on CTAN:
%%   http://tug.ctan.org/tex-archive/info/svg-inkscape
%%
\begingroup%
  \makeatletter%
  \providecommand\color[2][]{%
    \errmessage{(Inkscape) Color is used for the text in Inkscape, but the package 'color.sty' is not loaded}%
    \renewcommand\color[2][]{}%
  }%
  \providecommand\transparent[1]{%
    \errmessage{(Inkscape) Transparency is used (non-zero) for the text in Inkscape, but the package 'transparent.sty' is not loaded}%
    \renewcommand\transparent[1]{}%
  }%
  \providecommand\rotatebox[2]{#2}%
  \ifx\svgwidth\undefined%
    \setlength{\unitlength}{349.19917439bp}%
    \ifx\svgscale\undefined%
      \relax%
    \else%
      \setlength{\unitlength}{\unitlength * \real{\svgscale}}%
    \fi%
  \else%
    \setlength{\unitlength}{\svgwidth}%
  \fi%
  \global\let\svgwidth\undefined%
  \global\let\svgscale\undefined%
  \makeatother%
  \begin{picture}(1,0.25405775)%
    \put(0,0){\includegraphics[width=\unitlength]{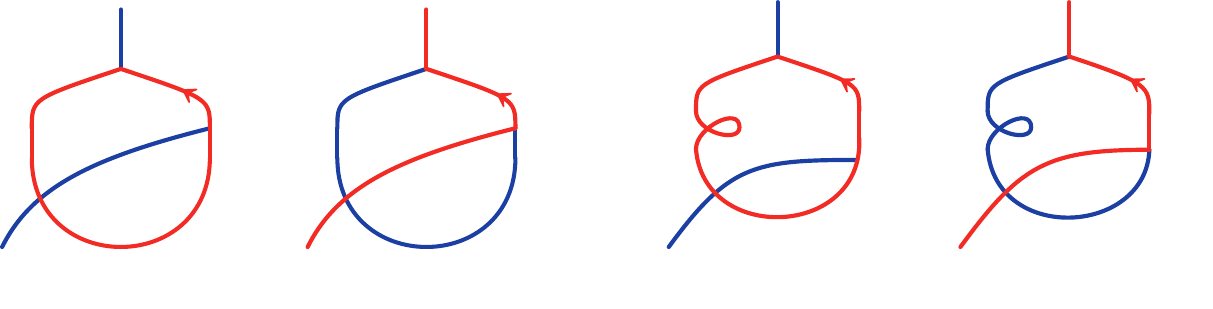}}%
    \put(0.0017449,0.00455283){\color[rgb]{0,0,0}\makebox(0,0)[lb]{\smash{$p^V \colon A \to A$}}}%
    \put(0.25375011,0.00455283){\color[rgb]{0,0,0}\makebox(0,0)[lb]{\smash{$p_V \colon V \to V$}}}%
    \put(0.55157442,0.00455283){\color[rgb]{0,0,0}\makebox(0,0)[lb]{\smash{$n^V \colon A \to A$}}}%
    \put(0.79212481,0.00455283){\color[rgb]{0,0,0}\makebox(0,0)[lb]{\smash{$n_V \colon V \to V$}}}%
  \end{picture}%
\endgroup%

\end{aligned}
\end{align}
We now define two subspaces of $V$: $\mathcal{Z}_{\lambda}(V)$, the set of elements $v \in V$ satisfying $m (a \otimes v)= m \circ \lambda (a \otimes v)$ for all $a \in A$, and $\overline{\mathcal{Z}}_{\lambda}(V)$, the set of elements $v \in V$ obeying the equation $m (a \otimes v)= m \circ \lambda (\varphi(a) \otimes v)$ for all $a \in A$. In diagrammatic terms, the elements $v$ of $\mathcal{Z}_{\lambda}(V)$ or $\overline{\mathcal{Z}}_{\lambda}(V)$ satisfy either
\begin{align}
\begin{aligned}%% Creator: Inkscape 0.48.2, www.inkscape.org
%% PDF/EPS/PS + LaTeX output extension by Johan Engelen, 2010
%% Accompanies image file 'Z_lambda_V.pdf' (pdf, eps, ps)
%%
%% To include the image in your LaTeX document, write
%%   \input{<filename>.pdf_tex}
%%  instead of
%%   \includegraphics{<filename>.pdf}
%% To scale the image, write
%%   \def\svgwidth{<desired width>}
%%   \input{<filename>.pdf_tex}
%%  instead of
%%   \includegraphics[width=<desired width>]{<filename>.pdf}
%%
%% Images with a different path to the parent latex file can
%% be accessed with the `import' package (which may need to be
%% installed) using
%%   \usepackage{import}
%% in the preamble, and then including the image with
%%   \import{<path to file>}{<filename>.pdf_tex}
%% Alternatively, one can specify
%%   \graphicspath{{<path to file>/}}
%% 
%% For more information, please see info/svg-inkscape on CTAN:
%%   http://tug.ctan.org/tex-archive/info/svg-inkscape
%%
\begingroup%
  \makeatletter%
  \providecommand\color[2][]{%
    \errmessage{(Inkscape) Color is used for the text in Inkscape, but the package 'color.sty' is not loaded}%
    \renewcommand\color[2][]{}%
  }%
  \providecommand\transparent[1]{%
    \errmessage{(Inkscape) Transparency is used (non-zero) for the text in Inkscape, but the package 'transparent.sty' is not loaded}%
    \renewcommand\transparent[1]{}%
  }%
  \providecommand\rotatebox[2]{#2}%
  \ifx\svgwidth\undefined%
    \setlength{\unitlength}{120.2421875bp}%
    \ifx\svgscale\undefined%
      \relax%
    \else%
      \setlength{\unitlength}{\unitlength * \real{\svgscale}}%
    \fi%
  \else%
    \setlength{\unitlength}{\svgwidth}%
  \fi%
  \global\let\svgwidth\undefined%
  \global\let\svgscale\undefined%
  \makeatother%
  \begin{picture}(1,0.47108055)%
    \put(0,0){\includegraphics[width=\unitlength]{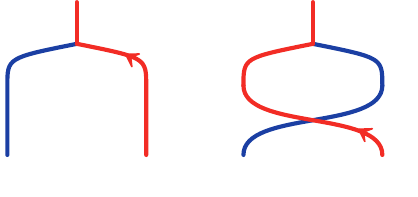}}%
    \put(0.43008901,0.23306478){\color[rgb]{0,0,0}\makebox(0,0)[lb]{\smash{$=$}}}%
    \put(-0.00237152,0.00685466){\color[rgb]{0,0,0}\makebox(0,0)[lb]{\smash{$a$}}}%
    \put(0.33029043,0.00685466){\color[rgb]{0,0,0}\makebox(0,0)[lb]{\smash{$v$}}}%
    \put(0.56315379,0.00685466){\color[rgb]{0,0,0}\makebox(0,0)[lb]{\smash{$a$}}}%
    \put(0.89581574,0.00685466){\color[rgb]{0,0,0}\makebox(0,0)[lb]{\smash{$v$}}}%
  \end{picture}%
\endgroup%
\end{aligned}, \text{ or} \hspace{5mm}\begin{aligned}%% Creator: Inkscape 0.48.2, www.inkscape.org
%% PDF/EPS/PS + LaTeX output extension by Johan Engelen, 2010
%% Accompanies image file 'Z_bar_lambda_V.pdf' (pdf, eps, ps)
%%
%% To include the image in your LaTeX document, write
%%   \input{<filename>.pdf_tex}
%%  instead of
%%   \includegraphics{<filename>.pdf}
%% To scale the image, write
%%   \def\svgwidth{<desired width>}
%%   \input{<filename>.pdf_tex}
%%  instead of
%%   \includegraphics[width=<desired width>]{<filename>.pdf}
%%
%% Images with a different path to the parent latex file can
%% be accessed with the `import' package (which may need to be
%% installed) using
%%   \usepackage{import}
%% in the preamble, and then including the image with
%%   \import{<path to file>}{<filename>.pdf_tex}
%% Alternatively, one can specify
%%   \graphicspath{{<path to file>/}}
%% 
%% For more information, please see info/svg-inkscape on CTAN:
%%   http://tug.ctan.org/tex-archive/info/svg-inkscape
%%
\begingroup%
  \makeatletter%
  \providecommand\color[2][]{%
    \errmessage{(Inkscape) Color is used for the text in Inkscape, but the package 'color.sty' is not loaded}%
    \renewcommand\color[2][]{}%
  }%
  \providecommand\transparent[1]{%
    \errmessage{(Inkscape) Transparency is used (non-zero) for the text in Inkscape, but the package 'transparent.sty' is not loaded}%
    \renewcommand\transparent[1]{}%
  }%
  \providecommand\rotatebox[2]{#2}%
  \ifx\svgwidth\undefined%
    \setlength{\unitlength}{120.6953125bp}%
    \ifx\svgscale\undefined%
      \relax%
    \else%
      \setlength{\unitlength}{\unitlength * \real{\svgscale}}%
    \fi%
  \else%
    \setlength{\unitlength}{\svgwidth}%
  \fi%
  \global\let\svgwidth\undefined%
  \global\let\svgscale\undefined%
  \makeatother%
  \begin{picture}(1,0.53559459)%
    \put(0,0){\includegraphics[width=\unitlength]{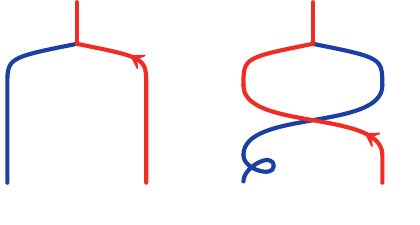}}%
    \put(0.42847433,0.29847239){\color[rgb]{0,0,0}\makebox(0,0)[lb]{\smash{$=$}}}%
    \put(-0.00236261,0.00682892){\color[rgb]{0,0,0}\makebox(0,0)[lb]{\smash{$b$}}}%
    \put(0.32905042,0.00682892){\color[rgb]{0,0,0}\makebox(0,0)[lb]{\smash{$a$}}}%
    \put(0.56103955,0.00682892){\color[rgb]{0,0,0}\makebox(0,0)[lb]{\smash{$b$}}}%
    \put(0.89245259,0.00682892){\color[rgb]{0,0,0}\makebox(0,0)[lb]{\smash{$a$}}}%
  \end{picture}%
\endgroup%
\end{aligned}, \label{eq:Z_lambda}
\end{align}
respectively. Similarly, we could have defined $\mathcal{Z}_{\lambda}(V^{\times})$ and $\overline{\mathcal{Z}}_{\lambda}(V^{\times})$ by inverting the direction of the arrows in the equations above. The proof of the following lemma is left as an exercise to the reader; the equations used to prove analogous statements in lemma \ref{lem:projector} can be followed step by step.
\begin{lemma}
The map $R.p_V$ is a projector $V \to V$ with image $\mathcal{Z}_{\lambda}(V)$. The map $R.n_V$ is a projector $V \to V$ with image $\overline{\mathcal{Z}}_{\lambda}(V)$. Further, $p_V \circ \varphi_V = \varphi_V \circ p_V$ and $n_V \circ \varphi_V = \varphi_V \circ n_V$. The image of the maps $p^V$ and $n^V$ is a subset of $\mathcal{Z}_{\lambda}(A)$ and $\overline{\mathcal{Z}}_{\lambda}(A)$, respectively.   
\end{lemma}

Our main objective is to determine whether by using defect lines we can differentiate all inequivalent spin structures of a given surface $\Sigma$. To do so we define the analogous of the elements $\eta$ and $\chi$, but now for $\Sigma_1-D$ with generating loops. Denote the partition function associated with $(\Sigma_1-D,l)$ as $Z(\Sigma_1-D,l,s)$.

\begin{proposition}
Fix the orientation of $l$ and let it come associated with the bimodule $V$. Then for each $(\Sigma_1-D,l)$ there are at most three different partition functions $Z(\Sigma_1-D,l,s)$. Further, these algebra elements belong to $\mathcal{Z}(A) \cap \mathcal{Z}_{\lambda}(A)$.
\end{proposition}
\begin{proof} The surface $\Sigma_1-D$ can be immersed into $\Rb^2 \subset \Rb^3$ in four non-equivalent ways, where the equivalence relation is regular homotopy. If defect loops are placed along the generating curves of $\Sigma_1-D$ their diagrammatic counterparts can have one of two forms,
\begin{align}
\begin{aligned}
\label{eq:two_forms}
%% Creator: Inkscape 0.48.2, www.inkscape.org
%% PDF/EPS/PS + LaTeX output extension by Johan Engelen, 2010
%% Accompanies image file 'generalised_eta.pdf' (pdf, eps, ps)
%%
%% To include the image in your LaTeX document, write
%%   \input{<filename>.pdf_tex}
%%  instead of
%%   \includegraphics{<filename>.pdf}
%% To scale the image, write
%%   \def\svgwidth{<desired width>}
%%   \input{<filename>.pdf_tex}
%%  instead of
%%   \includegraphics[width=<desired width>]{<filename>.pdf}
%%
%% Images with a different path to the parent latex file can
%% be accessed with the `import' package (which may need to be
%% installed) using
%%   \usepackage{import}
%% in the preamble, and then including the image with
%%   \import{<path to file>}{<filename>.pdf_tex}
%% Alternatively, one can specify
%%   \graphicspath{{<path to file>/}}
%% 
%% For more information, please see info/svg-inkscape on CTAN:
%%   http://tug.ctan.org/tex-archive/info/svg-inkscape
%%
\begingroup%
  \makeatletter%
  \providecommand\color[2][]{%
    \errmessage{(Inkscape) Color is used for the text in Inkscape, but the package 'color.sty' is not loaded}%
    \renewcommand\color[2][]{}%
  }%
  \providecommand\transparent[1]{%
    \errmessage{(Inkscape) Transparency is used (non-zero) for the text in Inkscape, but the package 'transparent.sty' is not loaded}%
    \renewcommand\transparent[1]{}%
  }%
  \providecommand\rotatebox[2]{#2}%
  \ifx\svgwidth\undefined%
    \setlength{\unitlength}{164.34609375bp}%
    \ifx\svgscale\undefined%
      \relax%
    \else%
      \setlength{\unitlength}{\unitlength * \real{\svgscale}}%
    \fi%
  \else%
    \setlength{\unitlength}{\svgwidth}%
  \fi%
  \global\let\svgwidth\undefined%
  \global\let\svgscale\undefined%
  \makeatother%
  \begin{picture}(1,0.41455073)%
    \put(0,0){\includegraphics[width=\unitlength]{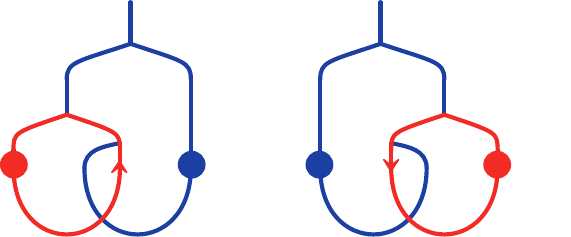}}%
    \put(0.42349653,0.11871445){\color[rgb]{0,0,0}\makebox(0,0)[lb]{\smash{or}}}%
    \put(0.93461303,0.11871445){\color[rgb]{0,0,0}\makebox(0,0)[lb]{\smash{$,$}}}%
  \end{picture}%
\endgroup%

\end{aligned}
\end{align}
where $\begin{aligned}%% Creator: Inkscape 0.48.2, www.inkscape.org
%% PDF/EPS/PS + LaTeX output extension by Johan Engelen, 2010
%% Accompanies image file 'blob_red2.pdf' (pdf, eps, ps)
%%
%% To include the image in your LaTeX document, write
%%   \input{<filename>.pdf_tex}
%%  instead of
%%   \includegraphics{<filename>.pdf}
%% To scale the image, write
%%   \def\svgwidth{<desired width>}
%%   \input{<filename>.pdf_tex}
%%  instead of
%%   \includegraphics[width=<desired width>]{<filename>.pdf}
%%
%% Images with a different path to the parent latex file can
%% be accessed with the `import' package (which may need to be
%% installed) using
%%   \usepackage{import}
%% in the preamble, and then including the image with
%%   \import{<path to file>}{<filename>.pdf_tex}
%% Alternatively, one can specify
%%   \graphicspath{{<path to file>/}}
%% 
%% For more information, please see info/svg-inkscape on CTAN:
%%   http://tug.ctan.org/tex-archive/info/svg-inkscape
%%
\begingroup%
  \makeatletter%
  \providecommand\color[2][]{%
    \errmessage{(Inkscape) Color is used for the text in Inkscape, but the package 'color.sty' is not loaded}%
    \renewcommand\color[2][]{}%
  }%
  \providecommand\transparent[1]{%
    \errmessage{(Inkscape) Transparency is used (non-zero) for the text in Inkscape, but the package 'transparent.sty' is not loaded}%
    \renewcommand\transparent[1]{}%
  }%
  \providecommand\rotatebox[2]{#2}%
  \ifx\svgwidth\undefined%
    \setlength{\unitlength}{8.4000037bp}%
    \ifx\svgscale\undefined%
      \relax%
    \else%
      \setlength{\unitlength}{\unitlength * \real{\svgscale}}%
    \fi%
  \else%
    \setlength{\unitlength}{\svgwidth}%
  \fi%
  \global\let\svgwidth\undefined%
  \global\let\svgscale\undefined%
  \makeatother%
  \begin{picture}(1,1.61712179)%
    \put(0,0){\includegraphics[width=\unitlength]{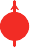}}%
  \end{picture}%
\endgroup%
\end{aligned}=\varphi_{V},\text{id}_{V}$, $\begin{aligned}%% Creator: Inkscape 0.48.2, www.inkscape.org
%% PDF/EPS/PS + LaTeX output extension by Johan Engelen, 2010
%% Accompanies image file 'blob_red1.pdf' (pdf, eps, ps)
%%
%% To include the image in your LaTeX document, write
%%   \input{<filename>.pdf_tex}
%%  instead of
%%   \includegraphics{<filename>.pdf}
%% To scale the image, write
%%   \def\svgwidth{<desired width>}
%%   \input{<filename>.pdf_tex}
%%  instead of
%%   \includegraphics[width=<desired width>]{<filename>.pdf}
%%
%% Images with a different path to the parent latex file can
%% be accessed with the `import' package (which may need to be
%% installed) using
%%   \usepackage{import}
%% in the preamble, and then including the image with
%%   \import{<path to file>}{<filename>.pdf_tex}
%% Alternatively, one can specify
%%   \graphicspath{{<path to file>/}}
%% 
%% For more information, please see info/svg-inkscape on CTAN:
%%   http://tug.ctan.org/tex-archive/info/svg-inkscape
%%
\begingroup%
  \makeatletter%
  \providecommand\color[2][]{%
    \errmessage{(Inkscape) Color is used for the text in Inkscape, but the package 'color.sty' is not loaded}%
    \renewcommand\color[2][]{}%
  }%
  \providecommand\transparent[1]{%
    \errmessage{(Inkscape) Transparency is used (non-zero) for the text in Inkscape, but the package 'transparent.sty' is not loaded}%
    \renewcommand\transparent[1]{}%
  }%
  \providecommand\rotatebox[2]{#2}%
  \ifx\svgwidth\undefined%
    \setlength{\unitlength}{8.4000037bp}%
    \ifx\svgscale\undefined%
      \relax%
    \else%
      \setlength{\unitlength}{\unitlength * \real{\svgscale}}%
    \fi%
  \else%
    \setlength{\unitlength}{\svgwidth}%
  \fi%
  \global\let\svgwidth\undefined%
  \global\let\svgscale\undefined%
  \makeatother%
  \begin{picture}(1,1.61712179)%
    \put(0,0){\includegraphics[width=\unitlength]{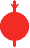}}%
  \end{picture}%
\endgroup%
\end{aligned}=\varphi_{V^{\times}},\text{id}_{V^{\times}}$ and $\begin{aligned}%% Creator: Inkscape 0.48.2, www.inkscape.org
%% PDF/EPS/PS + LaTeX output extension by Johan Engelen, 2010
%% Accompanies image file 'blob.pdf' (pdf, eps, ps)
%%
%% To include the image in your LaTeX document, write
%%   \input{<filename>.pdf_tex}
%%  instead of
%%   \includegraphics{<filename>.pdf}
%% To scale the image, write
%%   \def\svgwidth{<desired width>}
%%   \input{<filename>.pdf_tex}
%%  instead of
%%   \includegraphics[width=<desired width>]{<filename>.pdf}
%%
%% Images with a different path to the parent latex file can
%% be accessed with the `import' package (which may need to be
%% installed) using
%%   \usepackage{import}
%% in the preamble, and then including the image with
%%   \import{<path to file>}{<filename>.pdf_tex}
%% Alternatively, one can specify
%%   \graphicspath{{<path to file>/}}
%% 
%% For more information, please see info/svg-inkscape on CTAN:
%%   http://tug.ctan.org/tex-archive/info/svg-inkscape
%%
\begingroup%
  \makeatletter%
  \providecommand\color[2][]{%
    \errmessage{(Inkscape) Color is used for the text in Inkscape, but the package 'color.sty' is not loaded}%
    \renewcommand\color[2][]{}%
  }%
  \providecommand\transparent[1]{%
    \errmessage{(Inkscape) Transparency is used (non-zero) for the text in Inkscape, but the package 'transparent.sty' is not loaded}%
    \renewcommand\transparent[1]{}%
  }%
  \providecommand\rotatebox[2]{#2}%
  \ifx\svgwidth\undefined%
    \setlength{\unitlength}{8.40000381bp}%
    \ifx\svgscale\undefined%
      \relax%
    \else%
      \setlength{\unitlength}{\unitlength * \real{\svgscale}}%
    \fi%
  \else%
    \setlength{\unitlength}{\svgwidth}%
  \fi%
  \global\let\svgwidth\undefined%
  \global\let\svgscale\undefined%
  \makeatother%
  \begin{picture}(1,1.61712173)%
    \put(0,0){\includegraphics[width=\unitlength]{blob.pdf}}%
  \end{picture}%
\endgroup%
\end{aligned}=\varphi_A,\text{id}_A$. This means that one could have eight non-equivalent partition functions to associate with $(\Sigma_1-D,l)$. First it is established that the two forms of expression \eqref{eq:two_forms} are equivalent. By abuse of notation, equation \eqref{eq:closed_below} is used to refer not only to algebra diagrams but all defect diagrams instead. The properties \ref{it:phi_one}-\ref{it:phi_three} of $\varphi$, established in lemma~\ref{lem:varphi}, can also be extended to the map $\varphi_V$.
$$
\hspace{1cm}
%% Creator: Inkscape 0.48.2, www.inkscape.org
%% PDF/EPS/PS + LaTeX output extension by Johan Engelen, 2010
%% Accompanies image file 'generalised_eta_symmetry1.pdf' (pdf, eps, ps)
%%
%% To include the image in your LaTeX document, write
%%   \input{<filename>.pdf_tex}
%%  instead of
%%   \includegraphics{<filename>.pdf}
%% To scale the image, write
%%   \def\svgwidth{<desired width>}
%%   \input{<filename>.pdf_tex}
%%  instead of
%%   \includegraphics[width=<desired width>]{<filename>.pdf}
%%
%% Images with a different path to the parent latex file can
%% be accessed with the `import' package (which may need to be
%% installed) using
%%   \usepackage{import}
%% in the preamble, and then including the image with
%%   \import{<path to file>}{<filename>.pdf_tex}
%% Alternatively, one can specify
%%   \graphicspath{{<path to file>/}}
%% 
%% For more information, please see info/svg-inkscape on CTAN:
%%   http://tug.ctan.org/tex-archive/info/svg-inkscape
%%
\begingroup%
  \makeatletter%
  \providecommand\color[2][]{%
    \errmessage{(Inkscape) Color is used for the text in Inkscape, but the package 'color.sty' is not loaded}%
    \renewcommand\color[2][]{}%
  }%
  \providecommand\transparent[1]{%
    \errmessage{(Inkscape) Transparency is used (non-zero) for the text in Inkscape, but the package 'transparent.sty' is not loaded}%
    \renewcommand\transparent[1]{}%
  }%
  \providecommand\rotatebox[2]{#2}%
  \ifx\svgwidth\undefined%
    \setlength{\unitlength}{398.92421875bp}%
    \ifx\svgscale\undefined%
      \relax%
    \else%
      \setlength{\unitlength}{\unitlength * \real{\svgscale}}%
    \fi%
  \else%
    \setlength{\unitlength}{\svgwidth}%
  \fi%
  \global\let\svgwidth\undefined%
  \global\let\svgscale\undefined%
  \makeatother%
  \begin{picture}(1,0.24222593)%
    \put(0,0){\includegraphics[width=\unitlength]{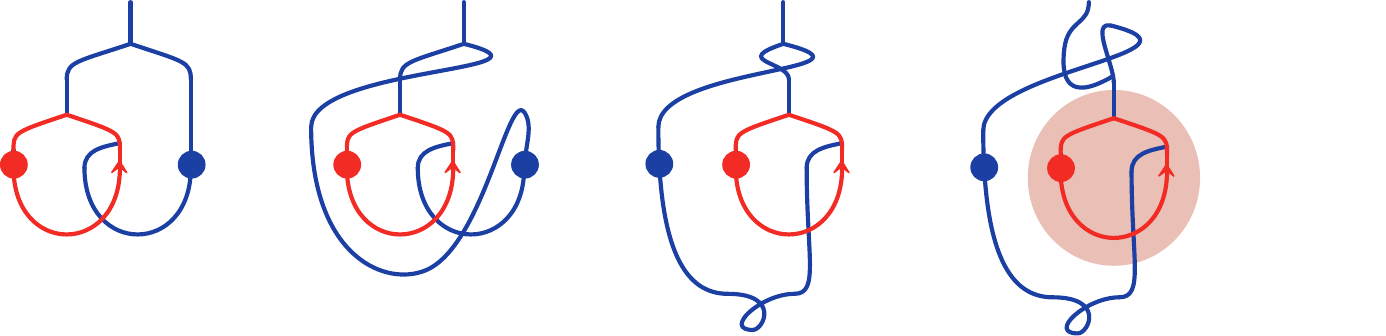}}%
    \put(0.17446923,0.12034931){\color[rgb]{0,0,0}\makebox(0,0)[lb]{\smash{$\overset{\eqref{eq:closed_below}}{=}$}}}%
    \put(0.64573668,0.12034931){\color[rgb]{0,0,0}\makebox(0,0)[lb]{\smash{$\overset{\text{\ref{fig:axiom_mult}}}{=}$}}}%
    \put(0.41511644,0.12034931){\color[rgb]{0,0,0}\makebox(0,0)[lb]{\smash{$\overset{\text{\ref{it:phi_three}}}{=}$}}}%
  \end{picture}%
\endgroup%

$$ 
$$
%% Creator: Inkscape 0.48.2, www.inkscape.org
%% PDF/EPS/PS + LaTeX output extension by Johan Engelen, 2010
%% Accompanies image file 'generalised_eta_symmetry2.pdf' (pdf, eps, ps)
%%
%% To include the image in your LaTeX document, write
%%   \input{<filename>.pdf_tex}
%%  instead of
%%   \includegraphics{<filename>.pdf}
%% To scale the image, write
%%   \def\svgwidth{<desired width>}
%%   \input{<filename>.pdf_tex}
%%  instead of
%%   \includegraphics[width=<desired width>]{<filename>.pdf}
%%
%% Images with a different path to the parent latex file can
%% be accessed with the `import' package (which may need to be
%% installed) using
%%   \usepackage{import}
%% in the preamble, and then including the image with
%%   \import{<path to file>}{<filename>.pdf_tex}
%% Alternatively, one can specify
%%   \graphicspath{{<path to file>/}}
%% 
%% For more information, please see info/svg-inkscape on CTAN:
%%   http://tug.ctan.org/tex-archive/info/svg-inkscape
%%
\begingroup%
  \makeatletter%
  \providecommand\color[2][]{%
    \errmessage{(Inkscape) Color is used for the text in Inkscape, but the package 'color.sty' is not loaded}%
    \renewcommand\color[2][]{}%
  }%
  \providecommand\transparent[1]{%
    \errmessage{(Inkscape) Transparency is used (non-zero) for the text in Inkscape, but the package 'transparent.sty' is not loaded}%
    \renewcommand\transparent[1]{}%
  }%
  \providecommand\rotatebox[2]{#2}%
  \ifx\svgwidth\undefined%
    \setlength{\unitlength}{511.0078125bp}%
    \ifx\svgscale\undefined%
      \relax%
    \else%
      \setlength{\unitlength}{\unitlength * \real{\svgscale}}%
    \fi%
  \else%
    \setlength{\unitlength}{\svgwidth}%
  \fi%
  \global\let\svgwidth\undefined%
  \global\let\svgscale\undefined%
  \makeatother%
  \begin{picture}(1,0.1890965)%
    \put(0,0){\includegraphics[width=\unitlength]{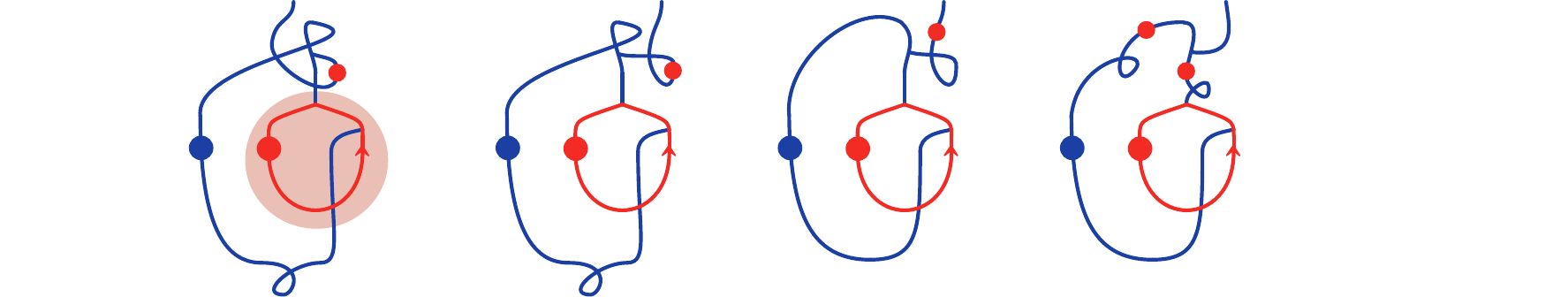}}%
    \put(-0.00055803,0.09395209){\color[rgb]{0,0,0}\makebox(0,0)[lb]{\smash{$\overset{\substack{p^V(a) \in \mathcal{Z}_{\lambda}(A) \\ n^V(a) \in \overline{\mathcal{Z}}_{\lambda}(A)}}{=}$}}}%
    \put(0.27341039,0.09395209){\color[rgb]{0,0,0}\makebox(0,0)[lb]{\smash{$\overset{\substack{\text{\ref{it:phi_three}} \\ \text{\ref{fig:axiom_mult}}\\ \text{\ref{fig:axiom_square}}}}{=}$}}}%
    \put(0.45344677,0.09395209){\color[rgb]{0,0,0}\makebox(0,0)[lb]{\smash{$\overset{\substack{\text{\ref{it:phi_three}} \\ \text{\ref{it:phi_one}}}}{=}$}}}%
    \put(0.63348316,0.09395209){\color[rgb]{0,0,0}\makebox(0,0)[lb]{\smash{$\overset{\text{\ref{it:phi_two}}}{=}$}}}%
  \end{picture}%
\endgroup%

$$ 
$$
\hspace{4mm}
%% Creator: Inkscape 0.48.2, www.inkscape.org
%% PDF/EPS/PS + LaTeX output extension by Johan Engelen, 2010
%% Accompanies image file 'generalised_eta_symmetry4.pdf' (pdf, eps, ps)
%%
%% To include the image in your LaTeX document, write
%%   \input{<filename>.pdf_tex}
%%  instead of
%%   \includegraphics{<filename>.pdf}
%% To scale the image, write
%%   \def\svgwidth{<desired width>}
%%   \input{<filename>.pdf_tex}
%%  instead of
%%   \includegraphics[width=<desired width>]{<filename>.pdf}
%%
%% Images with a different path to the parent latex file can
%% be accessed with the `import' package (which may need to be
%% installed) using
%%   \usepackage{import}
%% in the preamble, and then including the image with
%%   \import{<path to file>}{<filename>.pdf_tex}
%% Alternatively, one can specify
%%   \graphicspath{{<path to file>/}}
%% 
%% For more information, please see info/svg-inkscape on CTAN:
%%   http://tug.ctan.org/tex-archive/info/svg-inkscape
%%
\begingroup%
  \makeatletter%
  \providecommand\color[2][]{%
    \errmessage{(Inkscape) Color is used for the text in Inkscape, but the package 'color.sty' is not loaded}%
    \renewcommand\color[2][]{}%
  }%
  \providecommand\transparent[1]{%
    \errmessage{(Inkscape) Transparency is used (non-zero) for the text in Inkscape, but the package 'transparent.sty' is not loaded}%
    \renewcommand\transparent[1]{}%
  }%
  \providecommand\rotatebox[2]{#2}%
  \ifx\svgwidth\undefined%
    \setlength{\unitlength}{441.0546875bp}%
    \ifx\svgscale\undefined%
      \relax%
    \else%
      \setlength{\unitlength}{\unitlength * \real{\svgscale}}%
    \fi%
  \else%
    \setlength{\unitlength}{\svgwidth}%
  \fi%
  \global\let\svgwidth\undefined%
  \global\let\svgscale\undefined%
  \makeatother%
  \begin{picture}(1,0.1966669)%
    \put(0,0){\includegraphics[width=\unitlength]{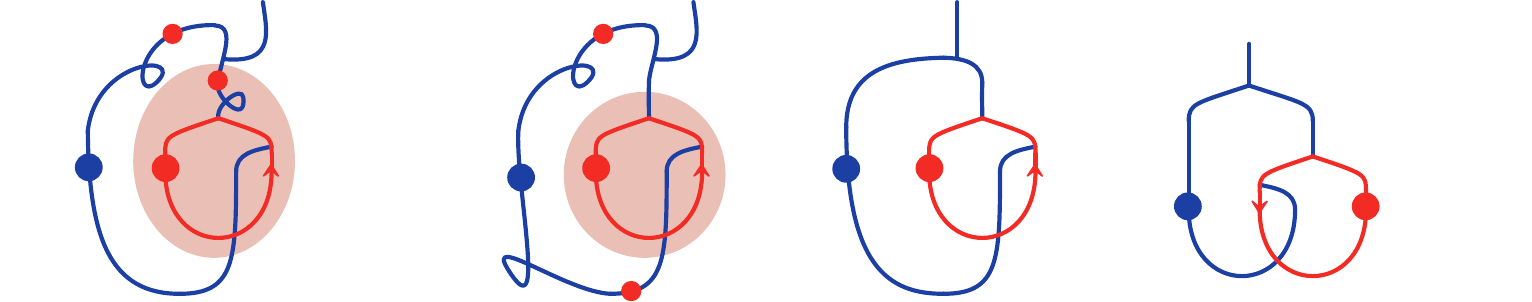}}%
    \put(-0.00064653,0.08643219){\color[rgb]{0,0,0}\makebox(0,0)[lb]{\smash{$=$}}}%
    \put(0.71581791,0.08643219){\color[rgb]{0,0,0}\makebox(0,0)[lb]{\smash{$=$}}}%
    \put(0.21701355,0.08643219){\color[rgb]{0,0,0}\makebox(0,0)[lb]{\smash{$\overset{\substack{p^V \circ \varphi = \varphi \circ p^V \\ n^V \circ \varphi = \varphi \circ n^V}}{=}$}}}%
    \put(0.49815782,0.08643219){\color[rgb]{0,0,0}\makebox(0,0)[lb]{\smash{$\overset{\text{\ref{it:phi_two}}}{=}$}}}%
  \end{picture}%
\endgroup%

$$ 
Next one shows that for $\begin{aligned}%% Creator: Inkscape 0.48.2, www.inkscape.org
%% PDF/EPS/PS + LaTeX output extension by Johan Engelen, 2010
%% Accompanies image file 'blob_red1.pdf' (pdf, eps, ps)
%%
%% To include the image in your LaTeX document, write
%%   \input{<filename>.pdf_tex}
%%  instead of
%%   \includegraphics{<filename>.pdf}
%% To scale the image, write
%%   \def\svgwidth{<desired width>}
%%   \input{<filename>.pdf_tex}
%%  instead of
%%   \includegraphics[width=<desired width>]{<filename>.pdf}
%%
%% Images with a different path to the parent latex file can
%% be accessed with the `import' package (which may need to be
%% installed) using
%%   \usepackage{import}
%% in the preamble, and then including the image with
%%   \import{<path to file>}{<filename>.pdf_tex}
%% Alternatively, one can specify
%%   \graphicspath{{<path to file>/}}
%% 
%% For more information, please see info/svg-inkscape on CTAN:
%%   http://tug.ctan.org/tex-archive/info/svg-inkscape
%%
\begingroup%
  \makeatletter%
  \providecommand\color[2][]{%
    \errmessage{(Inkscape) Color is used for the text in Inkscape, but the package 'color.sty' is not loaded}%
    \renewcommand\color[2][]{}%
  }%
  \providecommand\transparent[1]{%
    \errmessage{(Inkscape) Transparency is used (non-zero) for the text in Inkscape, but the package 'transparent.sty' is not loaded}%
    \renewcommand\transparent[1]{}%
  }%
  \providecommand\rotatebox[2]{#2}%
  \ifx\svgwidth\undefined%
    \setlength{\unitlength}{8.4000037bp}%
    \ifx\svgscale\undefined%
      \relax%
    \else%
      \setlength{\unitlength}{\unitlength * \real{\svgscale}}%
    \fi%
  \else%
    \setlength{\unitlength}{\svgwidth}%
  \fi%
  \global\let\svgwidth\undefined%
  \global\let\svgscale\undefined%
  \makeatother%
  \begin{picture}(1,1.61712179)%
    \put(0,0){\includegraphics[width=\unitlength]{blob_red1.pdf}}%
  \end{picture}%
\endgroup%
\end{aligned}=\text{id}_{V^{\times}}$ having $\begin{aligned}%% Creator: Inkscape 0.48.2, www.inkscape.org
%% PDF/EPS/PS + LaTeX output extension by Johan Engelen, 2010
%% Accompanies image file 'blob.pdf' (pdf, eps, ps)
%%
%% To include the image in your LaTeX document, write
%%   \input{<filename>.pdf_tex}
%%  instead of
%%   \includegraphics{<filename>.pdf}
%% To scale the image, write
%%   \def\svgwidth{<desired width>}
%%   \input{<filename>.pdf_tex}
%%  instead of
%%   \includegraphics[width=<desired width>]{<filename>.pdf}
%%
%% Images with a different path to the parent latex file can
%% be accessed with the `import' package (which may need to be
%% installed) using
%%   \usepackage{import}
%% in the preamble, and then including the image with
%%   \import{<path to file>}{<filename>.pdf_tex}
%% Alternatively, one can specify
%%   \graphicspath{{<path to file>/}}
%% 
%% For more information, please see info/svg-inkscape on CTAN:
%%   http://tug.ctan.org/tex-archive/info/svg-inkscape
%%
\begingroup%
  \makeatletter%
  \providecommand\color[2][]{%
    \errmessage{(Inkscape) Color is used for the text in Inkscape, but the package 'color.sty' is not loaded}%
    \renewcommand\color[2][]{}%
  }%
  \providecommand\transparent[1]{%
    \errmessage{(Inkscape) Transparency is used (non-zero) for the text in Inkscape, but the package 'transparent.sty' is not loaded}%
    \renewcommand\transparent[1]{}%
  }%
  \providecommand\rotatebox[2]{#2}%
  \ifx\svgwidth\undefined%
    \setlength{\unitlength}{8.40000381bp}%
    \ifx\svgscale\undefined%
      \relax%
    \else%
      \setlength{\unitlength}{\unitlength * \real{\svgscale}}%
    \fi%
  \else%
    \setlength{\unitlength}{\svgwidth}%
  \fi%
  \global\let\svgwidth\undefined%
  \global\let\svgscale\undefined%
  \makeatother%
  \begin{picture}(1,1.61712173)%
    \put(0,0){\includegraphics[width=\unitlength]{blob.pdf}}%
  \end{picture}%
\endgroup%
\end{aligned}=\varphi_A$ or $\begin{aligned}%% Creator: Inkscape 0.48.2, www.inkscape.org
%% PDF/EPS/PS + LaTeX output extension by Johan Engelen, 2010
%% Accompanies image file 'blob.pdf' (pdf, eps, ps)
%%
%% To include the image in your LaTeX document, write
%%   \input{<filename>.pdf_tex}
%%  instead of
%%   \includegraphics{<filename>.pdf}
%% To scale the image, write
%%   \def\svgwidth{<desired width>}
%%   \input{<filename>.pdf_tex}
%%  instead of
%%   \includegraphics[width=<desired width>]{<filename>.pdf}
%%
%% Images with a different path to the parent latex file can
%% be accessed with the `import' package (which may need to be
%% installed) using
%%   \usepackage{import}
%% in the preamble, and then including the image with
%%   \import{<path to file>}{<filename>.pdf_tex}
%% Alternatively, one can specify
%%   \graphicspath{{<path to file>/}}
%% 
%% For more information, please see info/svg-inkscape on CTAN:
%%   http://tug.ctan.org/tex-archive/info/svg-inkscape
%%
\begingroup%
  \makeatletter%
  \providecommand\color[2][]{%
    \errmessage{(Inkscape) Color is used for the text in Inkscape, but the package 'color.sty' is not loaded}%
    \renewcommand\color[2][]{}%
  }%
  \providecommand\transparent[1]{%
    \errmessage{(Inkscape) Transparency is used (non-zero) for the text in Inkscape, but the package 'transparent.sty' is not loaded}%
    \renewcommand\transparent[1]{}%
  }%
  \providecommand\rotatebox[2]{#2}%
  \ifx\svgwidth\undefined%
    \setlength{\unitlength}{8.40000381bp}%
    \ifx\svgscale\undefined%
      \relax%
    \else%
      \setlength{\unitlength}{\unitlength * \real{\svgscale}}%
    \fi%
  \else%
    \setlength{\unitlength}{\svgwidth}%
  \fi%
  \global\let\svgwidth\undefined%
  \global\let\svgscale\undefined%
  \makeatother%
  \begin{picture}(1,1.61712173)%
    \put(0,0){\includegraphics[width=\unitlength]{blob.pdf}}%
  \end{picture}%
\endgroup%
\end{aligned}=\text{id}_A$ is equivalent. This can be seen merely by recalling that $p^V \circ \varphi_A=p^V$. Finally one shows that all the partition functions $Z(\Sigma_1-D,l,s)$ are central elements and, as a consequence, they also belong to $\mathcal{Z}_{\lambda}(A)$.
$$
%% Creator: Inkscape 0.48.2, www.inkscape.org
%% PDF/EPS/PS + LaTeX output extension by Johan Engelen, 2010
%% Accompanies image file 'etacommutativity_defects1.pdf' (pdf, eps, ps)
%%
%% To include the image in your LaTeX document, write
%%   \input{<filename>.pdf_tex}
%%  instead of
%%   \includegraphics{<filename>.pdf}
%% To scale the image, write
%%   \def\svgwidth{<desired width>}
%%   \input{<filename>.pdf_tex}
%%  instead of
%%   \includegraphics[width=<desired width>]{<filename>.pdf}
%%
%% Images with a different path to the parent latex file can
%% be accessed with the `import' package (which may need to be
%% installed) using
%%   \usepackage{import}
%% in the preamble, and then including the image with
%%   \import{<path to file>}{<filename>.pdf_tex}
%% Alternatively, one can specify
%%   \graphicspath{{<path to file>/}}
%% 
%% For more information, please see info/svg-inkscape on CTAN:
%%   http://tug.ctan.org/tex-archive/info/svg-inkscape
%%
\begingroup%
  \makeatletter%
  \providecommand\color[2][]{%
    \errmessage{(Inkscape) Color is used for the text in Inkscape, but the package 'color.sty' is not loaded}%
    \renewcommand\color[2][]{}%
  }%
  \providecommand\transparent[1]{%
    \errmessage{(Inkscape) Transparency is used (non-zero) for the text in Inkscape, but the package 'transparent.sty' is not loaded}%
    \renewcommand\transparent[1]{}%
  }%
  \providecommand\rotatebox[2]{#2}%
  \ifx\svgwidth\undefined%
    \setlength{\unitlength}{602.90865375bp}%
    \ifx\svgscale\undefined%
      \relax%
    \else%
      \setlength{\unitlength}{\unitlength * \real{\svgscale}}%
    \fi%
  \else%
    \setlength{\unitlength}{\svgwidth}%
  \fi%
  \global\let\svgwidth\undefined%
  \global\let\svgscale\undefined%
  \makeatother%
  \begin{picture}(1,0.14151907)%
    \put(0,0){\includegraphics[width=\unitlength]{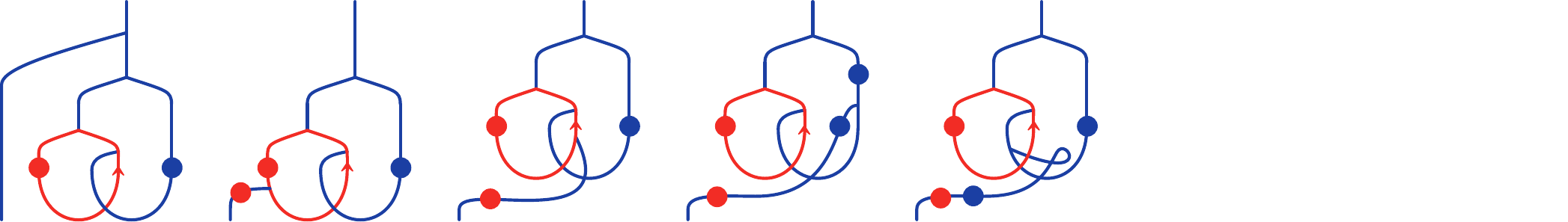}}%
    \put(0.12041635,0.06087158){\color[rgb]{0,0,0}\makebox(0,0)[lb]{\smash{$\overset{\substack{\eqref{fig:diagram_mult2} \\ \text{\ref{it:phi_two}}}}{=}$}}}%
    \put(0.26637544,0.06087158){\color[rgb]{0,0,0}\makebox(0,0)[lb]{\smash{$\overset{\text{\ref{fig:axiom_mult}}}{=}$}}}%
    \put(0.41233454,0.06087158){\color[rgb]{0,0,0}\makebox(0,0)[lb]{\smash{$\overset{\substack{\eqref{fig:diagram_mult2} \\ \text{\ref{it:phi_two}}}}{=}$}}}%
    \put(0.55829363,0.06087158){\color[rgb]{0,0,0}\makebox(0,0)[lb]{\smash{$\overset{\substack{\text{\ref{fig:axiom_mult}} \\ \text{\ref{it:phi_three}}}}{=}$}}}%
  \end{picture}%
\endgroup%

$$ 
$$
\hspace{25mm}
%% Creator: Inkscape 0.48.2, www.inkscape.org
%% PDF/EPS/PS + LaTeX output extension by Johan Engelen, 2010
%% Accompanies image file 'etacommutativity_defects_2.pdf' (pdf, eps, ps)
%%
%% To include the image in your LaTeX document, write
%%   \input{<filename>.pdf_tex}
%%  instead of
%%   \includegraphics{<filename>.pdf}
%% To scale the image, write
%%   \def\svgwidth{<desired width>}
%%   \input{<filename>.pdf_tex}
%%  instead of
%%   \includegraphics[width=<desired width>]{<filename>.pdf}
%%
%% Images with a different path to the parent latex file can
%% be accessed with the `import' package (which may need to be
%% installed) using
%%   \usepackage{import}
%% in the preamble, and then including the image with
%%   \import{<path to file>}{<filename>.pdf_tex}
%% Alternatively, one can specify
%%   \graphicspath{{<path to file>/}}
%% 
%% For more information, please see info/svg-inkscape on CTAN:
%%   http://tug.ctan.org/tex-archive/info/svg-inkscape
%%
\begingroup%
  \makeatletter%
  \providecommand\color[2][]{%
    \errmessage{(Inkscape) Color is used for the text in Inkscape, but the package 'color.sty' is not loaded}%
    \renewcommand\color[2][]{}%
  }%
  \providecommand\transparent[1]{%
    \errmessage{(Inkscape) Transparency is used (non-zero) for the text in Inkscape, but the package 'transparent.sty' is not loaded}%
    \renewcommand\transparent[1]{}%
  }%
  \providecommand\rotatebox[2]{#2}%
  \ifx\svgwidth\undefined%
    \setlength{\unitlength}{404.5546875bp}%
    \ifx\svgscale\undefined%
      \relax%
    \else%
      \setlength{\unitlength}{\unitlength * \real{\svgscale}}%
    \fi%
  \else%
    \setlength{\unitlength}{\svgwidth}%
  \fi%
  \global\let\svgwidth\undefined%
  \global\let\svgscale\undefined%
  \makeatother%
  \begin{picture}(1,0.21065564)%
    \put(0,0){\includegraphics[width=\unitlength]{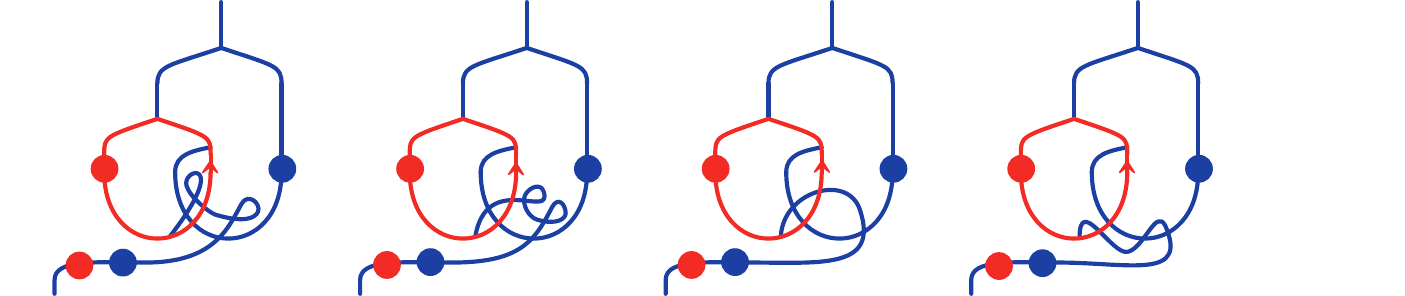}}%
    \put(0.65186451,0.09049097){\color[rgb]{0,0,0}\makebox(0,0)[lb]{\smash{$\overset{\text{\ref{it: RIII_defects}}}{=}$}}}%
    \put(-0.00070486,0.09049097){\color[rgb]{0,0,0}\makebox(0,0)[lb]{\smash{$\overset{\substack{\eqref{fig:diagram_mult2} \\ \text{\ref{it:C_defects}}}}{=}$}}}%
    \put(0.21681826,0.09049097){\color[rgb]{0,0,0}\makebox(0,0)[lb]{\smash{$\overset{\text{\ref{it:phi_three}}}{=}$}}}%
    \put(0.43434139,0.09049097){\color[rgb]{0,0,0}\makebox(0,0)[lb]{\smash{$\overset{\text{\ref{it:phi_one}}}{=}$}}}%
  \end{picture}%
\endgroup%

$$ 
$$
\hspace{25mm}
%% Creator: Inkscape 0.48.2, www.inkscape.org
%% PDF/EPS/PS + LaTeX output extension by Johan Engelen, 2010
%% Accompanies image file 'etacommutativity_defects_3.pdf' (pdf, eps, ps)
%%
%% To include the image in your LaTeX document, write
%%   \input{<filename>.pdf_tex}
%%  instead of
%%   \includegraphics{<filename>.pdf}
%% To scale the image, write
%%   \def\svgwidth{<desired width>}
%%   \input{<filename>.pdf_tex}
%%  instead of
%%   \includegraphics[width=<desired width>]{<filename>.pdf}
%%
%% Images with a different path to the parent latex file can
%% be accessed with the `import' package (which may need to be
%% installed) using
%%   \usepackage{import}
%% in the preamble, and then including the image with
%%   \import{<path to file>}{<filename>.pdf_tex}
%% Alternatively, one can specify
%%   \graphicspath{{<path to file>/}}
%% 
%% For more information, please see info/svg-inkscape on CTAN:
%%   http://tug.ctan.org/tex-archive/info/svg-inkscape
%%
\begingroup%
  \makeatletter%
  \providecommand\color[2][]{%
    \errmessage{(Inkscape) Color is used for the text in Inkscape, but the package 'color.sty' is not loaded}%
    \renewcommand\color[2][]{}%
  }%
  \providecommand\transparent[1]{%
    \errmessage{(Inkscape) Transparency is used (non-zero) for the text in Inkscape, but the package 'transparent.sty' is not loaded}%
    \renewcommand\transparent[1]{}%
  }%
  \providecommand\rotatebox[2]{#2}%
  \ifx\svgwidth\undefined%
    \setlength{\unitlength}{391.6796875bp}%
    \ifx\svgscale\undefined%
      \relax%
    \else%
      \setlength{\unitlength}{\unitlength * \real{\svgscale}}%
    \fi%
  \else%
    \setlength{\unitlength}{\svgwidth}%
  \fi%
  \global\let\svgwidth\undefined%
  \global\let\svgscale\undefined%
  \makeatother%
  \begin{picture}(1,0.21758016)%
    \put(0,0){\includegraphics[width=\unitlength]{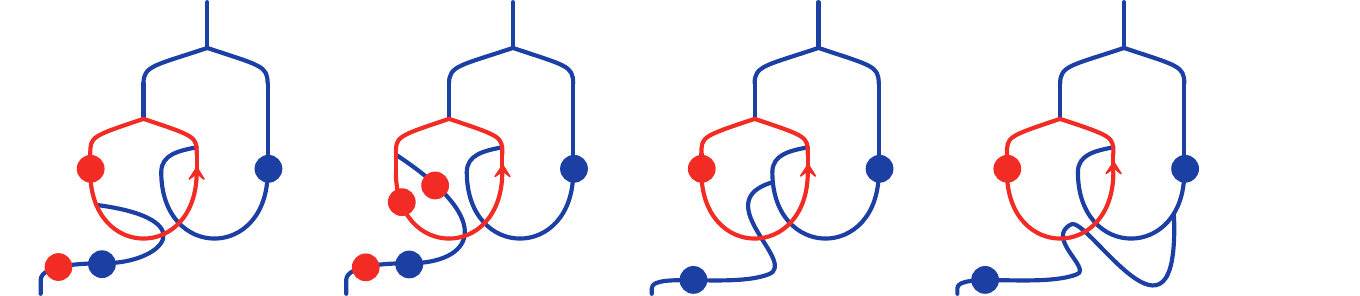}}%
    \put(0.4384063,0.08325309){\color[rgb]{0,0,0}\makebox(0,0)[lb]{\smash{$=$}}}%
    \put(0.21373292,0.08325309){\color[rgb]{0,0,0}\makebox(0,0)[lb]{\smash{$\overset{\text{\ref{it:phi_two}}}{=}$}}}%
    \put(-0.00072803,0.08325309){\color[rgb]{0,0,0}\makebox(0,0)[lb]{\smash{$\overset{\text{\ref{fig:axiom_square}}}{=}$}}}%
    \put(0.66307968,0.08325309){\color[rgb]{0,0,0}\makebox(0,0)[lb]{\smash{$\overset{\text{\ref{it:C_defects}}}{=}$}}}%
  \end{picture}%
\endgroup%

$$ 
$$
\hspace{15mm}
%% Creator: Inkscape 0.48.2, www.inkscape.org
%% PDF/EPS/PS + LaTeX output extension by Johan Engelen, 2010
%% Accompanies image file 'etacommutativity_defects_4.pdf' (pdf, eps, ps)
%%
%% To include the image in your LaTeX document, write
%%   \input{<filename>.pdf_tex}
%%  instead of
%%   \includegraphics{<filename>.pdf}
%% To scale the image, write
%%   \def\svgwidth{<desired width>}
%%   \input{<filename>.pdf_tex}
%%  instead of
%%   \includegraphics[width=<desired width>]{<filename>.pdf}
%%
%% Images with a different path to the parent latex file can
%% be accessed with the `import' package (which may need to be
%% installed) using
%%   \usepackage{import}
%% in the preamble, and then including the image with
%%   \import{<path to file>}{<filename>.pdf_tex}
%% Alternatively, one can specify
%%   \graphicspath{{<path to file>/}}
%% 
%% For more information, please see info/svg-inkscape on CTAN:
%%   http://tug.ctan.org/tex-archive/info/svg-inkscape
%%
\begingroup%
  \makeatletter%
  \providecommand\color[2][]{%
    \errmessage{(Inkscape) Color is used for the text in Inkscape, but the package 'color.sty' is not loaded}%
    \renewcommand\color[2][]{}%
  }%
  \providecommand\transparent[1]{%
    \errmessage{(Inkscape) Transparency is used (non-zero) for the text in Inkscape, but the package 'transparent.sty' is not loaded}%
    \renewcommand\transparent[1]{}%
  }%
  \providecommand\rotatebox[2]{#2}%
  \ifx\svgwidth\undefined%
    \setlength{\unitlength}{314.3671875bp}%
    \ifx\svgscale\undefined%
      \relax%
    \else%
      \setlength{\unitlength}{\unitlength * \real{\svgscale}}%
    \fi%
  \else%
    \setlength{\unitlength}{\svgwidth}%
  \fi%
  \global\let\svgwidth\undefined%
  \global\let\svgscale\undefined%
  \makeatother%
  \begin{picture}(1,0.27108977)%
    \put(0,0){\includegraphics[width=\unitlength]{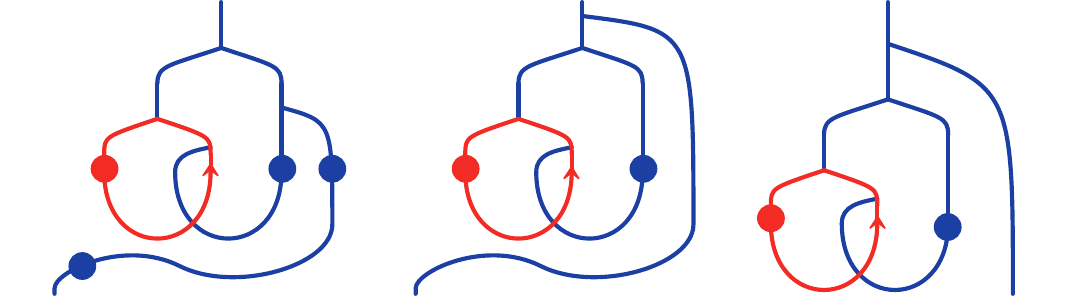}}%
    \put(-0.00090708,0.10372757){\color[rgb]{0,0,0}\makebox(0,0)[lb]{\smash{$\overset{\text{\ref{it: RII_defects}}}{=}$}}}%
    \put(0.32991625,0.10372757){\color[rgb]{0,0,0}\makebox(0,0)[lb]{\smash{$\overset{\text{\ref{it:phi_two}}}{=}$}}}%
    \put(0.64801561,0.10372757){\color[rgb]{0,0,0}\makebox(0,0)[lb]{\smash{$\overset{\eqref{eq:snake}}{=}$}}}%
  \end{picture}%
\endgroup%

$$ 
\end{proof}
The new elements are denoted as follows.
\begin{align}
\begin{aligned}
%% Creator: Inkscape 0.48.2, www.inkscape.org
%% PDF/EPS/PS + LaTeX output extension by Johan Engelen, 2010
%% Accompanies image file 'element_defects.pdf' (pdf, eps, ps)
%%
%% To include the image in your LaTeX document, write
%%   \input{<filename>.pdf_tex}
%%  instead of
%%   \includegraphics{<filename>.pdf}
%% To scale the image, write
%%   \def\svgwidth{<desired width>}
%%   \input{<filename>.pdf_tex}
%%  instead of
%%   \includegraphics[width=<desired width>]{<filename>.pdf}
%%
%% Images with a different path to the parent latex file can
%% be accessed with the `import' package (which may need to be
%% installed) using
%%   \usepackage{import}
%% in the preamble, and then including the image with
%%   \import{<path to file>}{<filename>.pdf_tex}
%% Alternatively, one can specify
%%   \graphicspath{{<path to file>/}}
%% 
%% For more information, please see info/svg-inkscape on CTAN:
%%   http://tug.ctan.org/tex-archive/info/svg-inkscape
%%
\begingroup%
  \makeatletter%
  \providecommand\color[2][]{%
    \errmessage{(Inkscape) Color is used for the text in Inkscape, but the package 'color.sty' is not loaded}%
    \renewcommand\color[2][]{}%
  }%
  \providecommand\transparent[1]{%
    \errmessage{(Inkscape) Transparency is used (non-zero) for the text in Inkscape, but the package 'transparent.sty' is not loaded}%
    \renewcommand\transparent[1]{}%
  }%
  \providecommand\rotatebox[2]{#2}%
  \ifx\svgwidth\undefined%
    \setlength{\unitlength}{292.20149807bp}%
    \ifx\svgscale\undefined%
      \relax%
    \else%
      \setlength{\unitlength}{\unitlength * \real{\svgscale}}%
    \fi%
  \else%
    \setlength{\unitlength}{\svgwidth}%
  \fi%
  \global\let\svgwidth\undefined%
  \global\let\svgscale\undefined%
  \makeatother%
  \begin{picture}(1,0.22665926)%
    \put(0,0){\includegraphics[width=\unitlength]{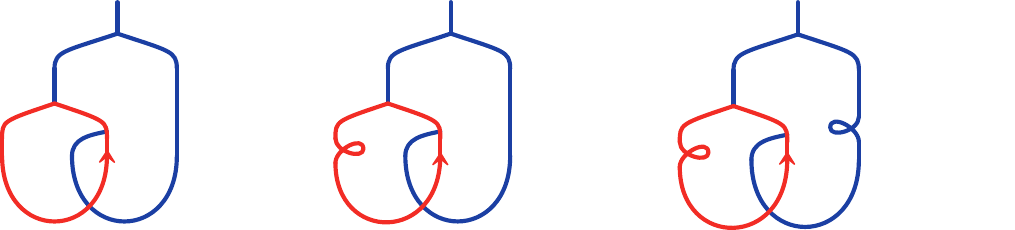}}%
    \put(0.18962443,0.08592135){\color[rgb]{0,0,0}\makebox(0,0)[lb]{\smash{$=\eta_{V}$}}}%
    \put(0.51816479,0.08592135){\color[rgb]{0,0,0}\makebox(0,0)[lb]{\smash{$=\rho_{V}$}}}%
    \put(0.86039439,0.08592135){\color[rgb]{0,0,0}\makebox(0,0)[lb]{\smash{$=\chi_{V}$}}}%
  \end{picture}%
\endgroup%

\end{aligned}
\end{align} 

%%%%%%%%%%%%%%%%%%%%%%%%%%%%%%%%%%%%%%%%%%%%%%%%%%%%%%%%%
%%%%%%%%%%%%%%%%%%%%%%%%%%%%%%%%%%%%%%%%%%%%%%%%%%%%%%%%%
\section{Group graded bimodules}  
%%%%%%%%%%%%%%%%%%%%%%%%%%%%%%%%%%%%%%%%%%%%%%%%%%%%%%%%%
%%%%%%%%%%%%%%%%%%%%%%%%%%%%%%%%%%%%%%%%%%%%%%%%%%%%%%%%%

Last section saw the introduction of a set of new preferred elements of $A$, $\eta_V$, $\rho_V$ and $\chi_V$. Given that any algebra $A$ is a bimodule over itself, it is straightforward to construct examples. However, we wish to find non-trivial examples by which we mean those not already studied for spin models without defects. To do so we introduce a new notion, that of a $H$-graded bimodule. 
\begin{definition}
Let $A$ be a $H$-graded algebra where $H$ is a finite group. Suppose $V$ is an $A$-$A$ bimodule that is also $H$-graded as a vector space. Then, $V$ is said to be a $H$-graded bimodule if $A_hV_lA_m \subset V_{hlm}$ for all $h,l,m \in H$.
\end{definition}
In section \S\ref{sec:graded_algebras} we introduced the notion of a bicharacter, a concept that was very useful in constructing examples of crossings for $A$. We will expand this concept so it can be used to determine the crossings that arise on a spin model with defects.
\begin{definition}
Let $W$ and $W'$ be $H$-graded. A map $\tilde{\lambda}_{W,W'}\colon H \times H \to k$ is said to be a $W$-$W'$ bicharacter if the crossing $\lambda \colon W \otimes W' \to W' \otimes W$ defined according to
\begin{align}
\lambda_{W,W'} (w_h \otimes w'_l)= \tilde{\lambda}_{W,W'}(h,l) w'_l \otimes w_h
\end{align} 
satisfies the axioms of definition~\ref{def:spin_defects}.
\end{definition}
We will understand what are the properties the $\tilde{\lambda}_{W,W'} \colon H \times H \to k$ must obey for a model with defect lines, when all the crossings on the model come from bicharacters. Note that an $A$-$A$ bicharacter is simply a bicharacter that satisfies the conditions of proposition~\ref{lem:graded}. We must, however, understand the properties of the remaining bicharacters for $W,W'=A,V,V^{\times}$.    

Let $W$ and $W'$ be such that a bilinear map $I^{-1} \colon W \otimes W' \to k$ exists. Such maps can be $B^{-1}$, $P^{-1}$ and $Q^{-1}$ as defined in equations \eqref{eq:snake}, \eqref{eq:H_inverse} and \eqref{eq:F_inverse}. Two subspaces $W_h$ and $W'_l$ are called orthogonal if $I(W_h,W'_l)=0$ and this fact is denoted as $W_h \perp W'_l$. By applying axiom \ref{it:B_defects} to $e_h \otimes e''_m \otimes e'_l \in W \otimes W'' \otimes W'$ we obtain the identity
\begin{align}
\tilde{\lambda}_{W'',W'}(m,l)I^{-1}(e_h,e'_l)e''_m=\tilde{\lambda}_{W,W''}(h,m)I^{-1}(e_h,e'_l)e''_m.
\end{align}
Therefore, for axiom \ref{it:B_defects} to be satisfied we must either have $W_h \perp W'_l$ or 
\begin{align}
\label{eq:general_bicharacter_relation}
\tilde{\lambda}_{W,W''}(h,m)=\tilde{\lambda}_{W'',W'}(m,l) \text{ for all } m \in H  \text{ and } W''=A,V,V^{\times}
\end{align}
which generalises property \ref{it:cr-axiom1} of proposition~\ref{lem:graded}.

Axioms \ref{it: RII_defects} and \ref{it: RIII_defects} have a straightforward interpretation: the former tell us 
$\tilde{\lambda}_{W,W'}(m,l)\tilde{\lambda}_{W',W}(l,m)=1_k$
whilst the latter is automatically satisfied given all $\tilde{\lambda}$ take values in $k$. This extends property~\ref{eq:graded_inverse} of a bicharacter.

Consider now the case when $W$ comes with a left action on $W'$, as are the maps $l$, $l^{\times}$ and $m$. This means we must have $W=A$. Then, axiom \ref{it:C_defects} applied to $e_h \otimes e'_m \otimes e''_l \in A \otimes W' \otimes W''$ translates into the condition $\tilde{\lambda}_{A,W'}(h,m)=\tilde{\lambda}_{W',W''}(m,l)\tilde{\lambda}_{W'',W'}(hl,m)$. Using axiom \ref{it: RII_defects} we can rewrite the condition as $\tilde{\lambda}_{W'',W'}(hl,m)=\tilde{\lambda}_{A,W'}(h,m)\tilde{\lambda}_{W'',W'}(l,m)$. On the other hand, if $W''=A$ has a right action on $W'$ instead, axiom \ref{it:C_defects} reads $\tilde{\lambda}_{W,W'}(hl,m)=\tilde{\lambda}_{W,W'}(h,m)\tilde{\lambda}_{W',A}(l,m)$. This means that in general we must have
\begin{align}
\label{eq:lambda_general}
\tilde{\lambda}_{W,W'}(hl,m)=\tilde{\lambda}_{A,W'}(h,m)\tilde{\lambda}_{W,W'}(l,m)=\tilde{\lambda}_{W,W'}(h,m)\tilde{\lambda}_{W',A}(l,m)
\end{align}

Finally, we address the ribbon condition \ref{it:axiom_ribbon_defects}. We write $I_1 \in W' \otimes W$ as $I_1=\sum e'_l \otimes e_h$ and $I_2 \in W \otimes W'$ as $I_2 = \sum f_h \otimes f'_l$ where the elements in both sums are subject to condition \eqref{eq:general_bicharacter_relation}. Note that $I_1$ can be either $B$, $P$ or $Q$; for each of these choices we will have $I_2$ equal to $B$, $Q$ or $P$. The curl maps on the left and right hand-side of axiom \ref{it:axiom_ribbon_defects} equate respectively to
\begin{align}
w_m &\mapsto \sum I_2^{-1}(e'_l,w_m)\tilde{\lambda}_{W,W}(h,m)e_h \\
w_m &\mapsto \sum I_1^{-1}(w_m,f'_l)\tilde{\lambda}_{W,W}(m,h)f_h
\end{align}
Consider the first of the curl maps. Since for non-zero $I_2^{-1}(e'_l,w_m)$ we must have $\tilde{\lambda}_{W',W''}(l,n)=\tilde{\lambda}_{W'',W}(n,m)$ we can conclude $\tilde{\lambda}_{W',W}(l,m)=\tilde{\lambda}_{W,W}(m,m)$. On the other hand, we know \eqref{eq:general_bicharacter_relation} holds which implies $\tilde{\lambda}_{W,W}(h,m)=\tilde{\lambda}_{W,W'}(m,l)$. Therefore we can conclude $\tilde{\lambda}_{W,W}(h,m)=\tilde{\lambda}_{W,W}(m,m)$. A similar treatment holds for the second form of the curling map. This allows to further identify the left and right curls as $w_m \mapsto \tilde{\lambda}_{W,W}(m,m)\sigma_W(w_m)$ and $w_m \mapsto \tilde{\lambda}_{W,W}(m,m)\sigma^{-1}_W(w_m)$, respectively. Therefore, the condition $\sigma_W^2=\text{id}$ must hold. 

Equation \eqref{eq:lambda_general} has strong implications on how $W$-$W'$ bicharacters relate to $A$-$A$ ones.
\begin{lemma} \label{lem:bi_simple}
Any $W$-$W'$ bicharacter satisfies
\begin{align} \label{eq:bi_simplification}
\tilde{\lambda}_{W,W'}(m,l)=\tilde{\lambda}_{W,W'}(1,1)\tilde{\lambda}_{A,A}(m,l).
\end{align}
Furthermore, $\tilde{\lambda}_{A,W}(1,1)=1$ for all $W$.
\end{lemma}
\begin{proof}
By restricting equation \eqref{eq:lambda_general} to $l=1$ we obtain
\begin{align}
\tilde{\lambda}_{W,W'}(h,m)=\tilde{\lambda}_{A,W'}(h,m)\tilde{\lambda}_{W,W'}(1,m)=\tilde{\lambda}_{W,W'}(h,m)\tilde{\lambda}_{W',A}(1,m).
\end{align}
The second identity guarantees we must have $\tilde{\lambda}_{W',A}(1,m)=1_k$ and, in particular, that $\tilde{\lambda}_{A,W'}(1,1)=1_k$. If we apply the first identity again to $\tilde{\lambda}_{W,W'}(1,m)=(\tilde{\lambda}_{W',W}(m,1))^{-1}$ we conclude $\tilde{\lambda}_{W,W'}(1,m)=\tilde{\lambda}_{W,W'}(1,1)\tilde{\lambda}_{W,A}(m,1)=\tilde{\lambda}_{W,W''}(1,1)$. By the same token, the relation $\tilde{\lambda}_{W',A}(m,h)=\tilde{\lambda}_{A,A}(m,h)$. The claimed result, equation \eqref{eq:bi_simplification}, follows. 
\end{proof}

Models with defects of the kind explored by Davydov, Kong and Runkel \cite{Runkel} use the canonical choice for every possible type of crossing. As we did with pure spin models, we are trying to establish a natural class of crossings that extends the canonical choice as to give rise to more interesting examples. One important step is to understand how partition functions for semi-simple objects can be generated from their simple components. Our first step is to understand how we can generate $W$-$W'$ bicharacters from $W_i$-$W_i'$ ones, where $W=W_1 \oplus W_2$ and $W'=W'_1 \oplus W'_2$. Note we also assume a direct sum of bilinear maps, $I=I_1 \oplus I_2$. However, the relation between $W$-$W'$ bicharacters and $A$-$A$ ones that we found through lemma~\ref{lem:bi_simple} makes our task simple -- as a direct consequence of lemma~\ref{lem:bicharacters_semi} we have the following.
\begin{corollary}
Let $\tilde{\lambda}_{W_i,W'_i} \colon H_i \times H_i \to k$ be $W_i$-$W_i'$ bicharacters. Then $W=W_1 \oplus W_2$ and $W'=W'_1 \oplus W'_2$ are $H$-graded with $H=H_1 \oplus H_2$ and it comes naturally equipped with a $W$-$W'$ bicharacter defined as
\begin{align}\label{eq:defects_semi_bicharacter}
\tilde{\lambda}_{W,W'}(g_1g_2,h_1h_2)=\tilde{\lambda}_{W_1,W'_1}(g_1,h_1)\tilde{\lambda}_{W_2,W'_2}(g_2,h_2)
\end{align}
where $g_1g_2$ and $h_1h_2$ belong to $H_1 \times H_2$. 
\end{corollary}

Ultimately, we are interested in studying the invariants objects like $\eta_V$, $\rho_V$ and $\chi_V$ can be used to create. We leave it as an exercise to the reader to deduce that if $\eta_{V_i} \in A_i$ are associated with $Z(T-S^1,s,A_i,V_i)$ then $Z(T-S^1,s,A_1\oplus A_2,V_1\oplus V_2)=\eta_{V_1}\oplus \eta_{V_2}$ with analogous conclusions valid also for $\rho_{V_i}$ and $\chi_{V_i}$.

We will now use the algebras of examples~\ref{ex:non_symmetric_Frobenius}-\ref{ex:G_general} to relate $\eta_V$, $\rho_V$, $\chi_V$ with the $\eta$ and $\chi$ of such models. The bimodules are $V=V^{\times}=A$ as vector spaces, with both left and right actions given by matrix multiplication -- note that for simple algebras any bimodule $V$ consists of a finite number of copies of $A$. The bilinear maps $P^{-1}$ and $Q^{-1}$ must therefore be non-degenerate linear maps $A \otimes A \to k$. Without loss of generality we can then write $P^{-1}(w,v)=\Tr(pwv)$ and $Q^{-1}(v,w)=\Tr(ovw)$ for some $p,o \in A$. However, the relations $\varphi_V(av)=\varphi(a)\varphi_V(v)$ and $\varphi_V(va)=\varphi_V(v)\varphi(a)$ imply that we must have $P^{-1}=Q^{-1}=B^{-1}$ or in other words $p=o=x$ such that $\varepsilon(a)=\Tr(xa)$ and $x^2=R\Tr(x)$.

Given the relation $\tilde{\lambda}_{A,V}=\tilde{\lambda}_{A,A}$ it is easy to conclude that for all examples~\ref{ex:non_symmetric_Frobenius}-\ref{ex:G_general} we have the relations:
\begin{align}
\eta_V=\eta ,\hspace{5mm}
\rho_V=\tilde{\lambda}_{V,V}(1,1)\eta, \hspace{5mm}
\chi_V=\tilde{\lambda}_{V,V}(1,1)\chi.
\end{align}
The constants $\tilde{\lambda}_{V,V}(1,1)$ satisfy $\tilde{\lambda}_{V,V}(1,1)=\pm 1$. The identities above therefore tell us that although models for $(\Sigma_g,l)$ can distinguish only a parity, the definition of parity itself might change. In particular, even in models for which $\chi=\eta$ we can now distinguish the parity of spin structures based on the number of defect curls present.

%%%%%%%%%%%%%%%%%%%%%%%%%%%%%%%%%%%%%%%%%%%%%%%%%%%%%%%%%
%%%%%%%%%%%%%%%%%%%%%%%%%%%%%%%%%%%%%%%%%%%%%%%%%%%%%%%%%
\chapter{Final remarks}\label{sec:cat}
%%%%%%%%%%%%%%%%%%%%%%%%%%%%%%%%%%%%%%%%%%%%%%%%%%%%%%%%%
%%%%%%%%%%%%%%%%%%%%%%%%%%%%%%%%%%%%%%%%%%%%%%%%%%%%%%%%%

This work has been driven by an effort to extend our knowledge of the two-dimensional world. Fundamentally, we have been able to understand the role played by spin structures when constructing Pachner-move-invariant state sum models. Nevertheless, this new information is merely a propeller for more questions. We shall address these first and foremost, here. 

\section{Summary of results}

The material of chapter \S\ref{chapter:pure_models}, devoted to the original constructions of two-dimensional models of the early 90s, is not new. Its main function was to present the reader with a pedagogical approach to how models on a triangulated surface are constructed. However, there was a clear intention to clarify the exact assumptions used and is this effort that allowed us to make the jump to the more interesting models of chapter \S\ref{sec:diagram}. The core of this thesis is chapter \S\ref{ch:spin}, where the spin calculus is developed -- we wish to highlight three of the results we can find therein.

\begin{description}
\item[Dependence on spin structure.] The spin calculus construction departs from the strategy of the preceding chapters in one fundamental area: no effort was made to present the models as in one-to-one correspondence with surfaces with spin structure, as it had been the case, for example, between FHK models and oriented surfaces up to homeomorphism. However, significant effort was made to first, motivate definition \ref{def:spin-model} of a spin model and two, conclude such models can only depend on the topology and spin structure of a given surface. Theorem \ref{spin-embedding} is a decisive step forward and a non trivial one, as highlighted by the non-existence of a similar result for spin defect models.    
\item[Dependence on spin parity.] Establishing this property was in fact a consistency test. From what is known of the mapping class group, the symmetry group of a surface with spin structure, it is possible to conclude from the outset that if a model is invariant under Pachner moves then it can only depend on the spin structure via parity \cite{Kusner}. It is nevertheless one of the achievements of this work to have clearly presented the algebraic translation of this dependence of spin structure through parity, a result encoded both by lemma \ref{lem:properties_eta_chi} and theorem \ref{theo:main}.   
\item[Non-trivial examples.] Although the results above suggest a framework that can produce spin-distinguishing models in theory, the strength of this work was built on the ability to construct explicit examples of such behaviour. The extensive study of group-graded algebras gave us a rich class of examples that allowed us to put to rest the question of existence of spin models which are not simply FHK, or even curl-free, models.    
\end{description}  

The two final chapters of this thesis try to expand the realm of planar, spherical and spin models to include defects -- loci where new information can be stored. The inspiration for this development comes most strongly from condensed matter theories where their presence has been used to model phase transitions, impurities and boundaries \cite{Kitaev}. It has received great attention from algebraic topological and conformal field theory groups and is another segment along which two-dimensional theories are being revisited \cite{Runkel}. The most important conclusions to take from chapters \S\ref{ch:defects} and \S\ref{ch:spin_defects} are as follows.

\begin{description}
%\item[Planar defect models are trivial.] This statement must be understood as true only for the narrow span of theories to which it applies. Therefore, we regard planar defect models as trivial when the data we work with are non-degenerate and when defect graphs are restricted to the interior of the disk (as assumed through chapters \S\ref{ch:defects} and \S\ref{ch:spin_defects}). They are considered trivial because they at most provide a modification of the Frobenius form used for the pure model. However, the spherical and spin defect models we can create recurring to the planar ones bring with them a lot more complexity.    
\item[Planar and spherical defect models are distinct.] For defect models the difference between developing a theory on a disk or on a sphere becomes clear -- this is a conceptual improvement over the original planar and spherical models, for which the data were shared. 
\item[Spin defect models are still largely an unknown.] The biggest conclusion we can draw from chapter \S\ref{ch:spin_defects} is that there is much to explore within spin defect models. However, being able to show that even simple defect patterns can have an impact on distinguishing spin structures through more than parity gives us a glimpse of what futures results could bring us.
\end{description}

\section{Future research directions}

There is one line of research that is in immediate need of attention: a more comprehensive look at the role of spin structures in surfaces. An obvious example is the need to understand what exactly is a spin defect model as defined in chapter \S\ref{ch:spin_defects} since as of yet we do not know what is the role of the mapping class group. What kind of invariance does regular homotopy  give rise to for spin defect graphs? Can we distinguish all spin structures?

However, there is another extension that would pose more conceptual questions: spin models for non-orientable surfaces, combining the knowledge of section \S\ref{sec:unoriented} and chapter \S\ref{ch:spin}. This is a more fundamental problem because a spin structure does not exist for non-orientable two-dimensional manifolds. Nonetheless, there is nothing on the algebraic side that would compel us to believe a spin graphical calculus depends on the existence of an orientation. This means it is our interpretation that is too naive: we are looking at a calculus more suited to be seen as reliant on pin structures (the generalisation of spin structures that is valid for all surfaces) \cite{Cimasoni2}. 

On a more practical level, we are also lacking a traditional gauge theory interpretation for the results of chapter \S\ref{ch:spin}. In particular, it would be of interest to explore the possible relation between example \ref{ex:H_awesome} and variations of the zero-area limit of Yang-Mills theory in two-dimensions, since such a relationship exists for the FHK models, both orientable and unoriented \cite{Witten}.

The categorical question should also not be sidestepped. As the categorically-minded reader might have already realised much of the thought-process employed (such as the duality between surfaces and vector spaces) and many of the axioms used (such as definitions \ref{def:spin-model} and \ref{def:spin_defects}) have been inspired by Topological Quantum Field Theory (TQFT) in the algebraic sense \cite{Kock}, and the axioms of several types of category \cite{Street}. However, we chose not to employ such a construction by design -- we think what was perhaps lost in generality has been gained in clarity of exposition. Nevertheless, much of the graphical calculus developed could be employed for categories not based on vector spaces, especially if one is interested in developing a richer class of examples --  another reason not to develop the categorical framework was our lack of interesting examples outside of the linear algebra realm.   

The mathematically-inclined reader will have noticed a double standard permeating this work: whilst in chapters \S\ref{chapter:pure_models}, \S\ref{sec:diagram} and \S\ref{ch:defects} there is a concerted effort to achieve some sort of classification of models, that same effort is not a part of chapters \S\ref{ch:spin} and \S\ref{ch:spin_defects}. There are two reasons for this flawed pattern. The first reason is the more hypothetical nature of the work in these two chapters. Although we know our axioms lead (or can potentially lead) to a certain behaviour under regular homotopy we have not been able to establish an `if and only if' premise -- for example, there is no claim spin models as defined here are the most general one could construct. This is a direction of work we wish to explore in the future. Second, there was one additional aspect of classification to be omitted: the classification of crossing maps. To resolve this problem we need to better understand the behaviour of these maps under automorphisms of the Frobenius algebra $A$ (that we know leave invariant the planar models) --  this is also the motivation behind the conjecture of example \ref{ex:H_awesome}.

A different angle worth exploring is presented in the work of Novak and Runkel \cite{Novak_Runkel} where a full TQFT for spin models is defined. Their conceptual understanding of how to encode the spin structure information is fundamentally different to ours. The approach of this thesis could be characterised as global: spin information is recorded through immersions; on the other hand Runkel and Novak propose a language where the spin structure data is provided locally as part of a new kind of triangulation. This idea has also been explored by Cimasoni and Reshetikhin \cite{Cimasoni} although the type of information recorded differs in these two works. Furthermore, Cimasoni has been able to extend such a framework to all surfaces \cite{Cimasoni2}.

An advantage of \cite{Novak_Runkel} is the treatment of surfaces with boundary from the outset. This thesis sidesteps the issues such a general treatment pose by choosing to deal only with spin models for manifolds without boundary. Developing a spin calculus for all surfaces whilst analysing the connections between spin models and spin TQFTs is high on our future agenda. Another work to keep in mind to achieve such a goal is that of Douglas, Schommer-Pries and Snyder \cite{Douglas}. Here, framed TQFTs are studied and the notion of spin is also developed in connection with immersions.

\appendix

%%%%%%%%%%%%%%%%%%%%%%%%%%%%
%%%%%%%%%%%%%%%%%%%%%%%%%%%%
\chapter{Equivalence of involutions}\label{app:TBA}
%%%%%%%%%%%%%%%%%%%%%%%%%%%%
%%%%%%%%%%%%%%%%%%%%%%%%%%%%

\begin{lemma} \label{lem:non_complex}
Let $A=\Mb_n(\Cb)$ with involution $\ast$ define a KM model. For $n$ odd there is one equivalence class of involutions; for $n$ even there are two. These classes are represented by $[a^{\ast}=a^{\tr}]$ in the first case and $[a^{\ast}=a^{\tr}]$, $[a^{\ast}=\Omega a^{\tr} \Omega^{-1}]$ in the second, where the matrix $\Omega$ is a standard anti-symmetric matrix.
\end{lemma}
\begin{proof}

It has already been established in lemma~\ref{lem:class_inv} that any involution $\ast$ for $\Mb_n(\Cb)$ acts as $a^{\ast}=sa^{\tr}s^{-1}$ with $s=\pm s^{\tr}$. If $\omega'$ is to be an isomorphism not only of FHK models but also of KM ones then it must preserve $\ast$ structures. Involutions $\ast$ and $\bullet$ are equivalent if they are related through an automorphism $\omega'$ as $\bullet=\omega' \circ \ast \circ (\omega')^{-1}$. Rewrite the $\bullet$ operation using the definition of $\omega'$:
\begin{align}
a^{\bullet}=(tst^{\tr})a^{\tr}(tst^{\tr})^{-1}.
\end{align}
Given the symmetry of $s$ and the freedom to choose $t$ one can use an appropriate transformation $tst^{\tr}$ to bring $s$ to a standard form. 

If $s$ is symmetric there exists an orthogonal matrix $o$ such that $oso^{\tr}$ is diagonal with entries $\lambda_k$. Since the algebra is complex, a diagonal matrix $d=\text{diag}\left(\frac{1}{\sqrt{\lambda_1}},\cdots,\frac{1}{\sqrt{\lambda_n}} \right)$ is well defined and can be used to rescale the eigenvalues of $s$. Then, for $t=do$ the transformation holds $tst^{\tr}=1$. In other words, the equivalence class can be represented by $\ast=\tr$.

If $s$ is skew-symmetric there is an orthogonal matrix $o$ such that $oso^{\tr}$ is block-diagonal where each block is constituted by a $2 \times 2$ matrix $\gamma=\left(\begin{smallmatrix}0 & 1\\ -1 &0\end{smallmatrix}\right)$ multiplied by some complex number $\theta_k$. Again, one can use a matrix $d=\text{diag}\left(\frac{1}{\sqrt{\theta_1}},\frac{1}{\sqrt{\theta_1}},\cdots,\frac{1}{\sqrt{\theta_{\frac{n}{2}}}},\frac{1}{\sqrt{\theta_{\frac{n}{2}}}}\right)$ to rescale the entries of $oso^{\tr}$. This means that by choosing $t=do$ we will bring $s$ to the standard skew-symmetric form $\Omega$ composed of $2 \times 2$ blocks $\gamma=\left(\begin{smallmatrix}0 & 1\\ -1 &0\end{smallmatrix}\right)$ along its diagonal. In other words, this equivalence class can be represented by $a^{\ast}=\Omega a^{\tr} \Omega^{-1}$.
\end{proof}

\begin{lemma}
Let $A=\Mb_{n}(D)$ with $D=\Rb$, $\Cb_{\Rb}$ or $\Hb_{\Rb}$ equipped with an involution $\ast$ determine a KM model. Then the model is isomorphic to one in which $a^{\ast}=sa^{\#}s^{-1}$ where the invertible element $s \in A$ and $\#$, a preferred choice of involution, satisfy the following.
\begin{enumerate}
\item \label{it:real1} For $D=\Rb$, $\#$ is the matrix transposition $\tr$. If $n$ is odd, $s$ takes the form
$
\eta(p,q)=\left(
\begin{smallmatrix}
1_{p\times p} & 0 \\
0 & - 1_{q \times q}
\end{smallmatrix}
\right)
$
where $p+q=n$; if $n$ is even then $s$ can also be the anti-symmetric matrix $\Omega$.
\item \label{it:real2} For $D=\Cb_{\Rb}$, $\#$ can be either matrix transposition $tr$ or hermitian conjugation $\dagger$. For $\#=\tr$, $s=1$ if $n$ is odd; if $n$ is even we can in addition have $s=\Omega$. For $\#=\dagger$, 
$
\eta(p,q)=\left(
\begin{smallmatrix}
1_{p\times p} & 0 \\
0 & - 1_{q \times q}
\end{smallmatrix}
\right)
$ where $p+q=n$.
\item \label{it:real3} For $D=\Hb$, $\#=\ddagger$, the quaternionic hermitian conjugation. Then $s$ is proportional to
$
\eta(p,q)=\left(
\begin{smallmatrix}
1_{p\times p} & 0 \\
0 & - 1_{q \times q}
\end{smallmatrix}
\right)
$
where $p+q=n$ and the proportionality constant is either $1$ or $\hat{\imath}$.
\end{enumerate} 
\end{lemma}
\begin{proof}

One starts with the case $D=\Rb$. Since $\tr$ is an involution also for real matrices the analysis of lemma \ref{lem:non_complex} leading to the conclusion that any $\ast$ structure is equivalent to one of the form $a^{\bullet}=(tst^{\tr})a^{\tr}(tst^{\tr})^{-1}$, where we are free to choose the invertible element $t \in A$, still follows. The fact $t$ is now a real matrix, however, will condition the transformations $tst^{\tr}$. 

If $s$ is symmetric there exists a real orthogonal matrix $o$ that will allow one to diagonalise it; in particular it is possible to choose the matrix $o$ so that the first $p$  diagonal entries $\lambda_1,\cdots,\lambda_p$ are positive eigenvalues ($p$ is the total number of positive eigenvalues). A matrix $d=\text{diag}\left(\frac{1}{\sqrt{\lambda_1}},\cdots,\frac{1}{\sqrt{\lambda_p}},\frac{1}{\sqrt{-\lambda_{p+1}}},\cdots,\frac{1}{\sqrt{-\lambda_n}} \right)$ is well defined and can be used to rescale the eigenvalues of $s$. However, the eigenvalue signs cannot be modified. The transformation $(do)s(do)^{tr}$ will result on the matrix $\eta(p,q)$ with $p+q=n$. 

If $n$ is even, $s$ could also be anti-symmetric. For such $s$ there exists an orthogonal matrix $o$ such that $oso^{\tr}$ is block-diagonal where each block constitutes of a $2 \times 2$ matrix $\gamma=\left(\begin{smallmatrix}0 & 1\\ -1 &0\end{smallmatrix}\right)$ multiplied by some number $\theta_k$. In particular, the $\theta_k$ are real and positive. This means the matrix $d=\text{diag}\left(\frac{1}{\sqrt{\theta_1}},\frac{1}{\sqrt{\theta_1}},\cdots,\frac{1}{\sqrt{\theta_{\frac{n}{2}}}},\frac{1}{\sqrt{\theta_{\frac{n}{2}}}}\right)$ is well defined and can be used to rescale the entries of $oso^{\tr}$. Therefore, by choosing $t=do$ we obtain $tst^{\tr}=\Omega$.

For $D=\mathbb{C}_{\Rb}$, one can either have $a^{\ast}=sa^{\tr}s^{-1}$ with $s=\pm s^{\tr}$ or $a^{\ast}=sa^{\dagger}s^{-1}$ with $s=\mu s^{\dagger}$, $|\mu|=1$. A treatment similar to that of lemma \ref{lem:non_complex} leads to the stated conclusions for $\ast=\tr$ and is therefore not repeated here. As previously studied, KM models are regarded as equivalent if their involutions $\ast,\bullet$ can be related via an inner automorphism $\omega'$ as $\bullet=\omega' \circ \ast \circ (\omega')^{-1}$. (One could also include the action of $\psi \colon \mathcal{Z}(D) \to \mathcal{Z}(D)$ on the automorphism since $\varepsilon \circ \psi = \varepsilon$; however, there is no loss of generality in not doing so for the current purpose.) 

For our specific $\ast=\dagger$ this means any $\bullet$ acting as
\begin{align}
a^{\bullet}=(tst^{\dagger})a^{\dagger}(tst^{\dagger})^{-1}
\end{align}
will give rise to the same state sum model. Consider for the moment the case $s=s^{\dagger}$. Then there exists unitary $u$ such that $usu^{\dagger}$ is a diagonal matrix with eigenvalues $\lambda_k$ as the non-zero entries. Since the eigenvalues of a hermitian matrix are always real, one can choose with no loss of generality the first $p$ eigenvalues to be positive and the remaining $q$ to be negative, with $p+q=n$. The matrix $d=\text{diag}\left(\frac{1}{\sqrt{\lambda_1}},\cdots,\frac{1}{\sqrt{\lambda_n}} \right)$ can be used to rescale them. Therefore, for $t=du$ one obtains $tst^{\dagger}=\eta(p,q)$. On the other hand, if $s=\mu s^{\dagger}$ for $|\mu|=1$ then there exists a hermitian matrix $h$ and a root of $\mu$ such that $s=\sqrt{\mu}h$. As before, one could then choose a matrix $t$ so that $tst^{\dagger}=\sqrt{\mu}\eta(p,q)$ and the involution reduces to $a^{\bullet}=\eta(p,q)a^{\dagger}\eta(p,q)$. On the other hand, we would have $(tst^{\dagger})^{-1}=\sqrt{\mu^{-1}}\eta(p,q)$. The involution would still be reduced to $a^{\bullet}=\eta(p,q)a^{\dagger}\eta(p,q)$. 

Finally one must handle the quaternionic case, for which any involution satisfies $a^{\bullet}=sa^{\ddagger}s^{-1}$ and $s=\pm s^{\ddagger}$. This relation for $s$ means $s$ is normal ($ss^{\ddagger}=s^{\ddagger}s$) and any quaternionic normal matrix is equivalent through a quaternionic unitary transformation $u$ to a diagonal matrix with complex entries \cite{Loring} -- here, complex denotes an entry of the form $a + \hat{\imath}b +\hat{\jmath}.0 +\hat{k}.0$. But it is already known that any hermitian (anti-hermitian) complex matrix is similar to $\eta(p,q)$ ($\hat{\imath}.\eta(p,q)$) through a complex unitary transformation. Therefore, by choosing the appropriate automorphism one can have either $a^{\bullet}=\eta(p,q)a^{\ddagger}\eta(p,q)$ or $a^{\bullet}=-\hat{\imath}\eta(p,q)a^{\ddagger}\hat{\imath}\eta(p,q)$.   
\end{proof}

%%%%%%%%%%%%%%%%%%%%%%%%%%%%
%%%%%%%%%%%%%%%%%%%%%%%%%%%%
\chapter{Finding all spin crossings for $\Cb\Zb_n$}\label{app:code}
%%%%%%%%%%%%%%%%%%%%%%%%%%%%
%%%%%%%%%%%%%%%%%%%%%%%%%%%%

\textit{Mathematica} code based on {`}Two-dimensional state sum models and spin structures{'} by John W. Barrett and Sara O. G. Tavares, prepared by Antony R. Lee and Sara O. G. Tavares. Please
contact the authors to reproduce any section of the code. 

We wish to study all crossings compatible with a spin state sum model in a special class of algebras: $\Cb\Zb_n$ where $\Zb_n$ is the cyclic group of order $n$ generated by $h$. The axioms 1 to 5 used throughout are those of definition~\ref{def:spin-model} for a spin crossing. Einstein{'}s summation convention is used throughout with other sums displayed explicitly.

We start by defining the order of the group through the variable dim.

\begin{doublespace}
\noindent\(\pmb{\quad \dim =3;}\)
\end{doublespace}

The variable A is a tensor in four indices that will store the information for the crossing map. In other words, \(\lambda :A \otimes A \rightarrow
A\otimes A\) can be written as \(\lambda =\Sigma _{ijkl}\lambda _{i,j}{}^{k,l}h^i\otimes h^j\otimes h^k\otimes h^l\) where the group elements,
powers of \(h\), are taken as the algebra basis. The components of $\lambda $ are then stored in A. 

\begin{doublespace}
\noindent\(\pmb{\quad A=\text{Array}\left[a_{\#\#}\&,\{\dim ,\dim ,\dim ,\dim \}\right];}\)
\end{doublespace}

\begin{doublespace}
\noindent\(\pmb{\quad \text{list0}=\text{Table}\left[0,\left\{i,1,\dim ^3\right\}\right];}\\
\pmb{\quad \text{list1}=\text{Table}\left[0,\left\{i,1,\dim ^4\right\}\right];}\\
\pmb{\quad \text{list2}=\text{Table}\left[0,\left\{i,1,\dim ^4-\dim ^3\right\}\right];}\)
\end{doublespace}

A combination of axioms 2 and 3 allows us to simplify the tensorial form of the crossing: it is easy to conclude \(\lambda _{0,j}{}^{k,l}=\delta
_0{}^l\delta _j{}^k\) . Although we will still enforce axioms 2 and 3 independently we can use this simplification to expedite calculations. We use
list0 to store the values \(\lambda _{0,j}{}^{k,l}\) must take. There are \(\dim ^3\) of these. Note that our index system records (powers of \(h\))
+ 1 instead of powers of \(h\). This is due to the internal labelling system of \textit{Mathematica} i.e. the zeroth element of an array does not
exist in \textit{Mathematica}. 

\begin{doublespace}
\noindent\(\pmb{\quad \text{tot}=0;}\\
\pmb{\quad \text{For}[j=1,j\leq \dim ,j\text{++},}\\
\pmb{\quad \text{For}[k=1,k\leq \dim ,k\text{++},}\\
\pmb{\quad \text{For}[l=1,l\leq \dim ,l\text{++},\text{tot}\text{++};}\\
\pmb{\quad \quad \text{list0}[[\text{tot}]]=\{A[[1,j,k,l]]=\text{KroneckerDelta}[1,l]\text{KroneckerDelta}[j,k]\}}\\
\pmb{\quad ]]]}\\
\pmb{\quad \text{Clear}[j,k,l]}\)
\end{doublespace}

Similarly we can conclude axiom 1 corresponds to \(\lambda _{i,j}{}^{k,l}=\lambda _{-k,i}{}^{l,-j}\) . We use list1 to store the equations the identity
gives rise to. As we said indices record (powers of \(h\)) + 1. To appropriately calculate \(-i\) we must therefore compute { }\([-i+1 (\text{mod}
\dim )]+1\). We recall dim is the order of the cyclic group.

\begin{doublespace}
\noindent\(\pmb{\quad \text{tot}=0;}\\
\pmb{\quad \text{For}[i=1,i\leq \dim ,i\text{++},}\\
\pmb{\quad \text{For}[j=1,j\leq \dim ,j\text{++},}\\
\pmb{\quad \text{For}[k=1,k\leq \dim ,k\text{++},}\\
\pmb{\quad \text{For}[l=1,l\leq \dim ,l\text{++},\text{tot}\text{++};}\\
\pmb{\quad \quad\text{list1}[[\text{tot}]]=\{A[[i,j,k,l]]\text{==}A[[\text{Mod}[-k+1,\dim ]+1,i,l,}\\
\pmb{\quad \quad\text{Mod}[-j+1,\dim]+1]]\}}\\
\pmb{\quad ]]]]}\\
\pmb{\quad \text{Clear}[i,j,k,l,\text{tot}]}\)
\end{doublespace}

We now store the components of the crossing on a new list, list2.

\begin{doublespace}
\noindent\(\pmb{\quad \text{tot}=0;}\\
\pmb{\quad \text{For}[i=2,i\leq \dim ,i\text{++},}\\
\pmb{\quad \text{For}[j=1,j\leq \dim ,j\text{++},}\\
\pmb{\quad \text{For}[k=1,k\leq \dim ,k\text{++},}\\
\pmb{\quad \text{For}[l=1,l\leq \dim ,l\text{++},\text{tot}\text{++};\text{list2}[[\text{tot}]]=A[[i,j,k,l]]}\\
\pmb{\quad ]]]]}\\
\pmb{\quad \text{Clear}[i,j,k,l,\text{tot}]}\)
\end{doublespace}

We will use the equations in list1 to simplify the variables in list2, the components of the crossing. The solutions to these equations are stored
in a new tensor, B. It is worth noting we expect there to be more variables than independent equations. Hence the resulting warning -- \textit{Solve::svars:
{``}Equations may not give solutions for all $\backslash ${''}solve$\backslash ${''} variables.{''}} -- when the program is run which will not affect
the outcome. 

The operations performed over list1 are designed to eliminate unnecessary equations. The \textit{Mathematica} instruction to delete duplicates is
self-explanatory; erasing the first entry of the list, however, might not be so. This is done to get rid of a trivial condition: an equation which
is already identified as being {`}True{'} and that always appears in the first entry of the list after duplicates are removed. This approach will
be followed throughout the program when dealing with systems of equations. Flatten reduces the produced list to its most basic form within \textit{Mathematica}.

\begin{doublespace}
\noindent\(\pmb{\quad \text{eqns}=\text{Delete}[\text{list1}\text{//}\text{DeleteDuplicates},1]\text{//}\text{Flatten};}\\
\pmb{\quad \text{sols}=\text{Solve}[\text{eqns},\text{list2}];}\)
\end{doublespace}

In the above, Solve is a Mathematica routine which solves a list of equations (eqns) for a given list of variables (list2). Finally, we store the
results in the array B (which can be viewed as a matrix for convenience).

\begin{doublespace}
\noindent\(\pmb{\quad B=\text{Flatten}[A\text{/.}\text{sols},1];}\\
\pmb{\quad B\text{//}\text{MatrixForm};}\)
\end{doublespace}

The map \(\varphi :A \rightarrow A\) is now defined. The components \(\varphi \left(h^i\right)= \Sigma _j\varphi ^i{}_j\left(h^j\right)\),
\(\varphi ^i{}_j= \lambda _{i,k}{}^{j,k}\) will be stored through the tensor f.

\begin{doublespace}
\noindent\(\pmb{\quad f[\text{i$\_$},\text{j$\_$}]\text{:=}\text{Sum}[B[[i,k,j,k]],\{k,1,\dim \}];}\)\\
\noindent\(\pmb{\quad \text{list3}=\text{Table}\left[0,\left\{i,1,\dim ^5\right\}\right];}\\
\pmb{\quad \text{list4}=\text{Table}\left[0,\left\{i,1,\dim ^4\right\}\right];}\\
\pmb{\quad \text{list5}=\text{Table}\left[0,\left\{i,1,\dim ^6\right\}\right];}\\
\pmb{\quad \text{list6}=\text{Table}\left[0,\left\{i,1,\dim ^2\right\}\right];}\\
\pmb{\quad \text{list7}=\text{Table}\left[0,\left\{i,1,\dim ^2\right\}\right];}\\
\pmb{\quad \text{list8}=\text{Table}\left[0,\left\{i,1,\dim ^2\right\}\right];}\)
\end{doublespace}

To store the equations axiom 2 imposes, \(\lambda _{g,h}{}^{m,-l+p}=\lambda _{h,l}{}^{-g+k,n}\lambda _{k,n}{}^{m,p}\), we use list3. There are \(\dim ^5\) of them.

\begin{doublespace}
\noindent\(\pmb{\quad \text{tot}=0;}\\
\pmb{\quad \text{For}[g=1,g\leq \dim ,g\text{++},}\\
\pmb{\quad \text{For}[h=1,h\leq \dim ,h\text{++},}\\
\pmb{\quad \text{For}[m=1,m\leq \dim ,m\text{++},}\\
\pmb{\quad \text{For}[p=1,p\leq \dim ,p\text{++},}\\
\pmb{\quad \text{For}[l=1,l\leq \dim ,l\text{++},}\\
\pmb{\quad \text{tot}\text{++};}\\
\pmb{\quad \text{list3}[[\text{tot}]]=\{B[[g,h,m,\text{Mod}[\text{Mod}[-l+1,\dim ]+(p-1),\dim ]+1]]}\\
\pmb{\quad \quad ==\text{Sum}[B[[h,l,\text{Mod}[\text{Mod}[-g+1,\dim ]+(k-1),\dim ]+1,n]]B[[k,n,m,p]],}\\
\pmb{\quad \quad \{k,1,\dim \},\{n,1,\dim \}]\}}\\
\pmb{\quad ]]]]];}\\
\pmb{\quad \text{Clear}[g,h,m,p,l,\text{tot}]}\)
\end{doublespace}

We can now introduce axiom 3, the Reidemeister II move: \(\lambda _{i,j}{}^{m,n}\lambda _{m,n}{}^{k,l}=\delta _i{}^k\delta _j{}^l\). list4 is used
to store these identities; there are \(\dim ^4\) of them. 

\begin{doublespace}
\noindent\(\pmb{\quad \text{tot}=0;}\\
\pmb{\quad \text{For}[i=1,i\leq \dim ,i\text{++},}\\
\pmb{\quad \text{For}[j=1,j\leq \dim ,j\text{++},}\\
\pmb{\quad \text{For}[k=1,k\leq \dim ,k\text{++},}\\
\pmb{\quad \text{For}[l=1,l\leq \dim ,l\text{++},\text{tot}\text{++};}\\
\pmb{\quad \text{list4}[[\text{tot}]]=\{\text{Sum}[B[[i,j,m,n]]B[[m,n,k,l]],\{m,1,\dim \},\{n,1,\dim \}]}\\
\pmb{\quad \quad ==\text{KroneckerDelta}[i,k]\text{KroneckerDelta}[j,l]\}}\\
\pmb{\quad ]]]];}\\
\pmb{\quad \text{Clear}[i,j,k,l,\text{tot}]}\)
\end{doublespace}

Axiom 4 translates into \(\lambda _{g,h}{}^{s,n}\lambda _{n,l}{}^{t,r}=\lambda _{h,l}{}^{m,n}\lambda _{g,m}{}^{o,p}\lambda
_{p,n}{}^{q,r}\lambda _{o,q}{}^{s,t}\). There are \(\dim ^6\) such identities and they are stored in list5.

\begin{doublespace}
\noindent\(\pmb{\quad \text{tot}=0;}\\
\pmb{\quad \text{For}[g=1,g\leq \dim ,g\text{++},}\\
\pmb{\quad \text{For}[h=1,h\leq \dim ,h\text{++},}\\
\pmb{\quad \text{For}[l=1,l\leq \dim ,l\text{++},}\\
\pmb{\quad \text{For}[s=1,s\leq \dim ,s\text{++},}\\
\pmb{\quad \text{For}[t=1,t\leq \dim ,t\text{++},}\\
\pmb{\quad \text{For}[r=1,r\leq \dim ,r\text{++},}\\
\pmb{\quad \text{tot}\text{++};\text{list5}[[\text{tot}]]=\{\text{Sum}[B[[g,h,s,n]]*B[[n,l,t,r]],\{n,1,\dim \}]}\\
\pmb{\quad \quad ==\text{Sum}[B[[h,l,m,n]]*B[[g,m,o,p]]*B[[p,n,q,r]]*B[[o,q,s,t]],}\\
\pmb{\quad \quad \{m,1,\dim \},\{n,1,\dim \},\{o,1,\dim \},\{p,1,\dim \},\{q,1,\dim \}]\}}\\
\pmb{\quad ]]]]]];}\\
\pmb{\quad \text{Clear}[g,h,l,k,s,t,r,\text{tot}]}\)
\end{doublespace}

Axiom 5, which translates into \(\varphi ^i{}_j=\varphi ^j{}_i\) or \(\varphi ^i{}_m\varphi ^m{}_j=\delta ^i{}_j\) is now imposed. The resulting
identities are stored in list6 and list7, respectively.

\begin{doublespace}
\noindent\(\pmb{\quad \text{tot}=0;}\\
\pmb{\quad \text{For}[i=1,i\leq \dim ,i\text{++},}\\
\pmb{\quad \text{For}[j=1,j\leq \dim ,j\text{++},}\\
\pmb{\quad \text{tot}\text{++};\text{list6}[[\text{tot}]]=\{f[i,j]==f[j,i]\};}\\
\pmb{\quad ]];}\\
\pmb{\quad \text{Clear}[i,j,\text{tot}]}\\
\pmb{}\\
\pmb{\quad \text{tot}=0;}\\
\pmb{\quad \text{For}[i=1,i\leq \dim ,i\text{++},}\\
\pmb{\quad \text{For}[j=1,j\leq \dim ,j\text{++},}\\
\pmb{\quad \text{tot}\text{++};\text{list7}[[\text{tot}]]=\{\text{Sum}[f[i,k]f[k,j],\{k,1,\dim \}]==\text{KroneckerDelta}[i,j]\};}\\
\pmb{\quad ]];}\\
\pmb{\quad \text{Clear}[i,j,\text{tot}]}\)
\end{doublespace}

To ensure internal consistency we impose some identities the axioms would make redundant. For example, we know the map \(p :A\rightarrow  A\) must
obey \(p \circ  \varphi =p\). We store the component equalities the expression gives rise to, \(\Sigma _l\text{  }\lambda _{h,g}{}^{l,-l+k+h}=\Sigma
_l \varphi _g{}^m\lambda _{h,m}{}^{l,-l+k+h}\) , in list8. There are \(\dim ^2\) of these.

\begin{doublespace}
\noindent\(\pmb{\quad \text{tot}=0;}\\
\pmb{\quad \text{For}[g=1,g\leq \dim ,g\text{++},}\\
\pmb{\quad \text{For}[k=1,k\leq \dim ,k\text{++},}\\
\pmb{\quad \text{tot}\text{++};}\\
\pmb{\quad \text{list8}[[\text{tot}]]=}\\
\pmb{\quad \quad \{\text{Sum}[B[[h,g,l,\text{Mod}[\text{Mod}[-l+1,\dim ]+(k-1)+(h-1),\dim ]+1]],}\\
\pmb{\quad \quad \{l,1,\dim \},\{h,1,\dim \}]}\\
\pmb{\quad ==\text{Sum}[f[g,m]*}\\
\pmb{\quad \quad B[[h,m,l,\text{Mod}[\text{Mod}[-l+1,\dim ]+(k-1)+(h-1),\dim ]+1]],}\\
\pmb{\quad \quad \{l,1,\dim \},\{h,1,\dim \},\{m,1,\dim \}]\};}\\
\pmb{\quad ]];}\\
\pmb{\quad \text{Clear}[g,k,\text{tot}];}\)
\end{doublespace}

We are finally ready to find all the spin crossings. We first concatenate the full set of equations (list3 to list8) into a single list. We then
solve these equations for the variables contained within the array B, the components of the crossing. The solutions are initially stored in sols2
and then extracted for viewing convenience in \(\text{CC}_i\).

\begin{doublespace}
\noindent\(\pmb{\quad \text{eqns2}=\text{Join}[}\\
\pmb{\quad \text{Delete}[\text{list3}\text{//}\text{Simplify}\text{//}\text{DeleteDuplicates},1]\text{//}\text{Flatten},}\\
\pmb{\quad \text{Delete}[\text{list4}\text{//}\text{DeleteDuplicates},1]\text{//}\text{Flatten},}\\
\pmb{\quad \text{Delete}[\text{list5}\text{//}\text{DeleteDuplicates}\text{//}\text{Simplify},1]\text{//}\text{Flatten},}\\
\pmb{\quad \text{Delete}[\text{list6}\text{//}\text{DeleteDuplicates},1]\text{//}\text{Flatten},}\\
\pmb{\quad \text{Delete}[\text{list7}\text{//}\text{DeleteDuplicates},1]\text{//}\text{Flatten},}\\
\pmb{\quad \text{Delete}[\text{list8}\text{//}\text{Simplify}\text{//}\text{DeleteDuplicates},1]\text{//}\text{Flatten}];}\\
\pmb{\quad \text{Parallelize}[\text{sols2}}\\
\pmb{\quad \quad =\text{Solve}[\text{eqns2},\text{Delete}[B\text{//}\text{Flatten}\text{//}\text{DeleteDuplicates},\{\{1\},\{2\}\}]]];}\)\\
\noindent\(\pmb{\quad \text{CC}=\text{Table}\left[\text{CC}_i=B\text{/.}\text{sols2}[[i]],\{i,1,\text{Length}[\text{sols2}]\}\right];}\)
\end{doublespace}

All possible spin crossings, for the inital order dim, are displayed.

\begin{doublespace}
\noindent\(\pmb{\quad \text{Table}[\text{CC}[[i]]\text{//}\text{MatrixForm},\{i,1,\text{Length}[\text{sols2}]\}]}\)
\end{doublespace}

We are one step away of constructing partition functions. We need only determine the values of $\eta$ and $\chi$ associated with each of the valid
spin crossings. Some of those will be topological, others will not be topological but will still not distinguish spin structures as well and, finally,
a few ones will distinguish the spin parity.

We start by defining the value of \(\varphi ^i{}_j\) for each possible crossing solution \(\text{CC}_l\). This information will be stored in the
variable gg[i,j,l]. 

\begin{doublespace}
\noindent\(\pmb{\quad \text{gg}[\text{i$\_$},\text{j$\_$},\text{l$\_$}]\text{:=}\text{Sum}[\text{CC}[[l]][[i,k,j,k]],\{k,1,\dim \}];}\)
\end{doublespace}

As an element of the algebra, $\eta $ can be written as \(\eta =\Sigma _k\eta _kh^k\). For each crossing solution \(\text{CC}_i\) we present the
components of $\eta $ stored through the variable $\eta $[i,k]. The multiplicative constant \(\left|\Zb_n|^{-2} R^{-2}\right.\)
is omitted.

\begin{doublespace}
\noindent\(\pmb{\quad \eta [\text{i$\_$},\text{k$\_$}]\text{:=}}\\
\pmb{\quad \text{Sum}[\text{CC}[[i]][[h,g,\text{Mod}[\text{Mod}[-l+1,\dim ]+h+k+g-3,\dim
]+1,l]],}\\
\pmb{\quad \quad \{h,1,\dim \},\{g,1,\dim \},\{l,1,\dim \}];}\)
\end{doublespace}

As an element of the algebra, $\chi$ can be written as \(\chi =\Sigma _k\chi _kh^k\). For each crossing solution \(\text{CC}_i\) we present the
components of $\chi $ stored through the variable $\chi $[i,k]. The multiplicative constant \(\left| \Zb_n|^{-2} R^{-2}\right.\)
is omitted.

\begin{doublespace}
\noindent\(\pmb{\quad \chi [\text{i$\_$},\text{k$\_$}]\text{:=}}\\
\pmb{\quad \text{Sum}[\text{gg}[h,p,i]\text{gg}[\text{Mod}[-g+1,\dim ]+1,q,i]}\\
\pmb{\quad \text{CC}[[i]][[\text{Mod}[-h+1,\dim ]+1,g,l,}\\
\pmb{\quad \quad \text{Mod}[ \text{Mod}[-l+1,\dim ]}\\
\pmb{\quad \quad +\text{Mod}[-p+1,\dim ]+\text{Mod}[-q+1,\dim ]+(k-1),\dim ]+1]],}\\
\pmb{\quad \{h,1,\dim \},\{g,1,\dim \},\{l,1,\dim \},\{p,1,\dim \},\{q,1,\dim \}];}\)
\end{doublespace}

Knowing the possible pairs ($\eta $,$\chi $) is all the information we need to compute partition functions. We invite the reader to verify all examples
do satisfy \(\eta ^2=\chi ^2\).

\begin{doublespace}
\noindent\(\pmb{\quad \text{Table}[\eta [i,k],\{i,1,\text{Length}[\text{sols2}]\},\{k,1,\dim \}]}\\
\pmb{\quad \text{Table}[\chi [i,k],\{i,1,\text{Length}[\text{sols2}]\},\{k,1,\dim \}]}\)
\end{doublespace}

\listoffigures

\printglossary[title=List of Symbols,style=mcolindex]

\renewcommand{\bibname}{References} % changes the header; default: Bibliography
\bibliographystyle{unsrt}
\bibliography{6_backmatter/references} % adjust this to fit your BibTex file

\end{document}